\pdfoutput=1
\documentclass[11pt]{book}
\usepackage{index}
\makeindex

\newindex{notation}{adx}{and}{Index of symbols}
\newindex{notion}{bdx}{bnd}{Index of notions}
\input{package.tex}
\newcommand\notion[1]{\textit{#1}\index[notion]{#1}}
\newcommand\wcnotion[2]{\textit{#1}\index[notion]{#2}}
\newcommand\wcnotionsym[3]{\textit{#1}\index[notation]{#2}\index[notion]{#3}}
\newcommand\wcsnotion[3]{\textit{#1}\index[notion]{#2!\textit{#3}}}
\newcommand\snotion[2]{\textit{#1}\index[notion]{#1!\textit{#2}}}
\newcommand\snotionsym[3]{\textit{#1}\index[notion]{#1!\textit{#3}}\index[notation]{#2!\textit{#3}}}

\newcommand\wcsnotionsym[4]{\textit{#1}\index[notation]{#2!\textit{#4}}\index[notion]{#3!\textit{#4}}}

\newcommand\wcnotation[2]{\textit{#1}\index[notation]{#2}}

\newcommand\sym[1]{\index[notation]{#1}}
\newcommand\ssym[2]{\index[notation]{#1!\textit{#2}}}

\newcommand{\exclam}{!}

\newcommand{\Zb}{\mathbb{Z}} 
 
\newcommand{\Nb}{\mathbb{N}}
\newcommand{\Tb}{\mathbf{T}} 
 
\newcommand{\Ib}{\mathbb{I}}

\newcommand{\Sb}{\mathbb{S}}

\def\-{\raisebox{.75pt}{-}}

\newcommand{\uvar}{\_}

\newcommand{\Db}{\mathbf{D}} 
\DeclareMathOperator*{\dom}{dom}
\DeclareMathOperator*{\codom}{codom}
\DeclareMathOperator{\tw}{tw}

\newcommand{\Rb}{\mathbf{R}} 
\newcommand{\Lb}{\mathbf{L}} 
\newcommand{\Fb}{\mathbf{F}} 
\DeclareMathOperator{\Gb}{G} 
  
\DeclareMathOperator{\N}{N}
\DeclareMathOperator{\T}{T}
\DeclareMathOperator{\J}{J}

\DeclareMathOperator*{\W}{W}
\DeclareMathOperator*{\Wm}{tW}
\DeclareMathOperator*{\Wseg}{W_{Seg}}
\DeclareMathOperator*{\Wsat}{W_{Sat}}

\DeclareMathOperator*{\M}{M}
\DeclareMathOperator*{\Mm}{tM}
\DeclareMathOperator*{\Mseg}{M_{Seg}}
\DeclareMathOperator*{\Msat}{M_{Sat}}

\DeclareMathOperator*{\I}{I}
\DeclareMathOperator*{\F}{F}

\DeclareMathOperator*{\CDA}{ADC}
\DeclareMathOperator*{\CDAB}{ADC_B}

\newcommand\omegacat{\omega\mbox{-$\cat$}}
\DeclareMathOperator\Set{Set}
\DeclareMathOperator\Sp{Sp}

\DeclareMathOperator*{\Sq}{Sq}

\DeclareMathOperator{\Hom}{Hom}

\DeclareMathOperator*{\Lfib}{LFib}

\DeclareMathOperator*{\LCartoperator}{LCart}

\newcommand{\LCart}{\mbox{$\LCartoperator$}}

\newcommand{\LCartc}{\mbox{$\LCartoperator$}^c}
\DeclareMathOperator*{\RCart}{RCart}

\newcommand{\uLCart}{\underline{\LCartoperator}}
\newcommand{\uLCartc}{\underline{\LCartoperator}^c}

\DeclareMathOperator{\uHom}{\underline{Hom}}

\DeclareMathOperator{\gHom}{\underline{Hom}_{\ominus}}
\DeclareMathOperator{\Map}{Map}

\DeclareMathOperator{\im}{Im}

\newcommand{\uni}{\underline{\omega}}
\newcommand\w[1]{\widehat{#1}}

\DeclareMathOperator*{\ev}{ev}
\DeclareMathOperator*{\Arr}{Arr}
\newcommand{\Noiun}{\N_{\tiny{(\omega,1)}}}

\newcommand{\colim}{\operatornamewithlimits{colim}}
\newcommand{\laxcolim}{\operatornamewithlimits{laxcolim}}
\newcommand{\laxlim}{\operatornamewithlimits{laxlim}}

\DeclareMathOperator{\Lan}{Lan}

\newcommand\iun{(\infty,1)}

\newcommand\io{(\infty,\omega)}
\newcommand\ioun{(\infty,\omega,1)}

\newcommand\zo{(0,\omega)}

\newcommand{\costar}{\mathbin{\overset{co}{\star}}}

\DeclareMathOperator\cst{cst}

\DeclareMathOperator\Fun{Fun}

\DeclareMathOperator\mcat{cat_m}

\DeclareMathOperator\cat{cat}

\DeclareMathOperator\grd{grd}

\newcommand\ocat{(\infty,\omega)\mbox{-$\cat$}}
\newcommand\ouncat{(\infty,\omega,1)\mbox{-$\cat$}}

\newcommand\ocatm{{(\infty,\omega)\mbox{-$\mcat$}}}

\newcommand\zocatm{(0,\omega)\mbox{-$\mcat$}}
\newcommand\zocat{(0,\omega)\mbox{-$\cat$}}
\DeclareMathOperator\zocatB{\zocat_B}
\newcommand\icat{(\infty,1)\mbox{-$\cat$}}

\newcommand\qcat{\mbox{Q$\cat$}}
\newcommand\ncat[1]{(\infty, #1)\mbox{-$\cat$}}
\newcommand\zncat[1]{(0, #1)\mbox{-$\cat$}}
\newcommand\igrd{\infty\mbox{-$\grd$}}
\newcommand\ngrd[1]{ #1\mbox{-$\grd$}}

\DeclareMathOperator{\OperatorinfiniPsh}{Psh^\infty}
\DeclareMathOperator{\OperatorinfinitPsh}{tPsh^\infty}
\DeclareMathOperator{\OperatorPsh}{Psh}

\DeclareMathOperator{\OperatormPsh}{mPsh}
\DeclareMathOperator{\OperatortPsh}{tPsh}
\newcommand\iPsh[1]{\OperatorinfiniPsh({#1})}

\newcommand\tiPsh[1]{\OperatorinfinitPsh({#1})}
\newcommand\Psh[1]{\OperatorPsh({#1})}

\newcommand\tPsh[1]{\OperatortPsh({#1})}

\newcommand\mPsh[1]{\OperatormPsh({#1})}

\DeclareMathOperator{\Sset}{\Psh{\Delta}}

\DeclareMathOperator{\U}{\mathbf{U}}
\DeclareMathOperator{\V}{\mathbf{V}}
\DeclareMathOperator{\Wcard}{\mathbf{W}}
\DeclareMathOperator{\Z}{\mathbf{Z}}

\newcommand{\ringpartial}{\mathring{\partial}}

\usepackage[colorlinks=true,linktocpage=true,linkcolor=black,citecolor=black]{hyperref}
\usepackage[capitalise]{cleveref}

\usepackage{fancyhdr}
\usepackage{titlesec}
\usepackage{textcase}

\title{\Huge{Categorical Theory of $(\infty,\omega)$-Categories}}
\author{Félix Loubaton\\ Max Planck Institute for Mathematics}
\date{}
\fancyhfoffset[RO,LE]{0.5cm}
\fancyhfoffset[LE,RO]{0.5cm}

\begin{document}

\maketitle
\dominitoc

\cleardoublepage
\phantomsection
\addcontentsline{toc}{chapter}{Contents} 
\tableofcontents

\cleardoublepage
\phantomsection
\addcontentsline{toc}{chapter}{Introduction} 
\chapter*{Introduction}
%
%
%
%
%
%
	
A \textit{category} consists of a set of objects and, for every pair of objects $a,b$, a set of morphisms $\hom_C(a,b)$ equipped with composition operations that satisfy associativity and identity laws.

\textit{$\iun$-Categories} are a homotopical generalization of categories. Intuitively, they are defined similarly to categories, except that the sets of objects and morphisms are replaced by spaces of objects and morphisms, and the associativity and identity laws are no longer satisfied strictly but up to homotopy.

Thanks notably to the work of Joyal (\cite{Joyal_Quasi-categories_and_Kan_complexes}) and Lurie (\cite{Lurie_Htt}), most important concepts and theorems in category theory now have their $\iun$-categorical analogues. These objects have become important tools in various fields of mathematics, including algebraic geometry, algebraic topology, and representation theory.

Another generalization of the notion of a category is obtained by replacing the set of morphisms between two objects $a$ and $b$ with a category of morphisms between $a$ and $b$. These new objects are called \textit{$2$-categories}. By replacing the sets of morphisms  with $(n-1)$-categories of morphisms, one can inductively define the notion of an \textit{$n$-category} for any integer $n$.

The notion of  \textit{$(\infty,n)$-category} is obtained by accomplishing these two generalizations simultaneously. These objects are now found in many fields, particularly in derived algebraic geometry, where the  6-functor formalism is expressed and manipulated using the theory of $(\infty,2)$-categories (\cite{Gaitsgory_A_study_on_DAG}), and in topological quantum fields theory, where $(\infty,n)$-categories are essential to the formulation and proof of the cobordism hypothesis (\cite{Baez_Higher-dimensional_algebra_and_topological_quantum_field_theory}, \cite{Lurie_on_the_classification_of_topological_field_theories}, \cite{Grady_the_geometric_cobordism_hypothesis}, \cite{Calaque_a_note_on_the_category_of_cobordism}).

\vspace{1cm}

The ambition of this work is to contribute to the study of $(\infty,n)$-categories. However, restricting to a given integer $n$ has its limitations.

A first example is given by the fact that $(\infty,n)$-categories organize into an $(\infty,n+1)$-category, and this richer structure plays an important role in the theory of $(\infty,n)$-categories.

A second example comes from the \textit{Gray tensor product}, a fundamental operation for studying \textit{lax} phenomena. Informally, given a concept in category theory (colimit, natural transformation), the lax variant is obtained by replacing all commutative diagrams appearing in the definition of this concept with diagrams that commute up to higher (not necessarily invertible) cells. The Gray tensor product is then the lax variant of the Cartesian product.

\begin{example*}[examples of some Gray tensor products]
We denote by $\Db_1$ the $1$-category generated by the $1$-graph
\[\begin{tikzcd}
	0 & 1
	\arrow[from=1-1, to=1-2]
\end{tikzcd}\]
and  by $\Db_2$ the $2$-category generated by the $2$-graph
\[\begin{tikzcd}
	0 & 1
	\arrow[""{name=0, anchor=center, inner sep=0}, curve={height=-12pt}, from=1-1, to=1-2]
	\arrow[""{name=1, anchor=center, inner sep=0}, curve={height=12pt}, from=1-1, to=1-2]
	\arrow[shorten <=3pt, shorten >=3pt, Rightarrow, from=0, to=1]
\end{tikzcd}\]
The Gray tensor product of $\Db_1$ with itself, denoted by $\Db_1\otimes\Db_1$, is the $2$-category generated by the diagram
\[\begin{tikzcd}
	00 & 01 \\
	10 & 11
	\arrow[from=1-1, to=2-1]
	\arrow[from=2-1, to=2-2]
	\arrow[from=1-1, to=1-2]
	\arrow[from=1-2, to=2-2]
	\arrow[shorten <=4pt, shorten >=4pt, Rightarrow, from=1-2, to=2-1]
\end{tikzcd}\]
The Gray tensor product of $\Db_2$ with $\Db_1$, denoted by $\Db_2\otimes\Db_1$, is the $3$-category generated by the diagram
\[\begin{tikzcd}
	00 & 01 & 00 & 01 \\
	10 & 11 & 10 & 11
	\arrow[from=1-1, to=1-2]
	\arrow[""{name=0, anchor=center, inner sep=0}, from=1-1, to=2-1]
	\arrow[from=2-1, to=2-2]
	\arrow[""{name=1, anchor=center, inner sep=0}, from=1-2, to=2-2]
	\arrow[shorten <=4pt, shorten >=4pt, Rightarrow, from=1-2, to=2-1]
	\arrow[""{name=2, anchor=center, inner sep=0}, from=1-3, to=2-3]
	\arrow[from=1-3, to=1-4]
	\arrow[""{name=3, anchor=center, inner sep=0}, from=1-4, to=2-4]
	\arrow[shorten <=4pt, shorten >=4pt, Rightarrow, from=1-4, to=2-3]
	\arrow[""{name=4, anchor=center, inner sep=0}, curve={height=30pt}, from=1-1, to=2-1]
	\arrow[from=2-3, to=2-4]
	\arrow[""{name=5, anchor=center, inner sep=0}, curve={height=-30pt}, from=1-4, to=2-4]
	\arrow["{ }"', shorten <=6pt, shorten >=6pt, Rightarrow, from=0, to=4]
	\arrow["{ }"', shorten <=6pt, shorten >=6pt, Rightarrow, from=5, to=3]
	\arrow[shift left=0.7, shorten <=6pt, shorten >=8pt, no head, from=1, to=2]
	\arrow[shift right=0.7, shorten <=6pt, shorten >=8pt, no head, from=1, to=2]
	\arrow[shorten <=6pt, shorten >=6pt, from=1, to=2]
\end{tikzcd}\]
\end{example*}
As we can see from the examples, unlike the Cartesian product which takes the maximum dimension of its inputs, the Gray tensor product adds them together and thus corresponds to a functor
\[
\begin{array}{ccc}
\ncat{n}\times \ncat{m}&\to &\ncat{n+m}.\\
(C,D)&\mapsto &C\otimes D
\end{array}
\]
One can handle this by considering a truncated version of the Gray tensor product, but we believe that avoiding such a drastic operation will lead to a more natural understanding of the complex combinatorics it encodes.
To understand higher categories in the most general way possible, we will thus focus directly on \textit{$(\infty,\omega)$-categories}\footnote{This notion is sometimes called $(\infty,\infty)$-categories. See page \pageref{section omeag and infinity} for a justification of the choice of denomination made in this text.}.

\subsubsection*{A brief definition of $(m,n)$-categories for $m\in \Nb\cup\{\infty\}$ and $n\in \Nb\cup\{\omega\}$}

A \textit{globular set} is the data of a diagram of sets
\[\begin{tikzcd}
	{X_0} & {X_1} & {X_2} & {...}
	\arrow["{\pi_0^+}"', shift right=2, from=1-2, to=1-1]
	\arrow["{\pi_1^+}"', shift right=2, from=1-3, to=1-2]
	\arrow["{\pi_3^+}"', shift right=2, from=1-4, to=1-3]
	\arrow["{\pi_0^-}", shift left=2, from=1-2, to=1-1]
	\arrow["{\pi_1^-}", shift left=2, from=1-3, to=1-2]
	\arrow["{\pi_3^-}", shift left=2, from=1-4, to=1-3]
\end{tikzcd}\]
with the relations $\pi_{n-1}^{\epsilon}\pi_{n}^+ = \pi_n^{\epsilon} \pi_{n}^-$ for any $n>0$ and $\epsilon \in \{+,-\}$. We also denote by $\pi^{\epsilon}_k$ the map $X_n \to X_k$ for $k< n$ obtained by composing any string of arrows starting with $\pi^\epsilon_{k}$. An \textit{$\omega$-category} is a globular set $X$ together with
\begin{enumerate}
\item operations of \textit{compositions}
\[ X_n\times_{X_k} X_n\to X_n ~~~(0\leq k<n) \]
which associate to two $n$-cells $(x,y)$ verifying $\pi_k^+(x) = \pi_k^-(y)$, an $n$-cell $x\circ_ky$,
\item as well as \textit{units}
\[X_n\to X_{n+1}\]
which associate to an $n$-cell $x$, an $(n+1)$-cell $\Ib_x$, 
\end{enumerate}
and satisfying some associativity and units axioms which will be expected by any reader familiar with $2$-categories.
A \textit{morphism of $\omega$-categories} is a map of globular sets commuting with both operations. The category of $\omega$-categories is denoted by \textit{$\omegacat$}.

The category $\Theta$ of Joyal is the full subcategory of $\omegacat$ spanned by the \textit{globular sums}. Roughly speaking, globular sums are the $\omega$-categories obtained by "directed" gluing of \textit{globes}. In particular, globes are the easiest example of globular sums. Here are a few examples of globes and globular sums, where we identify the pasting diagrams with the $\omega$-categories they generate.

\begin{example*}[some examples of globes]
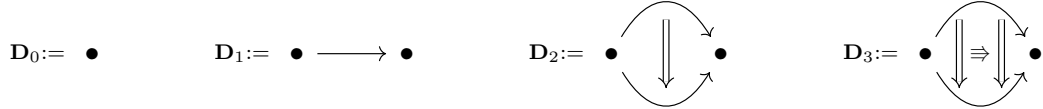
\label{exe:exemple 0}
\[\begin{tikzcd}
	{} & \bullet & {} & \bullet & \bullet & {} & \bullet & \bullet & {} & \bullet & \bullet \\
	\\
	{}
	\arrow[from=1-4, to=1-5]
	\arrow[""{name=0, anchor=center, inner sep=0}, curve={height=-24pt}, from=1-7, to=1-8]
	\arrow[""{name=1, anchor=center, inner sep=0}, curve={height=24pt}, from=1-7, to=1-8]
	\arrow["{\Db_0:=}"{description}, draw=none, from=1-1, to=1-2]
	\arrow["{\Db_1:=}"{description}, draw=none, from=1-3, to=1-4]
	\arrow["{\Db_2:=}"{description}, draw=none, from=1-6, to=1-7]
	\arrow[""{name=2, anchor=center, inner sep=0}, curve={height=-24pt}, from=1-10, to=1-11]
	\arrow[""{name=3, anchor=center, inner sep=0}, curve={height=24pt}, from=1-10, to=1-11]
	\arrow["{\Db_3:=}"{description}, draw=none, from=1-9, to=1-10]
	\arrow[shorten <=6pt, shorten >=6pt, Rightarrow, from=0, to=1]
	\arrow[""{name=4, anchor=center, inner sep=0}, shift left=3, shorten <=6pt, shorten >=6pt, Rightarrow, from=2, to=3]
	\arrow[""{name=5, anchor=center, inner sep=0}, shift right=3, shorten <=6pt, shorten >=6pt, Rightarrow, from=2, to=3]
	\arrow["\Rrightarrow"{description}, draw=none, from=5, to=4]
\end{tikzcd}\]
\end{example*}

\begin{example*}[some examples of globular sums]
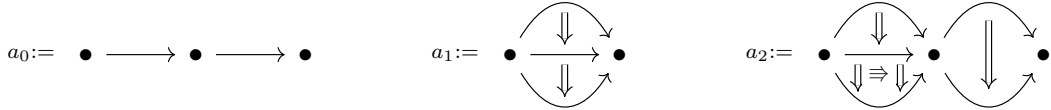
\label{exe:exemple 1}
\[\begin{tikzcd}
	{} & \bullet & \bullet & \bullet & {} & \bullet & \bullet & {} & \bullet & \bullet & \bullet
	\arrow[""{name=0, anchor=center, inner sep=0}, curve={height=-24pt}, from=1-6, to=1-7]
	\arrow[""{name=1, anchor=center, inner sep=0}, curve={height=24pt}, from=1-6, to=1-7]
	\arrow[""{name=2, anchor=center, inner sep=0}, from=1-6, to=1-7]
	\arrow[from=1-2, to=1-3]
	\arrow[from=1-3, to=1-4]
	\arrow[""{name=3, anchor=center, inner sep=0}, from=1-9, to=1-10]
	\arrow[""{name=4, anchor=center, inner sep=0}, curve={height=24pt}, from=1-9, to=1-10]
	\arrow[""{name=5, anchor=center, inner sep=0}, curve={height=-24pt}, from=1-9, to=1-10]
	\arrow[""{name=6, anchor=center, inner sep=0}, curve={height=-24pt}, from=1-10, to=1-11]
	\arrow[""{name=7, anchor=center, inner sep=0}, curve={height=24pt}, from=1-10, to=1-11]
	\arrow["{a_0:=}"{description}, draw=none, from=1-1, to=1-2]
	\arrow["{a_1:=}"{description}, draw=none, from=1-5, to=1-6]
	\arrow["{a_2:=}"{description}, draw=none, from=1-8, to=1-9]
	\arrow[shorten <=3pt, shorten >=3pt, Rightarrow, from=0, to=2]
	\arrow[shorten <=3pt, shorten >=3pt, Rightarrow, from=2, to=1]
	\arrow[""{name=8, anchor=center, inner sep=0}, shift left=3, shorten <=3pt, shorten >=5pt, Rightarrow, from=3, to=4]
	\arrow[""{name=9, anchor=center, inner sep=0}, shift right=3, shorten <=3pt, shorten >=5pt, Rightarrow, from=3, to=4]
	\arrow[shorten <=3pt, shorten >=3pt, Rightarrow, from=5, to=3]
	\arrow[shorten <=6pt, shorten >=6pt, Rightarrow, from=6, to=7]
	\arrow["\Rrightarrow"{description}, shift left=1, draw=none, from=9, to=8]
\end{tikzcd}\]
\end{example*}
\begin{example*}[some examples of morphisms between globular sums]
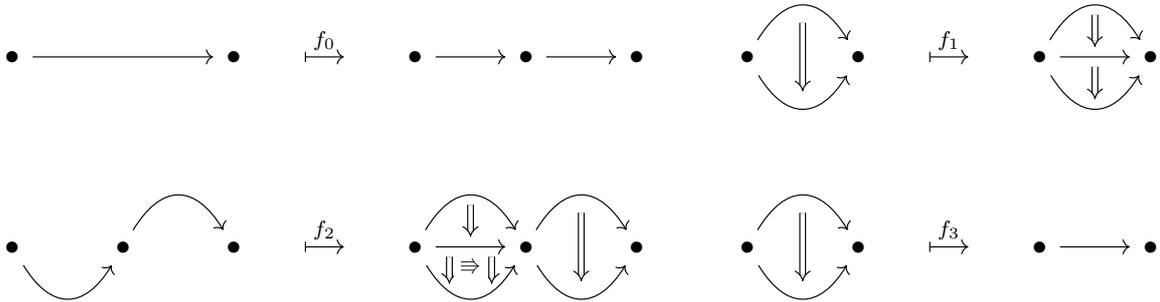
\[\begin{tikzcd}[column sep=0.367in]
	\bullet && \bullet && \bullet & \bullet & \bullet & \bullet & \bullet && \bullet & \bullet \\
	\\
	\\
	\bullet & \bullet & \bullet && \bullet & \bullet & \bullet & \bullet & \bullet && \bullet & \bullet
	\arrow[from=4-11, to=4-12]
	\arrow[""{name=0, anchor=center, inner sep=0}, curve={height=-24pt}, from=4-6, to=4-7]
	\arrow[""{name=1, anchor=center, inner sep=0}, curve={height=24pt}, from=4-6, to=4-7]
	\arrow[""{name=2, anchor=center, inner sep=0}, curve={height=24pt}, from=4-8, to=4-9]
	\arrow[""{name=3, anchor=center, inner sep=0}, curve={height=-24pt}, from=4-8, to=4-9]
	\arrow["{f_3}", shorten <=19pt, shorten >=19pt, maps to, from=4-9, to=4-11]
	\arrow[""{name=4, anchor=center, inner sep=0}, from=4-5, to=4-6]
	\arrow["{f_2}", shorten <=19pt, shorten >=19pt, maps to, from=4-3, to=4-5]
	\arrow[""{name=5, anchor=center, inner sep=0}, curve={height=-24pt}, from=1-8, to=1-9]
	\arrow[""{name=6, anchor=center, inner sep=0}, curve={height=24pt}, from=1-8, to=1-9]
	\arrow[""{name=7, anchor=center, inner sep=0}, curve={height=-24pt}, from=1-11, to=1-12]
	\arrow[""{name=8, anchor=center, inner sep=0}, curve={height=24pt}, from=1-11, to=1-12]
	\arrow[""{name=9, anchor=center, inner sep=0}, from=1-11, to=1-12]
	\arrow["{f_1}", shorten <=19pt, shorten >=19pt, maps to, from=1-9, to=1-11]
	\arrow[curve={height=-24pt}, from=4-2, to=4-3]
	\arrow[curve={height=24pt}, from=4-1, to=4-2]
	\arrow[from=1-1, to=1-3]
	\arrow["{f_0}", shorten <=19pt, shorten >=19pt, maps to, from=1-3, to=1-5]
	\arrow[from=1-5, to=1-6]
	\arrow[from=1-6, to=1-7]
	\arrow[""{name=10, anchor=center, inner sep=0}, curve={height=-24pt}, from=4-5, to=4-6]
	\arrow[""{name=11, anchor=center, inner sep=0}, curve={height=24pt}, from=4-5, to=4-6]
	\arrow[shorten <=6pt, shorten >=6pt, Rightarrow, from=0, to=1]
	\arrow[shorten <=6pt, shorten >=6pt, Rightarrow, from=3, to=2]
	\arrow[shorten <=3pt, shorten >=3pt, Rightarrow, from=7, to=9]
	\arrow[shorten <=3pt, shorten >=3pt, Rightarrow, from=9, to=8]
	\arrow[shorten <=6pt, shorten >=6pt, Rightarrow, from=5, to=6]
	\arrow[shorten <=3pt, shorten >=3pt, Rightarrow, from=10, to=4]
	\arrow[""{name=12, anchor=center, inner sep=0}, shift left=3, shorten <=3pt, shorten >=5pt, Rightarrow, from=4, to=11]
	\arrow[""{name=13, anchor=center, inner sep=0}, shift right=3, shorten <=3pt, shorten >=5pt, Rightarrow, from=4, to=11]
	\arrow["\Rrightarrow"{description}, shift left=1, shorten <=2pt, shorten >=2pt, from=13, to=12]
\end{tikzcd}\]
\end{example*}

 For $n\in \Nb\cup \{\omega\}$, we define $\Theta_n$ as the full subcategory of $\Theta$ whose objects correspond to $n$-categories. In particular, $\Theta_0$ is the terminal category, $\Theta_1$ is $\Delta$, and $\Theta_\omega$ is $\Theta$.

Let $\ngrd{m}$ be the $\iun$-category of $m$-groupoids, i.e. $m$-truncated $\infty$-groupoids, and $n\in \Nb\cup \{\omega\}$. A \textit{$(m,n)$-category} is a functor $\Theta_n^{op}\to \ngrd{m}$ that satisfies the \textit{Segal conditions} and \textit{completeness conditions}. We denote by $(m,n)\mbox{-$\cat$}$ the $\iun$-category of $(m,n)$-categories. Since we have not given a precise definition of $\Theta$, we cannot explicitly state these conditions, but we will try to explain their essence.

\textbf{Segal conditions.} As the diagrams given in the examples suggest, every globular sum is a colimit of globes. For instance, $a_2$ is the colimit of the following diagram
\[\begin{tikzcd}
	{\Db_2} \\
	{\Db_1} & {\Db_0} & {\Db_2} \\
	{\Db_3}
	\arrow["{i_1^+}", from=2-1, to=1-1]
	\arrow["{i_1^-}"', from=2-1, to=3-1]
	\arrow["{i_0^+}"', from=2-2, to=2-1]
	\arrow["{i_0^-}", from=2-2, to=2-3]
\end{tikzcd}\]
A functor $X:\Theta_n^{op}\to \ngrd{m}$ satisfies the \textit{Segal conditions} if it sends these colimits to limits. For instance, the presheaf $X$ must send $a_2$ to the limit of the diagram 
\[\begin{tikzcd}
	{X(\Db_2)} \\
	{X(\Db_1)} & {X(\Db_0)} & {X(\Db_2)} \\
	{X(\Db_3)}
	\arrow["{\pi_1^+}"', from=1-1, to=2-1]
	\arrow["{\pi_1^-}", from=3-1, to=2-1]
	\arrow["{\pi_0^+}", from=2-1, to=2-2]
	\arrow["{\pi_0^-}"', from=2-3, to=2-2]
\end{tikzcd}\]
The morphism $X(f_1)$ then sends a pair of $2$-cells sharing a common $1$-boundary to a $2$-cell, and then corresponds to a $1$-composition. Similarly, the morphism $X(f_0)$ is a $0$-composition, and the morphism $X(f_3)$ is a unit.

\textbf{Completeness conditions.} Let $X:\Theta_n^{op}\to \ngrd{m}$ be a functor satisfying the Segal conditions. Given an integer $k\leq n$, we have two notions of equivalence on  $k$-cells of $X$, i.e.  morphisms $1\to X(\Db_k)$. 

The first one, denoted by $\sim$, is the equivalence relation coming from the $m$-groupoid structure of $\Hom(1, X(\Db_k))$.

The second is more categorical and identifies \textit{isomorphic} elements, i.e. $k$-cells $a,b$ such that there exists $(k+1)$-cells $f:a\to b$, $g:b\to a$ and equivalences
$$g\circ_k f\sim id_a~~~~~~~~~\mbox{and}~~~~~~~~~f\circ_k g\sim id_b.$$
 The presheaf $X$ satisfies the completeness condition if these two notions of equivalence coincide. 

\vspace{1cm}
This notation is compatible with the one given in \cite{Rezk_a_cartesian_of_weak_n_categories} when $k\geq n$, but it also allows to give meaning to $(k,n)$-categories for $k<n$.
 For instance, $(0,\omega)$-categories correspond to $\Theta$-sets satisfying the Segal and completeness conditions. The first condition induce an inclusion of $(0,\omega)$-categories into $\omega$-categories and the latter forces isomorphisms to be identities.

\subsubsection*{$\io$-Categories and $(\infty,\infty)$-Categories}
\label{section omeag and infinity}
As stated earlier, this work is devoted to the concept of $\io$-categories.
These objects are sometimes called $(\infty,\infty)$-categories. However, we have chosen the notation "$\io$" for several reasons.

The first reason is that it emphasizes that these two infinities represent different concepts: the first concerns the infinity of the homotopical dimension, and the second the infinity of the categorical dimension.

The second reason is that we aim to generalize higher strict categories, which are usually called $\omega$-categories. It thus seemed appropriate that the symbol $\omega$ should appear in our notion.

The third reason is that some authors seem to want to remove the prefix "$\infty$" from homotopical concepts. In this paradigm, $(\infty,\infty)$-categories would be called $\infty$-categories, which could be confusing.

Finally, the notion of $(\infty,\infty)$-categories is sometimes considered ambiguous.
Indeed, Schommer-Pries and Rezk have independently argued (\cite{134099}) that there should be more than one notion of $(\infty,\infty)$-categories. The one we use here is commonly referred to as \textit{the inductive one}, in the sense that $\ocat$ is identified with the limit of the sequence:
$$\ncat{0}\xleftarrow{\tau_0} \ncat{1}\xleftarrow{}... \leftarrow\ncat{n} \xleftarrow{\tau_{n}}\ncat{n+1}\xleftarrow{}...$$
where the functors $\tau_n$ "forget" the cells of dimension $n$. For a more detailed discussion in the (semi-)strict case, we refer to \cite{Henry_an_inductive_model_structure_for_infini_categories}.

\subsubsection*{Language of $\iun$-categories}
In this text, we will freely use the language of $\iun$-categories. As there are currently several directions for the formalization of the language of $\iun$-categories (\cite{Riehl_element_of_infini_categories}, \cite{Riehl_A_type_theory_for_synthetic_-categories}, \cite{North_Towards_a_directed_homotopy_type_theory}, \cite{Cisinski-Univalent-Directed-Type-theory}), talking about "the" language of $\iun$-categories may be confusing.

In such case, the reader may consider that we are working within the quasi-category $\qcat$ of $\Tb$-small quasi-categories for $\Tb$ a Grothendieck universe. This quasi-category may be obtained either using the coherent nerve as described in \cite[chapter 3]{Lurie_Htt}, or by considering it as the codomain of the universal cocartesian fibration with $\Tb$-small fibers as done in \cite{Cisinski_The_universal_coCartesian_fibration}. In both cases, the straightening/unstraightening correspondence provides a morphism
$$\N(\Sset_{\Tb})\to \qcat$$
that exhibits $\qcat$ as the quasi-categorical localization of $\N(\Sset_{\Tb})$ with respect to the weak equivalences of the Joyal's model structure (\cite[theorem 8.13]{Cisinski_The_universal_coCartesian_fibration}). 

The constructions we use to build new objects - (co)limits of functor between quasi-categories, quasi-categories of functor, localization of quasi-categories, sub maximal Kan complex, full sub quasi-category, adjunction, left and right Kan extension, Yoneda lemma - are well documented in the Joyal model structure (see \cite{Lurie_Htt} or \cite{Cisinski_Higher_categories_and_homotopical_algebra}), and therefore have direct incarnation in the quasi-category $\qcat$.


\addcontentsline{toc}{subsection}{Summary of results} 
\subsection*{Summary of results}
\paragraph{Chapter 1.}

The first section is devoted to the definition of $\zo$-categories and of the category $\Theta$ of Joyal. We also show that the category $\Theta$ presents the category of $\zo$-categories, and we also exhibit an other presentation of this category (corollary \ref{cor:changing theta}).

The second section begins with a review of Steiner theory, which is an extremely useful tool for providing concise and computational descriptions of $\zo$-categories. Following Ara and Maltsiniotis, we employ this theory to define the Gray tensor product, denoted by $\otimes$, in $\zo$-categories. We then introduce the Gray operations, starting with the Gray cylinder $\uvar\otimes[1]$ which is the Gray tensor product with the directed interval $[1]:=0\to 1$. Then, we have the \textit{Gray cone} and the \textit{Gray $\circ$-cone}, denoted by $\uvar\star 1$ and $1\costar \uvar$, that send an $\zo$-category $C$ onto the following pushouts:
\[\begin{tikzcd}[ampersand replacement=\&]
	{C\otimes\{1\}} \& {C\otimes[1]} \& {C\otimes\{0\}} \& {C\otimes[1]} \\
	1 \& {C\star 1} \& 1 \& {1\costar C}
	\arrow[from=1-1, to=1-2]
	\arrow[from=1-1, to=2-1]
	\arrow[from=1-2, to=2-2]
	\arrow[from=1-3, to=1-4]
	\arrow[from=1-3, to=2-3]
	\arrow[from=1-4, to=2-4]
	\arrow[from=2-1, to=2-2]
	\arrow["\lrcorner"{anchor=center, pos=0.125, rotate=180}, draw=none, from=2-2, to=1-1]
	\arrow[from=2-3, to=2-4]
	\arrow["\lrcorner"{anchor=center, pos=0.125, rotate=180}, draw=none, from=2-4, to=1-3]
\end{tikzcd}\]

We also present a formula that illustrates the interaction between the suspension and the Gray cylinder. As this formula plays a crucial role in this text, we provide its intuition at this stage.

 If $A$ is any $\zo$-category, the \textit{suspension} of $A$, denoted by $[A,1]$, is the $\zo$-category having two objects - denoted by $0$ and $1$- and such that 
$$\Hom_{[A,1]}(0,1) := A,~~~\Hom_{[A,1]}(1,0) := \emptyset,~~~\Hom_{[A,1]}(0,0)=\Hom_{[A,1]}(1,1):=\{id\}.$$
We also define $[1]\vee[A,1]$ as the gluing of $[1]$ and $[A,1]$ along the $0$-target of $[1]$ and the $0$-source of $[A,1]$. We define similarly $[A,1]\vee[1]$.
These two objects come along with \textit{whiskerings}:
$$\triangledown:[A,1]\to [1]\vee [A,1] ~~~~\mbox{and}~~~~ \triangledown:[A,1] \to [A,1]\vee [1]$$ 
that preserve the extremal points.

The $\zo$-category $[1]\otimes [1]$ is induced by the diagram:
\[\begin{tikzcd}
	00 & 01 \\
	10 & 11
	\arrow[from=1-1, to=2-1]
	\arrow[from=2-1, to=2-2]
	\arrow[from=1-1, to=1-2]
	\arrow[from=1-2, to=2-2]
	\arrow[shorten <=4pt, shorten >=4pt, Rightarrow, from=1-2, to=2-1]
\end{tikzcd}\]
and is then equal to the colimit of the following diagram: 
$$[1]\vee [1]\xleftarrow{\triangledown} [1]\hookrightarrow [[1],1]\hookleftarrow[1]\xrightarrow{\triangledown } [1]\vee [1].$$
The $\zo$-category $ [[1],1]\otimes [1]$ is induced by the diagram:
\[\begin{tikzcd}
	00 & 01 & 00 & 01 \\
	10 & 11 & 10 & 11
	\arrow[from=1-1, to=1-2]
	\arrow[""{name=0, anchor=center, inner sep=0}, from=1-1, to=2-1]
	\arrow[from=2-1, to=2-2]
	\arrow[""{name=1, anchor=center, inner sep=0}, from=1-2, to=2-2]
	\arrow[shorten <=4pt, shorten >=4pt, Rightarrow, from=1-2, to=2-1]
	\arrow[""{name=2, anchor=center, inner sep=0}, from=1-3, to=2-3]
	\arrow[from=1-3, to=1-4]
	\arrow[""{name=3, anchor=center, inner sep=0}, from=1-4, to=2-4]
	\arrow[shorten <=4pt, shorten >=4pt, Rightarrow, from=1-4, to=2-3]
	\arrow[""{name=4, anchor=center, inner sep=0}, curve={height=30pt}, from=1-1, to=2-1]
	\arrow[from=2-3, to=2-4]
	\arrow[""{name=5, anchor=center, inner sep=0}, curve={height=-30pt}, from=1-4, to=2-4]
	\arrow["{ }"', shorten <=6pt, shorten >=6pt, Rightarrow, from=0, to=4]
	\arrow["{ }"', shorten <=6pt, shorten >=6pt, Rightarrow, from=5, to=3]
	\arrow[shift left=0.7, shorten <=6pt, shorten >=8pt, no head, from=1, to=2]
	\arrow[shift right=0.7, shorten <=6pt, shorten >=8pt, no head, from=1, to=2]
	\arrow[shorten <=6pt, shorten >=6pt, from=1, to=2]
\end{tikzcd}\]
and is then equal to the colimit of the following diagram: 
 $$[1]\vee[[1],1]\xleftarrow{\triangledown} [[1]\otimes\{0\},1]\hookrightarrow[[1]\otimes[1],1]\hookleftarrow [[1]\otimes\{1\},1]\xrightarrow{\triangledown}[[1],1]\vee[1]$$
We prove a formula that combines these two examples:

\begin{iprop}[\ref{prop:appendice formula for otimes}]
In the category of $\zo$-categories, there exists an isomorphism, natural in $A$, between $[A,1]\otimes[1]$ and the colimit of the following diagram
\[\begin{tikzcd}
	{[1]\vee[A,1]} & {[A\otimes\{0\},1]} & { [A\otimes[1],1]} & {[A\otimes\{1\},1]} & {[A,1]\vee[1]}
	\arrow["\triangledown"', from=1-2, to=1-1]
	\arrow[from=1-4, to=1-3]
	\arrow["\triangledown", from=1-4, to=1-5]
	\arrow[from=1-2, to=1-3]
\end{tikzcd}\]
\end{iprop} 

We also provide similar formulas for the {Gray cone}, the {Gray $\circ$-cone} and the Gray op-cone.
\begin{iprop}[\ref{prop:appendice formula for star}]
There is a natural identification between $1\costar [A,1]$ and the colimit of the following diagram
\[\begin{tikzcd}
	{[1]\vee[A,1]} & {[A,1]} & { [A\star 1,1]}
	\arrow["\triangledown"', from=1-2, to=1-1]
	\arrow[from=1-2, to=1-3]
\end{tikzcd}\]
There is a natural identification between $[A,1]\star 1$ and the colimit of the following diagram
\[\begin{tikzcd}
	{ [1\costar A,1]} & {[A,1]} & {[A,1]\vee[1]}
	\arrow[from=1-2, to=1-1]
	\arrow["\triangledown", from=1-2, to=1-3]
\end{tikzcd}\]
\end{iprop}

We conclude this chapter by studying the functor
\[ \uvar\otimes \uvar : \Psh{\Delta} \times \Psh{\Delta} \to \Psh{\Theta} \]
which is the left Kan extension of the functor
\[ \Delta \times \Delta \xrightarrow{\uvar\otimes \uvar} \zocat \xrightarrow{\iota} \Psh{\Theta} \]
where \(\iota\) is the inclusion of \(\zo\)-categories into presheaves on \(\Theta\). In particular, we demonstrate theorem \ref{theo:otimes presserves W}, which will be of crucial importance for defining the Gray tensor product for \(\io\)-categories.

\paragraph{Chapter 2.}

The first section is a recollection of results on presentable $\iun$-categories. In particular, we provide a very useful technical lemma that gives conditions for calculating colimits in $\infty$-presheaves using (strict) presheaves. We also present results on factorization systems and the localizations they induce. Finally, we conclude by giving some results on monomorphisms in $\iun$-categories.

In the second section, we define $\io$-categories and give some basic properties. 
We also define and study \textit{discrete Conduché functor}, which are morphisms having the unique right lifting property against 
units $\Ib_{n+1}:\Db_{n+1}\to \Db_n$ for any integer $n$, and against compositions $\triangledown_{k,n}:\Db_n\to \Db_n\coprod_{\Db_k}\Db_n$ for any pair of integers $k\leq n$. This notion was originally defined and studied in the context of strict $\omega$-category by Guetta in \cite{Guetta_conduche}. We then demonstrate the following result:
\begin{itheorem}[\ref{theo:pullback along conduche preserves colimits}]
Let $f:C\to D$ be a discrete Conduché functor. The pullback functor $f^*:\ocat_{/D}\to \ocat_{/C}$ preserves colimits.
\end{itheorem}

In the third section, using the Gray tensor product for $\zo$-categories, we construct a colimit-preserving functor
\[ \uvar\otimes\uvar:\ocat\times \icat\to \ocat \]
again called the Gray tensor product.
To be able to define lax phenomena, in particular lax colimits and limits, we will need to extend the functor $\uvar\otimes\uvar$ to \(\ocat\times \ocat\). Although one might at first be tempted to use the "general" Gray tensor product, this is not the operation that will be used subsequently, particularly for stating the universal property of lax colimits and limits. We then introduce a cocontinuous bifunctor:
\[ \ominus: \ocat\times \ocat\to \ocat \]
called the \textit{enhanced Gray tensor product}. This new operation can be seen as a generalization of the Gray product with $\iun$-categories.
 
 We then introduce the Gray operations, starting with the Gray cylinder $\uvar\otimes[1]$ which is the Gray tensor product with the directed interval $[1]:=0\to 1$. Then, we have the \textit{Gray cone} and the \textit{Gray $\circ$-cone}, denoted by $\uvar\star 1$ and $1\costar \uvar$, that send an $\io$-category $C$ onto the following pushouts:
\[\begin{tikzcd}[cramped]
	{C\otimes\{1\}} & {C\otimes[1]} & {C\otimes\{0\}} & {C\otimes[1]} \\
	1 & {C\star 1} & 1 & {1\costar C}
	\arrow[from=1-1, to=1-2]
	\arrow[from=1-1, to=2-1]
	\arrow[from=1-2, to=2-2]
	\arrow[from=1-3, to=1-4]
	\arrow[from=1-3, to=2-3]
	\arrow[from=1-4, to=2-4]
	\arrow[from=2-1, to=2-2]
	\arrow["\lrcorner"{anchor=center, pos=0.125, rotate=180}, draw=none, from=2-2, to=1-1]
	\arrow[from=2-3, to=2-4]
	\arrow["\lrcorner"{anchor=center, pos=0.125, rotate=180}, draw=none, from=2-4, to=1-3]
\end{tikzcd}\]

We then provide several formulas expressing the relationship between the Gray operations and the suspension $[\uvar,1]:\ocat\to \ocat$, identical to those in the strict case.
\begin{iprop}[\ref{prop:eq for cylinder}]
There is a natural identification between $[C,1]\otimes [1]$ and the colimit of the diagram
\begin{equation}
\begin{tikzcd}
	{[1]\vee [ C,1]} & {[C\otimes\{0\},1]} & {[C\otimes [1],1]} & {[C\otimes\{1\},1]} & {[C,1]\vee[1]}
	\arrow[from=1-2, to=1-1]
	\arrow[from=1-2, to=1-3]
	\arrow[from=1-4, to=1-3]
	\arrow[from=1-4, to=1-5]
\end{tikzcd}
\end{equation}
\end{iprop}

\begin{iprop}[\ref{prop:formula for the ominus}]
There is a  natural identification between $[C,1]\ominus[b,1]$ and the colimit of the following diagram
\begin{equation}
\begin{tikzcd}[column sep = 0.3cm]
	{[b,1]\vee[C,1]} & {[C\otimes\{0\}\times b,1]} & {[(C\otimes[1])\times b,1]} & {[C\otimes\{1\}\times b,1]} & {[C,1]\vee[b,1]}
	\arrow[from=1-2, to=1-3]
	\arrow[from=1-4, to=1-3]
	\arrow[from=1-4, to=1-5]
	\arrow[from=1-2, to=1-1]
\end{tikzcd}
\end{equation}
\end{iprop}

\begin{iprop}[\ref{prop:eq for Gray cone}]
There is a natural identification between $1\costar [C,1]$ and the colimit of the diagram
\begin{equation}
\begin{tikzcd}
	{[1]\vee [C,1]} & {[C,1]} & {[C\star 1,1]}
	\arrow[from=1-2, to=1-3]
	\arrow[from=1-2, to=1-1]
\end{tikzcd}
\end{equation}
There is a natural identification between $[C,1]\star 1$ and the colimit of the diagram
\begin{equation}
\begin{tikzcd}
	{[1\costar C,1]} & {[C,1]} & {[C,1]\vee[1]}
	\arrow[from=1-2, to=1-3]
	\arrow[from=1-2, to=1-1]
\end{tikzcd}
\end{equation}
\end{iprop}

 We conclude this chapter by proving results of strictification. In particular, we demonstrate the following theorem:
\begin{itheorem}[\ref{theo:strict stuff are pushout}]
Let $C\to D$ and $C\to E$ be two morphisms between strict $\io$-categories. The $\io$-categories
$D\coprod_{C}C\star 1$,  $1\costar C\coprod_CD$, and $D\coprod_{C\otimes\{0\}}C\otimes[1]\coprod_{C\otimes \{1\}}E$ are strict. In particular, $C\star 1$, $1\costar C$, and $C\otimes[n]$ for any integer $n$, are strict.
\end{itheorem}
In the process, we will demonstrate another fundamental equation combining $C\otimes[1]$, $1\costar C$, $C\star 1$, and $[C,1]$.
\begin{itheorem}[\ref{theo:formula between pullback of slice and tensor}]
Let $C$ be an $\io$-category. The five squares appearing in the following canonical diagram are both cartesian and cocartesian:
\[\begin{tikzcd}
	& {C\otimes\{0\}} & 1 \\
	{C\otimes\{1\}} & {C\otimes[1]} & {C\star 1} \\
	1 & {1\costar C} & {[C,1]}
	\arrow[from=2-3, to=3-3]
	\arrow[from=3-2, to=3-3]
	\arrow[from=2-2, to=3-2]
	\arrow[from=2-2, to=2-3]
	\arrow[from=1-2, to=1-3]
	\arrow[from=1-3, to=2-3]
	\arrow[from=1-2, to=2-2]
	\arrow[from=2-1, to=2-2]
	\arrow[from=3-1, to=3-2]
	\arrow[from=2-1, to=3-1]
\end{tikzcd}\]
where $[C,1]$ is the \textit{suspension of $C$}.
\end{itheorem}

\paragraph{Chapter 3.}

This chapter is dedicated to the study of \textit{marked $\io$-categories}, which are pairs $(C,tC)$, where $C$ is an $\io$-category and $tC:=(tC_n)_{n>0}$ is a sequence of full sub $\infty$-groupoids of $C_n$ that include identities and are stable under composition and whiskering with (possibly unmarked) cells of lower dimensions. There are two canonical ways to mark an $\io$-category $C$. In the first, denoted by $C^\flat$, we mark as little as possible. In the second, denoted by $C^\sharp$, we mark everything.

The first section of the chapter defines these objects and establishes analogs of many results on $\io$-categories to this new framework. In particular, the \textit{marked Gray cylinder} $\uvar\otimes [1]^\sharp$ is defined. If $A$ is an $\io$-category, the underlying $\io$-category of $A^\sharp\otimes[1]^\sharp$ is $A\times [1]$, and the underlying $\io$-category of $A^\flat\otimes[1]^\sharp$ is $A\otimes[1]$. By varying the marking, and at the level of underlying $\io$-categories, we "continuously" move from the cartesian product with the directed interval to the Gray tensor product with the directed interval.

The motivation for introducing markings comes from the notion of \textit{left (and right) cartesian fibrations}. A left cartesian fibration is a morphism between marked $\io$-categories such that only the marked cells of the codomain have cartesian lifting, and the marked cells of the domain correspond exactly to such cartesian lifting. For example, a left cartesian fibration $X\to A^\sharp$ is just a "usual" left cartesian fibration where we have marked the cartesian lifts of the domain, and every morphism $C^\flat \to D^\flat$ is a left cartesian fibration.

After defining and enumerating the stability properties enjoyed by this class of left (and right) cartesian fibration, we give several characterizations of this notion in theorem \ref{theo:other characterisation of left caresian fibration}. 

The more general subclass of left cartesian fibrations that still behaves well is the class of \textit{classified left cartesian fibrations}. 
This corresponds to left cartesian fibrations $X\to A$ such that there exists a cartesian square:
\[\begin{tikzcd}
	X & Y \\
	A & {A^\sharp}
	\arrow[from=1-1, to=2-1]
	\arrow[from=2-1, to=2-2]
	\arrow[from=1-1, to=1-2]
	\arrow[from=1-2, to=2-2]
	\arrow["\lrcorner"{anchor=center, pos=0.125}, draw=none, from=1-1, to=2-2]
\end{tikzcd}\]
 where the right vertical morphism is a left cartesian fibration and $A^\sharp$ is obtained from $A$ by marking all cells. In the second section, we prove the following fundamental result:

\begin{itheorem}[\ref{theo:pullback along un marked cartesian fibration}]
Let $p:X\to A$ be a classified left cartesian fibration. Then the functor $p^*:\ocatm_{/A}\to \ocatm_{/X}$ preserves colimits.
\end{itheorem}

The third subsection is devoted to the proof of the following theorem
\begin{itheorem}[\ref{theo:left cart stable by colimit}]
Let $A$ be an $\io$-category and $F:I\to \ocatm_{/A^\sharp}$ be a diagram that is pointwise a left cartesian fibration. The colimit 
$\colim_IF$, computed in $\ocatm_{/A^\sharp}$, is a left cartesian fibration over $A^\sharp$.
\end{itheorem}

In the fourth subsection we study \textit{smooth} and \textit{proper} morphisms and we obtain the following expected result:
\begin{iprop}[\ref{prop:quillent theorem A}]
For a morphism $X\to A^\sharp$, and an object $a$ of $A$, we denote by $X_{/a}$ the marked $\io$-category fitting in the following cartesian squares. 
\[\begin{tikzcd}[cramped]
	{X_{/a}} & X \\
	{A^\sharp_{/a}} & {A^\sharp}
	\arrow[from=1-1, to=1-2]
	\arrow[from=1-1, to=2-1]
	\arrow["\lrcorner"{anchor=center, pos=0.125}, draw=none, from=1-1, to=2-2]
	\arrow[from=1-2, to=2-2]
	\arrow[from=2-1, to=2-2]
\end{tikzcd}\]
We denote by $\bot:\ocatm\to \ocat$ the functor sending a marked $\io$-category to its localization by marked cells.
\begin{enumerate}
\item Let $E$, $F$ be two elements of $\ocatm_{/A^\sharp}$ corresponding to morphisms $X\to A^\sharp$, $Y\to A^\sharp$, and
 $\phi:E\to F$ a morphism between them. We denote by $\Fb E$ and $\Fb F$ the left cartesian fibrant replacement of $E$ and $F$. 
 
The induced morphism $\Fb\phi:\Fb E\to \Fb F$ is an equivalence if and only if for any object $a$ of $A$, the induced morphism 
$$\bot X_{/a}\to \bot Y_{/a}$$ 
is an equivalence of $\io$-categories.
\item A morphism $X\to A^\sharp$ is initial if and only if for any object $a$ of $A$, $\bot X_{/a}$ is the terminal $\io$-category.
\end{enumerate}
\end{iprop}

Finally, in the last subsection, for a marked $\io$-category $I$, we define and study a (huge) $\io$-category $\uLCartc(I)$ that has classified left cartesian fibrations as objects and morphisms between classified left cartesian fibrations as arrows.

\paragraph{Chapter 4.}

This chapter aims to establish generalizations of the fundamental categorical constructions to the $\io$ case. In this new theory, the Gray product plays an essential role. Firstly, it allows the definition of the notion of \textit{lax transformation}:

\begin{idefi}
Let $f, g: A \to B$ be two morphisms between $\io$-categories. A \textit{lax transformation} between $f$ and $g$ is given by a morphism
\[ \psi: A \otimes [1] \to B, \]
whose restriction to $A \otimes \{0\} \sim A$ is equivalent to $f$, and whose restriction to $A \otimes \{1\} \sim A$ is equivalent to $g$.

We can then show that a lax transformation corresponds to the following data:
\begin{enumerate}
\item[$-$] for every object $a$ in $A$, a morphism $f(a) \to g(a)$,
\item[$-$] for every $1$-cell $a \to b$, a $2$-cell in $B$ fitting into the following diagram:
\[
\begin{tikzcd}
	{f(a)} & {g(a)} \\
	{f(b)} & {g(b)}
	\arrow[from=2-1, to=2-2]
	\arrow[from=1-2, to=2-2]
	\arrow[from=1-1, to=2-1]
	\arrow[from=1-1, to=1-2]
	\arrow[shorten <=6pt, shorten >=6pt, Rightarrow, from=1-2, to=2-1]
\end{tikzcd}
\]
\item[$-$] for every $n$-cell in $A$, an $(n+1)$-cell in $B$ fitting into a more complex version of the above diagram,
\item[$-$] multiple coherences that all these cells satisfy.
\end{enumerate} 
The usefulness of the Gray product lies in the fact that it compactly encodes all these data and coherences.

\end{idefi}

The notion of lax transformation allows us to state the Grothendieck Lax construction.
\begin{itheorem}[\ref{theo:lcartc et ghom}]
\label{theo:lax gr}
Let $\uni$ be the $\io$-category of small $\io$-categories, and $A$ an $\io$-category. There is an equivalence
\[ \int_A: \gHom(A, \uni) \sim \uLCart^c(A) \]
where $\gHom(A, \uni)$ is the $\io$-category of morphisms from $A$ to $\uni$, with $1$-cells being the lax transformations $A \otimes [1] \to \uni$, and $\uLCart^c(A)$ is the $\io$-category of left cartesian fibrations over $A$, with $1$-cells being morphisms that do not necessarily preserve cartesian liftings.
\end{itheorem}

We also obtain a very precise construction of the functor $\int_A$. Given a functor $f: A \to \uni$, the left cartesian fibration $\int_A f$ is a colimit (calculated in $\ocatm_{/A}$) of a simplicial object whose value at $n$ is of the form
\[ \coprod_{x_0, \ldots, x_n: A_0} X(x_0) \times \hom_A(x_0, \ldots, x_n) \times A_{x_n/} \to A, \]
where $A_{x/}$ is the \textit{lax slice} of $A$ over $x$.
This formula is similar to the one given in \cite{Gepner_Lax_colimits_and_free_fibration} for $\iun$-categories, and to the one given in \cite{Warren_the_strict_omega_groupoid_interpretation_of_type_theory} for strict $\omega$-categories.

The result we provide is actually stronger than previously stated, as it allows us to choose "to what extent" the transformations between functors $A \to \uni$ are lax, which induces "to what extent" the morphisms between the fibrations preserve cartesian liftings. The result we have presented corresponds to the case where the transformations are "totally lax". Applying it to the case where the transformations between functors are "not lax at all" - that is, are natural transformations - we obtain the following corollary:

\begin{icor}[\ref{cor:lcar et hom}]
\label{cor:grd}
Let $A$ be an $\io$-category. There is an equivalence
\[ \uHom(A, \uni) \sim \uLCart(A) \]
where $\uHom(A, \uni)$ is the $\io$-category of morphisms from $A$ to $\uni$, with $1$-cells being the natural transformations $A \times [1] \to \uni$, and $\uLCart(A)$ is the $\io$-category of left cartesian fibrations over $A$, with $1$-cells being morphisms that preserve cartesian liftings.
\end{icor}

In the $(\infty, n)$-categorical case, the equivalence between  $\uHom(A, \uni)$ and $\uLCart(A)$ given in Corollary \ref{cor:lcar et hom} was already proven by Nuiten in \cite{Nuiten_on_straightening_for_segal_spaces} and by Rasekh in \cite{Rasekh_yoneda_lemma_for_simplicial_spaces}. In the $(\infty, 1)$-case, the equivalence  between  $\gHom(A, \uni)$ and $\uLCartc(A)$ was already proven by Haugseng-Hebestreit-Linskens-Nuiten in \cite{haugseng2023lax}.

\vspace{1cm}

Given a locally small $\io$-category $C$, we construct the Yoneda embedding $y: C \to \widehat{C}$ where $\widehat{C} := \uHom(C^t, \uni)$. We then prove the Yoneda lemma:

\begin{itheorem}[\ref{theo:Yoneda lemma}]
Let $C$ be a locally $\U$-small $\io$-category. 
The Yoneda embedding $C \to \widehat{C}$ is fully faithful. Furthermore,
there is an equivalence between the functor
\[ \hom_{\w{C}}(y_{\uvar}, \uvar): C^t \times \w{C} \to \uni \]
and the functor 
\[ \ev: C^t \times \w{C} \to \uni. \]

Given an object $c$ of $C$, the induced equivalence on fibers:
\[ \hom_{\widehat{C}}(y_c, y_c) \sim \hom_C(c, c) \]
sends $\{id_{y_c}\}$ to $\{id_c\}$.
\end{itheorem}

Having the Yoneda lemma at our disposal with all the correct functorialities is an extremely powerful tool for developing the theory of $\io$-categories, as it encodes many and complex coherences.

In the $(\infty,n)$-categorical context, the  equivalence, non-functorial in $c$, between the functors $\hom(y_c, \uvar): \w{C} \to \uni$ and $\ev(c, \uvar): \w{C} \to \uni$ for any object $c$ of $C$ is demonstrated in \cite{Rasekh_yoneda_lemma_for_simplicial_spaces}, \cite{Hinich_colimit_in_enriched_infini_categories}, and \cite{Heine_an_equivalence_between_enricherd_infini_categorories_and_categories_with_weak_action}.

\vspace{1cm}

We then define the notion of a \textit{lax colimit}:
\begin{idefi}
Let $f: I \to C$ be a morphism between $\io$-categories, and $tI$ a marking on $I$.

A \textit{lax colimit} for $f$ relative to $(tI_n)_{n>0}$ is the universal data of
\begin{enumerate}
\item[$-$] an object $\laxcolim_I F$,
\item[$-$] for every $1$-cell $i: a \to b$ in $I$, a diagram
\[\begin{tikzcd}
	{} & {F(b)} \\
	{F(a)} & {\laxcolim_I F}
	\arrow["{F(i)}", curve={height=-30pt}, from=2-1, to=1-2]
	\arrow[from=2-1, to=2-2]
	\arrow[shorten <=8pt, shorten >=8pt, Rightarrow, from=1-2, to=2-1]
	\arrow[draw=none, from=1-1, to=2-1]
	\arrow[from=1-2, to=2-2]
\end{tikzcd}\]
where the $2$-cell present is an equivalence if $i$ is in $tI_1$.
\item[$-$] for every $2$-cell $u: i \to j$, a diagram
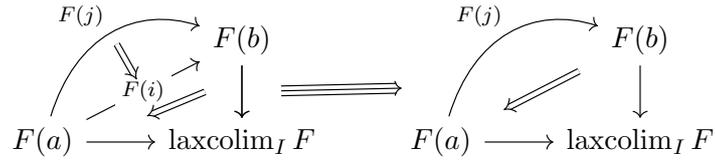
\[\begin{tikzcd}
	& {F(b)} & {} & {F(b)} \\
	{F(a)} & {\laxcolim_I F} & {F(a)} & {\laxcolim_I F}
	\arrow[""{name=0, anchor=center, inner sep=0}, "{F(i)}"{description}, from=2-1, to=1-2]
	\arrow[""{name=1, anchor=center, inner sep=0}, from=2-1, to=2-2]
	\arrow[from=1-2, to=2-2]
	\arrow[""{name=2, anchor=center, inner sep=0}, from=1-2, to=2-2]
	\arrow[""{name=3, anchor=center, inner sep=0}, "{F(j)}", curve={height=-30pt}, from=2-1, to=1-2]
	\arrow["{F(j)}", curve={height=-30pt}, from=2-3, to=1-4]
	\arrow[from=2-3, to=2-4]
	\arrow[from=1-4, to=2-4]
	\arrow[shorten <=8pt, shorten >=8pt, Rightarrow, from=1-4, to=2-3]
	\arrow[""{name=4, anchor=center, inner sep=0}, draw=none, from=1-3, to=2-3]
	\arrow[shift right=2, shorten <=12pt, shorten >=12pt, Rightarrow, from=2, to=1]
	\arrow[shorten <=4pt, shorten >=4pt, Rightarrow, from=3, to=0]
	\arrow[shift left=0.7, shorten <=14pt, shorten >=16pt, no head, from=2, to=4]
	\arrow[shorten <=14pt, shorten >=14pt, from=2, to=4]
	\arrow[shift right=0.7, shorten <=14pt, shorten >=16pt, no head, from=2, to=4]
\end{tikzcd}\]
where the $3$-cell present is an equivalence if $u$ is in $tI_2$,
\item[$-$] etc.
\end{enumerate}
Varying $tI$ thus allows us to adjust the "laxness" of the universal property that this object must satisfy.
\end{idefi}

We conclude this chapter by establishing generalizations of the standard results in category theory to the $\io$-categorical case, including the characterization of presheaves as completion by lax colimits, the calculation of lax Kan extensions using slices, the relationships between duality and (co)lax limits, as well as various characterizations and properties of adjunctions.

\adjustmtc
\adjustmtc
\chapter{ $\zo$-Categories and presheaves on $\Theta$}
\label{chapter:The category of zocategories}
\minitoc
\vspace{1cm}
%
%
%
%
%
%
%
%
%

The first section is devoted to the definition of $\zo$-categories and of the category $\Theta$ of Joyal. We also show that the category $\Theta$ presents the category of $\zo$-categories, and we also exhibit an other presentation of this category (corollary \ref{cor:changing theta}).

The second section begins with a review of Steiner theory, which is an extremely useful tool for providing concise and computational descriptions of $\zo$-categories. Following Ara and Maltsiniotis, we employ this theory to define the Gray tensor product, denoted by $\otimes$, in $\zo$-categories. We then introduce the Gray operations, starting with the Gray cylinder $\uvar\otimes[1]$ which is the Gray tensor product with the directed interval $[1]:=0\to 1$. Then, we have the \textit{Gray cone} and the \textit{Gray $\circ$-cone}, denoted by $\uvar\star 1$ and $1\costar \uvar$, that send an $\zo$-category $C$ onto the following pushouts:
\[\begin{tikzcd}[ampersand replacement=\&]
	{C\otimes\{1\}} \& {C\otimes[1]} \& {C\otimes\{0\}} \& {C\otimes[1]} \\
	1 \& {C\star 1} \& 1 \& {1\costar C}
	\arrow[from=1-1, to=1-2]
	\arrow[from=1-1, to=2-1]
	\arrow[from=1-2, to=2-2]
	\arrow[from=1-3, to=1-4]
	\arrow[from=1-3, to=2-3]
	\arrow[from=1-4, to=2-4]
	\arrow[from=2-1, to=2-2]
	\arrow["\lrcorner"{anchor=center, pos=0.125, rotate=180}, draw=none, from=2-2, to=1-1]
	\arrow[from=2-3, to=2-4]
	\arrow["\lrcorner"{anchor=center, pos=0.125, rotate=180}, draw=none, from=2-4, to=1-3]
\end{tikzcd}\]

We also present a formula that illustrates the interaction between the suspension and the Gray cylinder. As this formula plays a crucial role in this text, we provide its intuition at this stage.

 If $A$ is any $\zo$-category, the \textit{suspension} of $A$, denoted by $[A,1]$, is the $\zo$-category having two objects - denoted by $0$ and $1$- and such that 
$$\Hom_{[A,1]}(0,1) := A,~~~\Hom_{[A,1]}(1,0) := \emptyset,~~~\Hom_{[A,1]}(0,0)=\Hom_{[A,1]}(1,1):=\{id\}.$$
We also define $[1]\vee[A,1]$ as the gluing of $[1]$ and $[A,1]$ along the $0$-target of $[1]$ and the $0$-source of $[A,1]$. We define similarly $[A,1]\vee[1]$.
These two objects come along with \textit{whiskerings}:
$$\triangledown:[A,1]\to [1]\vee [A,1] ~~~~\mbox{and}~~~~ \triangledown:[A,1] \to [A,1]\vee [1]$$ 
that preserve the extremal points.

The $\zo$-category $[1]\otimes [1]$ is induced by the diagram:
\[\begin{tikzcd}
	00 & 01 \\
	10 & 11
	\arrow[from=1-1, to=2-1]
	\arrow[from=2-1, to=2-2]
	\arrow[from=1-1, to=1-2]
	\arrow[from=1-2, to=2-2]
	\arrow[shorten <=4pt, shorten >=4pt, Rightarrow, from=1-2, to=2-1]
\end{tikzcd}\]
and is then equal to the colimit of the following diagram: 
$$[1]\vee [1]\xleftarrow{\triangledown} [1]\hookrightarrow [[1],1]\hookleftarrow[1]\xrightarrow{\triangledown } [1]\vee [1].$$
The $\zo$-category $ [[1],1]\otimes [1]$ is induced by the diagram:
\[\begin{tikzcd}
	00 & 01 & 00 & 01 \\
	10 & 11 & 10 & 11
	\arrow[from=1-1, to=1-2]
	\arrow[""{name=0, anchor=center, inner sep=0}, from=1-1, to=2-1]
	\arrow[from=2-1, to=2-2]
	\arrow[""{name=1, anchor=center, inner sep=0}, from=1-2, to=2-2]
	\arrow[shorten <=4pt, shorten >=4pt, Rightarrow, from=1-2, to=2-1]
	\arrow[""{name=2, anchor=center, inner sep=0}, from=1-3, to=2-3]
	\arrow[from=1-3, to=1-4]
	\arrow[""{name=3, anchor=center, inner sep=0}, from=1-4, to=2-4]
	\arrow[shorten <=4pt, shorten >=4pt, Rightarrow, from=1-4, to=2-3]
	\arrow[""{name=4, anchor=center, inner sep=0}, curve={height=30pt}, from=1-1, to=2-1]
	\arrow[from=2-3, to=2-4]
	\arrow[""{name=5, anchor=center, inner sep=0}, curve={height=-30pt}, from=1-4, to=2-4]
	\arrow["{ }"', shorten <=6pt, shorten >=6pt, Rightarrow, from=0, to=4]
	\arrow["{ }"', shorten <=6pt, shorten >=6pt, Rightarrow, from=5, to=3]
	\arrow[shift left=0.7, shorten <=6pt, shorten >=8pt, no head, from=1, to=2]
	\arrow[shift right=0.7, shorten <=6pt, shorten >=8pt, no head, from=1, to=2]
	\arrow[shorten <=6pt, shorten >=6pt, from=1, to=2]
\end{tikzcd}\]
and is then equal to the colimit of the following diagram: 
 $$[1]\vee[[1],1]\xleftarrow{\triangledown} [[1]\otimes\{0\},1]\hookrightarrow[[1]\otimes[1],1]\hookleftarrow [[1]\otimes\{1\},1]\xrightarrow{\triangledown}[[1],1]\vee[1]$$
We prove a formula that combines these two examples:

\begin{iprop}[\ref{prop:appendice formula for otimes}]
In the category of $\zo$-categories, there exists an isomorphism, natural in $A$, between $[A,1]\otimes[1]$ and the colimit of the following diagram
\[\begin{tikzcd}
	{[1]\vee[A,1]} & {[A\otimes\{0\},1]} & { [A\otimes[1],1]} & {[A\otimes\{1\},1]} & {[A,1]\vee[1]}
	\arrow["\triangledown"', from=1-2, to=1-1]
	\arrow[from=1-4, to=1-3]
	\arrow["\triangledown", from=1-4, to=1-5]
	\arrow[from=1-2, to=1-3]
\end{tikzcd}\]
\end{iprop} 

We also provide similar formulas for the {Gray cone}, the {Gray $\circ$-cone} and the Gray op-cone.
\begin{iprop}[\ref{prop:appendice formula for star}]
There is a natural identification between $1\costar [A,1]$ and the colimit of the following diagram
\[\begin{tikzcd}
	{[1]\vee[A,1]} & {[A,1]} & { [A\star 1,1]}
	\arrow["\triangledown"', from=1-2, to=1-1]
	\arrow[from=1-2, to=1-3]
\end{tikzcd}\]
There is a natural identification between $[A,1]\star 1$ and the colimit of the following diagram
\[\begin{tikzcd}
	{ [1\costar A,1]} & {[A,1]} & {[A,1]\vee[1]}
	\arrow[from=1-2, to=1-1]
	\arrow["\triangledown", from=1-2, to=1-3]
\end{tikzcd}\]
\end{iprop}

We conclude this chapter by studying the functor
\[ \uvar\otimes \uvar : \Psh{\Delta} \times \Psh{\Delta} \to \Psh{\Theta} \]
which is the left Kan extension of the functor
\[ \Delta \times \Delta \xrightarrow{\uvar\otimes \uvar} \zocat \xrightarrow{\iota} \Psh{\Theta} \]
where \(\iota\) is the inclusion of \(\zo\)-categories into presheaves on \(\Theta\). In particular, we demonstrate theorem \ref{theo:otimes presserves W}, which will be of crucial importance for defining the Gray tensor product for \(\io\)-categories.

\section{Basic constructions}
\label{chapter:Basica construciton preliminaire}
\subsection{$\zo$-Categories}
\label{section:zocategories}

\begin{definition}
 A \notion{globular set} is a presheaf on the \textit{category of globes} $\Gb$, which is the category induces by the diagram
\[\begin{tikzcd}
	{\Db_0} & {\Db_1} & {\Db_2} & {...}
	\arrow["{i_0^+}", shift left=2, from=1-1, to=1-2]
	\arrow["{i_1^+}", shift left=2, from=1-2, to=1-3]
	\arrow["{i_3^+}", shift left=2, from=1-3, to=1-4]
	\arrow["{i_0^-}"', shift right=2, from=1-1, to=1-2]
	\arrow["{i_1^-}"', shift right=2, from=1-2, to=1-3]
	\arrow["{i_3^-}"', shift right=2, from=1-3, to=1-4]
\end{tikzcd}\]
with the relations $i_n^{+} i_{n-1}^\epsilon = i_n^{-} i_{n-1}^\epsilon $ for any $n>0$ and $\epsilon \in \{+,-\}$. 
For any $n>k$ and $\epsilon \in \{+,-\}$, we also denote by $i^{\epsilon}_k$ the composite $\Db_{k} \xrightarrow{i^{\epsilon}_k} \Db_{k+1}\xrightarrow{f} \Db_n$ where $f$ is any map. These and the identity arrows are the only maps in the category $\Gb$.

If $X$ is a globular set, we denotes by $X_n$ the set $X(\Db_n)$. Its elements are called \wcsnotion{$n$-cells}{cell@$n$-cell}{for $\zo$-categories}. The $0$-cells are  called \textit{objects}. The maps $X_n \to X_k$ induced by $i^\epsilon_k : \Db_k \to \Db_n$ is denoted by $\pi^\epsilon_k$.
\end{definition}

\begin{definition}
\label{defi:def of omega cat}
An \wcnotion{$\omega$-category}{category@$\omega$-category} is a globular set $X$ together with
\begin{enumerate}
\item operations of \textit{compositions}
\[ X_n\times_{X_k} X_n\to X_n ~~~(0\leq k<n) \]
which associate to two $n$-cells $(x,y)$ verifying $\pi_k^-(x) = \pi_k^+(y)$, a $n$-cells $x\circ_ky$,
\item as well as \textit{units}
\[X_n\to X_{n+1}\]
which associate to an $n$-cell $x$, a $(n+1)$-cell $\Ib_x$, 
\end{enumerate}
and satisfying the following axioms:
\begin{enumerate}

\item $\forall x \in X_n, \pi^\epsilon_n(\Ib_x) = x $.

\item $\pi^+_k (x \circ_n y) = \pi_k^{+}(x)$ and $\pi^-_k(x \circ_n y) = \pi_k^-(y)$ whenever the composition is defined and $k \leqslant n$.

\item $\pi^\epsilon_k (x \circ_n y) = \pi_k^{\epsilon}(x) \circ_n \pi^\epsilon_k(y)$ whenever the composition is defined and $k > n$.

\item $ x \circ_n \Ib_{\pi^-_n x} = x$ and $ \Ib_{\pi^+_n x} \circ_n x = x$.

\item $(x \circ_n y) \circ_n z = x \circ_n (y \circ_n z) $ as soon as one of these is defined.

\item If $k <n$
\[ (x \circ_n y) \circ_k ( z \circ_n w) = (x \circ_k z) \circ_n (y \circ_k w) \]
when the left-hand side is defined.

\end{enumerate}
A $n$-cell $a$ is \textit{non trivial}\index[notion]{non trivial $n$-cell} if is not in the image of the application $\Ib:X_{n-1}\to X_n$.

A \notion{morphism of $\omega$-categories} is a map of globular sets commuting with compositions and units. The category of $\omega$-categories is denoted by \textit{$\omegacat$}.
\end{definition}

\begin{definition}
\index[notion]{globe@$n$-globe}
By abuse of notation, we also denote by \wcnotation{$\Db_n$}{(da@$\Db_n$} the $\omega$-category obtained from the globular set represented by $\Db_n$ by adding units. This $\omega$-category
admits for any $k<n$ only two $k$-non-trivial cells, denoted by $e_k^-$ and $e_k^+$, and a single $n$-non-trivial cell, denoted by $e_n$, verifying:
\[
\begin{array}{rcl}
\pi_l^{-}(e_k^\epsilon)= e_l^{-}&\pi_l^{+}(e_k^\epsilon)= e_l^{+}& \mbox{ for $l\leq k<n$}\\
\pi_l^{-}(e_n)= e_l^{-}&\pi_l^{+}(e_n)= e_l^{+}& \mbox{ for $l\leq n$}\\
\end{array}
\]
\end{definition}
Remark furthermore that the $\omega$-category $\Db_n$ represents $n$-cells, in the sense that $\Hom(\Db_n,C)\cong C_n$. We will not make the difference between $n$-cells and the corresponding morphism $\Db_n\to C$. 

\begin{definition}
The $\omega$-category $\partial\Db_n$ is obtained from $\Db_n$ by removing the $n$-cell $e_n$. We thus have a morphism
\[i_n: \partial\Db_n\to \Db_n.\]
Note that $\partial \Db_0 = \emptyset$.

\end{definition}
\begin{definition}
We say that an $\zo$-category $X$ is a \notion{polygraph} if it can be constructed from the empty $\zo$-category by freely adding arrows with specified source and target. That is if $X$ can be obtained as a transfinite composition $\emptyset = X_0 \to X_1 \to \dots \to X_i \to \dots\to \colim X_i = X$ where for each $i$, the map $X_i \to X_{i+1}$ is a pushout of $\coprod_S \partial \Db_n \to \coprod_S \Db_{n+1}$.

 An arrow of a polygraph is said to be a \emph{generator} if it is one of the arrows that has been freely added at some stage.
 \end{definition}

Each cell in a polygraph can be written as a composite of generators or iterated unit of generators (not necessarily in a unique way). For a $n$-cell $f$, the set of generators of dimension $n$ that appear in such an expression (and even the number of times they appear) is the same for all such expressions. As a consequence, a composition of non trivial cells is always non trivial.

\begin{definition}
 \label{defi:dualities strict case}
 For any subset $S$ of $\Nb^*$, we define the functor $(\uvar)^S:\omegacat\to \omegacat$  sending a $\omega$-category $C$ to the category $C^S$ such that for any $n$, there is an isomorphism $C_n\to C_{n}^S$ that sends every $n$-cell $f$ to a cell $\overline{f}$ fulfilling
$$\pi_{n-1}^-(\overline{f})=\overline{\pi^+_{n-1}(f)}~~~~\pi_{n-1}^+(\overline{f})=\overline{\pi^-_{n-1}(f)}$$
if $i\in S$ and 
$$\pi_{n-1}^-(\overline{f})=\overline{\pi^-_{n-1}(f)}~~~~\pi_{n-1}^+(\overline{f})=\overline{\pi^+_{n-1}(f)}$$
if $i\notin S$.
These functors are called \snotion{dualities}{for $\zo$-categories} as they are inverse of themselves. Even if there are plenty of them, we will be interested in only a few of them. In particular, we have the \snotionsym{odd duality}{((b60@$(\uvar)^{op}$}{for $\zo$-categories} $(\uvar)^{op}$, corresponding to the set of odd integer, the \snotionsym{even duality}{((b50@$(\uvar)^{co}$}{for $\zo$-categories} $(\uvar)^{co}$, corresponding to the subset of non negative even integer, the \snotionsym{full duality}{((b80@$(\uvar)^{\circ}$}{for $\zo$-categories} $(\uvar)^{\circ}$, corresponding to $\Nb^*$ and the \snotionsym{transposition}{((b70@$(\uvar)^t$}{for $\zo$-categories} $(\uvar)^t$, corresponding to the singleton $\{1\}$. Eventually, we have equivalences
$$((\uvar)^{co})^{op}\sim (\uvar)^{\circ} \sim ((\uvar)^{op})^{co}.$$
\end{definition}

\begin{definition}
\label{defi:suspension zocat}
 Let $\Psh{\Gb}_{\bullet,\bullet}$ be the category of globular set with two distinguished points, i.e. of triples $(X,a,b)$ where $a$ and $b$ are elements of $X_0$.
Let $[\uvar,1]:\Gb\to \Psh{\Gb}_{\bullet,\bullet}$ be the functor sending $\Db_n$ on $(\Db_{n+1},\{0\},\{1\})$ and $i_n^{\epsilon}$ on $i_{n+1}^{\epsilon}$. This induces by left Kan extension a functor $[\uvar,1]:\Psh{\Gb}\to \Psh{\Gb}_{\bullet,\bullet}$ that we call the \textit{suspension}. We leave it to the reader to check that whenever $C$ has a structure of $\omega$-category, $[C,1]$ inherits one from it. This functor then induces a functor 
$$[\uvar,1]:\omegacat\to \omegacat$$
that we calls again the \snotionsym{suspension}{((d60@$[\uvar,1]$}{for $\zo$-categories}. Eventually, we denote by $i_0^-:\{0\}\to [C,1]$ (resp. $i_0^+:\{1\}\to [C,1]$) the morphism corresponding to the left point (resp. to the right point). For an integer $n$, we define by induction the functor $\Sigma^n:\Psh{\Gb}\to \Psh{\Gb}$\ssym{(sigma@$\Sigma^n$}{for $\zo$-categories} with the formula:
$$\Sigma^0:= id ~~~~~\Sigma^{n+1}:=\Sigma^n[\uvar,1].$$
\end{definition}

\begin{definition}
 Let $n$ be a non null integer.
A $n$-cells $f:s\to t$ is an \notion{equivalence} if there exists $n$-cells $g:t\to s$ and $g':t\to s$ such that 
$$f\circ_{n-1} g=\Ib_t~~~~~~g\circ_{n-1} f=\Ib_s$$
\end{definition}
\begin{definition}
A \wcnotion{$\zo$-category}{category1@$\zo$-category} is an $\omega$-category whose only equivalences are the identities.
These objects are called \textit{Gaunt $\omega$-categories} in \cite{Barwick_on_the_unicity_of_the_theory_of_higher_categories} and \textit{rigid $\omega$-categories} in \cite{Rezk_a_cartesian_of_weak_n_categories}. Remark that $\zo$-categories are stable under suspensions and dualities.

We denote by \wcnotation{$\zocat$}{((a20@$\zocat$}  the full subcategory of $\omegacat$ whose objects are the $\zo$-categories. 
\end{definition}

\begin{definition}
Let $n$ be an integer. An \wcnotion{$(0,n)$-category}{category2@$(0,n)$-category} is an $\zo$-category whose cell of dimension strictly higher than $n$ are units. The category of $(0,n)$-categories is denoted by \wcnotation{$\zncat{n}$}{((a10@$\zncat{n}$} and is the full subcategory of $\zocat$ whose objects are $(0,n)$-categories.
\end{definition}

\begin{construction}
\label{cons:dumb truncatoin}
 Remark that the category $\zncat{n}$ is the localization of $\zocat$ along morphisms $\Db_{k}\to \Db_{n}$ for $k\geq n$. We then have for any $n$ an adjunction 
\[\begin{tikzcd}
	{i_n:\zncat{n}} & {\zocat:\tau_n}
	\arrow[""{name=0, anchor=center, inner sep=0}, shift left=2, from=1-2, to=1-1]
	\arrow[""{name=1, anchor=center, inner sep=0}, shift left=2, from=1-1, to=1-2]
	\arrow["\dashv"{anchor=center, rotate=-90}, draw=none, from=1, to=0]
\end{tikzcd}\]
The right adjoint is called the \wcsnotionsym{$n$-truncation}{(tau@$\tau_n$}{truncation@$n$-truncation}{for $\zo$-categories}.
\end{construction}

\begin{construction}
\label{cons:inteligent truncatoin}
For any $n$, we define the colimit preserving functor $\tau^i_n:\zocat\to \zncat{n}$, called the \snotionsym{intelligent $n$-truncation}{(taui@$\tau^i_n$}{for $\zo$-categories}, sending $\Db_k$ on $\Db_{\min(n,k)}$. The functor $\tau^i_n$ fits in an adjunction
\[\begin{tikzcd}
	{\tau^i_n:\zocat} & {\zncat{n}:i_n}
	\arrow[""{name=0, anchor=center, inner sep=0}, shift left=2, from=1-1, to=1-2]
	\arrow[""{name=1, anchor=center, inner sep=0}, shift left=2, from=1-2, to=1-1]
	\arrow["\dashv"{anchor=center, rotate=-90}, draw=none, from=0, to=1]
\end{tikzcd}\]
\end{construction}
\begin{notation}
We will identify objects of $\zncat{n}$ with their image in $\zocat$ and we will then also note by $\tau_n$ and $\tau^i_n$ the composites $i_n\tau^i_n$ and $i_n\tau^i_n$.
\end{notation}

\begin{remark}
The family of truncation functor induces a sequence 
$$...\to \zncat{n+1}\xrightarrow{\tau_{n}} \zncat{n}\to...\to \zncat{1}\xrightarrow{\tau_{0}}\zncat{0}.$$
The canonical morphism
$$\zocat\to \lim_{n:\Nb}\zncat{n},$$
that sends an $\zo$-category $C$ to the sequence $(\tau_n C,\tau_n C\cong \tau_n \tau_{n+1}C)$, has an inverse given by the functor
$$\colim_{\Nb}:\lim_{n:\Nb}\zncat{n}\to \zocat$$
that sends a sequence $(C_n,  C_n\cong \tau_{n}C_{n+1}) $ to the colimit of the induced sequence:
$$i_0C_0\to i_1C_1\to...\to i_nC_n\to...$$	
 We then have an equivalence 
$$\zocat\cong \lim_{n:\Nb}\zncat{n}.$$
\end{remark}

\subsection{The category $\Theta$}
\label{subsection:the categoru theta}

\begin{definition}
 Let $n$ be a non negative integer and $\textbf{a}:=\{a_0,a_1,...,a_{n-1}\}$ a sequence of $\zo$-categories. We denote \wcnotation{$[\textbf{a},n]$}{((g00@$[\textbf{a},n]$} the colimit of the following diagram
\[\begin{tikzcd}
	& 1 && 1 && 1 \\
	{[a_0,1]} && {[a_1,1]} && {...} && {[a_{n-1},1]}
	\arrow["{i_0^+}"', from=1-2, to=2-1]
	\arrow["{i_0^-}", from=1-2, to=2-3]
	\arrow["{i_0^+}"', from=1-4, to=2-3]
	\arrow["{i_0^-}", from=1-4, to=2-5]
	\arrow["{i_0^+}"', from=1-6, to=2-5]
	\arrow["{i_0^-}", from=1-6, to=2-7]
\end{tikzcd}\]
where $[\uvar,1]$ is the suspension functor defined in \ref{defi:suspension zocat}.
\end{definition}

\begin{definition}
 \label{defi:les sommes glob}
We define \wcnotation{$\Theta$}{(theta@$\Theta$} as the smallest full subcategory of $\zocat$ that includes the terminal $\zo$-category $[0]$, and such that
for any non negative integer $n$, and any finite sequence $\textbf{a}:=\{a_0,a_1,...,a_{n-1}\}$ of objects of $\Theta$, it includes the $\zo$-category $[\textbf{a},n]$.
Objects of $\Theta$ are called \notion{globular sum}.
\end{definition}

\begin{definition}
\label{defi:thetaplus}
We denote by \wcnotation{$\Theta^+$}{(thetaplus@$\Theta^+$} the full subcategory of $\zocat$ whose objects are either globular sums or the empty $\zo$-category $\emptyset$.

We then have a canonical inclusion $\Theta\to \Theta^+$.
\end{definition}
\begin{remark}
\label{remark:morphism of theta}
A morphism $g:[\textbf{a},n]\to [\textbf{b},m]$ is exactly the data of a morphism $f:[n]\to [m]$, and for any integer $i$, a morphism
$$a_i\to \prod_{f(i)\leq k< f(i+1)}b_k.$$
\end{remark}

\begin{example}
\label{exemple:of globular sum}
For any $n$, $\Db_n$ is a globular sum. The $\zo$-category induced by the $\omega$-graph 
\[\begin{tikzcd}
	\bullet & \bullet & \bullet
	\arrow[from=1-1, to=1-2]
	\arrow[""{name=0, anchor=center, inner sep=0}, curve={height=-24pt}, from=1-2, to=1-3]
	\arrow[""{name=1, anchor=center, inner sep=0}, curve={height=24pt}, from=1-2, to=1-3]
	\arrow[""{name=2, anchor=center, inner sep=0}, from=1-2, to=1-3]
	\arrow[shorten <=3pt, shorten >=3pt, Rightarrow, from=0, to=2]
	\arrow[shorten <=3pt, shorten >=3pt, Rightarrow, from=2, to=1]
\end{tikzcd}\]
is a globular sum.
\end{example}

\begin{definition}
 For a globular sum $a$ and an integer $n$, we define $[a,n]:=[\{a,a,...,a\},n]$.\sym{((g10@$[a,n]$}
For a sequence of integer $\{n_0,..,n_k\}$ and a sequence of globular sum $\{a_0,..,a_k\}$, we define \wcnotation{$[a_0,n_0]\vee[a_1,n_1]\vee...\vee [a_k,n_k]$}{((g20@$[a_0,n_0]\vee[a_1,n_1]\vee...\vee [a_k,n_k]$} as the globular sum $[\{a_0,..,a_1,...,a_k,...\},n_0+n_1+...+n_k]$.

We denote by  $[n]$ the globular sum $[[0],n]$.
This induces a fully faithful functor $\Delta\to \Theta$ sending $[n]$ onto $[n]$.. 
\end{definition}

\begin{definition}
 We define by induction the \wcnotion{dimension}{dimension of a globular sum} of a globular sum $a$, denoted by $|a|$. The dimension of $[0]$ is $0$, and the dimension of $[\textbf{a},n]$ is the maximum of the set $\{|a_k|+1\}_{k< n}$. We denote by \wcnotation{$\Theta_n$}{(thetan@$\Theta_n$} the full subcategory of $\Theta$ whose objects are the globular sum of dimension inferior or equal to $n$.
 We set by convention $\Theta_\omega:=\Theta$.
\end{definition}
\begin{notation}
We set by convention $\omega +1:=\omega$.
\end{notation}

\vspace{1cm}

 An important property of the category $\Theta$ is that it is a \textit{Reedy elegant.}
\begin{definition}
 \label{defi:reedy}
 A \notion{Reedy category} is a small category $A$ equipped with two subcategories $A_+$, $A_-$ and a \textit{degree} function $d:ob(A)\to \Nb$ such that: 
\begin{enumerate}
\item for every non identity morphism $f:a\to b$, if $f$ belongs to $A_-$, $d(a)>d(b)$, and if $f$ belongs to $A_+$, $d(a)<d(b)$.
\item every morphism of $A$ uniquely factors as a morphism of $A_-$ followed by a morphism of $A_+$.
\end{enumerate}

A Reedy category $A$ is \wcnotion{elegant}{elegant Reedy category} if for any presheaf $X$ on $A$, for any $a\in A$ and any $c\in X(a)$, there exists a unique morphism $f:a\to a'\in A_{-}$ and a unique non degenerate object $c'\in X(a')$ such that $c=X(f)(c')$. 
\end{definition}

\begin{prop}
\label{prop:elelangat stable by slice}
Let $X$ be a presheaf on an elegant Reedy category $A$. The category $A_{/X}$ is an elegant Reedy category.
\end{prop}
\begin{proof}
We have a canonical projection $\pi:A_{/X}\to A$. A morphism is positive (resp. negative) if it's image by $\pi$ is. The degree of an element $c$ of $A_{/X}$ is the degree of $\pi(c)$. We leave it to the reader to check that this endows $A_{/X}$ with a structure of Reedy category. 

The fact that $A_{/X}$ is elegant is a direct consequence of the isomorphism $\Psh{A_{/X}}\cong \Psh{A}_{/X}$.
\end{proof}

\begin{prop}[Berger, Bergner-Rezk]
\label{prop:theta is elegan reedy}
For any $n\in \Nb\cup\{\omega\}$, the category $\Theta_n$ is an elegant Reedy category.

A morphism $g:[\textbf{a},n]\to [\textbf{b},m]$ is \wcnotion{degenerate}{degenerate morphism of $\Theta$} (i.e a morphism of $\Theta_{-}$) if the corresponding morphism $f:[n]\to [m]$ is a degenerate morphism of $\Delta$, and for any $i<n$ and any $f(i)\leq k<f(k+1)$, the corresponding morphism $a_i\to b_k$ is degenerate. Furthermore, a morphism is degenerate  if and only if it is a epimorphism in $\Psh{\Theta}$.

A morphism is in $\Theta^+$ if and only if it is a monomorphism in $\Psh{\Theta}$. 
\end{prop}
\begin{proof}
The Reedy structure is a consequence of lemma 2.4 of \cite{Berger_a_cellular_nerve}. The fact that for any $n<\omega$, $\Theta_n$ is elegant is 
\cite[corollary 4.5.]{Bergner_reedy_category_and_the_theta_construction}. As for any $n<\omega$, the inclusion $\Theta_n\to \Theta$ preserves strong pushout, the characterization of elegant Reedy category given by \cite[proposition 3.8.]{Bergner_reedy_category_and_the_theta_construction} implies that $\Theta$ is also elegant.
\end{proof}

\begin{cor}
\label{cor:thetaplus is reedy}
The category $\Theta^+$ is an elegant Reedy category.
A morphism \(g: a \to b\) is in \(\Theta^+_-\) if it is in \(\Theta_-\).
A morphism \(i: a \to b\) is in \(\Theta^+_+\) if it is in \(\Theta_+\) or if \(a\) is \(\emptyset\).
\end{cor}
\begin{proof}
This directly follows from the fact that \(\Theta\) is an elegant Reedy category and that \(\Theta^+_- = \Theta_-\).
\end{proof}

\begin{definition}
\label{defi:algebraic and globular}
We recall that a morphism $g:[\textbf{a},n]\to [\textbf{b},m]$ is exactly the data of a morphism $f:[n]\to [m]$, and for any integer $i$, a morphism
$$a_i\to \prod_{f(i)\leq k< f(i+1)}b_k.$$
The morphism $g$ is  \wcsnotion{globular}{globular morphism}{for $\zo$-categories} if for any $k<n$, $f(k+1)=f(k)+1$ and the morphism $a_k\to b_k$ is globular. The morphism $g$ is \wcnotion{algebraic}{algebraic morphism of $\Theta$} if it cannot be written as a composite $ig'$ where $i$ is a globular morphism.
\end{definition}

\begin{example}
\label{example:algebraic morphism}
The morphism 
\[\begin{tikzcd}
	\bullet & \bullet & \bullet && \bullet & \bullet & \bullet
	\arrow[from=1-5, to=1-6]
	\arrow[""{name=0, anchor=center, inner sep=0}, curve={height=-24pt}, from=1-6, to=1-7]
	\arrow[""{name=1, anchor=center, inner sep=0}, curve={height=24pt}, from=1-6, to=1-7]
	\arrow[""{name=2, anchor=center, inner sep=0}, from=1-6, to=1-7]
	\arrow[""{name=3, anchor=center, inner sep=0}, from=1-2, to=1-3]
	\arrow[""{name=4, anchor=center, inner sep=0}, curve={height=-24pt}, from=1-2, to=1-3]
	\arrow[from=1-1, to=1-2]
	\arrow[shorten <=14pt, shorten >=14pt, maps to, from=1-3, to=1-5]
	\arrow[shorten <=3pt, shorten >=3pt, Rightarrow, from=0, to=2]
	\arrow[shorten <=3pt, shorten >=3pt, Rightarrow, from=2, to=1]
	\arrow[shorten <=3pt, shorten >=3pt, Rightarrow, from=4, to=3]
\end{tikzcd}\]
is globular. This is not the case for the morphism
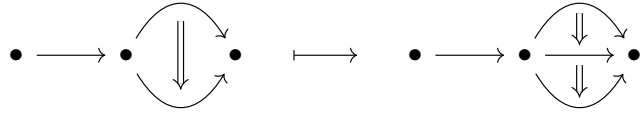
\[\begin{tikzcd}
	\bullet & \bullet & \bullet && \bullet & \bullet & \bullet
	\arrow[from=1-5, to=1-6]
	\arrow[""{name=0, anchor=center, inner sep=0}, curve={height=-24pt}, from=1-6, to=1-7]
	\arrow[""{name=1, anchor=center, inner sep=0}, curve={height=24pt}, from=1-6, to=1-7]
	\arrow[""{name=2, anchor=center, inner sep=0}, from=1-6, to=1-7]
	\arrow[""{name=3, anchor=center, inner sep=0}, curve={height=24pt}, from=1-2, to=1-3]
	\arrow[""{name=4, anchor=center, inner sep=0}, curve={height=-24pt}, from=1-2, to=1-3]
	\arrow[from=1-1, to=1-2]
	\arrow[shorten <=14pt, shorten >=14pt, maps to, from=1-3, to=1-5]
	\arrow[shorten <=3pt, shorten >=3pt, Rightarrow, from=0, to=2]
	\arrow[shorten <=3pt, shorten >=3pt, Rightarrow, from=2, to=1]
	\arrow[shorten <=6pt, shorten >=6pt, Rightarrow, from=4, to=3]
\end{tikzcd}\]
that sends the $2$-cell of the left globular sum on the $1$-composite of the two $2$-cells of the right globular sum.
\end{example}

\begin{prop}[{\cite[Proposition 3.3.10]{Ara_thesis}}]
\label{prop:algebraic ortho to globular}
Every morphism in $\Theta$ can be factored uniquely in an algebraic morphism followed by a globular morphism.
\end{prop}

\begin{remark}

Globular morphims belong to $\Theta_+$ (and so morphisms of $\Theta_-$ are algebraic) but the converse is false. For example, the second morphism of example \ref{example:algebraic morphism} is not globular but belongs to $\Theta_+$.
We then have two different factorizations  on $\Theta$: the one coming from the Reedy elegant structure (in a degenerate morphism followed by a monomorphism), and the one given in proposition \ref{prop:algebraic ortho to globular} (in an algebraic morphism followed by a globular morphism).
\end{remark}

\begin{definition}
\label{defi:definition of source et but}
We define for any globular sum $a$ and any integer $n$ a globular sum $s_n(a):=:t_n(a)$ and two morphisms
$$s_n(a)\to a\leftarrow t_n(a).$$
We first set $s_0(a):=:t_0(a) :=[0]$. The inclusion $s_0(a)\to a$ corresponds to the initial point and $t_0(a)\to a$ to the terminal point.
 For $n>0$, we define $s_n([\textbf{a},k]):=:t_n([\textbf{a},k]) :=[s_{n-1}(\textbf{a}),k]$ where $s_{n-1}(\textbf{a})$ is the sequence 
 $\{s_{n-1}(a_i)\}_{i<n}$.
 \end{definition}

\begin{example}
If $a$ is the globular sum of example \ref{exemple:of globular sum}, we have:
\[\begin{tikzcd}
	{s_1(a):=} & \bullet & \bullet & \bullet \\
	\\
	{a:=} & \bullet & \bullet & \bullet \\
	\\
	{t_1(a):=} & \bullet & \bullet & \bullet
	\arrow[from=3-2, to=3-3]
	\arrow[""{name=0, anchor=center, inner sep=0}, curve={height=-24pt}, from=3-3, to=3-4]
	\arrow[""{name=1, anchor=center, inner sep=0}, curve={height=24pt}, from=3-3, to=3-4]
	\arrow[""{name=2, anchor=center, inner sep=0}, from=3-3, to=3-4]
	\arrow[curve={height=-24pt}, from=1-3, to=1-4]
	\arrow[from=1-2, to=1-3]
	\arrow[from=5-2, to=5-3]
	\arrow[curve={height=24pt}, from=5-3, to=5-4]
	\arrow[shorten <=13pt, shorten >=13pt, maps to, from=1-3, to=3-3]
	\arrow[shorten <=13pt, shorten >=13pt, maps to, from=5-3, to=3-3]
	\arrow[shorten <=3pt, shorten >=3pt, Rightarrow, from=0, to=2]
	\arrow[shorten <=3pt, shorten >=3pt, Rightarrow, from=2, to=1]
\end{tikzcd}\]
\end{example}

\begin{definition}
\label{defi:definition of W}
The suspension functor $[\uvar,1]:\Theta\to \Theta$ induces by left Kan extension a functor 
$$[\uvar,1]:\Psh{\Theta}\to \Psh{\Theta}_{[\emptyset,1]/}$$
where $[\emptyset,1]$ is $\{0\}\coprod \{1\}$.

We define by induction on $a$ a $\Theta$-presheaf \wcnotation{$\Sp_a$}{(sp@$\Sp_{a}$} and a morphism $\Sp_a\to a$. 
If $a$ is $[0]$, we set $\Sp_{[0]}:=[0]$. For $n>0$, we define 
$\Sp_{[\textbf{a},n]}$ as the set valued presheaf on $\Theta$ obtained as the colimit of the diagram
\[\begin{tikzcd}
	& 1 && 1 && 1 \\
	{[ \Sp_{a_0},1]} && {[\Sp_{a_1},1]} && \cdots && {[\Sp_{a_{n-1}},1]}
	\arrow["{i_0^-}", from=1-2, to=2-3]
	\arrow["{i_0^+}"', from=1-4, to=2-3]
	\arrow["{i_0^+}"', from=1-2, to=2-1]
	\arrow["{i_0^-}", from=1-6, to=2-7]
	\arrow["{i_0^-}", from=1-4, to=2-5]
	\arrow["{i_0^+}"', from=1-6, to=2-5]
\end{tikzcd}\]
We define \wcnotation{$E^{eq}$}{(eeq@$E^{eq}$} as the set valued preheaves on $\Delta$ obtained as the colimit of the diagram
\[\begin{tikzcd}
	& {[1]} && {[1]} && {[1]} \\
	{[0]} && {[2]} && {[2]} && {[0]}
	\arrow[from=1-2, to=2-1]
	\arrow["{d^1}", from=1-2, to=2-3]
	\arrow["{d^1}"', from=1-6, to=2-5]
	\arrow["{d^0}"', from=1-4, to=2-3]
	\arrow["{d^2}", from=1-4, to=2-5]
	\arrow[from=1-6, to=2-7]
\end{tikzcd}\]
For any integer $n$, the functor $\Sigma^n:\Theta\to \Theta$, which is the $n$-iteration of $[\uvar,1]$, induces by left Kan extension a functor \ssym{(sigma@$\Sigma^n$}{for $\io$-categories}
$$\Sigma^n:\Psh{\Theta}\to \Psh{\Theta}_{\partial \Db_n/}.$$
 We define two sets of morphisms of $\Psh{\Theta}$:\sym{(w@$\W$}\sym{(wseg@$\Wseg$}\sym{(wsat@$\Wsat$}
$$\Wseg := \{\Sp_a\to a,~a\in\Theta\}~~~~\Wsat:= \{\Sigma^n E^{eq}\to \Db_n\}$$
and we set $$\W:=\Wseg\cup \Wsat.$$
For any $n$, we also define $$\mbox{$\W_n$}:= \W\cap \Theta_n.$$
\end{definition}

\begin{definition}
\label{defi:defi of delta theta}
We recall that for an integer $n$ and a globular sum $a$, we defined $[a,n]:=[\{a,a,...,a\},n]$.
 This defines a functor $i:\Delta[\Theta] \to \Theta$
sending $(n,a)$ on $[a,n]$ where 
 \wcnotation{$\Delta[\Theta]$}{(deltaTheta@$\Delta[\Theta]$} is the following pushout of category: 
\[\begin{tikzcd}
	{\{[0]\}\times \Theta} & \Delta\times\Theta \\
	1 & {\Delta[\Theta]}
	\arrow[from=1-1, to=1-2]
	\arrow[from=1-1, to=2-1]
	\arrow[from=1-2, to=2-2]
	\arrow[from=2-1, to=2-2]
	\arrow["\lrcorner"{anchor=center, pos=0.125, rotate=180}, draw=none, from=2-2, to=1-1]
\end{tikzcd}\]
For the sake of simplicity, we will also denote by $[a,n]$ (resp. $[n]$) the object of $\Delta[\Theta]$ corresponding to $(n,a)$ (resp. to $(n,[0])$).
 We define two sets of morphisms:\sym{(m@$\M$} \sym{(mseg@$\Mseg$}\sym{(msat@$\Msat$}
$$
\begin{array}{c}
\Mseg := \{[a,\Sp_n]\to [a,n],~a:\Theta\}\cup\{[f,1],~f\in \Wseg\}\\
\Msat:= \{E^{eq}\to [0]\} \cup\{[f,1],~f\in \Wsat\}
\end{array}$$
and we set $$\M := \Mseg \cup \Msat.$$

For an integer $n$, we define \wcnotation{$\Delta[\Theta_n]$}{(deltaThetan@$\Delta[\Theta_n]$} as the following pushout of category: 
\[\begin{tikzcd}
	{\{[0]\}\times\Theta_n} & {\Delta\times\Theta_n} \\
	1 & {\Delta[\Theta_n]}
	\arrow[from=1-1, to=1-2]
	\arrow[from=1-1, to=2-1]
	\arrow[from=1-2, to=2-2]
	\arrow[from=2-1, to=2-2]
	\arrow["\lrcorner"{anchor=center, pos=0.125, rotate=180}, draw=none, from=2-2, to=1-1]
\end{tikzcd}\]
and the functor $i$ induces a functor $\Delta[\Theta_n]\to \Theta_{n+1}$.
For any $n$, we define $$\mbox{$\M_n$}:= \M\cap \Delta[\Theta_n].$$ 
\end{definition}

\begin{definition}
Let $C$ be a category and $S$ a set of monomorphisms. A morphism is $f:x\to y$ is  \textit{$S$-local}
if it has the unique right lifting property against morphisms of $S$. An object $x$ is \textit{$S$-local} if $x\to 1$ is {$S$-local}, or equivalently, if for any $i:a\to b\in S$, the induced functor $\Hom(i,x):\Hom(b,x)\to \Hom(a,x)$ is an isomorphism.

We can easily check that $S$-local morphisms are stable by composition, left cancellation and pullback. As a consequence, any morphism between $S$-local objects is $S$-local. 
\end{definition}

\begin{construction}
Let $C$ be a presentable category and $S$ a set of monomorphisms with small codomains.
We define \textit{$C_{S}$} as the full subcategory of $C$ composed of $S$-local objects.
The theorem 4.1 of \cite{BOUSFIELD1977207} implies that  $\iota:C_S\to C$ is part of an adjunction
\[\begin{tikzcd}
	{\Fb_S:C} & {C_S:\iota}
	\arrow[""{name=0, anchor=center, inner sep=0}, shift left=2, from=1-1, to=1-2]
	\arrow[""{name=1, anchor=center, inner sep=0}, shift left=2, from=1-2, to=1-1]
	\arrow["\dashv"{anchor=center, rotate=-90}, draw=none, from=0, to=1]
\end{tikzcd}\]
where $\Fb_S:C\to C_S$ is the localization of $C$ by the smallest class of morphisms containing $S$ and stable under composition and colimit.
\end{construction}

\begin{theorem}[Berger]
\label{theo:theta and ocat}
Let $n\in \Nb\cup\{\omega\}$. 
The functor $\Psh{\Theta_n}\to \zncat{n}$ defined as the left Kan extension of the canonical inclusion $\Theta_n\to \zncat{n}$ induces an isomorphism 
$$\Psh{\Theta_n}_{\W_n}\cong \zncat{n}$$
\end{theorem}
\begin{proof}
This is  \cite[corollary 12.3]{Barwick_on_the_unicity_of_the_theory_of_higher_categories}.
\end{proof}

\begin{construction}
\label{cons:derived adj stric case}
Suppose given an other category $D$ fitting in an adjunction
\[\begin{tikzcd}
	{F:C} & {D:G}
	\arrow[""{name=0, anchor=center, inner sep=0}, shift left=2, from=1-1, to=1-2]
	\arrow[""{name=1, anchor=center, inner sep=0}, shift left=2, from=1-2, to=1-1]
	\arrow["\dashv"{anchor=center, rotate=-90}, draw=none, from=0, to=1]
\end{tikzcd}\]
with unit $\nu$ and counit $\epsilon$,
as well as a set of morphisms $T$ of $D$ such that $F(S)\subset T$. 
By adjunction property, it implies that for any $T$-local object $d\in D$, $G(d)$ is $S$-local.
The previous adjunction induces a derived adjunction
\[\begin{tikzcd}
	{\Lb F:C_S} & {D_T:\Rb G}
	\arrow[""{name=0, anchor=center, inner sep=0}, shift left=2, from=1-1, to=1-2]
	\arrow[""{name=1, anchor=center, inner sep=0}, shift left=2, from=1-2, to=1-1]
	\arrow["\dashv"{anchor=center, rotate=-90}, draw=none, from=0, to=1]
\end{tikzcd}\]
where $\Lb F$ is defined by the formula $c\mapsto \Fb_T F(c)$ and $\Rb G$ is the restriction of $G$ to $D_T$. The unit is given by $\nu\circ \Fb_S$ and the counit by the restriction of $\epsilon$ to $D_T$.
\end{construction}

\begin{construction}
\label{cons:derived adjunction case n strict}
Let $n\in \Nb\cup\{\omega\}$.
 The functor $i:\Delta[\Theta_n]\to \Theta_{n+1}$ defined in definition \ref{defi:defi of delta theta} induces an adjunction:
$$
\begin{tikzcd}
	{ i_!:\Psh{\Delta[\Theta_n]}} & {\Psh{\Theta_{n+1}}:i^*}
	\arrow[shift left=2, from=1-1, to=1-2]
	\arrow[shift left=2, from=1-2, to=1-1]
\end{tikzcd}
$$
where the left adjoint is the left Kan extension of the functor $\Delta[\Theta_n]\to \Theta\to \Psh{\Theta_{n+1}}$.
Remark that there is an obvious inclusion $i_!(\M_{n+1})\subset \W_{n+1}$. In virtue of  construction \ref{cons:derived adj stric case}, this induces an adjunction between derived categories:
\begin{equation}
\label{eq:derived adjunction case n strict}
\begin{tikzcd}
	{\Lb i_!:\Psh{\Delta[\Theta_n]}_{\M_{n+1}}} & {\Psh{\Theta_{n+1}}_{\W{n+1}}:\Rb i^*}
	\arrow[shift left=2, from=1-1, to=1-2]
	\arrow[shift left=2, from=1-2, to=1-1]
\end{tikzcd}
\end{equation}
The theorem \ref{theo:theta and ocat} and the corollary \ref{cor:changing theta} (which is proved in the next section) induce equivalences
$$\zocat\cong \Psh{\Theta_{n+1}}_{\W_{n+1}}\cong \Psh{\Delta[\Theta_n]}_{\M_{n+1}}.$$
\end{construction}

\subsection{The link between presheaves on $\Theta$ and on $\Delta[\Theta]$}

\begin{definition}
\label{defi:reedycof}
Let $C$ be a cocomplete category. A functor $F:A\to C$ is \textit{Reedy cofibrant}\index[notion]{Reedy cofibrant functor} if $A$ has a structure of  an elegant Reedy category (definition \ref{defi:reedy}) and for every object $a$, the  induced morphism $\colim_{\partial a}F\to F(a)$ is a monomorphism.
\end{definition}

\begin{definition}
\label{defi:precomplet}
A class of monomorphism $T$ of a  cocomplete category $C$ is \wcnotionsym{precocomplete}{(ss@$\overline{S}$}{precocomplete set of arrows} if
\begin{enumerate}
\item[$-$] It is closed by transfinite compositions and pushouts.
\item[$-$] It is closed by \notion{left cancellation}, i.e for any pair of composable morphisms $f$ and $g$, if $gf$ and $f$ are in $T$ , so is $g$.
\item[$-$] For any Reedy cofibrant diagram $F:A\to \Arr(C)$  that is pointwise in $S$, the morphism $\colim_AF$ is in $T$.
\end{enumerate} 
For a set of monomorphisms $S$, we denote $\overline{S}$ the smallest precocomplete class of morphisms containing $S$.
\end{definition}

\begin{lemma}
\label{lemma:precomplete included into cocomplete strict case}
If $S$ is a set of monomorphisms of a cocomplete category $C$, $\overline{S}$ is included in the smallest class of morphisms of $C$ stable under colimit and composition.
\end{lemma}
\begin{proof}
We denote $\widehat{S}$ the smallest class of morphisms of $C$ stable under colimit and composition. We have to show that $\widehat{S}$ is precomplete.

Let $\kappa$ be an cardinal and \(x_\uvar: \kappa \to C\) a functor such that for any \(\alpha < \kappa\), \(x_\alpha \to x_{\alpha+1}\) is in $\widehat{S}$. We show by transfinite composition that for any \(\alpha \leq \kappa\), \(x_0 \to x_\alpha\) is in $\widehat{S}$.

Suppose first that \(\alpha\) is of the shape \(\beta + 1\). The morphism \(x_0 \to x_{\beta + 1}\) is equal to the composite \(x_0 \to x_\beta \to x_{\beta + 1}\) and is then in $\widehat{S}$. If \(\alpha\) is a limit ordinal, \(x_0 \to x_\alpha\) is the colimit of the sequence of morphisms \(\{x_0 \to x_\beta\}_{\beta < \alpha}\) and is then in $\widehat{S}$.

Suppose now that we have a cocartesian square
\[\begin{tikzcd}
	a & c \\
	b & d
	\arrow[from=1-1, to=1-2]
	\arrow["i"', from=1-1, to=2-1]
	\arrow["j", from=1-2, to=2-2]
	\arrow[from=2-1, to=2-2]
	\arrow["\lrcorner"{anchor=center, pos=0.125, rotate=180}, draw=none, from=2-2, to=1-1]
\end{tikzcd}\]
where $i$ is in $\widehat{S}$. The morphism $j$ is then the horizontal colimit of the diagram
\[\begin{tikzcd}
	a & a & c \\
	b & a & c
	\arrow["i"', from=1-1, to=2-1]
	\arrow[from=1-2, to=1-1]
	\arrow[from=1-2, to=1-3]
	\arrow[from=1-2, to=2-2]
	\arrow[from=1-3, to=2-3]
	\arrow["i", from=2-2, to=2-1]
	\arrow[from=2-2, to=2-3]
\end{tikzcd}\]
and is then in $\widehat{S}$.

Let $f, g$ be two composable morphisms such that \(gf\) and \(f\) are in \(\widehat{S}\). The morphism \(g\) is then the horizontal colimit of the diagram:
\[\begin{tikzcd}
	b & a & a \\
	b & b & c
	\arrow["{id_b}", from=1-1, to=2-1]
	\arrow[from=1-2, to=1-1]
	\arrow[from=1-2, to=1-3]
	\arrow["f", from=1-2, to=2-2]
	\arrow["gf", from=1-3, to=2-3]
	\arrow[from=2-2, to=2-1]
	\arrow[from=2-2, to=2-3]
\end{tikzcd}\]
and is then in $\widehat{S}$. 

Eventually, as \(\widehat{S}\) is obviously closed under colimits of Reedy cofibrant diagrams, this concludes the proof.
\end{proof}

In the previous construction, we constructed an adjonction
$$
\begin{tikzcd}
	{ i_!:\Psh{\Delta[\Theta]}} & {\Psh{\Theta}:i^*}
	\arrow[shift left=2, from=1-1, to=1-2]
	\arrow[shift left=2, from=1-2, to=1-1]
\end{tikzcd}
$$
The aim of this subsection is to demonstrate the following theorem:

\begin{theorem}
\label{theo:unit and counit are in W}
For any $a\in \Theta$ and $b\in \Delta[\Theta]$, morphisms $i_!i^*a\to a$ and $b\to i^*i_! b$ are respectively in $\overline{\W}$ and $\overline{\M}$.
\end{theorem}

As a corollary, we directly have:
\begin{cor}
\label{cor:changing theta}
For any $n\in \Nb\cup\{\omega\}$, the adjunction 
$$\begin{tikzcd}
	{\Lb i_!:\Psh{\Delta[\Theta]_n}_{\M_n}} & {\Psh{\Theta_{n+1}}_{\W_n}:\Rb i^*}
	\arrow[shift left=2, from=1-1, to=1-2]
	\arrow[shift left=2, from=1-2, to=1-1]
\end{tikzcd}$$
given in \eqref{eq:derived adjunction case n strict} is an adjoint equivalence. 
\end{cor}
\begin{proof}
This is a consequence of theorem \ref{theo:unit and counit are in W} and lemma \ref{lemma:precomplete included into cocomplete strict case}, which states that $\overline{\W_n}$ (resp. $\overline{\M_n}$) is included in the smallest class containing $\W_n$ (resp. $\M_n$) and stable under composition and colimit.
\end{proof}

\begin{definition}
We denote by 
$$[\uvar,\uvar]: \Psh{\Theta}\times \Psh{\Delta}\to \Psh{\Delta[\Theta]}_{[\emptyset,n]/}$$
the left Kan extension of the functor $\Theta\times \Delta\to \Psh{\Delta[\Theta]}$ sending $(a,n)$ onto $[a,n]$, and where $[\emptyset,n]:=\coprod_{k\leq n}\{k\}$.
For an integer $n$, we denote
$$[\uvar,n]:\Psh{\Theta}^n\to \Psh{\Theta}_{[\emptyset,n]/}$$ 
the left Kan extension of the functor 
$\Theta^n\to\Psh{\Theta}_{[\emptyset,n]/}$ sending $\textbf{a}:=\{a_1,...,a_n\}$ onto $[\textbf{a},n]$,  and where $[\emptyset,n]:= \coprod_{k\leq n}\{k\}$. Eventually, we define 
$$[\uvar,d^0\cup d^n]:\Psh{\Theta}^n\to \Psh{\Theta}_{[\emptyset,n]/}$$ 
the left Kan extension of the functor 
$\Theta^n\to\Psh{\Theta}_{[\emptyset,n]/}$ sending $\textbf{a}:=\{a_1,...,a_n\}$ onto the colimit of the span.
$$[\{a_0,...,a_{n-2}\},{n-1}]\leftarrow [\{a_1,...,a_{n-2}\},{n-2}]\to [\{a_1,...,a_{n-1}\},{n-1}]$$
\end{definition}

\begin{lemma}
\label{lemma:the functor [] preserves classes}
The image of $\overline{\W}\times \overline{\W_1}$ by the functor $[\uvar,\uvar]:\Psh{\Theta}\times \Psh{\Delta}\to \Psh{\Delta[\Theta]}$ is included in $\overline{\W}$.
\end{lemma}
\begin{proof}
As $[\uvar,\uvar]$ preserves colimits and monomorphisms, it is enough to show that for any pair $f,g\in \W\times \W_1$, $[f,g]$ is in $\W$, which is obvious.
\end{proof}

\begin{lemma}
\label{lemma:i etoile of W is in M 0}
For any globular sum $v$, and any integer $n$,
the morphism $[v,d^0\cup d^n]\cup[\partial v,n]\to [v,n]$ appearing in the diagram
\[\begin{tikzcd}
	{[\partial v,d^0\cup d^n]} & {[v,d^0\cup d^n]} \\
	{[\partial v,n]} & {[\partial v,n]\cup[v,d^0\cup d^n]} \\
	&& {[ v,n]}
	\arrow[from=1-1, to=2-1]
	\arrow[from=1-2, to=2-2]
	\arrow[from=2-2, to=3-3]
	\arrow[curve={height=18pt}, from=2-1, to=3-3]
	\arrow[curve={height=-18pt}, from=1-2, to=3-3]
	\arrow[from=1-1, to=1-2]
	\arrow[from=2-1, to=2-2]
\end{tikzcd}\]
is in $\overline{\M}$.
\end{lemma}
\begin{proof}
If $v$ is $\emptyset$, this morphism is the identity. Suppose now that $v$ is a globular sum $a$. Remark that the morphism $[a,\Sp_n]\to [a,d^0\cup d^n]$ is in $\overline{\M}$. By left cancellation, this implies that $[a,d^0\cup d^n]\to [a,n]$ is in $\overline{\M}$. 
Let $X$ be a presheaf on $\Theta$. As $X$ is a colimit of globular sum indexed by the Reedy cofibrant diagram $\Theta^+_{/X}\to \Psh{\Theta}$ (definition \ref{defi:reedycof}), and as $[\uvar,d^0\cup d^n]\to [\uvar,n]$ preserve cofibrations, this implies that $[X,d^0\cup d^n]\to [X,n]$ is  in $\overline{\M}$.
 In particular, $[\partial v,d^0\cup d^n]\to [\partial v,n]$ is in $\overline{\M}$, and so is $[v,d^0\cup d^n]\to [\partial v,n]\cup[v,d^0\cup d^n]$ by stability by coproduct. A last use of the stability by left cancellation then concludes the proof.
\end{proof}

\begin{definition} 
Let  $[b,m]$ be an element of $\Delta[\Theta]$. We denote $\Hom^*(i([b,m]),[\textbf{a},n])$ the subset of $\Hom(i([b,m]),[\textbf{a},n])$ that consists of morphisms that preserve extremal objects. The explicit expression of morphism in $\Theta$ given in remark \ref{remark:morphism of theta} implies the bijection:
\begin{equation}
\label{eq:hom in theta}
\Hom_{\Theta}^*(i([b,m]),[\textbf{a},n])\cong\Hom_{\Delta}([n],[m])^*\times \prod_{i<n}\Hom_{\Theta}(b,a_i)
\end{equation}
where $\Hom_{\Delta}^*([n],[m])$ is the subset of $\Hom_{\Delta}([n],[m])$ consisting of morphisms that preserve extremal objects.

Let $\textbf{a}:=\{a_0,a_1,...,a_{n-1}\}$ be a finite sequence of globular sums. We define $\Theta^{\hookrightarrow}_{/\textbf{a}}$ as the category whose objects are collections of maps $\{b\to a_i\}_{ i< n}$ such that there exists no degenerate morphism $b\to b'$ factorizing all $b\to a_i$. Morphisms are monomorphisms $b\to b'$ making all induced triangles commute. 
\end{definition}

\begin{remark}
The bijection \eqref{eq:hom in theta}  induces a bijection between the objects of $\Theta^{\hookrightarrow}_{/\textbf{a}}$ and the morphisms $[b,n]\to i^*[\textbf{a},n]$ that are the identity on objects and that can not be factored through any 	degenerate morphism $[b,n]\to [\tilde{b},n]$.  	
\end{remark}

\begin{lemma}
\label{lemma:i etoile of W is in M 1}
For any morphism $p:[b,m]\to i^*[\textbf{a},n]$ in $\Psh{\Delta[\Theta]}$ that preserves extremal objects, there exists a unique pair $(\{b'\to a_i\}_{i<n},[f,i]:[b,m]\to [b',n])$ where $\{b'\to a_i\}_{i<n}$ is an element of $\Theta^{\hookrightarrow}_{/\textbf{a}}$, $f$ is a degenerate morphism, and such that the induced triangle
\[\begin{tikzcd}
	{[b,m]} & {[b',n]} \\
	& {i^*[\textbf{a},n]}
	\arrow["{[f,i]}", from=1-1, to=1-2]
	\arrow["{p'}", from=1-2, to=2-2]
	\arrow["p"', from=1-1, to=2-2]
\end{tikzcd}\]
commutes.
\end{lemma}
\begin{proof}
By adjunction and thanks to the bijection \eqref{eq:hom in theta}, $p$ corresponds to a pair $(j:[m]\to [n], \{b\to a_i\}_{i<n})$, and $i$ has to be equal to $j$.

Using once again this bijection, and the fact that degeneracies are epimorphisms, we have to show that there exists a unique degenerate morphism $g:b\to b'$ that factors the morphisms $b\to a_i$ for all $i<n$, and such that the induced family of morphisms $\{b'\to a_i\}_{i<n}$ is an element of $\Theta^{\hookrightarrow}_{/\textbf{a}}$.

As any infinite sequence of degenerate morphisms is constant at some point, the existence is immediate.

Suppose given two morphisms $b\to b'$, $b\to b''$ fulfilling the previous condition.
The proposition 3.8 of \cite{Bergner_reedy_category_and_the_theta_construction} implies that there exists a globular sum $\tilde{b}$ and two degenerate morphisms $b'\to \tilde{b}$ and $b''\to \tilde{b}$  such that the induced square
\[\begin{tikzcd}
	b & {b'} \\
	{b''} & {\tilde{b}}
	\arrow[from=1-1, to=2-1]
	\arrow[from=2-1, to=2-2]
	\arrow[from=1-1, to=1-2]
	\arrow[from=1-2, to=2-2]
\end{tikzcd}\]
is cartesian. The universal property of pushout implies that $b\to \tilde{b}$ also fulfills the previous condition. By definition of $b'$ and $b''$, this implies that they are equal to $\tilde{b}$, and this shows the uniqueness.
\end{proof}

\begin{lemma}
\label{lemma:i etoile of W is in M 0.5}
Let $\{b\to a_i\}_{ i< n}$ be an element of  $\Theta^{\hookrightarrow}_{/\textbf{a}}$  and $i:b'\to b$ a monomorphism of $\Theta$. The induced family $\{b'\to b\to a_i\}_{ i< n}$ is an object of $\Theta^{\hookrightarrow}_{/\textbf{a}}$. 
\end{lemma}
\begin{proof}
The lemma \ref{lemma:i etoile of W is in M 1} implies that there exists  a  unique degenerate morphism $j:b'\to \tilde{b}$ that factors all the morphism $b'\to b\to  a_i$ for $i<n$, and such the induced family of morphisms $\{\tilde{b}\to a_i\}_{i<n}$ is an element of $\Theta^{\hookrightarrow}_{/\textbf{a}}$.  We proceed by contradiction, and we then suppose that $j$ is different from the identity.

We then have, for any $i<n$, a commutative square
\[\begin{tikzcd}
	{b'} & b \\
	{\tilde{b}} & {a_i}
	\arrow["i", from=1-1, to=1-2]
	\arrow[from=1-2, to=2-2]
	\arrow["j"', from=1-1, to=2-1]
	\arrow[from=2-1, to=2-2]
\end{tikzcd}\]
As the morphism $j$ is degenerate and different of the identity, there exists an integer $k$ and a non trivial $k$-cell $d$ of $b'$ that is sent to an identity by $j$. Now, let $d'$ be a $k$-generator  of the polygraph $b$ that appears in the decomposition of $i(d)$. The commutativity of the previous square and the fact that the $\zo$-categories $a_i$ are polygraphs implies that for any $i$, the $k$-cell $d'$ is sent to an identity by the morphism $b\to a_i$.  As for any $i< n$ and any $l\geq k$,  there is no non trivial $l$-cell in $a_i$ whose $(k-1)$-source and $(k-1)$-target are the same, this implies that every $l$-cell of $b$ that is $(k-1)$-parallel with $d'$ is send to the identity by the morphism $b\to a_i$.

We denote $\bar{b}$ the globular sum obtained by crushing all $l$-cells of $b$ that are $(k-1)$-parallel with $d'$. The induced degenerate morphism $b\to \bar{b}$ factors all the morphisms $b\to a_i$ which is in contradiction with the fact that  $\{{b}\to a_i\}_{i<n}$ is an element of $\Theta^{\hookrightarrow}_{/\textbf{a}}$.
\end{proof}

\begin{definition}
We say that an element $\{v\to a_i\}_{i<n}$ in the category $\Theta^{\hookrightarrow}_{/\textbf{a}}$ is \textit{of height $0$} if $v\to a_0$ factors through $\partial a_0$ or $v\to a_{n-1}$ factors through $\partial a_{n-1}$. The \textit{height of an element $w$} is the maximal integer $m$ such that there exists a sequence 
$v_0\to v_1\to...\to v_m=w$ in $\Theta^{\hookrightarrow}_{/\textbf{a}}$ with $v_i\neq v_{i+1}$ for any $i<m$ and such that $v_0$ is of height $0$ and $v_1$ is not. As $\Theta$ is a Reedy category, all elements have finite height.
\end{definition}

\begin{lemma}
\label{lemma:i etoile of W is in M 1.5}
For any morphism $p:[b,m]\to i^*[\textbf{a},n]$ that preserves extremal objects, there exists a unique integer $k$, a unique element $\{b'\to a_i\}_{i<n}$ of height $k$, and a unique morphism $[f,i]:[b,m]\to [b',n]$ that doesn't factors through $[\partial b',n]$, and such that the induced triangle
\[\begin{tikzcd}
	{[b,m]} & {[b',n]} \\
	& {i^*[\textbf{a},n]}
	\arrow["{p'}", from=1-2, to=2-2]
	\arrow["{[f,i]}", from=1-1, to=1-2]
	\arrow[from=1-1, to=2-2]
\end{tikzcd}\]
commutes.

If $\{\tilde{b}\to a_i\}_{i<n}$ is any other object of non negative height, and $[\tilde{f},j]:[b,m]\to [\tilde{b},n]$ is a morphism that make the induced triangle
\[\begin{tikzcd}
	{[b,m]} & {[\tilde{b},n]} \\
	& {i^*[\textbf{a},n]}
	\arrow["{\tilde{p}}", from=1-2, to=2-2]
	\arrow["{[\tilde{f},j]}", from=1-1, to=1-2]
	\arrow[from=1-1, to=2-2]
\end{tikzcd}\]
commutative, then $\{\tilde{b}\to a_i\}_{i<n}$ is of height strictly superior to $k$ and $[\tilde{f},j]$ factors through $[\partial\tilde{b},n]$.
\end{lemma}
\begin{proof}
As $[f,i]$ doesn't factor through $[\partial b,n]$ if and only if $f$ is degenerate, the lemma \ref{lemma:i etoile of W is in M 1} implies the first assertion.
For the second one, suppose given an object $\{\tilde{b}\to a_i\}_{i<n}$ of non negative height and a morphism  $[\tilde{f},j]:[b,m]\to [\tilde{b},n]$ fulfilling the desired condition. The bijection \eqref{eq:hom in theta} directly implies that $j$ is equal to $i$, and the first assertion  implies that $\tilde{f}$ is non degenerate.

We can then factor $\tilde{f}:b\to \tilde{b}$ in a degenerate morphism $b\to \bar{b}$ followed by a monomorphism $ \bar{b}\to \tilde{b}$ which is not the identity. The lemma \ref{lemma:i etoile of W is in M 0.5} then implies that $\{\bar{b}\to \tilde{b}\to a_i\}_{i<n}$ is an element of  $\Theta^{\hookrightarrow}_{/\textbf{a}}$. The first assertion then implies that the two morphisms $[b,m]\to [b',n]$ and $[b,m]\to {[\bar{b},n]}$ are equals. As the  monomorphism $  {[{b}',n]}={[\bar{b},n]}\to [\tilde{b},n]$ is not the identity, this concludes the proof.
\end{proof}

\begin{lemma}
\label{lemma:i etoile of W is in M 2}
The morphism $i^*[\partial^0\textbf{a},n]\cup i^*[\partial^{n-1}\textbf{a},n]\to i^*[\textbf{a},n]$ is in $\overline{\M}$, where $\partial^j\textbf{a}$ corresponds to the sequence $\{a_1,..,\partial a_j,..,a_n\}$. 
\end{lemma}
\begin{proof}
For $k\in\Nb\cup\{\infty\}$, we define $x_k$ as the smallest sub object of $i^*[\textbf{a},n]$ such that for any element 
of height inferior or equal to $k$ of $\Theta^{\hookrightarrow}_{/\textbf{a}}$, the corresponding morphism $[b,n]\to i^*[\textbf{a},n]$ factors through $x_k$. In particular we have $x_0= i^*[\partial^0\textbf{a},n]\cup i^*[\partial^{n-1}\textbf{a},n]$, and the lemma \ref{lemma:i etoile of W is in M 1} implies that $x_{\infty} =i^*[\textbf{a},n]$.  We denote $(\Theta^{\hookrightarrow}_{/\textbf{a}})_{k}$ the set of element of $\Theta^{\hookrightarrow}_{/\textbf{a}}$ of height $k$.

Every morphism $[b,m]\to i^*[\textbf{a},n]$ that does not preserve extremal points then factors through $x_0$. 
The lemma \ref{lemma:i etoile of W is in M 1.5} implies that for any integer $k$, the canonical square 
\begin{equation}
\label{eq:lemma:i etoile of W is in M 2}
\begin{tikzcd}
	{\coprod_{(\Theta^{\hookrightarrow}_{/\textbf{a}})_{k+1}}[b,d^0\cup d^n]\cup[\partial b,n]} & {x_k} \\
	{\coprod_{(\Theta^{\hookrightarrow}_{/\textbf{a}})_{k+1}}[b,n]} & {x_{k+1}}
	\arrow[from=1-1, to=2-1]
	\arrow[from=2-1, to=2-2]
	\arrow[from=1-1, to=1-2]
	\arrow[from=1-2, to=2-2]
\end{tikzcd}
\end{equation}
is cocartesian.  The lemma \ref{lemma:i etoile of W is in M 0} and the stability under pushout of $\overline{\M}$  imply that $x_k\to x_{k+1}$ is in $\overline{\M}$.
 As $i^*[\textbf{a},n]$ is the transfinite composition of the sequence $x_0\to x_1\to...$, this implies that $x_0\to i^*[\textbf{a},n]$ is in $\overline{\M}$ which conclude the proof.
\end{proof}

\begin{lemma}
The morphism $i^*\Sp_a\to i^*a$ is in $\overline{\M}$ for any globular sum $a$.
\end{lemma}
\begin{proof}
Let $[\textbf{a},n]:= a$. As $\overline{\M}$ is closed under pushouts and composition, lemma \ref{lemma:i etoile of W is in M 2} implies that the morphism 
$$i^*[\{a_0,...,a_{n-2}\},n-1]\cup i^*[\{a_1,...,a_{n-1}\},n-1]\to i^*[\textbf{a},n]$$
is in $\widehat{\M}$. 
An easy induction on $n$ shows that this is also the case for the morphism 
$$[a_0,1]\cup... \cup [a_{n-1},1]= i^*[a_0,1]\cup... \cup i^*[a_{n-1},1]\to i^*[\textbf{a},n].$$
Now remark that $i^*\Sp_{[\textbf{a},n]}$ is equivalent to 
$$[\Sp_{a_0},1]\cup... \cup [\Sp_{a_{n-1}},1].$$
As the morphisms $[\Sp_i,1]\to [a_i,1]$ are by definition in $\M$, this concludes the proof.
\end{proof}

\begin{prop}
\label{prop:i etoile of W is in M}
There is an inclusion $i^*\W\subset \overline{\M}$.
\end{prop}
\begin{proof}
For Segal extensions, this is precisely the content of the last lemma. For saturation extensions, remark that $i^*\Wsat = \Msat$.
\end{proof}

\begin{proof}[Proof of theorem \ref{theo:unit and counit are in W}]
Let $a$ be a globe. We then have $i_!i^*a = a$. Suppose now that $a$ is any globular sum. We then have a commutative diagram
\[\begin{tikzcd}
	{i_!i^*\Sp_a} & {\Sp_a} \\
	{i_!i^*a} & a
	\arrow[from=1-1, to=2-1]
	\arrow[Rightarrow, no head, from=1-1, to=1-2]
	\arrow[from=1-2, to=2-2]
	\arrow[from=2-1, to=2-2]
\end{tikzcd}\]
where the upper horizontal morphism is an identity.
The proposition \ref{prop:i etoile of W is in M} and the fact that $i_!(\M)\subset \W$ implies that the vertical morphisms of the previous diagram are in $\overline{\W}$. By left cancellation, this implies that $i_!i^*a\to a$ belongs to $\overline{\W}$ for any globular sum. We proceed analogously to show that for any $b\in \Delta[\Theta]$, $b\to i^*i_! b$ is in $\overline{\M}$.
\end{proof}

\section{Gray Operations}
\subsection{Recollection on Steiner theory}

\label{section:Steiner thery} 

We present here the Steiner theory developed in \cite{Steiner_omega_categories_and_chain_complexes}.

\begin{definition}
An augmented directed complex $(K,K^*,e)$ is given by a complex of abelian groups $K$, with an augmentation $e$: $$\Zb \xleftarrow{e} K_0 \xleftarrow{\partial_0} K_1 \xleftarrow{\partial_1} K_2 \xleftarrow{\partial_2} K_3 \xleftarrow{\partial_3}. .. $$
and a graded set $K^* = (K^*_n)_{n\in\Nb}$ such that for any $n$, $K_n^*$ is a submonoid of $K_n$. A morphism of directed complexes between $(K,K^*,e)$ and $(L,L^*,e')$ is given by a morphism of augmented complexes of abelian groups $f : (K,e)\to (L,e')$ such that $f(K^*_n)\subset L^*_n$ for any $n$. We note by \wcnotation{$\CDA$}{(adc@$\CDA$} the category of augmented directed complexes. 
\end{definition}

Steiner then constructs an adjunction
\[\begin{tikzcd}
	{\lambda:\omegacat} & {\CDA:\nu}
	\arrow[""{name=0, anchor=center, inner sep=0}, shift left=2, from=1-1, to=1-2]
	\arrow[""{name=1, anchor=center, inner sep=0}, shift left=2, from=1-2, to=1-1]
	\arrow["\dashv"{anchor=center, rotate=-90}, draw=none, from=0, to=1]
\end{tikzcd}\]
The functor $\lambda$ is the simplest to define: \sym{(lambda@$\lambda:\omegacat\to \CDA$}

\begin{construction}
Let $C$ be a $\omega$-category.
We denote by $(\lambda C)_n$ the abelian group generated by the set $\{[x]_n: x\in C_n\}$ and the relations
$$[x\circ_m y]_n \sim [x]_n + [y]_n \mbox{ for $m<n$ }.$$
We define the morphism $\partial_n: (\lambda C)_{n+1}\to (\lambda C)_n$ on generators by the formula:
$$\partial_n([x]_{n+1}) := [d_n^+ x]_{n} - [d_n^- x]_{n}.$$
We can easily check that the morphism $\partial$ is a differential. We define an augmentation $e:(\lambda C)_{0}\to \Zb$ by setting $e([x]_0) = 1$ on generators. 
We denote by $(\lambda C)_n^*$ the additive submonoid generated by the elements $[x]_n$. We then set:
$$\lambda C := (\{(\lambda C)_n \}_{n\in \Nb},\{(\lambda C)^*_n \}_{n\in \Nb},e ).$$ This assignation lifts to a functor:
$$\begin{array}{ccccc}
\lambda &:& \omegacat&\to&\CDA\\
&&C&\mapsto&\lambda C.
\end{array}$$
\end{construction}
\begin{example}~
\begin{enumerate}
\item
For any integer $n$, $\lambda\Db_n$ is the augmented directed complex whose underlying chain complex is given by:
$$
\Zb\xleftarrow{e}
\Zb[e_0^-,e_0^+] \xleftarrow{\partial_0}
... \xleftarrow{\partial_{n-2}}
\Zb[e_{n-1}^-,e_{n-1}^+] \xleftarrow{\partial_{n-1}}
\Zb[e_{n}] \xleftarrow{\partial_{n}}
0\leftarrow ...$$
where for any $0<k<n$ and $\alpha\in\{-,+\}$
$$e(e_0^\alpha)=1~~~\partial_{k-1}(e_k^\alpha)= e_{k-1}^+-e_{k-1}^-~~~\partial_{n-1}(e_n)= e_{n-1}^+-e_{n-1}^-.$$
\item
The augmented directed complex $\lambda[n]$ has for underlying chain complex:
$$
\Zb\xleftarrow{e}
\Zb[v_0,v_1,...,v_n] \xleftarrow{\partial_0}
\Zb[v_{0,1},v_{1,2}...,v_{n-1,n}] \xleftarrow{\partial_{1}}
0\leftarrow ...$$
where for any $k<n$ and $\alpha\in\{-,+\}$
$$e(v_k)=e(v_n)=1~~~ \partial_{1}(v_{k,k+1})=v_{k+1}-v_k.$$
\end{enumerate}
\end{example}

 We now define the functor $\nu:\CDA\to \omegacat$. 
\begin{definition}
Let $(K,K^*,e)$ be an augmented directed complex.
A \textit{Steiner array} (or simply a \notion{array}) of dimension $n$ is the data of a finite double sequence: \sym{(nu@$\nu:\CDA\to \omegacat$}
$$\left(\begin{matrix}
x^-_0 &x^-_1&x^-_2&x^-_3 &...&x_n^-\\
x^+_0 &x^+_1&x^+_2&x^+_3 &...&x_n^+
\end{matrix}\right)$$
such that
\begin{enumerate}
\item $x^-_n=x^+_n$;
\item For any $i\leq n$ and $\alpha\in\{-,+\}$, $x_i^\alpha$ is an element of $K^*_i$;
\item For any $0<i\leq n$, $\partial_{i-1}(x_i^\alpha)= x_{i-1}^+ - x_{i-1}^-$;
\end{enumerate}
An array is said to be \wcnotion{coherent}{coherent array} if $e(x^+_0) = e(x^-_0) = 1$.
\end{definition}

\begin{definition}
Let $(K,K^*,e)$ be an augmented directed complex.
We define the globular set $\nu K$, whose $n$-cells are the coherent arrays of dimension $n$. The source and target maps are defined for $k<n$ by the formula: 

$$d^\alpha_k\begin{pmatrix}
x^-_0 &x^-_1&x^-_2&...&x^-_n\\
x^+_0 &x^+_1&x^+_2&...& x^+_n
\end{pmatrix} = \begin{pmatrix}
x^-_0 &x^-_1&x^-_2&...& x^-_{k-1}&x^\alpha_k\\
x^+_0 &x^+_1&x^+_2&...& x^+_{k-1}&x^\alpha_k\end{pmatrix}$$

There is an obvious group structure on the arrays:
$$\begin{pmatrix}
x^-_0 &x^-_1&...& x^-_n\\
x^+_0 &x^+_1&...& x^+_n
\end{pmatrix}
+
\begin{pmatrix}
y^-_0 &y^-_1&...& y^-_n\\
y^+_0 &y^+_1&...& y^+_n
\end{pmatrix}
=
\begin{pmatrix}
x^-_0+y^-_0 &x^-_1+ y^-_1&...&x^-_n+ y^-_n \\
x^+_0+y^+_0 &x^+_1+ y^+_1&...&x^+_n +y^+_n 
\end{pmatrix}
$$
\label{defi:definition of composition and units of nu k}

\begin{itemize}
\item[$-$]
For two coherent arrays $x$ and $y$ such that $d^-_k(x) =d^+_k(y) = z$, we define their $k$-composition by the following formula: 
$$x*_k y := x- z + y .$$ More explicitly:
$$\begin{pmatrix}
x^-_0 &...& x^-_n\\
x^+_0 &...& x^+_n
\end{pmatrix}
*_k
\begin{pmatrix}
y^-_0 &...& y^-_n\\
y^+_0 &...& y^+_n
\end{pmatrix}
 := 
\begin{pmatrix}
y^-_0&...&y_k^-& y_{k+1}^- + x_{k+1}^- & ...& y_{n}^- + x_{n}^-\\
x^+_0 &...&x_k^+& y_{k+1}^+ + x_{k+1}^+ & ...& y_{n}^+ + x_{n}^+ 
\end{pmatrix}
$$
\item[$-$]
For an integer $m>n$, we define the $m$-sized array $1^m_x$ as follows:
$$1^m_x :=
\begin{pmatrix}
x^-_0 &...& x^-_n& 0 &...&0\\
x^+_0 &...& x^+_n& 0 &...&0	
\end{pmatrix}$$
\end{itemize}
The globular set $\nu K$, equipped with these compositions and units is an $\omega$-category.
\end{definition}

\begin{construction}
We define the functor $\nu: \CDA \to \omegacat$ which associates to an augmented directed complex $K$, the $\omega$-category $\nu K$, and to a morphism of augmented directed complexes $f: K \to L$, the morphism of $\omega$-categories.
$$
\begin{array}{rccc}
\nu f : &\nu K &\to& \nu L\\
& \left(\begin{matrix}
x^-_0 &...&x_n^-\\
x^+_0&...&x_n^+
\end{matrix}\right) 
&\mapsto&
\left(\begin{matrix}
f_0(x^-_0) &...&f_n(x_n^-)\\
f_0(x^+_0)&...&f_n(x_n^+)
\end{matrix}\right) 
\end{array}
$$
\end{construction}

\begin{theorem}[Steiner]
\label{theo:ajdonction de steiner avec unite et counite explicite}
The functors $\lambda$ and $\nu$ form an adjoint pair 
\[\begin{tikzcd}
	{\lambda:\omegacat} & {\CDA:\nu}
	\arrow[""{name=0, anchor=center, inner sep=0}, shift left=2, from=1-1, to=1-2]
	\arrow[""{name=1, anchor=center, inner sep=0}, shift left=2, from=1-2, to=1-1]
	\arrow["\dashv"{anchor=center, rotate=-90}, draw=none, from=0, to=1]
\end{tikzcd}\]
For a $\omega$-category $C$, the unit of the adjunction is given by:
$$\begin{array}{rrcl}
~~~~~\eta :& C &\to & \nu \lambda C \\
& x\in C_n &\mapsto & 
\begin{pmatrix}
[d^-_0(x)]_0&...&[d^-_{n-1}(x)]_{n-1}&[x]_n\\
[d^+_0(x)]_0&...& [d^+_{n-1}(x)]_{n-1}&[x]_n
\end{pmatrix}
\end{array}
$$
For an augmented directed complex $K$, the counit is given by:
$$\begin{array}{rrcl}
\pi :& \lambda \nu K &\to & K~~~~~~~~~~~~~~~~ \\
& [x ]_n \in (\lambda \nu K)_n&\mapsto & x_n^+ = x_n^-
\end{array}
$$
\end{theorem}
\begin{proof}
This is \cite[theorem 2.11]{Steiner_omega_categories_and_chain_complexes}.
\end{proof}

\begin{definition}
A \snotion{basis}{for augmented directed complexes} for an augmented directed complex $(K,K^*,e)$ is a graded set $B = (B_n)_{n\in\Nb}$ such that for every $n$, $B_n$ is both a basis for the monoid $K_n^*$ and for the group $K_n$.
\end{definition}

\begin{remark}
The elements of $B_n$ can be characterized as the minimal elements of $K_n^*\backslash{0}$ for the following order relation:
	$$x\leq y \mbox{ iff } y-x \in K_n^*$$
This shows that if a basis exists, it is unique.
\end{remark}
\vspace{1cm}

We fix an augmented directed complex $(K,K^*,e)$ admitting a base $B$. Any element of $K_n$ can then be written uniquely as a sum $\sum_{b\in B_n} \lambda_b b$. This leads us to define new operations:
\begin{definition}
\label{defi:support}
For an element $x := \sum_{b\in B_n} \lambda_b b$ of $K_n$, we define the \textit{positive part} and the \textit{negative part}:
$$
\begin{array}{rcl}
(x)_+ &:=& \sum_{b\in B_n, \lambda_b> 0} ~\lambda_bb\\
(x)_- &:=& \sum_{b\in B_n, \lambda_b< 0} -\lambda_bb
\end{array}
$$
We then have $x = (x)_+ - (x)_-$. An element $x$ is \textit{positive} (resp. \textit{negative}) when $x =(x)_+$ (resp. when $x =-(x)_-$).
Let $y = \sum_{b\in B_n} \mu_b b$, we set : 
$$
\begin{array}{rcl}
x\wedge y &:=& \sum_{b\in B_n} \mbox{ min}(\lambda_b, \mu_b)~ b \\
\end{array}
$$
Eventually, we set \sym{(partialna@$\partial_n^+(\uvar)$}\sym{(partialnb@$\partial_n^-(\uvar)$}
$$
\begin{array}{rcl}
\partial_n^+(\uvar) &:=& (\partial_n(\uvar))_+ : K_{n+1}\to K^*_n\\
\partial_n^-(\uvar) &:= &(\partial_n(\uvar))_- : K_{n+1}\to K^*_n
\end{array}
$$

When an element $b$ of the basis is in the support of $x$, i.e $\lambda_b\neq 0$, we say that \textit{$b$ belongs to $x$}, which is denoted by $b\in x$.
\end{definition}

\begin{example}
For any integer $n$, $\lambda\Db_n$ admits a basis, given by the graded set $B_{\lambda\Db_n}$ fulfilling:
$$(B_{\lambda\Db_n})_k:= \left\{ 
\begin{array}{ll}
\{e_k^-,e_k^+\}&\mbox{ if $k<n$}\\
\{e_n\}&\mbox{ if $k=n$}\\
\emptyset&\mbox{ if $k>n$}\\
\end{array}\right.$$ 
The augmented directed complex $\lambda[n]$ also admits a basis, given by the graded set $B_{\lambda\Db_n}$ fulfilling:
$$(B_{\lambda\Db_n})_k:= \left\{ 
\begin{array}{ll}
\{v_0,v_1,...,v_n\}&\mbox{ if $k=0$}\\
\{v_{0,1},v_{1,2}...,v_{n-1,n}\}&\mbox{ if $k=1$}\\
\emptyset&\mbox{ if k>1}\\
\end{array} \right.$$ 
\end{example}

\begin{definition}
Let $a\in K^*_n$. We set by a decreasing induction on $k\leq n$ : 
 $$ \begin{array}{rclc}
 \langle a\rangle_k^\alpha &:= & a & \mbox{if $k = n$}\\
 &:= & \partial_k^\alpha\langle a\rangle^\alpha_{k+1} & \mbox{if not}
\end{array} 
$$
The array associated to $a$ is then: 
$$\langle a\rangle := \begin{pmatrix}
\langle a\rangle^-_0 &...&\langle a\rangle^-_{n-1}&a\\
\langle a\rangle^+_0 &...&\langle a\rangle^+_{n-1}&a
\end{pmatrix}$$
The basis is said to be \wcnotion{unitary}{unitary basis} when for any $b\in B$, the array $\langle b\rangle$ is coherent.
\end{definition}

\begin{definition}
 We define the relation $\odot$ on $B$ as being the smallest transitive and reflexive relation such that for any pair of elements of the basis $a,b$, 
$$a\odot b \mbox{ if } \mbox{($|a|>0$ and $b\in\langle a\rangle_{|a|-1}^-$)}~~\mbox{or}~~\mbox{($|b|>0$ and $a\in \langle b\rangle_{|b|-1}^+$)}$$
A basis is said to be \wcsnotion{loop free}{loop free basis}{for augmented directed complexes} the relation $\odot$ is a (partial) order on $B$.
\end{definition}

\begin{remark}
In \cite{Ara_Maltsiniotis_joint_et_tranche}, this notion is called \textit{strongly loop free}.
\end{remark}

\begin{example}
For any integer $n$, $\lambda\Db_n$ and $\lambda[n]$ admit a loop free and unitary basis.
\end{example}

\begin{definition}
 We now define the subcategory \wcnotation{$\CDAB$}{(adcb@$\CDAB$} of $\CDA$ composed of augmented directed complexes which admit a unitary and loop free basis. 
 \end{definition}

We will now describe the analog of the notion of basis for $\omega$-categories. 

\begin{definition}
A $\omega$-category $C$ is \wcnotion{generated by composition}{generated by composition} by a set $E\subset C$ when any cell can be written as a composition of elements of $E$ and iterated units of elements of $E$. This set is a \snotion{basis}{for $\io$-categories} if $\{[e]_{d(e)}\}_{e\in E}$ is a basis of the augmented directed complex $\lambda C$. 
\end{definition}

\begin{prop}
An $\omega$-category $C$ that admits a basis is an $\zo$-category.
\end{prop}
\begin{proof}
Let $C$ be an $\omega$-category that admits a basis $E$. Suppose that there exists a non trivial $n$-cell $\alpha$ that admits an inverse $\beta$. We then have $[\alpha]_n+ [\beta]_n=[\alpha \circ_{n-1} \beta]_n =0$. As $\lambda C$ is free, we have $[\alpha]_n=0$. This implies the equality $[e]_n=0$ for any element $e\in E$ of dimension $n$ that appears in a decomposition of $\alpha$. This is obviously in contradiction with the fact that $\{[e]_{d(e)}\}_{e\in E}$ is a basis of the augmented directed complex $\lambda C$. 
\end{proof}

\begin{definition}
\label{defi:loop free and atomic}
A basis $E$ of an $\zo$-category is : 
\begin{enumerate}
\item \wcsnotion{Loop free}{loop free basis}{for $\zo$-categories} when $\{[e]_{d(e)}\}_{e\in E}$ is.
\item \wcnotion{Atomic}{atomic basis} when $[d_n^+ e]_n \wedge [d_n^- e]_n = 0$ for any $e\in E$ and any natural number $n$ strictly smaller than the dimension of $e$. 
\end{enumerate}
\end{definition}

\begin{prop}
 If a loop free basis $E$ is atomic then $\{[e]\}_{e\in E}$ is unitary.
 \end{prop}
\begin{proof}
 This is \cite[proposition 4.6]{Steiner_omega_categories_and_chain_complexes}.
 \end{proof}

\begin{example}
For any integer $n$, $\Db_n$ and $[n]$ admit a loop free and atomic basis.
More generally, \cite[proposition 4.13]{Ara_Maltsiniotis_joint_et_tranche} states that 
any globular sum admits a loop free and atomic basis. 
\end{example}

\begin{definition}
Proposition $1.23$ of \cite{Ara_a_categorical_characterization_of_strong_Steiner_omega_categories} states that if an $\zo$-category admits a loop-free and atomic basis, it is unique.
We then define the category \wcnotation{$\zocatB$}{((a30@$\zocatB$} as the full subcategory of $\omegacat$ composed of $\zo$-categories admitting an atomic and loop-free basis.
\end{definition}

 \begin{theorem}[Steiner]
 \label{theorem:steiner}
 Once restricted to $\zocat_B$ and $\CDAB$, the adjunction 
\[\begin{tikzcd}
	{\lambda:\omegacat} & {\CDA:\nu}
	\arrow[""{name=0, anchor=center, inner sep=0}, shift left=2, from=1-1, to=1-2]
	\arrow[""{name=1, anchor=center, inner sep=0}, shift left=2, from=1-2, to=1-1]
	\arrow["\dashv"{anchor=center, rotate=-90}, draw=none, from=0, to=1]
\end{tikzcd}\]
becomes an adjoint equivalence, i.e. :
$$ \lambda_{|\zocatB } \circ \nu_{|\CDAB} \cong id_{|\CDAB}~~~~~~~ id_{|\zocatB }\cong \nu_{|\CDAB} \circ \lambda_{|\zocatB }$$
\end{theorem}
\begin{proof}
See \cite[theorem 5.11]{Steiner_omega_categories_and_chain_complexes}.
\end{proof}

\begin{remark}
If $K$ is an augmented directed complex admitting a unitary and loop-free basis $B$, then the $\zo$-category $\nu K$ admits an atomic and loop-free basis given by the set $\langle B\rangle := \{\langle b\rangle,b\in B\}$. Conversely if an $\zo$-category $C$ admits an atomic and loop-free basis $E$, then the augmented directed complex $\lambda C$ admits a unitary and loop-free basis given by the family of sets $[E_n] := \{[e]_{d(e)}, e\in E_n\}$. 
The isomorphisms
$$\lambda \nu K\cong K \mbox{~~~ and ~~~} C\cong \nu\lambda C$$
induce isomorphisms:
$$[\langle B\rangle ]\cong B \mbox{~~~ and ~~~} E \cong \langle [E]\rangle.$$
\end{remark}

\begin{definition}
Let $f:M\to N$ be a morphism between two augmented directed complexes admitting unitary and loop-free bases $B_M$ and $B_N$. The morphism $f$ is \wcnotion{quasi-rigid}{quasi-rigid morphism} if for any $n$, and any $b\in (B_M)_n$,
$$f_n(b)\neq 0 ~\Rightarrow ~ f_n(b)\in B_N\mbox{ and }\nu(f)\langle b\rangle = \langle f_n(b)\rangle.$$
\end{definition}

\begin{theorem}
\label{theo:Kan condition}
Suppose given a commutative square in $\CDAB$
\[\begin{tikzcd}
	K & {M_1} \\
	{M_0} & M
	\arrow["{k^0}", from=1-1, to=1-2]
	\arrow["{l^1}", from=1-2, to=2-2]
	\arrow["{k^0}"', from=1-1, to=2-1]
	\arrow["{l^0}"', from=2-1, to=2-2]
\end{tikzcd}\]
and such that all morphisms are quasi-rigid. Let $B_K,~B_{M_0},~B_{M_1},~B_{M}$ be the bases of $K,~M_0,~M_1,~ M$.

Then, this square is cocartesian if and only if for any $n$, the induced diagram of sets
\[\begin{tikzcd}
	{(B_{K})_n\cup\{0\}} & {(B_{M_1})_n\cup\{0\}} \\
	{(B_{M_0})_n\cup\{0\}} & {(B_{M})_n\cup\{0\}}
	\arrow["{k^0_n}", from=1-1, to=1-2]
	\arrow["{l^1_n}", from=1-2, to=2-2]
	\arrow["{k^0_n}"', from=1-1, to=2-1]
	\arrow["{l^0_n}"', from=2-1, to=2-2]
\end{tikzcd}\]
is cocartesian. Furthermore, the induced square in $\zocat$
\[\begin{tikzcd}
	{\nu K} & {\nu M_1} \\
	{\nu M_0} & {\nu M}
	\arrow["{\nu k^0}", from=1-1, to=1-2]
	\arrow["{\nu l^1}", from=1-2, to=2-2]
	\arrow["{\nu k^0}"', from=1-1, to=2-1]
	\arrow["{\nu l^0}"', from=2-1, to=2-2]
\end{tikzcd}\]
is cocartesian.
\end{theorem}
\begin{proof}
This is a combination of theorems 3.1.2 and 3.2.7 of \cite{Loubaton_condition_de_kan}.
\end{proof}

\subsection{$2$-Polygraphs and presheaves on $\Theta_2$}

The objective of this section is to prove the following theorem
\begin{theorem}
\label{theo: case of 1 and 2 category}
Let $k\leq 1$ be an integer, and let $C$ and $D$ be two $(0,2)$-categories admitting loop-free and atomic bases (definition \ref{defi:loop free and atomic}). Suppose there is a cocartesian square in $\zocat$ of shape:
\[\begin{tikzcd}
	{\partial [[k],1]} & C \\
	{[[k],1]} & D
	\arrow["x"', from=2-1, to=2-2]
	\arrow["{\partial x}", from=1-1, to=1-2]
	\arrow["j", from=1-2, to=2-2]
	\arrow["\lrcorner"{anchor=center, pos=0.125, rotate=180}, draw=none, from=2-2, to=1-1]
	\arrow[from=1-1, to=2-1]
\end{tikzcd}\]
Then, viewed as a morphism of $\Psh{\Theta_2}$, the morphism $j:C\cup x\to D$ is in $\overline{\W_2}$ which is the smallest precomplete class of morphism (definition \ref{defi:precomplet}) containing $\W_2$ ( definition \ref{defi:definition of W}).
\end{theorem}
Informally, this theorem shows that the square appearing in the previous statement is homotopically cocartesian. This result is therefore a special case of the similar but much more general theorem proved by Campion in \cite{campion2023infty}.

\vspace{1cm}

We  fix  a $(0,2)$-category $D$  admitting a loop free and atomic basis until the end of this section.

\begin{definition}
Let $v$ be a 2-cell of $ D$. The \textit{2-support of $v$}, denoted $B_2^v$, is the support of $[v]_2$ (definition \ref{defi:support}). 
The \textit{1-support of $ v$}, denoted $B_1^v$, is the union of the support of $[\pi_1^+v]_1$ with   $(\partial^-_1B_2^v)\cup B_2^v$.

For $i=1,2$, we define the relation $ <^v_i$ as the smallest transitive relation on $ B_i^v$ such that $ c<_i d$ whenever 
\[ \langle c \rangle^-_i  \wedge  \langle d \rangle^+_i  \neq 0. \]
\end{definition}

\begin{remark}
Remark that the two inclusions $(B_0^v,<^v_0)\to (B,\odot)$ and $(B_1^v,<^v_1)\to (B,\odot)$ are strictly increasing. As a consequence, $<^v_0$ and $<^v_1$ are (partial) orders. 
\end{remark}

\begin{remark}
The theorem \ref{theorem:steiner} implies that $B_1^v$ is also equal to the union of the support of $[\pi_1^-v]_1$ with $(\partial^+_1B_2^v)\cup B_2^v$.
\end{remark}

\begin{lemma}
\label{lemma:other characterization of <0}
Let $v$ be a 2-cell of $D$, and $b,b'$ be two elements of $B_1^v$. The assertion $b<^v_0b'$ holds if and only if there exists a well-defined $0$-composite 
$$b*_0....*_0 b'.$$
\end{lemma}
\begin{proof}
Straightforward.
\end{proof}

\begin{definition}
Given a finite set $E$ endowed with a strict order $<$, \textit{an ordering of $E$} is a bijective sequence $(x_i)_{i\leq n}$ of elements of $E$ such that for every $i<j$, $\neg (x_j<x_i)$. 
\end{definition}

\begin{theorem}
\label{theo:decomposition de condition de Kan}
Let $v$ be a $2$-cell of $D$, and $(w_i)_{i\leq n}$ an ordering of $B_2^v$.
There exists a decomposition of $v$ as 
$$v:=v_0*_1...*_1 v_{n}$$
such that for every $i<n$, $v_i$ is a $0$-composition of an element of $w_i$ with several $1$-generators of $D$.  Moreover, any element $z$ in $B_1^v$ of dimension $1$ appears in a decomposition of the source or a target of  $v_i$ for an integer $i\leq n$.

Moreover, for any decomposition of $v$ as 
$$v:=v_0'*_1...*_1 v'_{n}$$
such that $v_i'$ is a $0$-composition of a unique element $w_i'$ of $B_2^v$ with several $1$-generators of $D$, then the sequence $\{w_i\}_{i\leq n}$ is an ordering of $B_2^v$.
\end{theorem}
\begin{proof}
The first assertion is a consequence of \cite[theorem 2.47]{Loubaton_condition_de_kan}. 

To show the second assertion, suppose given such a decomposition. We will proceed by contradiction and then suppose that there exist $i<j$ such that $w'_j<w'_i$. We can suppose without loss of generality that $i=0$ and $j=n$. 

By a direct induction on $n$ using \cite[lemma 2.43]{Loubaton_condition_de_kan}, we have
$$\partial_1^+([v'_0]_2) \leq \partial_1^+([v_0'*_1...*_1 v'_{n}]_2)=\partial_1^+([v]_2)$$
$$\partial_1^-([v'_n]_2) \leq \partial_1^-([v_0'*_1...*_1 v'_{n}]_2)=\partial_1^-([v]_2)$$
Moreover, the inequality $w'_n<w'_0$ implies
$$\partial_1^+([v'_0]_2)\wedge \partial_1^-([v'_n]_2)\neq 0$$ 
and then
$$\partial_1^+([v]_2)\wedge \partial_1^-([v]_2)\neq 0$$
which is absurd as $\partial_1^+([v]_2)$ and $\partial_1^-([v]_2)$ are respectively defined as the positive part and the negative part of $\partial([v]_2)$.
\end{proof}

\begin{lemma}
\label{lem:functorialite de <}
Let $ D$ be a $ (0,2)$-category and $ f:C\to D$ be a morphism. 
Let $ v$ be a $2$-cell of $C$ and $ b,b'$ two elements in the $1$-support of $ v$.
\begin{enumerate}
\item $ b<^v_0b'$ implies that for all $ c\in B_1^{f(b)}$ and $ c'\in B_1^{f(b')}$, $ c<^{f(v)}_0c'$.
\item $ b<^v_1b'$ implies that for all $ c\in B_2^{f(b)}$ and $ c'\in B_2^{f(b')}$, $ \neg (c'<^{f(v)}_1 c)$.
\end{enumerate}
\end{lemma}
\begin{proof}
Suppose first that $ b<^v_0b'$.
According to lemma \ref{lemma:other characterization of <0}, we have a well defined $0$-composite 
$$b*_0...*_0b'$$
and so a well defined $0$-composite
$$f(b)*_0...*_0f(b')$$
 Let $ c\in B_1^{f(b)}$ and $ c'\in B_1^{f(b')}$.
Applying the decomposition given in theorem \ref{theo:decomposition de condition de Kan} to $f(b)$ and $f(b')$, we get a well-defined composite 
$$w*_0...*_0w'.$$
where $ w$ (resp. $w'$) is a $0$-composite of $c$ (resp. $c'$) with other generators. This then implies $c<^{f(v)}_0c'$. 

We now deal with the second case. Let $c\in B_2^{f(b)}$ and $c'\in B_2^{f(b')}$.
According to theorem  \ref{theo:decomposition de condition de Kan} there exists a decomposition of $v$ of shape 
\[ v:=v_0*v_1*_1....*_1v_n \]
where for all $i\leq n$, $v_0$ is a $0$-composite of a unique $2$-generator with $1$-generators. Moreover, the unique $i$ (resp. the unique $j$) such that $b$ belongs to $v_i$ (resp. such that $b'$ belongs to $v_j$) verifies $i<j$.

Applying the morphism $f$ and decomposing each $f(v_i)$ the same way, we get a decomposition
\[ f(v):=u_0*u_1*_1....*_1u_m \]
where for all $i\leq m$, $u_0$ is a $0$-composite of a $2$-generator with $1$-generators, and such that the unique $i$ (resp. the unique $j$) such that $c$ belongs to $u_i$ (resp. such that $c'$ belongs to $w_j$) verifies $i<j$. The second assertion of theorem \ref{theo:decomposition de condition de Kan} then implies that $ \neg (c'<^{f(v)}_1 c)$.
\end{proof}

\begin{lemma}
\label{lemma:2 incompatiblite1}
Let $v$ be a 2-cell, and $b,b'$ two different elements of the 2-support of $v$. Then $\neg (b<^v_1 b')\wedge \neg (b'<^v_1 b)$ implies that $(b<^v_0 b')\vee (b'<^v_0 b)$ holds.
\end{lemma}
\begin{proof}
We suppose that $\neg (b<^v_1 x)\wedge \neg (x<^v_1 b)$. We can then find an ordering with respect to $<^v_i$ of $B_2^v$ such that $b$ and $b'$ are one after the other. According to theorem \ref{theo:decomposition de condition de Kan}, we have a decomposition of $v$ of shape
 $...*_1 v_i*_1 v_{i+1}*_1....$ such that $v_i$ can be written as a $0$-composite of $b$ and $1$-generators and $v_{i+1}$ can be written in a $0$-composite of $b'$ and $1$-generators. We then have 
 \[ v_{i}:= ...*_0b*_0...~~~~ v_{i+1}:= ...*_0b'*_0... \]
 and then an equality between the following $1$-cells
 \[ ...*_0\pi^-_1b*_0...=\pi^-_1v_i=\pi^+_1v_{i+1}= ...*_0\pi^+_1b'*_0... \]
 As $\pi^-_1b\wedge \pi^+_1b'=0$, this implies that $\pi^-_1v_i=\pi^+_1v_{i+1}$ can be written as 
 \[ ... *_0 \pi^-_1b *_0 ... *_0 \pi^+_1b' *_0... \mbox{~~~or as~~~}... *_0 \pi^+_1b'*_0 ... *_0 \pi^-_1 b *_0...  \]
The cell $v_i*_1v_{i+1}$ can then be written as 
 \[ ... *_0 b *_0 ... *_0 b' *_0... \mbox{~~~or as~~~}... *_0 b'*_0 ... *_0 b *_0...  \]
This implies that $(b<_0 x)\vee (x<_0 b)$ holds.
\end{proof}

\begin{lemma}
\label{lemma:2 incompatiblite2}
Let $v$ be a $2$-cell, and $b,b'$ two elements of the $2$-support of $v$. Then $b<^v_0 b'$ implies that for all $\alpha\in \{-,+\}$, for all $c$ in $\langle b\rangle^\alpha_1$, $c<^v_0 b'$ holds.
\end{lemma}
\begin{proof}
By lemma \ref{lemma:other characterization of <0},
there exists a sequence $(b_i)_{i\leq n}$ such that $b_0=b$, $b_n=b'$ and for all $i<n$, $b_i$ and $b_{i+1}$ are $0$-composable. The sequence 
$$b*_0b_1*_0... *_0b_{n-1}*_0b'$$ is well defined, and then so is the sequence
$$\pi_1^\alpha b*_0b_1*_0... *_0b_{n-1}*_0b'.$$
As $\pi_1^\alpha b$ is a $0$-composite of $c$ with other elements of $B_1^v$, this concludes the proof.
\end{proof}

%

\begin{lemma}
\label{lemma:decompasiition}
 Let $r,u$ be two $2$-cells of $D$ such that $B_1^u\subset B_1^r$. Let $x$ in $B_2^r$.
Then there exists a unique decomposition of $u$ of shape
$$u= v*_{1}w*_{1}t$$ such that
\begin{enumerate}
\item for any element $b$ in $B_2^v$, $b<^r_1x$;
\item for any element $b$ in $B_2^t$, $x  <^r_1 b$;
\item for any element $b$ in $B_2^w$, $\neg (b<^r_1x) \vee  \neg (x<^r_1b)$
\end{enumerate}
If for any element of $b$ in $B_2^u$ different from $x$, $\neg (b<^r_1x) \vee  \neg (x<^r_1b)$, then  there exists a unique decomposition of $u$ of shape
$$u= v*_{0}w*_{0}t$$ such that
\begin{enumerate}
\item for any element $b$ in $B_1^v$, $b<^r_0x$;
\item for any element $b$ in $B_1^t$, $x  <^r_0 b$;
\item $w$ is either $x$ or a cell of lower dimension.
\end{enumerate}
\end{lemma}
\begin{proof}
We will construct these two decompositions at the same time. To this extend, we will use the Steiner theory recalled in section \ref{section:Steiner thery}.

Let $i$ be either $1$ or $0$. If $i=0$, we then suppose furthermore that for any element of $b$ in $B_2^u$ different from $x$, $\neg (b<^r_1x) \vee  \neg (x<^r_1b)$.
We denote by
$$\left(
\begin{array}{rcl}
u_0^-& u_1^-& u_2^-\\
u_0^+& u_1^+& u_2^+
\end{array}\right)$$
the array corresponding to the cell $u$.  
For any $i<j\leq 2$ and $\alpha\in\{-,+1\}$, we denote
$$
v_j^{\alpha}:=\sum \{b\in [u]_j^{\alpha}, ~b<_i x\} ~~~~~~
t_j^{\alpha}:=\sum \{b\in [u]_j^{\alpha},~ b>_i x\}  $$
$$
w_j^{\alpha}:=\sum \{b\in [u]_j^{\alpha}, ~\neg (b<_j x) \wedge \neg (b<_jx)\} 
$$
and 
$$
\begin{array}{lll}
v_i^{+}:= u_i^+& w_i^{+}:= v_i^{-} & t_i^{+}:= w_i^{-} \\
v_i^{-}:=  u_i^+- \partial (v_{i+1}^{-}) &
w_i^{-}:= v_i^{-} - \partial (w_{i+1}^{-}) &
t_i^{-}:= u_i^- 
\end{array}
$$
and for any $j<i$ and $\alpha\in\{-,+1\}$
$$
\begin{array}{lll}
v_j^{\alpha}:= u_j^{\alpha} & w_j^{\alpha}:= u_j^{\alpha} & t_j^{\alpha}:= u_j^{\alpha}\\
\end{array}
$$

By construction, 
we then have for any $i\leq j\leq 2$
$$u_j^{\alpha}=v_j^{\alpha}+w_j^{\alpha}+t_j^{\alpha}.$$
and 
$$\partial(v_{i+1}^-)= v_i^+-v_i^-~~~~
\partial(w_{i+1}^-)= w_i^+-w_i^- ~~~~
\partial(t_{i+1}^-)= t_i^+-t_i^-$$
and 
$$\partial(u_{i}^\alpha)= \partial(v_{i}^\alpha)=\partial(w_{i}^\alpha)=\partial(t_{i}^\alpha)$$
It then remains to show that  for any $i+1<j\leq 2$
\begin{equation}
\label{eq:first eq to show}
\partial v_j^\alpha = v_{j-1}^+-v_{j-1}^-~~~~~
\partial w_j^\alpha = w_{j-1}^+-w_{j-1}^-~~~~~
\partial t_j^\alpha = t_{j-1}^+-t_{j-1}^-
\end{equation}
and
\begin{equation}
\label{eq:seond eq to show}
v_i^-\geq 0~~~~ w_i^-\geq 0
\end{equation}
Indeed, if the assertions \eqref{eq:first eq to show} and \eqref{eq:seond eq to show} are fulfilled, this implies that the sequences $\{v_j^\beta\}$, $\{w_j^\beta\}$ and $\{t_j^\beta\}$  are arrays and then correspond respectively to the unique cells $v,w$ and $t$  fulfilling the desired condition. 

We first deal with the assertion \eqref{eq:first eq to show}.
Suppose first that there exists an integer $j$ such that $i+1<j\leq 2$. This implies that $i=0$. The lemma \ref{lemma:2 incompatiblite1} then implies that $w^\alpha_2=\lambda x$ with $\lambda\in \{0,1\}$.
By assumption, we have 
$$\partial (u^\beta_2) = u^+_1-u^-_1$$
and then 
$$\partial (v^\beta_2)+ \partial (w^\beta_2)  + \partial (t^\beta_2) = v^+_1-v^-_1+ w^+_1-w^-_1+ t^+_1-t^-_1$$
The lemma \ref{lemma:2 incompatiblite2} implies that  any element of the base belonging to $\partial (v^\beta_2)$ (resp. to $\partial (t^\beta_2)$) is $0$-inferior to $x$ (resp. $0$-superior to $x$). Moreover,  for any $b\in \partial (w^\beta_2)=\lambda \partial x$, we have $ \neg (b<^r_1x) \vee  \neg (x<^r_1b)$.

 The previous equality then implies
$$\partial (v^\beta_2) = v^+_1-v^-_1 ~~~~~~
\partial (w^\beta_2) =  w^+_1-w^-_1~~~~~~
 \partial (t^\beta_2) =  t^+_1-t^-_1$$

We now deal with the assertion \eqref{eq:first eq to show}.
We claim that we have
$$\partial^+ v^\alpha_{i+1}\wedge \partial^- w^\alpha_{i+1}=0~~~~~~~~\partial^+ w^\alpha_{i+1}\wedge \partial^- t^\alpha_{i+1}=0~~~~~~~~
\partial^+ v^\alpha_{i+1}\wedge \partial^- t^\alpha_{i+1}=0
$$
Indeed, suppose that $\partial^+ v^\alpha_{i+1}\wedge \partial^- w^\alpha_{i+1}\neq 0$. This implies that there exists an element of the base $b\in w^\alpha_{i+1}$ and $c\in v^\alpha_{i+1}$ such that $b<_i c$. As we have by definition $c<_i x$, this directly implies that $b<_i x$ which is absurd. We show similarly the two other equalities.
This implies that 
$$
\begin{array}{rcl}
u_i^{+} &\geq& \partial(u_{i+1}^-)\\
 &=& \partial^+(v_{i+1}^{-}+w_{i+1}^{-}+t_{i+1}^{-})\\
&=&\partial^+(v_{i+1}^{-}) + (\partial^+(w_{i+1}^{-}) - \partial^-(v_{i+1}^{-}))_+ +  (\partial^+(t_{i+1}^{-}) - \partial^-(w_{i+1}^{-}) -  \partial^-(v_{i+1}^{-}))_+
\end{array}$$
As a consequence, we have 
$$
\begin{array}{rcl}
v^-_{i} &=& u_i^+- \partial (v_{i+1}^{-})\\
&=& u_i^+- \partial^+ (v_{i+1}^{-})+ \partial^- (v_{i+1}^{-})\\
&\geq&  (\partial^+(w_{i+1}^{-}) - \partial^-(v_{i+1}^{-}))_+ +  (\partial^+(t_{i+1}^{-}) - \partial^-(w_{i+1}^{-}) -  \partial^-(v_{i+1}^{-}))_+ + \partial^- (v_{i+1}^{-})\\
&\geq&  (\partial^+(w_{i+1}^{-}) - \partial^-(v_{i+1}^{-}))_+ + \partial^- (v_{i+1}^{-})\\
&\geq & \partial^+(w_{i+1}^{-})\\
\end{array}
$$
and
$$w^-_{i} =  v_i^{-} - \partial (w_{i+1}^{-})= v_i^{-} - \partial^+ (w_{i+1}^{-})+ \partial^- (w_{i+1}^{-})\geq 0$$
The two assertions \eqref{eq:first eq to show} and \eqref{eq:seond eq to show} are then fulfilled, which concludes the proof.
\end{proof}

\begin{lemma}
\label{lemma:unicity of fact 1}
Let $C$ be a  $(0,2)$-category with a atomic and loop free basis. Let $x$ be a element of the base of $C$, and $y$ an element of the base of $D$. Let $f:C\to D$ be a morphism such that $\lambda f x=y$. Let $u$ be an $2$-cell of $C$. We denote by $u=:u_0*_0u_1*_0u_2$ and $f(u)=:v_0*_1v_1*_1v_2$ the decomposition given by the lemma \ref{lemma:decompasiition}. Then
$$f(u_0)=v_0~~~~ f(u_1)=v_1~~~~ f(u_2)=v_2$$
\end{lemma}
\begin{proof}
This is a direct consequence of lemma \ref{lemma:unicity of fact 1}.
\end{proof}

\begin{lemma}
\label{lemma:unicity of fact 2}
Let $C$ be a  $(0,2)$-category with a atomic and loop free basis. Let $x$ be a element of the base of $C$, and $y$ an element of the base of $D$. Let $f:C\to D$ be a morphism such that $y$ belongs to $\lambda f x$. Let $u$ be an $2$-cell of $C$. We denote by $u=:u_0*_1u_1*_1u_2$  and $f(u)=:v_0*_1v_1*_1v_2$ the decompositions given by lemma \ref{lemma:decompasiition}. For any $i\leq 2$,  we denote by $f(u_i)=:u_{i0}*_1u_{i1}*_1u_{i2}$ the decomposition given by lemma \textit{op cit}. Then
$$v_0=u_{00}~~~~v_1 = u_{01}*_1u_{02}*_1u_{10}*_1u_{11}*_1u_{12}*_1 u_{20}*_1u_{21}~~~~ v_2=u_{22}$$
\end{lemma}
\begin{proof}
This is a direct consequence of lemma \ref{lemma:unicity of fact 1}.
\end{proof}

\begin{notation}
Let $a$ be a globular sum of dimension lower or equal to $2$. We denote by $\triangledown$ the unique algebraic morphism $\Db_2\xrightarrow{} a$. The  $2$-cell $\triangledown$ is called the \textit{composite cell} of $a$.
\end{notation}

\begin{remark}
\label{rem:composite cell an algebraic morphism}
If $i:a\to a'$ is an algebraic morphism, and $f:a'\to C$ any morphism, the composite cell of $f:a'\to C$ is the same as the composite cell of $fi:a\to C$.
\end{remark}

\begin{definition}
Let $b$ be an element of the base of $D$.
A $2$-cell $v$ of $D$ is \textit{$0$-comparable} with $b$ if $b\in B_2^v$ and if for any $b'\in B_2^v$, the assertion 
$\neg (b<_1^v b') \wedge \neg (b'<_1^v b)$
holds.
\end{definition} 
\begin{lemma}
\label{lemma:simplification gamma,0}
Let $a$ be a globular sum of dimension lower or equal to $2$. Let $x$ be a $2$-cell of $D$.
Let $f:a\to D$ be a morphism such that $f(\triangledown)$ is $0$-comparable with $x$.
Then there exists a commutative triangle 
\[\begin{tikzcd}
	& {a'\vee [[1],1]\vee a''} \\
	a & D
	\arrow["f"', from=2-1, to=2-2]
	\arrow["i", from=2-1, to=1-2]
	\arrow["{f'\vee x\vee f''}", from=1-2, to=2-2]
\end{tikzcd}\]
Moreover, this factorization is functorial in $D$.
\end{lemma} 
\begin{proof}
Let $d$ be the (necessarily unique) element of the basis of $a$ such that $x\in [f(d)]_2$. Let $k\leq1$ and $j:[[k],1]\to \Sp_a$ be an element of the basis, i.e., a globular morphism.

If $j=d$, we consider the diagram
\[\begin{tikzcd}
	& {[[1],1]\vee[[1],1]\vee[[1],1]} \\
	{[[1],1]} & D
	\arrow["fj"', from=2-1, to=2-2]
	\arrow[from=2-1, to=1-2]
	\arrow["{f'\vee x\vee f''}", from=1-2, to=2-2]
\end{tikzcd}\]
and if $j$ is different of $d$ by share the same $0$-source and $0$-target, we consider the diagram
\[\begin{tikzcd}
	& {[[k],1]\vee[1]\vee[[k],1]} \\
	{[[k],1]} & D
	\arrow["fj"', from=2-1, to=2-2]
	\arrow[from=2-1, to=1-2]
	\arrow[from=1-2, to=2-2]
\end{tikzcd}\]
where these two decompositions are induced by lemma \ref{lemma:decompasiition}. If the $0$-source and $0$-target of $j$ are different of the one of $d$, we consider the diagram
\[\begin{tikzcd}
	& {[[k],1]} \\
	{[[k],1]} & D
	\arrow["fj"', from=2-1, to=2-2]
	\arrow[from=2-1, to=1-2]
	\arrow["fj", from=1-2, to=2-2]
\end{tikzcd}\] 
Taking the colimit over all such $j:[[k],1]\to a$, this induces a factorization
\[\begin{tikzcd}
	& {a'\vee [[1],1]\vee a''} \\
	a & D
	\arrow["f"', from=2-1, to=2-2]
	\arrow["i", from=2-1, to=1-2]
	\arrow["{f'\vee x\vee f''}", from=1-2, to=2-2]
\end{tikzcd}\]
fulfilling the desired property. Eventually, the functoriality of this factorization is a consequence of the unicity of the decomposition given in lemma \ref{lemma:decompasiition} and of lemma \ref{lem:functorialite de <}.
\end{proof}
\vspace{1cm}

Until the end of this section, we fix an other $(0,2)$-category  $C$ admitting a loop-free and atomic basis, and fitting in a cocartesian square of $\zocat$ of shape:
\[\begin{tikzcd}
	{\partial[[1],1]} & C \\
	{[[1],1]} & D
	\arrow["f", from=1-2, to=2-2]
	\arrow["x"', from=2-1, to=2-2]
	\arrow[from=1-1, to=2-1]
	\arrow["{\partial x}", from=1-1, to=1-2]
	\arrow["\lrcorner"{anchor=center, pos=0.125, rotate=180}, draw=none, from=2-2, to=1-1]
\end{tikzcd}\]

\begin{construction}
We define $\Gamma_0$ as the full subcategory of $(\Theta_2)_{/D}$ whose objects are morphisms $f:a\to D$ such that either $f$ factors through $C$, or the following conditions are fulfilled:
\begin{enumerate}
\item $f(\triangledown)$ is $0$-comparable with $x$.
\item $\Sp_a\to a\to D$ factors through the $\Theta$-set $C\cup x$.
\end{enumerate}
We define $\Gamma_1$ as the full subcategory of $(\Theta_2)_{/D}$ whose objects are morphisms $v:a\to D$ such that $\Sp_a\to a\to D$ factors through the $\Theta$-set $C\cup \colim_{\Gamma_0}a$.
\end{construction}

\begin{lemma}
\label{lem:injectif 1}
The canonical morphism of $\Theta$-sets $\iota:\colim_{\Gamma_0}a\to D$ is injective. Its image corresponds to morphisms $f:a\to D$ such that either $f$ factors through $C$, or the $2$-cell $f(\triangledown)$ is $0$-comparable with $x$.
\end{lemma}
\begin{proof}
First, remark that the morphism $C\cup\{x\}\to \colim_{\Gamma_0}a$ is injective. To complete the characterization of the image of $\iota$, let $f:a\to D$ be a morphism such that $f(\triangledown)$ is $0$-comparable with $x$.

Consider now the factorization $a\xrightarrow{i} a'\xrightarrow{g} D$ of $f$ given by lemma \ref{lemma:simplification gamma,0}. Every element of $\Sp a'$ is sent to either an element of $C$ or to $x$. This implies that $g$ belongs to $\Gamma_0$, which concludes the characterization of the image of $\iota$.

Now, for the injectivity, suppose that there exists another element $h:b\to D$ of $\Gamma_0$ and a decomposition $a\xrightarrow{j} b\xrightarrow{h} D$ of $f:a\to D$. Up to further factorization, we can suppose that $j$ is algebraic and, according to lemma \ref{lemma:simplification gamma,0}, that $j(\triangledown)$ is $0$-comparable with the (necessarily unique) element of the basis $c$ of $b$ such that $g(c)=x$. 

Using once again the factorization lemma \ref{lemma:simplification gamma,0} on the morphism $j$ and the object $c$, and using the functoriality of this factorization, we get a commutative diagram
\[\begin{tikzcd}
	a & b \\
	{a'} & D
	\arrow["i"', from=1-1, to=2-1]
	\arrow["g"', from=2-1, to=2-2]
	\arrow["h", from=1-2, to=2-2]
	\arrow["j", from=1-1, to=1-2]
	\arrow[from=2-1, to=1-2]
\end{tikzcd}\]
 completing the proof of injectivity.
\end{proof}

\begin{lemma}
\label{lem:injectif 2}
The canonical morphism of $\Theta$-sets $\iota:\colim_{\Gamma_1}a\to D$ is an equivalence.
\end{lemma}
\begin{proof}
First, remark that the morphism $C\cup  \colim_{\Gamma_0}a \to \colim_{\Gamma_1}a$ is injective. To complete the surjectivity of $\iota$, let $f:a\to D$ be a morphism that doesn't factor through $C$. In particular, this implies that $x$ belongs to $[f(\triangledown)]_2$. We denote by $c$ as the (necessary) unique element of the base of $a$ such that $x\in [f(c)]_2$.

Let $k\leq 1$ and $j:[[k],1]\to \Sp_a$ be an element of the basis. If $j$ is $c$, we consider the following diagram
$$[[1],1]\to [[3],1]\to D$$
induced by the decomposition of lemma \ref{lemma:decompasiition}. Moreover, lemma \ref{lem:injectif 1} implies that $l$ belongs to $\Gamma_1$.
If $j$ is different from $c$, we consider the diagram
$$[[k],1]\to [[k],1]\to D$$
Moreover, $fj$ factors through $C$ and then belongs to $\Gamma_1$. Taking the colimit over all such $j$, this induces a diagram
$$a\xrightarrow{i} a' \xrightarrow{g} D$$
whose composite is $f$ and such that $g$ is in $\Gamma_1$. This concludes the proof of the surjectivity of $\iota$.

To prove the injectivity, suppose now that there exists another element $h:b\to D$ and a decomposition $a\xrightarrow{j} b\xrightarrow{h} D$ of $f:a\to D$ with $h$ in $\Gamma_1$. Let $k\leq 1$ and $j:[[k],1]\to \Sp_a$ be an element of the basis. If $j$ is $c$, we consider the diagram
\[\begin{tikzcd}
	{[[1],1]} & {[[3],1]} & {a'} \\
	{[[3],1]} & {[[9],1]} \\
	b && D
	\arrow[from=2-1, to=2-2]
	\arrow[from=1-1, to=2-1]
	\arrow["g", from=1-3, to=3-3]
	\arrow["h"', from=3-1, to=3-3]
	\arrow["t", from=2-2, to=3-3]
	\arrow[from=2-1, to=3-1]
	\arrow[from=1-2, to=1-3]
	\arrow["{[\sigma,1]}", from=1-2, to=2-2]
	\arrow[from=1-1, to=1-2]
\end{tikzcd}\]
where the left vertical morphisms are induced by the decomposition of  lemma \ref{lemma:decompasiition}, the morphism $t$ obtained in applying for each $2$-cell the decomposition of lemma \textit{op cit}, and the morphism $\sigma$ send $0$ on $0$, $1$ on $1$, $2$ on $8$ and $3$ on $9$. The commutativity of this diagram is a consequence of  lemma \ref{lemma:unicity of fact 2}. 

If $j$ is different from $c$, we consider the diagram
\[\begin{tikzcd}
	{[[k],1]} & {[[k],1]} & {a'} \\
	{[[k],1]} & {[[k],1]} \\
	b && D
	\arrow[Rightarrow, no head, from=2-1, to=2-2]
	\arrow[Rightarrow, no head, from=1-1, to=2-1]
	\arrow[from=1-3, to=3-3]
	\arrow["h"', from=3-1, to=3-3]
	\arrow[from=2-2, to=3-3]
	\arrow[from=2-1, to=3-1]
	\arrow[from=1-2, to=1-3]
	\arrow[Rightarrow, no head, from=1-1, to=1-2]
	\arrow[Rightarrow, no head, from=1-2, to=2-2]
\end{tikzcd}\]
Taking the colimit over all such $j$, this induces a diagram
\[\begin{tikzcd}
	a & {a'} & {a'} \\
	{a'} & {a''} \\
	b && D
	\arrow[from=2-1, to=2-2]
	\arrow["i"', from=1-1, to=2-1]
	\arrow["g", from=1-3, to=3-3]
	\arrow["h"', from=3-1, to=3-3]
	\arrow[from=2-2, to=3-3]
	\arrow[from=2-1, to=3-1]
	\arrow[Rightarrow, no head, from=1-2, to=1-3]
	\arrow["i", from=1-1, to=1-2]
	\arrow[from=1-2, to=2-2]
\end{tikzcd}\]
where $a''\to D$ is in $\Gamma_1$,
which concludes the proof of injectivity.	
\end{proof}

\begin{lemma}
\label{lemma: lambdA gamma}
Let $f: a\to D$ be a morphism of $\Gamma_0$. We denote by $\Lambda^{\Gamma_0} a$  the subobject of $a$ composed of all $i\in {\Theta_2}_{/a}$ such that $fi$ factors through the  $\Theta_2$-set $C\cup x$. Then the morphism $\Lambda^{\Gamma_0} a\to a$ is in $\overline{\W_2}$.
\end{lemma}
\begin{proof}
If $f$ factors through $C$, then $\Lambda^{\Gamma_0}a$ is equal to $a$. Suppose then that there exists a (necessarily unique) element of the base $b$ such that $f(b)=x$. 

There exists a unique decomposition of $a$ as
$$a\cong a' \vee [[k]\vee[1]\vee[k'],1]\vee a''$$
where the cell $[[1],1]\to a$ is $b$ and where
$$[[k],1]\to a\to D~~~~\mbox{and}~~~~[[k'],1]\to a\to D$$
factors through $C$.

We then have
$$\Lambda^{\Gamma_0} a\cong a'\vee[[k]\coprod_{[0]} [1]\coprod_{[0]}[k'],1]\vee a''$$
As the functor $a'\vee[\uvar,1]\vee a:\Psh{\Delta}\to \Psh{\Theta}$ sends $\overline{\W_1}$ to $\overline{\W_2}$, and as $$[k]\coprod_{[0]} b\coprod_{[0]}[k']\to [k+1+k']$$ is in $\overline{\W_1}$, this concludes the proof.
\end{proof}

\begin{lemma}
\label{lemma: lambdA gamma2}
Let $f: a\to D$ be a morphism of $\Gamma_1$. We denote by $\Lambda^{\Gamma_1} a$  the subobject of $a$ composed of all $i\in \Theta_{/a}$ such that $fi$ factors through $\colim_{\Gamma_0}a$. Then the morphism $\Lambda^{\Gamma_1} a\to a$ is in $\overline{\W_2}$.
\end{lemma}
\begin{proof}
If $f$ factors through $C$, then $\Lambda^{\Gamma_1}a$ is equal to $a$. Suppose then that there exists a (necessarily unique) element of the base $b$ such that $x$ belongs to $[f(b)]_2$.

There exists a unique decomposition of $a$ as
$$a\cong a' \vee [[n]\vee[k]\vee[1]\vee[k']\vee[n'],1]\vee a'' $$
where the cell $[[1],1]\to a$ is $b$, and where $k$ and $k'$ are the maximal integers such that the image by the composite cell of 
$$[[k]\vee[1]\vee[k'],1]\to a$$
is $0$-comparable with $x$, and such that
$$[[k],1]\coprod [[k'],1]\to a\to D$$
factors through $C$.

We then have
$$\Lambda^{\Gamma_0} a\cong a'\vee[[n+k]\coprod_{[k]} [k+1+k']\coprod_{[k']}[k'+n'],1]\vee a''$$
As the functor $a'\vee[\uvar,1]\vee a:\Psh{\Delta}\to \Psh{\Theta}$ sends $\overline{\W_1}$ to $\overline{\W_2}$, and as 
$$[n+k]\coprod_{[k]} [k+1+k']\coprod_{[k']}[k'+n']\to [n+k +1+k'+n']$$ is in $\overline{\W_1}$, this concludes the proof.
\end{proof}

\begin{prop}
\label{prop: case of 2 category}
Let $C$ and $D$ be two $(0,2)$-categories admitting loop-free and atomic bases, fitting in a  cocartesian square of shape:
\[\begin{tikzcd}
	{\partial[[1],1]} & C \\
	{[[1],1]} & D
	\arrow["f", from=1-2, to=2-2]
	\arrow["x"', from=2-1, to=2-2]
	\arrow["{\partial x}", from=1-1, to=1-2]
	\arrow["\lrcorner"{anchor=center, pos=0.125, rotate=180}, draw=none, from=2-2, to=1-1]
	\arrow[from=1-1, to=2-1]
\end{tikzcd}\]
Then, viewed as a morphism of $\Psh{\Theta_2}$, the morphism $j:C\cup x\to D$ is in $\overline{\W_2}$.
\end{prop}
\begin{proof}
The category $\Gamma_0$ inherits from $\Theta_{/D}$ a structure of Reedy elegant category. Moreover, the two functors 
$$\begin{array}{ccccccc}
\Gamma_0&\to &\Psh{\Delta} &~~~~~&\Gamma_0&\to &\Psh{\Delta}\\
a\to D&\mapsto &\Lambda^{\Gamma_0}a &~~~~~& a\to D &\mapsto & a
\end{array}$$
are Reedy cofibrant (definition \ref{defi:reedycof}). The morphism 
$$C\cup x\cong \colim_{\Gamma_0}\Lambda^{\Gamma_0}a\to \colim_{\Gamma_0}a$$
is then in $\overline{\W_2}$. We proceed similarly to demonstrate that the morphism
$$ \colim_{\Gamma_0}a \cong \colim_{\Gamma_1}\Lambda^{\Gamma_1} a\to  \colim_{\Gamma_1}a\cong D$$
is in $\overline{\W_2}$, and where the second isomorphism is induced by lemma \ref{lem:injectif 2}. By stability by composition of $\overline{\W_2}$, this concludes the proof.
\end{proof}

\begin{prop}
\label{prop: case of 1 category}
Let $C$ and $D$ be two $(0,1)$-categories admitting loop-free and atomic bases,  fitting in a  cocartesian square of shape:
\[\begin{tikzcd}
	{\partial[1]} & C \\
	{[1]} & D
	\arrow["f", from=1-2, to=2-2]
	\arrow["x"', from=2-1, to=2-2]
	\arrow["{\partial x}", from=1-1, to=1-2]
	\arrow["\lrcorner"{anchor=center, pos=0.125, rotate=180}, draw=none, from=2-2, to=1-1]
	\arrow[from=1-1, to=2-1]
\end{tikzcd}\]
Then, viewed as a morphism of $\Psh{\Delta}$, the morphism $j:C\cup x\to D$ is in $\overline{\W_1}$.
\end{prop}
\begin{proof}
We denote by $\Upsilon$  the full subcategory of $\Delta_{/D}$ whose objects are morphisms $f:[n]\to D$ such that $\Sp_{[n]}\to [n]\to D$ factors through the $\Theta$-set $C\cup x$. 

Given $f:[n]\to D$ in $\Upsilon$, we denote by $\Lambda^{\Upsilon}[n]$ the subobject of $[n]$ composed of all $i\in \Delta_{/[n]}$ such that $fi$ factors through $C\cup x$. We can proceed as in lemma \ref{lemma: lambdA gamma} to show that the canonical morphism $\Lambda^{\Upsilon}[n]\to [n]$ is in $\overline{\W_1}$.

Now, remark that the category $\Upsilon$ inherits from $\Delta_{/D}$ a structure of Reedy elegant category. The two functors 
$$\begin{array}{ccccccc}
\Upsilon&\to &\Psh{\Delta} &~~~~~&\Upsilon&\to &\Psh{\Delta}\\
{[n]}\to D&\mapsto &\Lambda^{\Upsilon}[n] &~~~~~& [n]\to D &\mapsto & [n]
\end{array}$$
are Reedy cofibrant (definition \ref{defi:reedycof}). As the colimit of the first one is $C\cup x$ and the colimit of the second one is $D$, this concludes the proof.
\end{proof}

\begin{proof}[Proof of theorem \ref{theo: case of 1 and 2 category}]
If $n=0$, this is straightforward, and if $n=2$, it follows from proposition \ref{prop: case of 2 category}.

It then remains to prove the case $n=1$. 
Let $S$ be the set of generators of $C$ of dimension $2$. A repeated application of proposition \ref{prop: case of 2 category} and the stability by pushout  and transfinite composition of $\overline{\W_2}$ implies that the two vertical morphisms of the following square are in $\overline{\W_2}$:
\[\begin{tikzcd}
	{\tau_1C\cup x\cup_{y\in S} y} & {\tau_1D \cup_{y\in S} y} \\
	{C\cup x} & D
	\arrow[from=1-2, to=2-2]
	\arrow[from=1-1, to=2-1]
	\arrow[from=2-1, to=2-2]
	\arrow[from=1-1, to=1-2]
\end{tikzcd}\]
Moreover, the proposition \ref{prop: case of 1 category} implies that the canonical morphism 
$$\tau_1 C\cup x\to \tau_1 D$$ is in $\overline{\W_2}$, and so is the top horizontal morphism of the previous square. By stability of left cancellation of $\overline{\W_2}$, this concludes the proof.
\end{proof}

\subsection{Gray operations on augmented directed complexes}
We follow Steiner (\cite{Steiner_omega_categories_and_chain_complexes}) and Ara-Maltsiniotis (\cite{Ara_Maltsiniotis_joint_et_tranche}) for the definitions and first properties of Gray operations on augmented directed complexes.

\begin{definition}
Let $(K,K^*,e)$ and $(L,L^*,f)$ be two augmented directed complexes. We define the \snotionsym{Gray tensor product}{((d00@$\otimes$}{for augmented directed complexes} of $(K,K^*,e)$ and $(L,L^*,f)$ as the augmented directed complex
$$(K,K^*,e)\otimes (L,L^*,f):= (K\otimes L,(K\otimes L)^*,e\otimes f)$$
where 
\begin{enumerate}
\item[$-$] $K\otimes L$ is the chain complex whose value on $n$ is:
$$(K\otimes L)_n:= \oplus_{k+l=n}K_k\otimes L_l$$
and the differential is the unique graded group morphism fulfilling: 
$$\partial (x\otimes y):= \partial x\otimes y + (-1)^{|x|}x\otimes \partial y$$
where we set the convention $\partial x:=0$ if $|x|=0$.
\item[$-$] $(K\otimes L)^*$ is given on all integer $n$ by :
$$(K\otimes L)^*_n:= \oplus_{k+l=n}K_k^*\otimes L_l^*.$$
\item[$-$] $e\otimes f:K_0\otimes L_0\to \Zb$ is the unique morphism fulfilling 
$$(e\otimes f)(x\otimes y)= e(x)f(y).$$
\end{enumerate}
\end{definition}
The Gray tensor product induces a monoidal structure on $\CDA$. Its unit is given by $\lambda \Db_0$. Furthermore, Steiner shows that if $K$ and $L$ admit loop free and unitary bases, so does $K\otimes L$. The basis of $K\otimes L$ is given by the set of elements of shape $b\otimes b'$ where $b$ and $b'$ are respectively elements of the bases of $K$ and $L$.
 The monoidal structure then restricts to a monoidal structure on $\CDAB$.

\begin{notation}
 To simplify notation, the augmented directed complex $\lambda[1]$ will simply be denoted by $[1]$. 
\end{notation}
\begin{definition}
The induced functor 
$$\uvar\otimes [1]:\CDA\to \CDA$$
is called the \snotionsym{Gray cylinder}{((d30@$\uvar\otimes[1]$}{for augmented directed complexes}. 
For $(K,K^*,e)$ an augmented directed complex, we then have
$$(K,K^*,e)\otimes [1]:=(K\otimes [1] ,(K\otimes [1])^*,e)$$
where
\begin{enumerate}
\item[$-$] $K\otimes [1]$ is the chain complex whose value on $n$ is:
$$(K\otimes [1])_n:=\left\{
\begin{array}{ll}
\{x\otimes \{\epsilon\},x\in K_0,\epsilon=0,1\}&\mbox{if $n=0$}\\
\{x\otimes \{\epsilon\},x\in K_n,\epsilon=0,1\}\oplus \{x\otimes[1],x\in K_{n-1}\} &\mbox{if $n>0$}
\end{array}\right.$$
and the differential is the unique graded group morphism fulfilling: 
$$\partial (x\otimes [1]):= \partial x\otimes [1] + (-1)^{|x|}(x\otimes \{1\}-x\otimes \{0\} )~~~~~\partial (x\otimes\{\epsilon\}) = (\partial x)\otimes\{\epsilon\}$$
for $\epsilon\in\{0,1\}$, and
where we set the convention $\partial x:=0$ if $|x|=0$.
\item[$-$] $(K\otimes [1])^*$ is given on all integer $n$ by :
$$(K\otimes [1])^*_n:=\left\{
\begin{array}{ll}
\{x\otimes \{\epsilon\},x\in K^*_0,\epsilon=0,1\}&\mbox{if $n=0$}\\
\{x\otimes\{ \epsilon\},x\in K^*_n,\epsilon=0,1\}\oplus \{x\otimes[1],x\in K^*_{n-1}\} &\mbox{if $n>0$}
\end{array}\right.$$
\item[$-$] $e:(K\otimes [1])_0\to \Zb$ is the unique morphism fulfilling 
$$e(x\otimes \{0\})=e(x\otimes \{1\})= e(x).$$
\end{enumerate}
\end{definition}

\begin{prop}
\label{prop:non trivial automorphisme 0}
Let $A$ be an augmented directed complex admitting no non-trivial automorphisms. Then the augmented directed complexe $A\otimes [1]$ has no non-trivial automorphisms.
\end{prop}
\begin{proof}
Let $\phi:A \otimes[1]\to A\otimes [1]$ be an automorphism. The morphism $\phi$ then induces a bijection on the elements of the basis of $A\otimes [1]$.

Let $(E,F)$ be a partition of the set $(B_{A\otimes[1]})_0$ such that
\begin{enumerate}
\item there exists no element of $(B_{A\otimes[1]})_1$ whose source is in is $F$ and target in $E$.
\item for any $x,y\in E$ and $v\in (B_{A\otimes[1]})_1$ such that $\partial v=y-x$, there exist an element $w\in (B_{A\otimes[1]})_1$  such that $\partial^- w=y$ and an element $\alpha\in (B_{A\otimes[1]})_2$ with $\partial^+\alpha=w+v$.
\item for any $x,y\in F$ and $v\in (B_{A\otimes[1]})_1$ such that $\partial v=y-x$, there exist an element $w\in (B_{A\otimes[1]})_1$  such that $\partial^+ w=x$ and an element $\alpha\in (B_{A\otimes[1]})_2$ with $\partial^-\alpha=w+v$.
\end{enumerate}
Suppose now that there exists an object $a$ of $(B_A)_0$ such that $a\otimes \{1\}$ in $E$. As we have $\partial (a\otimes[1])=a\otimes\{1\}-a\otimes\{0\}$, $a\otimes\{0\}$ is in $E$. There exist then an element $w \in (B_{A\otimes[1]})_1$ and an element $\alpha \in  (B_{A\otimes[1]})_2$ with $\partial^-w = a\otimes\{1\}$ and  $\partial^+ \alpha = a\otimes [1]+w$. However, by construction of $A\otimes[1]$, there exist no such element $\alpha$. This implies that any element of $E$ is of shape $a\otimes\{0\}$ and we can show similarly that every element of $F$ is of shape $a\otimes\{1\}$. 

Conversely, we claim that the partition $((B_{A\otimes\{0\}})_0, (B_{A\otimes\{1\}})_0)$ fulfills these conditions.  The first one is obvious. For the second, there exist $a\in (B_A)_0$ and $u\in (B_A)_0$ such that $y=a\otimes\{0\}$ and $v:= u\otimes\{0\}$ and we then choose $w:=a\otimes[1]$ and $\alpha:= u\otimes [1]$. We proceed similarly for the last condition. 

The partition $((B_{A\otimes\{0\}})_0, (B_{A\otimes\{1\}})_0)$ is then the unique one fulfilling the previous three condition. As $\phi$ preserves such partition, this implies that $\phi(B_{A\otimes\{0\}})= B_{A\otimes\{0\}}$ and $\phi(B_{A\otimes\{1\}})= B_{A\otimes\{1\}}$.

Now, remark that for any element $e\in (A\otimes[1])^*_{n+1}$, there exists $x\in A^*_n$ such that $x\otimes[1] \leq e$ if and only if there exists $y\in A^*_{n-1}$ such that $y\otimes[1] \leq \partial^+e$. By a direct induction, this implies that there exists $x\in A^*_n$ such that $x\otimes[1]\leq  e$ if and only if $\partial^-_0e$ is in $A_0^*\otimes\{0\}$ and $\partial^+_0e$ is in $A_0^*\otimes\{1\}$.

Combined with the previous observation, this implies that for any element $x$ of the basis of $A_{n}$, $\phi(x\otimes\{\epsilon\})$ is of shape $x'\otimes\{\epsilon\}$ with $\epsilon\in\{0,1\}$.
The automorphism $\phi$ then induces by restriction  automorphisms $\phi_{|A\otimes\{0\}}:A\otimes\{0\}\to A\otimes\{0\}$ and $\phi_{|A\otimes\{1\}}:A\otimes\{1\}\to A\otimes\{1\}$, and the hypothesis implies that they are the identity.

We now show by induction on $n$ that $\phi_n:(A\otimes[1])_n\to (A\otimes[1])_n$ is the identity. Suppose the result true at the stage $n$. For any element $x$ of the basis of $A_{n}$, we then have 
$$\partial \phi(x\otimes[1]) = \phi(\partial (x\otimes[1])) = \partial (x\otimes[1]).$$
By the definition of the derivative of $A\otimes[1]$, and as $\phi$ preserves the basis, this forces the equality $\phi(x\otimes[1])=x\otimes[1]$. As we already know that for any element $x$ of the basis of $A_{n+1}$ we have $\phi(x\otimes\{\epsilon\})=x\otimes\{\epsilon\}$ for any $\epsilon\in\{0,1\}$, this concludes the induction.

We then have $\phi=id$ and $A\otimes[1]$ has no non trivial automorphisms.
\end{proof}

\begin{definition}
We define the \snotionsym{Gray cone}{((d40@$\uvar\star 1$}{for augmented directed complexes}
$$\begin{array}{ccc}
\CDA &\to&\CDA\\
K&\mapsto &K\star 1 
\end{array}
$$
where $K\star 1$ is defined as the following pushout: 
\begin{equation}
\label{eq:defin of cstar costar CDA}
\begin{tikzcd}
	{K\otimes\{1\}} & {K\otimes [1]} \\
	1 & {K\star 1}
	\arrow[from=1-1, to=2-1]
	\arrow[from=1-1, to=1-2]
	\arrow[from=2-1, to=2-2]
	\arrow[from=1-2, to=2-2]
	\arrow["\lrcorner"{anchor=center, pos=0.125, rotate=180}, draw=none, from=2-2, to=1-1]
\end{tikzcd}
\end{equation}

According to \cite[corollary 6.21]{Ara_Maltsiniotis_joint_et_tranche}, if $K$ admits a loop free and unitary basis, this is also the case for $K\star 1$. The {Gray cone}  then induces a functor:
$$\begin{array}{ccc}
\CDAB&\to&\CDAB\\
K&\mapsto &K\star 1 \\
\end{array}
$$
\end{definition}

\begin{remark}
 Unfolding the definition, we have
$$(K,K',e)\star 1:=(K\star 1, (K\star 1)^*,e)$$
where
\begin{enumerate}
\item[$-$] $K\star 1$  is the chain complex whose value on $n$ is:
$$(K\star 1)_n:=\left\{
\begin{array}{ll}
\Zb[\emptyset\star 1]\oplus \{x\star \emptyset,x\in K_0\}&\mbox{if $n=0$}\\
\{\emptyset\star x,x\in K_n\}\oplus \{x\star 1,x\in K_{n-1}\} &\mbox{if $n>0$}
\end{array}\right.$$
and the differentials are the unique graded group morphisms fulfilling: 
$$\begin{array}{rr}
\partial (x\star 1)= \partial x\star 1 + (-1)^{|x|} x\star \emptyset&\partial( x \star \emptyset )=\partial x\star \emptyset \\
\end{array}$$
where we set the convention $\partial x:=0$ if $|x|=0$.
\item[$-$] The graded monoids $(K\star 1)^*$ is given on any integer $n$ by :
$$(K\star 1)^*:=\left\{
\begin{array}{ll}
\Nb[\emptyset\star 1]\oplus \{x\star \emptyset,x\in K^*_0\}&\mbox{if $n=0$}\\
\{\emptyset\star x,x\in K^*_n\}\oplus \{x\star 1,x\in K^*_{n-1}\} &\mbox{if $n>0$}
\end{array}\right.$$

\item[$-$] The augmentation $e:(K\star 1)_0\to \Zb$ is the unique ones fulfilling 
$$
\begin{array}{cc}
e( \emptyset \star 1) =1 & e(x\star \emptyset)=e(x)\\
\end{array}$$
\end{enumerate}
The basis of $K\star 1$ is given by the reunion of $\emptyset\star 1$ and of the set of elements of shape $b\star 1$  where $b$ is an element of the basis of $K$. 
\end{remark}

\begin{prop}
\label{prop:non trivial automorphisme 1}
Let $A$ be an augmented directed complex admitting no non-trivial automorphisms. Then the augmented directed complexe $A\star 1$  has no non-trivial automorphisms.
\end{prop}
\begin{proof}
Let $\phi:A\star 1\to A\star 1$ be an automorphism. The morphism $\phi$ then induces a bijection on the elements of the basis of $A\star 1$.

 As the element $\emptyset\star 1\in (A\star 1)_0$ is the only element of the basis such that for all $v\in (A\star 1)_1$  $\partial_0^-(v)\neq \emptyset\star 1$, it is preserved by $\phi$. As a consequence, for any element $x$ of the basis of $A_0$, $\phi(x\star \emptyset)$ is of shape $x'\star \emptyset$. The morphism $\phi$ then preserves $(A\star \emptyset)_0$.

Now, remark that for any element $e\in (A\star 1)^*_{n+1}$, there exists $x\in A^*_n$ such that $x\star 1\leq e$ if and only if there exists $y\in A^*_{n-1}$ such that $y\star 1\leq \partial^+e$. By a direct induction, this implies that there exists $x\in (A\star 1)^*_n$ such that $x\star 1\leq e$ if and only if $\partial^+_0e\in \Zb[\emptyset\star 1]$.

Combined with the previous observation, this implies that for any element $x$ of the basis of $A_{n}$, $\phi(x\star \emptyset)$ is of shape $x'\star \emptyset$.
The automorphism $\phi$ then induces by restriction an automorphism $\phi_{|A\star\emptyset}:A\to A$, and the hypothesis implies that it is the identity.

We now show by induction on $n$ that $\phi_n:(A\star 1)_n\to (A\star 1)_n$ is the identity. Suppose the result true at the stage $n$. For any element $x$ of the basis of $A_{n}$, we then have 
$$\partial \phi(x\star 1) = \phi(\partial (x\star 1)) = \partial (x\star 1).$$
By the definition of the derivative of $A\star 1$, and as $\phi$ preserves the basis, this forces the equality $\phi(x\star 1)=x\star 1$. As we already know that for any element $x$ of the basis of $A_{n+1}$ we have $\phi(x\star \emptyset)=x\star \emptyset$, this concludes the induction.

We then have $\phi=id$ and $A\star 1$ has no non trivial automorphisms.

\end{proof}

\begin{definition}
We define the \snotionsym{suspension}{((d60@$[\uvar,1]$}{for augmented directed complexes} as the functor 
$$[\uvar,1]:\CDA\to \CDA$$
where $[K,1]$ is defined as the following pushout:
\begin{equation}
\label{eq:def of suspension cda}
\begin{tikzcd}
	{K\otimes \{0,1\}} & {K\otimes [1]} \\
	{1\coprod 1} & {[K,1]}
	\arrow[from=1-1, to=2-1]
	\arrow[from=2-1, to=2-2]
	\arrow[from=1-1, to=1-2]
	\arrow[from=1-2, to=2-2]
	\arrow["\lrcorner"{anchor=center, pos=0.125, rotate=180}, draw=none, from=2-2, to=1-1]
\end{tikzcd}
\end{equation}
We leave to the reader to check that $[K,1]$ admits a loop free and unitary basis when this is the case for $K$. This functor then induces a functor:
$$[\uvar,1]:\CDAB\to \CDAB$$
\end{definition}

\begin{remark}
 Unfolding the definition, we have
$$[(K,K',e),1]:=([K,1] ,([K,1])^*,e)$$
where
\begin{enumerate}
\item[$-$] $[K,1]$ is the chain complex whose value on $n$ is:
$$[K,1]:=\left\{
\begin{array}{ll}
 \Zb[\{0\},\{1\}]&\mbox{if $n=0$}\\
 \{[x,1],x\in K_{n-1}\} &\mbox{if $n>0$}
\end{array}\right.$$
and the differential is the unique graded group morphism fulfilling: 
$$\partial([x,1]):= \left\{
 \begin{array}{lll} 
 \{1\}-\{0\}&\mbox{if $|x|=0$}\\ 
 ~[\partial x,1]&\mbox{if $|x|>0$}
 \end{array}\right.
$$
\item[$-$] $([K,1])^*$ is given on all integer $n$ by:
$$([K,1])^*_n:=\left\{
\begin{array}{ll}
\Nb[0,1]&\mbox{if $n=0$}\\
 \{[x,1],x\in K^*_{n-1}\} &\mbox{if $n>0$}
\end{array}\right.$$
\item[$-$] $e:([K,1])_0\to \Zb$ is the unique morphism	 fulfilling 
$$e( 0)=e( 1)= e(x).$$
\end{enumerate}
The basis of $[K,1]$ is given by the reunion of $\{0\}$, $\{1\}$ and of the set of elements of shape $[b,1]$  where $b$ is an element of the basis of $K$. 
\end{remark}

\begin{prop}
\label{prop:non trivial automorphisme 2}
Let $A$ be a non null augmented directed complex admitting no non-trivial automorphisms. Then the augmented directed complex $[A,1]$ has no non-trivial automorphisms.
\end{prop}
\begin{proof}
Let $\phi:[A,1]\to [A,1]$ be an automorphism. As the element $\{1\}\in ([A,1])_0$ is the only element of the basis such that for all $v\in [A,1]_1$  $\partial_0^-(v)\neq \{1\}$, it is preserved by $\phi$. As a consequence, $\phi$ also preserves $\{0\}$. The induced morphism $\phi_0:[A,1]_0\to [A,1]_0$ is then the identity. 

Now, remark that $(\phi_{n+1})_{n\in \Nb}:A\to A$ is an automorphism and is then the identity. This implies that for all $n>0$, $\phi_n:[A,1]_n\to [A,1]_n$ is then identity, which concludes the proof.
\end{proof}

\begin{definition}
We define the \textit{wedges} as the functors
$$[\uvar,1]\vee[1]:\CDA\to \CDA~~~~~~ [1]\vee[\uvar,1]:\CDA\to \CDA$$
where $[K,1]\vee [1]$ and $[1]\vee[K,1]$ are defined as the following pushouts:
\[\begin{tikzcd}
	{\lambda [0]} & {[1]} && {\lambda [0]} & {[K,1]} \\
	{[K,1]} & { [K,1]\vee[1]} && {[1]} & {[1]\vee[K,1]}
	\arrow["{\{1\}}"', from=1-1, to=2-1]
	\arrow[from=2-1, to=2-2]
	\arrow["{\{0\}}", from=1-1, to=1-2]
	\arrow[from=1-2, to=2-2]
	\arrow["\lrcorner"{anchor=center, pos=0.125, rotate=180}, draw=none, from=2-2, to=1-1]
	\arrow["{\{0\}}", from=1-4, to=1-5]
	\arrow["{\{1\}}"', from=1-4, to=2-4]
	\arrow[from=1-5, to=2-5]
	\arrow[from=2-4, to=2-5]
	\arrow["\lrcorner"{anchor=center, pos=0.125, rotate=180}, draw=none, from=2-5, to=1-4]
\end{tikzcd}\]
Once again, we can easily check that $[K,1]\vee[1]$ and $[1]\vee[K,1]$ have a loop free and unitary basis when this is the case for $K$. These functors then induce functors
$$[\uvar,1]\vee[1]:\CDAB\to \CDAB~~~~~~ [1]\vee[\uvar,1]:\CDAB\to \CDAB$$
\end{definition}

\vspace{1cm}

 Unfolding the definition, we have
$$[(K,K',e),1]\vee [1]:=([K,1]\vee [1] ,([K,1]\vee [1])^*,e)$$ $$
[1]\vee(K,K',e),1]:=([1]\vee[K,1] ,([1]\vee[K,1])^*,e)$$
where
\begin{enumerate}
\item[$-$] $[K,1]\vee [1]$ and $[1]\vee[K,1]$ are the chain complexes whose value on $n$ are:
$$[K,1]\vee[1]:=\left\{
\begin{array}{ll}
\Zb[\{0\},\{1\},\{2\}]&\mbox{if $n=0$}\\
 \{[x,1],x\in K_{0}\}\oplus \Zb[e_1] &\mbox{if $n=1$}\\
 \{[x,1],x\in K_{n-1}\} &\mbox{if $n>1$}
\end{array}\right.$$
$$[1]\vee[K,1]:=\left\{
\begin{array}{ll}
\Zb[\{0\},\{1\},\{2\}]&\mbox{if $n=0$}\\
\Zb[e_1] \oplus \{[x,1],x\in K_{0}\} &\mbox{if $n=1$}\\
 \{[x,1],x\in K_{n-1}\} &\mbox{if $n>1$}
\end{array}\right.$$
and the differentials are the unique graded group morphism fulfilling: 
$$\partial_{[K,1]\vee[1]} (e_1):= \{2\}-\{1\}
~~~
\partial_{[K,1]\vee[1]} ([x,1]):=
\left\{
\begin{array}{ll}
 \{1\}-\{0\}&\mbox{if $|x|=0$}\\
 ~[\partial x,1]&\mbox{if $|x|>0$}\\
\end{array}\right.
$$
$$
\partial_{[1]\vee[K,1]} (e_1):= \{1\}-\{0\}
~~~
\partial_{[1]\vee[K,1]} ([x,1]):=
\left\{
\begin{array}{ll}
 \{2\}-\{1\}&\mbox{if $|x|=0$}\\
~ [\partial x,1]&\mbox{if $|x|>0$}\\
\end{array}\right.
$$
\item[$-$] $([K,1]\vee [1])^*$ and $([1]\vee[K,1])^*$ are given on all integer $n$ by:
$$([K,1]\vee[1])^*:=\left\{
\begin{array}{ll}
\{\{0\},\{1\},\{2\}\}&\mbox{if $n=0$}\\
 \{[x,1],x\in K_0^*\}\oplus \Nb[e_1] &\mbox{if $n=1$}\\
 \{[x,1],x\in K_{n-1}\} &\mbox{if $n>1$}
\end{array}\right.$$
$$([1]\vee[K,1])^*:=\left\{
\begin{array}{ll}
\{\{0\},\{1\},\{2\}\}&\mbox{if $n=0$}\\
\Nb[e_1]\oplus\cup \{[x,1],x\in K^*_{0}\} &\mbox{if $n=1$}\\
 \{[x,1],x\in K^*_{n-1}\} &\mbox{if $n>1$}
\end{array}\right.$$
\item[$-$] The augmentations $e$ are the unique morphism fulfilling 
$$e( \{0\})=e(\{ 1\})= e(\{2\})=1.$$
\end{enumerate}

\begin{definition}
There are two canonical morphisms 
$$\triangledown:\Sigma K\to \Sigma K \vee [1]
~~~~~~~ \triangledown:\Sigma K\to [1]\vee \Sigma K $$
that are the unique ones fulfilling
$$\triangledown(\{0\}):= \{0\}~~~\triangledown(\{1\}):= \{2\}~~~
\triangledown([x,1]):=\left\{ 
\begin{array}{ll}
~[x,1]+e_1&\mbox{if $|x|=0$}\\
~[x,1]&\mbox{if $|x|>0$}\\
\end{array}\right.$$
When we write $ \Sigma K\to \Sigma K \vee [1]$ and $\Sigma K\to [1]\vee \Sigma K$ and nothing more is specified, it will always mean that we considered the morphisms $\triangledown$.
\end{definition}

\begin{prop}
 \label{prop:appendice formula for otimes cda}
 Let $K$ be an augmented directed complex. 
 There is a natural transformation between the colimit of the following diagram
$$
\begin{tikzcd}
	{[1]\vee [K,1]} & {[K\otimes\{0\},1]} & {[K\otimes [1],1]} & {[K\otimes\{1\},1]} & {[K,1]\vee [1]}
	\arrow[from=1-2, to=1-1]
	\arrow[from=1-2, to=1-3]
	\arrow[from=1-4, to=1-3]
	\arrow[from=1-4, to=1-5]
\end{tikzcd}$$
and $[K,1]\otimes [1]$.
\end{prop}
\begin{proof}
The cone is induced by morphisms
$$
\begin{array}{rl}
&[1]\vee [K,1]\to [K,1]\otimes [1]\\
(\mbox{resp}.&[ K,1]\vee[1]\to [ K,1] \otimes [1])
\end{array}
$$ sending an element $x$ in the basis of $[1]$ to $\{0\}\otimes x$ (resp. $\{1\}\otimes x$), an element $y$ in the basis of $[ K,1]$ to $y\otimes\{1\}$ (resp. $y\otimes\{0\}$), 
and by the morphism 
$$f:[K\otimes [1],1]\to [K,1]\otimes [1]$$
defined by the formula 
$$f([x\otimes y,1]):= [ x,1]\otimes y$$ 
for $x$ in the basis of $K$ and $y$ in the basis of $[1]$.
We leave it to the reader to check the compatibilities of this three morphisms.
\end{proof}

\subsection{Gray operations on $\zo$-categories}
\label{section:definition of Gray operations}
We follow Ara-Maltsiniotis \cite{Ara_Maltsiniotis_joint_et_tranche} for the definitions and first properties of Gray operations on $\zo$-categories. Originally, these authors work with $\omega$-categories, and not with $\zo$-categories. However, this modification does not affect proof, and we then allow ourselves to use their results in our framework.

\begin{theorem}[Steiner, Ara-Maltsiniotis]
\label{theo:otimes in zocat}
There is a unique colimit preserving monoidal structure on $\zocat$,
up to a unique monoidal isomorphism, making the functor
$\nu_{|\CDAB}:\CDAB\to \zocat$
a monoidal functor, when $\CDAB$ is endowed with the monoidal structure given by the Gray tensor product.
\end{theorem}
\begin{proof}
This is \cite[theorem A.15]{Ara_Maltsiniotis_joint_et_tranche}.
\end{proof}

\begin{definition}
The monoidal product on $\zocat$ induced by the previous theorem is called the \snotionsym{Gray tensor product}{((d00@$\otimes$}{for $\zo$-categories} and is denoted by $\otimes$. It's unit is $ \Db_0$. If $C$ and $D$ are $\zo$-categories with an atomic and loop free basis, we have by construction
$$C\otimes D := \nu(\lambda C\otimes \lambda D).$$
\end{definition}

\begin{prop}
\label{prop:otimes and duality}
There are equivalences
$$(C\otimes D)^{op}\cong D^{op}\otimes C^{op}~~~~~~ (C\otimes D)^\circ\cong C^{\circ}\otimes D^{\circ}~~~~~~(C\otimes D)^{co}\cong D^{co}\otimes C^{co}$$
natural in $C,D:\zocat$.
\end{prop}
\begin{proof}
This is \cite[proposition A.20]{Ara_Maltsiniotis_joint_et_tranche}.
\end{proof}

\begin{definition}
\label{defi:of gray cylinder for strict}
The functor
$$\uvar\otimes[1]:\zocat\to \zocat$$
is called the \snotionsym{Gray cylinder}{((d30@$\uvar\otimes[1]$}{for $\zo$-categories}.
\end{definition}

\begin{prop}
\label{prop:comparaison betwen otimes and suspension}
Let $C$ be an $\io$-category.
The following canonical square 
\[\begin{tikzcd}
	{C\otimes\{0,1\}} & {C\otimes[1]} \\
	{1\coprod 1} & {[C,1]}
	\arrow[from=1-1, to=2-1]
	\arrow[from=1-1, to=1-2]
	\arrow[from=2-1, to=2-2]
	\arrow[from=1-2, to=2-2]
	\arrow["\lrcorner"{anchor=center, pos=0.125, rotate=180}, draw=none, from=2-2, to=1-1]
\end{tikzcd}\]
is cocartesian
\end{prop}
\begin{proof}
As all these functors commute with colimits, it is sufficient to demonstrate this assertion when $C$ is a globular sum, and \textit{a fortiori} when $C$ admits a loop free and atomic basis. In this case, remark that all the morphisms appearing in canonical cartesian square
\[\begin{tikzcd}
	{\lambda C\otimes\{0,1\}} & {\lambda C\otimes[1]} \\
	{1\coprod 1} & {[\lambda C,1]}
	\arrow[from=1-1, to=2-1]
	\arrow[from=1-1, to=1-2]
	\arrow[from=2-1, to=2-2]
	\arrow[from=1-2, to=2-2]
	\arrow["\lrcorner"{anchor=center, pos=0.125, rotate=180}, draw=none, from=2-2, to=1-1]
\end{tikzcd}\]
 are quasi-rigid. 
The results then follow from an application of theorem \ref{theo:Kan condition}.
\end{proof}

\begin{remark}
\label{defi:explicit Dbn otiomes [1]}
Applying the duality $(\uvar)^{op}$ to the computation achieved in appendix B.1 of \cite{Ara_Maltsiniotis_joint_et_tranche}, we can give an explicit expression of $\Db_n\otimes [ 1]$. As a polygraph, the generating arrows of $\Db_n\otimes [1]$ are:
$$ e^\epsilon_k\otimes\{0\}~~~~~e^\epsilon_k\otimes\{1\}~~~~~e^\epsilon_k\otimes[1]$$
 \[ a^-_0 \otimes e^\epsilon_k \qquad a^+_0 \otimes e^\epsilon_k \qquad a \otimes e^\epsilon_k \]
 where $\epsilon$ is either $+$ or $-$, $k \leqslant n$ and $e^+_n = e^-_n$. Their source and target are given as follows:
 \[ \pi^-( e^\epsilon_k \otimes\{0\}) = e^-_{k-1} \otimes\{0\} \qquad\qquad\qquad \pi^+(e^\epsilon_k \otimes\{0\}) = e^+_{k-1}\otimes\{0\} \]
 \[ \pi^-(e^\epsilon_k \otimes\{1\} ) = e^-_{k-1} \otimes\{1\}\qquad\qquad\qquad \pi^+(e^\epsilon_k\otimes\{1\} ) = e^+_{k-1}\otimes\{1\} \]
$$\pi^{-}(e^\epsilon_{2k}\otimes[1]) =...\circ_2(e^+_0\otimes[1])\circ_0(e^\epsilon_{2k}\otimes\{0\})\circ_1 (e^-_1\otimes[1])\circ_3... \circ_{2k-1}(e_{2k-1}^-\otimes[1])$$
$$\pi^{+}(e^\epsilon_{2k}\otimes[1]) = (e_{2k-1}^+\otimes[1])\circ_{2k-1}...\circ_3(e^+_1\otimes[1])\circ_1(e^\epsilon_{2k}\otimes\{1\})\circ_0 (e^-_0\otimes[1])\circ_2...$$
$$\pi^{-}(e^\epsilon_{2k+1}\otimes[1]) = ...\circ_3(e^+_1\otimes[1])\circ_1(e^\epsilon_{2k+1}\otimes\{1\})\circ_0 (e^-_0\otimes[1])\circ_2...\circ_{2k}(e_{2k}^-\otimes[1])$$
$$\pi^{+}(e^\epsilon_{2k+1}\otimes[1]) = (e_{2k}^+\otimes[1])\circ_{2k}...\circ_2(e^+_0\otimes[1])\circ_0(e^\epsilon_{2k+1}\otimes\{0\})\circ_1 (e^-_1\otimes[1])\circ_3...$$
 We did not put parenthesis in the expression above, to keep them shorter, the default convention is to do the composition $\circ_i$ in order of increasing values of $i$.
\end{remark}
 
\begin{example}
The $\zo$-category $\Db_1\otimes[1]$ is the polygraph: 
\[\begin{tikzcd}
	00 & 01 \\
	10 & 11
	\arrow[from=1-1, to=2-1]
	\arrow[from=2-1, to=2-2]
	\arrow[from=1-1, to=1-2]
	\arrow[from=1-2, to=2-2]
	\arrow[shorten <=4pt, shorten >=4pt, Rightarrow, from=1-2, to=2-1]
\end{tikzcd}\]
The $\zo$-category $\Db_2\otimes[1]$ is the polygraph: 
\[\begin{tikzcd}
	00 & 01 & 00 & 01 \\
	10 & 11 & 10 & 11
	\arrow[from=1-1, to=1-2]
	\arrow[""{name=0, anchor=center, inner sep=0}, from=1-1, to=2-1]
	\arrow[from=2-1, to=2-2]
	\arrow[""{name=1, anchor=center, inner sep=0}, from=1-2, to=2-2]
	\arrow[shorten <=4pt, shorten >=4pt, Rightarrow, from=1-2, to=2-1]
	\arrow[""{name=2, anchor=center, inner sep=0}, from=1-3, to=2-3]
	\arrow[from=1-3, to=1-4]
	\arrow[""{name=3, anchor=center, inner sep=0}, from=1-4, to=2-4]
	\arrow[shorten <=4pt, shorten >=4pt, Rightarrow, from=1-4, to=2-3]
	\arrow[""{name=4, anchor=center, inner sep=0}, curve={height=30pt}, from=1-1, to=2-1]
	\arrow[from=2-3, to=2-4]
	\arrow[""{name=5, anchor=center, inner sep=0}, curve={height=-30pt}, from=1-4, to=2-4]
	\arrow["{ }"', shorten <=6pt, shorten >=6pt, Rightarrow, from=0, to=4]
	\arrow["{ }"', shorten <=6pt, shorten >=6pt, Rightarrow, from=5, to=3]
	\arrow[shift left=0.7, shorten <=6pt, shorten >=8pt, no head, from=1, to=2]
	\arrow[shift right=0.7, shorten <=6pt, shorten >=8pt, no head, from=1, to=2]
	\arrow[shorten <=6pt, shorten >=6pt, from=1, to=2]
\end{tikzcd}\]
\end{example}

\begin{construction}
\label{cons:Gray cone for omega cat}
 We define the \snotionsym{Gray cone}{((d40@$\uvar\star 1$}{for $\zo$-categories} and  the \snotion{Gray $\circ$-cone}{for $\zo$-categories}\index[notation]{((d50@$1\overset{co}{\star}\_$!\textit{for $\zo$-categories}}:
$$\begin{array}{ccccccc}
\zocat &\to&\zocat_{\cdot}&&\zocat &\to&\zocat_{\cdot}\\
C&\mapsto &C\star 1 & &C &\mapsto &1\costar C
\end{array}
$$
where $C\star 1$, $1\costar C$ and $1\star C$ are defined as the following pushouts: 
\[\begin{tikzcd}
	{C\otimes\{1\}} & {C\otimes [1]} & {C\otimes\{0\}} & {C\otimes [1]} \\
	1 & {C\star 1} & 1 & {1\costar C}
	\arrow[from=1-1, to=1-2]
	\arrow[from=1-1, to=2-1]
	\arrow[from=1-2, to=2-2]
	\arrow[from=1-3, to=1-4]
	\arrow[from=1-3, to=2-3]
	\arrow[from=1-4, to=2-4]
	\arrow[from=2-1, to=2-2]
	\arrow["\lrcorner"{anchor=center, pos=0.125, rotate=180}, draw=none, from=2-2, to=1-1]
	\arrow[from=2-3, to=2-4]
	\arrow["\lrcorner"{anchor=center, pos=0.125, rotate=180}, draw=none, from=2-4, to=1-3]
\end{tikzcd}\]
\end{construction}

\begin{prop}
\label{prop:star and duality}
There is an equivalence
$$(C\star 1)^{\circ}\cong 1\costar C^{\circ} $$
natural in $C:\zocat$.
\end{prop}
\begin{proof}
This directly follows from the definition of these operations and from proposition \ref{prop:otimes and duality}. 
\end{proof}

\begin{example}
The $\zo$-categories $\Db_1\star 1$ and $1\costar \Db_1$ correspond respectively to the polygraphs: 
\[\begin{tikzcd}
	0 &&&& 0 \\
	1 & \star && \star & 1
	\arrow[from=1-1, to=2-1]
	\arrow[from=2-1, to=2-2]
	\arrow[""{name=0, anchor=center, inner sep=0}, from=1-1, to=2-2]
	\arrow[""{name=1, anchor=center, inner sep=0}, from=1-5, to=2-5]
	\arrow[from=2-4, to=1-5]
	\arrow[""{name=2, anchor=center, inner sep=0}, from=2-4, to=2-5]
	\arrow[shorten <=2pt, Rightarrow, from=0, to=2-1]
	\arrow[shift right=2, shorten <=4pt, shorten >=4pt, Rightarrow, from=1, to=2]
\end{tikzcd}\]
The $\zo$-categories $\Db_2\star 1$ and $1\costar \Db_2$ correspond respectively to the polygraphs: 
\[\begin{tikzcd}
	0 & {~} & 0 &&& 0 & {~} & 0 \\
	1 & \star & 1 & \star & \star & 1 & \star & 1
	\arrow[""{name=0, anchor=center, inner sep=0}, from=1-1, to=2-1]
	\arrow[from=2-1, to=2-2]
	\arrow[""{name=1, anchor=center, inner sep=0}, from=1-3, to=2-3]
	\arrow[""{name=2, anchor=center, inner sep=0}, curve={height=30pt}, from=1-1, to=2-1]
	\arrow[from=2-3, to=2-4]
	\arrow[""{name=3, anchor=center, inner sep=0}, from=1-1, to=2-2]
	\arrow[""{name=4, anchor=center, inner sep=0}, draw=none, from=1-2, to=2-2]
	\arrow[""{name=5, anchor=center, inner sep=0}, from=1-3, to=2-4]
	\arrow[from=1-6, to=2-5]
	\arrow[""{name=6, anchor=center, inner sep=0}, from=1-6, to=2-6]
	\arrow[""{name=7, anchor=center, inner sep=0}, from=2-5, to=2-6]
	\arrow[from=1-8, to=2-7]
	\arrow[""{name=8, anchor=center, inner sep=0}, from=1-8, to=2-8]
	\arrow[""{name=9, anchor=center, inner sep=0}, from=2-8, to=2-7]
	\arrow[""{name=10, anchor=center, inner sep=0}, curve={height=-30pt}, from=1-8, to=2-8]
	\arrow[""{name=11, anchor=center, inner sep=0}, draw=none, from=1-7, to=2-7]
	\arrow["{ }"', shorten <=6pt, shorten >=6pt, Rightarrow, from=0, to=2]
	\arrow[shorten <=2pt, shorten >=2pt, Rightarrow, from=3, to=2-1]
	\arrow[shift left=0.7, shorten <=6pt, shorten >=8pt, no head, from=4, to=1]
	\arrow[shift right=0.7, shorten <=6pt, shorten >=8pt, no head, from=4, to=1]
	\arrow[shorten <=6pt, shorten >=6pt, from=4, to=1]
	\arrow[shorten <=2pt, Rightarrow, from=5, to=2-3]
	\arrow[shorten <=6pt, shorten >=6pt, Rightarrow, from=10, to=8]
	\arrow[shift right=2, shorten <=4pt, shorten >=4pt, Rightarrow, from=8, to=9]
	\arrow[shift right=2, shorten <=4pt, shorten >=4pt, Rightarrow, from=6, to=7]
	\arrow[shift right=0.7, shorten <=6pt, shorten >=8pt, no head, from=6, to=11]
	\arrow[shorten <=6pt, shorten >=6pt, from=6, to=11]
	\arrow[shift left=0.7, shorten <=6pt, shorten >=8pt, no head, from=6, to=11]
\end{tikzcd}\]
\end{example}

\begin{prop}
Let $C$ be an $\zo$-category with an unitary and loop free basis. The canonical comparaison
$$ (\lambda C)\star 1\to \lambda (C\star 1) $$
is an equivalence.

Let $K$ be an augmented directed complex with a loop free and unitary basis. The canonical comparaisons 
$$(\nu K)\star 1\to \nu(K\star 1)$$
is an equivalence.
\end{prop}
\begin{proof}
The first assertion directly follows from the fact $\lambda$ commutes with colimits. For the second one,
we can easily check that all the morphisms appearing in the squares \eqref{eq:defin of cstar costar CDA} are quasi-rigid.
The results then follow from an application of theorem \ref{theo:Kan condition}.
\end{proof}

The following propositions express the link between the Gray operations and the suspension. They will play a fundamental role in the rest of this work.
\begin{prop}
 \label{prop:appendice formula for otimes} 
 Let $C$ be an $\zo$-category.
There is a natural identification between $[ C,1]\otimes [1]$ and the colimit of the following diagram
$$
\begin{tikzcd}
	{[1]\vee [ C,1]} & {[C\otimes\{0\},1]} & {[C\otimes [1],1]} & {[C\otimes\{1\},1]} & {[C,1]\vee[1]}
	\arrow[from=1-2, to=1-1]
	\arrow[from=1-2, to=1-3]
	\arrow[from=1-4, to=1-3]
	\arrow[from=1-4, to=1-5]
\end{tikzcd}$$
\end{prop}
\begin{proof}
As all these functors preserve colimits, it is sufficient to construct the comparison when $C$ is a globular sum, and to show that it is an equivalence when $C$ is a globe. 
As globular sums have atomic and loop free bases, the comparison is induced by proposition \ref{prop:appendice formula for otimes cda}. Using the explicit description of the $\zo$-category $\Db_n\otimes[1]$ given in definition \ref{defi:explicit Dbn otiomes [1]}, it is straightforward to see that it induces an equivalence on globes.
\end{proof}

\begin{prop}
 \label{prop:appendice formula for star} 
There is a natural identification between $1\costar [C,1]$ and the colimit of the following diagram
\[\begin{tikzcd}
	{[1]\vee [C,1]} & {[C,1]} & {[C\star 1,1]}
	\arrow[from=1-2, to=1-3]
	\arrow[from=1-2, to=1-1]
\end{tikzcd}\]
There is a natural identification between $[C,1]\star 1$ and the colimit of the following diagram
\[\begin{tikzcd}
	{[1\costar C,1]} & {[C,1]} & {[C,1]\vee[1]}
	\arrow[from=1-2, to=1-3]
	\arrow[from=1-2, to=1-1]
\end{tikzcd}\]
\end{prop}
\begin{proof}
This directly follows from the definition of these operations, from proposition \ref{prop:appendice formula for otimes}  and from proposition \ref{prop:star and duality}.
\end{proof}

\begin{prop}
\label{prop:cartesian squares}
Let $C$ be an $\zo$-category with an atomic and loop free basis. The two following canonical squares are cartesian:
\[\begin{tikzcd}
	1 & {1\costar C} & 1 & {C\star 1} \\
	{\{0\}} & {[C,1]} & {\{1\}} & {[C,1]}
	\arrow[from=1-1, to=1-2]
	\arrow[from=2-1, to=2-2]
	\arrow[from=1-1, to=2-1]
	\arrow[from=1-2, to=2-2]
	\arrow[from=1-3, to=1-4]
	\arrow[from=2-3, to=2-4]
	\arrow[from=1-3, to=2-3]
	\arrow[from=1-4, to=2-4]
\end{tikzcd}\]
The five squares appearing in the following canonical diagram are both cartesian and cocartesian:
\[\begin{tikzcd}
	& {C\otimes\{0\}} & 1 \\
	{C\otimes\{1\}} & {C\otimes[1]} & {C\star 1} \\
	1 & {1\costar C} & {[C,1]}
	\arrow[from=2-3, to=3-3]
	\arrow[from=3-2, to=3-3]
	\arrow[from=2-2, to=3-2]
	\arrow[from=2-2, to=2-3]
	\arrow[from=1-2, to=1-3]
	\arrow[from=1-3, to=2-3]
	\arrow[from=1-2, to=2-2]
	\arrow[from=2-1, to=2-2]
	\arrow[from=3-1, to=3-2]
	\arrow[from=2-1, to=3-1]
\end{tikzcd}\]
\end{prop}
\begin{proof}
The five squares are cocartesian by construction. 
Since the proofs of the cartesianess of all squares are identical, we will only show the proof for the square
\[\begin{tikzcd}
	{C\otimes[1]} & {C\star 1} \\
	{1\costar C} & {[C,1]}
	\arrow[from=1-2, to=2-2]
	\arrow[from=2-1, to=2-2]
	\arrow[from=1-1, to=2-1]
	\arrow[from=1-1, to=1-2]
\end{tikzcd}\]
To this extend, remark that for any integer $n$, the  following square is cartesian. 
\[\begin{tikzcd}
	{(B_{\lambda C\otimes[1]})_n\cup \{0\}} & {(B_{\lambda C\star 1})_n\cup \{0\}} \\
	{(B_{1\costar \lambda C})_n\cup \{0\}} & {(B_{[\lambda C,1]})_n\cup \{0\}}
	\arrow[from=1-1, to=1-2]
	\arrow[from=1-1, to=2-1]
	\arrow[from=1-2, to=2-2]
	\arrow[from=2-1, to=2-2]
\end{tikzcd}\]
This then implies that the following square in the category $\CDA$ is cartesian. 
\[\begin{tikzcd}
	{\lambda C\otimes[1]} & {\lambda C\star 1} \\
	{1\costar \lambda C} & {[\lambda C,1]}
	\arrow[from=1-1, to=1-2]
	\arrow[from=1-1, to=2-1]
	\arrow[from=1-2, to=2-2]
	\arrow[from=2-1, to=2-2]
\end{tikzcd}\]
As $\nu$ is a right adjoint, it preserves limits,  and as it commutes with Gray operation, this concludes the proof.
\end{proof}

\begin{lemma}
\label{lemma: pullback and sum}
Let $a$, $b$, $c$ and $d$ be four globular sums.
Suppose given a cartesian square:
\[\begin{tikzcd}
	a & b \\
	c & d
	\arrow[from=1-1, to=2-1]
	\arrow[from=2-1, to=2-2]
	\arrow[from=1-1, to=1-2]
	\arrow[from=1-2, to=2-2]
	\arrow["\lrcorner"{anchor=center, pos=0.125}, draw=none, from=1-1, to=2-2]
\end{tikzcd}\]
where the two horizontal morphisms are globular.
The two following squares are cartesian 
\[\begin{tikzcd}
	{b\coprod_aa\star 1} & {b\star 1} & {1\costar a\coprod_a b} & {1\costar b} \\
	{\Sigma c} & {\Sigma d} & {\Sigma c} & {\Sigma d}
	\arrow[from=1-1, to=2-1]
	\arrow[from=2-1, to=2-2]
	\arrow[from=1-1, to=1-2]
	\arrow[from=1-2, to=2-2]
	\arrow["\lrcorner"{anchor=center, pos=0.125}, draw=none, from=1-1, to=2-2]
	\arrow[from=1-3, to=2-3]
	\arrow[from=1-4, to=2-4]
	\arrow[from=2-3, to=2-4]
	\arrow[from=1-3, to=1-4]
	\arrow["\lrcorner"{anchor=center, pos=0.125}, draw=none, from=1-3, to=2-4]
\end{tikzcd}\]
\end{lemma}
\begin{proof}
We show only the cartesianess of the first square, as the cartesianess of the second one follows by applying the duality $(\uvar)^\circ$. A direct computation shows that for any integer $n$, the following square is cartesian
\[\begin{tikzcd}
	{\lambda b\coprod_{\lambda a}\lambda a\star 1} & {\lambda b\star 1} \\
	{\Sigma\lambda c} & {\Sigma \lambda d}
	\arrow[from=1-1, to=2-1]
	\arrow[from=2-1, to=2-2]
	\arrow[from=1-1, to=1-2]
	\arrow[from=1-2, to=2-2]
	\arrow["\lrcorner"{anchor=center, pos=0.125}, draw=none, from=1-1, to=2-2]
\end{tikzcd}\]
To conclude, one has to show that the canonical morphism
$$ \nu(\lambda b)\coprod_{\nu(\lambda a)}\nu(\lambda a\star 1)\to \nu (\lambda b\coprod_{\lambda a}\lambda a\star 1) $$
is an equivalence. 	
As $a\to b$ is globular, all the morphisms of the following cocartesian square are quasi-rigid. 
\[\begin{tikzcd}
	{\lambda a} & {\lambda b} \\
	{\lambda a\star 1 } & {\lambda b\coprod_{\lambda b}\lambda a\star 1 }
	\arrow[from=1-1, to=2-1]
	\arrow[from=1-1, to=1-2]
	\arrow[from=1-2, to=2-2]
	\arrow[from=2-1, to=2-2]
\end{tikzcd}\]
The results then follow from an application of theorem \ref{theo:Kan condition}.
\end{proof}

\vspace{1cm}

We are now willing to show the following  theorem: 
\begin{theorem}
\label{theo:appendince unicity of operation}
Let $F$ be an endofunctor of $\zocat$ such that the induced functor $\zocat\to \zocat_{F(\emptyset)/}$ is colimit preserving and $\psi$ an invertible natural transformation between $F(\Db_n)$ and $G(\Db_n)$ where $G$ is either the Gray cylinder, the Gray cone, the Gray $\circ$-cone or an iterated suspension.

Then, the natural transformation $\psi$ can be uniquely extended to an natural transformation between $F$ and $G$.  Moreover, this natural transformation is unique.
\end{theorem}
The previous theorem implies that the equations given in propositions \ref{prop:appendice formula for otimes} and \ref{prop:appendice formula for star} characterize respectively the Gray cylinder, the Gray cone and the Gray $\circ$-cone.

\begin{lemma}
\label{lemma:sub categgory of Theta}
A sub category $\Theta'$ of $\Theta$, stable by colimit is equal to $\Theta$ iff
\begin{enumerate}
\item for any integer $n$ and $\alpha\in\{-,+1\}$, $i_n^{\alpha}:\Db_n\to \Db_{n+1}$ belongs to $\Theta'$.
\item For any integer $n$, the unit $\Ib_n:\Db_{n+1}\to \Db_n$ belongs to $\Theta'$.
\item For any pair of integers $k<n$, the composition $\triangledown_{k,n}:\Db_n\to \Db_n\coprod_{k}\Db_n$ belongs to $\Theta'$.
\end{enumerate}
\end{lemma}
\begin{proof}
Suppose that $\Theta'$ fulfills these conditions.
As globular morphisms are compositions of pushouts along morphisms of shape $i_n^{\alpha}$, they belong to $\Theta'$.
 As algebraic morphisms are compositions of colimits of morphism of shape $\triangledown_{k,n}$ or $\Ib_n$, they belong to $\Theta'$.
The result then follows from proposition \ref{prop:algebraic ortho to globular} that states that every morphism factors as an algebraic morphism followed by a globular morphism.
\end{proof}

\begin{lemma}
\label{lemma:unit forced}
Let $n$ be an integer, and $G$ be either the Gray cylinder, the Gray cone, the Gray $\circ$-cone or an iterated suspension, and suppose
given a square 
\[\begin{tikzcd}
	{G(\Db_n)} \\
	& {G(\Db_{n+1})} & {G(\Db_n)} \\
	{G(\Db_n)}
	\arrow["f", from=2-2, to=2-3]
	\arrow["{G(i_n^-)}", from=1-1, to=2-2]
	\arrow["{G(i_n^+)}"', from=3-1, to=2-2]
	\arrow["id", curve={height=-18pt}, from=1-1, to=2-3]
	\arrow["id"', curve={height=18pt}, from=3-1, to=2-3]
\end{tikzcd}\]
Then, the morphism $f$ is $G(\Ib_n)$.
\end{lemma}
\begin{proof}
As the proof for any possibilities of $G$ are similar, we will show only the case $G:=\uvar\otimes [1]$.
As for any integer $n$, $\Db_n\otimes[1]$ admits a loop free and atomic basis, we can then show the desired assertion after applying the functor $\lambda$.
Remark first that the assumption implies that $\partial f((e_{n+1}\otimes \{\alpha\})=0$, and so $f((e_{n+1}\otimes \{\alpha\}) =0$. We also have $f(e_{n+1}\otimes[1])=0$ as $\lambda (\Db_n\otimes[1])_{n+2} =0$. This implies that $f$ is equal to $\lambda(G(\Ib_n))$.
\end{proof} 
\begin{lemma}
\label{lemma:comp forced}
Let $k<n$ be two integers, and $G$ be either the Gray cylinder, the Gray cone, the Gray $\circ$-cone or an iterated suspension, and suppose
given a square 
\[\begin{tikzcd}
	{G(\Db_{n-1})} && { G(\Db_{n-1}\coprod_k\Db_{n-1})} \\
	& {G(\Db_n)} && { G(\Db_{n}\coprod_k\Db_n)} \\
	{G(\Db_{n-1})} && { G(\Db_{n-1}\coprod_k\Db_{n-1})}
	\arrow["f", from=2-2, to=2-4]
	\arrow["{G(i_n^-)}"{description}, from=1-1, to=2-2]
	\arrow["{G(i_n^+)}"{description}, from=3-1, to=2-2]
	\arrow["{G(i_n^+)\coprod_k G(i_n^+)}"{description}, from=3-3, to=2-4]
	\arrow["{G(i_n^-)\coprod_k G(i_n^-)}"{description}, from=1-3, to=2-4]
	\arrow["{\triangledown_{n-1,k}}"', from=3-1, to=3-3]
	\arrow["{\triangledown_{n-1,k}}", from=1-1, to=1-3]
\end{tikzcd}\]
where we set $\triangledown_{n,n}:=id$.
Then, the morphism $f$ is $G(\triangledown_{n,k})$.
\end{lemma}
\begin{proof}
As the proof for any possibilities of $G$ are similar, we will show only the case $G:=\uvar\otimes [1]$.
As for any integer $n$, $\Db_n\otimes[1]$ admits a loop free and atomic basis, we can then show the desired assertion after applying the functor $\lambda$. Suppose first that $k<n-1$.
By assumption, we have 
$$
\begin{array}{rcl}
\partial f(e_n\otimes \{\alpha\})&=& \partial (e_n^0\otimes \{\alpha\} +e_n^1\otimes \{\alpha\})\\
\partial f(e_n\otimes [1])&=& \partial (e_n^0\otimes [1]) + \partial (e_n^1\otimes [1]) \\
\end{array}
$$
This forces the equalities
$$
\begin{array}{rcl}
 f(e_n\otimes \{\alpha\})&=& e_n^0\otimes \{\alpha\} +e_n^1\otimes \{\alpha\}\\
 f(e_n\otimes [1])&=& e_n^0\otimes [1] + e_n^1\otimes [1] \\
\end{array}
$$
and $f$ is then equal to $\triangledown_{n,k}\otimes[1]$. The case $k=n-1$ is similar.
\end{proof}

\begin{lemma}
\label{lemma:non trivial automorphisme 4}
The set of elements of $\zocatB$  admitting no non-trivial automorphisms is  stable by the Gray cylinder, the Gray cone, the Gray $\circ$-cone, and by the  iterated suspensions,
and contains globular sums.
\end{lemma}
\begin{proof}
Let $S$ be the set of elements of $\zocatB$ admitting no non-trivial automorphisms.  We first show that $S$ contains globular sums. We proceed by induction, supposing that $S$ contains any globular sum of dimension $k$. Let $[\textbf{a},n]$ be a globular sum of dimension $k+1$, and let $\phi:[\textbf{a},n]\to [\textbf{a},n]$ be an isomorphism. In particular, $\phi$ induces an automorphism on $[n]$, so $\phi{i}=i$ for any $i\leq n$. The automorphism $\phi$ then induces, for all $i<n$, an automorphism $\phi_i:[a_i,1]\cong [a_i,1]$. However, proposition \ref{prop:non trivial automorphisme 2} states that $S$ is stable by suspension and the induction hypothesis then implies that for any $i<n$, $[a_i,1]$ has no non-trivial automorphisms, so $\phi_i$ is the identity. This implies that $\phi$ is also the identity, which concludes the proof.

It remains to show that $S$ is stable under the Gray cylinder, the Gray cone, the Gray $\circ$-cone, and the iterated suspensions. 
Using theorem \ref{theorem:steiner}, it is sufficient to show that $\lambda(S)$ is stable under the Gray cylinder, the Gray cone, the Gray $\circ$-cone, and the iterated suspensions. This directly follows from propositions \ref{prop:non trivial automorphisme 0}, \ref{prop:non trivial automorphisme 1}, \ref{prop:non trivial automorphisme 2}, and from the fact that $S$ is closed under dualities.
\end{proof}

\begin{proof}[Proof of theorem \ref{theo:appendince unicity of operation}]
The lemma \ref{lemma:non trivial automorphisme 4} implies that for any globular sum $a$, $G(a)$ as no non trivial isomorphisms.
As every globular sum is a colimit of globes, $\psi$ uniquely extends to a (\textit{a priori} non-natural) invertible transformation, $\psi:F_{|\Theta}\to G_{|\Theta}$. Let $\Theta'$ be the maximal subcategory of $\Theta$ such that $\psi_{|\Theta'}$ is natural.  The category $\Theta'$ is closed under colimits. The assumption implies that $\Theta'$ fulfills the first condition of lemma \ref{lemma:sub categgory of Theta}. Lemma \ref{lemma:unit forced} implies that it fulfills the second condition, and an easy induction on $(n-k)$ using lemma \ref{lemma:comp forced} implies that it fulfills the last condition. Applying lemma \ref{lemma:sub categgory of Theta}, $\psi:F_{|\Theta}\to G_{|\Theta}$ is then natural. Eventually, we can extend $\psi$ to a natural transformation between $F$ and $G$ by colimits.
\end{proof}

\subsection{Gray tensor product of simplicial sets}

\begin{notation}
We denote  by
\[\begin{tikzcd}
	{\Psh{\Theta}} & \zocat
	\arrow["\Fb", shift left=2, from=1-1, to=1-2]
	\arrow["\iota", shift left=2, from=1-2, to=1-1]
\end{tikzcd}\]
the adjunction between presheaves on $\Theta$ and $\zo$-categories.
\end{notation}

\begin{construction}
\label{construction of the gray tensor on simplicial set}
We define the functor  $\uvar\otimes \uvar :\Psh{\Theta}\times \Psh{\Theta}\to \Psh{\Theta}$, called once again the \textit{Gray tensor product}, as the left Kan extension of the functor 
$$\Theta\times \Theta \xrightarrow{\uvar\otimes \uvar} \zocat\xrightarrow{\iota} \Psh{\Theta}$$
where $\otimes:\Theta\times \Theta\to \zocat$ is the Gray tensor product defined in theorem \ref{theo:otimes in zocat}.

By construction, the functor $\Fb$ preserves the Gray tensor product, and the functor $\iota$ preserves the Gray tensor product of globular sums.
\end{construction}

The aim of this section is to prove the following result:

\begin{theorem}
\label{theo:otimes presserves W}
The functor
$$\uvar\otimes\uvar:\Psh{\Delta}\times \Psh{\Delta}\to \Psh{\Theta_2}$$
sends $\W_1\times \W_1$ onto $\overline{\W_2}$, where $\W_1$ and $\W_2$ are defined in \ref{defi:definition of W}, and $\overline{(\uvar)}$ in \ref{defi:precomplet}.
\end{theorem}

Informally, this result implies that we can define a Gray tensor product for $(\infty,1)$-categories. It is therefore a special case of the main theorem of Campion's paper \cite{campion2023gray}.

\begin{prop}
\label{prop:explicit Gray}
The $\Theta$-set $[1]\otimes[1]$ is the colimit, computed in $\Psh{\Theta}$, of the diagram
\[\begin{tikzcd}
	{[2]} & {[1]} & {[[1],1]} & {[1]} & {[2]}
	\arrow["\triangledown"', from=1-2, to=1-1]
	\arrow["{[d^1,1]}", from=1-2, to=1-3]
	\arrow["{[d^0,1]}"', from=1-4, to=1-3]
	\arrow["\triangledown", from=1-4, to=1-5]
\end{tikzcd}\]
\end{prop}
\begin{proof}
We denote by $P$ the colimit of this diagram. Remark that $\Fb P$ is the $\zo$-category generated by the diagram
\[\begin{tikzcd}
	00 & 01 \\
	10 & 11
	\arrow[from=1-2, to=2-2]
	\arrow[from=2-1, to=2-2]
	\arrow[from=1-1, to=2-1]
	\arrow[from=1-1, to=1-2]
	\arrow[shorten <=4pt, shorten >=4pt, Rightarrow, from=1-2, to=2-1]
\end{tikzcd}\]
and we then have $\Fb P\cong [1]\otimes[1]$.  To conclude the proof, we have to show that $P$ is a $\zo$-category, i.e. that it has the unique right lifting property against $\W$. 

Let $f:[\textbf{a},n]\to P$ (resp. $f:\Sp_{[a,n]}\to p$) be a morphism. If there exists an integer $i<n$ such that  $f(i)=00$ and $f(i+1)=11$, then $f$ uniquely factors through $[[1],1]\to P$. If there exists an integer $i$ such that $f(i)=10$, then $f$ uniquely factors through the left inclusion $[2]\to P$. If there exists an integer $i$ such that $f(i)=01$, then $f$ uniquely factors through the right inclusion $[2]\to P$.
If none of these conditions is satisfied, then $f$ factors through $00$ or $11$.

As $[2]$ and $[[1],1]$ are $\zo$-categories, they have the unique right lifting property against $\W$, and so has $P$.
\end{proof}

\begin{lemma}
\label{lemma:when Gray and pullback work}
Let $C,D, E$ be three $\zo$-categories with loop-free and atomic bases. Let $f:C\to D$ be a morphism such that $f$ sends every generator of $C$ to a cell which is not a unit. The following square is then cartesian:
\[\begin{tikzcd}
	{E\otimes C } & {E\otimes  D} \\
	C & D
	\arrow["f"', from=2-1, to=2-2]
	\arrow[from=1-2, to=2-2]
	\arrow[from=1-1, to=2-1]
	\arrow["{f\otimes E}", from=1-1, to=1-2]
\end{tikzcd}\]
\end{lemma}
\begin{proof}
We can show this result at the level of the corresponding augmented directed complex, where it is an easy computation. 
\end{proof}
\begin{lemma}
\label{lemma:product is a nice colimit}
Let $(a)_{i\leq n}$ be a sequence of elements of $\Theta$. There exists a diagram $F:I\to \Theta_+$ such that the presheaf $a_0\times...\times a_n$ is the colimit of $F$. 
\end{lemma}
\begin{proof}
Let $I$ be the full subcategory of $\Theta_{/a_0\times...\times a_n}$ whose objects are $n$-tuples of morphisms $(j_i:b\to a_i)_{i\leq n}$ such that there exists no morphism $b\to b'$ in $\Theta_{-}$ that factors all the $j_i$. Morphisms are the ones such that $b\to b'$ is in $\Theta_{+}$.

Let $(j_i:b\to a_i)_{i\leq n}$ be any element of $\Theta_{/a_0\times...\times a_n}$. 
We claim that it is sufficient to show that there exists a unique degenerate morphism $g:b\to b'$ that factors the morphisms $b\to a_i$ for all $i<n$, and such that the induced family of morphisms $(b'\to a_i)_{i<n}$ is an element of $I$. 
Indeed, suppose that such an element exists and is unique. Let $(h_i:b''\to a_i)_{i\leq n}$ be an element of $I$ and $b\to b''$ be a morphism between $(j_i)_{i\leq n}$ and  $(h_i)_{i\leq n}$. We can factor $b\to b''$ as a degenerate morphism $p:b\to \tilde{b}$ followed by a monomorphism $l:\tilde{b}\to b''$. Lemma \ref{lemma:i etoile of W is in M 0.5} implies that $(h_il)_{i\leq n}$ is in $I$, and $p:(j_i)_{i\leq n}\to (h_il)_{i\leq n}$ is then equal to $g$. As all the operations are functorial, this provides a right retraction structure to the morphism $\{g\}\to I_{(j_i)_{i\leq n}/}$. This inclusion is then final and as a consequence, the morphism $\alpha: I\to  \Theta_{/a_0\times...\times a_n}$ is final. This will then conclude the proof.

As any infinite sequence of degenerate morphisms is constant at some point, the existence of such element is immediate.
Suppose given two morphisms $b\to b'$, $b\to b''$ fulfilling the previous condition. By proposition \ref{prop:theta is elegan reedy}, the category $\Theta$ is Reedy elegant  and 
the proposition 3.8 of \cite{Bergner_reedy_category_and_the_theta_construction} then implies that there exists a globular sum $\tilde{b}$ and two degenerate morphisms $b'\to \tilde{b}$ and $b''\to \tilde{b}$ such that the induced square
\[\begin{tikzcd}
	b & {b'} \\
	{b''} & {\tilde{b}}
	\arrow[from=1-2, to=2-2]
	\arrow[from=2-1, to=2-2]
	\arrow[from=1-1, to=2-1]
	\arrow[from=1-1, to=1-2]
\end{tikzcd}\]
is cocartesian. The universal property of pushout implies that $b\to \tilde{b}$ also fulfills the previous condition. By definition of $b'$ and $b''$, this implies that they are equal to $\tilde{b}$, and this shows the uniqueness.
\end{proof}

\begin{lemma}
\label{lemma:nice colimit in presheaves and Gray}
Let $C$ be a $\zo$-category such that there exists a diagram $F:I\to \Theta_+$ with $\iota(C)$ being  the colimit of $F$. Let $a$ be an element of $\Theta$. The canonical morphism $a\otimes \iota(C) \to \iota (a\otimes C)$ is an isomorphism.
\end{lemma}
\begin{proof}
The lemma \ref{lemma:when Gray and pullback work} implies that the natural transformation $a\otimes F(i)\to F(i)$ is cartesian. As a consequence, for any $i$, the square
\[\begin{tikzcd}
	{a\otimes F(i)} & {a\otimes (\colim_I F) \cong a\otimes \iota(C)} \\
	{F(i)} & {\colim_I F\cong \iota(C)}
	\arrow[from=1-2, to=2-2]
	\arrow[from=1-1, to=1-2]
	\arrow[from=2-1, to=2-2]
	\arrow[from=1-1, to=2-1]
	\arrow["\lrcorner"{anchor=center, pos=0.125}, draw=none, from=1-1, to=2-2]
\end{tikzcd}\]
is cartesian. 

Now, to show the desired result, we have to demonstrate that the $\Theta$-set $a\otimes\iota(C)$ already has a structure of $\io$-category, i.e. that it is $\W$-local. It is sufficient to show that for all $f:X\to Y$ in $\W$, any square
\[\begin{tikzcd}
	X & { a\otimes \iota(C)} \\
	Y & { \iota(C)}
	\arrow[from=1-2, to=2-2]
	\arrow[from=1-1, to=2-1]
	\arrow[from=1-1, to=1-2]
	\arrow[from=2-1, to=2-2]
\end{tikzcd}\]
admits a unique lift. Indeed, as $\iota(C)$ is an $\zo$-category, it is $\W$-local, and this will imply that $a\otimes \iota(C)$ also is. Suppose then given such a square. As every codomain of morphism of $\W$ is representable, there exists a (not necessarily unique) element $i$ of $I$, such that the bottom morphism factors as $Y\to F(i)\to \iota(C)$. The previous square then factors as
\[\begin{tikzcd}
	X & {a\otimes F(i)} & { a\otimes\iota(C)} \\
	Y & {F(i)} & { \iota(C)}
	\arrow[from=1-3, to=2-3]
	\arrow[from=1-2, to=1-3]
	\arrow[from=2-2, to=2-3]
	\arrow[from=1-2, to=2-2]
	\arrow["\lrcorner"{anchor=center, pos=0.125}, draw=none, from=1-2, to=2-3]
	\arrow[from=2-1, to=2-2]
	\arrow[from=1-1, to=1-2]
	\arrow[from=1-1, to=2-1]
\end{tikzcd}\]
where the right square is a pullback. The middle vertical morphism is $\W$-local because it's domain and codomain are, and this concludes the proof. 
\end{proof}
\begin{lemma}
\label{lemma:times et otimes}
Given $(a)_{i\leq n}$ and $b$ elements of $\Theta$, we have 
$$\iota(b\otimes (a_0\times ...\times a_n\otimes b)\cong b\otimes (a_0\times ...\times a_n)$$
\end{lemma}
\begin{proof}
This is a direct consequence of lemmas \ref{lemma:product is a nice colimit} and \ref{lemma:nice colimit in presheaves and Gray}
\end{proof}

\begin{lemma}
\label{lemma:times et otimes2}
Let $A,B, C$ be three presheaves on $\Theta$. We have a canonical morphism 
$$(A \otimes B)\otimes C\to A\otimes (B\times C)$$
\end{lemma}
\begin{proof}
It is sufficient to demonstrate the result when $A$, $B$ and $D$ are representable. In this case the lemma \ref{lemma:times et otimes} implies that $A\otimes (B\times C)$ is in the image of $\iota$. By adjunction, the desired comparaison morphism is induced by
$$\iota ((A \otimes B)\otimes C)\cong  \iota (A) \otimes \iota(B)\otimes \iota(C) \to \iota (A) \otimes (\iota(B)\times \iota(C))$$
\end{proof}

\begin{lemma}
\label{lemma:technical steiner}
Let $A$, $B$, $C$, $D$, and $E$ be presheaves on $\Theta$, and $k, m, n$ be integers. There exists a natural morphism
$$(\uvar)_A:\Hom([B,m],[C,n]\otimes D)\to \Hom([ B\otimes A,m],[C \otimes A,n]\otimes D)$$
such that for any pair of morphisms $f:[B,m]\to [n]\otimes D$ and $g:[F,k]\to [m]\otimes E$, 
$$\Fb((( f\otimes E)\circ g_B)_A)=\Fb((E\otimes f_A)\circ (g_B)_A)$$
\end{lemma}
\begin{proof}
It is sufficient to describe this morphism when $A,B,C,D$, and $E$ are representable. This allows us the use of Steiner theory to construct this application.
Let $f:[B,m]\to [C,n]\otimes D$ be a morphism. We set  $f_A:[B\otimes A,m]\to [C \otimes A,n]\otimes D$ as the unique morphism of $\zo$-categories such that for every $a\in B_A$, $b\in B_B$, and $m\in B_m$
$$\lambda f_A( [b\otimes a,m]):= \sum_{i\leq n} [c_i\otimes a_i,n_i]\otimes [d_i]$$
where $(c_i,d_i,n_i)$ is the unique sequence of elements of $B_C\times B_D\times B_{[n]}$ such that $\lambda f([b,m])= \sum_{i\leq n} \sum_{i\leq n} [c_i,n_i]\otimes [d_i]$.
The equality $\lambda f_A\partial=\partial\lambda f_A$ and the equality $\Fb(((E\otimes f)\circ g_B)_A)=\Fb((E\otimes f_A)\circ (g_B)_A)$ are straightforward computation using Steiner theory.
\end{proof}

\begin{lemma}
\label{lemma:otimes presserves W1}
Let $n$ and $m$ be two integers. The canonical morphism 
$$\Sp_{[n]}\otimes\Sp_{[m]}\to [n]\otimes[m]$$ 
is in $\overline{\W_2}$.
\end{lemma}
\begin{proof}
We recall that the functors $\tau_0$ and $\tau_1$ are defined in \ref{cons:dumb truncatoin} and correspond, respectively, to the functor that forgets cells of dimension strictly higher than $0$ and to the functor that forgets cells of dimension strictly higher than $1$.
Let $\Delta^{glob}$ be the subcategory of $\Delta$ whose morphisms are the globular ones.
We consider the functor $g:\Delta^{glob}\times \Delta^{glob}\to \Psh{\Theta_2}$ defined by the formula 
$$g([n],[m]):=\tau_0([n]\otimes [m])\cup_{x\in S_{n,m}}x$$
where $S_{n,m}$ is the set of $1$-generators of $\tau_1([n]\otimes[m])$.
We have a canonical transformation $g(n,m)\to \tau_1([n]\otimes[m])$ which is pointwise in $\widehat{\W_2}$ by repeated application of theorem \ref{theo: case of 1 and 2 category}. For any pair of integers $n,m$, the morphism 
$$g([n],[m])\cong \colim_{\Sp_{[n]}\times \Sp_{[m]}}g\to \tau_1(\Sp_{[n]}\otimes\Sp_{[m]})$$ 
then also belongs to $\overline{\W_2}$. 
By two out of three, so is the morphism
$$\tau_1(\Sp_{[n]}\otimes \Sp_{[m]})\to \tau_1([n]\otimes[m])$$
Remark now that we have a cocartesian square
\[\begin{tikzcd}
	{\tau_1(\Sp_{[n]}\otimes\Sp_{[m]})} & {\Sp_{[n]}\otimes\Sp_{[m]}} \\
	{\tau_1([n]\otimes[m])} & {\tau_1([n]\otimes[m])\cup_{x\in T_{n,m}}x}
	\arrow[from=1-1, to=2-1]
	\arrow[from=1-1, to=1-2]
	\arrow[from=2-1, to=2-2]
	\arrow[from=1-2, to=2-2]
\end{tikzcd}\]
where $T_{n,m}$ is the set of $2$-generators of  $[n]\otimes[m]$. The theorem \ref{theo: case of 1 and 2 category} implies that 
$$\tau_1([n]\otimes[m])\cup_{x\in T_{n,m}}x\to [n]\otimes[m]$$
is in $\widehat{\W_2}$, and by stability by composition and pushout, so is 
$$\Sp_{[n]}\otimes\Sp_{[m]}\to [n]\otimes[m].$$
\end{proof}

\begin{prop}
\label{prop:otimes and suspension in presheaves}
Let $K$ be a simplicial set. The canonical morphism
$$1\coprod_{K\otimes\{0\}}K\otimes[1]\coprod_{K\otimes\{1\}}1\to [K,1]$$
is in $\overline{\W_2}$.
\end{prop}
\begin{proof}
As $K$ is a colimit of representables indexed by the Reedy cofibrant diagram $\Delta^+_{/K}\to \Sset$ (definition \ref{defi:reedycof}), and as 
$1\coprod_{\uvar\otimes\{0\}}\uvar\otimes[1]\coprod_{\uvar\otimes\{1\}}1$ and  $[\uvar,1]$ preserve cofibrations,  it is sufficient to demonstrate the result when $K:=[n]$ for $n$ an integer or the empty simplicial set. The cases $K:=\emptyset$ and $K:=[0]$  are trivial. As  $[\uvar,1]$ and, by lemma \ref{lemma:otimes presserves W1}, $\uvar\otimes[1]$ send $\Sp_{[n]}\to [n]$ to $\overline{\W_2}$, it is sufficient to demonstrate the result when $[n]=[1]$. By proposition \ref{prop:explicit Gray}, the morphism $$1\coprod_{[1]\otimes\{0\}}[1]\otimes[1]\coprod_{[1]\otimes\{1\}}1\to [[1],1]$$ fits in the cocartesian square
\[\begin{tikzcd}
	{[0]\coprod_{[1]}[2]~~~\coprod~~~[0]\coprod_{[1]}[2]} & {1\coprod_{[1]\otimes\{0\}}[1]\otimes[1]\coprod_{[1]\otimes\{1\}}1} \\
	{[1]~~~\coprod~~~[1]} & {[[1],1]}
	\arrow[from=2-1, to=2-2]
	\arrow[from=1-2, to=2-2]
	\arrow[""{name=0, anchor=center, inner sep=0}, from=1-1, to=1-2]
	\arrow[from=1-1, to=2-1]
	\arrow["\lrcorner"{anchor=center, pos=0.125, rotate=180}, draw=none, from=2-2, to=0]
\end{tikzcd}\]
As the canonical morphisms $[0]\coprod_{[1]}[2]\to [1]$ and $[2]\coprod_{[1]}[0]\to [1]$ are in $\overline{\W_2}$, this concludes the proof.
\end{proof}

\begin{lemma}
\label{lemma:otimes presserves W2}
Let $n$ be an integer. The two morphisms
$$E^{eq}\otimes[n]\to [n] \quad\text{and}\quad [n]\otimes E^{eq}\to [n]$$
are in $\overline{\W_2}$.
\end{lemma}
\begin{proof}
As $\otimes$ sends spine inclusions to $\overline{\W_2}$, we can reduce to the case where $n=1$. By stability by pushouts along monomorphisms, and using lemma \ref{prop:otimes and suspension in presheaves}, the composite
$$E^{eq}\otimes[1]\to 1\coprod_{E^{eq}\otimes\{0\}}E^{eq}\otimes[1]\coprod_{E^{eq}\otimes\{1\}}1\to [E^{eq},1]$$
is in $\overline{\W_2}$. As $[E^{eq},1]\to [1]$ is in $\W_2$, this concludes the first assertion. We show the second one similarly.
\end{proof}

\begin{proof}[Proof of theorem \ref{theo:otimes presserves W}]
This is the content of lemmas \ref{lemma:otimes presserves W1} and \ref{lemma:otimes presserves W2}.
\end{proof}

%
%
%

\chapter{The $\iun$-category of $\io$-categories}

\minitoc
\vspace{1cm}
%
%
%
%
%
%
%
%
%

The first section is a recollection of results on presentable $\iun$-categories. In particular, we provide a very useful technical lemma that gives conditions for calculating colimits in $\infty$-presheaves using (strict) presheaves. We also present results on factorization systems and the localizations they induce. Finally, we conclude by giving some results on monomorphisms in $\iun$-categories.

In the second section, we define $\io$-categories and give some basic properties. 
We also define and study \textit{discrete Conduché functor}, which are morphisms having the unique right lifting property against 
units $\Ib_{n+1}:\Db_{n+1}\to \Db_n$ for any integer $n$, and against compositions $\triangledown_{k,n}:\Db_n\to \Db_n\coprod_{\Db_k}\Db_n$ for any pair of integers $k\leq n$. This notion was originally defined and studied in the context of strict $\omega$-category by Guetta in \cite{Guetta_conduche}. We then demonstrate the following result:
\begin{itheorem}[\ref{theo:pullback along conduche preserves colimits}]
Let $f:C\to D$ be a discrete Conduché functor. The pullback functor $f^*:\ocat_{/D}\to \ocat_{/C}$ preserves colimits.
\end{itheorem}

In the third section, using the Gray tensor product for $\zo$-categories, we construct a colimit-preserving functor
\[ \uvar\otimes\uvar:\ocat\times \icat\to \ocat \]
again called the Gray tensor product.
To be able to define lax phenomena, in particular lax colimits and limits, we will need to extend the functor $\uvar\otimes\uvar$ to \(\ocat\times \ocat\). Although one might at first be tempted to use the "general" Gray tensor product, this is not the operation that will be used subsequently, particularly for stating the universal property of lax colimits and limits. We then introduce a cocontinuous bifunctor:
\[ \ominus: \ocat\times \ocat\to \ocat \]
called the \textit{enhanced Gray tensor product}. This new operation can be seen as a generalization of the Gray product with $\iun$-categories.
 
 We then introduce the Gray operations, starting with the Gray cylinder $\uvar\otimes[1]$ which is the Gray tensor product with the directed interval $[1]:=0\to 1$. Then, we have the \textit{Gray cone} and the \textit{Gray $\circ$-cone}, denoted by $\uvar\star 1$ and $1\costar \uvar$, that send an $\io$-category $C$ onto the following pushouts:
\[\begin{tikzcd}[cramped]
	{C\otimes\{1\}} & {C\otimes[1]} & {C\otimes\{0\}} & {C\otimes[1]} \\
	1 & {C\star 1} & 1 & {1\costar C}
	\arrow[from=1-1, to=1-2]
	\arrow[from=1-1, to=2-1]
	\arrow[from=1-2, to=2-2]
	\arrow[from=1-3, to=1-4]
	\arrow[from=1-3, to=2-3]
	\arrow[from=1-4, to=2-4]
	\arrow[from=2-1, to=2-2]
	\arrow["\lrcorner"{anchor=center, pos=0.125, rotate=180}, draw=none, from=2-2, to=1-1]
	\arrow[from=2-3, to=2-4]
	\arrow["\lrcorner"{anchor=center, pos=0.125, rotate=180}, draw=none, from=2-4, to=1-3]
\end{tikzcd}\]

We then provide several formulas expressing the relationship between the Gray operations and the suspension $[\uvar,1]:\ocat\to \ocat$, identical to those in the strict case.
\begin{iprop}[\ref{prop:eq for cylinder}]
There is a natural identification between $[C,1]\otimes [1]$ and the colimit of the diagram
\begin{equation}
\begin{tikzcd}
	{[1]\vee [ C,1]} & {[C\otimes\{0\},1]} & {[C\otimes [1],1]} & {[C\otimes\{1\},1]} & {[C,1]\vee[1]}
	\arrow[from=1-2, to=1-1]
	\arrow[from=1-2, to=1-3]
	\arrow[from=1-4, to=1-3]
	\arrow[from=1-4, to=1-5]
\end{tikzcd}
\end{equation}
\end{iprop}

\begin{iprop}[\ref{prop:formula for the ominus}]
There is a  natural identification between $[C,1]\ominus[b,1]$ and the colimit of the following diagram
\begin{equation}
\begin{tikzcd}[column sep = 0.3cm]
	{[b,1]\vee[C,1]} & {[C\otimes\{0\}\times b,1]} & {[(C\otimes[1])\times b,1]} & {[C\otimes\{1\}\times b,1]} & {[C,1]\vee[b,1]}
	\arrow[from=1-2, to=1-3]
	\arrow[from=1-4, to=1-3]
	\arrow[from=1-4, to=1-5]
	\arrow[from=1-2, to=1-1]
\end{tikzcd}
\end{equation}
\end{iprop}

\begin{iprop}[\ref{prop:eq for Gray cone}]
There is a natural identification between $1\costar [C,1]$ and the colimit of the diagram
\begin{equation}
\begin{tikzcd}
	{[1]\vee [C,1]} & {[C,1]} & {[C\star 1,1]}
	\arrow[from=1-2, to=1-3]
	\arrow[from=1-2, to=1-1]
\end{tikzcd}
\end{equation}
There is a natural identification between $[C,1]\star 1$ and the colimit of the diagram
\begin{equation}
\begin{tikzcd}
	{[1\costar C,1]} & {[C,1]} & {[C,1]\vee[1]}
	\arrow[from=1-2, to=1-3]
	\arrow[from=1-2, to=1-1]
\end{tikzcd}
\end{equation}
\end{iprop}

 We conclude this chapter by proving results of strictification. In particular, we demonstrate the following theorem:
\begin{itheorem}[\ref{theo:strict stuff are pushout}]
Let $C\to D$ and $C\to E$ be two morphisms between strict $\io$-categories. The $\io$-categories
$D\coprod_{C}C\star 1$,  $1\costar C\coprod_CD$, and $D\coprod_{C\otimes\{0\}}C\otimes[1]\coprod_{C\otimes \{1\}}E$ are strict. In particular, $C\star 1$, $1\costar C$, and $C\otimes[n]$ for any integer $n$, are strict.
\end{itheorem}
In the process, we will demonstrate another fundamental equation combining $C\otimes[1]$, $1\costar C$, $C\star 1$, and $[C,1]$.
\begin{itheorem}[\ref{theo:formula between pullback of slice and tensor}]
Let $C$ be an $\io$-category. The five squares appearing in the following canonical diagram are both cartesian and cocartesian:
\[\begin{tikzcd}
	& {C\otimes\{0\}} & 1 \\
	{C\otimes\{1\}} & {C\otimes[1]} & {C\star 1} \\
	1 & {1\costar C} & {[C,1]}
	\arrow[from=2-3, to=3-3]
	\arrow[from=3-2, to=3-3]
	\arrow[from=2-2, to=3-2]
	\arrow[from=2-2, to=2-3]
	\arrow[from=1-2, to=1-3]
	\arrow[from=1-3, to=2-3]
	\arrow[from=1-2, to=2-2]
	\arrow[from=2-1, to=2-2]
	\arrow[from=3-1, to=3-2]
	\arrow[from=2-1, to=3-1]
\end{tikzcd}\]
where $[C,1]$ is the \textit{suspension of $C$}.
\end{itheorem}

\paragraph{Cardinality hypothesis.}
We fix during this chapter three Grothendieck universes $\U \in \V\in\Wcard$, such that $\omega\in \U$. 
All defined notions depend on a choice of cardinality. When nothing is specified, this corresponds to the implicit choice of the cardinality $\V$.
With this convention in mind, we denote by {$\Set$} the $\Wcard$-small $1$-category of $\V$-small sets, {$\igrd$} the $\Wcard$-small $\iun$-category of $\V$-small $\infty$-groupoids and {$\icat$} the $\Wcard$-small $\iun$-category of $\V$-small $\iun$-categories.

\section{Preliminaries}
\subsection{Explicit computation of some colimits}

\begin{definition}
For a category $B$, we denote by {$\Psh{B}$} the category of functors $B^{op}\to \Set$.
For a $\iun$-category $A$, we denote by \wcnotation{$\iPsh{A}$}{(psh@$\iPsh{\uvar}$} the $\iun$-category of functors $A^{op}\to \igrd$. A presheaf on $B$, (resp. a $\infty$-presheaves on $A$) is \textit{$\U$-small} if it is pointwise a $\U$-small set (resp. a $\U$-small $\infty$-groupoid).
\end{definition}
\begin{definition}
We have an adjunction:
\begin{equation}
\label{eq:adj betwen set and space}
\begin{tikzcd}
	{\pi_0:\igrd} & {\Set:\iota}
	\arrow[""{name=0, anchor=center, inner sep=0}, shift left=2, from=1-2, to=1-1]
	\arrow[""{name=1, anchor=center, inner sep=0}, shift left=2, from=1-1, to=1-2]
	\arrow["\dashv"{anchor=center, rotate=-90}, draw=none, from=1, to=0]
\end{tikzcd}
\end{equation}
 If $A$ is a $1$-category, the previous adjunction induces an adjunction:
\begin{equation}
\label{eq:adj betwen A set and A space}
\begin{tikzcd}
	{\pi_0:\iPsh{A}} & {\Psh{A}:\iota}
	\arrow[""{name=0, anchor=center, inner sep=0}, shift left=2, from=1-2, to=1-1]
	\arrow[""{name=1, anchor=center, inner sep=0}, shift left=2, from=1-1, to=1-2]
	\arrow["\dashv"{anchor=center, rotate=-90}, draw=none, from=1, to=0]
\end{tikzcd}
\end{equation}
\end{definition}
We recall that the notion of {elegant Reedy category} is defined in definition \ref{defi:reedy}.
The following lemma provides a powerful way to compute simple colimits in $\iun$-categories by reducing to computations in (stricts) categories. These techniques will be used freely in the rest of this text.

\begin{lemma}
\label{lemma:colimit computed in set presheaves}
Let $A$ be a $\V$-small category. We denote $\iota:\Psh{A}\to \iPsh{A}$ the canonical inclusion.
\begin{enumerate}
\item 
The functor $\iota$ preserves cocartesian square 
\[\begin{tikzcd}
	a & b \\
	c & d
	\arrow[from=1-1, to=2-1]
	\arrow[from=1-2, to=2-2]
	\arrow[from=1-1, to=1-2]
	\arrow[from=2-1, to=2-2]
	\arrow["\lrcorner"{anchor=center, pos=0.125, rotate=180}, draw=none, from=2-2, to=1-1]
\end{tikzcd}\]
where the left vertical morphism is a monomorphism.
\item 
The functor $\iota$ preserves colimit of finite diagrams of shape: 
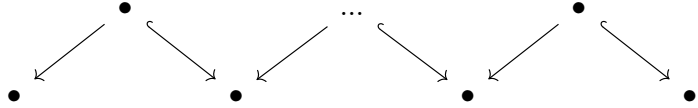
\[\begin{tikzcd}
	& \bullet && {...} && \bullet \\
	\bullet && \bullet && \bullet && \bullet
	\arrow[from=1-2, to=2-1]
	\arrow[hook, from=1-2, to=2-3]
	\arrow[from=1-4, to=2-3]
	\arrow[hook, from=1-4, to=2-5]
	\arrow[from=1-6, to=2-5]
	\arrow[hook, from=1-6, to=2-7]
\end{tikzcd}\]
where morphisms labeled $\hookrightarrow$ are monomorphisms.
\item The functor $\iota$ preserves transfinite composition. 
\item For any $\V$-small elegant Reedy category $I$, and any functor $F:I\to \Psh{A}$ that is Reedy cofibrant, i.e such that for any $i\in I$, $\colim_{\partial i}F\to F(i)$ is a monomorphism,
the canonical comparison 
$$\iota \colim F\to \colim \iota F$$
is an isomorphism. In particular, if $A$ is itself an elegant Reedy category, for any set-valued presheaf $X$ on $A$, there is an equivalence 
$$\iota(X)\sim \colim_{A_{/X}}a.$$ 
\end{enumerate}
\end{lemma}
\begin{proof}
For this result, we use model categories. We consider the interval induces by the constant functor $I:A\to \Psh{\Delta}$ with value $[1]$. We then consider the model structure on $\Psh{A\times \Delta}$ produced by \cite[theorem 1.3.22]{cisinski_prefaisceaux_comme_modele} and induces by the homotopical data $(I\times \uvar,\emptyset)$. This model structure represent $\iPsh{A}$.
To conclude, we then have to show that all the given colimits, seen as (simplicialy constant) presheaves on $\Delta\times A$ are also homotopy colimits of the same diagrams. 
The first three assertions are \cite[proposition 2.3.26 and 2.3.13]{Cisinski_Higher_categories_and_homotopical_algebra}. For the last one, using the characterization of elegant Reedy categories given by \cite[proposition 3.8]{Bergner_reedy_category_and_the_theta_construction} and \cite[proposition 15.10.2]{Hirschhorn_Model_categories_and_their_localizations}, it’s easy to see that they have fibrant constants in the sense of \cite[definition 15.10.1]{Hirschhorn_Model_categories_and_their_localizations}. We can then apply theorem 19.9.1 of \cite{Hirschhorn_Model_categories_and_their_localizations}.
\end{proof}

\begin{definition}
Let $C$ be a $\iun$-category.
A full sub $\infty$-groupoid of the $\infty$-groupoid of arrows of $C$ is \wcnotionsym{cocomplete}{(s@$\widehat{S}$}{cocomplete $\infty$-groupoid of arrows} if it is closed under colimit and composition. For a $\infty$-groupoid $S$, we define $\widehat{S}$ as the smallest cocomplete full sub $\infty$-groupoid of the $\infty$-groupoid of arrows containing $S$. 
\end{definition}

\begin{definition}
Let $C$ be a $\iun$-category.
A  full sub $\infty$-groupoid $U$ of the $\infty$-groupoid of arrows of $C$ is \wcnotion{closed under left cancellation}{closed under left or right cancellation} (resp. \textit{closed under right cancellation}), if for any pair of composable morphisms $f$ and $g$, if $gf$ and $f$ are in $U$, so is $g$ (resp. if $gf$ and $g$ are in $U$, so is $f$).
\end{definition}

\begin{lemma}
\label{lemma:closed under colimit imply saturated_modified}
Let $C$ be an $\iun$-category, and $U$ a cocomplete full sub-$\infty$-groupoid of the $\infty$-groupoid of arrows of $C$. The $\infty$-groupoid $U$ is stable under transfinite composition, pushout, and left cancellation.
\end{lemma}
\begin{proof}
Let $\kappa$ be an cardinal and $x_\uvar:\kappa\to C$ be a functor such that for any $\alpha<\kappa$, $x_\alpha\to x_{\alpha+1}$ is in $U$. We show by transfinite composition that for any $\alpha\leq \kappa$, $x_0\to x_{\alpha}$ is in $U$. 

Suppose first that $\alpha$ is of the form $\beta+1$. The morphism $x_0\to x_{\beta+1}$ is equal to the composite $x_0\to x_\beta\to x_{\beta+1}$ and is then in $U$. If $\alpha$ is a limit ordinal, we have $x_0\to x_{\alpha}$ as the colimit of the sequence of morphisms $\{x_0\to x_{\beta}\}_{\beta<\alpha}$ and is then in $U$.

Suppose now  given a cocartesian square
\[\begin{tikzcd}[row sep=scriptsize]
	a & b \\
	c & d
	\arrow["f"', from=1-1, to=2-1]
	\arrow[from=1-1, to=1-2]
	\arrow["{f'}", from=1-2, to=2-2]
	\arrow[from=2-1, to=2-2]
	\arrow["\lrcorner"{anchor=center, pos=0.125, rotate=180}, draw=none, from=2-2, to=1-1]
\end{tikzcd}\]
with $f$ in $U$. Remark that $f'$ is the horizontal colimit of the diagram
\[\begin{tikzcd}[row sep=scriptsize]
	a & a & b \\
	c & a & b
	\arrow[from=2-2, to=2-1]
	\arrow[from=2-2, to=2-3]
	\arrow["id", from=1-3, to=2-3]
	\arrow["id", from=1-2, to=2-2]
	\arrow[from=1-2, to=1-3]
	\arrow["id"', from=1-2, to=1-1]
	\arrow["f"', from=1-1, to=2-1]
\end{tikzcd}\]
and then is in $U$.

Eventually, suppose given $f:a\to b$, $g:b\to c$ such that $gf$ and $f$ are in $U$. As $g$ is the horizontal colimit of the following diagram
\[\begin{tikzcd}
	b & a & a \\
	b & b & c
	\arrow["{id_b}", from=1-1, to=2-1]
	\arrow[from=1-2, to=1-1]
	\arrow[from=1-2, to=1-3]
	\arrow["f", from=1-2, to=2-2]
	\arrow["gf", from=1-3, to=2-3]
	\arrow[from=2-2, to=2-1]
	\arrow[from=2-2, to=2-3]
\end{tikzcd}\]
it is in $U$.
\end{proof}

\begin{prop}
\label{prop:link beetwenwidehat and overline}
Let $T$ be a set of morphisms of $\Psh{A}$. 
We have an inclusion
$\iota(\overline{T}) \subset \widehat{T}$
where $\overline{T}$ is the smallest precomplete class of morphisms containing $T$ (definition \ref{defi:precomplet}) .
\end{prop}
\begin{proof}
This directly follows from the definition of a precomplete class of morphisms, and from lemmas \ref{lemma:colimit computed in set presheaves} and \ref{lemma:closed under colimit imply saturated_modified}.
\end{proof}

\subsection{Factorization sytems}
\label{section:Factorization system}
 For the rest of the section, we fix a \textit{presentable $\iun$-category} $C$, i.e a $\iun$-category $C$ that is a reflexive and $\V$-accessible localization of a $\iun$-category of $\infty$-presheaves on a $\V$-small $\iun$-category.

 We recall some standard results on factorization systems, which appear in many places in the literature, such as in section 5.5.5 of \cite{Lurie_Htt} for the $\iun$-case and  \cite{Joyal_factorisation} for the strict case.

\begin{definition}
Let $S$ be a $\V$-small $\infty$-groupoid of maps of $C$. We denote by $\Arr_S(C)$ the full sub $\iun$-category of $\Arr(C)$ whose objects correspond to arrows of $S$.
\end{definition}

\begin{definition}
A \notion{weak factorization system in $(L,R)$} is the data of two full sub $\infty$-groupoids $L$ and $R$ of the $\infty$-groupoid of arrows of $C$, stable under composition and containing equivalences, and of a section 
$\Arr_R(C)\to \Arr_L(C)\times_C \Arr_R(C)$ of the functor $ \Arr_L(C)\times_C \Arr_R(C)\to \Arr(C)$ sending two arrows onto their composite.
This is a \wcnotion{factorization system}{factorization system in $(L,R)$} if the functor $\Arr(C)\to \Arr_L(C)\times_C \Arr_R(C)$ is an equivalence. 
\end{definition}

Until the end of this section, we suppose given such factorization system in $(L,R)$.

\begin{definition}
\label{defi:of lift}
Let $i$ and $p$ be two morphisms, and consider a commutative square of shape:
\[\begin{tikzcd}
	a & b \\
	c & d
	\arrow["i"', from=1-1, to=2-1]
	\arrow["p", from=1-2, to=2-2]
	\arrow[from=2-1, to=2-2]
	\arrow[from=1-1, to=1-2]
\end{tikzcd}\]
A \wcnotion{lift}{lift in a square} is a factorization of this square as two commutative triangles:
\[\begin{tikzcd}
	a & b && b \\
	c && c & d
	\arrow[from=1-1, to=1-2]
	\arrow["i"', from=1-1, to=2-1]
	\arrow["h"', from=2-1, to=1-2]
	\arrow["h", from=2-3, to=1-4]
	\arrow[from=2-3, to=2-4]
	\arrow["p", from=1-4, to=2-4]
\end{tikzcd}\]

Equivalently, we can see a square of the previous shape as a morphism $s:1\to \Sq({i,p}):=\Hom(a,b)\times_{\Hom(a,d)}\Hom(c,d)$\sym{(sq@$\Sq(i,p)$} and a lift as the data of a morphism $h:1\to \Hom(c,d)$ and of a commutative triangle
\[\begin{tikzcd}
	& {\Hom(c,b)} \\
	1 & {\Sq(i,p)}
	\arrow["s"', from=2-1, to=2-2]
	\arrow["h", from=2-1, to=1-2]
	\arrow[from=1-2, to=2-2]
\end{tikzcd}\]

The \textit{$\infty$-groupoid of lift of $s$} is the fibers of $\Hom(c,b)\to \Sq(i,p)$ at $s$.
\end{definition}

\begin{definition}
Let $i$ and $p$ be two morphisms. The morphism \wcnotion{$i$ has the unique left lifting property against $p$}{unique left or right lifting property}, or equivalently, \textit{$p$ has the unique right lifting property against $i$}, if for any square $s\in \Sq(i,p)$, the $\infty$-groupoid of lift of $s$ is contractible. This is equivalent to asking for the morphism $\Hom(c,d)\to \Sq(i,p)$ to be an equivalence.
\end{definition}

\begin{lemma}
\label{lemma:when weak factorization system are factoryzation system}
Suppose that we have a weak factorization system in $(L',R')$ such that morphisms in $R'$ have the unique right lifting property against morphisms of $L'$. The weak factorization system is a factorization system.
\end{lemma}
\begin{proof}
Our goal is to demonstrate that the fibers of $\Arr_{L'}(C)\times_C\Arr_{R'}(C)\to \Arr(C)$ are contractible. Let $f$ be a morphism of $C$. As we have a weak factorization system, there exists an element in the fiber at $f$. Suppose given two elements in this fiber. This corresponds to a square
\[\begin{tikzcd}
	\cdot & \cdot \\
	\cdot & \cdot
	\arrow["i"', from=1-1, to=2-1]
	\arrow["p"', from=2-1, to=2-2]
	\arrow["{i'}", from=1-1, to=1-2]
	\arrow["{p'}", from=1-2, to=2-2]
\end{tikzcd}\]
Morphisms between these two factorizations correspond to lifts in the previous square, which are contractible by assumption, and the fiber is then contractible. 
\end{proof}
We recall that in this section, we suppose that we have a factorization system in $(L,R)$.
\begin{lemma}
\label{lemma:caracterisation of L and R with lifting property 1}
Morphisms in $L$ have the unique left lifting property with respect to morphisms in $R$. 
\end{lemma}
\begin{proof}
Let $i:a\to c$ be a morphim of $L$ and $p:b\to d$ a morphism of $R$.
The factorization functor induces an equivalence between squares $s\in \Sq(i,p)$ and diagrams of shape
\[\begin{tikzcd}[row sep=tiny]
	a && b \\
	& e \\
	c && d
	\arrow[from=1-1, to=3-1]
	\arrow[from=1-3, to=3-3]
	\arrow[from=1-1, to=2-2]
	\arrow[from=3-1, to=2-2]
	\arrow[from=2-2, to=3-3]
	\arrow[from=2-2, to=1-3]
\end{tikzcd}\]
where all the morphisms of the left triangle are in $L$ and the ones of the right triangle are in $R$.
Such diagrams are then in equivalence between composite $c\to e\to b$ where the first morphism is in $S$ and the second in $R$. Using once again the factorization functor, we can see that this data is exactly equivalent to a lift in the square $s$.
\end{proof}

We now show the converse of the previous lemma.

\begin{lemma}
\label{lemma:caracterisation of L and R with lifting property 2}
A morphism having the unique left lifting property against morphisms of $R$ is in $L$. Analogously, a morphism having the unique right lifting property against morphisms of $L$ is in $R$.
\end{lemma}
\begin{proof}
Let $f$ be a morphism having the unique left lifting property against morphisms in $R$. We factorize the morphism $f$ in $i\in L$ followed by $p\in R$ and we want to produce an equivalence $f\sim i$. The previous data induces by construction a square
\[\begin{tikzcd}
	a & b \\
	c & d
	\arrow["f"', from=1-1, to=2-1]
	\arrow["i", from=1-1, to=1-2]
	\arrow["p", from=1-2, to=2-2]
	\arrow[Rightarrow, no head, from=2-1, to=2-2]
	\arrow[dashed, from=2-1, to=1-2]
\end{tikzcd}\]
By hypothesis, this square admits a lift $l:c\to b$, that we factorize in a morphism $r'\in L$ followed by a morphism $p'\in R$. The commutativity of the lower triangle implies equivalences $pl'\sim pp'r'\sim id$, and by unicity, $r'\sim id$ and $pp'\sim id$. The lift $l$ is equivalent to $p'$ and is then in $R$. The commutativity of the upper triangle implies $lf\sim lpi \sim i$ and by unicity again, $p'p\sim id$. The morphism $p$ is then an isomorphism, this implies that $f\sim i$, and $f$ is then in $L$. We proceed similarly for the dual assertion.
\end{proof}

\begin{prop}
\label{prop:caracterisation of L and R with lifting property}
A morphism is in $L$ (resp. in $R$) if and only if it has the unique left lifting property against morphisms of $R$ (resp. the unique right lifting property against the morphisms of $R$).
\end{prop}
\begin{proof}
This is the content of lemma \ref{lemma:caracterisation of L and R with lifting property 1} and \ref{lemma:caracterisation of L and R with lifting property 2}.
\end{proof}

\begin{prop}
\label{prop:fonctorialite des relevement}
The forgetful functor from the $\iun$-category of squares with lifts, and whose left (resp. right) vertical morphism is in $L$ (resp. in $R$), to the $\iun$-category of squares whose left (resp. right) vertical morphism is in $L$ (resp. in $R$), is an equivalence.

Roughly speaking, the formation of the lift in squares whose left (resp. right) vertical morphism is in $L$ (resp. in $R$) is functorial.
\end{prop}
\begin{proof}
The $\iun$-category of squares with lifts, and whose left (resp. right) vertical morphism is in $L$ (resp. in $R$), is the $\iun$-category
$$ \mbox{$\Arr_L(C)$}\times_C  \mbox{$\Arr(C)$}\times_C  \mbox{$\Arr_R(C)$}$$
and the  $\iun$-category   whose left (resp. right) vertical morphism is in $L$ (resp. in $R$) of squares is the limit of the diagram
\[\begin{tikzcd}
	{\Arr_L(C)\times_C \Arr(C)} & {\Arr(C)} & {\Arr(C)\times_C\Arr_R(C)}
	\arrow["\triangledown"', from=1-3, to=1-2]
	\arrow["\triangledown", from=1-1, to=1-2]
\end{tikzcd}\]
The forgetful functor is induced by the commutative diagram
\[\begin{tikzcd}
	{ \Arr_L(C)\times_C  \Arr(C)\times_C  \Arr_R(C)} && {\Arr(C)\times_C\Arr_R(C)} \\
	{\Arr_L(C)\times_C \Arr(C)} && {\Arr(C)}
	\arrow["\triangledown", from=1-3, to=2-3]
	\arrow["{ \Arr_L(C)\times_C \triangledown}"', from=1-1, to=2-1]
	\arrow["{ \triangledown\times_C  \Arr_R(C)}", from=1-1, to=1-3]
	\arrow["\triangledown"', from=2-1, to=2-3]
\end{tikzcd}\]
and we then have to show that it is cartesian.

By definition of factorization system, the morphism 
$$\triangledown:  \mbox{$\Arr_L(C)$}\times_C  \mbox{$\Arr_R(C)$}\to \Arr(C)$$ is an equivalence. The previous square is then equivalent to the square
\[\begin{tikzcd}
	{ \Arr_L(C)\times_C \Arr(C)_L\times_C\Arr_R(C)\times_C  \Arr_R(C)} && {\Arr(C)_L\times_C\Arr_R(C)\times_C\Arr_R(C)} \\
	{\Arr_L(C)\times_C \Arr_L(C)\times_C  \Arr_R(C)} && {\Arr_L(C)\times_C  \Arr_R(C)}
	\arrow["{\Arr_L(C)\times_C\triangledown}", from=1-3, to=2-3]
	\arrow["{ \Arr_L(C)\times_C  \Arr_L(C)\times_C\triangledown}"', from=1-1, to=2-1]
	\arrow["{ \triangledown\times_C  \Arr_R(C)\times_C  \Arr_R(C)}", shift left=2, draw=none, from=1-1, to=1-3]
	\arrow["{ \triangledown\times_C  \Arr_R(C)}"', from=2-1, to=2-3]
	\arrow[from=1-1, to=1-3]
\end{tikzcd}\]
which is obviously cartesian.
\end{proof}

\begin{prop}
\label{prop:cloture of L recap}
The $\infty$-groupoid $L$ is stable under colimit, transfinite composition, pushout, left cancellation and retract. The $\infty$-groupoid $R$ is stable under limit, composition, pullback, right cancellation and retract.
\end{prop}
\begin{proof}
Let $p:b\to d$ be a morphism of $R$ and $\{i_j:a_j\to c_j\}_{j:J}$ a family of morphisms of $L$ indexed by a functor $J\to \Arr_L(C)$, admitting a colimit $\bar{i}:\bar{a}\to \bar{c}$. Both functors $r\mapsto \Sq(r,p)$ and $c\mapsto\Hom(c,b)$ send colimits on limits. This implies that the morphism \[\Hom(\bar{c},b)\to\Sq(\bar{i},p)\] is the limit in $\Arr(\Sp)$ of the family of morphisms 
$$\Hom(c_j,b)\to\Sq({i_j,p}).$$
Each of these morphisms is an equivalence by assumption, so that implies that $\Hom(\bar{c},b)\to\Sq({\bar{i},p})$ is an equivalence. As this is true for any $p$ in $R$, proposition \ref{prop:caracterisation of L and R with lifting property} implies that $\bar{i}$ is in $L$.

This implies that $L$ is closed under colimit. As it is obviously closed under composition, lemma \ref{lemma:closed under colimit imply saturated_modified} implies that $L$ is closed under transfinite composition, pushout, and left cancellation.

Consider now a retract diagram:
\[\begin{tikzcd}
	a & {a'} & a \\
	c & {c'} & c
	\arrow["i", from=1-1, to=2-1]
	\arrow["i", from=1-3, to=2-3]
	\arrow[from=2-1, to=2-2]
	\arrow["{i'}", from=1-2, to=2-2]
	\arrow[from=1-1, to=1-2]
	\arrow[from=1-2, to=1-3]
	\arrow[from=2-2, to=2-3]
	\arrow["id"', curve={height=12pt}, from=2-1, to=2-3]
	\arrow["id", curve={height=-12pt}, from=1-1, to=1-3]
\end{tikzcd}\]
such that $i'$ is in $L$. For any morphism $p:b\to d$ of $R$, this induces a retract diagram
\[\begin{tikzcd}
	{\Hom(c,b)} & {\Hom(c',b)} & {\Hom(c,b)} \\
	{\Sq(i,p)} & {\Sq(i',p)} & {\Sq(i,p)}
	\arrow["{ }", from=1-1, to=2-1]
	\arrow["{ }", from=1-3, to=2-3]
	\arrow[from=2-1, to=2-2]
	\arrow["{ }", from=1-2, to=2-2]
	\arrow[from=1-1, to=1-2]
	\arrow[from=1-2, to=1-3]
	\arrow[from=2-2, to=2-3]
	\arrow["id"', curve={height=12pt}, from=2-1, to=2-3]
	\arrow["id", curve={height=-12pt}, from=1-1, to=1-3]
\end{tikzcd}\]
As equivalences are stable under retract, $\Hom(c,b)\to \Sq(i,p)$ is an equivalence, and as it is true for any $p$ in $R$, $i$ is in $L$.

We other assertion follows by duality.
\end{proof}

 We fix an $\infty$-groupoid $S$ of arrows of $C$ with $\U$-small domain and codomain.
\begin{definition}
 We define \sym{(ls@$L_S$}$L_S := \widehat{S}$, i.e as the smallest full sub $\infty$-groupoid of arrows of $C$ stable under colimits, composition and including $S$, and \wcnotation{$R_S$}{(rs@$R_S$} as the full sub $\infty$-groupoid of arrows of $C$ having the unique right lifting property against morphisms of $S$. 
 \end{definition}

\begin{construction}[Small object Argument]
\label{cons:small object argument}
Let $f:x\to y$ be an arrow. We define by induction on $\U$ a sequence $\{x_\alpha\}_{\alpha<\U}$ sending $\emptyset$ on $x$.
For a limit ordinal $\alpha<\U$, we set $x_{\alpha}:= \colim_{\alpha'<\alpha}{x_{\alpha'}}$. For a successor ordinal, we define $x_{\alpha+1}$ as the pushout:
\[\begin{tikzcd}
	{\colim_{a\to b\in S}\big(\colim_{\Sq(a\to b,x_\alpha\to y)}a\underset{\colim_{\Hom(b,x_\alpha)}a}{\coprod} \colim_{\Hom(b,x_\alpha)}b\big)} & {x_\alpha} \\
	{\colim_{a\to b\in S}\big(\colim_{\Sq(a\to b,x_\alpha\to y)} b\big)} & {x_{\alpha+1}} \\
	&& y
	\arrow[""{name=0, anchor=center, inner sep=0}, from=1-1, to=1-2]
	\arrow[from=2-1, to=2-2]
	\arrow[from=1-2, to=2-2]
	\arrow[from=1-1, to=2-1]
	\arrow[curve={height=12pt}, from=2-1, to=3-3]
	\arrow[curve={height=-12pt}, from=1-2, to=3-3]
	\arrow[dashed, from=2-2, to=3-3]
	\arrow["\lrcorner"{anchor=center, pos=0.125, rotate=180}, draw=none, from=2-2, to=0]
\end{tikzcd}\]
Let $i:x\to\tilde{x}$ be the transfinite composition of this sequence. There is an induced morphism $p:\tilde{x}\to y$, and an equivalence $f\sim pi$. 
\end{construction}

\begin{prop}
\label{prop:factorization system from S}
The previous construction defines a factorization system between $L_S$ and $R_S$. 
\end{prop}
\begin{proof}
Let $f:x\to y $ be any morphism. The previous construction produces functorially morphisms $i:x\to \tilde{x}$ and $p:\tilde{x}\to y $ whose composite is $f$. The morphism $i$ is  in $L_S$ by lemma  \ref{lemma:closed under colimit imply saturated_modified}. We then need to show that $p$ has the unique right lifting property against any morphism of $L_S$. Let $j:a\to b$ be any morphism in $L_S$, $n$ an integer and consider a commutative square
\[\begin{tikzcd}
	{a\coprod_{\colim_{\Sb_{n-1}} a} \colim_{\Sb_{n-1}} b} & {\tilde{x}} \\
	b & y
	\arrow[from=1-1, to=1-2]
	\arrow["j"', from=1-1, to=2-1]
	\arrow["p", from=1-2, to=2-2]
	\arrow[from=2-1, to=2-2]
\end{tikzcd}\]
By stability by $\omega$-small colimits, the object $a\coprod_{\colim_{\Sb_{n-1}} a} \colim_{\Sb_{n-1}} b$ is $\U$-small. There exists then $\alpha<\U$ such that the top morphism factors through $x_\alpha$, and by construction there exists a morphism $l:b\to x_{\alpha+1}$ and a comutative square
\[\begin{tikzcd}
	{a\coprod_{\colim_{\Sb_{n-1}} a} \colim_{\Sb_{n-1}} b} & {x_\alpha} & {\tilde{x}} \\
	& {x_{\alpha+1}} \\
	b && y
	\arrow[from=1-1, to=1-2]
	\arrow["j"', from=1-1, to=3-1]
	\arrow[from=1-2, to=1-3]
	\arrow[from=1-2, to=2-2]
	\arrow["p", from=1-3, to=3-3]
	\arrow[from=2-2, to=1-3]
	\arrow[from=2-2, to=3-3]
	\arrow["l", dotted, from=3-1, to=2-2]
	\arrow[from=3-1, to=3-3]
\end{tikzcd}\]
The induced diagonal is a lift in the first square. This implies that $\Hom(b,x)\to \Sq(j,p)$ has the right lifting property against $\Sb_{n-1}\to 1$. Eventually, this implies that $\Hom(b,x)\to \Sq(j,p)$ is an equivalence of $\infty$-groupoid, and $p$ then has the unique right lifting property against $i$. We then have a weak factorization system, which is a factorization system according to lemma \ref{lemma:when weak factorization system are factoryzation system}. 
\end{proof}

\subsection{Reflexive localization}

For the rest of the section, we fix a \textit{presentable $\iun$-category} $C$, and an $\infty$-groupoid $S$ of arrows of $C$ with $\U$-small domain and codomain. 
 
We recall that $L_S$ is $\widehat{S}$, i.e., the smallest full sub-$\infty$-groupoid of arrows of $C$ stable under colimits, composition, and including $S$, and $R_S$ is the full sub-$\infty$-groupoid of arrows of $C$ having the unique right lifting property against morphisms of $S$. 

\begin{definition}
 An object $x$ is \wcnotation{$S$-local}{local@$S$-local} if $x\to 1$ is in $R_S$, or equivalently if for any $i:a\to b\in S$, the induced functor $\Hom(i,x):\Hom(b,x)\to \Hom(a,x)$ is an equivalence.

We define \wcnotation{$C_{S}$}{(cs@$C_S$} as the full sub $\iun$-category of $C$ consisting of $S$-local objects.
\end{definition}
\begin{remark}
As $R_S$ is closed under right cancellation, any morphism between $S$-local objects is in $R_S$.
\end{remark}

\begin{theorem}
\label{theo:adjunction between presheaves and local presheaves}
The inclusion $\iota:C_S\to C$ is part of an adjunction
\[\begin{tikzcd}
	{\Fb_S:C} & {C_S:\iota}
	\arrow[""{name=0, anchor=center, inner sep=0}, shift left=2, from=1-1, to=1-2]
	\arrow[""{name=1, anchor=center, inner sep=0}, shift left=2, from=1-2, to=1-1]
	\arrow["\dashv"{anchor=center, rotate=-90}, draw=none, from=0, to=1]
\end{tikzcd}\]
Moreover, $\Fb_S:C\to C_S$ is the localization of $C$ by $\widehat{S}$.\sym{(f@$\Fb$}
\end{theorem}
\begin{proof}
For an object $x$, the small object argument provides a factorization of $x\to 1$ into a morphism $x\to \Fb_S x$ of $L_S$ followed by a morphism $\Fb_S x\to 1$ in $R_S$. The object $\Fb_Sx$ is then in $C_S$. As the factorization is functorial, this defines a functor $\Fb_S:C\to C_S$, and a natural transformation $\nu:id\to \Fb_S$ constant on $S$-local objects. As $\Fb_S\iota$ is equivalent to the identity, this induces the claimed adjunction. 

For the second proposition, let $F:C\to D$ be a functor sending morphisms of $L_S$ on equivalences. We define $\Db(F):= F\circ \iota$, and we have a diagram
\[\begin{tikzcd}
	C && D \\
	& {C_S}
	\arrow["{\Fb_S}"', from=1-1, to=2-2]
	\arrow[""{name=0, anchor=center, inner sep=0}, "F", from=1-1, to=1-3]
	\arrow["{\Db(F)}"', from=2-2, to=1-3]
	\arrow[shorten <=5pt, shorten >=5pt, Rightarrow, from=0, to=2-2]
\end{tikzcd}\]
that commutes up to the natural transformation $F\circ_0 \nu:F\to D(F)\circ \Fb_S$. However, the natural transformation $\nu$ is pointwise in $L_S$, which implies that $F\circ \nu$ is pointwise an equivalence, and the previous diagram then commutes. Now, let $G:C_S\to D$ be any other functor such that $G\circ\Fb_S \sim F$. By precomposing with iota, this implies that $G\sim F\circ \iota$.
\end{proof}

\begin{cor}
 \label{cor:derived colimit preserving functor}
The $\iun$-category $C_S$ is cocomplete. Moreover, if $F:C\to D$ is a colimit preserving functor sending $S$ onto equivalences, the induced functor $\Db F:C_S\to D$ preserves colimits.
\end{cor}
\begin{proof}
The first assertion is a direct consequence of the adjunction given in theorem \ref{theo:adjunction between presheaves and local presheaves}.

This adjunction also implies that the colimit of a functor $G:A\to C_S$ is given by $\Fb_S(\colim_{a:A} \iota G(a))$.
As the canonical morphism $\colim_{a:A} \iota G(a)\to \Fb_S(\colim_{a:A} \iota G(a))$ is  by construction in $\widehat{S}$ this proves the second assertion.
\end{proof}

\begin{construction}
 Suppose given an adjunction between two $\iun$-categories
\[\begin{tikzcd}
	{F:C} & {D:G}
	\arrow[""{name=0, anchor=center, inner sep=0}, shift left=2, from=1-1, to=1-2]
	\arrow[""{name=1, anchor=center, inner sep=0}, shift left=2, from=1-2, to=1-1]
	\arrow["\dashv"{anchor=center, rotate=-90}, draw=none, from=0, to=1]
\end{tikzcd}\]
with unit $\nu$ and counit $\epsilon$,
as well as an $\infty$-groupoid of morphisms $S$ of $C$ and $T$ of $D$ such that $F(S)\subset \widehat{T}$. 
By adjunction property, it implies that for any $T$-local object $d\in D$, $G(d)$ is $S$-local.
The previous adjunction induces a derived adjunction\sym{(lf@$\Lb F$} \sym{(rg@$\Rb G$}
\[\begin{tikzcd}
	{\Lb F:C_S} & {D_T:\Rb G}
	\arrow[""{name=0, anchor=center, inner sep=0}, shift left=2, from=1-1, to=1-2]
	\arrow[""{name=1, anchor=center, inner sep=0}, shift left=2, from=1-2, to=1-1]
	\arrow["\dashv"{anchor=center, rotate=-90}, draw=none, from=0, to=1]
\end{tikzcd}\]
where $\Lb F$ is defined by the formula $c\mapsto \Fb_T F(c)$ and $\Rb G$ is the restriction of $G$ to $D_T$. The unit is given by $\nu\circ \Fb_S$ and the counit by the restriction of $\epsilon$ to $D_T$.
\end{construction}
\begin{example}
\label{exe:exe localization}
\index[notation]{(f0@$f_{\mbox{$\exclam$}}:C_{/c}\to C_{d/}$}
\index[notation]{(f1@$f^*:C_{/d}\to C_{c/}$}
\index[notation]{(f2@$f_*:C_{/c}\to C_{d/}$}
\index[notation]{(lf0@$\Lb f_{\mbox{$\exclam$}}:(C_{/c})_{S_{c/}}\to (C_{d/})_{S_{d/}}$}
\index[notation]{(rf1@$\Rb f^*:(C_{/d})_{S_{d/}}\to (C_{c/})_{S_{c/}}$}
\index[notation]{(lf2@$\Lb f^*:(C_{/d})_{S_{d/}}\to (C_{c/})_{S_{c/}}$}
\index[notation]{(rf3@$\Rb f_*:(C_{/c})_{S_{c/}}\to (C_{d/})_{S_{d/}}$}
Let $C$ be a presentable $\iun$-category, $S$ a full sub $\infty$-groupoid of morphisms of $\iPsh{A}$ with $\U$-small codomain and domain. 
Eventually, we set \wcnotation{$S_{/c}$}{(sc@$S_{/c}$} as the $\infty$-groupoid of morphisms of shape
\[\begin{tikzcd}
	& b \\
	a & c
	\arrow[from=2-1, to=2-2]
	\arrow[from=1-2, to=2-2]
	\arrow["s", from=2-1, to=1-2]
\end{tikzcd}\]
where $s:S$.

 A morphism $f:c\to d$ induces an adjunction
\[\begin{tikzcd}
	{f_!:C_{/c}} & {C_{/d}:f^*}
	\arrow[""{name=0, anchor=center, inner sep=0}, shift left=2, from=1-1, to=1-2]
	\arrow[""{name=1, anchor=center, inner sep=0}, shift left=2, from=1-2, to=1-1]
	\arrow["\dashv"{anchor=center, rotate=-90}, draw=none, from=0, to=1]
\end{tikzcd}\] 
where the left adjoint is the composition and the right adjoint is the pullback. By construction, $f_!(S_{/c})\subset S_{/d}$. The previous adjunction can then be derived, and induced an adjunction:
\[\begin{tikzcd}
	{\Lb f_!:(C_{/c})_{S_{/c}}} & {(C_{/d})_{S_{/d}}:\Rb f^*}
	\arrow[""{name=0, anchor=center, inner sep=0}, shift left=2, from=1-1, to=1-2]
	\arrow[""{name=1, anchor=center, inner sep=0}, shift left=2, from=1-2, to=1-1]
	\arrow["\dashv"{anchor=center, rotate=-90}, draw=none, from=0, to=1]
\end{tikzcd}\]
where the right adjoint is just the restriction of $f^*$ to $S_{/d}$-local objects.

If the functor $f^*:C_{/d}\to C_{/c}$ preserves colimits and $f^*(S_{/c})\subset S_{/d}$, the adjunction
\[\begin{tikzcd}
	{f^*:C_{/d}} & {C_{/c}:f_*}
	\arrow[""{name=0, anchor=center, inner sep=0}, shift left=2, from=1-1, to=1-2]
	\arrow[""{name=1, anchor=center, inner sep=0}, shift left=2, from=1-2, to=1-1]
	\arrow["\dashv"{anchor=center, rotate=-90}, draw=none, from=0, to=1]
\end{tikzcd}\]
induces an adjunction 
\[\begin{tikzcd}
	{\Lb f^*:(C_{/d})_{S_{/d}}} & {(C_{/c})_{S_{/c}}:\Rb f_*}
	\arrow[""{name=0, anchor=center, inner sep=0}, shift left=2, from=1-1, to=1-2]
	\arrow[""{name=1, anchor=center, inner sep=0}, shift left=2, from=1-2, to=1-1]
	\arrow["\dashv"{anchor=center, rotate=-90}, draw=none, from=0, to=1]
\end{tikzcd}\]

\end{example}

\subsection{Monomorphisms}
We fix a cocomplete $\iun$-category $C$ until the end of the section. 
\begin{definition}
\label{defi:mono abstrait}
A morphism $i: a \to b$ in $C$ is a \notion{monomorphism} if the square 
\[\begin{tikzcd}
	A & A \\
	A & B
	\arrow[from=1-1, to=1-2]
	\arrow[from=1-1, to=2-1]
	\arrow[from=1-2, to=2-2]
	\arrow[from=2-1, to=2-2]
\end{tikzcd}\]
is cartesian.
\end{definition}

\begin{definition}
\label{defi:epi abstrait}
A morphism $i: a \to b$ in $C$ is an \notion{epimorphism} if the square 
\[\begin{tikzcd}
	A & B \\
	B & B
	\arrow[from=1-1, to=1-2]
	\arrow[from=1-1, to=2-1]
	\arrow[from=1-2, to=2-2]
	\arrow[from=2-1, to=2-2]
\end{tikzcd}\]
is cocartesian.
\end{definition}
\begin{remark}
A morphism $f$ in $C$ is a monomorphism if and only if the corresponding morphism $f'$ in $C^{op}$ is an epimorphism.
\end{remark}

In what follows, we will focus on monomorphisms, but all results have obvious analogues for epimorphisms thanks to the previous remark.

\begin{prop}
\label{prop:mono closed by colimit}
Monomorphisms are closed by limits.
\end{prop}
\begin{proof}
This directly follows from the fact that cartesian squares are closed by limits.
\end{proof}

\begin{lemma}
\label{lemma:mono 1}
Let $i: a \to b$ be a monomorphism. For any integer $n$, 
the morphism
$$a \to \Pi_{\Sb_n}a \underset{\Pi_{\Sb_n}b}{\times} b$$
is an equivalence.
\end{lemma}
\begin{proof}
As cartesian squares are stable by limits, the $\infty$-groupoid of morphisms of spaces $S \to T$ such that 
\[\begin{tikzcd}
	{\Pi_{T}a} & {\Pi_{T}b} \\
	{\Pi_{S}a} & {\Pi_{S}b}
	\arrow[from=1-1, to=1-2]
	\arrow[from=1-1, to=2-1]
	\arrow[from=1-2, to=2-2]
	\arrow[from=2-1, to=2-2]
\end{tikzcd}\]
is cartesian is closed by colimit. We will now show by induction on $n$ that it includes $\Sb_n \to 1$. The case $n = -1$ is trivial, and the case $n = 0$ follows from the definition of monomorphisms.

Suppose the result is true at stage $n$. The induction step follows from the fact that $\Sb_{n+1} \to 1$ is the horizontal colimit of the diagram
\[\begin{tikzcd}
	1 & {\Sb_n} & 1 \\
	1 & 1 & 1
	\arrow[from=1-1, to=2-1]
	\arrow[from=1-2, to=1-1]
	\arrow[from=1-2, to=1-3]
	\arrow[from=1-2, to=2-2]
	\arrow[from=1-3, to=2-3]
	\arrow[from=2-2, to=2-1]
	\arrow[from=2-2, to=2-3]
\end{tikzcd}\]
\end{proof}

\begin{prop}
\label{prop:mono 2}
Consider a commutative square
\[\begin{tikzcd}
	a & c \\
	b & d
	\arrow[from=1-1, to=1-2]
	\arrow["i"', from=1-1, to=2-1]
	\arrow["p", from=1-2, to=2-2]
	\arrow[from=2-1, to=2-2]
\end{tikzcd}\]
where the right vertical morphism is a monomorphism. The groupoid of lifts of this square is either empty or contractible.
\end{prop}
\begin{proof}
We have to show that the fibers of the morphism $\Hom(b, c) \to Sq(i, p)$, defined in \ref{defi:of lift}, are either contractible or empty. This is equivalent to showing that $\Hom(b, c) \to Sq(i, p)$ has the right lifting property against $\Sb_n \to 1$ for any $n\geq 0$, which in turn is equivalent to showing that $a \to b$ has the left lifting property against $a \to \Pi_{\Sb_n}a \underset{\Pi_{\Sb_n}b}{\times} b$ for any $n\geq 0$. As this last morphism is an equivalence by lemma \ref{lemma:mono 1}, this concludes the proof.
\end{proof}
\begin{prop}
\label{prop:non trivial fact about mono}
Suppose given a $\iun$-category $J$ and a family of squares
\begin{equation}
\label{eq:sqj}
\begin{tikzcd}
	{a_j} & {c_j} \\
	{b_j} & {d_j}
	\arrow[from=1-1, to=1-2]
	\arrow["{i_j}"', from=1-1, to=2-1]
	\arrow["{p_j}", from=1-2, to=2-2]
	\arrow[from=2-1, to=2-2]
\end{tikzcd}
\end{equation}
natural in $j:J$ and such that for any $j$, $p_j$ is a monomorphism. The following are equivalent:
\begin{enumerate}
\item for any $j$, the square \eqref{eq:sqj} admits lifts $\alpha_j: b_j \to c_j$,
\item for any $j$, the square \eqref{eq:sqj} admits unique lifts $\alpha_j: b_j \to c_j$,
\item there exists a family of lifts $\alpha_j: b_j \to c_j$ natural in $j: J$,
\item there exists a unique family of lifts $\alpha_j: b_j \to c_j$ natural in $j: J$.
\end{enumerate}
\end{prop} 
\begin{proof}
The equivalence between $(1)$ and $(2)$ follows from proposition \ref{prop:mono 2}. The equivalence between $(3)$ and $(4)$ follows from proposition \ref{prop:mono 2} applied to the $\iun$-category $\Fun(J, C)$. 
The implication $(3) \Rightarrow (1)$ is obvious. 

To conclude the proof, we suppose $(2)$ and we want to demonstrate $(4)$. 
We denote
$$\uvar \boxtimes \uvar: \Fun(J^{op}, \igrd) \times \Fun(J, C) \to C$$
the functor that sends a presheaf $X$ and a functor $F: J \to C$ to
$$\colim_{J_{/X}} F.$$

We want to show that $p_\uvar: c_\uvar \to d_\uvar$ has the unique right lifting property against $i_\uvar \sim 1 \boxtimes i_\uvar: a_\uvar \to b_\uvar$ where $1$ denotes the terminal presheaf.
As the morphisms having the unique right lifting property against $p_\uvar$ are closed by colimit, it is sufficient to demonstrate that $p_\uvar$ has the unique right lifting property against $h_j \boxtimes i$ where $h_j$ denotes the presheaf represented by $j$.

Note that by adjunction, for any $j$, lifts in the square 
\[\begin{tikzcd}
	{h_j \boxtimes a_\uvar} & {c_\uvar} \\
	{h_j \boxtimes b_\uvar} & {d_\uvar}
	\arrow[from=1-1, to=1-2]
	\arrow["{h_j \boxtimes i}"', from=1-1, to=2-1]
	\arrow["{p_\uvar}", from=1-2, to=2-2]
	\arrow[dashed, from=2-1, to=1-2]
	\arrow[from=2-1, to=2-2]
\end{tikzcd}\]
correspond to lifts in the square
\[\begin{tikzcd}
	{a_j} & {c_j} \\
	{b_j} & {d_j}
	\arrow[from=1-1, to=1-2]
	\arrow["{i_j}"', from=1-1, to=2-1]
	\arrow["{p_j}", from=1-2, to=2-2]
	\arrow[dashed, from=2-1, to=1-2]
	\arrow[from=2-1, to=2-2]
\end{tikzcd}\]
and thus exist and are unique by assumption. This concludes the proof.
\end{proof}

\section{Basic constructions}
\label{chapter:Basica construciton}
\subsection{$\io$-Categories}
\label{section:iocategories}
The definitions of section \ref{subsection:the categoru theta} will be used freely here.

\subsubsection{Definition of $\io$-categories}

\begin{definition}
We denote by 
$$[\uvar,\uvar]: \iPsh{\Theta}\times \iPsh{\Delta}\to \iPsh{\Delta[\Theta]}$$
the extension by colimit of the functor $\Theta\times \Delta\to \iPsh{\Delta[\Theta]}$ sending $(a,n)$ onto $[a,n]$.
For an integer $n$, we denote
$$[\uvar,n]:\iPsh{\Theta}^n\to \iPsh{\Theta}$$ 
the extension by colimit of the functor 
$\Theta^n\to\iPsh{\Theta}$ sending $\textbf{a}:=\{a_1,...,a_n\}$ onto $[\textbf{a},n]$.
\end{definition}

\begin{construction}  We have an adjunction 
\begin{equation}
\label{eq:underived adjunction part}
\begin{tikzcd}
	{ i_!:\iPsh{\Delta[\Theta]}} & {\iPsh{\Theta}:i^*}
	\arrow[shift left=2, from=1-1, to=1-2]
	\arrow[shift left=2, from=1-2, to=1-1]
\end{tikzcd}
\end{equation}
where the left adjoint is the left Kan extension of the functor $\Delta[\Theta]\xrightarrow{i} \Theta\to \iPsh{\Theta}$. The sets of morphisms $\W$ and $\M$ are respectively defined in \ref{defi:definition of W} and \ref{defi:defi of delta theta}.
There is an obvious inclusion $i_!(\M)\subset \W$. The previous adjunction then induced a derived adjunction
\begin{equation}
\label{eq:derived adjunction}
\begin{tikzcd}
	{\Lb i_!:\Psh{\Delta[\Theta]}_{\M}} & {\Psh{\Theta}_{\W}:\Rb i^*}
	\arrow[shift left=2, from=1-1, to=1-2]
	\arrow[shift left=2, from=1-2, to=1-1]
\end{tikzcd}
\end{equation}
\end{construction}

\begin{prop}
\label{prop:infini changing theta}
The unit and counit of the adjunction \eqref{eq:underived adjunction part} are respectively in $\widehat{\M}$ and $\widehat{\W}$. As a consequence, the adjunction \eqref{eq:derived adjunction} is an adjoint equivalence.
\end{prop}
\begin{proof}
We denote by $\iota:\Psh{\Theta}\to \iPsh{\Theta}$ and $\iota:\Psh{\Delta[\Theta]}\to \iPsh{\Delta[\Theta]}$ the two canonical inclusions. By the proposition \ref{prop:link beetwenwidehat and overline}, we have inclusions $\iota(\overline{\W})\subset \widehat{\W}$ and $\iota(\overline{\M})\subset \widehat{\M}$. The result then directly follows from theorem \ref{theo:unit and counit are in W}.  
\end{proof}

\begin{definition}
 A \wcnotion{$\io$-category}{category4@$\io$-category} is a $\W$-local $\infty$-presheaf $C\in \iPsh{\Theta}$. We then define \index[notation]{((a60@$\ocat$}
$$\ocat := \iPsh{\Theta}_{\W}.$$
Proposition \ref{prop:infini changing theta} implies that $\ocat$ identifies itself with the full sub $\iun$-category of $\iPsh{\Delta[\Theta]}$ of $\M$-local objects:
$$\ocat \sim \iPsh{\Delta[\Theta]}_{\M}.$$
We recall that the sets of morphisms $\W$ and $\M$ are respectively defined in \ref{defi:definition of W} and \ref{defi:defi of delta theta}
\end{definition}

\begin{construction}
We denote by $\pi_0:\iPsh{\Theta}\to \Psh{\Theta}$ the functor sending an $\infty$-presheaf $X$ onto the presheaf
$$\pi_0X:a\mapsto\pi_0(X_a)$$
This functor admits a fully faithful right adjoint: $\N:\Psh{\Theta}\to \iPsh{\Theta}$. 
As $\pi_0$ preserves $\W$, it induces an adjoint pair: \sym{(pi@$\pi_0:\ocat\to \zocat$}\sym{n@$\N:\zocat\to \ocat$}
\[\begin{tikzcd}
	{\pi_0:\ocat} & {\zocat:\N}
	\arrow[""{name=0, anchor=center, inner sep=0}, shift left=2, from=1-1, to=1-2]
	\arrow[""{name=1, anchor=center, inner sep=0}, shift left=2, from=1-2, to=1-1]
	\arrow["\dashv"{anchor=center, rotate=-90}, draw=none, from=0, to=1]
\end{tikzcd}\]
where the right adjoint $\N$ is fully faithful.
Every $\zo$-category can then be seen as an $\io$-category and we will call \wcnotion{strict}{strict $\io$-category} the $\io$-categories lying in the image of this functor. 
\label{cons:strict}
\end{construction}

\subsubsection{Equivalence between $\io$-categories}

\begin{definition} A \wcsnotion{$n$-cell}{cell@$n$-cell}{for $\io$-categories} of an $\io$-category is a morphism $\Db_n\to C$.
If $C$ is an $\io$-category, we denote by $C_n$ the value of $C$ on $\Db_n$. 
\end{definition}

\begin{prop}
\label{prop:equivalences detected on globes}
Let $C,D$ be two $\io$-categories, and $f:C\to D$ any map. The morphism $f$ is an equivalence if and only if for any $n$, the induced morphism $f_n: C_n\to D_n$ is an equivalence. 
\end{prop}
\begin{proof}
This is a necessary condition. For the converse, let $f$ be a morphism fulfilling this condition. To show that $f$ is an equivalence, we have to show that for any globular sum $a$, $f_a: C_a\to D_a$ is an equivalence. This is true as 
$$f_a:C_a\to D_a~\sim~ \lim_{\Db_n\in\Sp_a}{f_n:C_n\to D_n}.$$	
\end{proof}

\begin{lemma}
\label{lemma:equivalence if unique right lifting property against globes.}
A functor is an equivalence if it has the unique right lifting property against $\emptyset\to \Db_n$ for any $n\geq 0$. 
\end{lemma}
\begin{proof}
This is a necessary condition. For the converse, let $f:C\to D$ be a morphism fulfilling this condition. By definition of  unique left lifting property, it implies that the induced morphism
$f_n:C_n\to D_n$ is an equivalence for any $n\geq 0$. Using proposition \ref{prop:equivalences detected on globes}, $f$ is an equivalence.
 \end{proof}
 
\subsubsection{Suspension}

\begin{construction} Let $\iPsh{\Theta}_{\bullet,\bullet}$ be the $(\infty,1)$-category of $\infty$-presheaves on $\Theta$ with two distinguished points, i.e. of triples $(C,a,b)$ where $a$ and $b$ are elements of $C_0$.
The functor $[\uvar,1]:\Theta\to \iPsh{\Theta}_{\bullet,\bullet}$ that sends $a$ onto $([a,1],\{0\},\{1\})$ induces by extension by colimit an adjunction
\begin{equation}
\label{eq:suspesnion betweenpresheaves}
\begin{tikzcd}
	{[\uvar,1]:\iPsh{\Theta}} & {\iPsh{\Theta}_{\bullet,\bullet}:\hom_{\uvar}(\uvar,\uvar)}
	\arrow[""{name=0, anchor=center, inner sep=0}, shift left=2, from=1-1, to=1-2]
	\arrow[""{name=1, anchor=center, inner sep=0}, shift left=2, from=1-2, to=1-1]
	\arrow["\dashv"{anchor=center, rotate=-90}, draw=none, from=0, to=1]
\end{tikzcd}
\end{equation}
As the left adjoint preserves representables, the right adjoint commutes with colimit. It is then easy to check on representables that the unit of this adjunction is an equivalence. As a consequence, the left adjoint is fully faithful.
\end{construction}

\begin{remark}
\label{rem:suspension preserevs colimits index by theta plus}
We recall that the category $\Theta^+$, defined in \ref{defi:thetaplus}, is the full subcategory of $\zocat$ whose objects are globular sums or the empty category. 
Now remark that for any $\Theta$-presheaves $C$, we have 
$$[C,1]\sim \colim_{\Theta^+_{/C}} [\uvar,1].$$
\end{remark}

\begin{lemma}
\label{lemma:hom of the suspension prequel}
Let $C$ be an $\infty$-presheaves on $\Theta$. The canonical morphisms
$$C\to \hom_{[C,1]}(0,1)~~\hom_{[C,1]}(0,0)\to 1~~~ \hom_{[C,1]}(1,1) \sim 1~~~~\emptyset\to \hom_{[C,1]}(1,0)$$
are equivalences.
\end{lemma}
\begin{proof}
As both $\hom$ and $[\uvar,1]$ preserve colimits, it is sufficient to check this property on representables, where it is an easy computation.
\end{proof}

\begin{prop}
\label{prop:supspension preserves cat}
The functor $[\uvar,1]:\iPsh{\Theta}\to \iPsh{\Theta}$ preserves $\io$-categories.
\end{prop}
\begin{proof}
By construction, for any pair of integers $k<n$, and any pair $([\textbf{a},n],b)$ where $[\textbf{a},n]$ is a globular sum and $b$ is a globular sum or the empty $\io$-category, we have cartesian squares
\[\begin{tikzcd}
	1 & {\Hom_{\Theta}([\textbf{a},n],[b,1])} & {\Hom_{\Theta}(a_k,b)} & {\Hom_{\Theta}([\textbf{a},n],[b,1])} \\
	{\{\epsilon\}} & {\Hom_{\Delta}([n],[1])} & {\{\alpha_k\}} & {\Hom_{\Delta}([n],[1])}
	\arrow[""{name=0, anchor=center, inner sep=0}, from=2-1, to=2-2]
	\arrow[from=1-1, to=1-2]
	\arrow[from=1-2, to=2-2]
	\arrow[from=1-1, to=2-1]
	\arrow[from=1-4, to=2-4]
	\arrow[""{name=1, anchor=center, inner sep=0}, from=2-3, to=2-4]
	\arrow[from=1-3, to=2-3]
	\arrow[from=1-3, to=1-4]
	\arrow["\lrcorner"{anchor=center, pos=0.125}, draw=none, from=1-1, to=0]
	\arrow["\lrcorner"{anchor=center, pos=0.125}, draw=none, from=1-3, to=1]
\end{tikzcd}\]
where $\epsilon$ denote any constant functor with value $0$ or $1$, and $\alpha_k$ the morphism that sends $k$ on $0$ and $k+1$ on $1$.
Let $C$ be an object of $\iPsh{\Theta}$. We then have $[C,1] \sim \colim_{\Theta^+_{/C}} [\uvar,1]$.
As the evalutation commutes with colimits and as the $\iun$-category $\igrd$ is locally cartesian closed,  the previous cartesian squares induce   cartesian squares
\begin{equation}
\label{eq:prop:supspension preserves cat}
\begin{tikzcd}[column sep =0.3cm]
	1 & {\Hom_{\iPsh{\Theta}}([\textbf{a},n],[C,1])} & {\Hom_{\iPsh{\Theta}}(a_k,C)} & {\Hom_{\iPsh{\Theta}}([\textbf{a},n],[C,1])} \\
	{\{\epsilon\}} & {\Hom_{\Delta}([n],[1])} & {\{\alpha_k\}} & {\Hom_{\Delta}([n],[1])}
	\arrow[""{name=0, anchor=center, inner sep=0}, from=2-1, to=2-2]
	\arrow[from=1-1, to=1-2]
	\arrow[from=1-2, to=2-2]
	\arrow[from=1-1, to=2-1]
	\arrow[from=1-4, to=2-4]
	\arrow[""{name=1, anchor=center, inner sep=0}, from=2-3, to=2-4]
	\arrow[from=1-3, to=2-3]
	\arrow[from=1-3, to=1-4]
	\arrow["\lrcorner"{anchor=center, pos=0.125}, draw=none, from=1-1, to=0]
	\arrow["\lrcorner"{anchor=center, pos=0.125}, draw=none, from=1-3, to=1]
\end{tikzcd}
\end{equation}
Using the fact that pullbacks are stable by limits and using the two previous squares in the case where $[\textbf{a},n]$ is a globe, we get two cartesian squares:
\[\begin{tikzcd}[column sep =0.3cm]
	1 & {\Hom_{\iPsh{\Theta}}(\Sp_{[\textbf{a},n]},[C,1])} & {\Hom_{\iPsh{\Theta}}(\Sp_{a_k},C)} & {\Hom_{\iPsh{\Theta}}(\Sp_{[\textbf{a},n]},[C,1])} \\
	{\{\epsilon\}} & {\Hom_{\Delta}([n],[1])} & {\{\alpha_k\}} & {\Hom_{\Delta}([n],[1])}
	\arrow[""{name=0, anchor=center, inner sep=0}, from=2-1, to=2-2]
	\arrow[from=1-1, to=1-2]
	\arrow[from=1-2, to=2-2]
	\arrow[from=1-1, to=2-1]
	\arrow[from=1-4, to=2-4]
	\arrow[""{name=1, anchor=center, inner sep=0}, from=2-3, to=2-4]
	\arrow[from=1-3, to=2-3]
	\arrow[from=1-3, to=1-4]
	\arrow["\lrcorner"{anchor=center, pos=0.125}, draw=none, from=1-1, to=0]
	\arrow["\lrcorner"{anchor=center, pos=0.125}, draw=none, from=1-3, to=1]
\end{tikzcd}\]
This directly implies that $[C,1]$ is $\Wseg$-local when $C$ is.

Furthermore, for any integer $n>0$, the cartesian squares \eqref{eq:prop:supspension preserves cat} induces cartesian squares
\[\begin{tikzcd}
	1 & {\Hom_{\iPsh{\Theta}}(\Sigma^nE^{eq},[C,1])} & {\Hom_{\iPsh{\Theta}}(\Sigma^{n-1}E^{eq},C)} & {\Hom_{\iPsh{\Theta}}(\Sigma^nE^{eq},[C,1])} \\
	{\{\epsilon\}} & {\Hom_{\Delta}([1],[1])} & {\{\alpha_0\}} & {\Hom_{\Delta}([1],[1])}
	\arrow[from=1-1, to=1-2]
	\arrow[from=1-1, to=2-1]
	\arrow[from=1-2, to=2-2]
	\arrow[from=1-3, to=1-4]
	\arrow[from=1-3, to=2-3]
	\arrow[from=1-4, to=2-4]
	\arrow[""{name=0, anchor=center, inner sep=0}, from=2-1, to=2-2]
	\arrow[""{name=1, anchor=center, inner sep=0}, from=2-3, to=2-4]
	\arrow["\lrcorner"{anchor=center, pos=0.125}, draw=none, from=1-1, to=0]
	\arrow["\lrcorner"{anchor=center, pos=0.125}, draw=none, from=1-3, to=1]
\end{tikzcd}\]
which implies that $[C,1]$  is local with respect to $\Sigma^nE^{eq}\to \Sigma^{n}1$ when $C$ is.

Eventually, suppose given a diagram of shape
\begin{equation}
\label{eq:proof of sigma preserves omega cat 3}
\begin{tikzcd}
	{E^{eq}} & {[C,1]} \\
	{1}
	\arrow[from=1-1, to=2-1]
	\arrow[from=1-1, to=1-2]
\end{tikzcd}
\end{equation}
The canonical morphism $E^{eq}\to [C,1]\xrightarrow{\pi} [1]$ then factors through $0$ or $1$. As the two fibers of $\pi$ are trivial,  the diagram \eqref{eq:proof of sigma preserves omega cat 3} admits a unique lift, which concludes the proof.
\end{proof}

\begin{construction}
\label{cons:def of suspension}
As $[\uvar,1]$ sends $\W$ to a subset of $\M$, the functor $\hom_{\uvar,\uvar}(\uvar)$ preserves $\io$-categories. Combined with the proposition \ref{prop:supspension preserves cat}, this implies that 
the adjunction \eqref{eq:suspesnion betweenpresheaves} restricts to an adjunction:
\begin{equation}
\label{eq:suspesnion between category}
\begin{tikzcd}
	{[\uvar,1]:\ocat} & {\ocat_{\bullet,\bullet}:\hom_{\uvar}(\uvar,\uvar)}
	\arrow[""{name=0, anchor=center, inner sep=0}, shift left=2, from=1-1, to=1-2]
	\arrow[""{name=1, anchor=center, inner sep=0}, shift left=2, from=1-2, to=1-1]
	\arrow["\dashv"{anchor=center, rotate=-90}, draw=none, from=0, to=1]
\end{tikzcd}
\end{equation}
The left adjoint is the \wcsnotionsym{suspension functor}{((d60@$[\uvar,1]$}{suspension}{for $\io$-categories}\ssym{(hom@$\hom_{\uvar}(\uvar,\uvar)$}{for $\io$-categories}.
\end{construction}

\begin{prop}
\label{prop:hom of the suspension}
Let $C$ be an $\io$-categories. We have natural equivalences
$$\hom_{[C,1]}(0,1)\sim C~~~~\hom_{[C,1]}(0,0)\sim \hom_{[C,1]}(1,1) \sim 1~~~~\hom_{[C,1]}(1,0)\sim \emptyset.$$
\end{prop}
\begin{proof}
This is a direct consequence of lemma \ref{lemma:hom of the suspension prequel}.
\end{proof}

\begin{construction}
\label{cons:wiskering}
Suppose given an $\io$-category $C$ and a $1$-cells $f:x'\to x$. 
As $C$ is an $\io$-category, for any globular sum $a$, the morphism
$$\Hom([1]\vee[a,1],C)\to \Hom([1], C)\times_{\Hom([0],C)}\Hom([a,1],C)$$
is an equivalence. 
This induces a morphism
$$\Hom(a,\hom_C(x,y))\to \Hom([1]\vee[a,1],(C,x',y))\to \Hom(a,\hom_{C}(x',y))$$	
where the two distinguished points of $[1]\vee[a,1]$ are the extremal ones, and where the left-hand morphism is the restriction of the inverse of the previous morphism. By the Yoneda lemma, this corresponds to a morphism
$$f_!:\hom_C(x',y)\to \hom_C(x,y).$$
Conversely, a $1$-cell $g:y\to y'$ induces a morphism
$$g_!:\hom_C(x,y)\to \hom_C(x,y').$$
\end{construction}

\subsubsection{Special colimits}

\begin{definition} \label{defi: spetial colimits}
We denote by $\iota$ the inclusion of $\ocat$ into $\iPsh{\Theta}$.
A functor $F:I\to \ocat$ has a \snotion{special colimit}{for $\io$-categories} if the canonical morphism 
\begin{equation}
\label{eq:special colimit}
\colim_{i:I}\iota F(i)\to \iota(\colim_{i:I}F(i))
\end{equation}
is an equivalence of $\infty$-presheaves. 

Similarly, we say that a functor $\psi: I\to \Arr(\ocat)$ has a \textit{special colimit} if the canonical morphism 
$$\colim_{i:I}\iota \psi(i)\to \iota(\colim_{i:I}\psi(i))$$
is an equivalence in the arrow $\iun$-category of $\iPsh{\Theta}$.
\end{definition}

\begin{remark}
A functor $F:I\to \ocat$ then has a special colimit if and only if $\colim_IF$, computed in $\iPsh{\Theta}$, already is an $\io$-category.\end{remark}

\begin{remark}
Let $F,G:I\to \ocat$ be two functors, and $\psi:F\to G$ a cartesian natural transformation admitting a special colimit. As $\iPsh{\Theta}$ is cartesian closed, this implies that for any object $i$ of $I$, the canonical square
\[\begin{tikzcd}
	{F(i)} & {\colim_IF} \\
	{G(i)} & {\colim_IG}
	\arrow[from=1-1, to=1-2]
	\arrow["{\psi(i)}"', from=1-1, to=2-1]
	\arrow["\lrcorner"{anchor=center, pos=0.125}, draw=none, from=1-1, to=2-2]
	\arrow["{\colim_I\psi}", from=1-2, to=2-2]
	\arrow[from=2-1, to=2-2]
\end{tikzcd}\]
is cartesian.
\end{remark}

\begin{example}
\label{exemple:every iocategor is a special colimit}
Let $C$ be an $\io$-category. The canonical diagram $\Theta_{/C}\to \ocat$ has a special colimit, given by $C$.
\end{example}
\begin{prop}
\label{prop:special colimit}
Let $F,G:I\to \ocat$ be two functors, and $\psi:F\to G$ a natural transformation. If $\psi$ is cartesian, and $G$ has a special colimit, then $\psi$ and $F$ have special colimits. 
\end{prop}
\begin{proof}
We have to show that $F$ has a special colimit, it will directly imply that $\psi$ also has one. The morphism \eqref{eq:special colimit} is always in $\widehat{\W}$. To conclude, one then has to show that $\colim_{i:I}\iota \psi(i)$ is $\W$-local. To this extend, it is enough to demonstrate that the canonical morphism 
$$\colim_{i:I}\iota \psi(i): \colim_{i:I}\iota F(i)\to \colim_{i:I}\iota G(i)$$ 
has the unique right lifting property against $\W$. We then consider a square
\begin{equation}
\label{eq:proof special colimit}
\begin{tikzcd}
	a & {\colim_{i:I}\iota F(i)} \\
	b & { \colim_{i:I}\iota G(i)}
	\arrow[from=1-1, to=1-2]
	\arrow["f"', from=1-1, to=2-1]
	\arrow["{\colim_{i:I}\iota \psi(i)}", from=1-2, to=2-2]
	\arrow[from=2-1, to=2-2]
\end{tikzcd}
\end{equation}
where $f\in W$. As the domain of $f$ is representable, there always exists $j:I$, such that the bottom horizontal morphism factors through $G(j)$. As $\psi$ is cartesian, the square \eqref{eq:proof special colimit} factors in two squares, where the right one is cartesian. 
\[\begin{tikzcd}
	a & {F(i)} & {\colim_{i:I}\iota F(i)} \\
	b & {G(i)} & { \colim_{i:I}\iota G(i)}
	\arrow[from=1-1, to=1-2]
	\arrow["f"', from=1-1, to=2-1]
	\arrow[from=1-2, to=1-3]
	\arrow["{\psi(i)}", from=1-2, to=2-2]
	\arrow["\lrcorner"{anchor=center, pos=0.125}, draw=none, from=1-2, to=2-3]
	\arrow["{\colim_{i:I}\iota \psi(i)}", from=1-3, to=2-3]
	\arrow[from=2-1, to=2-2]
	\arrow[from=2-2, to=2-3]
\end{tikzcd}\]
Lifts in the square \eqref{eq:proof special colimit} are then equivalent to lifts in the left square, which exist and are unique as $F(i)\to G(i)$ has the unique right lifting property against $\W$.
\end{proof}

\begin{prop}
\label{prop:example of a special colimit}
For any integer $n$, and globular sums $a$ and $b$, the equalizer diagram 
\[\begin{tikzcd}
	{\coprod_{k+l=n-1}[a,k]\vee[a\times b,1]\vee[a,l]} & {\coprod_{k+l=n}[a,k]\vee[ b,1]\vee[a,l]}
	\arrow[shift left=2, from=1-1, to=1-2]
	\arrow[shift right=2, from=1-1, to=1-2]
\end{tikzcd}\]
where the top diagram is induced by $[a\times b,1]\to [a,1]\vee[b,1]$ and to bottom one by $[a\times b,1]\to [b,1]\vee[a,1]$,
has a special colimit, which is $[a,n]\times [b,1]$.
\end{prop}
\begin{proof}
The lemma \ref{lemma:colimit computed in set presheaves} implies that the colimit of the previous diagram, computed in $\iPsh{\Theta}$ is strict. It is then enough to show that this colimit, computed in $\Psh{\Theta}$, is equivalent to $[a,n]\times [b,1]$. As this last object is $\W$-local, this will concludes the proof. The remaining combinatorial exercise is left to the reader. 
\end{proof}

\begin{prop}
\label{prop:example of a special colimit 2}
Any sequence of $\io$-categories has a special colimit. 
\end{prop}
\begin{proof}
Suppose given such sequence.
If the sequence is finite, this is obviously true. Suppose now that the sequence is non finite. As codomains and domains of morphism of $\W$ are $\omega$-small, the colimit of the sequence, computed in $\iPsh{\Theta}$ is $\W$-local, which concludes the proof.
\end{proof}

\begin{lemma}
\label{lemma:[ ,1] preserves spcial limits}
For any $\io$-category $C$, the functor $[\uvar,1]:\Theta^+_{/C}\to \ocat$ has a special colimit given by $[C,1]$.
\end{lemma}
\begin{proof}
This is a direct consequence of proposition \ref{prop:supspension preserves cat}.
\end{proof}

\begin{lemma}
\label{lemma:[ ,1] vee [ ,1]preserves spcial limits}
We denote by $$
\begin{array}{cl}
 & [\uvar,1]\vee[1]:\ocat\to \ocat_{[0]\amalg[1]/}\\
 \mbox{(resp.} & [1]\vee[\uvar,1]:\ocat\to \ocat_{[1]\amalg[0]/})
\end{array}$$ the colimit preserving functor that sends an element $a$ of $\Theta$ onto the globular sum $[a,1]\vee[1]$ (resp. $[1]\vee[a,1]$).

For any $\io$-category $C$, the functors $[\uvar,1]\vee[1]$  and $[1]\vee[\uvar,1]$ have special colimits given respectively by $[C,1]\vee[1]$ and $[1]\vee[C,1]$.
\end{lemma}
\begin{proof}
To prove this, we establish a result analogous to the one given in the proposition \ref{prop:supspension preserves cat}. We omit its proof because it is long but essentially identical.
\end{proof}

\begin{prop}
\label{prop:example of a special colimit3}
Suppose given two cartesian squares
\[\begin{tikzcd}
	{ B} & C & D \\
	{\{0\}} & {[1]} & {\{1\}}
	\arrow[from=1-2, to=2-2]
	\arrow[from=1-1, to=1-2]
	\arrow[from=1-3, to=2-3]
	\arrow[from=1-1, to=2-1]
	\arrow[from=1-3, to=1-2]
	\arrow[from=2-1, to=2-2]
	\arrow[from=2-3, to=2-2]
	\arrow["\lrcorner"{anchor=center, pos=0.125}, draw=none, from=1-1, to=2-2]
	\arrow["\lrcorner"{anchor=center, pos=0.125, rotate=-90}, draw=none, from=1-3, to=2-2]
\end{tikzcd}\]
The diagram 
\[\begin{tikzcd}
	{[1]\vee[B,1]} & {[ B,1]} & {[C,1]} & {[D,1]} & {[D,1]\vee[1]}
	\arrow["\triangledown", from=1-4, to=1-5]
	\arrow["\triangledown"', from=1-2, to=1-1]
	\arrow[from=1-2, to=1-3]
	\arrow[from=1-4, to=1-3]
\end{tikzcd}\]
has a special colimit.
\end{prop}
\begin{proof}
Remark firsts that the colimit, computed in $\iPsh{\Theta}$, of the diagram
\[\begin{tikzcd}
	{[1]\vee[1]} & {[1]} & {[[1],1]} & {[1]} & {[1]\vee[1]}
	\arrow["\triangledown"', from=1-2, to=1-1]
	\arrow["\triangledown", from=1-4, to=1-5]
	\arrow["{[\{0\},1]}", from=1-2, to=1-3]
	\arrow["{[\{1\},1]}"', from=1-4, to=1-3]
\end{tikzcd}\]
is strict. We leave it to the reader to check that the colimit of this diagram, computed in $\Psh{\Theta}$, already is an $\io$-category. This implies that the diagram admits a special colimit.

We are now willing to show that the two squares
\[\begin{tikzcd}
	{[B,1]} & {[C,1]} & {[D,1]} \\
	{[1]} & {[[1],1]} & {[1]}
	\arrow[from=1-1, to=1-2]
	\arrow[from=1-3, to=1-2]
	\arrow["{[\{0\},1]}"', from=2-1, to=2-2]
	\arrow["{[\{1\},1]}", from=2-3, to=2-2]
	\arrow[from=1-1, to=2-1]
	\arrow[from=1-2, to=2-2]
	\arrow["\lrcorner"{anchor=center, pos=0.125}, draw=none, from=1-1, to=2-2]
	\arrow[from=1-3, to=2-3]
	\arrow["\lrcorner"{anchor=center, pos=0.125, rotate=-90}, draw=none, from=1-3, to=2-2]
\end{tikzcd}\]
are cartesian. We recall that $\Theta^+$ is the subcategory of $\zocat$ whose objects are globular sums or the empty $\io$-category. Lemma \ref{lemma:[ ,1] preserves spcial limits} implies that $[C,1]$ is the special colimit of the functor $[\uvar,1]:\Theta^+_{/C}\to \ocat$. As $\iPsh{\Theta}$ is cartesian closed, the pullback along morphisms between $\io$-categories preserves special colimits. We can then restrict to the case where $C$ is a globular sum or the empty $\io$-category. In this case, it directly follows from the fact that $\Theta^+$ is closed by pullback and that $\Theta^+\to \ocat$ preserves them.

By the same argument, one can demonstrate that the two squares
\[\begin{tikzcd}
	{[B,1]} & {[1]\vee[B,1]} & {[D,1]} & {[D,1]\vee[1]} \\
	{[1]} & {[1]\vee[1]} & {[1]} & {[1]\vee[1]}
	\arrow[from=1-3, to=1-4]
	\arrow[from=1-1, to=1-2]
	\arrow["\triangledown"', from=2-1, to=2-2]
	\arrow["\triangledown"', from=2-3, to=2-4]
	\arrow[from=1-4, to=2-4]
	\arrow[from=1-3, to=2-3]
	\arrow[from=1-1, to=2-1]
	\arrow[from=1-2, to=2-2]
	\arrow["\lrcorner"{anchor=center, pos=0.125}, draw=none, from=1-1, to=2-2]
	\arrow["\lrcorner"{anchor=center, pos=0.125}, draw=none, from=1-3, to=2-4]
\end{tikzcd}\]
are cartesian by reducing to the case where $B$ and $D$ are globular sums or empty $\io$-categories.

The result then follows from proposition \ref{prop:special colimit}.
\end{proof}
\begin{prop}
\label{prop:example of a special colimit4}
Suppose given a cartesian square
\[\begin{tikzcd}
	{ B} & C \\
	{\{0\}} & {[1]}
	\arrow[from=1-2, to=2-2]
	\arrow[from=1-1, to=1-2]
	\arrow[from=1-1, to=2-1]
	\arrow[from=2-1, to=2-2]
	\arrow["\lrcorner"{anchor=center, pos=0.125}, draw=none, from=1-1, to=2-2]
\end{tikzcd}\]
The diagram 
\[\begin{tikzcd}
	{[1]\vee[B,1]} & {[ B,1]} & {[C,1]}
	\arrow["\triangledown"', from=1-2, to=1-1]
	\arrow[from=1-2, to=1-3]
\end{tikzcd}\]
has a special colimit.
\end{prop}
\begin{proof}
The proof is similar to the previous one. 
\end{proof}

\subsubsection{$(\infty,n)$-Categories and truncation functors}

\begin{construction}
  We have an adjunction 
\begin{equation}
\label{eq:underived adjunction case n}
\begin{tikzcd}
	{ i_!:\iPsh{\Delta[\Theta_{n-1}]}} & {\iPsh{\Theta_n}:i^*}
	\arrow[shift left=2, from=1-1, to=1-2]
	\arrow[shift left=2, from=1-2, to=1-1]
\end{tikzcd}
\end{equation}
where the left adjoint is the left Kan extension of the functor $\Delta[\Theta_{n-1}]\xrightarrow{i} \Theta_{n}\to \iPsh{\Theta_{n}}$. We recall that the sets of morphisms $\W_n$ and $\M_n$ are respectively defined in \ref{defi:definition of W} and \ref{defi:defi of delta theta}
Remark that there is an obvious inclusion $i_!(\M_n)\subset \W_n$. The previous adjunction then induced a derived adjunction
\begin{equation}
\label{eq:derived adjunction case n}
\begin{tikzcd}
	{\Lb i_!:\Psh{\Delta[\Theta_{n-1}]}_{\M}} & {\Psh{\Theta_{n}}_{\W}:\Rb i^*}
	\arrow[shift left=2, from=1-1, to=1-2]
	\arrow[shift left=2, from=1-2, to=1-1]
\end{tikzcd}
\end{equation}
\end{construction}

\begin{prop}
\label{prop:infini changing theta n}
The unit and counit of the adjunction \eqref{eq:underived adjunction case n} are respectively in $\widehat{\M}_n$ and $\widehat{\W}_n$. As a consequence, the adjunction \eqref{eq:derived adjunction case n} is an adjoint equivalence.
\end{prop}
\begin{proof}
We denote by $\iota:\Psh{\Theta_n}\to \iPsh{\Theta_n}$ and $\iota:\Psh{\Delta[\Theta_{n-1}]}\to \iPsh{\Delta[\Theta_{n-1}]}$ the two canonical inclusions. By the proposition \ref{prop:link beetwenwidehat and overline}, we have inclusions $\iota(\overline{\W_{n}})\subset \widehat{\W_{n}}$ and $\iota(\overline{\M_{n}})\subset \widehat{\M_{n}}$. The result then directly follows from theorem \ref{theo:unit and counit are in W}.  
\end{proof}

\begin{definition}
 Let $n>0$ be an integer. An \wcnotion{$(\infty,n)$-category}{category3@$(\infty,n)$-category} is a $\W_n$-local $\infty$-presheaf $C\in \iPsh{\Theta_n}$. We then define \sym{((a50@$\ncat{n}$}
$$\ncat{n} := \iPsh{\Theta_n}_{\W_n}.$$
Proposition \ref{prop:infini changing theta n} implies that $\ncat{n}$ identifies itself with the full sub $\iun$-category of $\iPsh{\Delta[\Theta_{n-1}]}$ of $\M_n$-local objects:
$$\ncat{n} \sim \iPsh{\Delta[\Theta_{n-1}]}_{\M_n}.$$
\end{definition}

\begin{construction}
The inclusion $i_n:\Theta_n\to \Theta$ fits in an adjunction
\[\begin{tikzcd}
	{\tau^i_n:\Theta} & {\Theta_n:i_n}
	\arrow[""{name=0, anchor=center, inner sep=0}, shift left=2, from=1-1, to=1-2]
	\arrow[""{name=1, anchor=center, inner sep=0}, shift left=2, from=1-2, to=1-1]
	\arrow["\dashv"{anchor=center, rotate=-90}, draw=none, from=0, to=1]
\end{tikzcd}\]
where the left adjoint sends $\Db_k$ on $\Db_{\min{(n,k)}}$.
By extension by colimits, this induces an adjoint pair 
\begin{equation}
\label{eq:inclusion of n cat pre}
\begin{tikzcd}
	{\tau^i_n:\iPsh{\Theta}} & {\iPsh{\Theta_n}:i_n.}
	\arrow[""{name=0, anchor=center, inner sep=0}, shift left=2, from=1-1, to=1-2]
	\arrow[""{name=1, anchor=center, inner sep=0}, shift left=2, from=1-2, to=1-1]
	\arrow["\dashv"{anchor=center, rotate=-90}, draw=none, from=0, to=1]
\end{tikzcd}
\end{equation}
where the two functors are colimit preserving.
As the image of every morphism of $\W$ by $\tau^i_n$ is in $\W_n$ or is an equivalence, and as the image of $\W_n$ by $i_n$ is included in $\W$, the previous adjunction induces by localization an adjunction
\begin{equation}
\label{eq:inclusion of n cat}
\begin{tikzcd}
	{\tau^i_n:\ocat} & {\ncat{n}:i_n}
	\arrow[""{name=0, anchor=center, inner sep=0}, shift left=2, from=1-1, to=1-2]
	\arrow[""{name=1, anchor=center, inner sep=0}, shift left=2, from=1-2, to=1-1]
	\arrow["\dashv"{anchor=center, rotate=-90}, draw=none, from=0, to=1]
\end{tikzcd}
\end{equation}
where the two adjoints are colimit preserving.
The left adjoint is called the \snotionsym{intelligent $n$-truncation}{(taui@$\tau^i_n$}{for $\io$-categories}.
\end{construction}
\begin{prop}
\label{ref:infini n a full sub cat}
The functor $i_n: \ncat{n}\to \ocat$ is fully faithful.
\end{prop}
\begin{proof}
We have to check that the unit of the adjunction \eqref{eq:inclusion of n cat} is an equivalence. As the two functors preserve colimits, we have to show that the restriction to $\Theta$ of the unit is an equivalence which is obvious.
\end{proof}
\begin{construction}
Being colimit preserving, the functor $i_n$ is also part of an adjunction
\begin{equation}
\begin{tikzcd}
	{i_n:\ncat{n}} & {\ocat:\tau_n}
	\arrow[""{name=0, anchor=center, inner sep=0}, shift left=2, from=1-2, to=1-1]
	\arrow[""{name=1, anchor=center, inner sep=0}, shift left=2, from=1-1, to=1-2]
	\arrow["\dashv"{anchor=center, rotate=-90}, draw=none, from=1, to=0]
\end{tikzcd}
\end{equation}
The right adjoint is called the \wcsnotionsym{$n$-truncation}{(tau@$\tau_n$}{truncation@$n$-truncation}{for $\io$-category}. 
\end{construction}

\begin{notation}
We will identify objects of $\ncat{n}$ with their image in $\ocat$ and we will then also note by $\tau_n$ and $\tau^i_n$ the composites $i_n\tau^i_n$ and $i_n\tau^i_n$.
\end{notation}

\begin{prop}
\label{prop:taun preserves special colimits}
The functor $\tau_n:\ocat\to \ocat$ preserves special colimits.
\end{prop}
\begin{proof}

As $i_n$ preserves representable objects, the functor $\tau_n:\ocat\to \ncat{n}$ preserves colimits. Since $i_n:\iPsh{\Theta_n}\to \iPsh{\Theta}$ preserves colimits and sends $\mathcal{W}_n$ to a subset of $\mathcal{W}$, $\tau_n$ preserves $\io$-categories. This concludes the proof.
\end{proof}

\begin{definition}
\label{defi:trivialization}
A \notion{trivialization map} is a morphism in the smallest cocomplete class of morphisms including the morphisms $\Ib:\Db_{n+1}\to \Db_n$ for any integer $n$.
\end{definition}

\begin{example}
For all $\io$-category $C$, the canonical morphism $C\to \tau^i_nC$ is a trivialization.
\end{example}

\begin{prop}
\label{prop:inteligent trucatio and a particular colimit}
Trivialization maps are epimorphisms (definition \ref{defi:epi abstrait}).
\end{prop}
\begin{proof}
Let $i:C\to D$ be a trivialization map. We have to show that the square
\[\begin{tikzcd}
	C & D \\
	D & D
	\arrow[from=1-1, to=1-2]
	\arrow[from=1-1, to=2-1]
	\arrow[from=1-2, to=2-2]
	\arrow[from=2-1, to=2-2]
\end{tikzcd}\]
is cocartesian.
Since cocartesian squares are closed under colimits, we can reduce to the case where $C\to D$ is $\Db_{n+1}\to \Db_{n}$ for an integer $n$. We then need to show that the square
\[\begin{tikzcd}
	{\Db_{n+1}} & {\Db_n} \\
	{\Db_n} & {\Db_n}
	\arrow["\Ib", from=1-1, to=1-2]
	\arrow["\Ib"', from=1-1, to=2-1]
	\arrow[from=1-2, to=2-2]
	\arrow[from=2-1, to=2-2]
\end{tikzcd}\]
is a pushout. By definition, this corresponds to showing that for any $\io$-category $A$, the canonical morphism 
$$\alpha:\Hom(\Db_n,A)\to \Hom(\Db_n,A)\times_{\Hom(\Db_{n+1},A)}\Hom(\Db_n,A)$$ is an equivalence. 
We consider the morphism
$$\beta:\Hom(\Db_n,A)\times_{\Hom(\Db_{n+1},A)}\Hom(\Db_n,A)\to \Hom(\Db_{n+1},A)\times_{\Hom(\Db_{n+1},A)}\Hom(\Db_n,A)\sim\Hom(\Db_n,A)$$
that sends a triplet $(p,q,\psi:p\circ\Ib\sim q\circ\Ib)$ to $(p\circ\Ib,q,\psi:p\circ\Ib\sim q\circ\Ib)$. As $\beta$ is a projection, we have $\beta \alpha\sim id$. 
Furthermore, as the endomorphism $\Ib\circ i_{n+1}^+:\Db_n\to \Db_n$ is the identity, $\alpha$ is equivalent to the morphism 
$$\Hom(\Db_n,A)\sim \Hom(\Db_{n+1},A)\times_{\Hom(\Db_{n+1},A)}\Hom(\Db_n,A)\to \Hom(\Db_n,A)\times_{\Hom(\Db_{n+1},A)}\Hom(\Db_n,A)$$
that sends $(p',q,\psi:p'\sim q\circ\Ib)$ to $(p'\circ i_{n+1}^+,q,\psi:p\circ\Ib\circ i_{n+1}^+\sim q\circ\Ib)$
and we then also have $\alpha\beta\sim id$.
\end{proof}

\begin{construction}
 The family of truncation functor induces a sequence 
$$...\to \ncat{n+1}\xrightarrow{\tau_{n}} \ncat{n}\to...\to \ncat{1}\xrightarrow{\tau_{0}}\ncat{0}$$
which induces an adjunction
\begin{equation}
\label{eq:inductivity}
\begin{tikzcd}
	{\colim_{n:\Nb}:\lim_{n:\Nb}\ncat{n}} & {\ocat:(\tau_n)_{n:\Nb}}
	\arrow[""{name=0, anchor=center, inner sep=0}, shift left=2, from=1-1, to=1-2]
	\arrow[""{name=1, anchor=center, inner sep=0}, shift left=2, from=1-2, to=1-1]
	\arrow["\dashv"{anchor=center, rotate=-90}, draw=none, from=0, to=1]
\end{tikzcd}
\end{equation}
where the left adjoint sends a sequence $(C_n, C_n\sim \tau_nC_{n+1})_{n:\Nb}$ to the colimit of the induced sequence
$$i_0C_0\to i_1C_1\to ... \to i_nC_n\to ..., $$
and the right adjoint sends an $\io$-category $C$ to the sequence $(\tau_nC,\tau_nC\sim \tau_{n}\tau_{n+1}C)_{n:\Nb}$. Indeed, we have equivalence
$$
\begin{array}{rcl}
\Hom(\colim_{n:\Nb}i_n C_n,D)&\sim& \lim_{n:\Nb}\Hom(C_n,\tau_n D)
\\&\sim& \Hom( (C_n, C_n\sim \tau_nC_{n+1})_{n:\Nb},(\tau_n D,\tau_n D\sim \tau_{n}\tau_{n+1}D)_{n:\Nb})
\end{array}$$
natural in $(C_n, C_n\sim \tau_nC_{n+1})_{n:\Nb}$ and $D$.
\end{construction}

\begin{prop}
\label{prop:infini omega a limit of infini n}
The adjunction \eqref{eq:inductivity} is an adjoint equivalence. As a consequence, we have an equivalence
$$\ocat\sim \lim_{n:\Nb}\ncat{n}.$$
\end{prop}
\begin{proof}
According to proposition \ref{prop:example of a special colimit 2}, any sequence $(C_n)_{n:\Nb}:\lim_{n:\Nb}\ncat{n}$ has a special colimit.
Let $k$ be an integer. According to proposition \ref{prop:taun preserves special colimits}, this implies the equivalence
$$\tau_k(\colim_{n:\Nb}C_n) \sim \colim_{n:\Nb}(\tau_kC_n).$$
Furthermore, the sequence $(\tau_kC_n)_{n:\Nb}$ is constant after the rank $k$. We then have 
$$\tau_k\colim_{n:\Nb}C_n \sim \tau_k C_n.$$
This directly implies that the unit of the adjunction \eqref{eq:inductivity} is an equivalence. 

To conclude, one has to show that the right adjoint is conservative, i.e that a morphism $f$ is an equivalence if and only if for any $n$, $\tau_n f$ is an equivalence. This last statement is a direct consequence of proposition \ref{prop:equivalences detected on globes}.
\end{proof}

\subsubsection{$\ocat$ is Cartesian closed}

\begin{prop}
\label{prop:cartesian product preserves W}
The cartesian product in $\ocat$ preserves colimits in both variables.
\end{prop}

\begin{definition}
For any $C$ in $\ocat$, we define 
$$\uHom(C,\uvar):\ocat\to \ocat$$
the right adjoint of the colimit-preserving functor $\uvar\times C:\ocat\to \ocat$.
\end{definition}

We  need several lemmas to prove proposition \ref{prop:cartesian product preserves W}.

\begin{lemma}
\label{lemma:product of representable in preshaves on Delta Theta}
Let $a$, $b$ be two globular sums, and $n,m$ two integer. The colimit in $\iPsh{\Delta[\Theta]}$ of the diagram 
\[\begin{tikzcd}
	{\coprod_{k\leq n}[a\times b,\{k\}\times [m]]} && {\coprod_{l\leq m}[a\times b,[n]\times \{l\}]} \\
	{\coprod_{k\leq n}[b,m]} & {[a\times b,[n]\times [m]]} & {\coprod_{l\leq m}[a,n]}
	\arrow[from=1-1, to=2-1]
	\arrow[from=1-1, to=2-2]
	\arrow[from=1-3, to=2-2]
	\arrow[from=1-3, to=2-3]
\end{tikzcd}\]
is $[a,n]\times [b,m]$.
\end{lemma}
\begin{proof}
The lemma \ref{lemma:colimit computed in set presheaves} implies that the object 
$$K:=\coprod_{k\leq n}[b,m]\coprod_{\coprod_{k\leq n}[a\times b,\{k\}\times [m]]}[a\times b,[n]\times [m]]$$
is strict. As the induced morphism 
$\coprod_{l\leq m}[a\times b,[n]\times \{l\}]\to K$, is a monomorphism, the lemma \textit{op cit} implies that the colimit of the diagram given in the statement is strict. We can then show the result in the category of set valued presheaves on $ \Delta[\Theta]$ and we leave this combinatorial exercise to the reader.
\end{proof}

\begin{lemma}
\label{lemma:technical cartesian product preserves W}
Let $f$ be a morphism of $\W_1$ and $n$ an integer. The morphism $f\times [n]$ is in $\widehat{\W_1}$.
\end{lemma}
\begin{proof}
Suppose first that $f$ is of shape $\Sp_m\to [m]$. Remark first that for any $k$, the $\Delta$-space $[k]\times [m]$ is $\W_1$-local as both $[k]$ and $[m]$ are. We then have $\Fb_{\W_1}([k]\times[m])\sim [k]\times [m]$ where $\Fb_{\W_1}$ is the left adjoint appearing in the adjunction given in theorem \ref{theo:adjunction between presheaves and local presheaves}. 
As the fibrant replacement preserves colimits and as the cartesian product in $\iun$-categories preserves colimits, we have a sequence of equivalences in $\icat$:
$$
\begin{array}{rcl}
\Fb_{\W_1}(\Sp_m\times [n])&\sim& \Fb_{\W_1}([1]\times [n])\coprod_{ \Fb_{\W_1}([0]\times [n])}...\coprod_{ \Fb_{\W_1}([0]\times [n])}\Fb_{\W_1}([1]\times [n])\\
&\sim& [1]\times [n]\coprod_{ [0]\times [n]}...\coprod_{ [0]\times [n]} [1]\times [n]\\
&\sim & [m]\times [n]
\end{array}
$$
As by construction the morphism $\Sp_m\times [n]\to \Fb_{\W_1}(\Sp_m\times [n])$ is in $\widehat{\W_1}$, this concludes the proof of this case. We proceed similarly for the case $f:=E^{eq}\to [0]$.
\end{proof}

\begin{proof}[Proof of proposition \ref{prop:cartesian product preserves W}]
As the cartesian product on $\iPsh{\Theta}$ preserves colimits in both variables, according to corollary \ref{cor:derived colimit preserving functor}, we then have to show that for any globular sum $a$, and any $f\in\W$, $f\times a$ is in $\widehat{\W}$.

We demonstrate by induction on $k$ that for any $f\in\W_k$ and any globular sum $a$, $f\times a$ is in $\W_k$. The case $k=0$ is trivial as $\W_0$ is the singleton $\{id_{[0]}\}$.

Suppose then the statement is true at this stage $k$. We recall that we denote $(i_!,i^*)$ the left and right adjoints between $\iPsh{\Delta[\Theta]}$ and $\iPsh{\Theta}$. As $i^*$ preserves cartesian product, proposition \ref{prop:infini changing theta} implies that it is enough to show that for any $f\in\M_{k+1}$ and any object $[b,n]$, $f\times [b,n]$ is in $\widehat{\M}$. 

Suppose first that $f$ is of shape $[a,1]\to [c,1]$ for $a\to c \in \W_k$. According to lemma \ref{lemma:product of representable in preshaves on Delta Theta}, the morphism $f\times[b,m]$ is the colimit in depth of the diagram 
\[\begin{tikzcd}[column sep =0.1cm]
	{\coprod_{k\leq 1}[a\times b,\{k\}\times [m]]} && {\coprod_{l\leq m}[a\times b,[1]\times \{l\}]} \\
	{\coprod_{k\leq 1}[b,m]} & {[a\times b,[1]\times [m]]} & {\coprod_{l\leq m}[a,1]} \\
	& {\coprod_{k\leq 1}[c\times b,\{k\}\times [m]]} && {\coprod_{l\leq m}[c\times b,[1]\times \{l\}]} \\
	& {\coprod_{k\leq 1}[b,m]} & {[c\times b,[1]\times [m]]} & {\coprod_{l\leq m}[c,1]}
	\arrow[from=1-1, to=2-1]
	\arrow[from=1-1, to=2-2]
	\arrow[from=1-3, to=2-2]
	\arrow[from=1-3, to=2-3]
	\arrow[from=2-1, to=4-2]
	\arrow[from=1-1, to=3-2]
	\arrow[from=2-2, to=4-3]
	\arrow[from=2-3, to=4-4]
	\arrow[from=1-3, to=3-4]
	\arrow[from=3-2, to=4-3]
	\arrow[from=3-4, to=4-3]
	\arrow[from=3-4, to=4-4]
	\arrow[from=3-2, to=4-2]
\end{tikzcd}\]
The lemma \ref{lemma:the functor [] preserves classes} and the induction hypothesis implies that all the depth morphisms are in $\widehat{\M}$.
By stability by colimit, this implies that $f\times[b,m]$ belongs to $\widehat{\M}$.

Suppose now that $f$ is of shape $[a,\Sp_n]\to [a,n]$. According to lemma \ref{lemma:product of representable in preshaves on Delta Theta}, the morphism $f\times[b,m]$ is the colimit in depth of the diagram 
\[\begin{tikzcd}[column sep = 0.1cm]
	{\coprod_{k\leq n}[a\times b,\{k\}\times [m]]} && {\coprod_{l\leq m}[a\times b,\Sp_n\times \{l\}]} \\
	{\coprod_{k\leq n}[b,m]} & {[a\times b,\Sp_n\times [m]]} & {\coprod_{l\leq m}[a,\Sp_n]} \\
	& {\coprod_{k\leq n}[a\times b,\{k\}\times [m]]} && {\coprod_{l\leq m}[a\times b,[n]\times \{l\}]} \\
	& {\coprod_{k\leq n}[b,m]} & {[a\times b,[n]\times [m]]} & {\coprod_{l\leq m}[a,n]}
	\arrow[from=1-1, to=2-1]
	\arrow[from=1-1, to=2-2]
	\arrow[from=1-3, to=2-2]
	\arrow[from=1-3, to=2-3]
	\arrow[from=3-2, to=4-3]
	\arrow[from=2-3, to=4-4]
	\arrow[from=1-3, to=3-4]
	\arrow[from=3-4, to=4-4]
	\arrow[from=3-4, to=4-3]
	\arrow[from=3-2, to=4-2]
	\arrow[from=2-1, to=4-2]
	\arrow[from=1-1, to=3-2]
	\arrow[from=2-2, to=4-3]
\end{tikzcd}\]
The lemma \ref{lemma:technical cartesian product preserves W} implies that $\Sp_n\times [m]\to [n]\times [m]$ is in $\widehat{\W_1}$. Combined with lemma \ref{lemma:the functor [] preserves classes}, this implies that all the morphisms in depth are in $\widehat{\M}$. By stability by colimit, so is $f\times[b,m]$.

It remains to show the case $f= E^{eq}\to [0]$. According to lemma \ref{lemma:product of representable in preshaves on Delta Theta}, the morphism $f\times[b,m]$ is the horizontal colimit of the diagram
\[\begin{tikzcd}
	{\coprod_{k\leq m} E^{eq}} & {\coprod_{k\leq m} [b,E^{eq}\times \{k\}]} & { [b,E^{eq}\times [m]]} \\
	{\coprod_{k\leq m}[0]} & {\coprod_{k\leq m}[0]} & {[b,m]}
	\arrow[from=1-2, to=1-3]
	\arrow[from=1-2, to=1-1]
	\arrow[from=2-2, to=2-1]
	\arrow[from=2-2, to=2-3]
	\arrow[from=1-3, to=2-3]
	\arrow[from=1-2, to=2-2]
	\arrow[from=1-1, to=2-1]
\end{tikzcd}\]
The lemma \ref{lemma:technical cartesian product preserves W} implies that $ E^{eq}\times [m]\to [m]$ is in $\widehat{\W_1}$. Combined with lemma \ref{lemma:the functor [] preserves classes}, this implies that all the vertical morphisms are in $\widehat{\M}$. By stability by colimit, so is $f\times[b,m]$.
\end{proof}

\begin{cor}
\label{cor:if codomain a groupoid, then f is exponentiable}
Let $C$ be an $\io$-category, $S$ an $\infty$-groupoid, and $f:C\to S$ any morphism.
The functor $f^*:\ocat_{/S}\to \ocat_{/C}$ preserves colimits. 
\end{cor}
\begin{proof}
As $\iPsh{\Theta}$ is locally cartesian closed, we just have to verify that for any cartesian squares:
\[\begin{tikzcd}
	{C''} & {C'} & C \\
	a & b & S
	\arrow["i"', from=2-1, to=2-2]
	\arrow[from=1-3, to=2-3]
	\arrow["j", from=1-1, to=1-2]
	\arrow[from=1-2, to=1-3]
	\arrow[from=2-2, to=2-3]
	\arrow[from=1-1, to=2-1]
	\arrow[from=1-2, to=2-2]
	\arrow["\lrcorner"{anchor=center, pos=0.125}, draw=none, from=1-1, to=2-2]
	\arrow["\lrcorner"{anchor=center, pos=0.125}, draw=none, from=1-2, to=2-3]
\end{tikzcd}\]
if $i$ is in $\W$, then $j$ is in $\widehat{\W}$. Suppose given such cartesian squares. As $b$ is a globular form, $\tau^i_0(b)\sim 1$ and 
as $S$ is an $\infty$-groupoid, there exists an object $s$ of $S$ such that the morphism $b\to S$ factor through $\{s\}\to S$. If we denote by $C_s$ the fiber of $f$ in $\{s\}$, the morphisms $i$ and $j$ then fit in the following cartesian squares:
\[\begin{tikzcd}
	{C_s\times a} & {C_s\times b} & {C_s} & C \\
	a & b & {\{s\}} & S
	\arrow["i"', from=2-1, to=2-2]
	\arrow[from=1-3, to=2-3]
	\arrow["j", from=1-1, to=1-2]
	\arrow[from=1-2, to=1-3]
	\arrow[from=2-2, to=2-3]
	\arrow[from=1-1, to=2-1]
	\arrow[from=1-2, to=2-2]
	\arrow["\lrcorner"{anchor=center, pos=0.125}, draw=none, from=1-1, to=2-2]
	\arrow["\lrcorner"{anchor=center, pos=0.125}, draw=none, from=1-2, to=2-3]
	\arrow[from=2-3, to=2-4]
	\arrow[from=1-3, to=1-4]
	\arrow[from=1-4, to=2-4]
	\arrow["\lrcorner"{anchor=center, pos=0.125}, draw=none, from=1-3, to=2-4]
\end{tikzcd}\]
The proposition \ref{prop:cartesian product preserves W} implies that $j$ verifies the desired property, which concludes the proof.
\end{proof}

The following proposition implies that a natural transformation is an equivalence if and only if it is pointwise one. 
\begin{prop}
\label{prop:cartesian square and times}
For any $\io$-categories $X$ and $C$, the following natural square is cartesian:
\[\begin{tikzcd}
	{\tau_0\uHom(X,C)} & {\uHom(X,C)} \\
	{\uHom(\tau_0X,\tau_0C)} & {\uHom(\tau_0X,C)}
	\arrow[from=1-1, to=2-1]
	\arrow[from=1-2, to=2-2]
	\arrow[from=2-1, to=2-2]
	\arrow[from=1-1, to=1-2]
\end{tikzcd}\]
\end{prop}
\begin{proof}
By adjunction, we have to show that for any $X$ and $Y$, the square
\[\begin{tikzcd}
	{\tau_0Y\times X} & {Y\times X} \\
	{\tau_0Y\times \tau_0^i X} & {Y\times \tau_0^iX}
	\arrow[from=1-1, to=1-2]
	\arrow[from=1-1, to=2-1]
	\arrow[from=1-2, to=2-2]
	\arrow[from=2-1, to=2-2]
\end{tikzcd}\]
is cocartesian. As $\tau_0^i$ preserves colimits, and as $\tau_0$ preserves special colimits, we can reduce to the case where $X$ and $Y$ are representable. Since these two functors send $\Wseg$ to weak equivalences, we can further reduce to the case where $X$  is $\Db_n$ and $Y$ is $\Db_m$ for $n$ and $m$ two integers. We are then reduced to showing that the following square is cocartesian.
\begin{equation}
\label{eq:proof of cartesian}
\begin{tikzcd}
	{(\tau_0\Db_n)\times\Db_m} & {\Db_n\times\Db_m} \\
	{\tau_0\Db_n} & {\Db_n}
	\arrow[from=1-1, to=2-1]
	\arrow[from=2-1, to=2-2]
	\arrow[from=1-1, to=1-2]
	\arrow[from=1-2, to=2-2]
\end{tikzcd}
\end{equation}
To show the cocartesianess of \eqref{eq:proof of cartesian}, remark that if either $n$ or $m$ is null, this is trivial. If not, proposition \ref{prop:example of a special colimit} states that $\Db_n\times\Db_m$ is the colimit of the span:
$$[\Db_{n-1},1]\vee[\Db_{m-1},1]\leftarrow [\Db_{n-1}\times \Db_{m-1},1]\to [\Db_{m-1},1]\vee[\Db_{n-1},1]$$
Using the two cartesian squares
\[\begin{tikzcd}
	{[\Db_{m-1},1]} & {[\Db_{m-1},1]\vee[\Db_{n-1},1]} & {[\Db_{m-1},1]} & {[\Db_{n-1},1]\vee[\Db_{m-1},1]} \\
	{[0]} & {[\Db_{n-1},1]} & {[0]} & {[\Db_{n-1},1]}
	\arrow[from=1-3, to=2-3]
	\arrow[from=1-3, to=1-4]
	\arrow[from=2-3, to=2-4]
	\arrow[from=1-4, to=2-4]
	\arrow[from=1-2, to=2-2]
	\arrow["\lrcorner"{anchor=center, pos=0.125, rotate=180}, draw=none, from=2-4, to=1-3]
	\arrow[from=1-1, to=2-1]
	\arrow[from=2-1, to=2-2]
	\arrow[from=1-1, to=1-2]
	\arrow["\lrcorner"{anchor=center, pos=0.125, rotate=180}, draw=none, from=2-2, to=1-1]
\end{tikzcd}\]
this implies that the pushout of the upper span of \eqref{eq:proof of cartesian} is then the colimit of the diagram:
\begin{equation}
\label{eq:proof of cartesian2}
[\Db_{n-1},1]\leftarrow [\Db_{n-1}\times \Db_{m-1},1]\to [\Db_{n-1},1]
\end{equation}
The proposition \ref{prop:inteligent trucatio and a particular colimit} states that the square
\[\begin{tikzcd}
	{ \Db_{m-1}} & 1 \\
	1 & 1
	\arrow[from=2-1, to=2-2]
	\arrow[from=1-1, to=1-2]
	\arrow[from=1-1, to=2-1]
	\arrow[from=1-2, to=2-2]
\end{tikzcd}\]
is cocartesian. Combined with proposition \ref{prop:cartesian product preserves W}, this implies that the square
\[\begin{tikzcd}
	{\Db_{n-1}\times \Db_{m-1}} & {\Db_{n-1}} \\
	{\Db_{n-1}} & {\Db_{n-1}}
	\arrow[from=2-1, to=2-2]
	\arrow[from=1-1, to=1-2]
	\arrow[from=1-1, to=2-1]
	\arrow[from=1-2, to=2-2]
\end{tikzcd}\]
is cocartesian. 
As a consequence, the colimit of the span \eqref{eq:proof of cartesian2}, and so of the upper span of \eqref{eq:proof of cartesian}, is $ [\Db_{n-1},1]\sim \Db_n$, which concludes the proof. 
\end{proof}

\subsubsection{Dualities}
\begin{definition}
\label{defi:dualities non strict case}
In definition \ref{defi:dualities strict case}, for any subset $S$ of $\Nb^*$, we have defined the duality $(\uvar)^S:\zocat\to \zocat$.
 These functors restrict to functors $\Theta\to \Theta$ that induce by extension by colimit functors \ssym{((b49@$(\uvar)^S$}{for $\io$-categories}
$$(\uvar)^S:\iPsh{\Theta}\to \iPsh{\Theta}$$
which are once again called \snotion{dualities}{for $\io$-categories}.
It is easy to see that this functor preserves $\io$-categories and then induces functors
$$(\uvar)^S:\ocat\to \ocat.$$

In particular, we have the \snotionsym{odd duality}{((b60@$(\uvar)^{op}$}{for $\io$-categories} $(\uvar)^{op}$, corresponding to the set of odd integer, the \snotionsym{even duality}{((b50@$(\uvar)^{co}$}{for $\io$-categories} $(\uvar)^{co}$, corresponding to the subset of non negative even integer, the \snotionsym{full duality}{((b80@$(\uvar)^{\circ}$}{for $\io$-categories} $(\uvar)^{\circ}$, corresponding to $\Nb^*$ and the \snotionsym{transposition}{((b70@$(\uvar)^t$}{for $\io$-categories} $(\uvar)^t$, corresponding to the singleton $\{1\}$. Eventually, we have equivalences
$$((\uvar)^{co})^{op}\sim (\uvar)^{\circ} \sim ((\uvar)^{op})^{co}.$$
\end{definition}

\subsubsection{Monomorphisms and surjections}

\begin{definition} 
\label{defi:mono non-abstrait}
A morphism $f:C\to D$ is an \notion{surjection} if it is in the smallest cocomplete $\infty$-groupoid of arrows of $\ocat$ that includes the codiagonal $\Db_n\coprod\Db_n\to \Db_n$ for any $n\geq 0$. A morphism is a \notion{monomorphism} if it has the unique right lifting property against surjections.
\end{definition}

\begin{remark}
 A morphism $i:C\to D$ is then a monomorphism if and only if for any $n$, $C_n\to D_n$ is a monomorphism.
 \end{remark}
 
\begin{remark}
\label{new texte}
As equivalence are detected on globes, the morphism
$$D\to C\times_DC$$ is an equivalence if and only if 
$$\Hom(\Db_n,D)\to \Hom(\Db_n,C)\times_{\Hom(\Db_n,D)}\Hom(\Db_n,C)$$ 
is an equivalence 
for any integer $n$. As this last condition is equivalent to having the unique right lifting property against $\Db_n\coprod\Db_n\to \Db_n$,  monomorphisms in the sense of definition \ref{defi:mono non-abstrait} correspond to monomorphisms in the sens of definition \ref{defi:mono abstrait}.
However, surjections are not necessarily epimorphisms.
\end{remark}

 \begin{construction}
The small object argument induces a factorization system:
\begin{equation}
\label{eq:epimonomorphism factorization}
C\to \im i\to D
\end{equation}
of any morphism $i:C\to D$, where the left map is a surjection, and the right one is a monomorphism. The object \wcnotation{$\im i$}{(im@$\im$} is called the \wcnotion{image of $i$}{image of a morphism}. 
\end{construction}

We then have by construction the following result:

\begin{prop}
A morphism is an equivalence if and only if it is both a monomorphism and a surjection.
\end{prop}

\begin{prop}
\label{prop:the image is stable under cartesian product}
The image is stable under the cartesian product.
\end{prop}
\begin{proof}
One has to show that both surjections and monomorphisms are stable under the functor $\uvar\times A$ for $A$ any $\io$-category. For monomorphisms, it is a direct consequence of the fact that this notion has been defined with a right lifting property. For surjections, as $\uvar\times A$ commutes with colimit, we can reduce to show that for any $n$, 
$$(\Db_n\coprod\Db_n)\times A \sim \Db_n\times A\coprod\Db_n\times A\to \Db_n \times A$$
is a surjection. 
However, the $\infty$-groupoid of object $B$ such that 
$B\coprod B\to B$ is a surjection is closed by colimits and contains globes. This $\infty$-groupoid then contains all the object and so in particular $\Db_n \times A$.
\end{proof}

\begin{lemma}
\label{lemma:id is an epi}
For any integer $n$, the morphism $\Ib:\Db_{n+1}\to \Db_n$ is a surjection. 
\end{lemma}
\begin{proof}
Remark first that we have a cocartesian square:
\[\begin{tikzcd}
	{\partial\Db_n\coprod \partial\Db_n} & {\Db_n\coprod\Db_n} \\
	{\partial\Db_n} & {\partial\Db_{n+1}}
	\arrow[""{name=0, anchor=center, inner sep=0}, from=1-1, to=1-2]
	\arrow[from=1-1, to=2-1]
	\arrow[from=1-2, to=2-2]
	\arrow[from=2-1, to=2-2]
	\arrow["\lrcorner"{anchor=center, pos=0.125, rotate=180}, draw=none, from=2-2, to=0]
\end{tikzcd}\]
As the left hand morphism is a surjection, so is the right one. By stability by left cancellation, this implies that $\partial\Db_{n+1}\to \Db_n$ is a surjection.
Now, the map $\Ib$ can be factored as: 
\[\begin{tikzcd}
	{\partial\Db_{n+1}} & {\Db_n} \\
	{\Db_{n+1}} & {\Db_{n+1}\coprod_{\partial\Db_{n+1}}\Db_n} & {\Db_n} \\
	& {\partial\Db_{n+2}} & {\Db_{n+1}}
	\arrow[from=1-1, to=2-1]
	\arrow[""{name=0, anchor=center, inner sep=0}, from=1-1, to=1-2]
	\arrow[from=2-1, to=2-2]
	\arrow[from=1-2, to=2-2]
	\arrow[from=3-2, to=2-2]
	\arrow[from=2-2, to=2-3]
	\arrow[from=3-2, to=3-3]
	\arrow["\lrcorner"{anchor=center, pos=0.125, rotate=-90}, draw=none, from=2-3, to=3-2]
	\arrow[from=3-3, to=2-3]
	\arrow["\lrcorner"{anchor=center, pos=0.125, rotate=180}, draw=none, from=2-2, to=0]
\end{tikzcd}\]
and by stability under pushouts and composition, this  implies that $\Ib$ is a surjection. 
\end{proof}

\begin{prop}
\label{prop:intelignet truncation is poitwise an epi}
For any integer $n$, the canonical natural transformation $id\to \tau^i_n$ is pointwise a surjection. 
\end{prop}
\begin{proof}
This is a direct consequence of lemma \ref{lemma:id is an epi}.
\end{proof}

\begin{prop}
\label{prop:canonical epi}
For any integer $n$, any $(\infty,n)$-category $C$, and any $\io$-category $D$, the canonical morphisms
$$\alpha:\coprod_{C_n}\Db_n\to C~~~~~\beta:\coprod_{n\in \Nb}\coprod_{D_n}\Db_n\to D$$
are surjections.
\end{prop}
\begin{proof}
Let $I$ be the image of $\alpha$. We are willing to show that the canonical morphism $j:I\to C$ is an equivalence.
According to lemma \ref{lemma:equivalence if unique right lifting property against globes.}, we have to show that $j$ has the unique right lifting property against $\emptyset\to \Db_k$ for any $k\leq n$. As $j$ is a monomorphism, lifts against it are always unique according to proposition \ref{prop:mono 2}, and it is then sufficient to show that $\alpha$ has the non-unique right lifting property against $\emptyset\to \Db_k$ for any $k\leq n$, which is obviously true.

We proceed similarly for $\beta$.
\end{proof}

\begin{prop}
\label{prop:truncation of surjection is pushout}
Let $i:A\to B$ be a surjection and $n$ an integer. The canonical square 
\[\begin{tikzcd}
	A & B \\
	{\tau^i_n(A)} & {\tau^i_n(B)}
	\arrow[from=1-1, to=2-1]
	\arrow["{\tau^i_n(i)}"', from=2-1, to=2-2]
	\arrow["i", from=1-1, to=1-2]
	\arrow[from=1-2, to=2-2]
\end{tikzcd}\]
is cocartesian. 
\end{prop}
\begin{proof}
We can reduce to the case where $i$ is $\Db_k\coprod\Db_k\to \Db_k$. If $n\geq k$, it is directly true, and we then suppose $n<k$. In this case, the colimit of the span:
$$\Db_n\coprod \Db_n \leftarrow \Db_k\coprod\Db_k\to \Db_k$$
is $\Db_n\coprod_{\Db_k}\Db_n$. The proposition \ref{prop:inteligent trucatio and a particular colimit} implies that this pushout is $\Db_n$, which concludes the proof.
\end{proof}

\subsubsection{Fully faithful functors}

\begin{definition}
A functor $f:C\to D$ is \snotion{fully faithful}{for $\io$-categories} if for any pair of objects $a,b\in C$, the induced morphism 
$\hom_C(a,b)\to \hom_D(fa,fb)$ is an equivalence. 
\end{definition}

\begin{prop}
\label{prop:ff 1}
A functor is fully faithful if and only if it has the unique right lifting property against $\{0\}\coprod \{1\}\to \Db_n$ for $n>0$.
\end{prop}
\begin{proof}
Let $f:C\to D$ be a functor having the unique right lifting property against $\{0\}\coprod \{1\}\to \Db_n$ for $n>0$. As $[\emptyset,1] =\{0\}\coprod \{1\}$ and $[\Db_n,1] = \Db_{n+1}$, 
this is equivalent to asking for any pair of objects $c,d$ and for any integer $n$, that $f(c,d):\hom_C(c,d)\to \hom_D(f(c),f(d)$ has the unique right lifting property against $\emptyset\to \Db_n$, which in turn is equivalent to $f$ being fully faithful according to lemma \ref{lemma:equivalence if unique right lifting property against globes.}.
\end{proof}

\begin{prop}
\label{prop:ff 2}
Fully faithful functors are stable under limits.
\end{prop}
\begin{proof}
This is a consequence of the fact that fully faithful functors are characterized by unique right lifting properties.
\end{proof}

\begin{lemma}
\label{lemma:ff 2}
Let $p:C\to D$ be a fully faithful functor. The induced morphism $C_0\to D_0$ is a monomorphism.
\end{lemma}
\begin{proof}
To this extent, we have to show that $p:C\to D$ has the unique right lifting property against $1\coprod 1\to 1$. This is equivalent to show that $p$ has the unique right lifting property against $\iota: 1\coprod 1 \to E^{eq}$.

The proposition \ref{prop:ff 1} implies that $p$ as the unique right lifting property against $1\coprod 1\to \Db_1$. By stability under colimits, $p$ has the unique right lifting property against $1\coprod 1\to \partial \Db_2$.
By left cancellation, this implies that  $p$ has the unique right lifting property against $\partial\Db_2\to \Db_1$. As $\iota$ is a composition of pushouts along  $1\coprod 1\to \Db_1$ and  $\partial\Db_2\to \Db_1$, this directly concludes the proof.
\end{proof}

\begin{prop}
\label{prop:fully faithful plus surjective on objet}
A morphism $f:C\to D$ is an equivalence if and only if it is fully faithful and induces a surjection on objects.
\end{prop}
\begin{proof}
This is necessary. Suppose that $f$ is fully faithful. According to 	\ref{prop:ff 1}, for any $n>0$, $f_n:C_n\to D_n$ is an equivalence. If $f$ induces a surjection on objects, lemma \ref{lemma:ff 2} implies that $f_0:C_0\to D_0$ is an equivalence. We can then apply proposition \ref{prop:equivalences detected on globes}.
\end{proof}

\subsection{Discrete Conduché functors}
\label{section:conduche}
\begin{notation}
We denote \wcnotation{$\triangledown_{k,n}$}{(nabla@$\triangledown_{k,n}$} the unique globular morphism between $\Db_n$ and $\Db_n\coprod_{\Db_k}\Db_n$.
\end{notation}
\begin{definition}
A morphism $f:C\to D$ between $\io$-categories is a \snotion{discrete Conduché functor}{for $\io$-categories} if it has the unique right lifting property against 
units $\Ib_{n+1}:\Db_{n+1}\to \Db_n$ for any integer $n$, and against compositions $\triangledown_{k,n}:\Db_n\to \Db_n\coprod_{\Db_k}\Db_n$ for any pair of integers $k\leq n$.
\end{definition}
\begin{lemma}
\label{lemma:technicalconduche have the rlp aginst alebraic morphism}
The two following full sub $\infty$-groupoids of morphisms of $\ocat$ are equivalent: 
\begin{enumerate}
\item The smallest cocomplete full sub $\infty$-groupoid of morphisms containing the family of morphism $\{\Ib_{n+1}:\Db_{n+1}\to \Db_n,\}$ and the family $\{\triangledown_{k,n}:\Db_n\to \Db_n\coprod_{\Db_k}\Db_n\, ~k\leq n\}$.
\item The smallest cocomplete full sub $\infty$-groupoid of morphisms containing algebraic morphisms of $\Theta$ (this notion is defined in definition \ref{defi:algebraic and globular}). 
\end{enumerate}
\end{lemma}
\begin{proof}
For any pair of integers $k\leq n$, $\Ib_{n+1}$ and $\triangledown_{k,n}$ are algebraic morphisms. This directly induces the inclusion of the fist $\infty$-groupoid in the second one. To conclude, one has to show that every algebraic morphism $i:a\to b$ is contained in the first $\infty$-groupoid.

We proceed by induction on $|a|+|b|$. Suppose first that there exists $n$ such that $a=\Db_n$. In this case two cases have to be considered. Either $n>0$ and $i$ factors as $\Db_n\xrightarrow{\Ib_n} \Db_{n-1}\xrightarrow{j} b$. The result then follows by the induction hypothesis. Suppose now that $i$ does not factor though $\Ib_n$. In this case, there exists $k$ such that $i$ factors as $\Db_n\xrightarrow{\triangledown_{k,n}} \Db_n\coprod_{\Db_k}\Db_n\xrightarrow{j} b$. The unique factorization system between algebraic and globular morphisms given in proposition \ref{prop:algebraic ortho to globular} produces a diagram 
\[\begin{tikzcd}
	&& {\Db_n} &&& {b_2} \\
	& {\Db_k} &&& {b_1} \\
	{\Db_n} &&& {b_0} \\
	&& { \Db_n\coprod_{\Db_k}\Db_n } &&& b
	\arrow["j"{description}, from=4-3, to=4-6]
	\arrow[hook, from=3-1, to=4-3]
	\arrow[hook, from=1-3, to=4-3]
	\arrow["{j_2}"', from=1-3, to=1-6]
	\arrow["{j_0}", from=3-1, to=3-4]
	\arrow[hook, from=3-4, to=4-6]
	\arrow[hook, from=1-6, to=4-6]
	\arrow[hook, from=2-2, to=4-3]
	\arrow[hook, from=2-2, to=1-3]
	\arrow[hook', from=2-2, to=3-1]
	\arrow[hook', from=2-5, to=3-4]
	\arrow[hook, from=2-5, to=1-6]
	\arrow[hook, from=2-5, to=4-6]
	\arrow["{j_1}"{description}, from=2-2, to=2-5]
\end{tikzcd}\]
where arrows labeled by $\hookrightarrow$ are globular and the other ones are algebraic. Remark that we have a cocartesian square in $\iun$-category of arrows of $\ocat$:
\[\begin{tikzcd}
	{j_1} & {j_2} \\
	{j_0} & j
	\arrow[from=1-1, to=2-1]
	\arrow[from=1-1, to=1-2]
	\arrow[from=2-1, to=2-2]
	\arrow[from=1-2, to=2-2]
\end{tikzcd}\]
As $j_0$, $j_1$ and $j_2$ are in the first $\infty$-groupoid by induction hypothesis, so is $j$. By stability by composition, the morphism $i$ is then in the first $\infty$-groupoid.

Suppose now that the domain of $i:a\to b$ is not a globe. Using once again the unique factorization system between algebraic and globular, we can construct a functor $\Sp_a\to \Arr(\Theta)$ whose value on $\Db_n\hookrightarrow a$ is given by the unique algebraic morphism $j$ fitting in a commutative square
\[\begin{tikzcd}
	{\Db_n} & {b'} \\
	a & b
	\arrow[hook, from=1-1, to=2-1]
	\arrow[hook, from=1-2, to=2-2]
	\arrow["i"', from=2-1, to=2-2]
	\arrow["j", from=1-1, to=1-2]
\end{tikzcd}\]
where arrows labeled by $\hookrightarrow$ are globular. By induction hypothesis, $j$ is in the first $\infty$-groupoid. The colimit of 
$\Sp_a\to \Arr(\Theta)\to \Arr(\ocat)$ is then in the first $\infty$-groupoid. As this colimit is $i$, this concludes the proof.
\end{proof}

\begin{prop}
\label{prop:conduche have the rlp aginst alebraic morphism}
A morphism $f:X\to Y$ is a discrete Conduché functor if and only if it as the unique right lifting property against algebraic morphism of $\Theta$ (this notion is defined in definition \ref{defi:algebraic and globular}). 
\end{prop}
\begin{proof}
Given a morphism $f$, the full sub $\infty$-groupoid of morphisms having the unique left lifting property against $f$ is cocomplete. The result is then a direct implication of lemma 
\ref{lemma:technicalconduche have the rlp aginst alebraic morphism}.
\end{proof}

\begin{example}
The proposition \ref{prop:algebraic ortho to globular} implies that a morphism $a\to b$ between globular sums is a discrete Conduché functor if and only if it is globular.
\end{example}

\begin{lemma}
\label{lemma:conduche technical}
Let $p:C\to a$ be discrete Conduché functor with $a$ a globular sum. We denote by $(\Theta_{/p})^{Cd}$ the full sub $\iun$-category of $\Theta_{/p}$ whose objects are triangles 
\[\begin{tikzcd}
	b & C \\
	& a
	\arrow[from=1-1, to=1-2]
	\arrow["p", from=1-2, to=2-2]
	\arrow[from=1-1, to=2-2]
\end{tikzcd}\]
where every arrow is a discrete Conduché functor.
The canonical inclusion of $\iun$-category $\iota:(\Theta_{/p})^{Cd}\to \Theta_{/p}$ is final.
\end{lemma}
\begin{proof}
To prove this statement, we will endow $\iota$ with a structure of right deformation retract. We then first build a right inverse of $\iota$.
Any triangle 
\[\begin{tikzcd}
	b & C \\
	& a
	\arrow[from=1-1, to=1-2]
	\arrow["p", from=1-2, to=2-2]
	\arrow[from=1-1, to=2-2]
\end{tikzcd}\]
induces a diagram of shape
\[\begin{tikzcd}
	b & C \\
	{b'} & a
	\arrow[from=1-1, to=1-2]
	\arrow[from=1-1, to=2-1]
	\arrow[from=2-1, to=2-2]
	\arrow["l", from=2-1, to=1-2]
	\arrow["p", from=1-2, to=2-2]
\end{tikzcd}\]
where $b'$ is obtained in factorizing $b\to a$ in a algebraic morphism followed by a globular morphism, and $l$ comes from the unique right lifting property of $p$ against algebraic morphisms. By right cancellation, this implies that $l$ is a discrete Conduché functor.

 As these two operations are functorial, this defines a retraction $r: \Theta_{/p}\to (\Theta_{/p})^{Cd}$ sending the triangle spotted by $b,C$ and $a$ to the triangle spotted by $b',C$ and $a$. Moreover, this retraction comes along with a natural transformation $id\to r\iota$. As right deformation retracts are final, this concludes the proof.
\end{proof}

\begin{lemma}
\label{lemma:conduche preserves W}
Let $p:C\to D$ be a discrete Conduché functor. Then for any globular sums $a$, and any cartesian squares in $\iPsh{\Theta}$:
\[\begin{tikzcd}
	{C''} & {C'} & C \\
	{\Sp_a} & a & D
	\arrow["{p''}", from=1-1, to=2-1]
	\arrow["{p'}", from=1-2, to=2-2]
	\arrow["p", from=1-3, to=2-3]
	\arrow["j", from=1-1, to=1-2]
	\arrow[from=1-2, to=1-3]
	\arrow[from=2-1, to=2-2]
	\arrow[from=2-2, to=2-3]
	\arrow["\lrcorner"{anchor=center, pos=0.125}, draw=none, from=1-2, to=2-3]
	\arrow["\lrcorner"{anchor=center, pos=0.125}, draw=none, from=1-1, to=2-2]
\end{tikzcd}\]
the morphism $j$ is in $\widehat{\Wseg}$.
\end{lemma}
\begin{proof}
By stability under pullback,
the morphism $p'$ is a discrete Conduché functor. 
Taking the notations of lemma \ref{lemma:conduche technical}, $p'$ is equivalent to $\colim_{(\Theta_{/p})^{Cd}}b\to a$ where this colimit is taken in $\iPsh{\Theta}_{/a}$. As $\iPsh{\Theta}$ is locally cartesian closed and as $\widehat{\W}$ is by definition closed by colimits, we can then reduce to the case where $p'$ is a discrete Conduché functor between globular sums, i.e a globular morphism $b\to a$.
In this case, the following canonical square is a pullback
\[\begin{tikzcd}
	{\Sp_b} & b \\
	{\Sp_a} & a
	\arrow[from=2-1, to=2-2]
	\arrow["{p'}", from=1-2, to=2-2]
	\arrow[from=1-1, to=2-1]
	\arrow[from=1-1, to=1-2]
	\arrow["\lrcorner"{anchor=center, pos=0.125}, draw=none, from=1-1, to=2-2]
\end{tikzcd}\]
and this concludes the proof.
\end{proof}

\begin{lemma}
\label{lemma:pulback of Wsat preresult}
Consider a cartesian square 
\[\begin{tikzcd}
	X & Y \\
	{\Sigma^{n} E^{eq}} & {\Db_{n}}
	\arrow[from=1-1, to=2-1]
	\arrow["j", from=1-1, to=1-2]
	\arrow[from=2-1, to=2-2]
	\arrow[from=1-2, to=2-2]
	\arrow["\lrcorner"{anchor=center, pos=0.125}, draw=none, from=1-1, to=2-2]
\end{tikzcd}\]
in $\iPsh{\Theta}$. The morphism $j$ is in $\widehat{\W}$.
\end{lemma}
\begin{proof}
 If we are in the case $n=0$, this directly follows from the preservation of $\W$ by cartesian product, demonstrated in the proof of proposition \ref{prop:cartesian product preserves W}.
We now suppose the result is true at stage $n$, and we first show that for any square
\[\begin{tikzcd}
	X & Y \\
	{[\Sigma^{n} E^{eq},1]} & {[\Db_{n+1},1]}
	\arrow[from=1-1, to=2-1]
	\arrow["j", from=1-1, to=1-2]
	\arrow[from=2-1, to=2-2]
	\arrow["p", from=1-2, to=2-2]
	\arrow["\lrcorner"{anchor=center, pos=0.125}, draw=none, from=1-1, to=2-2]
\end{tikzcd}\]
in $\iPsh{\Delta[\Theta]}$, $j$ is in $\widehat{\M}$.
As $\iPsh{\Delta[\Theta]}$ is locally cartesian closed and $\widehat{\M}$ closed under colimits, one can suppose that $Y$ is of shape $[a,k]$ and we denote $f:[k]\to [1]$ the morphism induced by $p$. By stability under pullback, $X$ is then set-valued. Furthermore, we can then check in $\Psh{\Delta[\Theta]}$ that this presheaf fits in a cocartesian square:
\[\begin{tikzcd}
	{[\Sigma^nE^{eq}\times_{\Db_n} a,f^{-1}(0)]\coprod [\Sigma^nE^{eq}\times_{\Db_n}a,f^{-1}(1)]} & {[\Sigma^nE^{eq}\times_{\Db_n} a,k]} \\
	{[a,f^{-1}(0)]\coprod [a,f^{-1}(1)]} & X
	\arrow[from=1-1, to=2-1]
	\arrow[from=1-1, to=1-2]
	\arrow[from=2-1, to=2-2]
	\arrow[from=1-2, to=2-2]
\end{tikzcd}\]
By induction hypothesis $[\Sigma^nE^{eq}\times_{\Db_n} a,l]\to [a,l]$ is in $\widehat{\M}$ for any integer $l$. As $X\to [a,k]$ is the colimit in depth of the diagram
\[\begin{tikzcd}[column sep = 0.7cm]
	{[\Sigma^nE^{eq}\times_{\Db_n} a,f^{-1}(0)]} && { [\Sigma^nE^{eq}\times_{\Db_n}a,f^{-1}(1)]} \\
	{[a,f^{-1}(0)]} & {[\Sigma^nE^{eq}\times_{\Db_n} a,k]} & {[a,f^{-1}(1)]} \\
	& {[a,f^{-1}(0)]} && {[a,f^{-1}(1)]} \\
	& {[a,f^{-1}(0)]} & {[a,k]} & {[a,f^{-1}(1)]}
	\arrow[from=4-3, to=3-2]
	\arrow[from=3-2, to=4-2]
	\arrow[from=1-1, to=2-1]
	\arrow[from=1-1, to=2-2]
	\arrow[from=1-3, to=2-2]
	\arrow[from=1-3, to=2-3]
	\arrow[from=3-4, to=4-3]
	\arrow[from=3-4, to=4-4]
	\arrow[from=2-1, to=4-2]
	\arrow[from=1-1, to=3-2]
	\arrow[from=2-2, to=4-3]
	\arrow[from=2-3, to=4-4]
	\arrow[from=1-3, to=3-4]
\end{tikzcd}\]
this implies that this morphism is in $\widehat{\M}$.

We now return to $\infty$-presheaves on $\Theta$. We recall that we denote by $(i_!,i^*)$ the adjunction between $\iPsh{\Delta[\Theta]}$ and $\iPsh{\Theta}$. Suppose given a cartesian square:
\[\begin{tikzcd}
	X & Y \\
	{\Sigma^{n+1} E^{eq}} & {\Db_{n+1}}
	\arrow[from=1-1, to=2-1]
	\arrow["j", from=1-1, to=1-2]
	\arrow[from=2-1, to=2-2]
	\arrow[from=1-2, to=2-2]
	\arrow["\lrcorner"{anchor=center, pos=0.125}, draw=none, from=1-1, to=2-2]
\end{tikzcd}\]
This induces two squares
\[\begin{tikzcd}
	{i^*X} & {i^*Y} & {i_!i^*X} & {i_!i^*Y} \\
	{[\Sigma^{n}E^{eq},1]} & {[\Db_n,1]} & X & Y
	\arrow[from=2-3, to=2-4]
	\arrow[from=1-3, to=2-3]
	\arrow[from=1-4, to=2-4]
	\arrow["{i_!i^*j}", from=1-3, to=1-4]
	\arrow[from=1-1, to=2-1]
	\arrow[from=1-2, to=2-2]
	\arrow[from=2-1, to=2-2]
	\arrow["{i^*j}", from=1-1, to=1-2]
	\arrow["\lrcorner"{anchor=center, pos=0.125}, draw=none, from=1-1, to=2-2]
\end{tikzcd}\]
Where the cartesianess of the left square comes from the fact that $i^*$ preserves cartesian squares as it is a right adjoint. We just have demonstrated that $i^*j$ is in $\widehat{\M}$. Using proposition \ref{prop:infini changing theta}, and by left cancellation, the right square implies that $j$ is in $\widehat{\W}$, which concludes the proof.
\end{proof}

\begin{prop}
\label{prop:pulback of Wsat}
Let $p:C\to D$ be a functor between $\io$-categories. Then for any globular sums $a$, and any cartesian squares in $\iPsh{\Theta}$:
\[\begin{tikzcd}
	{C''} & {C'} & C \\
	{\Sigma^nE^{eq}} & {\Db_n} & D
	\arrow["p", from=1-3, to=2-3]
	\arrow["j", from=1-1, to=1-2]
	\arrow[from=1-2, to=1-3]
	\arrow[from=2-1, to=2-2]
	\arrow[from=2-2, to=2-3]
	\arrow["\lrcorner"{anchor=center, pos=0.125}, draw=none, from=1-2, to=2-3]
	\arrow["\lrcorner"{anchor=center, pos=0.125}, draw=none, from=1-1, to=2-2]
	\arrow[from=1-1, to=2-1]
	\arrow[from=1-2, to=2-2]
\end{tikzcd}\]
the morphism $j$ is in $\widehat{\W}$.
\end{prop}
\begin{proof}
This is a direct consequence of lemma \ref{lemma:pulback of Wsat preresult}.
\end{proof}
\begin{theorem}
\label{theo:pullback along conduche preserves colimits}
Let $f:C\to D$ be a discrete Conduché functor. The pullback functor $f^*:\ocat_{/D}\to \ocat_{/C}$ preserves colimits.
\end{theorem}
\begin{proof}
As $\iPsh{\Theta}$ is locally cartesian closed, we can use the corollary \ref{cor:derived colimit preserving functor}. The hypotheses are provided by lemmas \ref{lemma:conduche preserves W} and proposition \ref{prop:pulback of Wsat}.
\end{proof}

\section{Gray Operations}
\subsection{Gray operations on $\io$-categories}

\begin{construction}
We construct by induction on $m\in \mathbb{N}$ a colimit-preserving functor
\begin{equation}
\label{eqgray second step}
\uvar\otimes \uvar:\iPsh{\Theta_m}\times \iPsh{\Delta}\to \iPsh{\Theta}
\end{equation}
If $m=1$, the functor is the left Kan extension of the functor 
$$\Delta\times \Delta  \xrightarrow{\uvar\otimes \uvar} \zocat\xrightarrow{\iota} \iPsh{\Theta}$$
where $\otimes:\Delta\times \Delta\to \zocat$ is the Gray tensor product defined in theorem \ref{theo:otimes in zocat}.

Suppose the functor constructed at the stage $m$. We define the colimit-preserving functor 
\begin{equation}
\label{eqgray first step}
\uvar\otimes \uvar:\iPsh{\Delta[\Theta_{m}]}\times \iPsh{\Delta}\to \iPsh{\Theta}
\end{equation}
whose value on $([a,k],[n])$ fits  in the cocartesian square
\[\begin{tikzcd}
	{\coprod_{i\leq k}\colim_{[l,p]\to \{k\}\otimes[n]}[a\otimes[l],p]} & {\colim_{[l,p]\to [k]\otimes [n]}[a\otimes[l],p]} \\
	{\coprod_{i\leq k}\colim_{[l,p]\to \{k\}\otimes[n]}[[l],p]} & {[a,k]\otimes[n]}
	\arrow[""{name=0, anchor=center, inner sep=0}, from=1-1, to=1-2]
	\arrow[from=1-1, to=2-1]
	\arrow[from=1-2, to=2-2]
	\arrow[from=2-1, to=2-2]
	\arrow["\lrcorner"{anchor=center, pos=0.125, rotate=180}, draw=none, from=2-2, to=0]
\end{tikzcd}\]
Eventually, we define the functor 
$$
\uvar\otimes \uvar:\iPsh{\Theta_{m+1}}\times \iPsh{\Delta}\to \iPsh{\Theta}$$
as the left Kan extension of the functor \eqref{eqgray first step} along the inclusion $$ \iPsh{\Delta[\Theta_{m}]}\times \iPsh{\Delta}\to \iPsh{\Theta_{m+1}}\times \iPsh{\Delta}.$$

As $\Theta$ is equivalent to $\cup_n\Theta_n$, the functor \eqref{eqgray second step} induces a by extension by colimit functor 
\begin{equation}
\label{eqgray thrid step}
\uvar\otimes \uvar:\iPsh{\Theta}\times \iPsh{\Delta}\to \iPsh{\Theta}
\end{equation}

\end{construction}
\begin{prop}
\label{prop:gray is colimit preserving}
The functor 
$$\uvar\otimes \uvar:\iPsh{\Theta}\times \iPsh{\Delta} \to \iPsh{\Theta}$$
constructed in \eqref{eqgray thrid step} sends $\W_1\times \W$ to $\widehat{\W}$.
\end{prop}
\begin{proof}
We show by induction on $m$ that the functor 
$$\uvar\otimes \uvar:\iPsh{\Theta_m}\times \iPsh{\Delta} \to \iPsh{\Theta}$$
constructed in \eqref{eqgray second step} sends $\W_1\times \W_m$ to $\widehat{\W}$.

For $m=0$, this is directly true as this functor corresponds to the inclusion of $\iPsh{\Delta}$ in $\iPsh{\Theta}$ induced by the canonical inclusion $\Delta\to \Theta$.

Suppose the result is true at the stage $m$. 
We first want to show that the functor \eqref{eqgray first step}:
$$\uvar\otimes\uvar:\iPsh{\Delta[\Theta_m]}\times\iPsh{\Delta}\to \iPsh{\Theta}$$
sends $\M_{m+1}\times W_1$ to $\widehat{\W}$.

We first fix an object $a$ in $\Theta_m$. The functor $ [a,\uvar]\otimes \uvar: \Sset \times \Sset \to  \iPsh{\Theta}$ is the composite
\[\begin{tikzcd}
	{\iPsh{\Delta}\times \iPsh{\Delta}} & {\iPsh{\Theta_2}} & {\iPsh{\Delta[\Delta]}} & {\iPsh{\Delta[\Theta]}} & {\iPsh{\Theta}}
	\arrow["\uvar\otimes\uvar", from=1-1, to=1-2]
	\arrow["{i^*}", from=1-2, to=1-3]
	\arrow["{\Delta[a\otimes\uvar]}", from=1-3, to=1-4]
	\arrow["{i_!}", from=1-4, to=1-5]
\end{tikzcd}\]
According to proposition \ref{prop:i etoile of W is in M}, lemma \ref{lemma:colimit computed in set presheaves}, theorem \ref{theo:otimes presserves W}, and by the induction hypothesis, this functor then sends $\W_1 \times \W_1$ to $\widehat{\W}$.

We now fix two integers $n$ and $k$. Let $i:X \to Y$ be a morphism in $\W_m$. The morphism $ [X,k]\otimes [n] \to  [Y,k]\otimes[n]$ is a colimit of natural transformations that is pointwise in $\widehat{\W}$ and so is also in $\widehat{\W}$.

This implies that the functor \eqref{eqgray second step} sends $\M_{m+1}\times W_1$ to $\widehat{\W}$. By propositions \ref{prop:i etoile of W is in M} and \ref{prop:link beetwenwidehat and overline}, this implies that the functor 
$$\uvar\otimes\uvar:\iPsh{\Theta_{m+1}}\times\iPsh{\Delta}\to \iPsh{\Theta}$$
sends $\W_{m+1}\times W_1$ to $\widehat{\W}$.
\end{proof}

\begin{definition}
By proposition \ref{prop:gray is colimit preserving}, the functor \eqref{eqgray thrid step} induces a functor 
$$\uvar\otimes\uvar: \ocat\times \icat \to\ocat,$$
called the \snotionsym{Gray tensor product}{((d00@$\otimes$}{for $\io$-categories}.
\end{definition}

\begin{prop}
\label{prop:Gray tensor product and duality}
There is an equivalence
$(C\otimes K)^\circ \sim C^\circ\otimes K^\circ$ natural in $C$ and $K$.
\end{prop}
\begin{proof}
It is sufficient to construct this equivalence on pairs $(a,[n])$ where $a$ is a globular sum. We then proceed by induction on the dimension in $a$, using the natural equivalence $[k]\otimes[l]\sim [k]^\circ\otimes [l]^{\circ}$  provided by proposition \ref{prop:otimes and duality}.
\end{proof}

\begin{construction}
\label{cons:almost assoc of gray}
We construct by induction on $m\in \Nb$ a natural transformation $\psi$ between the two functors 
$$\begin{array}{ll}
(\uvar\otimes \uvar)\otimes \uvar:&\iPsh{\Theta_m}\times \iPsh{\Delta}\times \iPsh{\Delta} \to \iPsh{\Theta}\\
\uvar\otimes (\uvar\times \uvar):&\iPsh{\Theta_m}\times \iPsh{\Delta}\times \iPsh{\Delta} \to \iPsh{\Theta}
\end{array}$$
If $m=0$, this corresponds to the canonical morphism $\uvar\otimes\uvar\to \uvar\times \uvar$. 
Suppose the natural transformation is defined at the stage $m$.
Let $a$ be an object of $\Theta_m$ and $l,m,n$ three integers. By construction, $([a,n]\otimes[m])\otimes [l]$ is a quotient of 
$$P_{a,l,m,n}:=\colim_{[[k_0],k_1]\to [n]\otimes[m]}\colim_{[[k_2],k_3]\to [k_1]\otimes[l]}[a\otimes [k_0]\otimes[k_2], k_3],$$
while $[a,n]\otimes ([m]\times [l])$ is a quotient of
$$Q_{a,l,m,n}:=\colim_{[[k_4],k_3]\to [m]\otimes ([n]\times[l])}[a\otimes [k_4], k_3].$$
Lemma \ref{lemma:technical steiner} and the induction hypothesis then induce a morphism $$P_{a,l,m,n}\to Q_{a,l,m,n}.$$ 
We can check that this morphism passes to the quotient and then induces a natural morphism 
$$([a,n]\otimes [m])\otimes[l]\to [a,n]\otimes([m]\otimes[l])$$
By extension by colimit, this induces, for any $C$ in $\iPsh{\Theta_m}$, and any pair $K,L$ of objects of $\iPsh{\Delta}$, a natural morphism 
$$C\otimes K\otimes L\to  C\otimes (K\times L).$$
\end{construction}
\vspace{1cm}

The Gray tensor product will allows us to defined numerous operations, called the \textit{Gray operations}.

\subsubsection{Gray cylinder}

\begin{definition}
\label{defi:of gray cylinder}
The restricted functor 
$$\uvar\otimes[1]:\ocat\to \ocat$$
is called the \snotionsym{Gray cylinder}{((d30@$\uvar\otimes[1]$}{for $\io$-categories}.
\end{definition}

\begin{prop}
\label{prop:eq for cylinder}
There is a natural identification between $[C,1]\otimes [1]$ and the colimit of the diagram
\begin{equation}
\begin{tikzcd}
	{[1]\vee [ C,1]} & {[C\otimes\{0\},1]} & {[C\otimes [1],1]} & {[C\otimes\{1\},1]} & {[C,1]\vee[1]}
	\arrow[from=1-2, to=1-1]
	\arrow[from=1-2, to=1-3]
	\arrow[from=1-4, to=1-3]
	\arrow[from=1-4, to=1-5]
\end{tikzcd}
\end{equation}
In the previous diagram, morphisms $[C,1]\to [1]\vee[C,1]$ and $[C,1]\to [C,1]\vee[1]$ are the whiskerings.
\end{prop}
\begin{proof}
This is a direct consequence of the construction of the Gray cylinder and of proposition \ref{prop:explicit Gray}.
\end{proof}

\begin{example}
The $\io$-category $\Db_1\otimes[1]$ corresponds to the polygraph
\[\begin{tikzcd}
	00 & 01 \\
	10 & 11
	\arrow[from=1-1, to=2-1]
	\arrow[from=2-1, to=2-2]
	\arrow[from=1-1, to=1-2]
	\arrow[from=1-2, to=2-2]
	\arrow[shorten <=4pt, shorten >=4pt, Rightarrow, from=1-2, to=2-1]
\end{tikzcd}\]
The $\io$-category $\Db_2\otimes[1]$ corresponds to the polygraph
\[\begin{tikzcd}
	00 & 01 & 00 & 01 \\
	10 & 11 & 10 & 11
	\arrow[from=1-1, to=1-2]
	\arrow[""{name=0, anchor=center, inner sep=0}, from=1-1, to=2-1]
	\arrow[from=2-1, to=2-2]
	\arrow[""{name=1, anchor=center, inner sep=0}, from=1-2, to=2-2]
	\arrow[shorten <=4pt, shorten >=4pt, Rightarrow, from=1-2, to=2-1]
	\arrow[""{name=2, anchor=center, inner sep=0}, from=1-3, to=2-3]
	\arrow[from=1-3, to=1-4]
	\arrow[""{name=3, anchor=center, inner sep=0}, from=1-4, to=2-4]
	\arrow[shorten <=4pt, shorten >=4pt, Rightarrow, from=1-4, to=2-3]
	\arrow[""{name=4, anchor=center, inner sep=0}, curve={height=30pt}, from=1-1, to=2-1]
	\arrow[from=2-3, to=2-4]
	\arrow[""{name=5, anchor=center, inner sep=0}, curve={height=-30pt}, from=1-4, to=2-4]
	\arrow["{ }"', shorten <=6pt, shorten >=6pt, Rightarrow, from=0, to=4]
	\arrow["{ }"', shorten <=6pt, shorten >=6pt, Rightarrow, from=5, to=3]
	\arrow[shift left=0.7, shorten <=6pt, shorten >=8pt, no head, from=1, to=2]
	\arrow[shift right=0.7, shorten <=6pt, shorten >=8pt, no head, from=1, to=2]
	\arrow[shorten <=6pt, shorten >=6pt, from=1, to=2]
\end{tikzcd}\]
\end{example}

\begin{prop}
\label{prop:Gray cylinder and op}
We have a natural diagram whose vertical morphisms are equivalences:
\[\begin{tikzcd}
	{(C\otimes\{1\})^\circ} & {(C\otimes[1])^\circ} & {(C\otimes\{0\})^\circ} \\
	{C^\circ\otimes\{0\}} & {C^\circ\otimes[1]} & {C^\circ\otimes\{1\}}
	\arrow["\sim", from=1-2, to=2-2]
	\arrow["\sim", from=1-3, to=2-3]
	\arrow["\sim", from=1-1, to=2-1]
	\arrow[from=1-3, to=1-2]
	\arrow[from=1-1, to=1-2]
	\arrow[from=2-1, to=2-2]
	\arrow[from=2-3, to=2-2]
\end{tikzcd}\]
\end{prop}
\begin{proof}
This is a direct consequence of proposition \ref{prop:Gray tensor product and duality}.
\end{proof}

\begin{definition}
We denote by 
$$\begin{array}{rcl}
\ocat&\to&\ocat\\
C&\mapsto &C^{[1]}
\end{array}$$
the right adjoint of the Gray cylinder.\sym{(c@$C^{[1]}$}
\end{definition}

\begin{remark}
The proposition \ref{prop:comparaison betwen otimes and suspension week case withou the cocartesian square.} provides, for any $\io$-category $C$ and any pair of objects $a$, $b$ of $C$, a canonical cartesian square
\[\begin{tikzcd}
	{\hom_C(a,b)} & {C^{[1]}} \\
	{\{a\}\times \{b\}} & {C\times C}
	\arrow[from=1-1, to=1-2]
	\arrow[from=1-1, to=2-1]
	\arrow["\lrcorner"{anchor=center, pos=0.125}, draw=none, from=1-1, to=2-2]
	\arrow[from=1-2, to=2-2]
	\arrow[from=2-1, to=2-2]
\end{tikzcd}\]
\end{remark}

\subsubsection{Enhanced Gray tensor product}
\begin{definition}
 For any $C:\ocat$, we denote by \wcnotation{$m_C$}{(mc@$m_C$} the colimit preserving functor 
$\ocat\to\ocat$ whose value on a representable $[a,n]$ is $[a\times C,n]$. Remark that the assignation $C\mapsto m_C$ is natural in $C$ and that $m_1$ is the identity.
\end{definition}

\begin{construction}
We define the colimit preserving functor: 
$$\begin{array}{ccc}
\ocat\times\ocat &\to&\ocat\\
(X,Y)&\mapsto &X\ominus Y,
\end{array}
$$
called the 
 \snotionsym{enhanced Gray tensor product}{((d20@$\ominus$}{for $\io$-categories}
where for any $\io$-category $C$ and any element $[b,n]$ of $\Delta[\Theta]$, $X\ominus [b,n]$ is the following pushout: 
\[\begin{tikzcd}
	{\coprod\limits_{k\leq n}m_b(C\otimes\{k\})} & {m_b(C\otimes[n])} \\
	{\coprod\limits_{k\leq n}m_1(C\otimes\{k\})} & {C\ominus[b,n]}
	\arrow[from=1-1, to=2-1]
	\arrow[""{name=0, anchor=center, inner sep=0}, from=1-1, to=1-2]
	\arrow[from=1-2, to=2-2]
	\arrow[from=2-1, to=2-2]
	\arrow["\lrcorner"{anchor=center, pos=0.125, rotate=180}, draw=none, from=2-2, to=0]
\end{tikzcd}\]
By construction, the functor $\uvar\ominus \uvar$ commutes with colimits in both variables. Moreover, for any $\io$-category $C$ and $\iun$-category $K$, we have a canonical identification $C\ominus K \sim C\otimes K$.
\end{construction}

\begin{prop}
\label{prop:formula for the ominus}
There is a  natural identification between $[C,1]\ominus[b,1]$ and the colimit of the following diagram
\begin{equation}
\begin{tikzcd}[column sep = 0.3cm]
	{[b,1]\vee[C,1]} & {[C\otimes\{0\}\times b,1]} & {[(C\otimes[1])\times b,1]} & {[C\otimes\{1\}\times b,1]} & {[C,1]\vee[b,1]}
	\arrow[from=1-2, to=1-3]
	\arrow[from=1-4, to=1-3]
	\arrow[from=1-4, to=1-5]
	\arrow[from=1-2, to=1-1]
\end{tikzcd}
\end{equation}
\end{prop}
\begin{proof}
This is a consequence of proposition \ref{prop:eq for cylinder}.
\end{proof}

\begin{cor}
\label{cor:ominus et op}
Let $A$ and $B$ two $\io$-categories. There is an equivalence  
$$(A\ominus B)^\circ \sim A^\circ\ominus B^\circ$$
natural in $A$ and $B$.
\end{cor}
\begin{proof}
It is sufficient to construct the equivalence when $A$ is a globular sum $a$ and $B$ is of shape $[b,n]$. 
Remark first that the proposition \ref{prop:Gray tensor product and duality} implies that $(a\otimes[n])^\circ$ and $a^\circ\otimes[n]^\circ$ are isomorphic.  The results then directly follows from the definition of the operation $\ominus$ and from the equivalence $(m_b(\uvar))^\circ\sim m_{b^\circ}((\uvar)^\circ)$.
\end{proof}

\subsubsection{Gray cone and $\circ$-cone}

\begin{construction}
\label{cons:of gray cone}
 We define the \snotionsym{Gray cone}{((d40@$\uvar\star 1$}{for $\io$-categories} and the \snotion{Gray $\circ$-cone}{for $\io$-categories}\index[notation]{((d50@$1\overset{co}{\star}\_$!\textit{for $\io$-categories}}:
$$\begin{array}{ccccccc}
\ocat &\to&\ocat_{\bullet}&&\ocat &\to&\ocat_{\bullet}\\
C&\mapsto &C\star 1 & &C &\mapsto &1\costar C
\end{array}
$$
where $C\star 1$ and $1\costar C$ are defined as the following pushout: 
\[\begin{tikzcd}
	{C\otimes\{1\}} & {C\otimes [1]} & {C\otimes\{0\}} & {C\otimes [1]} \\
	1 & {C\star 1} & 1 & {1\costar C}
	\arrow[from=1-1, to=2-1]
	\arrow[from=1-1, to=1-2]
	\arrow[from=2-1, to=2-2]
	\arrow[from=1-2, to=2-2]
	\arrow["\lrcorner"{anchor=center, pos=0.125, rotate=180}, draw=none, from=2-2, to=1-1]
	\arrow[from=1-3, to=2-3]
	\arrow[from=2-3, to=2-4]
	\arrow[from=1-3, to=1-4]
	\arrow[from=1-4, to=2-4]
	\arrow["\lrcorner"{anchor=center, pos=0.125, rotate=180}, draw=none, from=2-4, to=1-3]
\end{tikzcd}\]

\end{construction}

\begin{prop}
\label{prop:eq for Gray cone}
There is a natural identification between $1\costar [C,1]$ and the colimit of the diagram
\begin{equation}
\begin{tikzcd}
	{[1]\vee [C,1]} & {[C,1]} & {[C\star 1,1]}
	\arrow[from=1-2, to=1-3]
	\arrow[from=1-2, to=1-1]
\end{tikzcd}
\end{equation}
There is a natural identification between $[C,1]\star 1$ and the colimit of the diagram
\begin{equation}
\begin{tikzcd}
	{[1\costar C,1]} & {[C,1]} & {[C,1]\vee[1]}
	\arrow[from=1-2, to=1-3]
	\arrow[from=1-2, to=1-1]
\end{tikzcd}
\end{equation}
In each of the two previous diagrams, morphisms $[C,1]\to [1]\vee[C,1]$ and $[C,1]\to [C,1]\vee[1]$ are the whiskerings.
\end{prop}
\begin{proof}
This is a consequence of proposition \ref{prop:eq for cylinder}.
\end{proof}

\begin{example}
The $\io$-categories $\Db_1\star 1$ and $1\costar \Db_1$ correspond respectively to the polygraphs: 
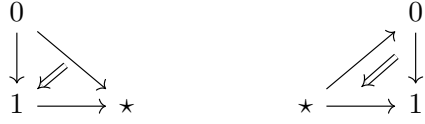
\[\begin{tikzcd}
	0 &&&& 0 \\
	1 & \star && \star & 1
	\arrow[from=1-1, to=2-1]
	\arrow[from=2-1, to=2-2]
	\arrow[""{name=0, anchor=center, inner sep=0}, from=1-1, to=2-2]
	\arrow[""{name=1, anchor=center, inner sep=0}, from=1-5, to=2-5]
	\arrow[from=2-4, to=1-5]
	\arrow[""{name=2, anchor=center, inner sep=0}, from=2-4, to=2-5]
	\arrow[shorten <=2pt, Rightarrow, from=0, to=2-1]
	\arrow[shift right=2, shorten <=4pt, shorten >=4pt, Rightarrow, from=1, to=2]
\end{tikzcd}\]
The $\io$-categories $\Db_2\star 1$ and $1\costar \Db_2$ correspond respectively to the polygraphs: 
\[\begin{tikzcd}
	0 & {~} & 0 &&& 0 & {~} & 0 \\
	1 & \star & 1 & \star & \star & 1 & \star & 1
	\arrow[""{name=0, anchor=center, inner sep=0}, from=1-1, to=2-1]
	\arrow[from=2-1, to=2-2]
	\arrow[""{name=1, anchor=center, inner sep=0}, from=1-3, to=2-3]
	\arrow[""{name=2, anchor=center, inner sep=0}, curve={height=30pt}, from=1-1, to=2-1]
	\arrow[from=2-3, to=2-4]
	\arrow[""{name=3, anchor=center, inner sep=0}, from=1-1, to=2-2]
	\arrow[""{name=4, anchor=center, inner sep=0}, draw=none, from=1-2, to=2-2]
	\arrow[""{name=5, anchor=center, inner sep=0}, from=1-3, to=2-4]
	\arrow[from=1-6, to=2-5]
	\arrow[""{name=6, anchor=center, inner sep=0}, from=1-6, to=2-6]
	\arrow[""{name=7, anchor=center, inner sep=0}, from=2-5, to=2-6]
	\arrow[from=1-8, to=2-7]
	\arrow[""{name=8, anchor=center, inner sep=0}, from=1-8, to=2-8]
	\arrow[""{name=9, anchor=center, inner sep=0}, from=2-8, to=2-7]
	\arrow[""{name=10, anchor=center, inner sep=0}, curve={height=-30pt}, from=1-8, to=2-8]
	\arrow[""{name=11, anchor=center, inner sep=0}, draw=none, from=1-7, to=2-7]
	\arrow["{ }"', shorten <=6pt, shorten >=6pt, Rightarrow, from=0, to=2]
	\arrow[shorten <=2pt, shorten >=2pt, Rightarrow, from=3, to=2-1]
	\arrow[shift left=0.7, shorten <=6pt, shorten >=8pt, no head, from=4, to=1]
	\arrow[shift right=0.7, shorten <=6pt, shorten >=8pt, no head, from=4, to=1]
	\arrow[shorten <=6pt, shorten >=6pt, from=4, to=1]
	\arrow[shorten <=2pt, Rightarrow, from=5, to=2-3]
	\arrow[shorten <=6pt, shorten >=6pt, Rightarrow, from=10, to=8]
	\arrow[shift right=2, shorten <=4pt, shorten >=4pt, Rightarrow, from=8, to=9]
	\arrow[shift right=2, shorten <=4pt, shorten >=4pt, Rightarrow, from=6, to=7]
	\arrow[shift right=0.7, shorten <=6pt, shorten >=8pt, no head, from=6, to=11]
	\arrow[shorten <=6pt, shorten >=6pt, from=6, to=11]
	\arrow[shift left=0.7, shorten <=6pt, shorten >=8pt, no head, from=6, to=11]
\end{tikzcd}\]
\end{example}

\begin{prop}
\label{prop:cone and op}
We have an
invertible natural transformation
$$ C\star 1\sim (1\costar C^{\circ})^\circ.$$
\end{prop}
\begin{proof}
 This directly follows from proposition \ref{prop:Gray cylinder and op} and from the construction of the Gray cone and $\circ$-cone.
\end{proof}

\begin{definition}
We  denote by 
$$\begin{array}{ccccccc}
\ocat_{\bullet} &\to& \ocat&&\ocat_{\bullet} &\to& \ocat\\
(C,c)&\mapsto &C_{/c} & &(C,c) &\mapsto &C_{c/}
\end{array}
$$
the right adjoints of the Gray cone and the Gray $\circ$-cone, respectively called the \wcsnotionsym{slice of $C$ over $c$}{(cc@$C_{/c}$}{slice over}{for $\io$-categories} and the \wcsnotionsym{slice of $C$ under $c$}{(cc@$C_{c/}$}{slice under}{for $\io$-categories}. 
\end{definition}

\begin{remark}
The proposition \ref{prop:cone and op} induces an
invertible natural transformation
$$C_{/c}\sim (C^{\circ}_{c/})^\circ.$$
\end{remark}

\begin{remark}
The proposition \ref{prop:comparaison betwen otimes and suspension week case withou the cocartesian square.} provides, for any $\io$-category $C$ and any pair of objects $a$, $b$ of $C$, two canonical cartesian squares
\[\begin{tikzcd}
	{\hom_C(a,b)} & {C_{/b}} & {\hom_C(a,b)} & {C_{a/}} \\
	{\{a\}} & C & {\{b\}} & C
	\arrow[from=1-1, to=1-2]
	\arrow[from=1-1, to=2-1]
	\arrow["\lrcorner"{anchor=center, pos=0.125}, draw=none, from=1-1, to=2-2]
	\arrow[from=1-2, to=2-2]
	\arrow[from=1-3, to=1-4]
	\arrow[from=1-3, to=2-3]
	\arrow["\lrcorner"{anchor=center, pos=0.125}, draw=none, from=1-3, to=2-4]
	\arrow[from=1-4, to=2-4]
	\arrow[from=2-1, to=2-2]
	\arrow[from=2-3, to=2-4]
\end{tikzcd}\]
\end{remark}

\subsubsection{$k$-Gray cylinder}
Let $C$ be an $\io$-category and $K$ a $(\infty,1)$-category.
There is a canonical morphism $C\otimes K\to C\times K$. In a way, one can see $C\times K$ as an intelligent truncated version of the Gray tensor product $C\otimes K$. We will make this intuition precise by constructing a hierarchy of Gray tensor products with $(\infty,1)$-categories. 
\begin{definition}
For $k\in \Nb\cup\{\omega\}$, we define the functor
$$\begin{array}{ccl}
\ocat \times \ncat{1}&\to &\ocat\\
(C,K)&\mapsto &C\otimes_k K
\end{array}$$
where $C\otimes_kK$ fits in the cocartesian square
\[\begin{tikzcd}
	{\colim_{n\geq k}(\tau_nC)\otimes K} & {C\otimes K} \\
	{\colim_{n\geq k}\tau^i_n((\tau_nC)\otimes K)} & {C\otimes_{k} K}
	\arrow[from=1-1, to=2-1]
	\arrow[from=1-1, to=1-2]
	\arrow[from=1-2, to=2-2]
	\arrow[from=2-1, to=2-2]
	\arrow["\lrcorner"{anchor=center, pos=0.125, rotate=180}, draw=none, from=2-2, to=1-1]
\end{tikzcd}\]

The induced functors $\uvar\otimes_k[1]:\ocat\to\ocat$ are called the \wcnotionsym{$k$-Gray cylinder}{((d10@$\otimes_n$}{Gray cylindera@$n$-Gray cylinder}.
\end{definition}

\begin{remark}
Remark that the endofunctor $\uvar\otimes_0[1]$ is the identity, 
the endofunctor $\uvar\otimes_1[1]$ is equivalent to $\uvar\times [1]$, and the endofunctor $\otimes_{\omega}[1]$ is just the normal Gray cylinder.
\end{remark}

\begin{prop}
\label{prop:otimesk preserves colimits}
For any integer $k>0$, $\uvar\otimes_k[1]$ preserves colimits. 
\end{prop}
\begin{proof}
In order to simplify the notation, for a functor $F:\ocat\to \ocat$, the $\infty$-presheaves $\colim_{\Theta_{/\Sigma^nE^{eq}}}\iota F$, where $\iota$ in the inclusion $\ocat\to \iPsh{\Theta}$, will just be denoted by $F(\Sigma^nE^{eq})$.

As $\tau$ and $\tau^i$ preserves colimits in $\iPsh{\Theta}$ and $\widehat{\Wseg}$, and as $\uvar\otimes[1]$ preserves colimits, we just have to show that for any $n$, $(\Sigma^nE^{eq})\otimes_k[1]\to (\Sigma^n1)\otimes_k[1]$ is in $\widehat{\W}$. 

We then proceed by induction on $k$. The cases $k=0$ and $k=1$ are trivial as $\uvar\otimes_0[1]$ is the identity and $\uvar\otimes_1[1]$ is the tensor product with $[1]$.

Suppose the result is true at the stage $k$ for $k>1$. If $n=0$, remark that $E^{eq}\otimes_k[1]$ (resp. $ 1\otimes_k[1]$) is equivalent to $E^{eq}\otimes[1]$ (resp. $ 1\otimes[1]$) and the morphism is then in $\widehat{\W}$. Now, if $n>0$, formula \eqref{eq:eq for k cylinder} implies that $(\Sigma^nE^{eq})\otimes_k[1]\to (\Sigma^n1)\otimes_k[1]$ is the colimit in depth of the following diagram:
\[\begin{tikzcd}[column sep=0.2cm]
	{[ \Sigma^{n-1}E^{eq}\otimes_{k-1}\{0\},1]} && {[ \Sigma^{n-1}E^{eq}\otimes_{k-1}\{1\},1]} \\
	{[1]\vee [ \Sigma^{n-1}E^{eq},1]} & {[ \Sigma^{n-1}E^{eq}\otimes_{k-1}[1],1]} & {[ \Sigma^{n-1}E^{eq},1]\vee[1]} \\
	& {[ \Sigma^{n-1}1\otimes_{k-1}\{0\},1]} && {[ \Sigma^{n-1}1\otimes_{k-1}\{1\},1]} \\
	& {[1]\vee [ \Sigma^{n-1}1,1]} & {[ \Sigma^{n-1}1\otimes_{k-1}[1],1]} & {[ \Sigma^{n-1}1,1]\vee[1]}
	\arrow[from=1-1, to=2-1]
	\arrow[from=1-1, to=2-2]
	\arrow[from=1-3, to=2-2]
	\arrow[from=1-3, to=2-3]
	\arrow[from=3-2, to=4-2]
	\arrow[from=3-2, to=4-3]
	\arrow[from=3-4, to=4-3]
	\arrow[from=3-4, to=4-4]
	\arrow[from=1-1, to=3-2]
	\arrow[from=2-2, to=4-3]
	\arrow[from=1-3, to=3-4]
	\arrow[from=2-1, to=4-2]
	\arrow[from=2-3, to=4-4]
\end{tikzcd}\]
by induction hypothesis, and using lemma \ref{lemma:the functor [] preserves classes}, all the morphisms in depth are in $\widehat{\W}$, and so is their colimit.
\end{proof}

\begin{prop}
\label{prop:eq:for k cylinder}
There is a natural identification between $[C,1]\otimes_{k+1} [1]$ and the colimit of the following diagram
\begin{equation}
\label{eq:eq for k cylinder}
\begin{tikzcd}
	{[1]\vee [ C,1]} & {[C\otimes_k\{0\},1]} & {[C\otimes_k [1],1]} & {[C\otimes_k\{1\},1]} & {[C,1]\vee[1]}
	\arrow[from=1-2, to=1-1]
	\arrow[from=1-2, to=1-3]
	\arrow[from=1-4, to=1-3]
	\arrow[from=1-4, to=1-5]
\end{tikzcd}
\end{equation}
\end{prop}
\begin{proof}
This is a consequence of proposition \ref{prop:eq for cylinder}.
\end{proof}

\begin{definition}
We denote by
$$(\uvar)^{[1]_k}:\ocat\to \ocat.$$
$k$-Gray cynlinder.
the right adjoint of the $k$-Gray cylinder.\sym{(ck@$C^{[1]_k}$}
\end{definition}

\subsection{Gray deformation retract}
\label{subsection:Gray deformation retract}

\begin{definition}
 A \wcnotion{left $k$-Gray deformation retract structure}{left or right $k$-Gray deformation retract structure} for a morphism $i:C\to D$ is the data of a \textit{retract}
 $r:D\to C$, a \textit{deformation} $\psi:D\otimes_k [1]\to D$, and equivalences
$$ri\sim id_C~~~~~\psi_{|D\otimes_k\{0\}}\sim ir~~~~~\psi_{|D\otimes_k\{1\}}\sim id_D~~~~~ \psi_{|C\otimes_k[1]}\sim i\cst_C
$$ 
A morphism $i:C\to D$ between $\io$-categories is a \wcnotion{left $k$-Gray deformation retract}{left or right $k$-Gray deformation retract} if it admits a left deformation retract structure. By abuse of notation, such data will just be denoted by $(i,r,\psi)$.
\end{definition}

\begin{definition}
We define dually the notion of \textit{right $k$-Gray deformation retract structure} and of \textit{right $k$-Gray deformation retract} in exchanging $0$ and $1$ in the previous definition.
\end{definition}

\begin{definition}
 A \textit{left $k$-Gray deformation retract structure} for a morphism $i:f\to g$ in the $\iun$-category of arrows of $\ocat$ is the data of a \textit{retract}
 $r:g\to f$, a \textit{deformation} $\psi:g\otimes_k [1]\to g$ and equivalences
$$ri\sim id_f~~~~~\psi_{|g\otimes_k\{0\}}\sim ir~~~~~\psi_{|g\otimes_k\{1\}}\sim id_D~~~~~ \psi_{|f\otimes_k[1]}\sim i\cst_C
$$ 
A morphism $i:C\to D$ between arrows of $\ocat$ is a \textit{left $k$-Gray deformation retract} if it admits a left deformation retract structure. By abuse of notation, such data will just be denoted by $(i,r,\psi)$.
\end{definition}

\begin{definition}
We define dually the notion of \textit{right $k$-Gray deformation retract structure} and of \textit{right $k$-Gray deformation retract} in exchanging $0$ and $1$ in the previous definition.
\end{definition}

\begin{example}
\label{example:canonical example of left deformation retract unmarked}
Let $k\in \Nb\cup\{\omega\}$ and 
let $C$ be an $(\infty,k)$-category. We consider the morphism $i:C\otimes\{0\}\to C\otimes[1]$. We define $r:C\otimes[1]\xrightarrow{C\otimes\Ib} C\otimes\{0\}$. Eventually, we set 
$$\psi:C\otimes[1]\otimes[1]\to C\otimes([1]\times [1])\xrightarrow{C\otimes \phi}C\otimes[1]$$
where the first morphism is constructed in \ref{cons:almost assoc of gray}, and
where $\phi:[1]\times[1]$ is the morphism sending $(i,j)$ on the minimum of $i$ and $j$. 

As $C$ is an $(\infty,k)$-category, $\psi$ factors through $C\otimes[1]\to \tau^i_k(C\otimes[1])\sim C\otimes_k[1]$. We denote by $\phi:C\otimes_k[1]\to C\otimes\{0\}$ the induced morphism.
The triple $(i,r,\phi)$ is a left $k$-Gray deformation retract structure. Conversely, 
 $C\otimes\{1\}\to C\otimes[1]$ is a right deformation retract.

One can show similarly that $1\to 1\costar C$ is a left $k$-Gray deformation retract, and $1\to C\star 1$ is a right $k$-Gray deformation retract.
\end{example}

\vspace{1cm} The $\infty$-groupoid of left and right Gray retracts enjoys many stability properties: 
\begin{prop}
\label{prop:left Gray deformation retract stable under pushout unmarked}
Let $(i_a,r_a,\psi_a)$ be a natural familly of left (resp. right) $k$-Gray deformation retract structures indexed by an $(\infty,1)$-category $A$.
The triple $(\colim_{A}i_a,\colim_{A}r_a,\colim_{A}\psi_a)$ is a left (resp. right) $k$-Gray deformation retract structure.
\end{prop}
\begin{proof}
This is an immediate consequence of the fact that $\uvar\otimes_k[1]$ preserves colimits.
\end{proof}
\begin{prop}
\label{prop:stability under pullback unmarked}
Suppose that we have a diagram 
\[\begin{tikzcd}
	X & Y & Z \\
	X & {Y'} & {Z'}
	\arrow[from=1-1, to=2-1]
	\arrow[from=1-2, to=2-2]
	\arrow[from=1-3, to=2-3]
	\arrow["p", from=1-1, to=1-2]
	\arrow["q"', from=1-3, to=1-2]
	\arrow["{p'}"', from=2-1, to=2-2]
	\arrow["{q'}", from=2-3, to=2-2]
\end{tikzcd}\]
such that $p\to p'$ and $q\to q'$ are left (resp. right) $k$-Gray deformation retract. The induced square $q^*p\to (q')^*p'$ is a left (resp. right) $k$-Gray deformation retract.
\end{prop}
\begin{proof}
The proof is an easy diagram chasing.
\end{proof}
\begin{prop}
\label{prop:stability by composition unmarked}
If $p\to p'$ and $p'\to p''$ are two left (resp. right) $k$-Gray deformation retracts, so is $p\to p''$.
\end{prop}
\begin{proof}
The proof is an easy diagram chasing.
\end{proof}

\begin{prop}
\label{prop:Gray deformation retract and passage to hom unmarked}
Let $(i:C\to D,r,\psi)$ be a left (resp. right) $(k+1)$-Gray deformation structure. For any $x: C$ and $y:D$ (resp. $x: D$ and $y:C$), the morphism
$$\begin{array}{cc}
&\hom_C(x,ry)\xrightarrow{i} \hom_D(ix,iry)\xrightarrow{{\psi_y}_!} \hom_D(ix,y)\\
(resp. &\hom_C(rx,y)\xrightarrow{i} \hom_D(irx,iy)\xrightarrow{{\psi_x}_!} \hom_D(x,iy))
\end{array}
$$
is a right (resp. left) $k$-Gray deformation retract, whose retract is given by 
$$\begin{array}{cc}
&\hom_D(ix,y)\xrightarrow{r}\hom_C(x,ry)\\
(resp. &\hom_D(x,iy)\xrightarrow{r}\hom_C(rx,y))
\end{array}$$
\end{prop}
\begin{proof}
By currying $\psi$, this induces a diagram
\[\begin{tikzcd}
	& C & D \\
	D & {D^{[1]_{k+1}}} \\
	&& D
	\arrow["r", curve={height=-6pt}, from=2-1, to=1-2]
	\arrow["i", from=1-2, to=1-3]
	\arrow["id"', curve={height=12pt}, from=2-1, to=3-3]
	\arrow[from=2-2, to=3-3]
	\arrow[from=2-2, to=1-3]
	\arrow["\psi"{description}, from=2-1, to=2-2]
\end{tikzcd}\]
For any pair of objects $(z,y)$ of $D$, according to proposition \prop{prop:eq:for k cylinder}, this induces a diagram
\[\begin{tikzcd}
	& {\hom_D(z,y)} \\
	{\hom_D(z,y)} & {\hom_D(irz,iry)\times_{\hom_D(irz,y)}\hom_D(irz,y)^{[1]_k}\times_{\hom_D(irz,y)} \hom_D(z,y)} \\
	{\hom_C(rz,ry)} & {\hom_D(irz,iry)}
	\arrow["r", from=2-1, to=3-1]
	\arrow["i", from=3-1, to=3-2]
	\arrow["id", curve={height=-12pt}, from=2-1, to=1-2]
	\arrow[from=2-2, to=1-2]
	\arrow[from=2-2, to=3-2]
	\arrow["\psi"{description}, from=2-1, to=2-2]
\end{tikzcd}\]
If $z$ is of shape $ix$, the diagram becomes
\[\begin{tikzcd}
	& {\hom_D(ix,y)} & {\hom_D(ix,y)} \\
	{\hom_D(ix,y)} & {\hom_D(ix,iry)\times_{\hom_D(ix,y)}\hom_D(ix,y)^{[1]_k}} & {\hom_D(ix,y)^{[1]_k}} \\
	{\hom_C(x,ry)} & {\hom_D(ix,iry)} & {\hom_D(ix,y)}
	\arrow["r"', from=2-1, to=3-1]
	\arrow["id", curve={height=-12pt}, from=2-1, to=1-2]
	\arrow[from=2-2, to=1-2]
	\arrow["\psi"{description}, from=2-1, to=2-2]
	\arrow["i"', from=3-1, to=3-2]
	\arrow[from=2-2, to=3-2]
	\arrow[""{name=0, anchor=center, inner sep=0}, "{{\psi_y}_!}"', from=3-2, to=3-3]
	\arrow[from=2-2, to=2-3]
	\arrow["id", from=1-2, to=1-3]
	\arrow[from=2-3, to=3-3]
	\arrow[from=2-3, to=1-3]
	\arrow["\lrcorner"{anchor=center, pos=0.125}, draw=none, from=2-2, to=0]
\end{tikzcd}\]
By decurrying, this induces a morphism $\phi:\hom_D(ix,y)\otimes_k[1]\to \hom_D(ix,y)$. We leave the reader verify that the triple $({\psi_y}_!i,r,\phi)$ is a right $k$-Gray deformation retract structure.
We proceed similarly for the other case.
\end{proof}

\begin{prop}
\label{prop:Gray deformation retract and passage to hom v2 unmarked}
For any left (resp. right) $(k+1)$-Gray deformation retract between $p$ and $p'$:
\[\begin{tikzcd}
	C & D \\
	{C'} & {D'}
	\arrow["p"', from=1-1, to=2-1]
	\arrow["i", from=1-1, to=1-2]
	\arrow["{p'}", from=1-2, to=2-2]
	\arrow["{i'}"', from=2-1, to=2-2]
\end{tikzcd}\]
and for any pair of objects $x: C$ and $y:D$ (resp. $x: D$ and $y:C$), the outer square of the following diagram
\[\begin{tikzcd}
	{\hom_{C}(x,ry)} & {\hom_{D}(ix,iry)} & {\hom_{D}(ix,y)} \\
	{\hom_{C'}(px,pr'y)} & {\hom_{D'}(p'i'x,p'i'r'y)} & {\hom_{D'}(p'i'x,p'y)}
	\arrow["{i'}"', from=2-1, to=2-2]
	\arrow["{{\psi'_{p'y}}_!}"', from=2-2, to=2-3]
	\arrow[from=1-1, to=2-1]
	\arrow[from=1-3, to=2-3]
	\arrow["{{\psi_y}_!}", from=1-2, to=1-3]
	\arrow["i", from=1-1, to=1-2]
	\arrow[from=1-2, to=2-2]
\end{tikzcd}\]
(resp.
\[\begin{tikzcd}
	{\hom_{C}(rx,y)} & {\hom_{D}(irx,iy)} & {\hom_{D}(x,iy)} \\
	{\hom_{C'}(pr'x,py)} & {\hom_{D'}(p'i'r'x,p'i'y)} & {\hom_{D'}(p'x,p'i'y)\big)}
	\arrow["{i'}"', from=2-1, to=2-2]
	\arrow["{{\psi'_{p'x}}_!}"', from=2-2, to=2-3]
	\arrow[from=1-1, to=2-1]
	\arrow[from=1-3, to=2-3]
	\arrow["{{\psi_x}_!}", from=1-2, to=1-3]
	\arrow["i", from=1-1, to=1-2]
	\arrow[from=1-2, to=2-2]
\end{tikzcd}\]
is a left (resp. right) $(k+1)$-Gray deformation retract, whose retract is given by
\[\begin{tikzcd}
	{\hom_{D}(ix,y)} & {\hom_{C}(x,ry)} & {(resp.\hom_{D}(x,iy)} & {\hom_{C}(rx,y)} \\
	{\hom_{D'}(p'i'x,p'y)\big)} & {\hom_{C'}(px,pr'y)} & {\hom_{D'}(p'x,p'i'y)} & {\hom_{C'}(pr'x,py)\big)}
	\arrow[from=1-3, to=2-3]
	\arrow["r", from=1-3, to=1-4]
	\arrow["{r'}"', from=2-3, to=2-4]
	\arrow[from=1-4, to=2-4]
	\arrow["{r'}"', from=2-1, to=2-2]
	\arrow["r", from=1-1, to=1-2]
	\arrow[from=1-2, to=2-2]
	\arrow[from=1-1, to=2-1]
\end{tikzcd}\]
\end{prop}
\begin{proof}
This comes from the fact that the construction of the retraction and the deformation in the previous proposition was functorial.
\end{proof}

\begin{prop}
\label{prop:suspension of left Gray deformation retract unmarked}
If $i$ is a left $k$-Gray deformation retract, $[i,1]$ is a right $(k+1)$-Gray deformation retract. Conversely, if $i$ is a right $k$-Gray deformation retract, $[i,1]$ is a left $(k+1)$-Gray deformation retract morphism.
\end{prop}
\begin{proof}
Let $(i:C\to D,r,\phi)$ be a left $k$-Gray deformation retract structure. Using proposition \ref{prop:eq:for k cylinder}, we define the morphism $\psi:[D,1]\otimes_{k+1}[1]\to [D,1]$ as the horizontal colimit of the following diagram:
\[\begin{tikzcd}
	{[1]^{}\vee[D,1]} & {[D\otimes_k\{0\},1]} & {[D\otimes_k[1],1]} & {[D\otimes_k\{1\},1]} & {[D,1]\vee[1]^{}} \\
	& {[C,1]} & {[D,1]} & {[D,1]}
	\arrow[from=1-4, to=1-5]
	\arrow["{[\phi,1]}"', from=1-3, to=2-3]
	\arrow[from=1-2, to=1-1]
	\arrow[from=1-2, to=1-3]
	\arrow[from=1-4, to=1-3]
	\arrow["{[i,1]}"', from=2-2, to=2-3]
	\arrow["{[r,1]}"', from=1-2, to=2-2]
	\arrow["{[id,1]}", from=2-4, to=2-3]
	\arrow["{[id,1]}", from=1-4, to=2-4]
	\arrow[from=1-1, to=2-2]
	\arrow[from=1-5, to=2-4]
\end{tikzcd}\]
Eventually, remark that the triple $([i,1],[r,1],\psi)$ is a right $(k+1)$-Gray deformation retract. The other assertion is demonstrated similarly.
\end{proof}

\begin{prop}
\label{prop:of left Gray deformation retract unmarked}
For any integer $n$,
if $n$ is even, $i_{n}^-:\Db_{n}\to \Db_{n+1}$ is a left $n$-Gray deformation retract and $i_{n}^+:\Db_{n}\to \Db_{n+1}$ is a right $n$-Gray deformation retract, and if $n$ is odd, $i_{n}^-$ is a right $n$-Gray deformation retract and $i_{n}^+$ is a left $n$-Gray deformation retract.
\end{prop}
\begin{proof}
It is obvious that $\{0\}\to [1]$ is a left $1$-Gray deformation retract and $\{1\}\to [1]$ is a right $1$-Gray deformation retract. A repeated application of \ref{prop:suspension of left Gray deformation retract unmarked} proves the assertion.
\end{proof}

\begin{prop}
\label{prop:when glob inclusion are left Gray deformation unmarked}
Let $a$ be a globular sum of dimension $(n+1)$. We denote by $s_n(a)$ and $t_n(a)$ the globular sum defined in definition \ref{defi:definition of source et but}.

If $n$ is even, $s_n(a)\to a$ is a left $n$-Gray deformation retract and $t_n(a)\to a$ is a right $n$-Gray deformation retract, and if $n$ is odd, $s_n(a)\to a$ is a right $n$-Gray deformation retract and $t_n(a)\to a$ is a left $n$-Gray deformation retract.
\end{prop}
\begin{proof}
This is a direct consequence of proposition \ref{prop:of left Gray deformation retract unmarked} and \ref{prop:left Gray deformation retract stable under pushout unmarked} as $s_n(a)\to a$ is a composition of pushouts of $i_{n}^-:\Db_{n}\to (\Db_{n+1})_t$. The other assertion is proved similarly.
\end{proof}

\subsection{Gray operations and strict objects}

\label{section:on preservation of strict}
Recall that in construction \ref{cons:strict} we defined an adjunction
\[\begin{tikzcd}
	{\pi_0:\ocat} & {\zocat:\N}
	\arrow[""{name=0, anchor=center, inner sep=0}, shift left=2, from=1-1, to=1-2]
	\arrow[""{name=1, anchor=center, inner sep=0}, shift left=2, from=1-2, to=1-1]
	\arrow["\dashv"{anchor=center, rotate=-90}, draw=none, from=0, to=1]
\end{tikzcd}\]
An $\io$-category lying in the image of the nerve functor $\N$ is called \textit{strict}.

\begin{prop}
\label{prop:pi0 preserves gray operation}
The functor $\pi_0$ preserves the Gray cylinder (defined in \ref{defi:of gray cylinder for strict} for $\zo$-categories and in \ref{defi:of gray cylinder} for $\io$-categories), the suspension (defined in \ref{defi:suspension zocat} for $\zo$-categories and in \ref{cons:def of suspension} for $\io$-categories), the Gray cone, and the Gray $\circ$-cone (defined in \ref{cons:Gray cone for omega cat} for $\zo$-categories and in \ref{cons:of gray cone} for $\io$-categories).
\end{prop}
\begin{proof}
We will demonstrate the result only in the case of the Gray cylinder, as the other cases are similar. 
As $\pi_0$ and $\uvar\otimes[1]$ preserve colimits, it is sufficient to demonstrate that the following triangle commutes up to an invertible natural transformation:
\[\begin{tikzcd}
	\Theta & \ocat \\
	& \zocat
	\arrow["{\uvar\otimes[1]}", from=1-1, to=1-2]
	\arrow["{\uvar\otimes[1]}"', from=1-1, to=2-2]
	\arrow["{\pi_0}", from=1-2, to=2-2]
\end{tikzcd}\]

Furthermore, an induction using  proposition \ref{prop:appendice formula for otimes} and proposition \ref{prop:eq for cylinder} provides a natural isomorphism between $\pi_0(\Db_n\otimes[1])$ and $\Db_n\otimes[1]$. The theorem \ref{theo:appendince unicity of operation} then provides the desired invertible natural transformation.
\end{proof}

The strict categories play an important role as they allow us to make explicit calculations. In particular, it will be very useful to know which cocontinuous functors preserve them.
\begin{prop}
\label{prop:criter stricte easy}
An $\io$-category $C$ is strict if and only if $C_0$ is a set and for any pair of objects $x,y$, $\hom_C(x,y)$ is strict.
\end{prop}
\begin{proof}
By definition, an $\io$-category is strict if and only if, for any globular sum $[\textbf{b},n]$, $\Hom([\textbf{b},n],C)$ is a set. However, as $C$ is $\W$-local, we have an equivalence between $\Hom([\textbf{b},n],C)$ and 
$$\coprod_{x_0,x_1,...,x_n\in C_0}\Hom(b_1, \hom_C(x_0,x_1))\times...\times \Hom(b_n, \hom_C(x_{n-1},x_n))$$
As all the objects of the previous expression are set by hypothesis, and as the inclusion of set into $\infty$-groupoid is stable under coproduct and product, $\Hom([b,n],C)$ is a set.
\end{proof}

\begin{prop}
\label{prop:suspension preserves stricte}
If $C$ is a strict $\io$-category, so is $[C,1]$. 
\end{prop}
\begin{proof}
If $C$ is a globular sum of the empty $\io$-category, this is obvious.

Suppose now that $C$ is any strict $\io$-category. Proposition \ref{prop:elelangat stable by slice} and corollary \ref{cor:thetaplus is reedy} imply that $[\uvar,1]:\Theta^+_{/C}\to \iPsh{\Theta}$ is Reedy cofibrant. According to proposition \ref{prop:supspension preserves cat}, the colimit of this diagram, computed in $\iPsh{\Theta}$, is $[C,1]$.
Lemma \ref{lemma:colimit computed in set presheaves} then implies that this object is strict.
\end{proof}

\begin{lemma}
\label{lemma:gray operation on globes are strict}
For any $n$, $\Db_n\otimes[1]$, $\Db_n\star 1$ and $1\costar \Db_n$ are strict.
\end{lemma}
\begin{proof}
We proceed by induction on $n$. The result is obviously true for $n=0$. Suppose it is true as the stage $n$. According to proposition \ref{prop:eq for cylinder}, $\Db_n\otimes[1]$ is the colimit of the following diagram
\begin{equation}
\label{eq:tensor of globuees is sitrict}
\begin{tikzcd}
	{[1]\vee \Db_n} & {\Db_n} & {[\Db_{n-1}\otimes [1],1]} & {\Db_n} & {\Db_n\vee[1]}
	\arrow[from=1-2, to=1-1]
	\arrow[from=1-2, to=1-3]
	\arrow[from=1-4, to=1-3]
	\arrow[from=1-4, to=1-5]
\end{tikzcd}
\end{equation}
The induction hypothesis and proposition \ref{prop:suspension preserves stricte} implies that all the objects are strict.
The proposition \ref{prop:cartesian squares} then implies that the diagram
\[\begin{tikzcd}
	{\Db_{n-1}} & {\Db_{n-1}\otimes[1]} & {\Db_{n-1}} \\
	{\{0\}} & {[1]} & {\{1\}}
	\arrow[from=1-1, to=2-1]
	\arrow[from=1-2, to=2-2]
	\arrow[from=1-3, to=2-3]
	\arrow[from=2-1, to=2-2]
	\arrow[from=1-1, to=1-2]
	\arrow[from=1-3, to=1-2]
	\arrow[from=2-3, to=2-2]
\end{tikzcd}\]
verifies the hypothesis of proposition \ref{prop:example of a special colimit3}. 
This implies that the colimit of \eqref{eq:tensor of globuees is sitrict} is special. The inductive hypothesis and lemma \ref{lemma:colimit computed in set presheaves} imply that its colimit is then strict. As this colimit is $\Db_n\otimes[1]$, this concludes the proof by induction.

We proceed similarly for the Gray cone and the Gray $\circ$-cone.
\end{proof}

\begin{prop}
\label{prop:comparaison betwen otimes and suspension week case withou the cocartesian square.}
There exists a natural transformation $C\otimes [1]\to [C,1]$ whose restriction to $C\otimes\{0\}$ (resp. to $C\otimes\{1\}$) is constant on $\{0\}$ (resp. on $\{1\}$) and such that the following square is cocartesian:
\begin{equation}
\label{eq:cocartesian link betwen gray and suspension}
\begin{tikzcd}
	{C\otimes\{0\}\coprod C\otimes\{1\}} & {C\otimes[1]} \\
	{\{0\}\coprod \{1\}} & {[C,1]}
	\arrow[from=1-1, to=1-2]
	\arrow[from=1-1, to=2-1]
	\arrow[from=1-2, to=2-2]
	\arrow[from=2-1, to=2-2]
	\arrow["\lrcorner"{anchor=center, pos=0.125, rotate=180}, draw=none, from=2-2, to=1-1]
\end{tikzcd}
\end{equation}
\end{prop}
\begin{proof}
As $\uvar\otimes[1]$ and $[\uvar,1]$ preserve colimits, it is sufficient to construct the natural transformation on globular sums. As $\pi_0$ commutes with the Gray cylinder, proposition \ref{prop:comparaison betwen otimes and suspension} induces a comparison
$$\pi_0 (a\otimes[1]) \to [a,1]$$ natural in $a:\Theta$. By adjunction, this induces a natural transformation 
$$a\otimes[1]\to [a,1].$$ It remains to demonstrate the cocartesianness of the square \eqref{eq:cocartesian link betwen gray and suspension}. As cocartesian squares are stable under colimits, we can reduce to the case where $C$ is equivalent to $\Db_n$ for $n$ an integer. We then proceed by induction on $n$. The case $n=0$ is trivial, as proposition \ref{prop:eq for cylinder} implies that $\Db_0\otimes[1]$ is equivalent to $[1]$. Suppose the result is true at the stage $n$. Using once again proposition \ref{prop:eq for cylinder}, we have that 
$$\Db_{n+1}\otimes[1]\coprod_{\Db_{n+1}\otimes (\{0\}\amalg \{1\})} \{0\}\amalg \{1\}$$ is equivalent to the colimit of the diagram
\[\begin{tikzcd}
	{[1]} & {[\Db_n\otimes\{0\},1]} & {[\Db_n\otimes[1],1]} & {[\Db_n\otimes\{1\},1]} & {[1]}
	\arrow[from=1-2, to=1-1]
	\arrow[from=1-2, to=1-3]
	\arrow[from=1-4, to=1-3]
	\arrow[from=1-4, to=1-5]
\end{tikzcd}\]
and by induction hypothesis, it is equivalent to 
$$[\Db_{n+1},1],$$
which concludes the proof.
\end{proof}

\begin{lemma}
\label{lemma:strictification2}
Let $\alpha$ be $-$ if $n$ is even (resp. odd) and $+$ if $n$ is odd (resp.even).
Consider a cartesian square
\begin{equation}
\label{eq:lemma:strictification2}
\begin{tikzcd}
	{C_0} & D \\
	{\Db_{n}} & {\Db_{n+1}}
	\arrow["p"', from=1-1, to=2-1]
	\arrow["{p'}", from=1-2, to=2-2]
	\arrow[from=1-1, to=1-2]
	\arrow["\lrcorner"{anchor=center, pos=0.125}, draw=none, from=1-1, to=2-2]
	\arrow["{i_{n}^\alpha}"', from=2-1, to=2-2]
\end{tikzcd}
\end{equation}
such that $p\to p'$ is a left $(n+1)$-Gray deformation retract (resp. a right $(n+1)$-Gray deformation retract). Let $C_1$ be the $\io$-category fitting in the pullback 
\begin{equation}
\label{eq:lemma:strictification3}
\begin{tikzcd}
	{C_1} & D \\
	{\Db_{n}} & {\Db_{n+1}}
	\arrow["p"', from=1-1, to=2-1]
	\arrow["{p'}", from=1-2, to=2-2]
	\arrow[from=1-1, to=1-2]
	\arrow["\lrcorner"{anchor=center, pos=0.125}, draw=none, from=1-1, to=2-2]
	\arrow["{i_{n}^{1-\alpha}}"', from=2-1, to=2-2]
\end{tikzcd}
\end{equation}
Then if $C_0$ and $C_1$ are strict, so is $D$.
\end{lemma}
\begin{proof}
We denote by $(i,r,\phi)$ the deformation retract structure corresponding to $C_0\to D$. 
We show this result by induction, and let's start with the case $n=0$. 
This corresponds to the case where $C_0\to D$ fits in a pullback diagram.
\[\begin{tikzcd}
	{C_0} & D \\
	{\{0\}} & {[1]}
	\arrow[from=1-1, to=1-2]
	\arrow[from=2-1, to=2-2]
	\arrow[from=1-1, to=2-1]
	\arrow[from=1-2, to=2-2]
\end{tikzcd}\]
 Let $x,y$ be two objects of $D$. Suppose first that $x$ and $y$ are over the same object of $[1]$. In this case, $\hom_D(x,y)$ is equivalent to either $\hom_{C_0}(x,y)$ or $\hom_{C_1}(x,y)$ and is then strict. If $x$ is over $1$ and $y$ over $0$, the $\infty$-groupoid $\hom_D(x,y)$ is empty. If $x$ is over $0$ and $y$ is over $1$, $\hom_D(x,y)$ is equivalent to $\hom_{C_0}(x,ry)$ according to \ref{prop:Gray deformation retract and passage to hom unmarked} and is then strict by hypothesis. Eventually, $\tau_0(D)$ is equivalent to $\tau_0(C_1)$ and is then a set. According to \ref{prop:criter stricte easy}, this implies that $D$ is strict.

Suppose now the result is true at stage $(n-1)$.
Let $p'\to p$ be a square verifying the condition.  Remark that, at the level of objects, the inclusion $C_0\to D$, its retract, and its deformation, are the identity.

Let $x$ and $y$ be two objects of $D$. 
As before, the only interesting case is when $x$ is over $0$ and $y$ is over $1$. In this case,
applying $\hom(\uvar,\uvar)$ to the square \eqref{eq:lemma:strictification2}, we get a cartesian square
\[\begin{tikzcd}
	{\hom_{C_0}(x,y)} & {\hom_D(x,y)} \\
	{\Db_{n-1}} & {\Db_{n}}
	\arrow[from=1-1, to=1-2]
	\arrow["{i_{n-1}^\alpha}"', from=2-1, to=2-2]
	\arrow[from=1-1, to=2-1]
	\arrow[from=1-2, to=2-2]
\end{tikzcd}\]
which is a right $n$-Gray deformation retract according to proposition \ref{prop:Gray deformation retract and passage to hom unmarked}.
Applying $\hom(\uvar,\uvar)$ to the square \eqref{eq:lemma:strictification3}, we get a cartesian square
\[\begin{tikzcd}
	{\hom_{C_1}(x,y)} & {\hom_D(x,y)} \\
	{\Db_{n-1}} & {\Db_{n}}
	\arrow[from=1-1, to=1-2]
	\arrow["{i_{n-1}^{1-\alpha}}"', from=2-1, to=2-2]
	\arrow[from=1-1, to=2-1]
	\arrow[from=1-2, to=2-2]
\end{tikzcd}\]
As $C_1$ is strict, so is $\hom_{C_1}(x,y)$.
 We can then apply the induction hypothesis, which implies that $\hom_D(x,y)$ is strict. As $\tau_0 D$ is equivalent to $\tau_0 C_{0}$, it is a set. We can apply proposition \ref{prop:criter stricte easy} which implies that $D$ is strict. 
\end{proof}

\begin{construction}
For an integer $n>0$,
we define by induction 
\begin{enumerate}
\item[$-$]
a left $(n+1)$-Gray retract structure for the inclusion 
\begin{equation}
\label{eq:Gray retract structurure for Gray cone}
\Db_n\star\emptyset \cup \Db_{n-1}\star 1 \to \Db_n\star 1
\end{equation}
where the gluing is performed along $i_n^\alpha:\Db_{n-1}\star\emptyset\to \Db_n\star\emptyset$ with $\alpha$ being $+$ if $n$ is odd and $-$ if not, 
\item[$-$]
a right $(n+1)$-Gray retract structure for the inclusion
\begin{equation}
\label{eq:Gray retract structurure for circ Gray cone}
1\costar\Db_{n-1} \cup \emptyset\costar\Db_{n}\to 1 \costar \Db_n
\end{equation}
where the gluing is performed along $i_n^\alpha:\emptyset\costar \Db_{n-1}\to \emptyset\costar \Db_n$ with $\alpha$ being $-$ if $n$ is odd and $+$ if not. 
\end{enumerate}
If $n=1$, the first morphism corresponds to the inclusion
\[\begin{tikzcd}
	\bullet & {} & \bullet \\
	\bullet & \bullet & \bullet & \bullet
	\arrow[""{name=0, anchor=center, inner sep=0}, from=1-3, to=2-3]
	\arrow[""{name=1, anchor=center, inner sep=0}, from=2-3, to=2-4]
	\arrow[""{name=2, anchor=center, inner sep=0}, from=1-3, to=2-4]
	\arrow[from=1-1, to=2-1]
	\arrow[from=2-1, to=2-2]
	\arrow[""{name=3, anchor=center, inner sep=0}, draw=none, from=1-2, to=2-2]
	\arrow[shift right=2, shorten <=2pt, shorten >=2pt, Rightarrow, from=2, to=1]
	\arrow[shorten <=6pt, shorten >=6pt, maps to, from=3, to=0]
\end{tikzcd}\]
and the second one to the inclusion:
\[\begin{tikzcd}
	& \bullet & {} & \bullet \\
	\bullet & \bullet & \bullet & \bullet
	\arrow[""{name=0, anchor=center, inner sep=0}, from=1-2, to=2-2]
	\arrow[""{name=1, anchor=center, inner sep=0}, from=2-3, to=2-4]
	\arrow[""{name=2, anchor=center, inner sep=0}, from=2-3, to=1-4]
	\arrow[from=1-4, to=2-4]
	\arrow[from=2-1, to=1-2]
	\arrow[""{name=3, anchor=center, inner sep=0}, draw=none, from=1-3, to=2-3]
	\arrow[shift left=2, shorten <=2pt, shorten >=2pt, Rightarrow, from=2, to=1]
	\arrow[shorten <=6pt, shorten >=6pt, maps to, from=0, to=3]
\end{tikzcd}\]
The propositions \ref{prop:of left Gray deformation retract unmarked} and \ref{prop:left Gray deformation retract stable under pushout unmarked} imply that the first morphism is a left $2$-Gray deformation retract and the second one a right $2$-Gray deformation retract. 
Suppose now that these two morphisms are constructed at stage $n$.
The proposition \ref{prop:eq for Gray cone} implies that $\Db_{n+1}\star\emptyset \cup \Db_{n}\star 1 \to \Db_{n+1}\star 1$ fits in the cocartesian square
\[\begin{tikzcd}
	{[1\costar\Db_{n-1}\cup \emptyset\star\Db_n,1]} & {\Db_n\star\emptyset\cup \Db_{n-1}\star 1} \\
	{[1\costar\Db_n,1]} & { \Db_n\star 1}
	\arrow[from=2-1, to=2-2]
	\arrow[from=1-1, to=1-2]
	\arrow[from=1-1, to=2-1]
	\arrow[from=1-2, to=2-2]
\end{tikzcd}\]
The induction hypothesis and the propositions \ref{prop:suspension of left Gray deformation retract unmarked} and \ref{prop:left Gray deformation retract stable under pushout unmarked} endow this morphism with a left $(n+2)$-Gray retract structure. We constructs similarly the right $(n+2)$-Gray retract structure for the inclusion $1\costar\Db_{n-1} \cup \emptyset\costar\Db_{n}\to 1 \costar \Db_n$.
\end{construction}

\begin{prop}
\label{prop:strict stuff are stable under Gray cone}
Let $C$ be a strict $\io$-category, $a$ a globular sum, and $f:a\to C$ any morphism. The $\io$-categories $C\coprod_a a\star 1$ and $1\costar a\coprod_a C$ are strict.
In particular, $a\star 1$ and $1\costar a$ are strict.
\end{prop}
\begin{proof}
We will prove the result by induction on the number of non-identity cells of $a$. Remark that for any globular sum $b$, there exists a globular sum $a$, an integer $n$, and a cocartesian square composed of globular morphism
\begin{equation}
\label{eq:first cocartesian square}
\begin{tikzcd}
	{\Db_{n-1}} & a \\
	{\Db_{n}} & b
	\arrow["{i^\alpha_{n-1}}"', from=1-1, to=2-1]
	\arrow["l"', from=2-1, to=2-2]
	\arrow["\lrcorner"{anchor=center, pos=0.125, rotate=180}, draw=none, from=2-2, to=1-1]
	\arrow[from=1-1, to=1-2]
	\arrow[from=1-2, to=2-2]
\end{tikzcd}
\end{equation}
with $\alpha=+$ if $n$ is odd, and $\alpha=-$ if $n$ is even, and such that $l$ admits a retract $r$. As $i^\alpha_{n-1}$ is globular, the pullback along this morphism preserves colimits according to theorem \ref{theo:pullback along conduche preserves colimits}. We then have a cartesian square:
\[\begin{tikzcd}
	a & b \\
	{\Db_{n-1}} & {\Db_{n}}
	\arrow["r", from=1-2, to=2-2]
	\arrow["{i^\alpha_{n-1}}"', from=2-1, to=2-2]
	\arrow[from=1-1, to=2-1]
	\arrow[from=1-1, to=1-2]
\end{tikzcd}\]
 We also define $a'$ as the pullback:
\[\begin{tikzcd}
	{a'} & b \\
	{\Db_{n-1}} & {\Db_{n}}
	\arrow["r", from=1-2, to=2-2]
	\arrow["{i^{-\alpha}_{n-1}}"', from=2-1, to=2-2]
	\arrow[from=1-1, to=2-1]
	\arrow[from=1-1, to=1-2]
\end{tikzcd}\]
and remark that $a'$ is a globular sum. Eventually, we fix a morphism $b\to C$. As $a$ and $a'$ are sub globular sum of $b$, the number of non-identity cells in each of them is strictly less than the one of $b$.
 We then suppose that for any strict $\io$-category $C$, and any morphism $b\to C$, the two induced $\io$-category 
 $C\coprod_a a\star 1$ and $C\coprod_{a'}a'\star 1$ are strict.
 
We are first willing to demonstrate that the two following squares are cartesian:
\begin{equation}
\label{the two annoying square}
\begin{tikzcd}
	{b\coprod_{a}a\star1} & {b\star 1} & {b\coprod_{a'}a'\star1} & {b\star 1} \\
	{[\Db_{n-1},1]} & {[\Db_{n},1]} & {[\Db_{n-1},1]} & {[\Db_{n},1]}
	\arrow[from=1-1, to=2-1]
	\arrow[from=1-1, to=1-2]
	\arrow["{[i^\alpha_{n-1},1]}"', from=2-1, to=2-2]
	\arrow[from=1-2, to=2-2]
	\arrow[from=1-3, to=2-3]
	\arrow[from=1-3, to=1-4]
	\arrow["{[i^{-\alpha}_{n-1},1]}"', from=2-3, to=2-4]
	\arrow[from=1-4, to=2-4]
\end{tikzcd}
\end{equation}
As $i^{\alpha}_{n-1}$, $i^{-\alpha}_{n-1}$, $[i^{\alpha}_{n-1},1]$, and $[i^{-\alpha}_{n-1},1]$ are globular, they are Conduché functors, and pullback along them preserves colimits according to theorem \ref{theo:pullback along conduche preserves colimits}.
The subcategory of $\Theta$ consisting of globular sums $b$ such that the two squares of \eqref{the two annoying square} are cartesian is then closed under colimits, and we can then reduce to the case where $b$ is a globe.

In this case, lemma \ref{lemma:gray operation on globes are strict}, proposition \ref{prop:suspension preserves stricte}, and the induction hypothesis imply that all the objects of these squares are strict. We can then show the cartesianess in $\zocat$, where it follows from lemma \ref{lemma: pullback and sum}.
This concludes the proof of the cartesianess of the square \eqref{the two annoying square}.

We claim that the square
\[\begin{tikzcd}
	{\{0\}} & {\{0\}} \\
	{[\Db_{n-1},1]} & {[\Db_n,1]}
	\arrow[from=1-1, to=1-2]
	\arrow[from=1-1, to=2-1]
	\arrow[from=1-2, to=2-2]
	\arrow[from=2-1, to=2-2]
\end{tikzcd}\]
is cartesian. This directly follows from proposition \ref{prop:cartesian squares} as all the objects are strict according to proposition \ref{prop:suspension preserves stricte}. This implies that the square
\begin{equation}
\label{eq:third square}
\begin{tikzcd}
	{C\otimes\{0\}} & {C\otimes\{0\}} \\
	{[\Db_{n-1},1]} & {[\Db_n,1]}
	\arrow[from=1-1, to=1-2]
	\arrow[from=1-1, to=2-1]
	\arrow[from=1-2, to=2-2]
	\arrow[from=2-1, to=2-2]
\end{tikzcd}
\end{equation}
is cartesian. Using one more time the fact that pullback along $[i^-_{n-1},1]$ and $[i^+_{n-1},1]$ preserves colimits, the cartesian squares \eqref{the two annoying square}
and \eqref{eq:third square} induces two cartesian squares:
\begin{equation}
\label{eq:two square in the proof of strict}
\begin{tikzcd}
	{C\coprod_{a}a\star1} & {C\coprod_bb\star 1} & {C\coprod_{a'}a'\star1} & {C\coprod_bb\star 1} \\
	{[\Db_{n-1},1]} & {[\Db_{n},1]} & {[\Db_{n-1},1]} & {[\Db_{n},1]}
	\arrow[from=1-1, to=2-1]
	\arrow[from=1-1, to=1-2]
	\arrow["{[i^\alpha_{n-1},1]}"', from=2-1, to=2-2]
	\arrow[from=1-2, to=2-2]
	\arrow[from=1-3, to=2-3]
	\arrow[from=1-3, to=1-4]
	\arrow["{[i^{-\alpha}_{n-1},1]}"', from=2-3, to=2-4]
	\arrow[from=1-4, to=2-4]
\end{tikzcd}
\end{equation}
By the induction hypothesis, the two top left objects are strict. 
Eventually, remark that the cocartesian \eqref{eq:first cocartesian square} induces  a cocartesian square
\[\begin{tikzcd}
	{\Db_n\coprod_{\Db_{n-1}}\Db_{n-1}\star 1} & {\Db_{n}\star 1} \\
	{C\coprod_{a}a\star1} & {C\coprod_bb\star 1}
	\arrow[from=2-1, to=2-2]
	\arrow[from=1-1, to=2-1]
	\arrow[from=1-1, to=1-2]
	\arrow[from=1-2, to=2-2]
	\arrow["\lrcorner"{anchor=center, pos=0.125, rotate=180}, draw=none, from=2-2, to=1-1]
\end{tikzcd}\] As the top horizontal morphism is a left $(n+1)$-Gray deformation retract, proposition \ref{prop:left Gray deformation retract stable under pushout unmarked} implies that the left square of \eqref{eq:two square in the proof of strict} also is. Lemma \ref{lemma:strictification2} then implies that $C\coprod_bb\star 1$ is strict. We prove similarly that $1\costar b\coprod_bC$ is strict.
\end{proof}

\vspace{1cm} 
We want to generalize proposition \ref{prop:strict stuff are stable under Gray cone} and provide an analogue for the Gray Cylinder. 
In what follows, we will utilize the results of sections \ref{section:Colimit of left cartesian fibrations} (specifically the corollaries \ref{cor:cor of the past101}, \ref{cor:cor of the past10}, \ref{cor:cor of the past3}).
We assure the reader that this is not a tautology, as the proofs of these results are not based on the following propositions and theorems.

\begin{lemma}
\label{lemma: first comming grom the future}
Let $j:C\to D$ be a morphism between $\io$-categories. The following squares are cartesian:
\[\begin{tikzcd}
	{1\costar C\coprod_CD} & {1\costar D} & {D\coprod_CC\star1} & {D\star 1} \\
	{[C,1]} & {[D,1]} & {[C,1]} & {[D,1]}
	\arrow[from=1-1, to=1-2]
	\arrow[from=1-1, to=2-1]
	\arrow["\lrcorner"{anchor=center, pos=0.125}, draw=none, from=1-1, to=2-2]
	\arrow[from=1-2, to=2-2]
	\arrow[from=1-3, to=1-4]
	\arrow[from=1-3, to=2-3]
	\arrow["\lrcorner"{anchor=center, pos=0.125}, draw=none, from=1-3, to=2-4]
	\arrow[from=1-4, to=2-4]
	\arrow["{[j,1]}"', from=2-1, to=2-2]
	\arrow[from=2-3, to=2-4]
\end{tikzcd}\]
\end{lemma}
\begin{proof}
This is the content of corollary \ref{cor:cor of the past10}.
\end{proof}

\begin{lemma}
\label{lemma: second comming grom the future}
Let $j:C\to D$ be a morphism between $\io$-categories. The following squares are cartesian:
\[\begin{tikzcd}
	{D\coprod_{C\otimes \{0\}}C\otimes[1]} & {D\star 1} && {C\otimes[1]\coprod_{C\otimes\{1\}}D} & {1\costar D} \\
	{1\costar C} & {[D,1]} && {C\star 1} & {[D,1]}
	\arrow[from=1-1, to=1-2]
	\arrow[from=1-1, to=2-1]
	\arrow["\lrcorner"{anchor=center, pos=0.125}, draw=none, from=1-1, to=2-2]
	\arrow[from=1-2, to=2-2]
	\arrow[from=1-4, to=1-5]
	\arrow[from=1-4, to=2-4]
	\arrow["\lrcorner"{anchor=center, pos=0.125}, draw=none, from=1-4, to=2-5]
	\arrow[from=1-5, to=2-5]
	\arrow[from=2-1, to=2-2]
	\arrow[from=2-4, to=2-5]
\end{tikzcd}\]
\end{lemma}
\begin{proof}
We will show only the cartesianness of the first square; the second follows by duality. The $\io$-category $1\costar C$ is equivalent to the colimit $\colim_{\Theta^{+}_{/C}}1\costar \uvar$ where $\Theta^+$ defined in \ref{defi:thetaplus}. Corollary \ref{cor:cor of the past3} states that pullback along the morphism $D\star 1\to [D,1]$ preserves colimits. We can then reduce to the case where $C$ is the empty $\io$-category or a globe. If $C$ is the empty $\io$-category, the result follows from lemma \ref{lemma: first comming grom the future}. It remains to demonstrate the case where $C$ is of shape $\Db_n$ for $n$ an integer.

 Remark now that the square factors as
\[\begin{tikzcd}
	{D\coprod_{\Db_n\otimes\{0\}}\Db_n\otimes[1]} & {D\coprod_{\Db_n}\Db_n\star 1} & {D\star 1} \\
	{1\costar \Db_n} & {[\Db_n,1]} & {[D,1]}
	\arrow[from=1-1, to=1-2]
	\arrow[from=1-1, to=2-1]
	\arrow[from=1-2, to=1-3]
	\arrow[from=1-2, to=2-2]
	\arrow[from=1-3, to=2-3]
	\arrow[from=2-1, to=2-2]
	\arrow[from=2-2, to=2-3]
\end{tikzcd}\]
The lemma \ref{lemma: first comming grom the future} implies that the right one is cartesian, and it is sufficient to demonstrate that the left one also is cartesian.

Using  corollary \ref{cor:cor of the past3} that states that pullback along the morphism $1\costar \Db_n\to [\Db_n,1]$ preserves colimits, it is sufficient to demonstrate that the two following squares are cartesian.
\[\begin{tikzcd}
	{\Db_n\otimes[1]} & {\Db_n\star1} & D & D \\
	{1\costar\Db_n} & {[\Db_n,1]} & {1\costar\Db_n} & {[\Db_n,1]}
	\arrow[from=1-1, to=1-2]
	\arrow[from=1-1, to=2-1]
	\arrow[from=1-2, to=2-2]
	\arrow[from=1-3, to=1-4]
	\arrow[from=1-3, to=2-3]
	\arrow[from=1-4, to=2-4]
	\arrow[from=2-1, to=2-2]
	\arrow[from=2-3, to=2-4]
\end{tikzcd}\]
For the first one, as all the objects are strict according to lemma \ref{lemma:gray operation on globes are strict}. We can then show the result in $\zocat$ where it follows from proposition \ref{prop:cartesian squares}. Remark now that the second square factors as
\[\begin{tikzcd}
	D & D \\
	1 & {\{0\}} \\
	{1\costar \Db_n} & {[\Db_n,1]}
	\arrow[from=1-1, to=1-2]
	\arrow[from=1-1, to=2-1]
	\arrow[from=1-2, to=2-2]
	\arrow[from=2-1, to=2-2]
	\arrow[from=2-1, to=3-1]
	\arrow[from=2-2, to=3-2]
	\arrow[from=3-1, to=3-2]
\end{tikzcd}\]
The top square is obviously cartesian. Remark that all objects of the second one are strict according to lemma \ref{lemma:gray operation on globes are strict}, and we can then show the cartesianess in $\zocat$ where it follows from proposition \ref{prop:cartesian squares}.
\end{proof}

\begin{lemma}
\label{lemma: thrid comming grom the future}
Let $j:C\to D$ and $k:C\to E$ be two morphisms between $\io$-categories. The following square is cartesian:
\[\begin{tikzcd}
	{D\coprod_{C\otimes \{0\}}C\otimes[1]\coprod_{C\otimes\{1\}}E} & {D\coprod_CC\star 1} \\
	{1\costar C\coprod_C E} & {[C,1]}
	\arrow[from=1-1, to=1-2]
	\arrow[from=1-1, to=2-1]
	\arrow[from=1-2, to=2-2]
	\arrow[""{name=0, anchor=center, inner sep=0}, from=2-1, to=2-2]
	\arrow["\lrcorner"{anchor=center, pos=0.125}, draw=none, from=1-1, to=0]
\end{tikzcd}\]
\end{lemma}
\begin{proof}
By remarking that $[\emptyset,1] \sim \{0\}\amalg \{1\}$ and $\emptyset \star 1\sim 1$, lemma \ref{lemma: first comming grom the future} induces a cartesian square:
\[\begin{tikzcd}
	{D\coprod 1} & {D\star 1} \\
	{\{0\}\coprod\{1\}} & {[D,1]}
	\arrow[from=1-1, to=1-2]
	\arrow[from=1-1, to=2-1]
	\arrow["\lrcorner"{anchor=center, pos=0.125}, draw=none, from=1-1, to=2-2]
	\arrow[from=1-2, to=2-2]
	\arrow[from=2-1, to=2-2]
\end{tikzcd}\]

This implies that the squares:
\[\begin{tikzcd}
	C & E & {D\star 1} \\
	{C\times\{1\}} & {E\times\{1\}} & {[D,1]}
	\arrow[from=1-1, to=1-2]
	\arrow[from=1-1, to=2-1]
	\arrow["\lrcorner"{anchor=center, pos=0.125}, draw=none, from=1-1, to=2-2]
	\arrow[from=1-2, to=1-3]
	\arrow[from=1-2, to=2-2]
	\arrow["\lrcorner"{anchor=center, pos=0.125}, draw=none, from=1-2, to=2-3]
	\arrow[from=1-3, to=2-3]
	\arrow[from=2-1, to=2-2]
	\arrow[from=2-2, to=2-3]
\end{tikzcd}\]
are cartesian. According to corollary \ref{cor:cor of the past3}, pullback along the morphism $D\star 1 \to [D,1]$ preserves colimits. Combined with lemma \ref{lemma: second comming grom the future}, this implies that the outer square of the diagram
\[\begin{tikzcd}
	{D\coprod_{C\otimes \{0\}}C\otimes[1]\coprod_{C\otimes\{1\}}E} & {D\coprod_CC\star 1} & {D\star 1} \\
	{1\costar C\coprod_C E} & {[C,1]} & {[D,1]}
	\arrow[from=1-1, to=1-2]
	\arrow[from=1-1, to=2-1]
	\arrow[from=1-2, to=1-3]
	\arrow[from=1-2, to=2-2]
	\arrow[from=1-3, to=2-3]
	\arrow[from=2-1, to=2-2]
	\arrow[from=2-2, to=2-3]
\end{tikzcd}\]
is cartesian. Furthermore, the lemma \ref{lemma: first comming grom the future} states the right hand square is cartesian. By the cancellation property of cartesian squares, this concludes the proof.
\end{proof}

\begin{theorem}
\label{theo:strict stuff are pushout}
Let $C\to D$ and $C\to E$ be two morphisms between strict $\io$-categories. The $\io$-categories
$D\coprod_{C}C\star 1$,  $1\costar C\coprod_CD$, and $D\coprod_{C\otimes\{0\}}C\otimes[1]\coprod_{C\otimes \{1\}}E$ are strict. In particular, $C\star 1$, $1\costar C$, and $C\otimes[n]$ for any integer $n$, are strict.
\end{theorem}
\begin{proof}
The proposition \ref{prop:suspension preserves stricte} implies that $[C,1]$ and $[D,1]$ are strict. As right adjoints preserve strict objects, so are $[C,1]_{0/}$,  $[C,1]_{/1}$,  $[D,1]_{0/}$, and $[D,1]_{/1}$. Corollary \ref{cor:cor of the past101} then implies that $1\costar C$, $C\star 1$, $1\costar D$, and $C\star 1$ are strict. As strict objects are closed by limits, lemmas \ref{lemma: first comming grom the future}, \ref{lemma: second comming grom the future} and \ref{lemma: thrid comming grom the future} conclude the proof.
\end{proof}

\begin{cor}
\label{cor:characterisaiont of Gray operation}
Let $F$ be an endofunctor of $\ocat$ such that the induced functor $\ocat\to \ocat_{F(\emptyset)/}$ is colimit preserving and $\psi$ an invertible natural transformation between $F(\Db_n)$ and $H(\Db_n)$ where $H$ is either the Gray cylinder, the Gray cone, the Gray $\circ$-cone or an iterated suspension.

Then, the natural transformation $\psi$ can be uniquely extended to an natural transformation between $F$ and $G$.  Moreover, this natural transformation is unique.
\end{cor}
\begin{proof}
We denote by $\Theta^+$ the category obtained from $\Theta$ by adding an initial element $\emptyset$. 
Remark first that the theorem \ref{theo:appendince unicity of operation} implies that we have an invertible natural transformation 
$$\pi_0 \circ F_{|\Theta^+}\to \pi_0 \circ H_{|\Theta^+}.$$
The theorem \ref{theo:strict stuff are pushout} and \ref{prop:suspension preserves stricte} imply that the canonical morphism 
$$H_{|\Theta^+}\to \N\circ \pi_0\circ 	H_{|\Theta^+}$$ is an equivalence. The two previous morphisms then induce a comparison:
$$F_{|\Theta^+}\to \N\circ \pi_0 \circ F_{|\Theta^+}\to H_{|\Theta^+}$$
By extension by colimits, this produces a natural transformation $\phi:F\to H$ extending $\psi$. The full sub $\infty$-groupoid of objects $C$ such that $\phi_C:F(C)\to H(C)$ is an equivalence is closed by colimits, contains globes, and so is the maximal sub 	$\infty$-groupoid.
\end{proof}
The previous corollary implies that the propositions \ref{prop:eq for cylinder} and \ref{prop:eq for Gray cone} characterize respectively the Gray cylinder, the Gray cone, and the Gray $\circ$-cone. 

\begin{cor}
\label{cor:crushing of Gray tensor is identitye}
The colimit preserving endofunctor $F:\ocat\to \ocat$, sending $[a,n]$ to the colimit of the span
$$\coprod_{k\leq n}\{k\}\leftarrow \coprod_{k\leq n}a\otimes\{k\}\to a\otimes[n]$$
is equivalent to the identity.
\end{cor}
\begin{proof}
The proposition \ref{prop:comparaison betwen otimes and suspension week case withou the cocartesian square.} implies that the restriction of $F$ to globes is equivalent to the restriction of the identity to globes. As the identity is the $0$-iterated suspension, we can apply corollary \ref{cor:characterisaiont of Gray operation}.
\end{proof}

\begin{remark}
\label{rem:somthing is the iun category}
The last corollary implies that for any $\io$-category $C$ and any globular sum $a$, the simplicial $\infty$-groupoid
$$\begin{array}{rcl}
\Delta^{op}&\to &\igrd\\
~[n]~&\mapsto &\Hom([a,n],C)
\end{array} $$
is a $\iun$-category.
\end{remark}

\begin{theorem}
\label{theo:formula between pullback of slice and tensor}
Let $C$ be an $\io$-category. The two following canonical squares are cartesian:
\[\begin{tikzcd}
	1 & {1\costar C} & 1 & {C\star 1} \\
	{\{0\}} & {[C,1]} & {\{1\}} & {[C,1]}
	\arrow[from=1-1, to=1-2]
	\arrow[from=2-1, to=2-2]
	\arrow[from=1-1, to=2-1]
	\arrow[from=1-2, to=2-2]
	\arrow[from=1-3, to=1-4]
	\arrow[from=2-3, to=2-4]
	\arrow[from=1-3, to=2-3]
	\arrow[from=1-4, to=2-4]
\end{tikzcd}\]
The five squares appearing in the following canonical diagram are both cartesian and cocartesian:
\[\begin{tikzcd}
	& {C\otimes\{0\}} & 1 \\
	{C\otimes\{1\}} & {C\otimes[1]} & {C\star 1} \\
	1 & {1\costar C} & {[C,1]}
	\arrow[from=2-3, to=3-3]
	\arrow[from=3-2, to=3-3]
	\arrow[from=2-2, to=3-2]
	\arrow[from=2-2, to=2-3]
	\arrow[from=1-2, to=1-3]
	\arrow[from=1-3, to=2-3]
	\arrow[from=1-2, to=2-2]
	\arrow[from=2-1, to=2-2]
	\arrow[from=3-1, to=3-2]
	\arrow[from=2-1, to=3-1]
\end{tikzcd}\]
\end{theorem}
\begin{proof}
The five squares of the second diagram are cocartesian by construction and by proposition \ref{prop:comparaison betwen otimes and suspension week case withou the cocartesian square.}. 

By remarking that $[\emptyset,1] \sim \{0\}\amalg \{1\}$ and $\emptyset \star 1\sim 1$, lemma \ref{lemma: first comming grom the future} induces a cartesian square:
\[\begin{tikzcd}
	{C\coprod 1} & {C\star 1} \\
	{\{0\}\coprod\{1\}} & {[C,1]}
	\arrow[from=1-1, to=1-2]
	\arrow[from=1-1, to=2-1]
	\arrow["\lrcorner"{anchor=center, pos=0.125}, draw=none, from=1-1, to=2-2]
	\arrow[from=1-2, to=2-2]
	\arrow[from=2-1, to=2-2]
\end{tikzcd}\]

This implies that the squares:
\[\begin{tikzcd}
	C & {C\star 1} & 1 & {C\star 1} \\
	{\{0\}} & {[C,1]} & {\{1\}} & {[C,1]}
	\arrow[from=1-1, to=1-2]
	\arrow[from=1-1, to=2-1]
	\arrow[from=1-2, to=2-2]
	\arrow[from=1-3, to=1-4]
	\arrow[from=1-3, to=2-3]
	\arrow[from=1-4, to=2-4]
	\arrow[from=2-1, to=2-2]
	\arrow[from=2-3, to=2-4]
\end{tikzcd}\]
are cartesian. We can similarly demonstrate that the square:
\[\begin{tikzcd}
	1 & {1\costar C} & C & {1\costar C} \\
	{\{0\}} & {[C,1]} & {\{1\}} & {[C,1]}
	\arrow[from=1-1, to=1-2]
	\arrow[from=1-1, to=2-1]
	\arrow[from=1-2, to=2-2]
	\arrow[from=1-3, to=1-4]
	\arrow[from=1-3, to=2-3]
	\arrow[from=1-4, to=2-4]
	\arrow[from=2-1, to=2-2]
	\arrow[from=2-3, to=2-4]
\end{tikzcd}\]
are cartesian. The lemma \ref{lemma: second comming grom the future} implies that the square:
\[\begin{tikzcd}
	{C\otimes[1]} & {C\star 1} \\
	{1\costar C} & {[C,1]}
	\arrow[from=1-1, to=1-2]
	\arrow[from=1-1, to=2-1]
	\arrow[from=1-2, to=2-2]
	\arrow[from=2-1, to=2-2]
\end{tikzcd}\]
is cartesian. The cartesianess of the remaining square follows by right cancellation of cartesian squares.
\end{proof}

%

\chapter{The $\iun$-category of marked $\io$-categories}	

\minitoc
\vspace{1cm}
%
%
%
%
%
%
%
%
%
%

This chapter is dedicated to the study of \textit{marked $\io$-categories}, which are pairs $(C,tC)$, where $C$ is an $\io$-category and $tC:=(tC_n)_{n>0}$ is a sequence of full sub $\infty$-groupoids of $C_n$ that include identities and are stable under composition and whiskering with (possibly unmarked) cells of lower dimensions. There are two canonical ways to mark an $\io$-category $C$. In the first, denoted by $C^\flat$, we mark as little as possible. In the second, denoted by $C^\sharp$, we mark everything.

The first section of the chapter defines these objects and establishes analogs of many results on $\io$-categories to this new framework. In particular, the \textit{marked Gray cylinder} $\uvar\otimes [1]^\sharp$ is defined. If $A$ is an $\io$-category, the underlying $\io$-category of $A^\sharp\otimes[1]^\sharp$ is $A\times [1]$, and the underlying $\io$-category of $A^\flat\otimes[1]^\sharp$ is $A\otimes[1]$. By varying the marking, and at the level of underlying $\io$-categories, we "continuously" move from the cartesian product with the directed interval to the Gray tensor product with the directed interval.

The motivation for introducing markings comes from the notion of \textit{left (and right) cartesian fibrations}. A left cartesian fibration is a morphism between marked $\io$-categories such that only the marked cells of the codomain have cartesian lifting, and the marked cells of the domain correspond exactly to such cartesian lifting. For example, a left cartesian fibration $X\to A^\sharp$ is just a "usual" left cartesian fibration where we have marked the cartesian lifts of the domain, and every morphism $C^\flat \to D^\flat$ is a left cartesian fibration.

After defining and enumerating the stability properties enjoyed by this class of left (and right) cartesian fibration, we give several characterizations of this notion in theorem \ref{theo:other characterisation of left caresian fibration}. 

The more general subclass of left cartesian fibrations that still behaves well is the class of \textit{classified left cartesian fibrations}. 
This corresponds to left cartesian fibrations $X\to A$ such that there exists a cartesian square:
\[\begin{tikzcd}
	X & Y \\
	A & {A^\sharp}
	\arrow[from=1-1, to=2-1]
	\arrow[from=2-1, to=2-2]
	\arrow[from=1-1, to=1-2]
	\arrow[from=1-2, to=2-2]
	\arrow["\lrcorner"{anchor=center, pos=0.125}, draw=none, from=1-1, to=2-2]
\end{tikzcd}\]
 where the right vertical morphism is a left cartesian fibration and $A^\sharp$ is obtained from $A$ by marking all cells. In the second section, we prove the following fundamental result:

\begin{itheorem}[\ref{theo:pullback along un marked cartesian fibration}]
Let $p:X\to A$ be a classified left cartesian fibration. Then the functor $p^*:\ocatm_{/A}\to \ocatm_{/X}$ preserves colimits.
\end{itheorem}

The third subsection is devoted to the proof of the following theorem
\begin{itheorem}[\ref{theo:left cart stable by colimit}]
Let $A$ be an $\io$-category and $F:I\to \ocatm_{/A^\sharp}$ be a diagram that is pointwise a left cartesian fibration. The colimit 
$\colim_IF$, computed in $\ocatm_{/A^\sharp}$, is a left cartesian fibration over $A^\sharp$.
\end{itheorem}

In the fourth subsection we study \textit{smooth} and \textit{proper} morphisms and we obtain the following expected result:
\begin{iprop}[\ref{prop:quillent theorem A}]
For a morphism $X\to A^\sharp$, and an object $a$ of $A$, we denote by $X_{/a}$ the marked $\io$-category fitting in the following cartesian squares. 
\[\begin{tikzcd}[cramped]
	{X_{/a}} & X \\
	{A^\sharp_{/a}} & {A^\sharp}
	\arrow[from=1-1, to=1-2]
	\arrow[from=1-1, to=2-1]
	\arrow["\lrcorner"{anchor=center, pos=0.125}, draw=none, from=1-1, to=2-2]
	\arrow[from=1-2, to=2-2]
	\arrow[from=2-1, to=2-2]
\end{tikzcd}\]
We denote by $\bot:\ocatm\to \ocat$ the functor sending a marked $\io$-category to its localization by marked cells.
\begin{enumerate}
\item Let $E$, $F$ be two elements of $\ocatm_{/A^\sharp}$ corresponding to morphisms $X\to A^\sharp$, $Y\to A^\sharp$, and
 $\phi:E\to F$ a morphism between them. We denote by $\Fb E$ and $\Fb F$ the left cartesian fibrant replacement of $E$ and $F$. 
 
The induced morphism $\Fb\phi:\Fb E\to \Fb F$ is an equivalence if and only if for any object $a$ of $A$, the induced morphism 
$$\bot X_{/a}\to \bot Y_{/a}$$ 
is an equivalence of $\io$-categories.
\item A morphism $X\to A^\sharp$ is initial if and only if for any object $a$ of $A$, $\bot X_{/a}$ is the terminal $\io$-category.
\end{enumerate}
\end{iprop}

Finally, in the last subsection, for a marked $\io$-category $I$, we define and study a (huge) $\io$-category $\uLCartc(I)$ that has classified left cartesian fibrations as objects and morphisms between classified left cartesian fibrations as arrows.

\paragraph{Cardinality hypothesis.}
We fix during this chapter two Grothendieck universes $\V\in\Wcard$, such that $\omega\in \U$. When nothing is specified, this corresponds to the implicit choice of the cardinality $\V$.
We then denote by $\Set$ the $\Wcard$-small $1$-category of $\V$-small sets, $\igrd$ the $\Wcard$-small $\iun$-category of $\V$-small $\infty$-groupoids and $\icat$ the $\Wcard$-small $\iun$-category of $\V$-small $\iun$-categories.

\section{Marked $\io$-categories}
\subsection{Definition of marked $\io$-categories}
\begin{definition}
A \notion{marked $\zo$-category} is a pair $(C,tC)$ where $C$ is an $\zo$-category and $tC:=(tC_n)_{n>0}$ is a sequence of subsets of $C_n$, containing identities, stable by composition and by whiskering with (possibly unmarked) cells of lower dimension.
A $n$-cell $a:\Db_n\to (C,tC)$ is \wcsnotion{marked}{marked $n$-cell}{for marked $\zo$-categories} if it belongs to $tC_n$.

A \textit{morphism of marked $\io$-categories} $f:(C,tC)\to (D,tD)$ is the data of a morphism on the underlying $\zo$-categories such that $f(tC_n)\subset tD_n$.
The category of marked $\zo$-categories is denoted by \wcnotation{$\zocatm$}{((a40@$\zocatm$}. 
\end{definition}

\begin{construction}
 There are two canonical ways to mark an $\zo$-category. For $C\in \zocat$, we define $C^\sharp := (C,(C_n)_{n>0})$ and $C^\flat := (C,(\Ib(C_{n-1})_{n>0}))$. The first one corresponds to the case where all cells are marked, and the second one where only the identities are marked. These two functors fit in the following adjoint triple:
\ssym{((b10@$(\uvar)^\sharp$}{for (marked) $\zo$-categories}\ssym{((b20@$(\uvar)^\flat$}{for (marked) $\zo$-categories}\ssym{((b30@$(\uvar)^\natural$}{for (marked) $\zo$-categories}
\[\begin{tikzcd}
	{(\uvar)^{\flat}: \zocat} & {\zocatm:(\uvar)^\natural} & {(\uvar)^\natural:\zocatm} & { \zocat	:(\uvar)^\sharp}
	\arrow[""{name=0, anchor=center, inner sep=0}, shift left=2, from=1-1, to=1-2]
	\arrow[""{name=1, anchor=center, inner sep=0}, shift left=2, from=1-2, to=1-1]
	\arrow[""{name=2, anchor=center, inner sep=0}, shift left=2, from=1-3, to=1-4]
	\arrow[""{name=3, anchor=center, inner sep=0}, shift left=2, from=1-4, to=1-3]
	\arrow["\dashv"{anchor=center, rotate=-90}, draw=none, from=0, to=1]
	\arrow["\dashv"{anchor=center, rotate=-90}, draw=none, from=2, to=3]
\end{tikzcd}\]
where $(\uvar)^\natural$ is the obvious forgetfull functor.
To simplify notations, for a marked $\zo$-category $C$, the marked $\zo$-categories $(C^\natural)^\flat$ and $(C^\natural)^\sharp$ will be simply denoted by $C^\flat$ and $C^\sharp$.
\end{construction}

\begin{example}
For $n$ an integer, we denote by $(\Db_n)_t$ the marked $\zo$-category whose underlying $\zo$-category is $\Db_n$ and whose only non-trivial marked cell is the top dimensional one.
\end{example}

\begin{definition}
We define the category \wcnotation{$t\Theta$}{(tTheta@$t\Theta$} as the full subcategory of $\zocatm$ whose objects are of shape $a^\flat$ for $a$ a globular sum, or \wcnotation{$(\Db_n)_t$}{(dn@$(\Db_n)_t$} for an integer $n\in\Nb$. Remark that this subcategory is dense in $\zocatm$.
\end{definition}

\begin{definition}
 We define the $\iun$-category of \wcnotion{stratified $\infty$-presheaves on $\Theta$}{stratified $\infty$-presheaf on $\Theta$}, noted by \wcnotation{$\tiPsh{\Theta}$}{(tPsh@$\tiPsh{\Theta}$}, as the full sub $\iun$-category of $\iPsh{t\Theta}$ whose objects correspond to $\infty$-presheaves $X$ such that the induced morphism
$X((\Db_n)_t)\to X(\Db_n)$
is a monomorphism. 
\end{definition}

\begin{prop}
\label{prop:marked presheaves are locally cartesian closed}
The $\iun$-category $\tiPsh{\Theta}$ is locally cartesian closed. 
\end{prop}
\begin{proof}
The $\iun$-category $\tiPsh{\Theta}$ is the localization of the $\iun$-category $\iPsh{t\Theta}$ along the set of map $\widehat{I}$ with $$I:=\{(\Db_n)_t\coprod_{\Db_n}(\Db_n)_t\to(\Db_n)_t\}_n.$$
As $\iPsh{t\Theta}$ is locally cartesian closed, we have to show that for any integer $n>0$ and any cartesian square in $\iPsh{t\Theta}$:
\[\begin{tikzcd}
	{X'} & X \\
	{(\Db_n)_t\coprod_{\Db_n}(\Db_n)_t} & {(\Db_n)_t}
	\arrow["{ }", from=1-2, to=2-2]
	\arrow[from=1-1, to=2-1]
	\arrow["{ }", from=1-1, to=1-2]
	\arrow[from=2-1, to=2-2]
	\arrow["\lrcorner"{anchor=center, pos=0.125}, draw=none, from=1-1, to=2-2]
\end{tikzcd}\]
the top horizontal morphism is in $\widehat{I}$. Using once again the locally cartesian closeness of $\iPsh{t\Theta}$, it is sufficient to show that for any integer $n>0$ and for any morphism 
$j:b\to (\Db_n)_t$ between elements of $t\Theta$, the morphism $i$ appearing in the following cartesian square of $\iPsh{t\Theta}$ is an equivalence or is in $I$:
\[\begin{tikzcd}
	B & b \\
	{(\Db_n)_t\coprod_{\Db_n}(\Db_n)_t} & {(\Db_n)_t}
	\arrow["j", from=1-2, to=2-2]
	\arrow[from=1-1, to=2-1]
	\arrow["i", from=1-1, to=1-2]
	\arrow[from=2-1, to=2-2]
\end{tikzcd}\]
Two cases have to be considered. If $j$ is the identity this is trivially true. If $j$ is any other morphism, it factors through $\Db_n\to (\Db_n)_t$. Remark now that we have two cartesian squares
\[\begin{tikzcd}
	b & b \\
	{\Db_n} & {\Db_n} \\
	{\Db_n} & {(\Db_n)_t}
	\arrow[from=1-1, to=1-2]
	\arrow[from=1-1, to=2-1]
	\arrow["\lrcorner"{anchor=center, pos=0.125}, draw=none, from=1-1, to=2-2]
	\arrow[from=1-2, to=2-2]
	\arrow[from=2-1, to=2-2]
	\arrow[from=2-1, to=3-1]
	\arrow["\lrcorner"{anchor=center, pos=0.125}, draw=none, from=2-1, to=3-2]
	\arrow[from=2-2, to=3-2]
	\arrow[from=3-1, to=3-2]
\end{tikzcd}\]
This implies that $b\sim \Db_n\times_{(\Db_{n})_t} b$ and then that $B\sim b\coprod_bb\sim b\sim b$. In particular, the morphism $i$ is the identity, which concludes the proof.
 \end{proof}

\begin{construction}
For a stratified $\infty$-presheaf $X$ on $\Theta$, we denote by $tX_n$ the $\infty$-groupoid $X((\Db_n)_t)$.
A stratified $\infty$-presheaves on $\Theta$ is then the data of a pair $(X,tX)$ such that $X\in \iPsh{\Theta}$ and $tX:=(tX_n)_{n>0}$ is a sequence of $\infty$-groupoid such that for any $n>0$, $tX_n$ is a full sub $\infty$-groupoid of $X_n$ including all units.

For $X\in \iPsh{\Theta}$, we define $X^\sharp := (X,(X_n)_{n>0})$ and $X^\flat := (X,(\Ib (X_{n-1})_{n>0})$ and we have an adjoint triple 
\ssym{((b10@$(\uvar)^\sharp$}{for (marked) $\io$-categories}\ssym{((b20@$(\uvar)^\flat$}{for (marked) $\io$-categories}\ssym{((b30@$(\uvar)^\natural$}{for (marked) $\io$-categories}
\[\begin{tikzcd}
	{(\uvar)^{\flat}: \iPsh{\Theta}} & {\tiPsh{\Theta}:(\uvar)^\natural} & {(\uvar)^\natural:\tiPsh{\Theta}} & { \Psh{\Theta}:(\uvar)^\sharp}
	\arrow[""{name=0, anchor=center, inner sep=0}, shift left=2, from=1-1, to=1-2]
	\arrow[""{name=1, anchor=center, inner sep=0}, shift left=2, from=1-2, to=1-1]
	\arrow[""{name=2, anchor=center, inner sep=0}, shift left=2, from=1-3, to=1-4]
	\arrow[""{name=3, anchor=center, inner sep=0}, shift left=2, from=1-4, to=1-3]
	\arrow["\dashv"{anchor=center, rotate=-90}, draw=none, from=0, to=1]
	\arrow["\dashv"{anchor=center, rotate=-90}, draw=none, from=2, to=3]
\end{tikzcd}\]
where $(\uvar)^\natural$ is the obvious forgetful functor.
\end{construction}

\begin{definition} We define the category \wcnotation{$t\Delta[t\Theta]$}{(tDeltaTheta@$t\Delta[t\Theta]$} as the pullback
\[\begin{tikzcd}
	{t\Delta[t\Theta]} & t\Theta \\
	{\Delta[\Theta]} & \Theta
	\arrow["{(\uvar)^\natural}", from=1-2, to=2-2]
	\arrow[from=2-1, to=2-2]
	\arrow[from=1-1, to=2-1]
	\arrow["\lrcorner"{anchor=center, pos=0.125}, draw=none, from=1-1, to=2-2]
	\arrow[from=1-1, to=1-2]
\end{tikzcd}\]
The objects of $t\Delta[t\Theta]$ then are of shape $[1]^\sharp$ or $[a,n]$ with $a\in t\Theta$ and $n\in \Delta$.
The $(\infty,1)$-category of \wcnotion{stratified presheaves on $\Delta[\Theta]$}{stratified $\infty$-presheaf on $\Delta[\Theta]$}, denoted by $\tiPsh{\Delta[\Theta]}$, is the full sub $\iun$-category of $\iPsh{t\Delta[t\Theta]}$ whose objects correspond to $\infty$-presheaves $X$ such that the induced morphism
$X((\Db_n)_t)\to X(\Db_n)$
is a monomorphism. 
\end{definition}

\begin{prop}
\label{prop:marked presheaves are locally cartesian closed2}
The $\iun$-category $\tiPsh{\Delta[\Theta]}$ is locally cartesian closed. 
\end{prop}
\begin{proof}
The proof is almost identical to the one of proposition \ref{prop:marked presheaves are locally cartesian closed}
\end{proof}

\begin{definition} For a stratified $\infty$-presheaf $X$ on $\Delta[\Theta]$, we denote by $tX_1$ the   $\infty$-groupoid $X([1]^\sharp)$, and for any $n>1$, we denote by $tX_n$ the $\infty$-groupoid $X((\Db_n)_t)$.

A stratified $\infty$-presheaf on $\Delta[\Theta]$ is then the data of a pair $(X,tX)$ such that $X\in \iPsh{\Delta[\Theta]}$
and $tX:=(tX_n)_{n>0}$ is a sequence of $\infty$-groupoid such that for any $n>0$, $tX_n$ is a full sub $\infty$-groupoid of $X_n$ including all units.

 For $X\in \iPsh{\Delta[\Theta]}$, we define once again $X^\sharp := (X,( X_n)_{n>0})$ and $X^\flat := (X, (\Ib (X_{n-1}))_{n>0})$ and we still have an adjoint triple
\[\begin{tikzcd}[column sep=0.5cm]
	{(\uvar)^\natural\iPsh{\Delta[\Theta]}} & {\tiPsh{\Delta[\Theta]}:(\uvar)^\natural} & {(\uvar)^\natural:\tiPsh{\Delta[\Theta]}} & {\iPsh{\Delta[\Theta]}:(\uvar)^\sharp}
	\arrow[""{name=0, anchor=center, inner sep=0}, shift left=2, from=1-1, to=1-2]
	\arrow[""{name=1, anchor=center, inner sep=0}, shift left=2, from=1-2, to=1-1]
	\arrow[""{name=2, anchor=center, inner sep=0}, shift left=2, from=1-3, to=1-4]
	\arrow[""{name=3, anchor=center, inner sep=0}, shift left=2, from=1-4, to=1-3]
	\arrow["\dashv"{anchor=center, rotate=-90}, draw=none, from=0, to=1]
	\arrow["\dashv"{anchor=center, rotate=-90}, draw=none, from=2, to=3]
\end{tikzcd}\]
where $(\uvar)^\natural$ is the obvious forgetfull functor.
\end{definition}

\begin{construction}
We once again have an adjunction:
\[\begin{tikzcd}
	{i_!:\tiPsh{\Delta[\Theta]}} & {\tiPsh{\Theta}:i^*}
	\arrow[shift left=2, from=1-1, to=1-2]
	\arrow[shift left=2, from=1-2, to=1-1]
\end{tikzcd}\]
induced by the canonical inclusion $t\Delta[t\Theta]\to t\Theta$.
For an integer $n$, we define the functor \sym{((b40@$(\uvar)^{\sharp_n}$}$(\uvar)^{\sharp_n}:\iPsh{\Theta}\to \tiPsh{\Theta}$ and $(\uvar)^{\sharp_n}:\iPsh{\Delta[\Theta]}\to \tiPsh{\Delta[\Theta]}$ sending a $\infty$-presheaf $X$ onto $ (X, (X^n_k)_{k>0})$ where $X^n_k:= \Ib(X_{k-1})$ if $k<n$, and $X^n_k:=X_k$ if not. We eventually set \sym{(tw@$\Wm$}\sym{(tm@$\Mm$}
$$\Wm:= \coprod_{n}(\Wseg)^{\sharp_n}\coprod (\Wsat)^\flat~~~~~\Mm:= \coprod_{n}(\Mseg)^{\sharp_n}\coprod(\Msat)^\flat$$
As $i_!(\Mm)$ is contained in $\Wm$, the previous adjunction induces a derived one:
\begin{equation}
\label{eq:derived marked adjunction theta and delta theta}
\begin{tikzcd}
	{\Lb i_!:\tiPsh{\Delta[\Theta]}_{\Mm}} & {\tiPsh{\Theta}_{\Wm}:i^*\Rb}
	\arrow[""{name=0, anchor=center, inner sep=0}, shift left=2, from=1-1, to=1-2]
	\arrow[""{name=1, anchor=center, inner sep=0}, shift left=2, from=1-2, to=1-1]
	\arrow["\dashv"{anchor=center, rotate=-90}, draw=none, from=0, to=1]
\end{tikzcd}
\end{equation}
\end{construction}

\begin{prop}
\label{prop:derived marked adjunction theta and delta theta}
The derived adjunction \eqref{eq:derived marked adjunction theta and delta theta}
is an adjoint equivalence.
\end{prop}
\begin{proof}
It is enough to show that for any element $a:t\Delta[t\Theta]$ and any $b:t\Theta$, $a\to i^*i_!a$ and $i_!i^*b\to b$ are respectively in $\widehat{\Mm}$ and $\widehat{\Wm}$. If $a$ is of shape $[b,n]^\flat$, this is a direct consequence of proposition \ref{prop:infini changing theta}, and if $a$ is $(\Db_n)_t$ the unit is the identity. We proceed similarly for $i_!i^*b\to b$.
\end{proof}

\begin{remark}
 The inclusion $t\Theta\to \zocatm$ induces an adjunction
\[\begin{tikzcd}
	{\tPsh{\Theta}} & \zocatm
	\arrow[""{name=0, anchor=center, inner sep=0}, shift left=2, from=1-1, to=1-2]
	\arrow[""{name=1, anchor=center, inner sep=0}, shift left=2, from=1-2, to=1-1]
	\arrow["\dashv"{anchor=center, rotate=-90}, draw=none, from=0, to=1]
\end{tikzcd}\]
and we can easily check that this induces an equivalence between $\zocatm$ and the sub-category of $\tPsh{\Theta}$ of $\Wm$-local objects.
Together with proposition \ref{prop:derived marked adjunction theta and delta theta}, this induces equivalences
$$\tPsh{\Theta}_{\Mm} \cong \tPsh{\Delta[\Theta]}_{\Wm}\cong \zocatm$$
\end{remark}

\begin{definition} A \notion{marked $\io$-category} is a $\Wm$-local stratified $\infty$-presheaves on $\Theta$. We then define
$$\ocatm:=\tiPsh{\Theta}_{\Wm}.$$

Proposition \ref{prop:derived marked adjunction theta and delta theta} implies that $\ocatm$ identifies itself with the full sub $\iun$-category of $\tiPsh{\Delta[\Theta]}$ of $\Mm$-local objects:
$$\ocatm \sim \tiPsh{\Delta[\Theta]}_{\Mm}.$$

Unfolding the definition, a marked $\io$-category is a pair $(C,tC)$ where $C$ is an $\io$-category and $tC:=(tC_n)_{n>0}$ is a sequence of full sub $\infty$-groupoids of $C_n$, containing identities, stable by composition and by whiskering with cells of lower dimension.
A $n$-cell $a:\Db_n\to (C,tC)$ is \wcsnotion{marked}{marked $n$-cell}{for marked $\io$-categories} if it belongs to the image of $tC_n$.
\end{definition}

\begin{construction}
There are two obvious ways to mark a $\io$-category. For $C\in \ocat$, we define $C^\sharp := (C,(C_n)_{n>0})$ and $C^\flat := (C,(\Ib(C_{n-1})_{n>0}))$. The first one corresponds to the case where all cells are marked, and the second one where only the identities are marked. These two functors fit in the following adjoint triple:
\[\begin{tikzcd}[column sep=0.7cm]
	{(\uvar)^{\flat}: \ocat} & {\ocatm:(\uvar)^\natural} & {(\uvar)^\natural:\ocatm} & { \ocat	:(\uvar)^\sharp}
	\arrow[""{name=0, anchor=center, inner sep=0}, shift left=2, from=1-1, to=1-2]
	\arrow[""{name=1, anchor=center, inner sep=0}, shift left=2, from=1-2, to=1-1]
	\arrow[""{name=2, anchor=center, inner sep=0}, shift left=2, from=1-3, to=1-4]
	\arrow[""{name=3, anchor=center, inner sep=0}, shift left=2, from=1-4, to=1-3]
	\arrow["\dashv"{anchor=center, rotate=-90}, draw=none, from=0, to=1]
	\arrow["\dashv"{anchor=center, rotate=-90}, draw=none, from=2, to=3]
\end{tikzcd}\]
where $(\uvar)^\natural$ is the obvious forgetful functor.	
To simplify notations, for a marked $\io$-category $C$, the marked $\io$-categories $(C^\natural)^\flat$ and $(C^\natural)^\sharp$ will be simply denoted by $C^\flat$ and $C^\sharp$.
\end{construction}

\begin{construction}
Following definition \ref{defi:dualities non strict case}, for any subset $S$ of $\Nb^*$, we define the duality\ssym{((b49@$(\uvar)^S$}{for marked $\io$-categories}
$$(\uvar)^S:\ocatm\to \ocatm$$
whose value on $(C,tC)$ is $(C^S,tC)$.
In particular, we have the \snotionsym{odd duality}{((b60@$(\uvar)^{op}$}{for marked $\io$-categories} $(\uvar)^{op}$, corresponding to the set of odd integer, the \snotionsym{even duality}{((b50@$(\uvar)^{co}$}{for marked $\io$-categories} $(\uvar)^{co}$, corresponding to the subset of non negative even integer, the \snotionsym{full duality}{((b80@$(\uvar)^{\circ}$}{for marked $\io$-categories} $(\uvar)^{\circ}$, corresponding to $\Nb^*$ and the \snotionsym{transposition}{((b70@$(\uvar)^t$}{for marked $\io$-categories} $(\uvar)^t$, corresponding to the singleton $\{1\}$. Eventually, we have equivalences
$$((\uvar)^{co})^{op}\sim (\uvar)^{\circ} \sim ((\uvar)^{op})^{co}.$$
\end{construction}

\begin{remark}
Given a functor $F:I\to \ocatm$, the colimit of $F$ is given by the marked $\io$-category $(C,tC)$ with 
$$C:=\colim_{I}F^\natural$$
and for any $n$, $(tC)_n$ is the image of the morphism 
$$\colim_I tF_n\to (\colim_{I}F)^\natural_n.$$
The case of the limit is easier as we have 
$$\lim_{I}F := (\lim_{I}F^\natural,(\lim_{I}(tF_n)_{n>0}).$$
In particular, if $(C,tC)$ and $(D,tD)$ are two marked $\io$-categories, we have
$$(C,tC)\times (D,tD):= (C\times D, (tC_n\times tD_n)_{n>0}).$$
\end{remark}
\begin{prop}
\label{prop:cartesian product preserves W marked version}
The cartesian product in $\ocatm$ preserves colimits in both variables.
\end{prop}
\begin{proof}
Let $F:I\to \ocatm$ be a diagram and $C$ a marked $\io$-category. The underlying $\io$-categories of $\colim_I (F\times C)$ and $(\colim_IF)\times C$ are the same as the cartesian product preserves colimits in $\ocat$. The equivalence of the two markings 
 is a direct consequence of the fact that the cartesian product in $\igrd$ preserves both colimits and the formation of image.
\end{proof}

\begin{definition}
For any $C$ in $\ocatm$, we define 
$$\uHom(C,\uvar):\ocatm\to \ocatm$$
as the right adjoint of the colimit-preserving functor $\uvar\times C:\ocatm\to \ocatm$.
\end{definition}

\begin{construction}
\label{cons:strict marked}
 We denote again $\pi_0:\tiPsh{\Theta}\to \tPsh{\Theta}$ the colimit preserving functor sending a stratified $\infty$-presheaf $X$ to the stratified presheaf $a\mapsto \pi_0(X_a)$. As this functor preserves $\Wm$, it induces an adjoint pair:
\sym{(pi@$\pi_0:\ocatm\to \zocatm$}\sym{n@$\N:\zocatm\to \ocatm$}
\[\begin{tikzcd}
	{\pi_0:\ocat} & {\zocat:\N}
	\arrow[""{name=0, anchor=center, inner sep=0}, shift left=2, from=1-1, to=1-2]
	\arrow[""{name=1, anchor=center, inner sep=0}, shift left=2, from=1-2, to=1-1]
	\arrow["\dashv"{anchor=center, rotate=-90}, draw=none, from=0, to=1]
\end{tikzcd}\]
where the right adjoint $\N$ is fully faithful.
A marked $\io$-category lying in the image of the nerve is called \wcnotion{strict}{strict marked $\io$-category}.
Remark eventually that the following square is cartesian
\[\begin{tikzcd}
	\zocatm & \ocatm \\
	\zocat & \ocat
	\arrow["\N", from=1-1, to=1-2]
	\arrow["{(\uvar)^\natural}", from=1-2, to=2-2]
	\arrow["{(\uvar)^\natural}"', from=1-1, to=2-1]
	\arrow["\N"', from=2-1, to=2-2]
\end{tikzcd}\]
A marked $\io$-category is then strict if and only if it's underlying $\io$-category is.
\end{construction}

\begin{construction}
 The \wcsnotionsym{marked suspension}{((d60@$[\uvar,1]$}{suspension}{for marked $\io$-categories} is the colimit preserving functor $$[\uvar,1]:\ocatm\to \ocatm_{\bullet,\bullet}$$ sending $a^\flat$ onto $[a,1]^\flat$ and $(\Db_n)_t$ to $([\Db_n,1])_t$. 
It then admits a right adjoint: \ssym{(hom@$\hom_{\uvar}(\uvar,\uvar)$}{for marked $\io$-categories}
$$\begin{array}{lll}
\ocatm_{\bullet,\bullet}&\to& \ocatm\\
(C,a,b)&\mapsto &\hom_C(a,b)
\end{array}
$$

With the same computation than the one of construction \ref{cons:wiskering}, we show that for a marked $\io$-category $C$, any $1$-cell $f:x\to x'$ induces for any object $y$, a morphism
$$f_!:\hom_C(x',y)\to \hom_C(x,y).$$
Conversely, a $1$-cell $g:y\to y'$ induces for any object $x$ a morphism
$$g_!:\hom_C(x,y)\to \hom_C(x,y')$$
\end{construction}

 In section \ref{section:iocategories}, we define the notion of fully faithful morphism of $\io$-categories. There is an equivalent notion for marked $\io$-categories:

\subsubsection{Fully faithful functor}
 We now give some adaptation of the result on fully faithful functors to the case of marked $\io$-categories without proofs, as they are obvious modifications to this new framework.
 
 \begin{definition}
A morphism $f:C\to D$ is \snotion{fully faithful}{for marked $\io$-categories} if for any pair of objects $x,y$, the morphism of marked $\io$-categories $\hom_C(x,y)\to \hom_D(fx,fy)$ is an equivalence, and if a $1$-cell $v$ is marked whenever $f(v)$ is.
\end{definition}

\begin{prop}
\label{prop:ff 1 marked case}
A morphism is fully faithful if and only if it has the unique right lifting property against $\emptyset\to \Db_n$ and $\Db_n\to (\Db_n)_t$ for $n>0$.
\end{prop}

\begin{prop}
\label{prop:ff 2 marked case}
Fully faithful morphisms are stable under limits.
\end{prop}

\begin{prop}
\label{prop:fully faithful plus surjective on objet marked case}
A morphism $f:C\to D$ is an equivalence if and only if it is fully faithful and surjective on objects.
\end{prop}

\subsubsection{Marked discrete Conduché functor}

\begin{definition} A morphism $f:C\to D$ between marked $\io$-categories is a \snotion{discrete Conduché functor}{for marked $\io$-categories} if for any triplet of integers $k< n\leq m$, $f$ has the unique right lifting property against $$\Ib_{m+1}:\Db_{m+1}^\flat\to \Db_{m}^\flat ~~\mbox{ and }~~\triangledown^{\sharp_n}_{k,m}:\Db_{m}^{\sharp_n}\to \Db_{m}^{\sharp_n}\coprod_{ \Db_{k}^\flat} \Db_{m}^{\sharp_n}.$$
\end{definition}

\begin{example}
If $f$ is a discrete Conduché functor between marked $\io$-categories, $f^\sharp$ is a discrete Conduché functor. Conversely, if $g$ is a discrete Conduché functor between $\io$-categories, so are $g^\sharp$, $g^\flat$ and $g^{\sharp_n}$ for any integer $n$.
\end{example}

\begin{definition}
A \notion{marked globular sum} is a marked $\io$-category whose underlying $\io$-category is a globular sum and such that for any pair of integers $k\leq n$, and any pair of $k$-composable $n$-cells $(x,y)$, $x\circ_k y$ is marked if and only if $x$ and $y$ are marked.

A morphism $i:a\to b$ between marked globular sum is \wcsnotion{globular}{globular morphism}{for marked $\zo$-categories} if the morphism $i^\natural$ is globular.
\end{definition}

\begin{remark}
The proposition \ref{prop:algebraic ortho to globular} implies that a morphism $a\to b$ between marked globular sums is a discrete Conduché functor if and only if it is globular.
\end{remark}

\begin{lemma}
\label{lemma:pullback by conduch marked preserves colimitpre}
Let $p:C\to D^\flat$ be a discrete Conduché functor between marked $\io$-categories. The canonical morphism $(C^\natural)^\flat\to C$ is an equivalence. 
\end{lemma}
\begin{proof}
Suppose given a marked $n$-cell $v:\Db_n\to C$. As the marking on $D$ is trivial, this induces a commutative square
\[\begin{tikzcd}
	{\Db_n} & { C^\natural} \\
	{\Db_{n-1}} & D
	\arrow[from=1-1, to=2-1]
	\arrow["v", from=1-1, to=1-2]
	\arrow["{p^\natural}", from=1-2, to=2-2]
	\arrow[from=2-1, to=2-2]
	\arrow["l"{description}, dashed, from=2-1, to=1-2]
\end{tikzcd}\]
that admits a lift $l$ as $p^\natural$ is a discrete Conduché functor, which concludes the proof.
\end{proof}

\begin{prop}
\label{prop:pullback by conduch marked preserves colimit}
Let $p:C\to D$ be a discrete Conduché functor between marked $\io$-categories. The pullback functor $p^*:\ocatm_{/D}\to\ocatm_{/C}$ preserves colimits.
\end{prop}
\begin{proof}
As $\tiPsh{\Theta}$ is locally cartesian closed, one has to show that for any pair of cartesian squares
\[\begin{tikzcd}
	{C''} & {C'} & C \\
	{D''} & {D'} & D
	\arrow["p", from=1-3, to=2-3]
	\arrow["i"', from=2-1, to=2-2]
	\arrow["j", from=1-1, to=1-2]
	\arrow[from=1-2, to=1-3]
	\arrow[from=2-2, to=2-3]
	\arrow[from=1-2, to=2-2]
	\arrow[from=1-1, to=2-1]
	\arrow["\lrcorner"{anchor=center, pos=0.125}, draw=none, from=1-1, to=2-2]
	\arrow["\lrcorner"{anchor=center, pos=0.125}, draw=none, from=1-2, to=2-3]
\end{tikzcd}\]
if $i$ is $\Wm$, then $j$ is in $\widehat{\Wm}$. Suppose first that $i$ is in $\Wsat^\flat$. According of the lemma \ref{lemma:pullback by conduch marked preserves colimitpre} the $\io$-categories $C'$ and $C''$ are  of shape $(E)^\flat$ and $(E')^\flat$ for $E$ and $E'$ two $\io$-categories. The proposition \ref{prop:pulback of Wsat} then implies that $i$ is in $\widehat{\W^\flat}\subset \widehat{\Wm}$. If $i$ is in $(\Wseg)^{\sharp_n}$ the proof is an easy adaptation of the one of lemma \ref{lemma:conduche preserves W}.
\end{proof}

\subsubsection{Special colimits}

 We now give some adaptation of the result on special colimits stated in the previous chapter to the case of marked $\io$-categories without proofs, as they are easy modifications.
\begin{definition}
We denote by $\iota$ the inclusion of $\ocatm$ into $\tiPsh{\Theta}$.
A functor $F:I\to \ocatm$ has a \snotion{special colimit}{for marked $\io$-categories} if the canonical morphism 
\begin{equation}
\label{eq:special colimit marked case}
\colim_{i:I}\iota F(i)\to \iota(\colim_{i:I}F(i))
\end{equation}
is an equivalence of stratified presheaves. 

Similarly, we say that a functor $\psi: I\to \Arr(\ocatm)$ has a \textit{special colimit} if the canonical morphism 
$$\colim_{i:I}\iota \psi(i)\to \iota(\colim_{i:I}\psi(i))$$
is an equivalence in the arrow $\iun$-category of $\tiPsh{\Theta}$.
\end{definition}

\begin{remark}
\label{rem on cartesian and special colim}Let $F,G:I\to \ocatm$ be two functors, and $\psi:F\to G$ a cartesian natural transformation admitting a special colimit. As $\tiPsh{\Theta}$ is cartesian closed, this implies that for any object $i$ of $I$, the canonical square
\[\begin{tikzcd}
	{F(i)} & {\colim_IF} \\
	{G(i)} & {\colim_IG}
	\arrow[from=1-1, to=1-2]
	\arrow["{\psi(i)}"', from=1-1, to=2-1]
	\arrow["\lrcorner"{anchor=center, pos=0.125}, draw=none, from=1-1, to=2-2]
	\arrow["{\colim_I\psi}", from=1-2, to=2-2]
	\arrow[from=2-1, to=2-2]
\end{tikzcd}\]
is cartesian.
\end{remark}

\begin{example}
Let $C$ be a marked $\io$-category. The canonical diagram $t\Theta_{/C}\to \ocat$ has a special colimit, given by $C$.
\end{example}
\begin{prop}
\label{prop:special colimit marked case}
Let $F,G:I\to \ocatm$ be two functors, and $\psi:F\to G$ a natural transformation. If $\psi$ is cartesian, and $G$ has a special colimit, then $\psi$ and $F$ have special colimits. 
\end{prop}

\begin{prop}
\label{prop:example of a special colimit marked case}
For any integer $n$, and element $a\in t\Theta$ and $b\in \Theta$, the equalizer diagram 
\[\begin{tikzcd}
	{\coprod_{k+l=n-1}[a,k]\vee[a\times b^\sharp,1]\vee[a,l]} & {\coprod_{k+l=n}[a,k]\vee[ b,1]^\sharp\vee[a,l]}
	\arrow[shift left=2, from=1-1, to=1-2]
	\arrow[shift right=2, from=1-1, to=1-2]
\end{tikzcd}\]
where the top diagram is induced by $[a\times b^\sharp,1]\to [a,1]\vee[b,1]^\sharp$ and to bottom one by $[a\times b^\sharp,1]\to [b,1]^\sharp\vee[a,1]$,
has a special colimit, which is $[a,n]\times [b,1]^\sharp$.
\end{prop}

\begin{prop}
\label{prop:example of a special colimit 2 marked case}
Any sequence of marked $\io$-categories has a special colimit. 
\end{prop}

\begin{prop}
\label{prop:example of a special colimit4 marked case}
Suppose given a cartesian square
\[\begin{tikzcd}
	{ B} & C \\
	{\{0\}} & {[1]^\sharp}
	\arrow[from=1-2, to=2-2]
	\arrow[from=1-1, to=1-2]
	\arrow[from=1-1, to=2-1]
	\arrow[from=2-1, to=2-2]
	\arrow["\lrcorner"{anchor=center, pos=0.125}, draw=none, from=1-1, to=2-2]
\end{tikzcd}\]
The diagram 
\[\begin{tikzcd}
	{[1]^\sharp\vee[B,1]} & {[B,1]} & {[C,1]}
	\arrow["\triangledown"', from=1-2, to=1-1]
	\arrow[from=1-2, to=1-3]
\end{tikzcd}\]
has a special colimit.
\end{prop}

\begin{prop}
\label{prop:example of a special colimit3 marked case}
Suppose given two cartesian squares
\[\begin{tikzcd}
	{ B} & C & D \\
	{\{0\}} & {[1]^\sharp} & {\{1\}}
	\arrow[from=1-2, to=2-2]
	\arrow[from=1-1, to=1-2]
	\arrow[from=1-3, to=2-3]
	\arrow[from=1-1, to=2-1]
	\arrow[from=1-3, to=1-2]
	\arrow[from=2-1, to=2-2]
	\arrow[from=2-3, to=2-2]
	\arrow["\lrcorner"{anchor=center, pos=0.125}, draw=none, from=1-1, to=2-2]
	\arrow["\lrcorner"{anchor=center, pos=0.125, rotate=-90}, draw=none, from=1-3, to=2-2]
\end{tikzcd}\]
The diagram 
\[\begin{tikzcd}
	{[1]^\sharp\vee[B,1]} & {[B,1]} & {[C,1]} & {[D,1]} & {[D,1]\vee[1]^\sharp}
	\arrow["\triangledown", from=1-4, to=1-5]
	\arrow["\triangledown"', from=1-2, to=1-1]
	\arrow[from=1-2, to=1-3]
	\arrow[from=1-4, to=1-3]
\end{tikzcd}\]
has a special colimit.
\end{prop}

\subsection{Gray tensor product of marked $\io$-categories}
\begin{construction}
We define the  \snotionsym{Gray tensor product}{((d00@$\otimes$}{for marked $\io$-categories}.
$$\uvar\otimes (\uvar)^\sharp:\ocatm\times \icat \to \ocatm$$
sending a marked $\io$-category $C$ and a $\iun$-category $K$ to the marked $\io$-category $C\otimes K^\sharp$, such that $(C\otimes K^\sharp)^\natural$ fits in the cocartesian square
\[\begin{tikzcd}
	{\coprod_{ tC}\Db_n\otimes K} & {C^\natural\otimes K} \\
	{\coprod_{ tC}\tau^i_n(\Db_n\otimes K)} & {(C\otimes K^\sharp)^\natural}
	\arrow[from=1-1, to=1-2]
	\arrow[from=1-1, to=2-1]
	\arrow[from=2-1, to=2-2]
	\arrow[from=1-2, to=2-2]
	\arrow["\lrcorner"{anchor=center, pos=0.125, rotate=180}, draw=none, from=2-2, to=1-1]
\end{tikzcd}\]
and such that $t(C\otimes K^\sharp)_n$ consists of $n$-cells lying in the image of the morphism 
$$\tau_{n-1}C\otimes K\coprod (tC)_n\otimes K_0\to (C\otimes K^\sharp)^\natural.$$
\end{construction}

\begin{prop}
\label{prop:Gray tensor product and dualitymarked}
There is an equivalence
$$(C\otimes K^\sharp)^\circ \sim C^\circ\otimes (K^\sharp)^\circ$$
natural in $C:\ocatm$ and $K:\icat$.
\end{prop}
\begin{proof}
Proposition \ref{prop:Gray tensor product and dualitymarked} provides an invertible transformation 
$$(C^\natural\otimes K)^\circ\sim (C^\natural)^\circ\otimes K.$$
The result follows from the definition of the Gray tensor product for marked $\io$-categories.
\end{proof}

\begin{prop}
\label{prop:otimes marked preserves colimits}
The functor $\uvar\otimes(\uvar)^\sharp:\ocatm\times \icat\to \ocatm$ preserves colimits. 
\end{prop}
\begin{proof}
By construction, we have two cocartesian squares:
\[\begin{tikzcd}
	{\coprod_{\colim tF }\Db_n\otimes K} & {\colim (F^\natural\otimes K)} \\
	{\coprod_{\colim tF }\tau^i_{n}(\Db_n\otimes K)} & {\colim (F\otimes K^\sharp)^\natural} \\
	{\coprod_{t(\colim F)}\Db_n\otimes K} & {(\colim F^\natural)\otimes K} \\
	{\coprod_{ t(\colim F)}\tau^i_{n}(\Db_n\otimes K)} & {((\colim F)\otimes K^\sharp)^\natural}
	\arrow[from=1-1, to=2-1]
	\arrow[from=3-1, to=4-1]
	\arrow[from=1-2, to=2-2]
	\arrow[from=3-2, to=4-2]
	\arrow[""{name=0, anchor=center, inner sep=0}, from=3-1, to=3-2]
	\arrow[""{name=1, anchor=center, inner sep=0}, from=1-1, to=1-2]
	\arrow[from=2-1, to=2-2]
	\arrow[from=4-1, to=4-2]
	\arrow["\lrcorner"{anchor=center, pos=0.125, rotate=180}, draw=none, from=2-2, to=1]
	\arrow["\lrcorner"{anchor=center, pos=0.125, rotate=180}, draw=none, from=4-2, to=0]
\end{tikzcd}\]
By the preservation of colimit by the Gray tensor product for $\io$-categories and by the functor $(\uvar)^\natural$, we have an equivalence 
$$\colim (F^\natural\otimes K)\sim (\colim F^\natural)\otimes K$$
However, the canonical morphism $\colim tF \to t(\colim F)$ is a surjection, and according to proposition \ref{prop:truncation of surjection is pushout}, the following canonical square
is cocartesian
\[\begin{tikzcd}
	{\coprod_{\colim tF }\Db_n\otimes K} & {\coprod_{t(\colim F)}\Db_n\otimes K} \\
	{\coprod_{\colim tF }\tau^i_{n}(\Db_n\otimes K)} & {\coprod_{ t(\colim F)}\tau^i_{n}(\Db_n\otimes K)}
	\arrow[from=1-1, to=2-1]
	\arrow[from=1-1, to=1-2]
	\arrow[from=2-1, to=2-2]
	\arrow[from=1-2, to=2-2]
\end{tikzcd}\]
Combined with the first two cocartesian squares, this implies that that $\colim (F\otimes K^\sharp)^\natural$ and $((\colim F)\otimes K^\sharp)^\natural$ are equivalent.

According to proposition \ref{prop:intelignet truncation is poitwise an epi} and by construction, the morphisms
 $$\colim (\tau^i_n F^\natural\otimes K)\to \tau^i_n(\colim F^\natural\otimes K)~~\mbox{ and }~~\colim (tF\otimes K_0) \to t(\colim F\otimes K_0)$$ are surjections. The marked $\io$-categories $\colim (F\otimes K^\sharp)$ and $(\colim F)\otimes K^\sharp$
 then have the same marked cells.
\end{proof}

\begin{prop}
\label{prop:associativity of Gray amput 1}
Let $C$ be a $\io$-category and $K$ a  $(\infty,1)$-categories.   The underlying $\io$-category of $C^\flat\otimes K^\sharp$ is $C\otimes K$.
\end{prop}
\begin{proof}
This is obvious.
\end{proof}

\subsubsection{Marked Gray cylinder}
\begin{definition}
The Gray tensor product for marked $\io$-category restricts to a functor
$$\uvar\otimes[1]^\sharp:\ocatm \to\ocatm$$ called the \wcsnotion{marked Gray cylinder}{Gray cylinder}{for marked $\io$-categories}\sym{((d30@$\uvar\otimes[1]^\sharp$}.
\end{definition}

\begin{prop}
\label{prop:eq for cylinder marked}
Let $C$ be a marked $\io$-category.
There is a natural identification between $[C,1]\otimes [1]^\sharp$ and the colimit of the diagram
\begin{equation}
\begin{tikzcd}
	{[1]^\sharp\vee [ C,1]} & {[C\otimes\{0\},1]} & {[C\otimes [1]^\sharp,1]} & {[C\otimes\{1\},1]} & {[C,1]\vee[1]^\sharp}
	\arrow[from=1-2, to=1-1]
	\arrow[from=1-2, to=1-3]
	\arrow[from=1-4, to=1-3]
	\arrow[from=1-4, to=1-5]
\end{tikzcd}
\end{equation}
In the previous diagram, morphisms $[C,1]\to [1]^\sharp\vee[C,1]$ and $[C,1]\to [C,1]\vee[1]^\sharp$ are the whiskerings.
\end{prop}
\begin{proof}
This is a direct consequence of proposition \ref{prop:eq for cylinder} and the definition of the marked Gray cylinder.
\end{proof}

\begin{example}
In all the following diagrams, marked cells are represented by crossed-out arrows.

The object $\Db_1^\flat\otimes[1]^\sharp$ corresponds to the diagram
\[\begin{tikzcd}
	00 & 01 \\
	10 & 11
	\arrow[from=1-1, to=2-1]
	\arrow["{/}"{marking}, from=2-1, to=2-2]
	\arrow["{/}"{marking}, from=1-1, to=1-2]
	\arrow[from=1-2, to=2-2]
	\arrow["{/}"{marking}, shorten <=4pt, shorten >=4pt, Rightarrow, from=1-2, to=2-1]
\end{tikzcd}\]
the object $(\Db_1)^\sharp\otimes[1]^\sharp$ corresponds to the diagram
\[\begin{tikzcd}
	00 & 01 \\
	10 & 11
	\arrow["{/}"{marking}, from=1-1, to=2-1]
	\arrow["{/}"{marking}, from=2-1, to=2-2]
	\arrow["{/}"{marking}, from=1-1, to=1-2]
	\arrow["{/}"{marking}, from=1-2, to=2-2]
\end{tikzcd}\]
the object $\Db_2^\flat\otimes[1]^\sharp$ corresponds to the diagram
\[\begin{tikzcd}
	00 & 01 & 00 & 01 \\
	10 & 11 & 10 & 11
	\arrow["{/}"{marking}, from=1-1, to=1-2]
	\arrow[""{name=0, anchor=center, inner sep=0}, from=1-1, to=2-1]
	\arrow["{/}"{marking}, from=2-1, to=2-2]
	\arrow[""{name=1, anchor=center, inner sep=0}, from=1-2, to=2-2]
	\arrow["{/}"{marking}, shorten <=4pt, shorten >=4pt, Rightarrow, from=1-2, to=2-1]
	\arrow[""{name=2, anchor=center, inner sep=0}, from=1-3, to=2-3]
	\arrow["{/}"{marking}, from=1-3, to=1-4]
	\arrow[""{name=3, anchor=center, inner sep=0}, from=1-4, to=2-4]
	\arrow["{/}"{marking}, shorten <=4pt, shorten >=4pt, Rightarrow, from=1-4, to=2-3]
	\arrow[""{name=4, anchor=center, inner sep=0}, curve={height=30pt}, from=1-1, to=2-1]
	\arrow["{/}"{marking}, from=2-3, to=2-4]
	\arrow[""{name=5, anchor=center, inner sep=0}, curve={height=-30pt}, from=1-4, to=2-4]
	\arrow["{ }"', shorten <=6pt, shorten >=6pt, Rightarrow, from=0, to=4]
	\arrow["{ }"', shorten <=6pt, shorten >=6pt, Rightarrow, from=5, to=3]
	\arrow[shift left=0.7, shorten <=6pt, shorten >=8pt, no head, from=1, to=2]
	\arrow[shift right=0.7, shorten <=6pt, shorten >=8pt, no head, from=1, to=2]
	\arrow["{/}"{marking}, shorten <=6pt, shorten >=6pt, from=1, to=2]
\end{tikzcd}\]
and the object $(\Db_2)_t\otimes[1]^\sharp$ corresponds to the diagram
\[\begin{tikzcd}
	00 & 01 & 00 & 01 \\
	10 & 11 & 10 & 11
	\arrow["{/}"{marking}, from=1-1, to=1-2]
	\arrow[""{name=0, anchor=center, inner sep=0}, from=1-1, to=2-1]
	\arrow["{/}"{marking}, from=2-1, to=2-2]
	\arrow[""{name=1, anchor=center, inner sep=0}, from=1-2, to=2-2]
	\arrow["{/}"{marking}, shorten <=4pt, shorten >=4pt, Rightarrow, from=1-2, to=2-1]
	\arrow[""{name=2, anchor=center, inner sep=0}, from=1-3, to=2-3]
	\arrow["{/}"{marking}, from=1-3, to=1-4]
	\arrow[""{name=3, anchor=center, inner sep=0}, from=1-4, to=2-4]
	\arrow["{/}"{marking}, shorten <=4pt, shorten >=4pt, Rightarrow, from=1-4, to=2-3]
	\arrow[""{name=4, anchor=center, inner sep=0}, curve={height=30pt}, from=1-1, to=2-1]
	\arrow["{/}"{marking}, from=2-3, to=2-4]
	\arrow[""{name=5, anchor=center, inner sep=0}, curve={height=-30pt}, from=1-4, to=2-4]
	\arrow["{=}"{marking}, draw=none, from=1, to=2]
	\arrow["{ /}"{marking}, shorten <=6pt, shorten >=6pt, Rightarrow, from=5, to=3]
	\arrow["{ /}"{marking}, shorten <=6pt, shorten >=6pt, Rightarrow, from=0, to=4]
\end{tikzcd}\]
\end{example}

\begin{prop}
\label{prop:otimes et op marked version}
There is diagram
\[\begin{tikzcd}
	{(C\otimes\{1\})^\circ} & {(C\otimes[1]^\sharp)^\circ} & {(C\otimes\{0\})^\circ} \\
	{C^\circ\otimes\{0\}} & {C^\circ\otimes[1]^\sharp} & {C^\circ\otimes\{1\}}
	\arrow["\sim", from=1-2, to=2-2]
	\arrow["\sim", from=1-3, to=2-3]
	\arrow["\sim", from=1-1, to=2-1]
	\arrow[from=1-3, to=1-2]
	\arrow[from=1-1, to=1-2]
	\arrow[from=2-1, to=2-2]
	\arrow[from=2-3, to=2-2]
\end{tikzcd}\]
natural in $C:\ocatm$,
where all vertical arrows are equivalences. 
\end{prop}
\begin{proof}
This directly follows from proposition \ref{prop:Gray tensor product and dualitymarked}.
\end{proof}

\begin{prop}
\label{prop:comparaison betwen otimes and suspension week case withou the cocartesian square.marked}
Let $C$ be an $\io$-category. There exists a natural transformation $C^\flat\otimes [1]^\sharp\to [C^\flat,1]$ whose restriction to $C^\flat\otimes\{0\}$ (resp. to $C^\flat\otimes\{1\}$) is constant on $\{0\}$ (resp. on $\{1\}$) and such that the following square is cocartesian:
\begin{equation}
\begin{tikzcd}
	{C^\flat\otimes\{0\}\coprod C^\flat\otimes\{1\}} & {C^\flat\otimes[1]^\sharp} \\
	{\{0\}\coprod \{1\}} & {[C^\flat,1]}
	\arrow[from=1-1, to=1-2]
	\arrow[from=1-1, to=2-1]
	\arrow[from=1-2, to=2-2]
	\arrow[from=2-1, to=2-2]
	\arrow["\lrcorner"{anchor=center, pos=0.125, rotate=180}, draw=none, from=2-2, to=1-1]
\end{tikzcd}
\end{equation}
\end{prop}
\begin{proof}
This is a direct consequence of proposition \ref{prop:comparaison betwen otimes and suspension week case withou the cocartesian square.} and of the definition of the marked Gray cylinder.
\end{proof}

\begin{definition}
We  denote by 
$$\begin{array}{rcl}
\ocatm&\to&\ocatm\\
C&\mapsto &C^{[1]^\sharp}
\end{array}$$
the right adjoint of the marked Gray cylinder.\sym{(c@$C^{[1]^\sharp}$}
\end{definition}

\begin{remark}The proposition \ref{prop:comparaison betwen otimes and suspension week case withou the cocartesian square.marked} provides, for any $\io$-category $C$ and any pair of objects $a$ and $b$ of $C$, a canonical cartesian square
\[\begin{tikzcd}
	{\hom_C(a,b)^\flat} & {(C^\sharp)^{[1]}} \\
	{\{a\}\times \{b\}} & {C^\sharp\times C^\sharp}
	\arrow[from=1-1, to=1-2]
	\arrow[from=1-1, to=2-1]
	\arrow["\lrcorner"{anchor=center, pos=0.125}, draw=none, from=1-1, to=2-2]
	\arrow[from=1-2, to=2-2]
	\arrow[from=2-1, to=2-2]
\end{tikzcd}\]
\end{remark}

\subsubsection{Marked Enhanced Gray tensor product}
\begin{definition}
 For any $C:\ocat$, we denote by \wcnotation{$m_{C^\sharp}$}{(mc@$m_{C^\sharp}$} the colimit preserving functor 
$\ocatm\to\ocatm$ whose value on $[a,n]^\flat$ is $[a\times C^\sharp,n]$, on $[1]^\sharp$ is $[C,1]^\sharp$, and on $[(\Db_n)_t,1]$ is $
[(\Db_n)_t\times C^\sharp,1]$.
Remark that the assignation $C\mapsto m_{C^\sharp}$ is natural in $C$ and that $m_1$ is the identity.
\end{definition}

\begin{construction}
We define the colimit preserving functor:
$$\begin{array}{ccc}
\ocatm\times\ocatm &\to& \ocatm\\
(X,Y)&\mapsto &X\ominus Y^\sharp
\end{array}
$$
called the \textit{marked} \snotionsym{enhanced Gray tensor product}{((d20@$\ominus$}{for marked $\io$-categories},
where for any marked $\io$-category $C$ and element $[b,n]$ of $\Delta[\Theta]$, $C\ominus [b,n]^\sharp$ is the following pushout: 
\begin{equation}
\label{eq: def of ominus marked}
\begin{tikzcd}
	{\coprod\limits_{k\leq n}m_{b^\sharp}(C\otimes\{k\})} & {m_{b^\sharp}(C\otimes[n]^\sharp)} \\
	{\coprod\limits_{k\leq n}m_1(C\otimes\{k\})} & {C\ominus[b,n]^\sharp}
	\arrow[from=1-1, to=2-1]
	\arrow[""{name=0, anchor=center, inner sep=0}, from=1-1, to=1-2]
	\arrow[from=1-2, to=2-2]
	\arrow[from=2-1, to=2-2]
	\arrow["\lrcorner"{anchor=center, pos=0.125, rotate=180}, draw=none, from=2-2, to=0]
\end{tikzcd}
\end{equation}
By construction, the functor $\uvar\ominus\uvar$ commutes with colimits in both variables. Moreover, for any marked $\io$-category $C$ and $\iun$-category $K$, we have a canonical identification $C\ominus K^{\sharp} \sim C\otimes K^\sharp$.
\end{construction}

\begin{prop}
There is a natural identification between $[C,1]\ominus[b,1]^\sharp$ and the colimit of the following diagram 
\begin{equation}
\label{eq:formula for the ominus marked case}
\begin{tikzcd}[column sep = 0.3cm]
	{[b,1]^\sharp\vee[C,1]} & {[C\otimes\{0\}\times b^\sharp,1]} & {[(C\otimes[1]^\sharp)\times b^\sharp),1]} & {[C\otimes\{1\}\times b^\sharp,1]} & {[C,1]\vee[b,1]^\sharp}
	\arrow[from=1-2, to=1-3]
	\arrow[from=1-4, to=1-3]
	\arrow[from=1-4, to=1-5]
	\arrow[from=1-2, to=1-1]
\end{tikzcd}
\end{equation}
\end{prop}
\begin{proof}
This directly follows from the construction of the marked enhanced Gray tensor product and from proposition \ref{prop:eq for cylinder marked}.
\end{proof}

\begin{prop}
\label{prop:ominus and opmarked}
There is an equivalence 
$$(C\ominus B^\sharp)^\circ\sim C^\circ\ominus (B^\circ)^\sharp$$
natural in $C$ and $B$.
\end{prop}
\begin{proof}
It it sufficient to construct this equivalence when $B$ is of shape $[b,n]$.
The proposition \ref{prop:Gray tensor product and dualitymarked} induces an equivalence
$$(C\otimes[n]^\sharp)^\circ\sim C^\circ\otimes ([n]^\circ)^\sharp.$$
 The results then directly follows from the definition of the operation $\ominus$ and from the equivalence $(m_{b^\sharp}(\uvar))^\circ\sim m_{(b^\sharp)^\circ}((\uvar)^\circ)$.
\end{proof}

\subsubsection{Marked Gray cone and $\circ$-cone}

\begin{construction}
\label{cons:slice and joint}
We define  the \wcsnotionsym{marked Gray cone}{((d40@$\uvar\star 1$}{Gray cone}{for marked $\io$-categories} and the \wcsnotion{marked Gray $\circ$-cone}{Gray $\circ$-cone}{for marked $\io$-categories}\index[notation]{((d50@$1\overset{co}{\star}\_$!\textit{for marked $\io$-categories}}:
$$\uvar\star 1:\ocatm\to\ocatm~~~~~~1\costar \uvar:\ocatm\to \ocatm,$$
where for any marked $\io$-category $C$, $C\star 1$ and $1\costar C$, fit in the following cocartesian square
\[\begin{tikzcd}
	{C\otimes\{1\}} & {C\otimes [1]^\sharp} & {C\otimes\{0\}} & {C\otimes [1]^\sharp} \\
	1 & {C\star 1} & 1 & {1\costar C}
	\arrow[from=1-1, to=1-2]
	\arrow[from=1-3, to=1-4]
	\arrow[from=1-4, to=2-4]
	\arrow[from=1-3, to=2-3]
	\arrow[from=2-3, to=2-4]
	\arrow[from=1-2, to=2-2]
	\arrow[from=1-1, to=2-1]
	\arrow[from=2-1, to=2-2]
	\arrow["\lrcorner"{anchor=center, pos=0.125, rotate=180}, draw=none, from=2-2, to=1-1]
	\arrow["\lrcorner"{anchor=center, pos=0.125, rotate=180}, draw=none, from=2-4, to=1-3]
\end{tikzcd}\]
\end{construction}

\begin{prop}
\label{prop: equation fullfill by cylinder and join marked version}
For any marked $\io$-category $C$, there is a natural identification between $1\costar [C,1]$ and the colimit of the diagram
\begin{equation}
\label{eq:eq for Gray cone marked version}
\begin{tikzcd}
	 {[1]^\sharp\vee [C,1]} & {[C,1]} & {[C\star 1,1]} 
	\arrow[from=1-2, to=1-1]
	\arrow[from=1-2, to=1-3]
\end{tikzcd}
\end{equation}
There is a natural identification between $[C,1] \star 1$ and the colimit of the diagram
\begin{equation}
\label{eq:eq for cojoin marked version}
\begin{tikzcd}
	 {[1\costar C,1]}& {[C,1]} & {[C,1]\vee[1]^\sharp} 
	\arrow[from=1-2, to=1-1]
	\arrow[from=1-2, to=1-3]
\end{tikzcd}
\end{equation}
\end{prop}
\begin{proof}
This directly follows from proposition \ref{prop:eq for cylinder marked} and from the definition of the Gray cone and $\circ$-cone.\end{proof}

\begin{example}
In all the following diagrams, marked cells are represented by crossed-out arrows.

The objects $\Db_1^\flat\star 1$ and $1\costar \Db_1^\flat$ correspond respectively the diagrams
\[\begin{tikzcd}
	0 &&&& 0 \\
	1 & \star && \star & 1
	\arrow[from=1-1, to=2-1]
	\arrow["{/}"{marking}, from=2-1, to=2-2]
	\arrow[""{name=0, anchor=center, inner sep=0}, "{/}"{marking}, from=1-1, to=2-2]
	\arrow[""{name=1, anchor=center, inner sep=0}, from=1-5, to=2-5]
	\arrow["{/}"{marking}, from=2-4, to=1-5]
	\arrow[""{name=2, anchor=center, inner sep=0}, "{/}"{marking}, from=2-4, to=2-5]
	\arrow["{/}"{marking}, shift right=2, shorten <=4pt, shorten >=4pt, Rightarrow, from=1, to=2]
	\arrow["{/}"{marking}, shorten <=2pt, Rightarrow, from=0, to=2-1]
\end{tikzcd}\]
the objects $(\Db_1)_t\star 1$ and $1\costar (\Db_1)_t$ correspond respectively the diagrams
\[\begin{tikzcd}
	0 &&&& 0 \\
	1 & \star && \star & 1
	\arrow["{/}"{marking}, from=1-1, to=2-1]
	\arrow["{/}"{marking}, from=2-1, to=2-2]
	\arrow["{/}"{marking}, from=1-1, to=2-2]
	\arrow["{/}"{marking}, from=1-5, to=2-5]
	\arrow["{/}"{marking}, from=2-4, to=1-5]
	\arrow["{/}"{marking}, from=2-4, to=2-5]
\end{tikzcd}\]
the objects $\Db_2^\flat\star 1$ and $1\costar \Db_2^\flat$ correspond respectively the diagrams
\[\begin{tikzcd}
	0 & {} & 0 &&& 0 & {} & 0 \\
	1 & \star & 1 & \star & \star & 1 & \star & 1
	\arrow[""{name=0, anchor=center, inner sep=0}, from=1-1, to=2-1]
	\arrow["{/}"{marking}, from=2-1, to=2-2]
	\arrow[""{name=1, anchor=center, inner sep=0}, from=1-3, to=2-3]
	\arrow[""{name=2, anchor=center, inner sep=0}, curve={height=30pt}, from=1-1, to=2-1]
	\arrow["{/}"{marking}, from=2-3, to=2-4]
	\arrow[""{name=3, anchor=center, inner sep=0}, "{/}"{marking}, from=1-1, to=2-2]
	\arrow[""{name=4, anchor=center, inner sep=0}, draw=none, from=1-2, to=2-2]
	\arrow[""{name=5, anchor=center, inner sep=0}, "{/}"{marking}, from=1-3, to=2-4]
	\arrow["{/}"{marking}, from=1-6, to=2-5]
	\arrow[""{name=6, anchor=center, inner sep=0}, from=1-6, to=2-6]
	\arrow[""{name=7, anchor=center, inner sep=0}, "{/}"{marking}, from=2-5, to=2-6]
	\arrow["{/}"{marking}, from=1-8, to=2-7]
	\arrow[""{name=8, anchor=center, inner sep=0}, from=1-8, to=2-8]
	\arrow[""{name=9, anchor=center, inner sep=0}, "{/}"{marking}, from=2-8, to=2-7]
	\arrow[""{name=10, anchor=center, inner sep=0}, curve={height=-30pt}, from=1-8, to=2-8]
	\arrow[""{name=11, anchor=center, inner sep=0}, draw=none, from=1-7, to=2-7]
	\arrow["{ }"', shorten <=6pt, shorten >=6pt, Rightarrow, from=0, to=2]
	\arrow["{/}"{marking}, shorten <=2pt, shorten >=2pt, Rightarrow, from=3, to=2-1]
	\arrow[shift left=0.7, shorten <=6pt, shorten >=8pt, no head, from=4, to=1]
	\arrow[shift right=0.7, shorten <=6pt, shorten >=8pt, no head, from=4, to=1]
	\arrow["{/}"{marking}, shorten <=6pt, shorten >=6pt, from=4, to=1]
	\arrow["{/}"{marking}, shorten <=2pt, Rightarrow, from=5, to=2-3]
	\arrow[shorten <=6pt, shorten >=6pt, Rightarrow, from=10, to=8]
	\arrow["{/}"{marking}, shift right=2, shorten <=4pt, shorten >=4pt, Rightarrow, from=8, to=9]
	\arrow["{/}"{marking}, shift right=2, shorten <=4pt, shorten >=4pt, Rightarrow, from=6, to=7]
	\arrow[shift right=0.7, shorten <=6pt, shorten >=8pt, no head, from=6, to=11]
	\arrow["{/}"{marking}, shorten <=6pt, shorten >=6pt, from=6, to=11]
	\arrow[shift left=0.7, shorten <=6pt, shorten >=8pt, no head, from=6, to=11]
\end{tikzcd}\]
and the objects $(\Db_2)_t\star 1$ and $1\costar (\Db_2)_t$ correspond respectively the diagrams
\[\begin{tikzcd}
	0 & {} & 0 &&& 0 & {} & 0 \\
	1 & \star & 1 & \star & \star & 1 & \star & 1
	\arrow[""{name=0, anchor=center, inner sep=0}, from=1-1, to=2-1]
	\arrow["{/}"{marking}, from=2-1, to=2-2]
	\arrow[""{name=1, anchor=center, inner sep=0}, from=1-3, to=2-3]
	\arrow[""{name=2, anchor=center, inner sep=0}, curve={height=30pt}, from=1-1, to=2-1]
	\arrow["{/}"{marking}, from=2-3, to=2-4]
	\arrow[""{name=3, anchor=center, inner sep=0}, "{/}"{marking}, from=1-1, to=2-2]
	\arrow[""{name=4, anchor=center, inner sep=0}, draw=none, from=1-2, to=2-2]
	\arrow[""{name=5, anchor=center, inner sep=0}, "{/}"{marking}, from=1-3, to=2-4]
	\arrow["{/}"{marking}, from=1-6, to=2-5]
	\arrow[""{name=6, anchor=center, inner sep=0}, from=1-6, to=2-6]
	\arrow[""{name=7, anchor=center, inner sep=0}, "{/}"{marking}, from=2-5, to=2-6]
	\arrow["{/}"{marking}, from=1-8, to=2-7]
	\arrow[""{name=8, anchor=center, inner sep=0}, from=1-8, to=2-8]
	\arrow[""{name=9, anchor=center, inner sep=0}, "{/}"{marking}, from=2-8, to=2-7]
	\arrow[""{name=10, anchor=center, inner sep=0}, curve={height=-30pt}, from=1-8, to=2-8]
	\arrow[""{name=11, anchor=center, inner sep=0}, draw=none, from=1-7, to=2-7]
	\arrow["{/}"{marking}, shorten <=6pt, shorten >=6pt, Rightarrow, from=0, to=2]
	\arrow["{/}"{marking}, shorten <=2pt, shorten >=2pt, Rightarrow, from=3, to=2-1]
	\arrow["{/}"{marking}, shorten <=2pt, Rightarrow, from=5, to=2-3]
	\arrow["{/}"{marking}, shorten <=6pt, shorten >=6pt, Rightarrow, from=10, to=8]
	\arrow["{/}"{marking}, shift right=2, shorten <=4pt, shorten >=4pt, Rightarrow, from=8, to=9]
	\arrow["{/}"{marking}, shift right=2, shorten <=4pt, shorten >=4pt, Rightarrow, from=6, to=7]
	\arrow["{=}"{description}, draw=none, from=4, to=1]
	\arrow["{=}"{marking}, draw=none, from=6, to=11]
\end{tikzcd}\]
\end{example}

\begin{prop}
\label{prop:star marked and dualiti}
There is a natural equivalence 
$$(C\star 1)^\circ \sim 1\costar C^\circ$$
\end{prop}
\begin{proof}
This directly follows from proposition \ref{prop:otimes et op marked version} and from the definition of the Gray cone and $\circ$-cone.
\end{proof}

\begin{definition}
\label{defi: marked slice}
We  denote by 
$$\begin{array}{ccccccc}
\ocatm_{\bullet} &\to&\ocatm&&\ocatm_{\bullet} &\to&\ocatm\\
(C,c)&\mapsto &C_{/c} & &(C,c) &\mapsto &C_{c/}
\end{array}
$$
the right adjoints of Gray cone and of the Gray $\circ$-cone, respectively called the \wcsnotionsym{slice of $C$ over $c$}{(cc@$C_{c/}$}{slice over}{for marked $\io$-categories} and the \wcsnotionsym{slice of $C$ under $c$}{(cc@$C_{/c}$}{slice under}{for marked $\io$-categories}.
\end{definition}

\begin{remark}
The proposition \ref{prop:star marked and dualiti} induces an  invertible natural transformation:
$$C_{/c}\sim (C^{\circ}_{c/})^\circ.$$
\end{remark}

\begin{remark}
\label{rem:cartesian square slicdes}
The proposition \ref{prop:comparaison betwen otimes and suspension week case withou the cocartesian square.marked} provides, for any $\io$-category $C$ and any pair of objects $a$, $b$ of $C$, two canonical cartesian squares
\[\begin{tikzcd}
	{\hom_C(a,b)^\flat} & {C^\sharp_{/b}} & {\hom_C(a,b)^\flat} & {C^\sharp_{a/}} \\
	{\{a\}} & C & {\{b\}} & C
	\arrow[from=1-1, to=1-2]
	\arrow[from=1-1, to=2-1]
	\arrow["\lrcorner"{anchor=center, pos=0.125}, draw=none, from=1-1, to=2-2]
	\arrow[from=1-2, to=2-2]
	\arrow[from=1-3, to=1-4]
	\arrow[from=1-3, to=2-3]
	\arrow["\lrcorner"{anchor=center, pos=0.125}, draw=none, from=1-3, to=2-4]
	\arrow[from=1-4, to=2-4]
	\arrow[from=2-1, to=2-2]
	\arrow[from=2-3, to=2-4]
\end{tikzcd}\]
\end{remark}

\subsubsection{Marked Gray operations and strict objects}

Recall that in construction \ref{cons:strict marked}, we defined an adjunction
\[\begin{tikzcd}
	{\pi_0:\ocatm} & {\zocatm:\N}
	\arrow[""{name=0, anchor=center, inner sep=0}, shift left=2, from=1-1, to=1-2]
	\arrow[""{name=1, anchor=center, inner sep=0}, shift left=2, from=1-2, to=1-1]
	\arrow["\dashv"{anchor=center, rotate=-90}, draw=none, from=0, to=1]
\end{tikzcd}\]
A marked $\io$-category lying in the image of the nerve is called \textit{strict}.
Remark eventually that the following square is cartesian
\[\begin{tikzcd}
	\zocatm & \ocatm \\
	\zocat & \ocat
	\arrow["\N", from=1-1, to=1-2]
	\arrow["{(\uvar)^\natural}", from=1-2, to=2-2]
	\arrow["{(\uvar)^\natural}"', from=1-1, to=2-1]
	\arrow["\N"', from=2-1, to=2-2]
\end{tikzcd}\]
A marked $\io$-category is then strict if and only if it's underlying $\io$-category is. 

\begin{prop}
\label{prop:suspension preserves stricte marked case}
If $C$ is a strict marked $\io$-category, so is $[C,1]$.
\end{prop}
\begin{proof}
This directly follows from proposition \ref{prop:suspension preserves stricte} and from the canonical equivalence $[C,1]^\natural \sim [C^\natural,1]$.
\end{proof}

\begin{lemma}
\label{lemma:a otimes 1 is strict}
Let $C$ be a marked $\io$-category.
The canonical squares
\[\begin{tikzcd}
	C & {C\otimes[1]^\sharp} & C \\
	{\{0\}} & {[1]^\sharp} & {\{1\}}
	\arrow[from=1-1, to=2-1]
	\arrow[from=2-1, to=2-2]
	\arrow[from=2-3, to=2-2]
	\arrow[from=1-2, to=2-2]
	\arrow[from=1-3, to=1-2]
	\arrow[from=1-1, to=1-2]
	\arrow[from=1-3, to=2-3]
\end{tikzcd}\]
are cartesian. 
\end{lemma}
\begin{proof}
As the morphisms $\{\epsilon\}\to [1]$ for $\epsilon\leq 1$ are discrete Conduché functors, pullback along them preserves colimits, and we can then reduce to the case where $C$ is of the shape $[1]^\sharp$ or $[a,1]$ with  $a$ is an element of $t\Theta$. 
The case $C:=[1]^\sharp$ is obvious as we have $[1]^\sharp\otimes[1]^\sharp\sim [1]^\sharp\times[1]^\sharp$ according to  proposition \ref{prop:associativity of Gray amput 1.5}. We then focus on the case $C:=[a,1]$.

We claim that for any marked $\io$-category $D$, the square
\begin{equation}
\label{eq:lemma:a otimes 1 is strict}
\begin{tikzcd}
	{\{\epsilon\}} & {[D,1]} \\
	{\{\epsilon\}} & {[1]^\sharp}
	\arrow[from=1-1, to=2-1]
	\arrow[from=2-1, to=2-2]
	\arrow[from=1-1, to=1-2]
	\arrow[from=1-2, to=2-2]
\end{tikzcd}
\end{equation}
is cartesian. To show this, as  the morphisms $\{\epsilon\}\to [1]$, are discrete Conduché functors one can reduce to the case where $D$ is in $\Theta_t$ or the empty $\io$-category, where it is obvious.

We now return to the proof of the assertion.  
Using the proposition \ref{prop:eq for cylinder marked}, the morphism $[a,1]\otimes[1]^\sharp\to [1]^\sharp$ is the horizontal colimit of the following diagram:
\[\begin{tikzcd}
	{[1]^\sharp\vee[a,1]} & {[a\otimes\{0\},1]} & {[a\otimes[1]^\sharp,1]} & {[a\otimes\{1\},1]} & {[a,1]\vee[1]^\sharp} \\
	{[1]^\sharp} & {[1]^\sharp} & {[1]^\sharp} & {[1]^\sharp} & {[1]^\sharp}
	\arrow[from=1-2, to=1-1]
	\arrow[from=1-4, to=1-5]
	\arrow[from=1-4, to=1-3]
	\arrow[from=1-2, to=1-3]
	\arrow[from=2-2, to=2-1]
	\arrow[from=2-2, to=2-3]
	\arrow[from=2-4, to=2-3]
	\arrow[from=2-4, to=2-5]
	\arrow["{s^0}", from=1-5, to=2-5]
	\arrow[from=1-4, to=2-4]
	\arrow[from=1-2, to=2-2]
	\arrow[from=1-3, to=2-3]
	\arrow["{s^1}"', from=1-1, to=2-1]
\end{tikzcd}\]

The results is then a direct application of the cartesian square \eqref{eq:lemma:a otimes 1 is strict} and of the fact  that pullbacks along morphisms $\{\epsilon\}\to [1]$ for $\epsilon\leq 1$ preserves colimits.
\end{proof}

\begin{prop}
\label{prop:tensor of glboer are strics}
For any object $a$ of $t\Theta$, the marked $\io$-categories $a\otimes [1]^\sharp$, $a\star 1$ and $1\costar a$ are strict. 
\end{prop}
\begin{proof}
We will show only the the strictness of the object $a\otimes[1]^\sharp$, as the proofs for $a\star 1$ and $1\costar a$ are similar.

Suppose first that $a$ is of shape $b^\flat$.
The first assertion of proposition \ref{prop:associativity of Gray amput 1} implies that the underlying $\io$-categories of $b^\flat\otimes [1]^\sharp$ is $b\otimes [1]$ which is strict according to theorem \ref{theo:strict stuff are pushout}.

To conclude, we have to show that for any integer $n$, $(\Db_n)_t\otimes[1]^\sharp$ is strict. We proceed by induction.
Suppose first that $a$ is $(\Db_1)_t$. The proposition \ref{prop:associativity of Gray amput 1.5} implies that 
$(\Db_1)_t \otimes[1]^\sharp$ is $([1]\times[1])^\sharp$ which is a strict object. 

Suppose now that $(\Db_n)_t\otimes[1]^\sharp$ is strict. 
The proposition \ref{prop:eq for cylinder marked} stipulates that $(\Db_{n+1})_t\otimes[1]^\sharp$ is the colimit of the diagram.
$$\begin{tikzcd}[column sep = 0.2cm]
	{[1]^\sharp\vee [(\Db_n)_t,1]} & {[(\Db_n)_t\otimes\{0\},1]} & {[(\Db_n)_t\otimes [1]^\sharp,1]} & {[(\Db_n)_t\otimes\{1\},1]} & {[(\Db_n)_t,1]\vee[1]^\sharp}
	\arrow[from=1-2, to=1-1]
	\arrow[from=1-2, to=1-3]
	\arrow[from=1-4, to=1-3]
	\arrow[from=1-4, to=1-5]
\end{tikzcd}$$
The induction hypothesis and the proposition \ref{prop:suspension preserves stricte marked case} implies that all the objects are strict. According to proposition \ref{prop:example of a special colimit3 marked case}, whose hypotheses are provided by lemma \ref{lemma:a otimes 1 is strict}, this diagram admits a special colimit. As all the morphisms are monomorphism, this implies that $(\Db_{n+1})_t\otimes[1]^\sharp$ is strict, which concludes the proof.
\end{proof}

\begin{theorem}
\label{theo:strict stuff are pushout marked}
Let $C\to D$ be a morphism between strict $\io$-categories. The marked $\io$-categories
$D^\flat\coprod_{C^\flat}C^\flat\star 1$, $1\costar C^\flat\coprod_{C^\flat}D^\flat$, and $D^\flat\coprod_{D^\flat\otimes\{0\}}D^\flat\otimes[1]^\sharp$ are strict. In particular, $C^\flat\star 1$, $1\costar C^\flat$, and $C^\flat\otimes[n]^\sharp$ for any integer $n$, are strict.
\end{theorem}
\begin{proof}
The proposition \ref{prop:associativity of Gray amput 1} and theorem \ref{theo:strict stuff are pushout marked} imply that the underlying $\io$-categories of these marked $\io$-categories are strict.
\end{proof}

\begin{prop}
\label{prop:crushing of Gray tensor is identitye marked case}
The colimit preserving endofunctor $F:\ocat\to \ocatm$, sending $[a,n]$ to the colimit of the span
$$\coprod_{k\leq n}\{k\}\leftarrow \coprod_{k\leq n}a^\flat\otimes\{k\}\to a^\flat\otimes[n]^\sharp$$
is equivalent to the functor $(\uvar)^\sharp:\ocat\to \ocatm$.
\end{prop}
\begin{proof}
This is a direct consequence of proposition \ref{prop:associativity of Gray amput 1}, of corollary \ref{cor:crushing of Gray tensor is identitye} and of the definition of the marking of the Gray tensor product for marked $\io$-categories.
\end{proof}
\begin{remark}
The last proposition implies that for any marked $\io$-category $C$ and any globular sum $a$, the simplicial $\infty$-groupoid
$$\begin{array}{rcl}
\Delta^{op}&\to &\igrd\\
~[n]~&\mapsto &\Hom([a,n]^\sharp,C)
\end{array} $$
is a $\iun$-category.
\end{remark}

\begin{theorem}
\label{theo:formula between pullback of slice and tensor marked case}
Let $C$ be an $\io$-category. The two following canonical squares are cartesian:
\[\begin{tikzcd}
	1 & {1\costar C^\flat} & 1 & {C^\flat\star 1} \\
	{\{0\}} & {[C,1]^\sharp} & {\{1\}} & {[C,1]^\sharp}
	\arrow[from=1-1, to=1-2]
	\arrow[from=2-1, to=2-2]
	\arrow[from=1-1, to=2-1]
	\arrow[from=1-2, to=2-2]
	\arrow[from=1-3, to=1-4]
	\arrow[from=2-3, to=2-4]
	\arrow[from=1-3, to=2-3]
	\arrow[from=1-4, to=2-4]
\end{tikzcd}\]
The five squares appearing in the following canonical diagram are both cartesian and cocartesian:
\[\begin{tikzcd}
	& {C^\flat\otimes\{0\}} & 1 \\
	{C^\flat\otimes\{1\}} & {C^\flat\otimes[1]^\sharp} & {C^\flat\star 1} \\
	1 & {1\costar C^\flat} & {[C,1]^\sharp}
	\arrow[from=2-3, to=3-3]
	\arrow[from=3-2, to=3-3]
	\arrow[from=2-2, to=3-2]
	\arrow[from=2-2, to=2-3]
	\arrow[from=1-2, to=1-3]
	\arrow[from=1-3, to=2-3]
	\arrow[from=1-2, to=2-2]
	\arrow[from=2-1, to=2-2]
	\arrow[from=3-1, to=3-2]
	\arrow[from=2-1, to=3-1]
\end{tikzcd}\]
\end{theorem}
\begin{proof}
This is a direct consequence of proposition \ref{prop:associativity of Gray amput 1}, of theorem \ref{theo:formula between pullback of slice and tensor} and of the definition of the marking of the Gray tensor product for marked $\io$-categories.
\end{proof}

\subsection{Gray tensor product and Cartesian product}

\begin{prop}
\label{prop:associativity of Gray amput 1.5}
Let $C$ be a $\io$-category and $K$ a  $(\infty,1)$-categories.   The canonical morphism $C^\sharp\otimes K^\sharp\to C^\sharp\times K^\sharp$ is an equivalence.
\end{prop}
\begin{proof}
As the Gray tensor product and the cartesian product preserve colimits, it is sufficient to demonstrate the equivalence when $C$ is a globe $\Db_n$ and $K$ is $[1]$. 
We proceed by induction on $n$. If $n=0$, this is trivial. For $n=1$, proposition \ref{prop:eq for cylinder marked} implies that $[1]^\flat\otimes[1]^\sharp$ is the colimit of the diagram:

\[\begin{tikzcd}
	{[1]^\sharp\vee[1]} & {[1]} & {[[1]^\sharp,1]} & {[1]} & {[1]\vee[1]^\sharp}
	\arrow["\triangledown"', from=1-2, to=1-1]
	\arrow[from=1-2, to=1-3]
	\arrow[from=1-4, to=1-3]
	\arrow["\triangledown", from=1-4, to=1-5]
\end{tikzcd}\]
By construction, we have a cocartesian square:
\[\begin{tikzcd}
	{[1]^\flat\otimes\{0\}\coprod [1]^\flat\otimes\{1\}} & {\tau^i_1([1]^\flat\otimes[1]^\sharp)} \\
	{[1]^\sharp\otimes\{0\}\coprod [1]^\sharp\otimes\{1\}} & {[1]^\sharp\otimes[1]^\sharp}
	\arrow[from=1-1, to=1-2]
	\arrow[from=1-1, to=2-1]
	\arrow[from=1-2, to=2-2]
	\arrow[from=2-1, to=2-2]
	\arrow["\lrcorner"{anchor=center, pos=0.125, rotate=180}, draw=none, from=2-2, to=1-1]
\end{tikzcd}\]
This implies that $[1]^\sharp\otimes[1]^\sharp$ is then equivalent to the colimit of the diagram:
\[\begin{tikzcd}
	{[1]^\sharp\vee[1]^\sharp} & {[1]^\sharp} & {[1]^\sharp\vee[1]^\sharp}
	\arrow["\triangledown"', from=1-2, to=1-1]
	\arrow["\triangledown", from=1-2, to=1-3]
\end{tikzcd}\]
which is equivalent to $[1]^\sharp\times[1]^\sharp$ by proposition \ref{prop:example of a special colimit}.

We now suppose that
$$\Db_n^\sharp\otimes[1]^\sharp\sim \Db_n^\sharp\times[1]^\sharp.$$
Combined with proposition \ref{prop:eq for cylinder marked}, this implies that $[\Db_n^\sharp,1]\otimes[1]^\sharp$ is the colimit of the diagram:
\[\begin{tikzcd}
	{[1]^\sharp\vee[\Db_n^\sharp,1]} & {[\Db_n^\sharp\times\{0\},1]} & {[\Db_n^\sharp\times [1]^\sharp,1]} & {[\Db_n^\sharp\times\{1\},1]} & {[\Db_n^\sharp\,1]\vee[1]^\sharp}
	\arrow[from=1-2, to=1-1]
	\arrow[from=1-2, to=1-3]
	\arrow[from=1-4, to=1-3]
	\arrow[from=1-4, to=1-5]
\end{tikzcd}\]
By construction, we have a cocartesian square:
\[\begin{tikzcd}
	{[\{0\},1]\otimes[1]^\sharp\coprod [\{1\},1]\otimes[1]^\sharp} & {[\Db_n^\sharp,1]\otimes[1]^\sharp} \\
	{[\{0\},1]^\sharp\times[1]^\sharp\coprod [\{1\},1]^\sharp\times[1]^\sharp} & {[\Db_n,1]^\sharp\otimes[1]^\sharp}
	\arrow[""{name=0, anchor=center, inner sep=0}, from=1-1, to=1-2]
	\arrow[from=1-1, to=2-1]
	\arrow[from=1-2, to=2-2]
	\arrow[from=2-1, to=2-2]
	\arrow["\lrcorner"{anchor=center, pos=0.125, rotate=180}, draw=none, from=2-2, to=0]
\end{tikzcd}\]

Eventually, the cocartesian square
\[\begin{tikzcd}
	{(\{0\}\amalg\{1\})\times[1]} & {\Db_n\times[1]} \\
	{\{0\}\amalg\{1\}} & {\Db_n}
	\arrow[from=1-1, to=1-2]
	\arrow[from=1-1, to=2-1]
	\arrow[from=1-2, to=2-2]
	\arrow[from=2-1, to=2-2]
	\arrow["\lrcorner"{anchor=center, pos=0.125, rotate=180}, draw=none, from=2-2, to=1-1]
\end{tikzcd}\]
 constructed in the proof of proposition \ref{prop:cartesian square and times} then implies that 
$[\Db_n,1]^\sharp\otimes[1]^\sharp$ is the colimit of the diagram:
\[\begin{tikzcd}
	{[1]^\sharp\vee[\Db_n^\sharp,1]} & {[\Db_n^\sharp,1]} & {[\Db_n^\sharp\,1]\vee[1]^\sharp}
	\arrow[from=1-2, to=1-1]
	\arrow[from=1-2, to=1-3]
\end{tikzcd}\]
and is then equivalent to $[\Db_n,1]^\sharp\times[1]^\sharp$ by proposition \ref{prop:example of a special colimit}.
\end{proof}
\begin{prop}
\label{prop:associativity of Gray2}
Let $D$ be an $\io$-category, $C$ a marked $\io$-category and $K$ an $\iun$-category.
The canonical morphism
$(D^\sharp\times C)\otimes K^\sharp\to D^\sharp\times (C\otimes K^\sharp)$ is an equivalence.
\end{prop}
\begin{proof}
As $\times$ and $\otimes$ preserve colimits, we can reduce to the case where $D$ is an element of $\Theta$, $C$ of $t\Theta$ and $K$ of $\Delta$, and we proceed by induction on the dimension of $D$. Remark first that if $D$ is $[0]$, the result is obvious, and if it is $(\Db_1)_t$, the result follows from  proposition \ref{prop:associativity of Gray amput 1.5}. Suppose then the result is true at the stage $n$. Using once again the fact that $\times$ and $\otimes$ preserve colimits, we can reduce to the case where
 $D^\sharp$ is $[a,1]^\sharp$, $C$ is $[b,1]$ with $b$ an element of $\Theta_t$ of dimension $n$, and $K^\sharp$ is $[1]^\sharp$.
 
 The proposition \ref{prop:example of a special colimit marked case} implies that $([a,1]^\sharp\times[b,1])\otimes[1]^\sharp$ is the colimit of the sequence: 
\begin{equation}
\label{eq:prop:associativity of Gray2}
\begin{tikzcd}
	{([a,1]^\sharp\vee [b,1])\otimes[1]^\sharp} & {[a^\sharp\times b,1]\otimes[1]^\sharp} & {([b,1]\vee [a,1]^\sharp)\otimes[1]^\sharp}
	\arrow[from=1-2, to=1-1]
	\arrow[from=1-2, to=1-3]
\end{tikzcd}
\end{equation}
The marked $\io$-category $([a,1]^\sharp\vee [b,1])\otimes[1]^\sharp$ is then the colimit of the diagram 
\[\begin{tikzcd}
	{[a,1]^\sharp\times [1]^\sharp} & {[1]^\sharp} & {[b,1]\otimes[1]^\sharp}
	\arrow[from=1-2, to=1-1]
	\arrow[from=1-2, to=1-3]
\end{tikzcd}\]
and using the propositions \ref{prop:eq for cylinder marked} and \ref{prop:example of a special colimit marked case}, $([a,1]^\sharp\vee [b,1])\otimes[1]^\sharp$ is the colimit of the diagram
\[\begin{tikzcd}
	{[1]^\sharp\vee[a,1]^\sharp\vee[b,1]} & {[a,1]^\sharp\vee[b,1]} & {[a,1]^\sharp\vee[1]^\sharp\vee[b,1]} \\
	&& {[a,1]^\sharp\vee[b\otimes\{0\},1]} \\
	&& {[a,1]^\sharp\vee[b\otimes[1]^\sharp,1]} \\
	&& {[a,1]^\sharp\vee[b\otimes\{1\},1]} \\
	&& {[a,1]^\sharp\vee[b,1]\vee[1]^\sharp}
	\arrow[from=4-3, to=5-3]
	\arrow[from=4-3, to=3-3]
	\arrow[from=2-3, to=3-3]
	\arrow[from=2-3, to=1-3]
	\arrow[from=1-2, to=1-3]
	\arrow[from=1-2, to=1-1]
\end{tikzcd}\]
Similarly, $([b,1]\vee [a,1]^\sharp)\otimes[1]^\sharp$ is the colimit of the diagram
\[\begin{tikzcd}
	{[1]^\sharp\vee[b,1]\vee[a,1]^\sharp} \\
	{[b\otimes\{0\},1]\vee[a,1]^\sharp} \\
	{[b\otimes[1]^\sharp,1]\vee[a,1]^\sharp} \\
	{[b\otimes\{1\},1]\vee[a,1]^\sharp} \\
	{[b,1]\vee[1]^\sharp\vee[a,1]^\sharp} & {[b,1]\vee[a,1]} & {[b,1]\vee[a,1]^\sharp\vee[1]^\sharp}
	\arrow[from=2-1, to=1-1]
	\arrow[from=2-1, to=3-1]
	\arrow[from=4-1, to=3-1]
	\arrow[from=4-1, to=5-1]
	\arrow[from=5-2, to=5-1]
	\arrow[from=5-2, to=5-3]
\end{tikzcd}\]
Eventually, the proposition \ref{prop:eq for cylinder marked} and the induction hypothesis imply that $[a^\sharp\times b,1]\otimes[1]^\sharp$ is the colimit of the diagram
\[\begin{tikzcd}
	{[1]^\sharp\vee[a^\sharp\times b,1]} \\
	& {[a^\sharp\times b\otimes\{0\},1]} \\
	& {[a^\sharp\times (b\otimes[1]^\sharp),1]} \\
	& {[a^\sharp\times b\otimes\{1\},1]} \\
	&& {[a^\sharp\times b,1]\vee[1]^\sharp}
	\arrow[from=2-2, to=1-1]
	\arrow[from=2-2, to=3-2]
	\arrow[from=4-2, to=3-2]
	\arrow[from=4-2, to=5-3]
\end{tikzcd}\]

All put together, $([a,1]^\sharp\times[b,1])\otimes[1]^\sharp$ is the colimit of the diagram
\[\begin{tikzcd}[column sep =0.1cm]
	{[1]^\sharp\vee[b,1]\vee[a,1]^\sharp} & {[1]^\sharp\vee[a^\sharp\times b,1]} & {[1]^\sharp\vee[a,1]^\sharp\vee[b,1]} & {[a,1]^\sharp\vee[b,1]} & {[a,1]^\sharp\vee[1]^\sharp\vee[b,1]} \\
	{[b\otimes\{0\},1]\vee[a,1]^\sharp} && {[a^\sharp\times b\otimes\{0\},1]} && {[a,1]^\sharp\vee[b\otimes\{0\},1]} \\
	{[b\otimes[1]^\sharp,1]\vee[a,1]^\sharp} && {[a^\sharp\times (b\otimes[1]^\sharp),1]} && {[a,1]^\sharp\vee[b\otimes[1]^\sharp,1]} \\
	{[b\otimes\{1\},1]\vee[a,1]^\sharp} && {[a^\sharp\times b\otimes\{1\},1]} && {[a,1]^\sharp\vee[b\otimes\{1\},1]} \\
	{[b,1]\vee[1]^\sharp\vee[a,1]^\sharp} & {[b,1]\vee[a,1]} & {[b,1]\vee[a,1]^\sharp\vee[1]^\sharp} & {[a^\sharp\times b,1]\vee[1]^\sharp} & {[a,1]^\sharp\vee[b,1]\vee[1]^\sharp}
	\arrow[from=4-5, to=5-5]
	\arrow[from=4-5, to=3-5]
	\arrow[from=2-5, to=3-5]
	\arrow[from=2-5, to=1-5]
	\arrow[from=5-4, to=5-5]
	\arrow[from=5-4, to=5-3]
	\arrow[from=2-3, to=2-5]
	\arrow[from=2-3, to=1-4]
	\arrow[from=1-4, to=1-5]
	\arrow[from=4-3, to=5-4]
	\arrow[from=4-3, to=3-3]
	\arrow[from=2-3, to=3-3]
	\arrow[from=3-3, to=3-5]
	\arrow[from=4-3, to=4-5]
	\arrow[from=1-4, to=1-3]
	\arrow[from=2-3, to=2-1]
	\arrow[from=3-3, to=3-1]
	\arrow[from=4-3, to=4-1]
	\arrow[from=4-1, to=3-1]
	\arrow[from=2-1, to=3-1]
	\arrow[from=1-2, to=1-3]
	\arrow[from=5-2, to=5-3]
	\arrow[from=5-2, to=5-1]
	\arrow[from=1-2, to=1-1]
	\arrow[from=4-1, to=5-1]
	\arrow[from=2-1, to=1-1]
	\arrow[from=2-3, to=1-2]
	\arrow[from=4-3, to=5-2]
\end{tikzcd}\]

Now, using the formula given in proposition \ref{prop:example of a special colimit marked case}, and taking the colimit line by line of the previous diagram, $([a,1]^\sharp\times[b,1])\otimes[1]^\sharp$ is the colimit of the diagram 
\[\begin{tikzcd}
	{[a,1]^\sharp\times([1]^\sharp\vee[b,1])} \\
	{[a,1]^\sharp\times[b\otimes\{0\},1]} \\
	{[a,1]^\sharp\times[b\otimes[1]^\sharp,1]} \\
	{[a,1]^\sharp\times[b\otimes\{1\},1]} \\
	{[a,1]^\sharp\times([b,1]\vee[1]^\sharp)}
	\arrow[from=2-1, to=1-1]
	\arrow[from=2-1, to=3-1]
	\arrow[from=4-1, to=3-1]
	\arrow[from=4-1, to=5-1]
\end{tikzcd}\]
Using for the last times proposition \ref{prop:eq for cylinder marked}, $([a,1]^\sharp\times[b,1])\otimes[1]^\sharp$ is equivalent to $[a,1]^\sharp\times([b,1]\otimes[1]^\sharp)$.
\end{proof}

\begin{prop}
 \label{prop:associativity of Gray amput2}
Let $D$ be a marked $\io$-category and $K,L$ two $(\infty,1)$-categories. 
There is a natural equivalence
$(D\otimes K^\sharp)\otimes L^\sharp\to D\otimes(K\times L)^\sharp$.
\end{prop}
\begin{proof}
The construction \ref{cons:almost assoc of gray} induces a natural transformation
$$\phi_{D,K,L}:((D^\natural\otimes K)\otimes L)^\flat \to (D^\natural\otimes(K\times L))^\flat\to D\otimes(K\times L)^\sharp.$$
Remark now that the canonical natural transformation $\psi_{D,K,L}:((D^\natural\otimes K)\otimes L)^\flat\to (D\otimes K^\sharp)\otimes L^\sharp$ is pointwise a colimit of morphisms of shape $\Ib:\Db_{n+1}^\flat\to \Db_n^\flat$ and $\Db_n^\flat\to (\Db_n)_t$ for $n$ an integer. Remark now that proposition \ref{prop:inteligent trucatio and a particular colimit} and the definition of marked $\io$-categories imply that the morphisms $\Ib:\Db_{n+1}^\flat\to \Db_n^\flat$ and $\Db_n^\flat\to (\Db_n)_t$ are epimorphisms in the sense of definition \ref{defi:epi abstrait}.

As a consequence, by the dual of proposition \ref{prop:mono closed by colimit}, $\psi$ is pointwise an epimorphism. The dual of proposition \ref{prop:non trivial fact about mono} then implies that $\phi$ factors through $\psi$ if and only if for any $D,K,L$, $\phi_{D,K,L}$ uniquely factors through $\psi_{D,K,L}$. 

Remark that $\psi_{D,K,L}$ obviously factors if $D,$ $K$ or $L$ is $[0]$.
As the marked Gray tensor product and the cartesian product commute with colimits, the full sub $\infty$-groupoid of elements $(D,K,L)$ of $\ocatm\times \icat\times \icat$ such that $\phi_{D,K,L}$ factors through $\psi_{D,K,L}$ and induces the desired equivalence is closed by colimits. 

 It is then sufficient to show that it includes $([1]^\sharp,[1],[1])$ and $([a,1],[1],[1])$ for $a\in t\Theta$. We can then proceed as in the proof of proposition \ref{prop:associativity of Gray2}, making these two objects explicit thanks to the equations given in proposition \ref{prop:eq for cylinder marked}. As the proof takes up a lot of space and is very similar to that of proposition \textit{prop:associativity of Gray2}, we leave it to the reader.
\end{proof}

\begin{prop}
\label{prop:associativity of ominus}
Let $C$ be a $\io$-category, $D$ a marked $\io$-category and $[b,n]$ a globular sum.
\begin{enumerate}
\item The underlying $\io$-category of $C^\flat\ominus [b,n]^\sharp$ is $C\ominus [b,n]$.
\item The canonical morphism $C^\sharp\ominus [b,n]^\sharp\to C^\sharp\times [b,n]^\sharp$ is an equivalence.
\item The canonical morphism
$(C^\sharp\times D)\ominus [b,n]^\sharp\to C^\sharp\times (D\ominus [b,n]^\sharp)$ is an equivalence.
\end{enumerate}
\end{prop}
\begin{proof}
This is a consequence of propositions \ref{prop:associativity of Gray amput 1},  \ref{prop:associativity of Gray amput 1.5}, \ref{prop:associativity of Gray2} and \ref{prop:associativity of Gray amput2} and of the construction of $\ominus$.
\end{proof}

\subsection{Marked Gray deformation retract}
We provide analogous results for section \ref{subsection:Gray deformation retract}, with proofs that are entirely similar and, therefore, omitted.

\begin{definition} A \wcnotion{left Gray deformation retract structure}{left or right Gray deformation retract structure} for a morphism $i:C\to D$ between marked $\io$-categories is the data of a \textit{retract}
 $r:D\to C$, a \textit{deformation} $\psi:D\otimes [1]^\sharp\to D$, and equivalences
$$ri\sim id_C~~~~~\psi_{|D\otimes\{0\}}\sim ir~~~~~\psi_{|D\otimes\{1\}}\sim id_D~~~~~ \psi_{|C\otimes[1]^\sharp}\sim i\cst_C
$$ 
A morphism $i:C\to D$ between marked $\io$-categories is a \wcnotion{left Gray deformation retract}{left or right Gray deformation retract} if it admits a left deformation retract structure. By abuse of notation, such data will just be denoted by $(i,r,\psi)$.

We define dually the notion of \textit{right Gray deformation retract structure} and of \textit{right Gray deformation retract} in exchanging $0$ and $1$ in the previous definition.

We define similarly the notion of \notion{left or right deformation retract} by replacing $\otimes$ by $\times$.

\end{definition}

\begin{definition}
 A \textit{left Gray deformation retract structure for a morphism $i:f\to g$} in the $\iun$-category of arrows of $\ocatm$ is the data of a \textit{retract}
 $r:g\to f$, a \textit{deformation} $\psi:g\otimes [1]^\sharp\to g$ and equivalences
$$ri\sim id_f~~~~~\psi_{|g\otimes\{0\}}\sim ir~~~~~\psi_{|g\otimes\{1\}}\sim id_D~~~~~ \psi_{|f\otimes[1]^\sharp}\sim i\cst_C
$$ 
A morphism $i:C\to D$ between two arrows of $\ocatm$ is a \textit{left Gray deformation retract} if it admits a left deformation retract structure. By abuse of notation, such data will just be denoted by $(i,r,\psi)$.

We define dually the notion of \textit{right Gray deformation retract structure} and of \textit{right Gray deformation retract} in exchanging $0$ and $1$ in the previous definition.

We define similarly the notion of \notion{left and right deformation retract} by replacing $\otimes$ by $\times$.

\end{definition}

\begin{example}
\label{example:canonical example of left deformation retract}
Let $C$ be a marked $\io$-category. The morphism $C\otimes\{0\}\to C\otimes[1]^\sharp$ is a left Gray deformation retract.
Indeed, the retract is given by $C\otimes\Ib:C\otimes[1]^\sharp\to C\otimes\{0\}$, and the deformation is induced by
$$(C\otimes[1]^\sharp)\otimes[1]^\sharp\sim C\otimes([1]\times [1])^\sharp\xrightarrow{C\otimes\psi^\sharp} C\otimes[1]^\sharp$$
where the first equivalence is the one of proposition \ref{prop:associativity of Gray amput2}, and $\psi:[1]\times[1]\to [1]$ is the unique morphism sending $(\epsilon,\epsilon')$ to $\epsilon\wedge \epsilon'$.

Similarly, the morphism $C\otimes\{1\}\to C\otimes[1]^\sharp$ is a right deformation retract.
\end{example}

\vspace{1cm} Left and right Gray retracts enjoy many stability properties: 
\begin{prop}
\label{prop:left Gray deformation retract stable under pushout}
Let $(i_a,r_a,\psi_a)$ be a natural family of left (resp. right) Gray deformation retract structures indexed by an $(\infty,1)$-category $A$.
The triple $(\colim_{A}i_a,\colim_{A}r_a,\colim_{A}\psi_a)$ is a left (resp. right) $k$-Gray deformation retract structure.
\end{prop}

\begin{prop}
\label{prop:stability under pullback}
Suppose given a diagram
\[\begin{tikzcd}
	X & Y & Z \\
	X & {Y'} & {Z'}
	\arrow[from=1-1, to=2-1]
	\arrow[from=1-2, to=2-2]
	\arrow[from=1-3, to=2-3]
	\arrow["p", from=1-1, to=1-2]
	\arrow["q"', from=1-3, to=1-2]
	\arrow["{p'}"', from=2-1, to=2-2]
	\arrow["{q'}", from=2-3, to=2-2]
\end{tikzcd}\]
such that $p\to p'$ and $q\to q'$ are left (resp. right) Gray deformation retract. The induced square $q^*p\to (q')^*p'$ is a left (resp. right) $k$-Gray deformation retract.
\end{prop}

\begin{prop}
\label{prop:stability by composition }
If $p\to p'$ and $p'\to p''$ are two left (resp. right) Gray deformation retracts, so is $p\to p''$.
\end{prop}

\begin{prop}
\label{prop:Gray deformation retract and passage to hom}
Let $(i:C\to D,r,\psi)$ be a left (resp. right) Gray deformation structure. For any $x: C$ and $y:D$ (resp. $x: D$ and $y:C$), the morphism
$$\begin{array}{cc}
&\hom_C(x,ry)\xrightarrow{i} \hom_D(ix,iry)\xrightarrow{{\psi_y}_!} \hom_D(ix,y)\\
(resp. &\hom_C(rx,y)\xrightarrow{i} \hom_D(irx,iy)\xrightarrow{{\psi_x}_!} \hom_D(x,iy))
\end{array}
$$
is a right (resp. left) Gray deformation retract, whose retract is given by 
$$\begin{array}{cc}
&\hom_D(ix,y)\xrightarrow{r}\hom_C(x,ry)\\
(resp. &\hom_D(x,iy)\xrightarrow{r}\hom_C(rx,y))
\end{array}$$

If $(i:C\to D,r,\psi)$ is a left (resp. right) deformation structure, for any $x: C$ and $y:D$ (resp. $x: D$ and $y:C$), the two morphisms above are inverses one of each other.
\end{prop}

\begin{prop}
\label{prop:Gray deformation retract and passage to hom v2}
For any left (resp. right) Gray deformation retracts between $p$ and $p'$:
\[\begin{tikzcd}
	C & D \\
	{C'} & {D'}
	\arrow["p"', from=1-1, to=2-1]
	\arrow["i", from=1-1, to=1-2]
	\arrow["{p'}", from=1-2, to=2-2]
	\arrow["{i'}"', from=2-1, to=2-2]
\end{tikzcd}\]
and for any pair of objects $x: C$ and $y:D$ (resp. $x: D$ and $y:C$), the outer square of the following diagram
\[\begin{tikzcd}
	{\hom_{C}(x,ry)} & {\hom_{D}(ix,iry)} & {\hom_{D}(ix,y)} \\
	{\hom_{C'}(px,pr'y)} & {\hom_{D'}(p'i'x,p'i'r'y)} & {\hom_{D'}(p'i'x,p'y)}
	\arrow["{i'}"', from=2-1, to=2-2]
	\arrow["{{\psi'_{p'y}}_!}"', from=2-2, to=2-3]
	\arrow[from=1-1, to=2-1]
	\arrow[from=1-3, to=2-3]
	\arrow["{{\psi_y}_!}", from=1-2, to=1-3]
	\arrow["i", from=1-1, to=1-2]
	\arrow[from=1-2, to=2-2]
\end{tikzcd}\]
(resp.
\[\begin{tikzcd}
	{\hom_{C}(rx,y)} & {\hom_{D}(irx,iy)} & {\hom_{D}(x,iy)} \\
	{\hom_{C'}(pr'x,py)} & {\hom_{D'}(p'i'r'x,p'i'y)} & {\hom_{D'}(p'x,p'i'y)\big)}
	\arrow["{i'}"', from=2-1, to=2-2]
	\arrow["{{\psi'_{p'x}}_!}"', from=2-2, to=2-3]
	\arrow[from=1-1, to=2-1]
	\arrow[from=1-3, to=2-3]
	\arrow["{{\psi_x}_!}", from=1-2, to=1-3]
	\arrow["i", from=1-1, to=1-2]
	\arrow[from=1-2, to=2-2]
\end{tikzcd}\]
is a left (resp. right) Gray deformation retract, whose retract is given by
\[\begin{tikzcd}
	{\hom_{D}(ix,y)} & {\hom_{C}(x,ry)} & {(resp.\hom_{D}(x,iy)} & {\hom_{C}(rx,y)} \\
	{\hom_{D'}(p'i'x,p'y)\big)} & {\hom_{C'}(px,pr'y)} & {\hom_{D'}(p'x,p'i'y)} & {\hom_{C'}(pr'x,py)\big)}
	\arrow[from=1-3, to=2-3]
	\arrow["r", from=1-3, to=1-4]
	\arrow["{r'}"', from=2-3, to=2-4]
	\arrow[from=1-4, to=2-4]
	\arrow["{r'}"', from=2-1, to=2-2]
	\arrow["r", from=1-1, to=1-2]
	\arrow[from=1-2, to=2-2]
	\arrow[from=1-1, to=2-1]
\end{tikzcd}\]
If $p\to p'$ is a left (resp. right) deformation structure, for any $x: C$ and $y:D$ (resp. $x: D$ and $y:C$), the two morphisms above are inverses one of each other.
\end{prop}

\begin{prop}
\label{prop:suspension of left Gray deformation retract}
If $i$ is a left Gray deformation retract, $[i,1]$ is a right Gray deformation retract. Conversely, if $i$ is a right Gray deformation retract, $[i,1]$ is a left Gray deformation retract morphism.
\end{prop}

\begin{prop}
\label{prop:when glob inclusion are left Gray deformation}
Let $a$ be a globular sum of dimension $(n+1)$. We denote by $s_n(a)$ and $t_n(a)$ the globular sum defined in \ref{defi:definition of source et but}. If $n$ is even, $s_n(a)^\flat\to a^{\sharp_n}$ is a left Gray deformation retract, and $t_n(a)^\flat\to a^{\sharp_n}$ is a right Gray deformation retract. Dually, if $n$ is odd, $t_n(a)^\flat\to a^{\sharp_n}$ is a left Gray deformation retract, and $s_n(a)^\flat\to a^{\sharp_n}$ is a right Gray deformation retract.
\end{prop}

\begin{prop}
\label{prop:exemple of right deformation retract}
Let $i:C\to D$ be a left Gray deformation retract and $A$ a marked $\io$-category.
The morphism $A\times i$ is a left Gray deformation retract. 
\end{prop}
\begin{proof}
Let $r$ and $\psi$ be retracts and deformation of $i$.
We define $\psi_A$ as the composite
$$(A\times D)\otimes[1]^\sharp\to A\times (D\otimes[1]^\sharp)\xrightarrow{A\times \psi} A\times D$$
Remark that the triple $(A\times i,A\times r,\psi_A)$ is a left Gray deformation retract structure.
\end{proof}

\section{Cartesian fibrations}
\subsection{Initial and final morphisms}

\begin{definition} We denote by \wcnotation{$\I$}{(i@$\I$} the set of morphisms of shape $X\otimes \{0\}\to X\otimes [1]^\sharp$ for $X$ being either $\Db_n^\flat$ or $(\Db_n)_t$. A morphism is \wcnotion{initial}{initial morphism} if it is in $\widehat{\I}$. Conversely, we denote by \wcnotation{$\F$}{(f@$\F$} the set of morphisms of shape $X\otimes \{1\}\to X\otimes [1]^\sharp$ for $X$ being either $\Db_n^\flat$ or $(\Db_n)_t$. A morphism is \wcnotion{final}{final morphism} if it is in $\widehat{\F}$.

Initial and final morphisms are stable under colimits, retract, composition and  left cancellation according to lemma \ref{lemma:closed under colimit imply saturated_modified}.
\end{definition}

\begin{remark}
The proposition \ref{prop:otimes et op marked version} implies that the full duality $(\uvar)^\circ$ sends final (resp. initial) morphisms to initial (resp. final) morphisms.
\end{remark}

\begin{example}
\label{exe:the easiest example of initial and finla morphism}
By stability of initial and final morphisms by colimits, for any marked $\io$-category $C$, $C\otimes\{0\}\to C\otimes[1]^\sharp$ is initial, and $C\otimes\{1\}\to C\otimes[1]^\sharp$ is final.
\end{example}

\begin{prop}
\label{prop:left Gray deformation retract are initial}
Left Gray deformation retracts (resp. left deformation retract) are initial and right Gray deformation retracts (resp. right deformation retract) are final. 
\end{prop}
\begin{proof} 
Let $i:C\to D$ be a left Gray deformation retract. The diagram
\[\begin{tikzcd}
	C & {D\otimes\{0\}} & C \\
	{D\otimes\{1\}} & {D\otimes [1]^\sharp} & D
	\arrow["i"', from=1-1, to=2-1]
	\arrow[from=2-1, to=2-2]
	\arrow[from=1-2, to=2-2]
	\arrow["\psi"', from=2-2, to=2-3]
	\arrow["i", from=1-1, to=1-2]
	\arrow["r", from=1-2, to=1-3]
	\arrow["i", from=1-3, to=2-3]
\end{tikzcd}\]
expresses $i$ as a retract of $D\otimes \{0\}\to D\otimes [1]^\sharp$, which is an initial morphism according to example \ref{exe:the easiest example of initial and finla morphism}. The morphism  $i$ is then initial. 

As left deformation retracts are left Gray deformation retracts, they are initial.
The case of right (Gray) deformation retracts follows by duality.
\end{proof}

\begin{cor}
\label{cor:when glob inclusion are final and initial}
Let $a$ be a globular sum of dimension $(n+1)$. We denote by $s_n(a)$ and $t_n(a)$ the globular sum defined in \ref{defi:definition of source et but}. If $n$ is even, $s_n(a)^\flat\to a^{\sharp_n}$ is initial, and $t_n(a)^\flat\to a^{\sharp_n}$ is final. Dually, if $n$ is odd, $t_n(a)^\flat\to a^{\sharp_n}$ is initial, and $s_n(a)^\flat\to a^{\sharp_n}$ is final
\end{cor}
\begin{proof}
This is a direct consequence of propositions \ref{prop:when glob inclusion are left Gray deformation} and \ref{prop:left Gray deformation retract are initial}.
\end{proof}

\begin{definition}
A \wcnotion{marked trivialization}{marked trivializations} is a morphism in the smallest cocomplete class of morphisms that includes $\Ib_n:(\Db_{n+1})_t\to \Db_n^\flat$ for any $n$.
\end{definition}

\begin{prop}
\label{prop:trivialization are initial}
Marked trivializations are  both initial and final.
\end{prop}
\begin{proof}
As final and initial morphisms are closed under colimits, it is sufficient to demonstrate that for any $n$, the morphism $\Ib_n:(\Db_{n+1})_t\to \Db_n^\flat$ is both initial and final.
According to lemma \ref{cor:when glob inclusion are final and initial} there exists $\alpha\in\{-,+\}$ such that 
 $i_{n}^\alpha:(\Db_n)^\flat\to (\Db_{n+1})_t$ is initial.
 As $\Ib_n$ is a retraction of this morphism, and as initial morphisms are closed under left cancellation, $\Ib_n$ is initial. The second case follows by duality.
\end{proof}

\begin{prop}
\label{prop:cotimes 1 to c is a trivialization}
Let $C$ be a marked $\io$-category.
The morphism $C\otimes[1]^\sharp\to C$ is a marked trivialization. In particular, this morphism is both initial and final. 
\end{prop}
\begin{proof}
As marked trivializations are stable under colimits, it is sufficient to demonstrate the result for $C$ being either $\Db_n^\flat$ or $(\Db_{n+1})_t$ for $n$ an integer. We will then proceed by induction. Suppose first that $C$ is $\Db_0^\flat$ or $(\Db_1)_t$. The first case is trivial, for the second one, remark that $(\Db_1)_t\otimes[1]^\sharp\sim [1]^\sharp\times[1]^\sharp\to [1]^\sharp$ is the horizontal colimit of the diagram
\[\begin{tikzcd}
	{[2]^\sharp} & {[1]^\sharp} & {[2]^\sharp} \\
	{[1]^\sharp} & {[1]^\sharp} & {[1]^\sharp}
	\arrow["{s^0}"', from=1-1, to=2-1]
	\arrow[from=1-2, to=2-2]
	\arrow["{s^1}", from=1-3, to=2-3]
	\arrow[from=1-2, to=1-1]
	\arrow[from=1-2, to=1-3]
	\arrow[from=2-2, to=2-1]
	\arrow[from=2-2, to=2-3]
\end{tikzcd}\]
and is then  a marked trivialization. Suppose now the result is true at the stage $(n-1)$. Let $C$ be $\Db_n^\flat$ (resp.$(\Db_{n+1})_t$). We set $D:=\Db_{n-1}^\flat$ (resp. $D:=(\Db_{n})_t$). We then have $C\sim [D,1]$. The proposition \ref{prop:eq for cylinder marked} implies that $C\otimes[1]^\sharp\to C$ is the horizontal colimit of the diagram:
\[\begin{tikzcd}
	{[1]^\sharp\vee[D,1]} & {[D\otimes\{0\},1]} & {[D\otimes[1]^\sharp,1]} & {[D\otimes\{1\},1]} & {[D,1]\vee[1]^\sharp} \\
	{[D,1]} & {[D,1]} & {[D,1]} & {[D,1]} & {[D,1]}
	\arrow[from=2-2, to=2-1]
	\arrow[from=2-2, to=2-3]
	\arrow[from=2-4, to=2-3]
	\arrow[from=2-4, to=2-5]
	\arrow[from=1-1, to=2-1]
	\arrow[from=1-3, to=2-3]
	\arrow[from=1-4, to=2-4]
	\arrow[from=1-5, to=2-5]
	\arrow[from=1-2, to=1-1]
	\arrow[from=1-2, to=1-3]
	\arrow[from=1-4, to=1-3]
	\arrow[from=1-4, to=1-5]
	\arrow[from=1-2, to=2-2]
\end{tikzcd}\]
The leftest and rightest morphisms obviously are   marked trivializations
As marked trivializations are stable by suspension, the induction hypothesis implies that the middle vertical morphisms of the previous diagram are in $K$, which concludes the proof.
\end{proof}
\begin{prop}
\label{prop:cotimes 1 to ctimes 1 is a trivialization}
Let $C$ be a marked $\io$-category.
The morphism $C\otimes[1]^\sharp\to C\times [1]^\sharp$ is a marked trivialization. In particular, this morphism is both initial and final. 
\end{prop}
\begin{proof}
As marked trivializations are stable by colimit, it is sufficient to demonstrate the result for $C$ being either $\Db_n^\flat$ or $(\Db_{n+1})_t$ for $n$ an integer. If $C$ is either $(\Db_0)^\flat$ or $(\Db_1)_t$ the considered morphism is the identity. We then suppose that $n>0$. Let $C$ be $\Db_n^\flat$ (resp.$(\Db_{n+1})_t$). We set $D:=\Db_{n-1}^\flat$ (resp. $D:=(\Db_{n})_t$). We then have $C\sim [D,1]$. The propositions \ref{prop:eq for cylinder marked} and  \ref{prop:example of a special colimit marked case} imply that $C\otimes[1]^\sharp\to C\times[1]^\sharp$ is the horizontal colimit of the diagram:
\[\begin{tikzcd}
	{[1]^\sharp\vee[D,1]} & {[D\otimes\{0\},1]} & {[D\otimes[1]^\sharp,1]} & {[D\otimes\{1\},1]} & {[D,1]\vee[1]^\sharp} \\
	{[1]^\sharp\vee[D,1]} & {[D,1]} & {[D,1]} & {[D,1]} & {[D,1]\vee[1]^\sharp}
	\arrow[from=2-2, to=2-1]
	\arrow[from=2-2, to=2-3]
	\arrow[from=2-4, to=2-3]
	\arrow[from=2-4, to=2-5]
	\arrow[from=1-1, to=2-1]
	\arrow[from=1-3, to=2-3]
	\arrow[from=1-4, to=2-4]
	\arrow[from=1-5, to=2-5]
	\arrow[from=1-2, to=1-1]
	\arrow[from=1-2, to=1-3]
	\arrow[from=1-4, to=1-3]
	\arrow[from=1-4, to=1-5]
	\arrow[from=1-2, to=2-2]
\end{tikzcd}\]
The proposition \ref{prop:cotimes 1 to c is a trivialization} then states that the middle vertical morphisms of the previous diagram are marked trivializations, which concludes the proof.
\end{proof}

\begin{prop}
\label{prop:suspension of initial}
If $i$ is an initial morphism, $[i,1]$ is a final morphism. Conversely, if $i$ is a final morphism, $[i,1]$ is an initial morphism.
\end{prop}
\begin{proof}
As the suspension preserves colimits, we can restrict to the case where $i$ is of shape $C\otimes\{0\}\to C\otimes[1]^\sharp$, and this is then a consequence of propositions \ref{prop:suspension of left Gray deformation retract} and \ref{prop:left Gray deformation retract are initial}.
\end{proof}

\begin{prop}
\label{prop:initial stable under product}
For any marked $\io$-category $K$, the functor $K\times\uvar:\ocatm\to \ocatm$ preserves initial and final morphisms. 
\end{prop}
\begin{proof}
The functor $K\times\uvar$ preserves colimits and this is then enough to show that 
it preserves left and right Gray deformation retracts, which is a consequence of proposition \ref{prop:exemple of right deformation retract}.
\end{proof}

\subsection{Left and right cartesian fibrations}
\label{subsection Left and right cartesian fibration}

\begin{definition}
A \wcnotion{left cartesian fibration}{left or right cartesian fibration} is a morphism $f:C\to D$ between marked $\io$-categories having the unique right lifting property against initial morphisms.
A \textit{right cartesian fibration} is a morphism $f:C\to D$ between marked $\io$-categories having the unique right lifting property against final morphisms.
\end{definition}
\begin{remark}
Left and right cartesian fibrations are stable under limits, retract, composition and  right cancellation according to the result of section \ref{section:Factorization system}.  
\end{remark}

\begin{remark}
The proposition \ref{prop:otimes et op marked version} implies that the full duality $(\uvar)^\circ$ sends left (resp. right) cartesian fibrations to right (resp. left) cartesian fibrations.
\end{remark}

\begin{construction}
\label{cons:of fb for fibration}
The construction \ref{cons:small object argument} produces a unique factorization system between initial (resp final) morphisms and left (resp. right) cartesian fibrations. If $f:A\to B$ is any morphism, we will denote by $\Fb f: A'\to B$ the left cartesian fibration obtained via this factorization system. This morphism will be called the \notion{left Cartesian replacement} of $f$.
\end{construction}

\begin{prop}
\label{prop:left fib over flat}
If $f:C\to D^\flat$ is a left cartesian fibration, then the canonical morphism $(C^\natural)^\flat\to C$ is an equivalence. Conversely, any morphism $C^\flat \to D^\flat$ is a left cartesian fibration.
\end{prop}
\begin{proof}
The first assertion is a consequence of the fact that marked trivializations are initial. The second assertion is a direct consequence of proposition \ref{prop:cotimes 1 to c is a trivialization}.
\end{proof}

\begin{prop}
\label{prop:cartesian fibration between arrow}
Let $p:X\to C$ be a morphism, and $x,y$ two objects of $X$. Then, if $p$ is a right (resp. left) cartesian fibration, the induced morphism $p:\hom_X(x,y)\to \hom_C(x,y)$ is a left (resp. right) cartesian fibration.
\end{prop}
\begin{proof}
This is a direct consequence of proposition \ref{prop:suspension of initial}.
\end{proof}

\begin{prop}
\label{prop:left Gray transfomration stable under pullback along cartesian fibration}
Consider a cocartesian square
\[\begin{tikzcd}
	{X''} & {X'} & X \\
	{Y''} & {Y'} & Y
	\arrow["p"', from=1-3, to=2-3]
	\arrow["{p''}"', from=1-1, to=2-1]
	\arrow["{p'}"', from=1-2, to=2-2]
	\arrow["j", from=1-1, to=1-2]
	\arrow[from=1-2, to=1-3]
	\arrow["i"', from=2-1, to=2-2]
	\arrow[from=2-2, to=2-3]
	\arrow["\lrcorner"{anchor=center, pos=0.125}, draw=none, from=1-1, to=2-2]
	\arrow["\lrcorner"{anchor=center, pos=0.125}, draw=none, from=1-2, to=2-3]
\end{tikzcd}\]
If $p$ is a left (resp. right) cartesian fibration and $i$ is a right (resp. left) Gray deformation retract, then $p''\to p'$ is a right (resp. left) Gray deformation retract. Moreover, this left (resp. right) Gray deformation retract structure is functorial in $p$.

Similarly, if $p$ is a left (resp. right) cartesian fibration and $i$ is a right (resp. left) deformation retract, then $p''\to p'$ is a right (resp. left) deformation retract.  This left (resp. right) deformation retract structure is functorial in $p$. 
\end{prop}
\begin{proof}
We suppose that $p$ is a right cartesian fibration. By stability under pullbacks, so is $p'$.
Let $(i:C\to D,r,\phi)$ be a left Gray deformation retract structure.
We define the morphism $\psi$ as the lift of the following commutative square:
\[\begin{tikzcd}
	{X''\otimes[1]^\sharp\cup X'\otimes\{0\}} && {X'} \\
	{X'\otimes[1]^\sharp} & {Y''\otimes[1]^\sharp} & {Y'}
	\arrow[from=1-1, to=2-1]
	\arrow["{p'}", from=1-3, to=2-3]
	\arrow["{(X''\otimes\Ib)\cup id}", from=1-1, to=1-3]
	\arrow[from=2-1, to=2-2]
	\arrow[from=2-2, to=2-3]
	\arrow["\psi"{description}, dotted, from=2-1, to=1-3]
\end{tikzcd}\]
Remark that the restriction of $\psi$ to $X'\otimes\{1\}$ factors through $X''$ and then defines a retract $s:Y\to X$ of $j$. This provides a right Gray deformation structure for $p\to p''$. We proceed similarly for the dual case.

The functoriality of the Gray deformation retract structure comes from the fact that only functorial operations were used. Indeed, pullbacks, pushouts and the Gray tensor product are functorial. The formation of the lift $\psi$ is also functorial according to proposition \ref{prop:fonctorialite des relevement}.

To verify the second claim, one may utilize the same proof, exchanging $\otimes$ with $\times$.
\end{proof}

\begin{cor}
\label{cor:morphism between is an equivalence when equivalence on fiber}
Let $p:X\to B^\sharp$ and $q:Y\to B^\sharp$ be two left cartesian fibrations and $\phi:p\to q$ a morphism over $ B^\sharp$. The morphism $\phi$ is an equivalence if and only if, for any object $b$ of $B$, the induced morphism $\{b\}^*\phi :\{b\}^*X\to \{b\}^*Y$ is an equivalence.
\end{cor}
\begin{proof}
As $\tiPsh{\Theta}$ is locally cartesian closed, pullback commutes with special colimits, and as every $\io$-category is the special colimit of its $k$-truncation for $k\in \Nb$ according to proposition \ref{prop:example of a special colimit 2 marked case} , one can suppose that $B$ is a marked $(\infty,k)$-category for $k<\omega$, and we then proceed by induction on $k$. 
The initialization is trivial.
Suppose then the result is true for $(\infty,k)$-categories and that $B$ is an $(\infty,k+1)$-category. Remark first that $\phi$ induces an equivalence between $\tau_0(X)$ and $\tau_0(Y)$. 

Let $x$ and $y$ be two objects of $X$ and
 $v:[1]^\sharp\to B^\sharp$ be a cell whose source is $px$ and target $py$. This induces cartesian squares
\[\begin{tikzcd}
	{X_1} & {X_v} & X \\
	{Y_1} & {Y_v} & Y \\
	{\{1\}} & {[1]^\sharp} & {B^\sharp}
	\arrow["{\phi_1}", from=1-1, to=2-1]
	\arrow["\phi", from=1-3, to=2-3]
	\arrow["{\phi_v}", from=1-2, to=2-2]
	\arrow[from=2-3, to=3-3]
	\arrow[from=2-2, to=3-2]
	\arrow[from=2-1, to=3-1]
	\arrow[from=3-1, to=3-2]
	\arrow["v"', from=3-2, to=3-3]
	\arrow[from=2-1, to=2-2]
	\arrow[from=1-1, to=1-2]
	\arrow[from=2-2, to=2-3]
	\arrow[from=1-2, to=1-3]
	\arrow["\lrcorner"{anchor=center, pos=0.125}, draw=none, from=1-1, to=2-2]
	\arrow["\lrcorner"{anchor=center, pos=0.125}, draw=none, from=1-2, to=2-3]
	\arrow["\lrcorner"{anchor=center, pos=0.125}, draw=none, from=2-2, to=3-3]
	\arrow["\lrcorner"{anchor=center, pos=0.125}, draw=none, from=2-1, to=3-2]
\end{tikzcd}\]
By hypothesis, $\phi_1$ is an equivalence.
According to proposition \ref{prop:left Gray transfomration stable under pullback along cartesian fibration}, $\phi_1\to \phi_v$ is a right deformation retract, and according to proposition \ref{prop:Gray deformation retract and passage to hom}, this induces a  square
\[\begin{tikzcd}
	{\hom_{X_1}(x,ry)} & {\hom_{X_v}(x,y)} \\
	{\hom_{Y_1}(\phi x,r\phi y)} & {\hom_{Y_v}(\phi x,\phi y)}
	\arrow[from=1-1, to=1-2]
	\arrow[from=1-1, to=2-1]
	\arrow[from=1-2, to=2-2]
	\arrow[from=2-1, to=2-2]
\end{tikzcd}\]
where horizontal morphisms are equivalences. By hypothesis, the left vertical one is an equivalence, and then, by two out of three, so is the right vertical one. 

We then have, for any $1$-cell $v$, the following cartesian squares
\[\begin{tikzcd}
	{\hom_{X_v}(x,y)} & {\hom_{X}(x,y)} \\
	{\hom_{Y_v}(\psi x,\psi y)} & {\hom_{Y}(\psi x,\psi y)} \\
	{\{v\}} & {\hom_{B}(px,py)^\sharp}
	\arrow["\sim"', from=1-1, to=2-1]
	\arrow[from=3-1, to=3-2]
	\arrow[from=2-2, to=3-2]
	\arrow[from=1-2, to=2-2]
	\arrow[from=2-1, to=2-2]
	\arrow[from=2-1, to=3-1]
	\arrow[from=1-1, to=1-2]
\end{tikzcd}\]
where the arrow labeled  by $\sim$ is an equivalence. As $\hom_{B}(px,py)^\sharp$ is an $(\infty,k)$-category, the induction hypothesis implies that $\hom_{X}(x,y)\to \hom_{Y}(\psi x,\psi y)$  is an equivalence. The morphism $\phi$ is then fully faithful, and as we already know that it is essentially surjective, this concludes the proof.
\end{proof}

\vspace{1cm}  We have by construction \ref{cons:of fb for fibration} a factorization system in initial morphism followed by left cartesian fibration, and another one in final morphism followed by right cartesian fibration. We are willing to find an explicit expression for such factorization in some easy cases. We then fix $i:C^\flat \to D$ with $D$ being any marked $\io$-category.

\begin{definition}
If $C^\flat\to D$ is a functor between marked $\io$-categories, we define $D_{/C^{\flat}}$ and $D_{C^{\flat}/}$ as the following pullbacks 
\[\begin{tikzcd}
	{D_{C^{\flat}/}} & {D^{[1]^\sharp}} && {D_{/C^{\flat}}} & {D^{[1]^\sharp}} \\
	{C^{\flat}} & D && {C^{\flat}} & D
	\arrow["{(i_0^-)_!}", from=1-2, to=2-2]
	\arrow[from=2-1, to=2-2]
	\arrow[from=1-1, to=2-1]
	\arrow[from=1-1, to=1-2]
	\arrow["\lrcorner"{anchor=center, pos=0.125}, draw=none, from=1-1, to=2-2]
	\arrow["{(i_1^-)_!}", from=1-5, to=2-5]
	\arrow[from=1-4, to=2-4]
	\arrow[from=2-4, to=2-5]
	\arrow[from=1-4, to=1-5]
	\arrow["\lrcorner"{anchor=center, pos=0.125}, draw=none, from=1-4, to=2-5]
\end{tikzcd}\]
If $C$ is the terminal $\io$-category, this notation is compatible with the one of the slice over and under introduced in paragraph.
\end{definition}

\begin{lemma}
\label{lemma:explicit factoryzation 1}
The morphism $i:C^\flat\to D_{/C^{\flat}}$ appearing in the following diagram
\[\begin{tikzcd}
	& D \\
	{C^{\flat}} & {D_{C^{\flat}/}} & {D^{[1]^\sharp}} \\
	& {C^{\flat}} & D
	\arrow["{(i_0^-)_!}", from=2-3, to=3-3]
	\arrow[from=3-2, to=3-3]
	\arrow[from=2-2, to=3-2]
	\arrow[from=2-2, to=2-3]
	\arrow["\lrcorner"{anchor=center, pos=0.125}, draw=none, from=2-2, to=3-3]
	\arrow[curve={height=-12pt}, from=2-1, to=1-2]
	\arrow[curve={height=-12pt}, from=1-2, to=2-3]
	\arrow["id"', from=2-1, to=3-2]
	\arrow["i", dashed, from=2-1, to=2-2]
\end{tikzcd}\]
is initial.
\end{lemma}
\begin{proof}
Using proposition \ref{prop:associativity of Gray amput2}, we have a natural transformation
$$(\uvar\otimes[1]^\sharp)\otimes[1]^\sharp \sim \uvar\otimes([1]^\sharp\times[1]^\sharp)
\xrightarrow{\uvar\otimes\psi} \uvar\otimes[1]^\sharp$$
where $\psi$ sends $(\epsilon,\epsilon')$ on $\max(\epsilon,\epsilon')$. 
This induces squares:
\[\begin{tikzcd}
	{D^{}_{/C^\flat}} & {D^{[1]^\sharp}} & {(D^{[1]^\sharp})^{[1]^\sharp}} \\
	{C^\flat} & D & {D^{[1]^\sharp}}
	\arrow[from=1-1, to=1-2]
	\arrow[from=1-1, to=2-1]
	\arrow["\lrcorner"{anchor=center, pos=0.125}, draw=none, from=1-1, to=2-2]
	\arrow[from=1-2, to=1-3]
	\arrow[from=1-2, to=2-2]
	\arrow[from=1-3, to=2-3]
	\arrow[from=2-1, to=2-2]
	\arrow["{D^{s_0}}"', from=2-2, to=2-3]
\end{tikzcd}\]
As we have two squares
\[\begin{tikzcd}
	{(D^{}_{/C^\flat})^{[1]^\sharp}} & {(D^{[1]^\sharp})^{[1]^\sharp}} & {C^\flat} & D \\
	{(C^{\flat})^{[1]^\sharp}} & {D^{[1]^\sharp}} & {(C^{\flat})^{[1]^\sharp}} & {D^{[1]^\sharp}}
	\arrow[from=1-1, to=1-2]
	\arrow[from=1-1, to=2-1]
	\arrow["\lrcorner"{anchor=center, pos=0.125}, draw=none, from=1-1, to=2-2]
	\arrow[from=1-2, to=2-2]
	\arrow[from=1-3, to=1-4]
	\arrow[from=1-3, to=2-3]
	\arrow[from=1-4, to=2-4]
	\arrow["{D^{s_0}}"', from=2-1, to=2-2]
	\arrow[from=2-3, to=2-4]
\end{tikzcd}\]
This induces, by the universal property of pullback, a morphism $:D^{}_{/C^\flat})\to (D^{}_{/C^\flat})^{[1]^\sharp}$ and, by adjunction, a morphism $\phi:D^{}_{/C^\flat})\otimes[1]^\sharp\to D^{}_{/C^\flat}).$
We set $r:D_{C^{\flat}/}\to C^\flat$ as the canonical projection. Eventually, remark that $(i,r,\phi)$ is a left Gray deformation retract. According to proposition \ref{prop:left Gray deformation retract are initial}, this concludes the proof.
\end{proof}

\begin{lemma}
\label{lemma:explicit factoryzation 2}
The composite $q:D_{C^{\flat}/}\to D^{[1]^\sharp}\xrightarrow{D^{i_0^+}} D$ is a left cartesian fibration.
\end{lemma}
\begin{proof}
Consider a commutative diagram
\begin{equation}
\label{eq:lemma:explicit factoryzation 2}
\begin{tikzcd}
	{K\otimes\{0\}} & {D_{C^{\flat}/}} \\
	{K\otimes[1]^\sharp} & D
	\arrow[from=1-1, to=2-1]
	\arrow[from=1-1, to=1-2]
	\arrow[from=2-1, to=2-2]
	\arrow[from=1-2, to=2-2]
\end{tikzcd}
\end{equation}
The $\infty$-groupoid of lifts of this previous diagram is equivalent to the $\infty$-groupoid of pairs consisting of a commutative triangle 
\[\begin{tikzcd}
	{K\otimes\{0\}\otimes\{0\}} \\
	{K\otimes[1]^\sharp\otimes\{0\}} & {C^\flat}
	\arrow[from=1-1, to=2-1]
	\arrow["f", from=1-1, to=2-2]
	\arrow[dashed, from=2-1, to=2-2]
\end{tikzcd}\]
where $f$ is induced by $K\otimes\{0\}\to D_{C^{\flat}/}\to C^{\flat}$,
 and a lift in the induced diagram
\[\begin{tikzcd}
	{K\otimes\{0\}\otimes[1]^\sharp\cup K\otimes[1]^\sharp\otimes\{1\} \cup K\otimes[1]^\sharp\otimes\{0\}} & D \\
	{K\otimes[1]^\sharp\otimes[1]^\sharp} & 1
	\arrow[from=1-1, to=1-2]
	\arrow[from=1-1, to=2-1]
	\arrow[from=2-1, to=2-2]
	\arrow[from=1-2, to=2-2]
	\arrow[dashed, from=2-1, to=1-2]
\end{tikzcd}\]
According to proposition \ref{prop:cotimes 1 to c is a trivialization}, the morphism $K\otimes[1]^\sharp\otimes\{0\}\to C^\flat$ factors through a morphism $K\to C^\flat$, and is then uniquely determined by $f:K\otimes\{0\}\otimes\{0\}\to C^\flat$,  and proposition \ref{prop:associativity of Gray amput2} provides a natural equivalence between $(K\otimes[1]^\sharp)\otimes[1]^\sharp$ and $K\otimes([1]^\sharp\times [1]^\sharp)$. The $\infty$-groupoid of lifts of the diagram \eqref{eq:lemma:explicit factoryzation 2} is then equivalent  to the  $\infty$-groupoid of lifts of  the left square of the following diagram
\[\begin{tikzcd}[column sep =0.7cm]
	{K\otimes\{0\}\otimes[1]^\sharp\cup K\otimes[1]^\sharp\otimes\{1\} \cup K\otimes[1]^\sharp\otimes\{0\}} & {K\otimes[1]^\sharp\cup K\otimes[1]^\sharp} & D \\
	{K\otimes[1]^\sharp\otimes[1]^\sharp} & {K\otimes[2]^\sharp} & 1
	\arrow[from=1-1, to=2-1]
	\arrow[from=1-3, to=2-3]
	\arrow[from=2-1, to=2-2]
	\arrow[""{name=0, anchor=center, inner sep=0}, from=1-1, to=1-2]
	\arrow[from=1-2, to=2-2]
	\arrow[from=1-2, to=1-3]
	\arrow[from=2-2, to=2-3]
	\arrow[dashed, from=2-2, to=1-3]
	\arrow["\lrcorner"{anchor=center, pos=0.125, rotate=180}, draw=none, from=2-2, to=0]
\end{tikzcd}\]
As $K\otimes[1]^\sharp\coprod_{K\otimes[0]}K\otimes[1]^\sharp\to K\otimes[2]^\sharp$ is an equivalence, this   $\infty$-groupoid is contractible.
\end{proof}

\begin{prop}
\label{prop:explicit factoryzation}
The factorisation of $p:C^\flat \to D$
in an initial morphism followed by a left cartesian fibration is
$$C^\flat \xrightarrow{i} D_{C^\flat/}\xrightarrow{q} D,$$
and its factorization in a final morphism and a right cartesian fibration is 
$$C^\flat \xrightarrow{i} D_{/C^\flat}\xrightarrow{q} D.$$
\end{prop}
\begin{proof}
This is a direct application of lemma \ref{lemma:explicit factoryzation 1} and \ref{lemma:explicit factoryzation 1} and of their dual version.
\end{proof}
\begin{remark}
The more important example of the previous proposition is the case $C:=\{a\}$. In this case, the corresponding left cartesian fibration is the slice of $D$ under $a$
$$D_{a/}\to D$$
defined in \ref{defi: marked slice},
 and the corresponding right cartesian fibration is the slice of $D$ over $a$
$$D_{/a}\to D$$
also defined in \ref{defi: marked slice},
\end{remark}
\begin{definition}
Let $p:X\to Y$ be a morphism between $\io$-categories. A marked $1$-cell $v:x\to x'$ is \wcnotion{left cancellable}{left or right cancellable $1$-cell} if for any $y$, the following natural square is cartesian:
\[\begin{tikzcd}
	{\hom_X(x',y)} & {\hom_X(x,y)} \\
	{\hom_Y(px',py)} & {\hom_Y(px,py)}
	\arrow["{v_!}", from=1-1, to=1-2]
	\arrow["{p(v)_!}"', from=2-1, to=2-2]
	\arrow[from=1-2, to=2-2]
	\arrow[from=1-1, to=2-1]
\end{tikzcd}\]

Conversely, a $1$-cell $v:y\to y'$ is \textit{right cancellable} if for any $x$, the following natural square is cartesian:
\[\begin{tikzcd}
	{\hom_X(x,y)} & {\hom_X(x,y')} \\
	{\hom_Y(px,py)} & {\hom_Y(px,py')}
	\arrow["{v_!}", from=1-1, to=1-2]
	\arrow["{p(v)_!}"', from=2-1, to=2-2]
	\arrow[from=1-2, to=2-2]
	\arrow[from=1-1, to=2-1]
\end{tikzcd}\]
\end{definition}

\begin{notation}
Given $a$ an element of $\Theta_t$. The morphism $[1]^\sharp\hookrightarrow[1]^\sharp\vee[a,1]^\sharp$ and $[a,1]^\sharp\hookrightarrow[1]^\sharp\vee[a,1]^\sharp$ denote the canonical globular morphisms.
\end{notation}

\begin{lemma}
\label{lemma:technical lemma on cancellable cell 1}
Let $p$ be a morphism. 
The following conditions are equivalent:
\begin{enumerate}
\item $p$ has the unique right lifting property against $\{0\}\to [1]^\sharp$ and marked $1$-cells are left cancellable.
\item $p$ has the unique right lifting property against $[a,1]\xrightarrow{\triangledown} [1]^\sharp\vee[a,1]$ for any object $a$ of $t\Theta$.
\item $p$ has the unique right lifting property against $[a,1]\xrightarrow{\triangledown} [1]^\sharp\vee[a,1]$ and $[1]^\sharp\xrightarrow{\triangledown}[1]^\sharp\vee[1]^\sharp$ for any object $a$ of $t\Theta$.
\end{enumerate}
Conversely, the following are equivalent:
\begin{enumerate}
\item[(1)'] $p$ has the unique right lifting property against $\{1\}\to [1]^\sharp$ and marked $1$-cells are right cancellable.
\item[(2)'] $p$ has the unique right lifting property against $[a,1]\xrightarrow{\triangledown} [a,1]\vee[1]^\sharp$ for any object $a$ of $t\Theta$.
\item[(3)'] $p$ has the unique right lifting property against $[a,1]\xrightarrow{\triangledown}[a,1]\vee[1]^\sharp$ and $[1]^\sharp\xrightarrow{\triangledown}[1]^\sharp\vee[1]^\sharp$ for any object $a$ of $t\Theta$.
\end{enumerate}
\end{lemma}
\begin{proof}
 The fact that $1$-cells are left cancellable is equivalent to asking that $i$ has the unique right lifting property against 
$$[a,1]\amalg_{\{0\}} [1]^\sharp\to [1]^\sharp\vee[a,1]$$
 for any object $a$ of $t\Theta$, and where $[a,1]\amalg_{\{0\}} [1]^\sharp$ is the pushout of the span
\[\begin{tikzcd}
	{[a,1]} & {\{0\}} & {[1]}
	\arrow["{i_0^-}"', from=1-2, to=1-1]
	\arrow["{i_0^-}", from=1-2, to=1-3]
\end{tikzcd}\]
 Suppose that $p$ fulfills $(1)$.
The class of morphisms having the unique right lifting property against $p$ is closed under pushout. The morphism $p$ then has the unique left lifting property against $$[a,1] \to [a,1]\amalg_{\{0\}} [1]^\sharp.$$ As remarked above, the fact that marked $1$-cells are left cancellable implies that $p$ has the unique right lifting property against 
$$[a,1]\amalg_{\{0\}} [1]^\sharp\to [1]^\sharp\vee[a,1].$$
By stability by composition, $p$ then has the unique right lifting property against  
$$[a,1]\xrightarrow{\triangledown} [1]^\sharp\vee[a,1].$$
This implies $(1)\Rightarrow (2)$.

Suppose now that $p$ fulfills $(2)$. Remark that we have a retract
\[\begin{tikzcd}
	{\{0\}} & {[1]} & {\{0\}} \\
	{[1]^\sharp} & {[1]^\sharp\vee[1]} & {[1]^\sharp}
	\arrow[from=1-1, to=2-1]
	\arrow[hook, from=2-1, to=2-2]
	\arrow["id\vee\Ib"', from=2-2, to=2-3]
	\arrow[from=1-1, to=1-2]
	\arrow["\triangledown"', from=1-2, to=2-2]
	\arrow["\Ib", from=1-2, to=1-3]
	\arrow[from=1-3, to=2-3]
\end{tikzcd}\]
and as the class of morphisms having the unique right lifting property against $p$ is closed under retracts, this implies that $p$ has the unique right lifting property against $\{0\}\to [1]^\sharp$. 
By stability under pushout, $p$ has the unique right lifting property against 
$$[a,1] \to [a,1]\amalg_{\{0\}} [1]^\sharp.$$
By stability under left cancellation, $p$ then has the unique right lifting property against 
$$[a,1]\amalg_{\{0\}} [1]^\sharp\to [1]^\sharp\vee[a,1].$$
As remarked above, this implies that $1$-cells are left cancellable. We then have $(1)\Leftrightarrow (2)$.

There is an obvious implication $(3)\Rightarrow (2)$.  To show the converse, we suppose that $p$ fulfills $(1)$ and $(2)$. As the
class of morphisms having the unique right lifting property against $p$ is closed under pushout and composition,
 it then contains $\{0\}\to [1]^\sharp$ and $\{0\}\to [1]^\sharp\to  [1]^\sharp\vee[1]^\sharp$. By left cancellation, it also includes $ [1]^\sharp\xrightarrow{\triangledown}[1]^\sharp\vee[1]^\sharp$.
The proof of the equivalence of $(1)'$, $(2)'$ and $(3)'$ is symetrical.
\end{proof}

\begin{lemma}
\label{lemma:technical lemma on cancellable cell 3}
Let $p:X\to Y$ be a morphism having the unique right lifting property against marked trivializations, such that for any element $a$ of $t\Theta$, and any cartesian squares: 	
\[\begin{tikzcd}
	{X''} & {X'} & X \\
	{[a,1]} & {[1]^\sharp\vee[a,1]} & Y
	\arrow["p"', from=1-3, to=2-3]
	\arrow["{p''}"', from=1-1, to=2-1]
	\arrow["{p'}"', from=1-2, to=2-2]
	\arrow["k", from=1-1, to=1-2]
	\arrow[from=1-2, to=1-3]
	\arrow[hook, from=2-1, to=2-2]
	\arrow[from=2-2, to=2-3]
	\arrow["\lrcorner"{anchor=center, pos=0.125}, draw=none, from=1-1, to=2-2]
	\arrow["\lrcorner"{anchor=center, pos=0.125}, draw=none, from=1-2, to=2-3]
\end{tikzcd}\]
the square $p''\to p'$ is a right deformation retract.
Then, $p$ has the unique right lifting property against $[a,1]\xrightarrow{\triangledown} [1]^\sharp\vee[a,1]$ for any object $a$ of $t\Theta$. 
\end{lemma}
\begin{proof}
Suppose given a square
\[\begin{tikzcd}
	{[a,1]} & X \\
	{[1]^\sharp\vee[a,1]} & Y
	\arrow["p"', from=1-2, to=2-2]
	\arrow["\triangledown"', from=1-1, to=2-1]
	\arrow[from=1-1, to=1-2]
	\arrow["g"', from=2-1, to=2-2]
\end{tikzcd}\]
and let $p'$ and $p''$ be the morphisms appearing in the following cartesian squares:
\[\begin{tikzcd}
	{X''} & {X'} & X \\
	{[a,1]} & {[1]^\sharp\vee[a,1]} & Y
	\arrow["p"', from=1-3, to=2-3]
	\arrow["{p''}"', from=1-1, to=2-1]
	\arrow["{p'}"', from=1-2, to=2-2]
	\arrow["k", from=1-1, to=1-2]
	\arrow[from=1-2, to=1-3]
	\arrow[hook, from=2-1, to=2-2]
	\arrow["g"', from=2-2, to=2-3]
	\arrow["\lrcorner"{anchor=center, pos=0.125}, draw=none, from=1-1, to=2-2]
	\arrow["\lrcorner"{anchor=center, pos=0.125}, draw=none, from=1-2, to=2-3]
\end{tikzcd}\]
Let $(k:X''\to X',r,\phi)$ the left deformation retract existing by hypothesis.

We first demonstrate that any diagram of shape
\begin{equation}
\label{eq in lemma:technique}
\begin{tikzcd}
	{\{0\}} & {X'} \\
	{ [1]^\sharp} & {[1]^\sharp\vee[a,1]}
	\arrow["{\{x\}}", from=1-1, to=1-2]
	\arrow[from=1-1, to=2-1]
	\arrow["{p'}"', from=1-2, to=2-2]
	\arrow[hook, from=2-1, to=2-2]
\end{tikzcd}
\end{equation}
admits a unique lifting given by $\phi(x):x\to r(x)$. Let $l:[1]^\sharp\to X'$ be any lifting. Remark that $l(1)$ is in $X''$.
The deformation $\phi$ then induces a square in $X'$ of shape
\[\begin{tikzcd}
	x & {l(1)} \\
	{r(x)} & {l(1)}
	\arrow["{/}"{marking, allow upside down}, from=1-1, to=1-2]
	\arrow["l", draw=none, from=1-1, to=1-2]
	\arrow["{/}"{marking, allow upside down}, from=1-1, to=2-1]
	\arrow["{\phi(x)~~}"', draw=none, from=1-1, to=2-1]
	\arrow[Rightarrow, no head, from=1-2, to=2-2]
	\arrow["{/}"{marking, allow upside down}, from=2-1, to=2-2]
\end{tikzcd}\]
Moreover, the marked $1$-cell $r(x)\to r(x)$ is over a unit, and as $p'$ has the unique right lifting property against marked trivializations, this implies that this $1$-cell is an equivalence. As all the operations to produce the previous square were functorial, we have constructed a retraction structure for the inclusion of $\{\phi(x):x\to r(x)\}$ into the space of lifts of the square \eqref{eq in lemma:technique}.

To conclude the proof of the proposition, one has to demonstrate that the induced diagram
\[\begin{tikzcd}
	{[a,1]} & {X'} \\
	{[1]^\sharp\vee[a,1]} & {[1]^\sharp\vee[a,1]}
	\arrow["j", from=1-1, to=1-2]
	\arrow["\triangledown"', from=1-1, to=2-1]
	\arrow["{p'}"', from=1-2, to=2-2]
	\arrow["id"', from=2-1, to=2-2]
\end{tikzcd}\]
admits a unique lifting. 
We denote by $x_0$ and $x_2$ the image of the object of $[a,1]$ via the morphism $j$. As demonstrated above, the unique marked $1$-cell in $X'$ over $[1]^\sharp\hookrightarrow [1]^\sharp\vee[a,1]$ with $x_0$ for source is $\phi(x_0):x_0\to r(x_0)$.
The $\infty$-groupoid of lifts of this diagram is then equivalent to the $\infty$-groupoid of lifts of the following diagram
\[\begin{tikzcd}
	\emptyset & {\hom_{X'}(rx_0,x_2)} \\
	a & {\hom_{X'}(x_0,x_2)}
	\arrow["{{\phi_{x_0}}_!}", from=1-2, to=2-2]
	\arrow[from=2-1, to=2-2]
	\arrow[from=1-1, to=2-1]
	\arrow[from=1-1, to=1-2]
\end{tikzcd}\]
However, the right vertical morphism is an isomorphism according to proposition \ref{prop:Gray deformation retract and passage to hom} which concludes the proof.
\end{proof}
\begin{definition}
\label{defi: of Fg et Ig}
 Keeping in mind the last lemma, we define \wcnotation{$\I_{g}$}{(ig@$\I_g$} and \wcnotation{$\F_{g}$}{(fg@$\F_g$} as the smallest sets of morphisms of $\zocatm$ fullfilling these conditions:
\begin{enumerate}
\item for any $a\in \Theta^t$, the globular morphism $[a,1]\hookrightarrow[1]^\sharp\vee[a,1]$ is in $\F_g$ and the globular morphism $[a,1]\hookrightarrow[a,1]\vee[1]^\sharp$ is in $\I_g$
\item for any $i$ in $\F_g$, $[i,1]$ is in $\I_g$, for any $j$ in $\I_g$, $[i,1]$ is in $\F_g$,
\end{enumerate}
\end{definition}

\begin{remark}
Propositions \ref{prop:left Gray deformation retract stable under pushout} and \ref{prop:suspension of left Gray deformation retract} then imply that morphisms of $\I_g$ are left Gray deformation retracts and morphisms of $\F_g$ are right Gray deformation retracts.
\end{remark}

\begin{definition}
We extend by induction the definition of right and left cancellable to cells of any dimension as follows: a $n$-cell $v$ is \wcnotion{left or right cancellable}{left cancellable $n$-cell} (resp. \textit{right cancellable}) if the corresponding $(n-1)$-cell of $\hom_X(x,y)$ is left cancellable (resp. right cancellable) for the morphism $\hom_X(x,y)\to \hom_Y(px,py)$, where $x$ and $y$ denote the $0$-sources and $0$-but of $v$.
\end{definition}

\begin{lemma}
\label{lemma:technical lemma on cancellable cell 2}
Let $p':X'\to Y'$ be a morphism such that $p$ has the unique right lifting property against marked trivializations and suppose that we have a left Gray deformation retract $p'\to p$. We denote by $(r:Y'\to Y,i,\phi)$ the left deformation retract structure induced on the codomain, and suppose that the deformation $\phi:Y\otimes[1]^\sharp\to Y$ factors through $\psi:Y\times[1]^\sharp\to Y$. Then, the square $p'\to p$ is a left deformation retract.
\end{lemma}
\begin{proof}
Proposition \ref{prop:cotimes 1 to ctimes 1 is a trivialization} states that $Y\otimes[1]^\sharp\to Y\times [1]^\sharp$ is a marked trivialization. There is then a lift in the following diagram:
\[\begin{tikzcd}
	{X\otimes[1]^\sharp} && X \\
	{X\times[1]^\sharp} & {Y\times[1]^\sharp} & Y
	\arrow[from=1-1, to=2-1]
	\arrow[from=2-1, to=2-2]
	\arrow[from=1-3, to=2-3]
	\arrow["{\phi'}", from=1-1, to=1-3]
	\arrow["\psi"', from=2-2, to=2-3]
	\arrow["{\psi'}"{description}, dotted, from=2-1, to=1-3]
\end{tikzcd}\]
where $\phi'$ is the deformation induced on domains.
This endows $p'\to p$ with a structure of left deformation retract, where the retraction is the same, and the deformation is given by $(\psi',\psi)$.
\end{proof}

\begin{theorem}
\label{theo:other characterisation of left caresian fibration}
Consider the following shape of diagram
\begin{equation}
\label{eq:prop:other characterisation of left caresian fibration}
\begin{tikzcd}
	{X''} & {X'} & X \\
	{Y''} & {Y'} & Y
	\arrow["p"', from=1-3, to=2-3]
	\arrow["{p''}"', from=1-1, to=2-1]
	\arrow["{p'}"', from=1-2, to=2-2]
	\arrow[from=1-1, to=1-2]
	\arrow[from=1-2, to=1-3]
	\arrow["i"', from=2-1, to=2-2]
	\arrow[from=2-2, to=2-3]
	\arrow["\lrcorner"{anchor=center, pos=0.125}, draw=none, from=1-1, to=2-2]
	\arrow["\lrcorner"{anchor=center, pos=0.125}, draw=none, from=1-2, to=2-3]
\end{tikzcd}
\end{equation}
The following are equivalent:
\begin{enumerate}
\item The morphism $p$ is a left cartesian fibration.
\item $p$ has the unique right lifting property against marked trivialization, and for any diagram of shape \eqref{eq:prop:other characterisation of left caresian fibration},
if $i$ is a right Gray deformation retract, so is $p''\to p'$. 
\item $p$ has the unique right lifting property against marked trivialization and, for any diagram of shape \eqref{eq:prop:other characterisation of left caresian fibration},
if $i$ is in $\F_g$, the square $p''\to p'$ is a right Gray deformation retract.
\item For any even integer $n$, $p$ has the unique right lifting property against $i_n^+:\Db_{n}\to (\Db_{n+1})_t$ and marked $n$-cells are right cancellable; for any odd integer $p$ has the unique right lifting property against $i_n^-:\Db_{n}\to (\Db_{n+1})_t$ and marked $n$-cells are left cancellable.
\item $p$ as the unique right lifting property against $\{0\}\to [1]^\sharp$, marked $1$-cells are left cancellable, and
for any pair of objects $(x,y)$ of $X$, $\hom_X(x,y)\to \hom_Y(px,py)$ is a right cartesian fibration.
\end{enumerate} 
Conversely, the following are equivalent:
\begin{enumerate}
\item[(1)'] The morphism $p$ is a right cartesian fibration.
\item[(2)'] $p$ has the unique right lifting property against marked trivialization and for any diagram of shape \eqref{eq:prop:other characterisation of left caresian fibration},
if $i$ is a left Gray deformation retract, so is $p''\to p'$.
\item[(3)'] $p$ has the unique right lifting property against marked trivialization, and for any diagram of shape \eqref{eq:prop:other characterisation of left caresian fibration},
if $i$ is in $\I_g$, the square $p''\to p'$ is a left Gray deformation retract.
\item[(4)'] For any even integer $n$, $p$ has the unique right lifting property against $i_n^-:\Db_{n}\to (\Db_{n+1})_t$ and marked $n$-cells are left cancellable; for any odd integer $p$ has the unique right lifting property against $i_n^+:\Db_{n}\to (\Db_{n+1})_t$ and marked $n$-cells are right cancellable.
\item[(5)'] $p$ as the unique right lifting property against $\{1\}\to [1]^\sharp$, marked $1$-cells are right cancellable, and
for any pair of objects $(x,y)$ of $X$, $\hom_X(x,y)\to \hom_Y(px,py)$ is a left cartesian fibration.
\end{enumerate} 

\end{theorem}
\begin{proof}
The implication from $(1)$ to $(2)$ and $(1)'$ to $(2)'$ is the content of proposition \ref{prop:left Gray transfomration stable under pullback along cartesian fibration}.

The implication from $(2)$ to $(3)$ and $(2)'$ to $(3)'$ comes from the fact that $\I_g$ (resp. $\F_g$) consists of right (resp. left) Gray deformation retracts.

Suppose now that $p$ fulfills condition $(3)$. Lemma \ref{lemma:technical lemma on cancellable cell 2} implies that if $i$ is of shape $[a,1]\hookrightarrow [1]^\sharp\vee[a,1]$ for $a:t\Theta$, $p''\to p'$ is a right deformation retract. Lemma 
\ref{lemma:technical lemma on cancellable cell 3} and \ref{lemma:technical lemma on cancellable cell 1} then imply that $p$ has the unique right lifting property against $\{0\}\to [1]^\sharp$ and marked $1$-cells are left cancellable.

We are now willing to show that for any pair of objects $(x,y)$, $\hom_X(x,y)\to \hom_Y(px,py)$ fulfills condition $(3)'$, and an obvious induction will complete the proof of $(3)\Rightarrow (4)$.
We then consider $x,y$ two objects of $X$, $i:b\to a$ in $\I_g$ and any morphism $a\to \hom_Y(px,py)$. The previous data induces a pullback square
\[\begin{tikzcd}
	{X''} & {X'} & X \\
	{[b,1]} & {[a,1]} & Y
	\arrow["p"', from=1-3, to=2-3]
	\arrow["{p''}"', from=1-1, to=2-1]
	\arrow["{p'}"', from=1-2, to=2-2]
	\arrow[from=1-1, to=1-2]
	\arrow[from=1-2, to=1-3]
	\arrow["{[i,1]}"', from=2-1, to=2-2]
	\arrow[from=2-2, to=2-3]
	\arrow["\lrcorner"{anchor=center, pos=0.125}, draw=none, from=1-1, to=2-2]
	\arrow["\lrcorner"{anchor=center, pos=0.125}, draw=none, from=1-2, to=2-3]
\end{tikzcd}\]
where the bottom right morphism sends $\{0\}$ to $px$ and $\{1\}$ to $py$.
By construction, $[i,1]$ is in $\F_g$, and so by assumption, the morphism $p'\to p''$ is a right Gray deformation retract. Applying the functor $\hom_{\uvar}(\uvar,\uvar)$ we get the following pullback diagram:
\[\begin{tikzcd}
	{\hom_{X''}(x,y)} & {\hom_{X'}(x,y)} & {\hom_{X}(x,y)} \\
	b & a & {\hom_{Y}(px,py)}
	\arrow["{\tilde{p}}"', from=1-3, to=2-3]
	\arrow["{\tilde{p}''}"', from=1-1, to=2-1]
	\arrow["{\tilde{p}'}"', from=1-2, to=2-2]
	\arrow[from=1-1, to=1-2]
	\arrow[from=1-2, to=1-3]
	\arrow["i"', from=2-1, to=2-2]
	\arrow[from=2-2, to=2-3]
	\arrow["\lrcorner"{anchor=center, pos=0.125}, draw=none, from=1-1, to=2-2]
	\arrow["\lrcorner"{anchor=center, pos=0.125}, draw=none, from=1-2, to=2-3]
\end{tikzcd}\]
and the dual version of proposition \ref{prop:Gray deformation retract and passage to hom v2} implies that $\tilde{p}''\to \tilde{p}'$ is a left Gray deformation retract. As this is true for any $i:b\to a$ in $\I_g$, for any object of $X$, and any $a\to \hom_Y(px,py)$, this implies that $\hom_X(x,y)\to \hom_Y(px,py)$ fulfills condition $(3)'$. As mentioned above, an obvious induction induces $(3)\Rightarrow (4)$. We show similarly $(3)'\Rightarrow (4)'$.

Now let's show $(4)\Rightarrow (1)$ and $(4)'\Rightarrow (1)'$. We show by induction on $n$ that for any 
element $a$ of $t\Gb_n:=\{\Db_k\}_{0\leq k\leq n}\cup \{(\Db_k)_t\}_{1\leq k\leq n}$, if $p$ fulfills $(4)$ (resp. $(4)'$) $p$ has the unique right lifting property against
 $a\otimes\{0\}\to a\otimes[1]^\sharp$ (against $a\otimes\{1\}\to a\otimes[1]^\sharp$).
 
Suppose then that this is true at the stage $n$, and suppose that $p$ fulfills $(4)$. 
Let $a$ be an object of $t\Gb_n$. 	Remark that according to the proposition \ref{prop:eq for cylinder marked}, $[a,1]\otimes\{0\}\to [a,1]\otimes[1]^\sharp$ fits in the sequence of pushouts
\[\begin{tikzcd}
	{[0]} & {[a,1]\otimes \{0\}} \\
	{[1]^\sharp} & {[a,1]\vee[1]^\sharp} & {[a\otimes\{1\},1]} \\
	{[a,1]} & {[a,1]\vee[1]^\sharp\cup[a\otimes[1]^\sharp,1]} & {[a\otimes[1]^\sharp,1]} \\
	{[1]^\sharp\vee[a,1]} & {[a,1]\otimes[1]^\sharp}
	\arrow["{i_0^-}"', from=1-1, to=2-1]
	\arrow[""{name=0, anchor=center, inner sep=0}, "{i_0^+}", from=1-1, to=1-2]
	\arrow[from=2-1, to=2-2]
	\arrow[hook, from=1-2, to=2-2]
	\arrow[""{name=1, anchor=center, inner sep=0}, from=2-3, to=2-2]
	\arrow[from=3-3, to=3-2]
	\arrow[from=2-2, to=3-2]
	\arrow[from=2-3, to=3-3]
	\arrow[""{name=2, anchor=center, inner sep=0}, from=3-1, to=3-2]
	\arrow[from=4-1, to=4-2]
	\arrow[from=3-2, to=4-2]
	\arrow["\triangledown"', from=3-1, to=4-1]
	\arrow["\lrcorner"{anchor=center, pos=0.125, rotate=180}, draw=none, from=2-2, to=0]
	\arrow["\lrcorner"{anchor=center, pos=0.125, rotate=90}, draw=none, from=3-2, to=1]
	\arrow["\lrcorner"{anchor=center, pos=0.125, rotate=180}, draw=none, from=4-2, to=2]
\end{tikzcd}\]
By induction hypothesis, for any pair of objects $(x,y)$ of $X$, $\hom_X(x,y)\to \hom_Y(px,py)$ has the unique right lifting property against $a\otimes\{1\}\to a\otimes[1]^\sharp$ for $a\in t\Gb_n$. Furthermore, lemma \ref{lemma:technical lemma on cancellable cell 1} implies that $p$ has the unique right lifting property against $\triangledown:[a,1]\to [1]^\sharp\vee[a,1]$. The morphism $p$ then has the unique right lifting property against $[a\otimes\{1\},1]\to [a\otimes[1]^\sharp,1]$ for $a\in t\Gb_n$. The class of morphisms having the unique right lifting property against $p$ being closed under colimits, this implies that it includes $[a,1]\otimes\{0\}\to [a,1]\otimes[1]^\sharp$. To conclude, one has to show that $p$ has the unique right lifting property against $[1]^\sharp\times\{0\}\to [1]^\sharp\times [1]^\sharp$. Remark that according to proposition \ref{prop:example of a special colimit marked case}, $[1]^\sharp\times\{0\}\to [1]^\sharp\times[1]^\sharp$ fits in the sequence of pushouts:
\[\begin{tikzcd}
	{[0]} & {[1]^\sharp\times\{0\}} \\
	{[1]^\sharp} & {[1]^\sharp\vee[1]^\sharp} & {[1]^\sharp} \\
	& {[1]^\sharp\times[1]^\sharp} & {[1]^\sharp\vee[1]^\sharp}
	\arrow["{i_0^-}"', from=1-1, to=2-1]
	\arrow["{i_0^+}", from=1-1, to=1-2]
	\arrow["\triangledown"', from=2-3, to=2-2]
	\arrow["\triangledown", from=2-3, to=3-3]
	\arrow[hook, from=2-1, to=2-2]
	\arrow[from=2-2, to=3-2]
	\arrow[from=3-3, to=3-2]
	\arrow[hook, from=1-2, to=2-2]
	\arrow["\lrcorner"{anchor=center, pos=0.125, rotate=180}, draw=none, from=2-2, to=1-1]
	\arrow["\lrcorner"{anchor=center, pos=0.125, rotate=90}, draw=none, from=3-2, to=2-3]
\end{tikzcd}\]
According to lemma \ref{lemma:technical lemma on cancellable cell 1}, $p$ has the unique right lifting property against $\triangledown:[1]^\sharp\to [1]^\sharp\vee[1]^\sharp$ and so also against $[1]^\sharp\times\{0\}\to [1]^\sharp\times [1]^\sharp$. This concludes the proof of the implication $(4)\Rightarrow (1)$. We show similarly $(4)'\Rightarrow (1)'$.

Eventually, the implication $(1)\Rightarrow (5)$ and $(1)'\Rightarrow (5)'$ are a consequence of proposition \ref{prop:cartesian fibration between arrow} and of the implications $(1)\Rightarrow (4)$ and $(1)'\Rightarrow (4)'$. The implications $(5)\Rightarrow (4)$ and $(5)'\Rightarrow (4)'$ are a consequence of the equivalences  $(1)'\Leftrightarrow (4)'$ and $(1)\Leftrightarrow (4)$ applied to the morphisms $\hom_X(x,y)\to \hom_Y(px,py)$ for all objects $x,y$.
\end{proof}

\begin{cor}
\label{cor:on the fact that fib are define against representable}
A morphism $p:X\to A^\sharp$ is a left cartesian fibration if and only if for any globular sum $b$ and morphism $j:b\to A$, $j^*p$ is a left cartesian fibration over $b^\sharp$.
\end{cor} 
\begin{proof}
This is a direct consequence of the equivalence between conditions $(1)$ and $(3)$ of theorem \ref{theo:other characterisation of left caresian fibration}, and the fact that the codomains of marked trivializations and the codomains of morphisms of $\F_g$ are marked globular sums. 
\end{proof}

\subsection{Cartesian fibration are exponentiable}
\label{subsection:A criterion to be a left cartesian fibration}

We recall that a {marked globular sum} is a marked $\io$-category whose underlying $\io$-category is a globular sum and such that for any pair of integers $k\leq n$, and any pair of $k$-composable $n$-cells $(x,y)$, $x\circ_k y$ is marked if and only if $x$ and $y$ are marked.

 A morphism $i:a\to b$ between marked globular sums is {globular} if the morphism $i^\natural$ is {globular}.

 A globular morphism $i$ between marked globular sums is then a discrete Conduché functor, which implies according to proposition \ref{prop:pullback by conduch marked preserves colimit} that the functor $i^*:\ocatm_{/b}\to \ocatm_{/a}$ preserves colimits.

\begin{definition}
 Let $b$ be a globular sum and $f:X\to b^\sharp$ a morphism. We say that $f$ is \wcnotion{$b$-exponentiable}{expo@$b$-exponentiable} if 
the canonical morphism $$\colim_{i:\Sp_b^\sharp} {i}^*f\to f$$ is an equivalence. 
\end{definition}

\begin{remark}
\label{rem:a expo implies a expo with other marking}
Let $b$ be a marked globular sum and $f:X\to b^\sharp$ a $b$-exponentiable morphism. As the canonical morphism $\iota:b\to b^{\sharp}$ is a marked discrete Conduché functor, it is exponentiable by proposition \ref{prop:pullback by conduch marked preserves colimit}, and  we have an equivalence
$$\colim_{i:\iota^*\Sp_b^\sharp} {i}^*f\to \iota^*f$$
\end{remark}

\begin{prop}
\label{prop:exponantiable stable under colim}
Let $F:I\to \ocatm_{/b^\sharp}$ be a functor which is pointwise $b$-exponentiable. The morphism $\colim_I F$ is $b$-exponentiable 
\end{prop}
\begin{proof}
Remark that all morphisms $\Db_n^\sharp\to b^\sharp$ in $\Sp^\sharp_b$ are globular, and so are discrete Conduché functors. We then have a sequence of equivalences
$$\colim_{i:\Sp_b^\sharp} {i}^*\colim_I F\sim \colim_I\colim_{i:\Sp_b^\sharp} {i}^*F\sim \colim_I F.$$
\end{proof}

\begin{prop}
\label{prop:how to create exponentiable}
Let $a$ be a globular sum, and $f:X\to a^\sharp$ be a morphism. The induced morphism $\colim_{i:\Sp_{a}^\sharp}i^*f$ is $a$-exponentiable.
\end{prop}
\begin{proof}
As marked globular morphisms are marked discrete Conduché functors, for any $j:\Db_n^\sharp\to a^\sharp\in \Sp_a$, $j^*\colim_{i:\Sp_{a}^\sharp}i^*f$ is equivalent to $j^*f$. We then have a sequence of equivalences
$$ \colim_{j:\Sp_{a}^\sharp}j^* \colim_{i:\Sp_{a}^\sharp}i^*f \sim \colim_{j:\Sp_{a}^\sharp}j^*f .$$
\end{proof}

\begin{prop}
\label{prop:exponantiable stable under pullback}
Let $f:X\to b^\sharp$ be $b$-exponentiable and $j:a^\sharp\to b^\sharp$ a globular morphism. The morphism $j^*f:X\to a^\sharp$ is $a$-exponentiable.
\end{prop}
\begin{proof}
The data induces a diagram in $\tiPsh{\Theta}$:
\[\begin{tikzcd}[cramped]
	& {X'} && X \\
	{X'''} && {X''} \\
	& {\Sp_b^\sharp} && {b^\sharp} \\
	{\Sp_a^\sharp} && {a^\sharp}
	\arrow["i", from=1-2, to=1-4]
	\arrow[from=1-4, to=3-4]
	\arrow[from=2-1, to=1-2]
	\arrow["{~~~~j}", from=2-1, to=2-3]
	\arrow[from=2-1, to=4-1]
	\arrow[from=2-3, to=1-4]
	\arrow[from=2-3, to=4-3]
	\arrow[from=3-2, to=1-2]
	\arrow[from=3-2, to=3-4]
	\arrow[from=4-1, to=3-2]
	\arrow[from=4-1, to=4-3]
	\arrow[from=4-3, to=3-4]
\end{tikzcd}\]
where, by stability by composition and right cancelation, all squares are cartesian. The hypothesis implies that $i$ is in $\widehat{\W}$. As $j$ is a conduché fibration,  the morphism $i'$ is in $\widehat{\W}$. This directly implies that $j^*f$ is $a$-exponentiable.
\end{proof}

\begin{lemma}
\label{lemma:unseful}
Suppose given a cocartesian square
\[\begin{tikzcd}[cramped]
	a & c \\
	d & b
	\arrow["k", from=1-1, to=1-2]
	\arrow["l"', from=1-1, to=2-1]
	\arrow["i", from=1-2, to=2-2]
	\arrow["j"', from=2-1, to=2-2]
	\arrow["\lrcorner"{anchor=center, pos=0.125, rotate=180}, draw=none, from=2-2, to=1-1]
\end{tikzcd}\]
where all objects are marked globular sums and all morphisms are globular, and let $f:X\to b^\sharp$ be a $b^\natural$-exponentiable morphism. We denote $\iota:b\to b^\sharp$ the canonical morphism.
The morphism $\iota^*f$ is the coproduct in the category of arrows of the span
$$j^*f\leftarrow (ik)^*f\to  i^*f.$$
\end{lemma}
\begin{proof}
This directly follows from remark \ref{rem:a expo implies a expo with other marking} and proposition \ref{prop:exponantiable stable under pullback}.
\end{proof}

\begin{lemma}
\label{lemma:technical lemma exponentiability}
Let $i:c\to d$ be a morphism in the set $\F_g$ defined in \ref{defi: of Fg et Ig}, $b$ a globular sum, and $f:d\to b^\sharp$ any morphism. 
Then, there exists a commutative square
\[\begin{tikzcd}
	{c'} & {d'} & {b^\sharp} \\
	c & d
	\arrow["h", from=2-1, to=1-1]
	\arrow["g", from=2-2, to=1-2]
	\arrow["{i'}", from=1-1, to=1-2]
	\arrow["i"', from=2-1, to=2-2]
	\arrow[from=1-2, to=1-3]
	\arrow["f"', from=2-2, to=1-3]
\end{tikzcd}\]
\begin{enumerate}
\item $d\to d'$ is a finite composition of pushouts of morphism of shape $i_n^\alpha:\Db_n\to (\Db_{n+1})_t$ with $n$ an integer and $\alpha:=+$ if $n$ is even, and $-$ if not.
\item $d'\to b^\sharp$ is globular.
\item $h\to g$ is a right Gray deformation retract.
\end{enumerate}
\end{lemma}
\begin{proof}
We obtain $(d')^\natural$ by factorizing $f^\natural$ into an algebraic morphism $g^\natural$ followed by a globular morphism. The marking $d'$ is the smaller one that makes $g$ a morphism of marked $\zo$-categories. By construction, $c\to d$ fits in a cocartesian square
\[\begin{tikzcd}
	{\Db_n^\flat} & c \\
	{ (\Db_{n+1})_t} & d
	\arrow[from=1-2, to=2-2]
	\arrow["{i^\alpha_n}"', from=1-1, to=2-1]
	\arrow[from=2-1, to=2-2]
	\arrow[from=1-1, to=1-2]
	\arrow["\lrcorner"{anchor=center, pos=0.125, rotate=180}, draw=none, from=2-2, to=1-1]
\end{tikzcd}\]
where all morphisms are globular, and where $\alpha$ is $+$ if $n$ is even, and $-$ if not. As the procedure is similar for any $n$, we will suppose that $n=0$, and $d$ is then equivalent to $[1]^\sharp\vee[a,1]$ for $a\in t\Theta$. The fact that $g$ is algebraic implies that there exists a marked globular sum $c'$ and an integer $k$, such that $d'$ is of shape $[k]^\sharp\vee c'$ and such that $gi$ factors through $c'$. These data verify the desired condition.
\end{proof}

\begin{prop}
\label{prop:criterion to be left cartesian firbation}
Let $p:X\to b^\sharp$ be a morphism exponentiable in $b$. Consider also the following shape of diagram
\begin{equation}
\label{eq:prop:criterion to be left cartesian firbation}
\begin{tikzcd}
	{X''} & {X'} & X \\
	c & {d} & {b^\sharp}
	\arrow["p"', from=1-3, to=2-3]
	\arrow["i"', from=2-1, to=2-2]
	\arrow["j"', from=2-2, to=2-3]
	\arrow["{p'}"', from=1-2, to=2-2]
	\arrow["{p''}"', from=1-1, to=2-1]
	\arrow[from=1-2, to=1-3]
	\arrow[from=1-1, to=1-2]
	\arrow["\lrcorner"{anchor=center, pos=0.125}, draw=none, from=1-1, to=2-2]
	\arrow["\lrcorner"{anchor=center, pos=0.125}, draw=none, from=1-2, to=2-3]
\end{tikzcd}
\end{equation}
The following are equivalent.
\begin{enumerate}
\item For any globular morphism $i:[a,1]^\sharp\to b^\sharp$, $i^*p$ is a left cartesian fibration.
\item For any diagram of shape \eqref{eq:prop:criterion to be left cartesian firbation}, if $i$ is  $i_n^\alpha:\Db_n\to (\Db_{n+1})_t$ with $n$ an integer and $\alpha:=+$ if $n$ is even and $-$ if not, and $j$ is  globular, then $p''\to p'$ is a right Gray deformation retract.
\item For any diagram of shape \eqref{eq:prop:criterion to be left cartesian firbation}, if $i$ marked globular and is a finite composition of pushouts of morphism of shape $i_n^\alpha:\Db_n\to (\Db_{n+1})_t$ with $n$ an integer and $\alpha:=+$ if $n$ is even and $-$ if not, and $j$ is  globular, then $p''\to p'$ is a right Gray deformation retract.
\item For any diagram of shape \eqref{eq:prop:criterion to be left cartesian firbation}, if $i$ is in  the set $\F_g$ defined in \ref{defi: of Fg et Ig}, then $p''\to p'$ is a right Gray deformation retract.
\item The morphism $p$ is a left cartesian fibration.
\end{enumerate}
\end{prop}
\begin{proof}
The implication $(1)\Rightarrow (2)$
comes from theorem \ref{theo:other characterisation of left caresian fibration} as morphisms of shape $ i_n^\alpha$ are right Gray deformation retracts according to proposition \ref{prop:when glob inclusion are left Gray deformation}, and as every globular morphism $\Db_{n+1}\to b$ factors through a globular morphism $[a,1]\to b$.

We suppose that the second condition is fulfilled. As right Gray deformation retracts are stable under composition according to proposition \ref{prop:stability by composition }, we can restrict to the case where $i':c\to d$ fits in a cocartesian square
\[\begin{tikzcd}
	{\Db_n^\flat} & c \\
	{ (\Db_{n+1})_t} & d
	\arrow[from=1-2, to=2-2]
	\arrow["{i^\alpha_n}"', from=1-1, to=2-1]
	\arrow[from=2-1, to=2-2]
	\arrow[from=1-1, to=1-2]
	\arrow["\lrcorner"{anchor=center, pos=0.125, rotate=180}, draw=none, from=2-2, to=1-1]
\end{tikzcd}\]
where all morphisms are globular, and where $\alpha$ is $+$ if $n$ is even, and $-$ if not.
Let $p_0$ and $p_1$ be the morphism fitting in cocartesian squares
\[\begin{tikzcd}
	{X_0} & {X_1} & X \\
	{\Db_{n}^\flat} & {(\Db_{n+1})_t} & { b^\sharp}
	\arrow["p"', from=1-3, to=2-3]
	\arrow[from=2-2, to=2-3]
	\arrow["{p_1}"', from=1-2, to=2-2]
	\arrow[from=1-2, to=1-3]
	\arrow["\lrcorner"{anchor=center, pos=0.125}, draw=none, from=1-2, to=2-3]
	\arrow["{ i_n^\alpha}"', from=2-1, to=2-2]
	\arrow["{p_0}"', from=1-1, to=2-1]
	\arrow[from=1-1, to=1-2]
\end{tikzcd}\]
This defines a diagram in the $(\infty,1)$-category of arrows of $\ocatm$:
\[\begin{tikzcd}
	{p_0} & {p_0} & {p''} \\
	{p_1} & {p_0} & {p''}
	\arrow[from=2-2, to=2-1]
	\arrow[from=1-1, to=2-1]
	\arrow[from=1-2, to=1-1]
	\arrow[from=1-2, to=1-3]
	\arrow[from=1-2, to=2-2]
	\arrow[from=2-2, to=2-3]
	\arrow[from=1-3, to=2-3]
\end{tikzcd}\]
By assumption, $p$ is $b$-exponentiable and the proposition \ref{prop:exponantiable stable under pullback} then implies that $p'$ is $d$-exponentiable. The lemma \ref{lemma:unseful} implies that the morphism $p''\to p'$ is the horizontal colimit of the previous diagram. As $p$ fulfills condition $(2)$, $p_0\to p_1$ is a right Gray deformation retract, and proposition \ref{prop:left Gray deformation retract stable under pushout} implies that $p''\to p'$ also is a right Gray deformation retract. This proves $(2)\Rightarrow (3)$.

Suppose now that condition $(3)$ is fulfilled and let $i$ be in $\F_g$. Consider the diagram
\[\begin{tikzcd}
	{c'} & {d'} & {b^\sharp} \\
	c & d
	\arrow["h", from=2-1, to=1-1]
	\arrow["g", from=2-2, to=1-2]
	\arrow["{i'}", from=1-1, to=1-2]
	\arrow["i"', from=2-1, to=2-2]
	\arrow[from=1-2, to=1-3]
	\arrow["f"', from=2-2, to=1-3]
\end{tikzcd}\]
induced by lemma \ref{lemma:technical lemma exponentiability}. 
We denote by $\tilde{p}''$ and $\tilde{p}'$ the morphisms fitting in the following cartesian squares.
\[\begin{tikzcd}
	{\tilde{X}''} & {\tilde{X}'} & X \\
	{ c'} & { d'} & { b^\sharp}
	\arrow["{ i'}"', from=2-1, to=2-2]
	\arrow[from=2-2, to=2-3]
	\arrow["p"', from=1-3, to=2-3]
	\arrow["{\tilde{p}'}"', from=1-2, to=2-2]
	\arrow["{\tilde{p}''}"', from=1-1, to=2-1]
	\arrow[from=1-1, to=1-2]
	\arrow[from=1-2, to=1-3]
	\arrow["\lrcorner"{anchor=center, pos=0.125}, draw=none, from=1-2, to=2-3]
	\arrow["\lrcorner"{anchor=center, pos=0.125}, draw=none, from=1-1, to=2-2]
\end{tikzcd}\]
As $p$ fulfills $(3)$, $\tilde{p}''\to \tilde{p}'$ is a right Gray deformation retract. By construction, 
the square $h\to g$ also is a right Gray deformation retract. 
As $p''$ and $p'$ are respectively the pullback of $\tilde{p}''$ along $h$ and the pullback of $\tilde{p}'$ along $g$, the dual version of \ref{prop:stability under pullback} implies that $p''\to p'$ is a right Gray deformation retract.

The implication $(4)\Rightarrow (5)$ is induced by theorem \ref{theo:other characterisation of left caresian fibration}. Eventually, the implication $(5)\Rightarrow (1)$ is a consequence of the preservation of left cartesian fibration under pullback.
\end{proof}

\begin{cor}
\label{cor:fibration over representable are expenitalbe}
Any left cartesian fibration $p$ over $a^\sharp$ is $a$-exponentiable. 
\end{cor}
\begin{proof}
We define $q:=\colim_{i:\Sp_a^\sharp}i^*p$. This morphism comes with a canonical comparison $q\to p$. According to proposition \ref{prop:how to create exponentiable}, $q$ is $a$-exponentiable.  For any globular morphism $j:[b,1]^\sharp\to a$, we have $j^*q\sim j^*p$ as $j$ is a discrete Conduché functor. In particular, $j^*q$ is a left cartesian fibration and  $q$ then verifies the first condition of proposition \ref{prop:criterion to be left cartesian firbation}. This implies that $q$ is a left cartesian fibration.

As all morphisms $j:1\to a^\sharp$ are marked globular, and so are discrete Conduché functors, 
there are equivalences
$$j^*\colim_{i:\Sp_a^\sharp}i^*p\sim j^*p$$
and the morphism $q\to p$ induces an equivalence on fiber. This morphisms is then an equivalence according to corollary \ref{cor:morphism between is an equivalence when equivalence on fiber}.
\end{proof}

\begin{lemma}
\label{lemma:pulback of Wsat}
Let $f:A\to B^\sharp$ be a left cartesian fibration, $n$ an integer, and consider a diagram of $\tiPsh{\Theta}$ of shape
\[\begin{tikzcd}
	{A''} & {A'} & A \\
	{(\Sigma^nE^{eq})^\flat} & {\Db_n^\flat} & {B^\sharp}
	\arrow["f", from=1-3, to=2-3]
	\arrow["{f'}", from=1-2, to=2-2]
	\arrow["{f''}", from=1-1, to=2-1]
	\arrow["i"', from=2-1, to=2-2]
	\arrow["j", from=1-1, to=1-2]
	\arrow[from=1-2, to=1-3]
	\arrow[from=2-2, to=2-3]
	\arrow["\lrcorner"{anchor=center, pos=0.125}, draw=none, from=1-2, to=2-3]
	\arrow["\lrcorner"{anchor=center, pos=0.125}, draw=none, from=1-1, to=2-2]
\end{tikzcd}\]
Then $j$ is in $\widehat{\Wm}$.
\end{lemma}
\begin{proof}
As $f'$ and $f''$ are left cartesian fibrations, the only marked cell in $A'$ and $A''$ are the identities according to proposition \ref{prop:left fib over flat}. We can then suppose that the left square lies in $\ocat$, and then apply proposition \ref{prop:pulback of Wsat}.
\end{proof}

\begin{lemma}
\label{lemma:pullback along markkin}
Let $b$ be a globular sum, and $n$ an integer. For any cartesian squares in $\iPsh{\Theta}$,
\[\begin{tikzcd}
	{A''} & {A'} & {b^{\sharp_n}} \\
	{B''} & {B'} & {b^{\sharp}}
	\arrow[from=1-3, to=2-3]
	\arrow["{ }", from=1-2, to=2-2]
	\arrow[from=1-1, to=2-1]
	\arrow["i"', from=2-1, to=2-2]
	\arrow["j", from=1-1, to=1-2]
	\arrow[from=1-2, to=1-3]
	\arrow[from=2-2, to=2-3]
	\arrow["\lrcorner"{anchor=center, pos=0.125}, draw=none, from=1-2, to=2-3]
	\arrow["\lrcorner"{anchor=center, pos=0.125}, draw=none, from=1-1, to=2-2]
\end{tikzcd}\]
if $i$ is in $\widehat{\Wm}$, so is $j$.
\end{lemma}
\begin{proof}
As $\tiPsh{\Theta}$ is cartesian closed, one can suppose that $i$ is in $\W$. In this case the diagram can be seen as a diagram in $\Psh{\Theta}$. The proof is an easy verification of all the possible cases.
\end{proof}

\begin{prop}
\label{prop:W stable under pullback}
For any cartesian square of $\tiPsh{\Theta}$,
\[\begin{tikzcd}
	{A''} & {A'} & A \\
	{B''} & {B'} & {B^\sharp}
	\arrow["f", from=1-3, to=2-3]
	\arrow[from=1-2, to=2-2]
	\arrow[from=1-1, to=2-1]
	\arrow["i"', from=2-1, to=2-2]
	\arrow["j", from=1-1, to=1-2]
	\arrow[from=1-2, to=1-3]
	\arrow[from=2-2, to=2-3]
	\arrow["\lrcorner"{anchor=center, pos=0.125}, draw=none, from=1-2, to=2-3]
	\arrow["\lrcorner"{anchor=center, pos=0.125}, draw=none, from=1-1, to=2-2]
\end{tikzcd}\]
where $f$ is a left cartesian fibration, if $i$ is in $\widehat{\Wm}$, so is $j$.
\end{prop}
\begin{proof}
As $\tiPsh{\Theta}$ is cartesian closed, one can suppose that $i$ is in $\W$. Several cases have to be considered.
If $i$ is of shape $(\Sigma^nE^{eq})^\flat\to \Db_n^\flat$, this is lemma \ref{lemma:pulback of Wsat}. 
Suppose now that $i$ is of shape $\Sp_b^{\sharp_n}\to b^{\sharp_n}$. This induces a diagram
\[\begin{tikzcd}
	{A''} && {A'} && A \\
	& {A''''} && {A'''} \\
	{\Sp_b^{\sharp_n}} && {b^{\sharp_n}} && {B^\sharp} \\
	& {\Sp_b^{\sharp}} && {b^{\sharp}}
	\arrow[""{name=0, anchor=center, inner sep=0}, "f", from=1-5, to=3-5]
	\arrow[from=1-3, to=3-3]
	\arrow[from=1-1, to=3-1]
	\arrow["i"'{pos=0.6}, from=3-1, to=3-3]
	\arrow["j", from=1-1, to=1-3]
	\arrow[from=1-3, to=1-5]
	\arrow[from=3-3, to=3-5]
	\arrow["\lrcorner"{anchor=center, pos=0.125}, draw=none, from=1-1, to=3-3]
	\arrow[from=3-1, to=4-2]
	\arrow[from=3-3, to=4-4]
	\arrow["{i'}"', from=4-2, to=4-4]
	\arrow[from=4-4, to=3-5]
	\arrow[from=2-2, to=4-2]
	\arrow["{j'}"{pos=0.4}, from=2-2, to=2-4]
	\arrow[from=2-4, to=4-4]
	\arrow[from=2-4, to=1-5]
	\arrow[from=1-1, to=2-2]
	\arrow[from=1-3, to=2-4]
	\arrow["\lrcorner"{anchor=center, pos=0.125}, draw=none, from=2-4, to=3-5]
	\arrow["\lrcorner"{anchor=center, pos=0.125}, draw=none, from=1-3, to=0]
\end{tikzcd}\]
where all squares are cartesian. Corollary \ref{cor:fibration over representable are expenitalbe} implies that $j'$ is in $\widehat{\W}$, and according to lemma \ref{lemma:pullback along markkin}, so is $j$.
\end{proof}

\begin{definition}
\label{defi:classified}
 A left cartesian fibration $A\to B$ is \wcnotion{classified}{classified left cartesian fibration} if there exists a cocartesian square: 
\[\begin{tikzcd}
	A & {A'} \\
	B & {B^\sharp}
	\arrow[from=1-1, to=1-2]
	\arrow[from=1-2, to=2-2]
	\arrow[from=1-1, to=2-1]
	\arrow[from=2-1, to=2-2]
	\arrow["\lrcorner"{anchor=center, pos=0.125}, draw=none, from=1-1, to=2-2]
\end{tikzcd}\]
\end{definition}

\begin{theorem}
\label{theo:pullback along un marked cartesian fibration}
Let $p:A\to B$ be a classified left cartesian fibration. The functor $p^*:\ocatm_{/B}\to \ocatm_{/A}$ preserves colimits. 
\end{theorem}
\begin{proof}
As $\tPsh{\Theta}$ is locally cartesian closed, it is enough to show that the functor $p^*:\tiPsh{\Theta}_{/B}\to \tiPsh{\Theta}_{/A}$ sends $\Wm$ onto $\widehat{\Wm}$.
As morphisms fulfilling this property are stable under pullback, one can suppose that $p$ is of shape $B\to A^\sharp$, then applies proposition \ref{prop:W stable under pullback}.
\end{proof}

\begin{cor}
\label{cor:fib over a colimit}
Let $B$ be the colimit of a diagram $F:I\to \ocat$, and
 $p:X\to \colim_i B_i$ a left cartesian fibration. The canonical morphism
 $$ \colim_{i:B_i\to B}i^*p\to p$$
 is an equivalence.
\end{cor}
\begin{proof}
This morphism corresponds to the square
\[\begin{tikzcd}
	{\colim_{i:I}p^*B_i} & X \\
	{\colim_{i:I}B_i} & {B^\sharp}
	\arrow[from=2-1, to=2-2]
	\arrow[from=1-1, to=2-1]
	\arrow[from=1-1, to=1-2]
	\arrow["p", from=1-2, to=2-2]
\end{tikzcd}\]
The lower horizontal morphism is an equivalence by hypothesis, and the upper one is an equivalence as $p^*$ preserves colimits.
\end{proof}

\subsection{Colimits of cartesian fibrations}
\label{section:Colimit of left cartesian fibrations}
Through this section, we will identify any marked $\io$-category $C$ with the canonical induced morphism $C\to1$. If $f:X\to Y$ is a morphism, $f\times C$ then corresponds to the canonical morphism $X\times C\to Y$.

\begin{lemma}
\label{lemma: colimit of fib over b}
Let $b$ be a globular sum and $F:I\to \ocatm_{/b^\sharp}$ be a diagram that is pointwise a left cartesian fibration. The induced morphism 
$\colim_IF$ is a left cartesian fibration over $b^\sharp$. 
\end{lemma}
\begin{proof}
We denote $X:I\to \ocatm$ the diagram induced by $F$ by taking the domain.
Remark first that proposition \ref{prop:exponantiable stable under colim} and corollary \ref{cor:fibration over representable are expenitalbe} imply that $\colim_IF$ is $b$-exponentiable. 
Let $n$ be an integer. Suppose given cartesian squares
\[\begin{tikzcd}
	{Y'} & Y & {\colim_IX} \\
	{\Db_n^\flat} & {(\Db_{n+1})_t} & {b^\sharp}
	\arrow["f", from=1-1, to=1-2]
	\arrow[from=1-1, to=2-1]
	\arrow["\lrcorner"{anchor=center, pos=0.125}, draw=none, from=1-1, to=2-2]
	\arrow[from=1-2, to=1-3]
	\arrow[from=1-2, to=2-2]
	\arrow["\lrcorner"{anchor=center, pos=0.125}, draw=none, from=1-2, to=2-3]
	\arrow["{\colim_IF}", from=1-3, to=2-3]
	\arrow["{i_n^\alpha}"', from=2-1, to=2-2]
	\arrow["j"', from=2-2, to=2-3]
\end{tikzcd}\]
where $\alpha$ is $+$ is $n$ is even and $-$ if not and with $j$ globular. According to proposition \ref{prop:criterion to be left cartesian firbation}, we have to show that $f$ is a right Gray deformation retract to conclude. As $F$ is pointwise a left cartesian fibration, proposition \ref{prop:left Gray transfomration stable under pullback along cartesian fibration} implies that for any $i:I$, the morphism $f(i)$ appearing in the cartesian squares:
\[\begin{tikzcd}
	{Y'} & Y & {X(i)} \\
	{\Db_n^\flat} & {(\Db_{n+1})_t} & {b^\sharp}
	\arrow["{F(i)}", from=1-3, to=2-3]
	\arrow["{i_n^\alpha}"', from=2-1, to=2-2]
	\arrow[from=1-2, to=2-2]
	\arrow[from=1-2, to=1-3]
	\arrow["{f(i)}", from=1-1, to=1-2]
	\arrow[from=1-1, to=2-1]
	\arrow["j"', from=2-2, to=2-3]
	\arrow["\lrcorner"{anchor=center, pos=0.125}, draw=none, from=1-1, to=2-2]
	\arrow["\lrcorner"{anchor=center, pos=0.125}, draw=none, from=1-2, to=2-3]
\end{tikzcd}\]
is a right Gray deformation retract, and that the corresponding Gray deformation retract structure is functorial in $i:I$.
 As $j$ and $ji_n^\alpha$ are marked globular, they are discrete Conduché functors, and so exponentiable according to proposition \ref{prop:pullback by conduch marked preserves colimit}. The following canonical morphism
 $$\colim_I f(i)\to f$$
 is then an equivalence. As right Gray deformation retract structures are stable by colimits, this concludes the proof.
\end{proof}

\begin{lemma}
\label{lemma: colimit of fib over b2}
Let $A$ be an $\io$-category and $F:I\to \ocatm_{/A^\sharp}$ be a diagram that is pointwise a left cartesian fibration. Let $i:a^\sharp\to b^\sharp$ be a morphism between globular sums and $j:b^\sharp\to A^\sharp$ any morphism.
The canonical comparison $$\colim_I (ji)^*F\to i^*\colim_I j^*F$$
is an equivalence.
\end{lemma}
\begin{proof}
Lemma \ref{lemma: colimit of fib over b} implies that the two morphisms are left cartesian fibrations. As equivalences between these morphisms are detected on fibers, we can suppose that $a$ is $[0]$. In this case, the morphism $i$ is a discrete Conduché functor, and is then exponentiable according to proposition \ref{prop:pullback by conduch marked preserves colimit}. This directly concludes the proof.
\end{proof}

\begin{theorem}
\label{theo:left cart stable by colimit}
Let $A$ be an $\io$-category and $F:I\to \ocatm_{/A^\sharp}$ be a diagram that is pointwise a left cartesian fibration. The induced morphism 
$\colim_IF$ is a left cartesian fibration over $A^\sharp$.
\end{theorem}
\begin{proof}
Consider the functor $\psi:\Theta_{/A}\to \Arr(\ocatm)$ whose value on $j:b\to A$ is $\colim_I j^*F$.
As $F$ is pointwise a left cartesian fibration, the corollary \ref{cor:fib over a colimit} induces equivalences
$$\colim_{\Theta_{/A}}\psi:= \colim_{j:b\to A}\colim_I j^*F\sim \colim_I \colim_{j:b\to A}j^*F\sim \colim_I F$$

 The functor $\psi$  is cartesian according to lemma \ref{lemma: colimit of fib over b2}, and as $\codom \psi$ as a special colimit (given by $A^\sharp$), so has $\psi$ according to proposition \ref{prop:special colimit marked case}. As seen in remark \ref{rem on cartesian and special colim}, the fact that $\tiPsh{\Theta}$ is cartesian closed  implies that for any $j:b\to A$, the following canonical morphism
$$\colim_I j^* F=: \psi(j)\to j^*\colim_{\Theta_{/A}}\psi\sim j^* \colim_I F$$
is an equivalence. As the left object is a left cartesian fibration according to lemma \ref{lemma: colimit of fib over b}, so is the right one.
As this is true for any $j:b\to A$, the corollary \ref{cor:on the fact that fib are define against representable} implies that $ \colim_I F$ is a left cartesian fibration.
\end{proof}

\begin{cor}
\label{cor:inclusion of lcatt into the slice preserves colimits}
Let $A$ be an $\io$-category. The inclusion $\LCart(A^\sharp)\to \ocatm_{/A^\sharp}$ preserves both colimits and limits.
\end{cor}
\begin{proof}
The preservation of limits is a consequence of the fact that that this inclusion is a right adjoint. The preservation of colimits is a direct consequence of the theorem \ref{theo:left cart stable by colimit}.
\end{proof}

\vspace{1cm}
We now use the last theorem to provide an alternative explicit expression of the left cartesian fibration $\Fb h^0_{[C,1]}$. We obtain this in the theorem \ref{theo:equivalence between slice and join}.

\begin{prop}
\label{prop:appendice version equivalence betwen slice and join strict word}
Let $C$ be an $\zo$-category with an atomic and loop free basis such that $1\costar C$ is strict. The canonical projection $\gamma:1\costar C^\flat \to [C,1]^\sharp$ is a left cartesian fibration.
\end{prop}
\begin{proof}
Let $C$ be such $\zo$-category.
The hypothesis and the proposition \ref{prop:suspension preserves stricte marked case} imply that both the domain and the codomain of $\gamma$ are strict. We can then show the result in $\zocatm$ and use Steiner theory.
By construction, the basis of $1\costar \lambda C$ is given by the graduated set: 
$$(B_{1\costar \lambda C})_n:=
\left\{
\begin{array}{ll}
\{\emptyset \costar c,c\in (B_{C})_0\}\cup \{\emptyset \costar c,c \in (B_C)_0\}&\mbox{if $n=0$}\\
\{1 \costar c,c\in (B_{C})_{n-1}\}\cup \{\emptyset \costar c,c\in (B_C)_n\} &\mbox{if $n>0$}\\
\end{array}\right.
$$
where $B_C$ is the basis of $C$. The derivative is induced by: 
$$\partial (1\costar c):= 1\costar \partial c + (-1)^{|c|}\emptyset\otimes c~~~~~~~~~~\partial(\emptyset\star c):= \emptyset\costar \partial c$$
where we set the convention $\partial c:=0$ if $|c|=0$.
Let $n$ be an integer and $x$ an element of $(1\costar \lambda C)_n$. The induced morphism $\Db_n\to 1\costar C^\flat$ is marked if and only if there is no element of shape $\emptyset\star c$ in the support of $x$.

For an integer $n>0$, we define $s_n: (\Sigma \lambda C)_n\to (1\costar \lambda C)_n$ as the unique group morphism fulfilling $$s_n(\Sigma c):= 1\costar c$$ for $c$ any element of $\lambda C_{n-1}$. Remark that for any non negative integer $n$, and any element $d$ of $(1\costar \lambda C)_n$, $s_n(d)$ is contained in $d$. However, the family of morphism $\{s_n\}_{n\in \Nb}$ does not commute with the derivative. Let $n$ be an integer and $x$ an element of $(1\costar \lambda C)_n$. The induced morphism $\Db_n\to 1\costar C^\flat$ is therefore marked if and only if $x$ is equal to $s_n\gamma_n(x)$.

Eventually,
we recall that $(\Db_n)_t\otimes[1]^\sharp$ is the colimit of the diagram:
\[\begin{tikzcd}
	{(\Db_n)_t\otimes\{0\}\coprod (\Db_n)_t\otimes\{1\}} & {\Db_n^\flat\otimes\{0\}\coprod \Db_n^\flat\otimes\{1\}} & {\tau^ i_ n(\Db_n^\flat\otimes[1]^\sharp)}
	\arrow[from=1-2, to=1-1]
	\arrow[from=1-2, to=1-3]
\end{tikzcd}\]
We then have to show that for any integer $n$, any diagram of shape 
\[\begin{tikzcd}
	{\lambda\Db_n\otimes\{0\}\cup \lambda\partial\Db_n\otimes[1]} & {1\costar \lambda C} \\
	{\lambda\Db_n\otimes[1]} & {\Sigma \lambda C}
	\arrow[from=1-1, to=2-1]
	\arrow["f"', from=2-1, to=2-2]
	\arrow["g", from=1-1, to=1-2]
	\arrow[from=1-2, to=2-2]
\end{tikzcd}\]
with $f(e_n\otimes[1])$ and $f(e^\alpha_k\otimes[1])$ for $\alpha\in\{-,+\}$ and $k<n$ correponding to a marked cell, admits a unique lifting $l$ with the following extra condition: if $n>0$, if $f(e_n\otimes[1])$ is null and if $g(e_n\otimes\{0\})$ corresponds to a marked cell, then $l(e_n\otimes[1])$ is null and $l(e_n\otimes\{1\})$ corresponds to a marked cell.

Suppose first that $n=0$. We set $l_0:\lambda (\Db_0\otimes[1])_0\to (1\costar \lambda C)_0$ as the unique group morphism extending $g_0$ and such that 
$$l_0(e_0\otimes\{1\}):= \partial s_1(f_1(e_0\otimes[1])+ g_0(e_0\otimes\{1\}).$$
We also define $l_1:\lambda (\Db_0\otimes[1])_1\to (1\costar \lambda C)_1$ as the group morphism characterized by: 
$$l_1(e_0\otimes[1]):= s_1(f_1(e_0\otimes[1])).$$
For $k>1$, we set $l_k:\lambda (\Db_0\otimes[1])_k\to (1\costar \lambda C)_k$ as the constant morphism on $0$.
We directly deduce the equality $\partial l= l \partial$.
We then have defined the desired lifting, which is obviously the unique one possible.

Suppose now that $n>0$. We set $l_k:=g_k:\lambda (\Db_n\otimes[1])_k\to (1\costar \lambda C)_k$ for $k<n$ and $l_n:\lambda (\Db_n\otimes[1])_n\to (1\costar \lambda C)_n$ as the unique group morphism extending $g_n$ and such that 
$$l_n(e_n\otimes\{1\}) := (-1)^\alpha \partial s_{n+1}( f(e_n\otimes[1])) - (-1)^\alpha s_{n}( f((\partial e_n)\otimes[1])) + g_n(e_n\otimes\{0\})$$
where $\alpha$ is $+$ if $n$ is even and $-$ if not.
We define $l_{n+1}:\lambda (\Db_n\otimes[1])_{n+1}\to (1\costar \lambda C)_{n+1}$ as the group morphism characterized by: 
$$l_{n+1}(e_n\otimes[1]):= s_{n+1}(f_{n+1}(e_n\otimes[1])).$$
Eventually, for $k>n$, we set $l_k:\lambda (\Db_n\otimes[1])_k\to (1\costar \lambda C)_k$ as the constant morphism on $0$.

For an integer $k<n$ and $\alpha\in\{-,+\}$, as the $(k+1)$-cell corresponding to $g_{k+1}(e_k^\alpha\otimes [1])$ is marked, we have an equality
$$g_{k+1}(e_k^\alpha\otimes [1]) = s_{k+1}f_{k+1}(e_k^\alpha\otimes [1]).$$
This then implies the equalities
$$\begin{array}{rcl}
\partial(l_{n+1}(e_n\otimes[1])) &=& l_{n+1}(\partial (e_n\otimes[1]))\\
\partial(l_{n}(e_n\otimes \{1\}))&=& g_{n-1}(\partial e_n\otimes \{1\})
\end{array}$$
As it was the only non trivial case, we have $l\partial = \partial l.$
We then have defined the desired lifting, which is obviously the unique one possible. Moreover, if we suppose that $f(e_n\otimes[1])$ is null and $g(e_n\otimes\{0\})$ corresponds to a marked cell, this implies that 
$ s_{n+1}( f(e_n\otimes[1])) =0$ and that the $g_n(e_n\otimes\{0\})$ is in the image of $s_n$. The object $f(e_n\otimes[1])$ also is in the image of $s_n$ and so corresponds to a marked cell.
\end{proof}

\begin{lemma}
\label{lemma:equivalence betwen slice and join strict word}
There is a unique morphism $1\costar C^\flat\to [C,1]^\sharp_{0/}$ fitting in a square
\[\begin{tikzcd}
	1 & {[C,1]^\sharp_{0/}} \\
	{1\costar C^\flat} & {[C,1]^\sharp}
	\arrow[from=1-1, to=2-1]
	\arrow[from=1-1, to=1-2]
	\arrow[from=1-2, to=2-2]
	\arrow[from=2-1, to=2-2]
	\arrow[dotted, from=2-1, to=1-2]
\end{tikzcd}\]
This morphism is an equivalence whenever $C$ is a globular sum.
\end{lemma}
\begin{proof}
We have by construction a cocartesian square
\[\begin{tikzcd}
	{C^\flat\otimes\{0\}} & {C^\flat\otimes[1]^\sharp} \\
	1 & {1\costar C^\flat}
	\arrow[from=1-1, to=2-1]
	\arrow[from=1-1, to=1-2]
	\arrow[from=2-1, to=2-2]
	\arrow[from=1-2, to=2-2]
	\arrow["\lrcorner"{anchor=center, pos=0.125, rotate=180}, draw=none, from=2-2, to=1-1]
\end{tikzcd}\]
which implies that $1\to 1\costar C^\flat$ is initial. This directly implies the first assertion.
We now prove the second assertion. We suppose that $C$ is a globular sum $a$. The $\io$-categories $1\costar a$ is strict according to proposition \ref{prop:strict stuff are stable under Gray cone}. Proposition \ref{prop:appendice version equivalence betwen slice and join strict word} states that the canonical morphism $1\costar a^\flat \to [a,1]^\sharp$ is a left Cartesian fibration. As the comparison map is initial by left cancellation, it is an equivalence. 
\end{proof}

\begin{prop}
\label{prop:equivalence betwen slice and join strict word2}
Let $j:C\to D$ be a morphism between $\io$-categories. The following diagram is cartesian
\[\begin{tikzcd}
	{1\costar C^\flat\coprod_{C^\flat}D^\flat} & {[D,1]^\sharp_{0/}} \\
	{[C,1]^\sharp} & {[D,1]^\sharp}
	\arrow[from=1-1, to=1-2]
	\arrow[from=1-1, to=2-1]
	\arrow[from=1-2, to=2-2]
	\arrow["{[j,1]^\sharp}"', from=2-1, to=2-2]
\end{tikzcd}\]
\end{prop}
\begin{proof}
Suppose first that $C$ is a globular form $b$.
The lemma \ref{lemma:equivalence betwen slice and join strict word} implies that the morphism $1\costar b^\flat\to [b,1]^\sharp$ is equivalent to $\Fb h_0^{[b,1]}$.
We then have to check that the canonical morphism 
\begin{equation}
\label{eq:in a technical lemma}
\Fb h_0^{[b,1]}\coprod_{b^\flat}C^\flat\to [j,1]^*\Fb h_0^{[D,1]}
\end{equation}
is an equivalence. According to theorem \ref{theo:left cart stable by colimit}, the two objects are left Cartesian fibrations, and we then have to check that this morphism induces equivalences on fibers. Furthermore, remark that the two morphisms $\{0\}\to [b,1]^\sharp$ and $\{1\}\to [b,1]^\sharp$ are discrete Conduché functors and then exponentiable according to proposition \ref{prop:pullback by conduch marked preserves colimit}. The fibers on $0$ and $1$ of the morphism \eqref{eq:in a technical lemma} then correspond to the equivalences
$$1\coprod_{\emptyset}\emptyset\sim 1~~~~\mbox{ and }~~~~ b\coprod_bD\sim D.$$

Suppose now that $C$ is any $\io$-category. Remark that we have an equivalence 
$$\colim_{b\to C}[b,1]\sim [C,1]$$
where $b$ ranges over globular sums or the empty $\io$-category. The theorem \ref{theo:pullback along un marked cartesian fibration} and the already proved case then induce equivalences
$$[C,1]^\sharp\times_{[D,1]^\sharp}[D,1]^\sharp_{0/}\sim \colim_{i:b\to C}1\costar b^\flat\coprod_{b^\flat}D^\flat \sim 1\costar C^\flat\coprod_{C^\flat}D^\flat$$
over $[C,1]^\sharp$. This concludes the proof.

\end{proof}

\begin{theorem}
\label{theo:equivalence between slice and join}
Let $C$ be an $\io$-category. The left Cartesian fibration 
$$\Fb h^0_{[C,1]}:[C,1]^\sharp_{0/}\to [C,1]^\sharp$$ 
is equivalent to the projection 
$$1\costar C^\flat\to [C,1]^\sharp.$$
\end{theorem}
\begin{proof}
This directly follows from proposition \ref{prop:equivalence betwen slice and join strict word2} applied to the morphism $id:C\to C$.
\end{proof}

\begin{cor}
\label{cor:cor of the past101}
Let $C$ be an $\io$-category. 
We have equivalences 
$$1\costar C\sim [C,1]_{0/}~~~~~ C\star 1\sim [C,1]_{/1}.$$
\end{cor}
\begin{proof}
By forgetting the marking, theorem \ref{theo:equivalence between slice and join} implies that $1\costar C$ is equivalent to $[C,1]_{0/}$. The proof of the second assertion is dual.
\end{proof}

\begin{cor}
\label{cor:cor of the past10}
Let $j:C\to D$ be a morphism between $\io$-categories. The following squares are cartesian:
\[\begin{tikzcd}
	{1\costar C\coprod_CD} & {1\costar D} & {D\coprod_CC\star1} & {D\star 1} \\
	{[C,1]} & {[D,1]} & {[C,1]} & {[D,1]}
	\arrow[from=1-1, to=1-2]
	\arrow[from=1-1, to=2-1]
	\arrow["\lrcorner"{anchor=center, pos=0.125}, draw=none, from=1-1, to=2-2]
	\arrow[from=1-2, to=2-2]
	\arrow[from=1-3, to=1-4]
	\arrow[from=1-3, to=2-3]
	\arrow["\lrcorner"{anchor=center, pos=0.125}, draw=none, from=1-3, to=2-4]
	\arrow[from=1-4, to=2-4]
	\arrow["{[j,1]}"', from=2-1, to=2-2]
	\arrow[from=2-3, to=2-4]
\end{tikzcd}\]
\end{cor}
\begin{proof}
To show the cartesianness of the first square, we apply the functor $(\underline{\phantom{x}})^\natural$ to the cartesian square given in proposition \ref{prop:equivalence betwen slice and join strict word2} and the equivalence given in theorem \ref{theo:equivalence between slice and join}. The proof of the cartesianness of the second square is dual.
\end{proof}

\begin{cor}
\label{cor:cor of the past3}
Let $C$ be an $\io$-category. We denote by $\gamma:C\star 1\to [C,1]$ and $\gamma':1\costar C\to [C,1]$ the two canonical projections. The functors $\gamma^*:\ocat_{/[C,1]}\to \ocat_{/C\star 1}$ and $\gamma^*:\ocat_{/[C,1]}\to \ocat_{/1\costar C}$ preserve colimits. 
\end{cor}
\begin{proof}
We have a cocartesian square
\[\begin{tikzcd}
	{(1\costar C)^\flat} & {1\costar C^\flat} \\
	{[C,1]^\flat} & {[C,1]^\sharp}
	\arrow["{\gamma^\flat}"', from=1-1, to=2-1]
	\arrow[from=1-2, to=2-2]
	\arrow[from=2-1, to=2-2]
	\arrow[from=1-1, to=1-2]
	\arrow["\lrcorner"{anchor=center, pos=0.125}, draw=none, from=1-1, to=2-2]
\end{tikzcd}\]
The theorem \ref{theo:equivalence between slice and join} implies that the right hand morphism is a left cartesian fibration, and $\gamma^\flat$ is then a classified left cartesian fibration. The result is then a direct consequence of theorem \ref{theo:pullback along un marked cartesian fibration}.
The proof of the second assertion is dual.
\end{proof}

\subsection{Smooth and proper morphisms}
\begin{definition}
For a marked $\io$-category $C$, we denote by \textit{$\LCart(C)$} \sym{(lcart@$\LCart(\uvar)$}\sym{(rcart@$\RCart(\uvar)$} (resp. $\RCart(C)$) the full sub $\iun$-category of $\ocatm_{/C}$ whose objects are left cartesian fibrations. By theorem \ref{theo:adjunction between presheaves and local presheaves}, we can equivalently define $\LCart(C)$ as the localization of $\ocatm_{/C}$ along $\widehat{\I_{/C}}$.
\end{definition}

\begin{definition}
 For $E$, $F$ two objects of $\LCart(C)$ corresponding respectively to two left cartesian fibrations
$p:X\to C$ and $q:X\to C$, we denote by \wcnotation{$\Map(E,F)$}{(map@$\Map(\uvar,\uvar)$} the $\io$-category fitting in the cocartesian square:
\[\begin{tikzcd}
	{\Map(E,F)} & {\uHom(X,Y)} \\
	{\{p\}} & {\uHom(X,C)}
	\arrow["{q_!}", from=1-2, to=2-2]
	\arrow[from=2-1, to=2-2]
	\arrow[from=1-1, to=1-2]
	\arrow[from=1-1, to=2-1]
	\arrow["\lrcorner"{anchor=center, pos=0.125}, draw=none, from=1-1, to=2-2]
\end{tikzcd}\]
\end{definition}

 We recall that a left cartesian fibration $X\to C$ is \textit{classified} when there exists a cartesian square: 
\[\begin{tikzcd}
	X & {X'} \\
	C & {C^\sharp}
	\arrow[from=1-1, to=1-2]
	\arrow[from=1-2, to=2-2]
	\arrow[from=1-1, to=2-1]
	\arrow[from=2-1, to=2-2]
	\arrow["\lrcorner"{anchor=center, pos=0.125}, draw=none, from=1-1, to=2-2]
\end{tikzcd}\]

\begin{definition}
We denote by \wcnotation{$\LCartc(C)$}{(lcart@$\LCartc(\uvar)$} the full sub $\iun$-category of $\LCart(C)$ whose objects are classified left cartesian fibrations.
\end{definition}

\begin{construction} Remark that every morphism $f:C\to D$ induces an adjunction
\[\begin{tikzcd}
	{f_!:\ocat_{/C}} & {\ocat_{/D}:f^*}
	\arrow[shift left=2, from=1-1, to=1-2]
	\arrow[shift left=2, from=1-2, to=1-1]
\end{tikzcd}\]
where the left adjoint $f_!$ is the composition and the right one is the pullback.
This induces an adjunction at the level of localized $\iun$-category:
\[\begin{tikzcd}
	{\Lb f_!:\LCart(C)} & {\LCart(D):\Rb f^*=f^*}
	\arrow[shift left=2, from=1-1, to=1-2]
	\arrow[shift left=2, from=1-2, to=1-1]
\end{tikzcd}\]
\end{construction}

\begin{definition} A morphism $f:C\to D$ is \wcnotion{smooth}{smooth morphism} if $f^*:\ocatm_{/D}\to \ocatm_{/C}$ preserves colimits, and for every cartesian square of the form
\begin{equation}
\label{eq:smooth diagram}
\begin{tikzcd}
	{C''} & {C'} & C \\
	{D''} & {D'} & D
	\arrow["{v'}", from=1-1, to=1-2]
	\arrow[from=1-2, to=1-3]
	\arrow["f", from=1-3, to=2-3]
	\arrow[from=2-2, to=2-3]
	\arrow["v"', from=2-1, to=2-2]
	\arrow[from=1-2, to=2-2]
	\arrow[from=1-1, to=2-1]
	\arrow["\lrcorner"{anchor=center, pos=0.125}, draw=none, from=1-1, to=2-2]
	\arrow["\lrcorner"{anchor=center, pos=0.125}, draw=none, from=1-2, to=2-3]
\end{tikzcd}
\end{equation}
if $v$ is inital, so is $v'$.
\end{definition}

\begin{remark}
As seen in example \ref{exe:exe localization},
when $f$ is smooth, the functor $f^*$ admits a left adjoint
\[\begin{tikzcd}
	{f^*:\ocatm_{/D}} & {\ocatm_{/C}:f_*}
	\arrow[""{name=0, anchor=center, inner sep=0}, shift left=2, from=1-2, to=1-1]
	\arrow[""{name=1, anchor=center, inner sep=0}, shift left=2, from=1-1, to=1-2]
	\arrow["\dashv"{anchor=center, rotate=-90}, draw=none, from=1, to=0]
\end{tikzcd}\]
and as $f^*$ preserves initial morphisms, this induces a derived adjunction:
\[\begin{tikzcd}
	{\Lb f^*:\LCart(D)} & {\LCart(C):\Rb f_*}
	\arrow[""{name=0, anchor=center, inner sep=0}, shift left=2, from=1-2, to=1-1]
	\arrow[""{name=1, anchor=center, inner sep=0}, shift left=2, from=1-1, to=1-2]
	\arrow["\dashv"{anchor=center, rotate=-90}, draw=none, from=1, to=0]
\end{tikzcd}\]
where $\Rb f_*$ is just the restriction of $f_*$.
\end{remark}

\begin{prop}
\label{prop:projection are smooth}
Let $I, J$ be two marked $\io$-categories. The projection $I\times J\to I$ is smooth. 
\end{prop}
\begin{proof}
This is a direct consequence of the fact that cartesian product preserves colimits and initial morphisms.
\end{proof}
\begin{prop}
\label{prop:left cartesian fibration are smooth}
Classified right cartesian fibrations are smooth.
\end{prop}
\begin{proof}
The theorem \ref{theo:pullback along un marked cartesian fibration} states that $f^*$ preserves colimits. Suppose given a diagram of shape \eqref{eq:smooth diagram}. As initial morphisms are the smallest cocomplete class containing morphism $I$, and as $f^*$ preserves colimits, one can suppose that $v$ belongs to $I$, and then is a left Gray deformation retract. To conclude, one applies proposition
\ref{prop:left Gray transfomration stable under pullback along cartesian fibration}.
\end{proof}

\begin{definition} A morphism $f:C\to D$ is \wcnotion{proper}{proper morphism} if $f^*:\ocatm_{/D}\to \ocatm_{/C}$ preserves colimits and for every cartesian square of the form
\begin{equation}
\label{eq:proper diagram}
\begin{tikzcd}
	{C''} & {C'} & C \\
	{D''} & {D'} & D
	\arrow["{v'}", from=1-1, to=1-2]
	\arrow[from=1-2, to=1-3]
	\arrow["f", from=1-3, to=2-3]
	\arrow[from=2-2, to=2-3]
	\arrow["v"', from=2-1, to=2-2]
	\arrow[from=1-2, to=2-2]
	\arrow[from=1-1, to=2-1]
	\arrow["\lrcorner"{anchor=center, pos=0.125}, draw=none, from=1-1, to=2-2]
	\arrow["\lrcorner"{anchor=center, pos=0.125}, draw=none, from=1-2, to=2-3]
\end{tikzcd}
\end{equation}
if $v$ is final, so is $v'$.
\end{definition}
\begin{remark}
As $(\uvar)^\circ$ sends initial to final morphisms, a morphism $f$ is then proper if and only if $f^{\circ}$ is smooth. Propositions \ref{prop:projection are smooth} and \ref{prop:left cartesian fibration are smooth} then imply that projections and classified right cartesian fibrations are proper.
\end{remark}

\begin{definition}
\label{defi:of bot}
We denote by $\bot:\ocatm\to \ocat$ the left Kan extension of the functor $t\Theta\to \ocat$ that sends $a^\flat$ on $a$ and $(\Db_{n+1})_t$ on $\Db_n$. Roughly speaking, $\bot$ sends a marked $\io$-category to it's localization by marked cells. By abuse of notation, we also denote\sym{((g3@$\bot$} $\bot: 
\Arr(\ocatm)\to \ocat$, the composite functor 
$$\Arr(\ocatm)\xrightarrow{\dom}\ocatm\xrightarrow{\bot} \ocat.$$
\end{definition}

\begin{lemma}
\label{lemma:bot send initial and final to we}
The  functor $\bot$ preserves colimits and sends initial and final morphisms to equivalences.
\end{lemma}
\begin{proof}
The functor $\bot$ obviously preserves colimits. It is then sufficient to show that it sends $C\otimes[1]^\sharp\to C$ to an equivalence for any marked $\io$-category $C$. As $\bot$ sends marked trivializations to equivalences, this directly follows from Proposition \ref{prop:cotimes 1 to c is a trivialization}.
\end{proof}

\begin{construction}
The previous lemma implies that  for any object $E$ of $\LCart(A)$ and for any morphism $i:A\to B$, we then have a canonical equivalence 
\begin{equation}
\label{eq:bot kill pull}
\bot \Lb i_! E\sim \bot E.
\end{equation}
Let $A$ be an $\io$-category and $a:1\to A^\sharp$ an object of $A$. 
According to proposition \ref{prop:explicit factoryzation}, the factorisation of $a:1\to A^\sharp$ in a final morphism followed by a right cartesian fibration is given by the canonical inclusion $\{a\}\to A^\sharp_{/a}$ and the canonical projection $\pi_a:A^\sharp_{/a}\to A^\sharp$.
Let $E$ be an object of $\LCart(A^\sharp)$ corresponding to a left cartesian fibration $p:X\to A^\sharp$.
We then have a diagram
\[\begin{tikzcd}
	{X_a} & {X_{/a}} & X \\
	{\{a\}} & {A^\sharp_{/a}} & {A^\sharp}
	\arrow["i", from=1-1, to=1-2]
	\arrow[from=1-1, to=2-1]
	\arrow["\lrcorner"{anchor=center, pos=0.125}, draw=none, from=1-1, to=2-2]
	\arrow[from=1-2, to=1-3]
	\arrow[from=1-2, to=2-2]
	\arrow["\lrcorner"{anchor=center, pos=0.125}, draw=none, from=1-2, to=2-3]
	\arrow["p", from=1-3, to=2-3]
	\arrow[from=2-1, to=2-2]
	\arrow["{\pi_a}"', from=2-2, to=2-3]
\end{tikzcd}\]
and the morphism $i$ is final as $p$ is proper. As $\bot$ sends final morphisms to equivalences, we then have an invertible natural transformation: 
\begin{equation}
\label{eq:explicit derived fiber}
\Rb a^*E\sim \bot \Rb a^*E\sim \bot \Rb \pi_a^*E
\end{equation}
\end{construction}

\begin{prop}
\label{prop:fiber preserves colimits}
The functor $\Rb a^*:\LCart(A^\sharp)\to \LCart(1)\sim \ocat$ preserves colimits. 
\end{prop}
\begin{proof}
As $\pi_a$ is a right cartesian fibration, it is smooth and $\Rb \pi_a^*$ then preserves colimits. The functor $\bot$ also preserves them. The result then follows from the equivalence \eqref{eq:explicit derived fiber}.
\end{proof}

\begin{construction}
 Let $E$ be an object of $\ocatm_{/A^\sharp}$ corresponding to a morphism $X\to A^\sharp$. We denote $\tilde{X}\to A^\sharp$ the left cartesian replacement of $E$ (constructed in \ref{cons:of fb for fibration}). We then have a diagram
\[\begin{tikzcd}
	{X_{/a}} & {\tilde{X}_{/a}} & {A^\sharp_{/a}} \\
	X & {\tilde{X}} & {A^\sharp}
	\arrow[from=1-1, to=1-2]
	\arrow[from=1-1, to=2-1]
	\arrow["\lrcorner"{anchor=center, pos=0.125}, draw=none, from=1-1, to=2-2]
	\arrow[from=1-2, to=1-3]
	\arrow[from=1-2, to=2-2]
	\arrow["\lrcorner"{anchor=center, pos=0.125}, draw=none, from=1-2, to=2-3]
	\arrow["{\pi_a}", from=1-3, to=2-3]
	\arrow[from=2-1, to=2-2]
	\arrow["{\Fb E}"', from=2-2, to=2-3]
\end{tikzcd}\]
 As $\pi_a$ is smooth, the canonical morphism 
$X_{/a}\to \tilde{X}_{/a}$ is initial. Combined with \eqref{eq:explicit derived fiber}, this induces an equivalence:
\begin{equation}
\label{eq:explicit derived fiber2}
\Rb a^*(\Fb E)\sim \bot X_{/a}
\end{equation}
\end{construction}
\begin{prop}
\label{prop:quillent theorem A}
For a morphism $X\to A^\sharp$, and an object $a$ of $A$, we denote by $X_{/a}$ the marked $\io$-category fitting in the following cartesian square: 
\[\begin{tikzcd}
	{X_{/a}} & X \\
	{A^\sharp_{/a}} & {A^\sharp}
	\arrow[from=1-1, to=1-2]
	\arrow[from=1-1, to=2-1]
	\arrow["\lrcorner"{anchor=center, pos=0.125}, draw=none, from=1-1, to=2-2]
	\arrow[from=1-2, to=2-2]
	\arrow[from=2-1, to=2-2]
\end{tikzcd}\]
We denote by $\bot:\ocatm\to \ocat$ the functor sending a marked $\io$-category to its localization by marked cells.
\begin{enumerate}
\item Let $E$, $F$ be two elements of $\ocatm_{/A^\sharp}$ corresponding to morphisms $X\to A^\sharp$, $Y\to A^\sharp$, and
 $\phi:E\to F$ a morphism between them. The induced morphism $\Fb\phi:\Fb E\to \Fb F$ is an equivalence if and only if for any object $a$ of $A$, the induced morphism 
$$\bot X_{/a}\to \bot Y_{/a}$$ 
is an equivalence of $\io$-categories.
\item A morphism $X\to A^\sharp$ is initial if and only if for any object $a$ of $A$, $\bot X_{/a}$ is the terminal $\io$-category.
\end{enumerate}
\end{prop}
\begin{proof}
The first assertion is a direct consequence of the equation \eqref{eq:explicit derived fiber2} and of the fact that equivalences between left cartesian fibrations are detected on fibers.

A morphism $p:X\to A$ is initial if and only if $\Fb p$ is equivalent to the identity of $A^\sharp$, and according to the first assertion, if and only if for any object $a$ of $A$, the canonical morphism $\bot X_{/a}\to \bot A^\sharp_{/a}$ is an equivalence. However, the canonical morphism $\{a\}\to A_{/a}^\sharp$ is final, and $\bot A^\sharp_{/a}$ is then the terminal $\io$-category. This concludes the proof of the second assertion.
\end{proof}

\begin{construction}
Suppose given a commutative square of marked $\io$-categories: 
\begin{equation}
\label{eq:BC data}
\begin{tikzcd}
	A & C \\
	{B^\sharp} & {D^\sharp}
	\arrow["j", from=1-1, to=1-2]
	\arrow["u", from=1-2, to=2-2]
	\arrow["v"', from=1-1, to=2-1]
	\arrow["i"', from=2-1, to=2-2]
\end{tikzcd}
\end{equation}
This induces a square 
\begin{equation}
\label{eq:BC lax commutative square}
\begin{tikzcd}
	{\LCartc(C)} & {\LCartc(A)} \\
	{\LCart(D^\sharp)} & {\LCart(B^\sharp)}
	\arrow["{\Rb j^*}", from=1-1, to=1-2]
	\arrow["{\Lb  u_!}"', from=1-1, to=2-1]
	\arrow["{\Lb v_!}", from=1-2, to=2-2]
	\arrow["{\Rb i^*}"', from=2-1, to=2-2]
	\arrow[shorten <=8pt, shorten >=8pt, Rightarrow, from=1-2, to=2-1]
\end{tikzcd}
\end{equation}
that commutes up to a natural transformation 
\begin{equation}
\label{eq:BC nat}
\begin{array}{rcl}
\Lb v_!\circ \Rb j^*&\to &\Lb v_!\circ \Rb j^* \circ \Rb u^* \circ \Lb u_!\\
&\sim & \Lb v_!\circ \Rb v^* \circ \Rb i^* \circ \Lb u_!\\
&\to& \Rb i^* \circ \Lb u_!
\end{array}
\end{equation}
\end{construction}
\begin{definition}
A square \eqref{eq:BC data} verifies the \notion{Beck-Chevaley condition} if this natural transformation \eqref{eq:BC nat} is an equivalence. This square verifies the \notion{weak Beck-Chevaley condition} if the natural transformation once composed with $\bot$ becomes an equivalence.
\end{definition}

\begin{prop}
\label{prop:base change}
If the square \eqref{eq:BC data} is cartesian and $i$ is smooth, then it verifies the Beck-Chevaley condition.
\end{prop}
\begin{proof}
By construction, $\Lb v_!\circ \Rb j^*$ sends an object $E$ of $\LCartc(C)$ onto the fibrant replacement of $ v_!j^* E$. 
As $i$ is smooth, $\Rb i^* \circ \Lb u_!$ sends an object $E$ of $\LCart(C)$ onto the fibrant replacement of $i^*u_! E$. As pullbacks are stable under composition, we have $i^*u_!\sim v_!j^*$.
\end{proof}

\begin{lemma}
\label{lemma:smoth technical 1}
A square \eqref{eq:BC data} where both $j$ and $i$ are final verifies the weak Beck-Chevaley condition.
\end{lemma}
\begin{proof}
As  $\bot$ sends initial and final morphisms to equivalences, for any $E: \LCartc(A)$ and any $F:\LCartc(C)$, we have equivalences
$$\bot \Lb v_! E \sim \bot E~~~\mbox{ and }~~~\bot \Lb v_! F \sim \bot F.$$
Moreover, as classified left cartesian fibrations are proper, for any $G:\LCartc(C)$ and $H:\LCart(D^\sharp)$,  we have equivalences
$$\bot \Rb j^*G \sim \bot G~~~\mbox{ and }~~~\bot \Rb i^* H \sim \bot H.$$
This implies  the result.
\end{proof}

\begin{lemma}
\label{lemma:smoth technical 2}
Suppose given a cartesian square 
\[\begin{tikzcd}
	A & C \\
	{B^\sharp} & {D^\sharp}
	\arrow["j", from=1-1, to=1-2]
	\arrow["u", from=1-2, to=2-2]
	\arrow["v"', from=1-1, to=2-1]
	\arrow["i"', from=2-1, to=2-2]
\end{tikzcd}\]
such that for any object $b$ of $B^\sharp$, the outer square of the induced diagram
\[\begin{tikzcd}
	{A_{b/}} & A & C \\
	{B^\sharp_{/b}} & {B^\sharp} & {D^\sharp}
	\arrow["j", from=1-2, to=1-3]
	\arrow["u", from=1-3, to=2-3]
	\arrow["v"', from=1-2, to=2-2]
	\arrow["i"', from=2-2, to=2-3]
	\arrow["{v'}"', from=1-1, to=2-1]
	\arrow["{\pi_b}"', from=2-1, to=2-2]
	\arrow["{\pi_b'}", from=1-1, to=1-2]
	\arrow["\lrcorner"{anchor=center, pos=0.125}, draw=none, from=1-1, to=2-2]
	\arrow["\lrcorner"{anchor=center, pos=0.125}, draw=none, from=1-2, to=2-3]
\end{tikzcd}\]
verifies the weak Beck Chevaley condition. Then the right hand square verifies the Beck Chevaley condition.
\end{lemma}
\begin{proof}
Let $E$ be an element of $\LCart(C)$. Using the hypothesis, the fact that $\pi_a$ is a right cartesian fibration, and so smooth,, we have a sequence of equivalences: 
$$\begin{array}{rcll}
\bot \Rb \pi_b^*  \Lb v_! \Rb j^*E&\sim &\bot \Lb v'_! \Rb {\pi'_b}^*  \Rb j^*E&(\ref{prop:base change})\\
&\sim & \bot \Rb \pi_b^*  \Rb i  \Lb u_! E&\mbox{(hypothesis)}
\end{array}$$
Using the equivalence \eqref{eq:explicit derived fiber}, this implies that for any element $b$ of $B$, we have an equivalence 
$$ \Rb b^*  \Lb v_! \Rb j^*E\to \Rb b^*  \Rb i  \Lb u_!E$$
which concludes the proof as equivalences between left cartesian fibrations are detected fiberwise.
\end{proof}

\begin{prop}
\label{prop:BC condition}
Let $i:I\to A^\sharp$ and $j:C^\sharp\to D^\sharp$ be two morphisms. The square 
\[\begin{tikzcd}
	{C^\sharp\times I} & {D^\sharp\times I} \\
	{C^\sharp\times A^\sharp} & {D^\sharp\times A^\sharp}
	\arrow[from=1-1, to=2-1]
	\arrow[from=2-1, to=2-2]
	\arrow[from=1-1, to=1-2]
	\arrow[from=1-2, to=2-2]
\end{tikzcd}\]
verifies the Beck-Chevaley condition.
\end{prop}
\begin{proof}
According to lemma \ref{lemma:smoth technical 2}, one has to show that for any pair $(a,c)$ where $a$ is an object of $A^\sharp$ and $c$ of $C^\sharp$, the induced cartesian square
\[\begin{tikzcd}
	{C^\sharp_{c/}\times I_{a/}} & {D^\sharp\times I} \\
	{C^\sharp_{c/}\times A_{a/}^\sharp} & {D^\sharp\times A^\sharp}
	\arrow[from=1-2, to=2-2]
	\arrow[from=1-1, to=2-1]
	\arrow[from=1-1, to=1-2]
	\arrow[from=2-1, to=2-2]
\end{tikzcd}\]
verifies the weak Beck-Chevaley condition. Remark that this square factors as two cartesian squares:
\[\begin{tikzcd}
	{C^\sharp_{c/}\times I_{a/}} & {D^\sharp_{j(c)/}\times I_{a/}} & {D^\sharp\times I} \\
	{C^\sharp_{c/}\times A_{a/}^\sharp} & {D^\sharp_{j(c)/}\times A_{a/}^\sharp} & {D^\sharp\times A^\sharp}
	\arrow[from=1-3, to=2-3]
	\arrow[from=1-1, to=2-1]
	\arrow[from=2-1, to=2-2]
	\arrow[from=2-2, to=2-3]
	\arrow[from=1-1, to=1-2]
	\arrow[from=1-2, to=1-3]
	\arrow[from=1-2, to=2-2]
\end{tikzcd}\]
The two morphisms $\{c\}\to C^\sharp_{c/}$ and $\{c\}\to D^\sharp_{j(c)/}$ are initial, and by stability by left cancellation, so is $C^\sharp_{c/}\to D^\sharp_{j(c)/}$. By stability by cartesian product, the two horizontal morphisms of the left square are initial. Lemma \ref{lemma:smoth technical 1} then implies that the left square verifies the weak Beck-Chevaley condition. According to proposition \ref{prop:base change}, the right square fulfills the Beck-Chevaley condition, and so \textit{a fortiori}, the weak one. The outer square then verified the weak Beck-Chevaley condition, which concludes the proof.
\end{proof}

\begin{construction}
Suppose given a commutative square of marked $\io$-categories:
\begin{equation}
\label{eq:BC data}
\begin{tikzcd}
	A & {C^\sharp} \\
	B & {D^\sharp}
	\arrow["j", from=1-1, to=1-2]
	\arrow["u", from=1-2, to=2-2]
	\arrow["v"', from=1-1, to=2-1]
	\arrow["i"', from=2-1, to=2-2]
\end{tikzcd}
\end{equation}
where  $j$ and $i$ are smooth. This induces a square
\begin{equation}
\label{eq:BC lax commutative square2}
\begin{tikzcd}
	{\LCartc(B)} & {\LCart(D^\sharp)} \\
	{\LCartc(A)} & {\LCart(C^\sharp)}
	\arrow["{\Rb j_*}"', from=2-1, to=2-2]
	\arrow["{\Lb  u^*}", from=1-2, to=2-2]
	\arrow["{\Lb v^*}"', from=1-1, to=2-1]
	\arrow["{\Rb i_*}", from=1-1, to=1-2]
	\arrow[shorten <=8pt, shorten >=8pt, Rightarrow, from=1-2, to=2-1]
\end{tikzcd}
\end{equation}
that commutes up to a natural transformation 
\begin{equation}
\label{eq:BC nat2}
\begin{array}{rcl}
\Lb u^*\circ \Rb i_*&\to & \Rb j_*\circ\Lb j^* \circ \Lb u^*\circ \Rb i_*\\
&\sim &\Rb j_*\circ\Lb v^*\circ\Lb i^*  \circ \Rb i_*\\
&\to &\Rb j_*\circ\Lb v^*
\end{array}
\end{equation}
\end{construction}

\begin{definition}
A square \eqref{eq:BC data} verifies the \notion{opposed Beck-Chevaley condition} if $i$ and $j$ are smooth and  the natural transformation \eqref{eq:BC nat2} is an equivalence.
\end{definition}

\begin{prop}
\label{prop:base change2}
If the square \eqref{eq:BC nat2} is cartesian,  and $i$ and $j$ are smooth, then it verifies the opposed Beck-Chevaley condition.
\end{prop}
\begin{proof}
By adjunction, it is sufficient to show that the induced natural transformation
$$\Lb v_!\circ \Rb j^*\to \Rb i^* \circ \Lb u_!:\LCart(C^\sharp)\to \LCart(B)$$
is an equivalence. By construction, $\Lb v_!\circ \Rb j^*$ sends an object $E$ of $\LCart(C^\sharp)$ onto the fibrant replacement of $ v_!j^* E$. 
As $i$ is smooth, $\Rb i^* \circ \Lb u_!$ sends an object $E$ of $\LCart(C^\sharp)$ onto the fibrant replacement of $i^*u_! E$. As pullbacks are stable under composition, we have $i^*u_!\sim v_!j^*$.
\end{proof}

\begin{prop}
\label{prop:BC condition 2}
Let $i:I\to A^\sharp$ be a smooth morphism and $j:C^\sharp\to D^\sharp$ any morphism. The square
\[\begin{tikzcd}
	{ C^\sharp\times I} & { C^\sharp\times A^\sharp} \\
	{D^\sharp\times I} & { D^\sharp\times A^\sharp}
	\arrow[from=1-1, to=2-1]
	\arrow[from=1-2, to=2-2]
	\arrow[from=1-1, to=1-2]
	\arrow[from=2-1, to=2-2]
\end{tikzcd}\]
verifies the opposed Beck-Chevaley condition.
\end{prop}
\begin{proof}
As $id_{C^\sharp}\times i$ and $id_{D^\sharp}\times i$ are pullbacks of $i$, they are smooth. The result  is then follows from proposition \ref{prop:base change2}.
\end{proof}

\subsection{The $\Wcard$-small $\io$-category of $\V$-small left cartesian fibrations}

 Let $I$ be a marked $\io$-category, and $a$ a globular sum. We recall that the pullback along the canonical projection $\pi_a:I\times a^\flat\to I$ induces an adjunction
\[\begin{tikzcd}
	{{\pi_a}_!:\ocat_{/I\times a^\flat}} & {\ocatm_{/I}:{\pi_a}^*}
	\arrow[""{name=0, anchor=center, inner sep=0}, shift left=2, from=1-1, to=1-2]
	\arrow[""{name=1, anchor=center, inner sep=0}, shift left=2, from=1-2, to=1-1]
	\arrow["\dashv"{anchor=center, rotate=-90}, draw=none, from=0, to=1]
\end{tikzcd}\]
\begin{lemma}
\label{lemma:to show fully faithfullness1}
Let $E$ and $F$ be two objects of $\ocatm_{/I}$ and $\psi:\pi_{[a,1]}^*E\to \pi_{[a,1]}^*F$ an equivalence.
The exists a unique commutative diagram of shape
\[\begin{tikzcd}
	{(\pi_{[a,1]})_!\pi_{[a,1]}^*E} & {(\pi_{[a,1]})_!\pi_{[a,1]}^*F} \\
	E & F
	\arrow["{(\pi_{[a,1]})_!\psi}", from=1-1, to=1-2]
	\arrow["\epsilon", from=1-2, to=2-2]
	\arrow["\epsilon"', from=1-1, to=2-1]
	\arrow["\phi"', dashed, from=2-1, to=2-2]
\end{tikzcd}\]
Moreover, the arrow $\phi$ is an equivalence.
\end{lemma}
\begin{proof}

We denote by $p:X\to I$ and $q:Y\to I$ the morphisms corresponding to $E$ and $F$.
Unfolding the definition, we have to show the existence and uniqueness of lifts in the diagram
\begin{equation}
\label{eq:square to show fully faithfulness}
\begin{tikzcd}
	{X\times [a,1]^\flat} & {Y\times [a,1]^\flat} & Y \\
	X && I
	\arrow["{\dom \psi}", from=1-1, to=1-2]
	\arrow[from=1-1, to=2-1]
	\arrow[from=1-2, to=1-3]
	\arrow["q", from=1-3, to=2-3]
	\arrow["\alpha"{description}, dashed, from=2-1, to=1-3]
	\arrow["p"', from=2-1, to=2-3]
\end{tikzcd}
\end{equation}
and that the unique lift is an equivalence.

Remark that $[a,1]\to 1$ in a trivialization in the sense of definition \ref{defi:trivialization}. This morphism is then an epimorphism according to proposition \ref{prop:inteligent trucatio and a particular colimit}. As $(\uvar)^{\flat}$ and $(\uvar)\times A$ are left adjoints, they preserve epimorphisms, and $X\times [a,1]^\flat\to X$ is then an epimorphism. The dual of proposition \ref{prop:mono 2} then implies that  the space of lifts of the square \eqref{eq:square to show fully faithfulness} is empty or contractible. 

Suppose now that we already know that lifts in squares of shape \eqref{eq:square to show fully faithfulness} always exist. By the unicity of lifts, an inverse to a lift $\alpha$ in this square is given by a lift $\beta$ in the square
\[\begin{tikzcd}
	{Y\times [a,1]^\flat} & {X\times [a,1]^\flat} & X \\
	Y && I
	\arrow["{\dom \phi}", from=1-1, to=1-2]
	\arrow[from=1-1, to=2-1]
	\arrow[from=1-2, to=1-3]
	\arrow["p", from=1-3, to=2-3]
	\arrow["\beta"{description}, from=2-1, to=1-3]
	\arrow["q"', from=2-1, to=2-3]
\end{tikzcd}\]
where $\phi:E\to F$ is an inverse of $\psi$. As a consequence, if lifts exist, they are equivalences.

It then remains to show that there exists a lift. We will denote also $[a,1]$ the canonical morphism $[a,1]\to 1$.
Let $\phi$ be an inverse of $\psi$. We denote $\tilde{\psi}: E\times [a,1]\to F$ and $\tilde{\phi}:F\times [a,1]\to E$ the morphisms induced by the adjunction from $\psi$ and $\phi$. For $\epsilon\in\{0,1\}$, we denote by ${\psi}_{\{\epsilon\}}: E\times \{\epsilon\}\to F$ and ${\phi}_{\{\epsilon\}}:F\times \{\epsilon\}\to E$ the induced morphisms. In particular, $\psi_{\epsilon}$ and $\phi_{\epsilon}$ are inverses of each other.

By construction, we have a commutative diagram
\[\begin{tikzcd}
	{E\times [a,1]^\flat\times [a,1]^\flat} & {F\times[a,1]^\flat} \\
	{E\times[a,1]^\flat} & E
	\arrow["{\tilde{\psi}\times[a,1]^\flat}", from=1-1, to=1-2]
	\arrow["{\tilde{\phi}}", from=1-2, to=2-2]
	\arrow["{X\times \triangledown}", from=2-1, to=1-1]
	\arrow["\pi"', from=2-1, to=2-2]
\end{tikzcd}\]
where $\triangledown$ is the diagonal. This corresponds to a commutative diagram in the $\iun$-category $[n]\mapsto \Hom_{\ocatm_{/I}}(E\times [a,n]^\flat,E)$:
\[\begin{tikzcd}
	{id_E\sim\phi_0\circ \psi_0} & {\phi_0\circ \psi_1} \\
	{\phi_1\circ \psi_0} & {\phi_1\circ \psi_1\sim id_E}
	\arrow["{\phi_0\circ_0\tilde\psi}", from=1-1, to=1-2]
	\arrow["{\tilde\phi\circ_0\psi_0}"', from=1-1, to=2-1]
	\arrow["{id_{id_E}}"{description}, from=1-1, to=2-2]
	\arrow["{\tilde\phi\circ_0\psi_1}", from=1-2, to=2-2]
	\arrow["{\phi_1\circ_0\psi}"', from=2-1, to=2-2]
\end{tikzcd}\]
As $\phi_0\circ\psi_0\sim \text{id}$ and $\phi_1\circ\psi_1\sim \text{id}$, the previous diagram induces two commutative triangles in the $\iun$-category $[n]\mapsto \Hom_{\ocatm_{/I}}(E\times [a,n]^\flat,F)$:
\[\begin{tikzcd}
	{\psi_1} &&& {\psi_0} & {\psi_1} \\
	{\psi_0} & {\psi_1} &&& {\psi_0}
	\arrow["{\psi_1\circ_0\tilde\phi\circ_0\psi_0}"', from=1-1, to=2-1]
	\arrow["{id_{\psi_1}}", from=1-1, to=2-2]
	\arrow["\tilde\psi", from=1-4, to=1-5]
	\arrow["{id_{\psi_0}}"', from=1-4, to=2-5]
	\arrow["{\psi_0\circ_0\tilde\phi\circ_0\psi_1}", from=1-5, to=2-5]
	\arrow["\tilde\psi"', from=2-1, to=2-2]
\end{tikzcd}\]
Viewed as a $1$-cell of $[n]\mapsto \Hom_{\ocatm_{/I}}(E\times [a,n]^\flat,F)$, $\tilde{\psi}$ is then an equivalence. This implies the existence of lifts in the following diagram
\[\begin{tikzcd}[cramped]
	{[a,1]^\flat} & {\Map(E,F)} \\
	1
	\arrow["\tilde\psi", from=1-1, to=1-2]
	\arrow[from=1-1, to=2-1]
	\arrow["\alpha"', dashed, from=2-1, to=1-2]
\end{tikzcd}\]
which corresponds to a lift in the square \eqref{eq:square to show fully faithfulness}.
\end{proof}

\begin{lemma}
\label{lemma:to show fully faithfullness2}
Let $I$ be a marked $\io$-category and $a$ a globular form. 
The canonical morphisms of $\infty$-groupoids:
$$\pi_{[a,1]}^*:\tau_0\ocatm_{/I}\to \tau_0\ocatm_{/I\times [a,1]^\flat}$$
$$\pi_{[a,1]}^*:\tau_0 \Arr(\ocatm_{/I})\to \tau_0\Arr(\ocatm_{/I\times [a,1]^\flat})$$
are fully faithful.
\end{lemma}
\begin{proof}
Let $E$ and $F$ be two objects of $\ocatm_{/I}$. The morphism 
$$\Hom_{\tau_0\ocatm_{/I}}(E,F) \to \Hom_{\tau_0\ocatm_{/I\times[a,1]^\flat}}(\pi_{[a,1]}^*E,\pi_{[a,1]}^*F) $$ has an inverse that sends $\psi:\pi_{[a,1]}^*E\to \pi_{[a,1]}^*F$ onto the morphism $\phi:E\to F$ appearing in the commutative square provided by lemma \ref{lemma:to show fully faithfullness1}.

The second assertion is demonstrated similarly.
\end{proof}

\begin{prop}
\label{prp:to show fully faithfullness3}
Let $I$ be a marked $\io$-category and $a$ a globular form. We denote by $\pi_a:I\times a^\flat\to I$ the canonical projection.
The canonical morphisms of $\infty$-groupoids:
$$\Rb{\pi_a}^*:\tau_0\LCartc(I)\to \tau_0\LCartc(I\times a^\flat)$$
$$\Rb{\pi_a}^*:\tau_0 \Arr(\LCartc(I))\to \tau_0\Arr(\LCartc(I\times a^\flat))$$
are fully faithful.
\end{prop}
\begin{proof}
Let $[\textbf{b},n]:= a$. Considere first the adjunction:
\[\begin{tikzcd}
	{\LCartc(I\times [b_0,1]^\flat)\times_{\LCartc(I)}...\times_{\LCartc(I)}\LCartc(I\times [b_{n-1},1]^\flat)} \\
	{\LCartc(I^\flat\times [\textbf{b},n])}
	\arrow[""{name=0, anchor=center, inner sep=0}, shift left=2, from=2-1, to=1-1]
	\arrow[""{name=1, anchor=center, inner sep=0}, "{\colim_I}", shift left=2, from=1-1, to=2-1]
	\arrow["\dashv"{anchor=center, rotate=-180}, draw=none, from=1, to=0]
\end{tikzcd}\]
The corollary \ref{cor:fib over a colimit} implies that the counit of this adjunction is an equivalence.
This implies that the right adjoint
$$\LCartc(I^\flat\times [\textbf{b},n])\to \LCartc(I\times [b_0,1]^\flat)\times_{\LCartc(I)}...\times_{\LCartc(I)}\LCartc(I\times [b_{n-1},1]^\flat)$$ is fully faithful. 
By right cancellation and using the fact that fully faithful functors are stable by limits, it is sufficient to show that for any $k<n$, 
$$\Rb{\pi_{[b_i,1]}}^*:\tau_0\LCartc(I)\to \tau_0\LCartc(I\times [b_k,1]^\flat)$$
is fully faithful. 
Moreover, for any such $k$, we have a commutative square
\[\begin{tikzcd}
	{\tau_0\LCartc(I)} & {\tau_0\LCartc(I\times [b_k,1]^\flat)} \\
	{\tau_0\ocatm_{/I}} & {\tau_0\ocatm_{/I\times [b_k,1]^\flat}}
	\arrow["{\Rb{\pi_{[b_k,1]}}^*}", from=1-1, to=1-2]
	\arrow[from=1-1, to=2-1]
	\arrow[from=1-2, to=2-2]
	\arrow["{{\pi_{[b_k,1]}}^*}"', from=2-1, to=2-2]
\end{tikzcd}\]
whose vertical morphisms are fully faithful by construction. The results the follows from lemma \ref{lemma:to show fully faithfullness2} by right cancellation.

The second assertion is demonstrated similarly.

\end{proof}

\begin{definition}
For an $\io$-category $A$ and a globular sum $a$, we define $\LCart(A^\sharp;a)$ as the full sub $\iun$-category of $\LCartc(A^\sharp\times a^\flat)$ whose objects are of shape $E\times id_a^\flat$ for $E$ an object of $\LCart(A^\sharp)$. The proposition \ref{prp:to show fully faithfullness3} implies that the canonical morphism 
$$\tau_0\LCart(A^\sharp)\to \tau_0\LCart(A^\sharp;a)$$
is an equivalence of $\infty$-groupoid.
\end{definition}

\begin{definition}
 We define \wcnotation{$\uLCart(A^\sharp)$}{(lcart@$\uLCart(\uvar)$} as the $\Wcard$-small $\io$-category whose value on $[a,n]$ is given by:
$$\uLCart(A^\sharp)([a,n]):=\Hom([n],\LCart(A^\sharp;a)).$$
For a marked $\io$-category $I$ and a globular sum $a$, we define similarly $\LCartc(I;a)$ as the full sub $\iun$-category of $\LCartc(I\times a^\flat)$ whose objects are of shape $E\times id_a^\flat$ for $E$ an object of $\LCartc(I)$. The proposition \ref{prp:to show fully faithfullness3} implies that the canonical morphism 
$$\tau_0\LCartc(I)\to \tau_0\LCartc(I;a)$$
is an equivalence of $\infty$-groupoid. We define \wcnotation{$\uLCartc(I)$}{(lcartc@$\uLCartc(\uvar)$} as the $\Wcard$-small $\io$-category whose value on $[a,n]$ is given by:
$$\uLCartc(I)([a,n]):=\Hom([n],\LCartc(I;a)).$$
These two definitions are compatible as we have an equivalence between $\uLCartc(A^\sharp)$ and $\uLCart(A^\sharp)$.
\end{definition}
\begin{remark}
 Let $E$ and $F$ be two objects of $\uLCartc(I)$, and $a$ a globular sum. Remark that a morphism $[a,1]\to \uLCartc(I)$ corresponds to a morphism $E\times id_a^\flat\to F\times id_a^\flat$, and so to a morphism $X\times a^\flat\to Y$ over $I$ where $X$ and $Y$ are respectively the domain of $E$ and $F$. We then have an equivalence: 
\begin{equation}
\hom_{\uLCart(I)}(E,F)\sim \Map(E,F). 
\end{equation}
This then implies that $\LCartc(I)$ is locally $\V$-small.
\end{remark}

\begin{construction}
Let $i:I\to J$ be a morphism between marked $\io$-category, $a$ a globular sum, and $p$ a classified left cartesian fibration over 
$a^\flat\times J$. Remark that we have a canonical equivalence $$\Rb (i\times id_{a^\flat})^*(p\times id_{a^\flat})\sim (\Rb i^*p)\times id_{a^\flat}$$ natural in $a:\Theta^{op}$. The functor $\Rb (i\times id_{a^\flat})^*$ then restricts to a functor 
$$(i_a)^*:\LCartc(J;a)\to \LCartc(I;a)$$
natural in $a:\Theta^{op}$, and then to a morphism of $\io$-categories:
\begin{equation}
\label{eq:i pullback}
i^*:\uLCartc(J)\to \uLCartc(I)
\end{equation}
\index[notation]{(f5@$f^*:	\uLCartc(J)\to \uLCartc(I)$}
\end{construction}

\begin{construction}
Let $i:I\to A^\sharp$ be a morphism between marked $\io$-categories. We are now willing to construct a morphism $i_!:\uLCartc(I)\to \uLCart(A^\sharp)$ which corresponds to $\Lb i_!:\LCartc(I)\to \LCart(A^\sharp)$ on the maximal sub $\iun$-category.

We denote by $E_0$ and $E_1$ the $\iun$-categories fitting in the cartesian square: 
\[\begin{tikzcd}
	{E_0} & \Theta & {E_1} & \Theta \\
	{\Arr^{fib}(\ocatm)} & \ocatm & {\Arr^{fib}(\ocatm)} & \ocatm
	\arrow["\codom"', from=2-1, to=2-2]
	\arrow["{\psi_0}", from=1-2, to=2-2]
	\arrow[from=1-1, to=2-1]
	\arrow[from=1-1, to=1-2]
	\arrow["\lrcorner"{anchor=center, pos=0.125}, draw=none, from=1-1, to=2-2]
	\arrow["{\psi_1}", from=1-4, to=2-4]
	\arrow[from=1-3, to=2-3]
	\arrow["\codom"', from=2-3, to=2-4]
	\arrow[from=1-3, to=1-4]
	\arrow["\lrcorner"{anchor=center, pos=0.125}, draw=none, from=1-3, to=2-4]
\end{tikzcd}\]
where $\Arr^{fib}(\ocatm)$ is the full sub $\iun$-category of $\Arr(\ocatm)$ whose objects are classified left cartesian fibrations, and where $\psi_0$ and $\psi_1$ send respectively $a$ on $I\times a^\flat$ and $A^\sharp\times a^\flat$. 
The morphism $i$ induces an adjunction
\begin{equation}
\label{eq:adj i pull}
\begin{tikzcd}
	{i_!:E_0} & {E_1:i^*}
	\arrow[""{name=0, anchor=center, inner sep=0}, shift left=2, from=1-1, to=1-2]
	\arrow[""{name=1, anchor=center, inner sep=0}, shift left=2, from=1-2, to=1-1]
	\arrow["\dashv"{anchor=center, rotate=-90}, draw=none, from=0, to=1]
\end{tikzcd}
\end{equation}
where the left adjoint sends a left cartesian fibration $p$ over $I\times a^\flat$ to $\Lb (i\times id_a^\flat)_!p$ and the right adjoint sends a left cartesian fibration $q$ over $A^\sharp\times a^\flat$ to $\Rb (i\times id_a^\flat)^* q$.
\end{construction}

\begin{lemma}
\label{lemma:technical lemma i pull}
Let $p$ be a left cartesian fibration over $I^\sharp$. We have an equivalence $$\Lb (i\times id_{a^\flat})_!(p\times id_{a^\flat})\sim (\Lb i_! p)\times id_{a^\flat}.$$
Let $q$ be a left cartesian fibration over $A^\sharp$. We have an equivalence $$\Rb (i\times id_{a^\flat})^*(q\times id_{a^\flat})\sim (\Rb i^* q)\times id_{a^\flat}.$$
\end{lemma}
\begin{proof}
The first assertion is straightforward as the cartesian product with $a^\flat$ preserves initial morphisms and left cartesian fibrations. The second assertion is obvious.
\end{proof}

\begin{construction}
We define $\tilde{E_0}$ and $\tilde{E_1}$ as the full sub $\iun$-categories of $E_0$ and $E_1$ whose objects are respectively of shape $p\times id_a^\flat$ and $q\times id_a^\flat$ for $p$ and $q$ classified left cartesian fibrations over $I$ and $A^\sharp$.
The last lemma implies that \eqref{eq:adj i pull} restricts to an adjunction
\begin{equation}
\label{eq:adj i pull2}
\begin{tikzcd}
	{i_!:\tilde{E_0}} & {\tilde{E_1}:i^*}
	\arrow[""{name=0, anchor=center, inner sep=0}, shift left=2, from=1-1, to=1-2]
	\arrow[""{name=1, anchor=center, inner sep=0}, shift left=2, from=1-2, to=1-1]
	\arrow["\dashv"{anchor=center, rotate=-90}, draw=none, from=0, to=1]
\end{tikzcd}
\end{equation}
\end{construction}

\begin{lemma}
\label{lemma:technical lemma i pull2} $~$
\begin{enumerate}
\item
Let $q\to q'$ be a morphism in $\tilde{E_0}$ corresponding to a cartesian square. The induced morphism $i_!(q)\to i_!(q')$ also corresponds to a cartesian square. 
\item
Let $q\to q'$ be a morphism in $\tilde{E_1}$ corresponding to a cartesian square. The induced morphism $i^*(q)\to i^*(q')$ also corresponds to a cartesian square. 
\end{enumerate}
\end{lemma}
\begin{proof}
Cartesian morphisms in $\tilde{E_0}$ corresponds to cartesian squares
\[\begin{tikzcd}[cramped]
	{X\times a^{\flat}} & {X\times b^{\flat}} \\
	{I\times a^{\flat}} & {I\times b^{\flat}}
	\arrow[from=1-1, to=1-2]
	\arrow["{p\times id_a^\flat}"', from=1-1, to=2-1]
	\arrow["{p\times id_b^\flat}", from=1-2, to=2-2]
	\arrow[from=2-1, to=2-2]
\end{tikzcd}\]
and cartesian morphisms in $\tilde{E_1}$ corresponds to cartesian squares
\[\begin{tikzcd}[cramped]
	{Y\times a^{\flat}} & {Y\times b^{\flat}} \\
	{A^\sharp\times a^{\flat}} & {A^\sharp\times b^{\flat}}
	\arrow[from=1-1, to=1-2]
	\arrow["{q\times id_a^\flat}"', from=1-1, to=2-1]
	\arrow["{q\times id_b^\flat}", from=1-2, to=2-2]
	\arrow[from=2-1, to=2-2]
\end{tikzcd}\]
The results directly follows from lemma \ref{lemma:technical lemma i pull}.
\end{proof}
\begin{construction}
The canonical projection $\tilde{E_0}\to \Theta$ and $\tilde{E_1}\to \Theta$ are Grothendieck fibrations in $\iun$-categories. The cartesian lifting is given by cartesian squares. Moreover, their Grothendieck deconstructions correspond respectively to 
$a\mapsto \LCartc(I;a)$ and $a\mapsto \LCart(A^\sharp;b)$. As both $i_!$ and $i^*$ preserve cartesian lifting according to lemma \ref{lemma:technical lemma i pull2}, they induce by Grothendieck deconstruction a family of adjunction
\begin{equation}
\label{eq:adj i pull3}
\begin{tikzcd}
	{(i_a)_!:\LCartc(I;a)} & {\LCart(A^\sharp;a):(i_a)^*}
	\arrow[""{name=0, anchor=center, inner sep=0}, shift left=2, from=1-1, to=1-2]
	\arrow[""{name=1, anchor=center, inner sep=0}, shift left=2, from=1-2, to=1-1]
	\arrow["\dashv"{anchor=center, rotate=-90}, draw=none, from=0, to=1]
\end{tikzcd}
\end{equation}
natural in $a:\Theta^{op}$. The family of functors $(i_a)_!$ then induces a morphism of $\io$-category\index[notation]{(f4@$f_{\mbox{$\exclam$}}:\uLCartc(I)\to \uLCart(A^\sharp)$}
\begin{equation}
\label{eq:i pull}
i_!:\uLCartc(I)\to \uLCart(A^\sharp)
\end{equation}
which corresponds to $\Lb i_!:\LCartc(I)\to \LCart(A^\sharp)$ on the maximal sub $\iun$-category.
The unit and counit of adjunction \eqref{eq:adj i pull3} induce morphisms
\begin{equation}
\label{eq:i pull unit an counit}
\mu:id\to i^*i_!~~~~ \epsilon:i_!i^*\to id
\end{equation}
and equivalences
$(\epsilon\circ_0 i_!)\circ_1(i_!\circ_0 \mu) \sim id_{i_!}$ and $(i^*\circ_0 \epsilon)\circ_1 (\mu \circ_0 i^* )\sim id_{i^*}$.
\end{construction}

\begin{construction} Let $j:C^\sharp\to D^\sharp$ be a morphism between $\io$-categories. We claim that the commutative square 
\[\begin{tikzcd}
	{\uLCart(D^\sharp\times A^\sharp)} & {\uLCartc(D^\sharp\times I)} \\
	{\uLCart(C^\sharp\times A^\sharp)} & {\uLCartc(C^\sharp\times I)}
	\arrow["{(j\times id_{I})^*}", from=1-2, to=2-2]
	\arrow["{( id_{D^\sharp}\times i)^*}", from=1-1, to=1-2]
	\arrow["{( id_{C^\sharp}\times i)^*}"', from=2-1, to=2-2]
	\arrow["{(j\times id_{A^\sharp})^*}"', from=1-1, to=2-1]
\end{tikzcd}\]
induces a commutative square
\begin{equation}
\label{eq:commutative pull push}
\begin{tikzcd}
	{\uLCartc(D^\sharp\times I)} & {\uLCartc(C^\sharp\times I)} \\
	{\uLCart(D^\sharp\times A^\sharp)} & {\uLCart(C^\sharp\times A^\sharp)}
	\arrow["{( id_{D^\sharp}\times i)_!}"', from=1-1, to=2-1]
	\arrow["{(j\times id_{I})^*}", from=1-1, to=1-2]
	\arrow["{( id_{C^\sharp}\times i)_!}", from=1-2, to=2-2]
	\arrow["{(j\times id_{A^\sharp})^*}"', from=2-1, to=2-2]
\end{tikzcd}
\end{equation}
\textit{A priori}, the natural transformations \eqref{eq:i pull unit an counit} implies that this square commutes up the natural transformation:
$$
\begin{array}{rcl}
( id_{C^\sharp}\times i)_!\circ (j\times id_{I})^*&\to &( id_{C^\sharp}\times i)_! \circ (j\times id_{I})^* \circ ( id_{D^\sharp}\times i)^*\circ ( id_{D^\sharp}\times i)_!\\
&\sim &( id_{C^\sharp}\times i)_!\circ ( id_{C^\sharp}\times i)^*\circ (j\times id_{A^\sharp})^*\circ ( id_{D^\sharp}\times i)_!\\
&\to&(j\times id_{A^\sharp})^*\circ ( id_{D^\sharp}\times i)_!
\end{array}
$$
Proposition \ref{prop:BC condition} implies that this natural transformation is pointwise an equivalence, and so is globally an equivalence.
\end{construction}

\begin{construction} We now suppose that the morphism $i:I\to A^\sharp$ is smooth, and we are willing to construct a morphism $i_*:\uLCart(A^\sharp)\to \uLCart(I)$ which corresponds to $\Rb i_*:\LCartc(I)\to \LCart(A^\sharp)$ on the sub maximal $\iun$-categories.. 

As smooth morphisms are stable by pullback, the maps $i\times id_b^\flat$ are smooth for any $b:\Theta$. The morphism $i^*:E_0\to E_1$ then preserves colimits and fits into an adjunction
\begin{equation}
\label{eq:adj i pullstar}
\begin{tikzcd}
	{i^*:E_1} & {E_0:i_*}
	\arrow[""{name=0, anchor=center, inner sep=0}, shift left=2, from=1-1, to=1-2]
	\arrow[""{name=1, anchor=center, inner sep=0}, shift left=2, from=1-2, to=1-1]
	\arrow["\dashv"{anchor=center, rotate=-90}, draw=none, from=0, to=1]
\end{tikzcd}
\end{equation}
where the left adjoint sends a left cartesian fibration $p$ over $A^\sharp\times a^\flat$ to $ (i\times id_a^\flat)^*p$ and the right adjoint sends a left cartesian fibration $q$ over $I\times a^\flat$ to $\Rb (i\times id_a^\flat)_* q$.
\end{construction}
\begin{lemma}
\label{lemma:technical lemma i pullstar}
Let $p$ be a left cartesian fibration over $I$. We have an equivalence $$\Rb (i\times id_{a^\flat})_*(p\times id_{a^\flat})\sim (\Rb i_* p)\times id_{a^\flat}.$$
\end{lemma}
\begin{proof}
This is a straightforward calculation.
\end{proof}

\begin{construction}
We recall that $\tilde{E_0}$ and $\tilde{E_1}$ are defined as the full sub $\iun$-categories of $E_0$ and $E_1$ whose objects are respectively of shape $p\times id_a^\flat$ and $q\times id_a^\flat$ for $p$ and $q$ classified left cartesian fibrations over $I$ and $A^\sharp$.
The lemma \ref{lemma:technical lemma i pullstar} and the second assertion of lemma \ref{lemma:technical lemma i pull} imply that the adjunction \eqref{eq:adj i pullstar} restricts to an adjunction
\begin{equation}
\label{eq:adj i pull2star}
\begin{tikzcd}
	{i^*:\tilde{E_1}} & {\tilde{E_0}:i_*}
	\arrow[""{name=0, anchor=center, inner sep=0}, shift left=2, from=1-1, to=1-2]
	\arrow[""{name=1, anchor=center, inner sep=0}, shift left=2, from=1-2, to=1-1]
	\arrow["\dashv"{anchor=center, rotate=-90}, draw=none, from=0, to=1]
\end{tikzcd}
\end{equation}
\end{construction}

\begin{lemma}
\label{lemma:technical lemma i pull2star} 
Let $q\to q'$ be a morphism in $\tilde{E_0}$ corresponding to a cartesian square. The induced morphism $i_*(q)\to i_*(q')$ also corresponds to a cartesian square. 
\end{lemma}
\begin{proof}
The proof is similar to that of the lemma \ref{lemma:technical lemma i pull2}, using lemma \ref{lemma:technical lemma i pullstar} instead of  lemma \ref{lemma:technical lemma i pull}.
\end{proof}

\begin{construction}
The lemmas \ref{lemma:technical lemma i pull2} and \ref{lemma:technical lemma i pull2star} imply that the two adjoints of \eqref{eq:adj i pull2star} preserve the cartesian cells of the Grothendieck fibrations $\tilde{E_0}\to \Theta$ and $\tilde{E_1}\to \Theta$. These two adjoints then induce by
 Grothendieck deconstruction a family of adjunction
\begin{equation}
\label{eq:adj i pull3star}
\begin{tikzcd}
	{(i_a)^*:\LCart(A^\sharp;a)} & {\LCartc(I;a):(i_a)_*}
	\arrow[""{name=0, anchor=center, inner sep=0}, shift left=2, from=1-1, to=1-2]
	\arrow[""{name=1, anchor=center, inner sep=0}, shift left=2, from=1-2, to=1-1]
	\arrow["\dashv"{anchor=center, rotate=-90}, draw=none, from=0, to=1]
\end{tikzcd}
\end{equation}
natural in $a:\Theta^{op}$. The family of functors $(i_a)_*$ then induces a morphism of $\io$-categories\index[notation]{(f6@$f_*:\uLCartc(I)\to \uLCart(A^\sharp)$}
\begin{equation}
\label{eq:i push op}
i_*:\uLCartc(I)\to \uLCart(A^\sharp)
\end{equation}
which is equivalent to $\Rb i_*:\LCartc(I)\to \LCart(A^\sharp)$ on the sub maximal $\iun$-categories.
The unit and counit of adjunction \eqref{eq:adj i pull3star} induce natural transformation
\begin{equation}
\label{eq:i pull unit an counit op}
\mu: id\to i_*i^*~~~~ \epsilon:i^*i_*\to id
\end{equation}
and equivalences
$(\epsilon\circ_0 i^*)\circ_1(i^*\circ_0 \mu) \sim id_{i^*}$ and $(i_*\circ_0 \epsilon)\circ_1 (\mu \circ_0 i_* )\sim id_{i_*}$.
\end{construction}

\begin{construction} Let $j:C^\sharp\to D^\sharp$ be a morphism between $\io$-categories. We claim that the commutative square 
\[\begin{tikzcd}
	{\uLCart(D^\sharp\times A^\sharp)} & {\uLCartc(D^\sharp\times I)} \\
	{\uLCart(C^\sharp\times A^\sharp)} & {\uLCartc(C^\sharp\times I)}
	\arrow["{(j\times id_{I})^*}", from=1-2, to=2-2]
	\arrow["{( id_{D^\sharp}\times i)^*}", from=1-1, to=1-2]
	\arrow["{( id_{C^\sharp}\times i)^*}"', from=2-1, to=2-2]
	\arrow["{(j\times id_{A^\sharp})^*}"', from=1-1, to=2-1]
\end{tikzcd}\]
induces a commutative square
\begin{equation}
\label{eq:commutative pull push op}
\begin{tikzcd}
	{\uLCartc(D^\sharp\times I)} & {\uLCart(D^\sharp\times A^\sharp)} \\
	{\uLCartc(C^\sharp\times I)} & {\uLCart(C^\sharp\times A^\sharp)}
	\arrow["{( id_{D^\sharp}\times i)_*}", from=1-1, to=1-2]
	\arrow["{(j\times id_{I})^*}"', from=1-1, to=2-1]
	\arrow["{( id_{C^\sharp}\times i)_*}"', from=2-1, to=2-2]
	\arrow["{(j\times id_{A^\sharp})^*}", from=1-2, to=2-2]
\end{tikzcd}
\end{equation}

\textit{A priori}, the natural transformations \eqref{eq:i pull unit an counit op} implies that this square commutes up the natural transformation:
$$
\begin{array}{rcl}
(j\times id_{A^\sharp})^*\circ (id_{D^\sharp}\times i)_*&\to & (id_{C^\sharp}\times i)_*\circ(id_{C^\sharp}\times i)^*\circ(j\times id_{A^\sharp})^*\circ (id_{D^\sharp}\times i)_*\\
&\sim &(id_{C^\sharp}\times i)_*\circ(j\times id_I)^*\circ(id_{D^\sharp}\times i)^*\circ (id_{D^\sharp}\times i)_*\\
&\to &(id_{C^\sharp}\times i)_*\circ(j\times id_I)^*
\end{array}
$$
Proposition \ref{prop:BC condition 2} implies that this natural transformation is pointwise an equivalence, and so is globally an equivalence.
\end{construction}

%
%

\chapter{The $\io$-category of $\io$-categories}	

\minitoc
\vspace{1cm}
%
%
%
%
%
%
%
%
%
%
%

This chapter aims to establish generalizations of the fundamental categorical constructions to the $\io$ case. In this new theory, the Gray product plays an essential role. Firstly, it allows the definition of the notion of \textit{lax transformation}:

\begin{idefi}
Let $f, g: A \to B$ be two morphisms between $\io$-categories. A \textit{lax transformation} between $f$ and $g$ is given by a morphism
\[ \psi: A \otimes [1] \to B, \]
whose restriction to $A \otimes \{0\} \sim A$ is equivalent to $f$, and whose restriction to $A \otimes \{1\} \sim A$ is equivalent to $g$.

We can then show that a lax transformation corresponds to the following data:
\begin{enumerate}
\item[$-$] for every object $a$ in $A$, a morphism $f(a) \to g(a)$,
\item[$-$] for every $1$-cell $a \to b$, a $2$-cell in $B$ fitting into the following diagram:
\[
\begin{tikzcd}
	{f(a)} & {g(a)} \\
	{f(b)} & {g(b)}
	\arrow[from=2-1, to=2-2]
	\arrow[from=1-2, to=2-2]
	\arrow[from=1-1, to=2-1]
	\arrow[from=1-1, to=1-2]
	\arrow[shorten <=6pt, shorten >=6pt, Rightarrow, from=1-2, to=2-1]
\end{tikzcd}
\]
\item[$-$] for every $n$-cell in $A$, an $(n+1)$-cell in $B$ fitting into a more complex version of the above diagram,
\item[$-$] multiple coherences that all these cells satisfy.
\end{enumerate} 
The usefulness of the Gray product lies in the fact that it compactly encodes all these data and coherences.

\end{idefi}

The notion of lax transformation allows us to state the Grothendieck Lax construction.
\begin{itheorem}[\ref{theo:lcartc et ghom}]
\label{theo:lax gr}
Let $\uni$ be the $\io$-category of small $\io$-categories, and $A$ an $\io$-category. There is an equivalence
\[ \int_A: \gHom(A, \uni) \sim \uLCart^c(A) \]
where $\gHom(A, \uni)$ is the $\io$-category of morphisms from $A$ to $\uni$, with $1$-cells being the lax transformations $A \otimes [1] \to \uni$, and $\uLCart^c(A)$ is the $\io$-category of left cartesian fibrations over $A$, with $1$-cells being morphisms that do not necessarily preserve cartesian liftings.
\end{itheorem}

We also obtain a very precise construction of the functor $\int_A$. Given a functor $f: A \to \uni$, the left cartesian fibration $\int_A f$ is a colimit (calculated in $\ocatm_{/A}$) of a simplicial object whose value at $n$ is of the form
\[ \coprod_{x_0, \ldots, x_n: A_0} X(x_0) \times \hom_A(x_0, \ldots, x_n) \times A_{x_n/} \to A, \]
where $A_{x/}$ is the \textit{lax slice} of $A$ over $x$.
This formula is similar to the one given in \cite{Gepner_Lax_colimits_and_free_fibration} for $\iun$-categories, and to the one given in \cite{Warren_the_strict_omega_groupoid_interpretation_of_type_theory} for strict $\omega$-categories.

The result we provide is actually stronger than previously stated, as it allows us to choose "to what extent" the transformations between functors $A \to \uni$ are lax, which induces "to what extent" the morphisms between the fibrations preserve cartesian liftings. The result we have presented corresponds to the case where the transformations are "totally lax". Applying it to the case where the transformations between functors are "not lax at all" - that is, are natural transformations - we obtain the following corollary:

\begin{icor}[\ref{cor:lcar et hom}]
\label{cor:grd}
Let $A$ be an $\io$-category. There is an equivalence
\[ \uHom(A, \uni) \sim \uLCart(A) \]
where $\uHom(A, \uni)$ is the $\io$-category of morphisms from $A$ to $\uni$, with $1$-cells being the natural transformations $A \times [1] \to \uni$, and $\uLCart(A)$ is the $\io$-category of left cartesian fibrations over $A$, with $1$-cells being morphisms that preserve cartesian liftings.
\end{icor}

In the $(\infty, n)$-categorical case, the equivalence between  $\uHom(A, \uni)$ and $\uLCart(A)$ given in Corollary \ref{cor:lcar et hom} was already proven by Nuiten in \cite{Nuiten_on_straightening_for_segal_spaces} and by Rasekh in \cite{Rasekh_yoneda_lemma_for_simplicial_spaces}. In the $(\infty, 1)$-case, the equivalence  between  $\gHom(A, \uni)$ and $\uLCartc(A)$ was already proven by Haugseng-Hebestreit-Linskens-Nuiten in \cite{haugseng2023lax}.

\vspace{1cm}

Given a locally small $\io$-category $C$, we construct the Yoneda embedding $y: C \to \widehat{C}$ where $\widehat{C} := \uHom(C^t, \uni)$. We then prove the Yoneda lemma:

\begin{itheorem}[\ref{theo:Yoneda lemma}]
Let $C$ be a locally $\U$-small $\io$-category. 
The Yoneda embedding $C \to \widehat{C}$ is fully faithful. Furthermore,
there is an equivalence between the functor
\[ \hom_{\w{C}}(y_{\uvar}, \uvar): C^t \times \w{C} \to \uni \]
and the functor 
\[ \ev: C^t \times \w{C} \to \uni. \]

Given an object $c$ of $C$, the induced equivalence on fibers:
\[ \hom_{\widehat{C}}(y_c, y_c) \sim \hom_C(c, c) \]
sends $\{id_{y_c}\}$ to $\{id_c\}$.
\end{itheorem}

Having the Yoneda lemma at our disposal with all the correct functorialities is an extremely powerful tool for developing the theory of $\io$-categories, as it encodes many and complex coherences.

In the $(\infty,n)$-categorical context, the  equivalence, non-functorial in $c$, between the functors $\hom(y_c, \uvar): \w{C} \to \uni$ and $\ev(c, \uvar): \w{C} \to \uni$ for any object $c$ of $C$ is demonstrated in \cite{Rasekh_yoneda_lemma_for_simplicial_spaces}, \cite{Hinich_colimit_in_enriched_infini_categories}, and \cite{Heine_an_equivalence_between_enricherd_infini_categorories_and_categories_with_weak_action}.

\vspace{1cm}

We then define the notion of a \textit{lax colimit}:
\begin{idefi}
Let $f: I \to C$ be a morphism between $\io$-categories, and $tI$ a marking on $I$.

A \textit{lax colimit} for $f$ relative to $(tI_n)_{n>0}$ is the universal data of
\begin{enumerate}
\item[$-$] an object $\laxcolim_I F$,
\item[$-$] for every $1$-cell $i: a \to b$ in $I$, a diagram
\[\begin{tikzcd}
	{} & {F(b)} \\
	{F(a)} & {\laxcolim_I F}
	\arrow["{F(i)}", curve={height=-30pt}, from=2-1, to=1-2]
	\arrow[from=2-1, to=2-2]
	\arrow[shorten <=8pt, shorten >=8pt, Rightarrow, from=1-2, to=2-1]
	\arrow[draw=none, from=1-1, to=2-1]
	\arrow[from=1-2, to=2-2]
\end{tikzcd}\]
where the $2$-cell present is an equivalence if $i$ is in $tI_1$.
\item[$-$] for every $2$-cell $u: i \to j$, a diagram
\[\begin{tikzcd}
	& {F(b)} & {} & {F(b)} \\
	{F(a)} & {\laxcolim_I F} & {F(a)} & {\laxcolim_I F}
	\arrow[""{name=0, anchor=center, inner sep=0}, "{F(i)}"{description}, from=2-1, to=1-2]
	\arrow[""{name=1, anchor=center, inner sep=0}, from=2-1, to=2-2]
	\arrow[from=1-2, to=2-2]
	\arrow[""{name=2, anchor=center, inner sep=0}, from=1-2, to=2-2]
	\arrow[""{name=3, anchor=center, inner sep=0}, "{F(j)}", curve={height=-30pt}, from=2-1, to=1-2]
	\arrow["{F(j)}", curve={height=-30pt}, from=2-3, to=1-4]
	\arrow[from=2-3, to=2-4]
	\arrow[from=1-4, to=2-4]
	\arrow[shorten <=8pt, shorten >=8pt, Rightarrow, from=1-4, to=2-3]
	\arrow[""{name=4, anchor=center, inner sep=0}, draw=none, from=1-3, to=2-3]
	\arrow[shift right=2, shorten <=12pt, shorten >=12pt, Rightarrow, from=2, to=1]
	\arrow[shorten <=4pt, shorten >=4pt, Rightarrow, from=3, to=0]
	\arrow[shift left=0.7, shorten <=14pt, shorten >=16pt, no head, from=2, to=4]
	\arrow[shorten <=14pt, shorten >=14pt, from=2, to=4]
	\arrow[shift right=0.7, shorten <=14pt, shorten >=16pt, no head, from=2, to=4]
\end{tikzcd}\]
where the $3$-cell present is an equivalence if $u$ is in $tI_2$,
\item[$-$] etc.
\end{enumerate}
Varying $tI$ thus allows us to adjust the "laxness" of the universal property that this object must satisfy.
\end{idefi}

We conclude this chapter by establishing generalizations of the standard results in category theory to the $\io$-categorical case, including the characterization of presheaves as completion by lax colimits, the calculation of lax Kan extensions using slices, the relationships between duality and (co)lax limits, as well as various characterizations and properties of adjunctions.

\paragraph{Cardinality hypothesis.}
We fix during this chapter three Grothendieck universes $\U \in \V\in\Wcard$, such that $\omega\in \U$. 
All defined notions depend on a choice of cardinality. When nothing is specified, this corresponds to the implicit choice of the cardinality $\V$.
We denote by $\Set$ the $\Wcard$-small $1$-category of $\V$-small sets, $\igrd$ the $\Wcard$-small $\iun$-category of $\V$-small $\infty$-groupoids and $\icat$ the $\Wcard$-small $\iun$-category of $\V$-small $\iun$-categories. 

\section{Lax Grothendieck construction}
\label{section:Univalence}
\subsection{Internal category}
\begin{definition} For $X$ an object of $\iPsh{\Theta}$ and $K$ a simplicial $\infty$-groupoid, we define the simplicial object $\langle X, K\rangle$ of $\ocat$ whose value on $n$ is given by \index[notation]{((g20@$\langle a,n\rangle$}
$$\langle X,K\rangle_n := X\times K_n$$
If $K$ is the representable $[n]$, this object is simply denoted by $\langle X,n\rangle$.
We also define the following set of morphism of $\iPsh{\Delta\times \Theta}$:\sym{(t@$\T$}
$$\T:= \{\langle a,f\rangle,~ a\in \Theta, f\in \mbox{$\W_1$}\} \cup \{\langle g,n\rangle,~ g\in \W, [n]\in \Delta\}$$
\end{definition}
\begin{definition}  A \wcnotion{$\ioun$-category}{category5@$\ioun$-category} is a $\T$-local $\infty$-presheaf $C\in \iPsh{\Theta\times \Delta}$. We then naturally define \sym{((a80@$\ouncat$}
$$\ouncat := \iPsh{\Theta\times \Delta}_{\T}.$$
\end{definition}

\begin{remark}
Unfolding the definition, an $\ioun$-category is a simplicial object $C:\Delta^{op}\to \ocat$
such that the induced morphisms
$$C_0\to\lim_{[k]\to E^{eq}}C_k~~~~\mbox{ and }~~~C_n\to C_1\times_{C_0}\times...\times_{C_0}C_1~n\in \Nb$$
are equivalences. 
Remark that we have a cartesian square
\[\begin{tikzcd}
	\ouncat & {\Fun(\Theta^{op},\icat)} \\
	{\ocat\times \ocat} & {\Fun(\Theta^{op},\igrd)\times \Fun(\Theta^{op},\igrd)}
	\arrow[from=1-2, to=2-2]
	\arrow[""{name=0, anchor=center, inner sep=0}, from=2-1, to=2-2]
	\arrow[from=1-1, to=2-1]
	\arrow[from=1-1, to=1-2]
	\arrow["\lrcorner"{anchor=center, pos=0.125}, draw=none, from=1-1, to=0]
\end{tikzcd}\]
where the lower horizontal morphism is induced by the canonical inclusion of $\io$-category onto $\infty$-presheaves on $\Theta$, and the right vertical one is induced by the functor that maps an $\iun$-category to the pair consisting of the $\infty$-groupoid of objects and the $\infty$-groupoid of arrows.
\end{remark}

\begin{definition} 
A morphism $p:X\to A$ between two $\infty$-presheaves on $\Theta\times \Delta$ is a \notion{left fibration} if it has the unique right lifting property against the set of morphism \sym{(j@$\J$}
$$\J:=\{\langle a,\{0\}\rangle \to \langle a,n\rangle~,a\in\Theta, [n]\in\Delta\}\cup \{\langle g,0\rangle,~ g\in\W\}$$
\end{definition} 

\begin{remark}
Unfolding the notation, this is equivalent to asking that $X_0\to A_0$ is $\W$-local, and that the natural square 
\[\begin{tikzcd}
	{X_n} & {X_{\{0\}}} \\
	{A_n} & {A_{\{0\}}}
	\arrow[from=1-1, to=1-2]
	\arrow[from=1-1, to=2-1]
	\arrow[from=2-1, to=2-2]
	\arrow[from=1-2, to=2-2]
\end{tikzcd}\]
is cartesian. 
\end{remark}

\begin{prop}
\label{prop:if left fib the fib}
We have an inclusion $T\subset \widehat{J}$.
\end{prop}
\begin{proof}
Let $a$ be an object of $\Theta$.
The  $\infty$-groupoid of morphisms $i$ of $\iPsh{\Delta}$ such that $\langle a,i\rangle$ is in $\widehat{J}$ contains by definition $\{0\}\to [n]$, and is closed by colimits and left cancelation. This $\infty$-groupoid then contains all initial morphism between $\infty$-presheaves on $\Delta$. As morphisms of $\W_1$ are initial, $\widehat{J}$ includes morphisms of shape $\langle a, f\rangle$ for  $a\in \Theta$ and $f\in \W_1$.

Let $g:a\to b$ be a morphism of $\W$ and $n$ an integer. We have a commutative square
\[\begin{tikzcd}
	{\langle a,\{0\}\rangle} & {\langle a,n\rangle} \\
	{\langle b,\{0\}\rangle} & {\langle b,n\rangle}
	\arrow["{\langle g,\{0\}\rangle}"', from=1-1, to=2-1]
	\arrow["{\langle g,n\rangle}", from=1-2, to=2-2]
	\arrow[from=1-1, to=1-2]
	\arrow[from=2-1, to=2-2]
\end{tikzcd}\]
The two horizontal morphisms are in $\widehat{J}$. By left cancellation, this implies that  $\langle g,n\rangle$ is in $ \widehat{J}$ which concludes the proof.
\end{proof}

\begin{remark}
If $X\to A$ is a left fibration, with $A$ a $\ioun$-category, the last proposition implies that $X$ is also a $\ioun$-category. We denote by \wcnotation{$\Lfib(A)$}{(lfib@$\Lfib(\uvar)$} the full sub $\iun$-category of $\ouncat_{/A}$ whose objects are left fibrations.
\end{remark}

\begin{prop}
\label{prop:lfib and W}
There is a canonical equivalence: 
$$\Lfib(\langle a,C \rangle)\sim \Fun(C,\ocat_{/a})$$
natural in $a:\Theta^{op}$ and $C:\icat^{op}$.
\end{prop}
\begin{proof}
Let $a$ be an object of $\Theta^{op}$ and $C$ an $\iun$-category. We have a canonical equivalence 
$$\iPsh{\Theta\times \Delta}_{/\langle a , C\rangle}\sim \iPsh{\Theta_{/a}\times \Delta_{/C}}\sim \Fun(\Theta_{/a}^{op},\iPsh{\Delta}_{/C})$$
The previous equivalence induces an equivalence
$$(\iPsh{\Theta\times \Delta}_{/\langle a,C\rangle})_{\{\langle b,\{0\}\rangle \to \langle b,[n]\rangle\}_{/\langle a , C\rangle}} \sim \Fun(\Theta_{/a}^{op}, (\iPsh{\Delta}_{/C})_{\I^0_{/C}})$$
where $\I^0_{/C}$ corresponds to the $\infty$-groupoid of morphisms of $\iPsh{\Delta}_{/C}$ of shape
\[\begin{tikzcd}
	& {[n]} \\
	{\{0\}} & C
	\arrow[from=2-1, to=2-2]
	\arrow[from=2-1, to=1-2]
	\arrow[from=1-2, to=2-2]
\end{tikzcd}\]
for $n$ any integer.
The $\iun$-category $(\iPsh{\Delta}_{/C})_{\I^0_{/C}}$ is equivalent to the $\iun$-category of Grothendieck $\V$-small opfibrations fibered in $\infty$-groupoid over $C$, which is itself equivalent to $\Fun(C,\igrd)$ according to the Grothendieck construction. 
We then have an equivalence 
\begin{equation}
\label{eq:lfib and W}
(\iPsh{\Theta\times \Delta}_{/\langle a,C\rangle})_{\{\langle b,\{0\}\rangle \to \langle b,[n]\rangle\}_{/\langle a , C\rangle}} \sim \Fun (\Theta_{/a}^{op}, \Fun(C,\igrd))\sim \Fun (C, \iPsh{\Theta}_{/a})
\end{equation}
By definition, $\Lfib(\langle a,C\rangle)$ is the fully faithful sub $\iun$-category of the left hand $\iun$-category corresponding to objects that are local with respect to the set of morphism
 $\{\langle g,0\rangle, g\in \W\}_{/\langle a,C\rangle}.$ 
Such $\infty$-presheaves corresponds via the equivalence \eqref{eq:lfib and W} to functors $C\to \iPsh{\Theta}_{/a}$ that are pointwise $\W_{/a}$-local. As $\W_{/a}$-local $\infty$-presheaves on $\Theta_{/a}$ corresponds to elements of $\ocat_{/a}$, we have an equivalence
$$\Lfib(\langle a,C\rangle)\sim \Fun(C,\ocat_{/a}).$$
\end{proof}

\begin{construction}
A morphism $f:A\to B$ between two $\infty$-presheaves on $\Theta\times \Delta$ induces an adjunction
\begin{equation}
\label{eq:adj between left fibration}
\begin{tikzcd}
	{f_!:\ouncat{/A}} & {\ouncat_{/B}:f^*}
	\arrow[shift left=2, from=1-1, to=1-2]
	\arrow[shift left=2, from=1-2, to=1-1]
\end{tikzcd}
\end{equation}
where $f_!$ is the composition and $f^*$ is the pullback.
As $\Lfib(A)$ is the localization of $\ouncat_{/A}$ along the class of morphisms $\widehat{\J_{/A}}$,
the previous adjunction induces a derived adjunction:
\begin{equation}
\label{eq:derived adj between left fibration}
\begin{tikzcd}
	{\Lb f_!:\Lfib(A)} & {\Lfib(B):\Rb f^*}
	\arrow[shift left=2, from=1-1, to=1-2]
	\arrow[shift left=2, from=1-2, to=1-1]
\end{tikzcd}
\end{equation}
where $\Lb f_!$ sends $E$ onto $\Fb f_!E$ and $\Rb f^*$ is just the restriction of $f^*$ to $\Lfib(B)$.
\end{construction}

\begin{construction}
We denote by $\pi_!:\Fun(\Delta^{op},\iPsh{\Theta})\to \iPsh{\Delta[\Theta]}$ the functor induced by extention by colimits by the canonical morphism $\pi:\Delta\times \Theta\to \Delta[\Theta]$. We also define $\Noiun:\iPsh{\Delta[\Theta]}\to \Fun(\Delta^{op},\iPsh{\Theta})$ as the right adjoint of $\pi_!$. As $\pi_!$ preserves representable, \wcnotation{$\Noiun$}{(noiun@$\Noiun$} preserves colimits. Remark that the image of $T$ by $\pi_!$ is contained in $\widehat{\M}$, and $\Noiun$ induces then by restriction a functor
$$\Noiun:\ocat\to \ouncat.$$
If $C$ is an $\io$-category, $\Noiun C$ corresponds to the simplicial object in $\ocat$:
\[\begin{tikzcd}
	\cdots & {\coprod_{x_0,x_1,x_2:\tau_0C}\hom_C(x_0,x_1,x_2)} & {\coprod_{x_0,x_1:\tau_0C}\hom_C(x_0,x_1)} & {\coprod_{x_0:\tau_0C}1}
	\arrow[from=1-2, to=1-3]
	\arrow[shift left=2, from=1-3, to=1-4]
	\arrow[shift left=4, from=1-2, to=1-3]
	\arrow[shift right=2, from=1-3, to=1-2]
	\arrow[shift right=4, from=1-2, to=1-3]
	\arrow[shift left=2, from=1-3, to=1-2]
	\arrow[from=1-4, to=1-3]
	\arrow[shift right=2, from=1-3, to=1-4]
\end{tikzcd}\]
\end{construction}
\begin{notation}
If $p:X\to \Noiun C$ is a left fibration, and $x$ an object of $C$, we will denote by $X(x)$ the fiber of $p_0:X_0\to \Noiun C$ on $x$, and $E(x)$ the canonical morphism $X(x)\to 1$. Unfolding the definitions, and using corollary \ref{cor:if codomain a groupoid, then f is exponentiable}, we then have for any integer $n$ a canonical equivalence:
$$X_n \sim \coprod_{x_0,...,x_n}X(x_0)\times \hom_C(x_0,...,x_n)$$
\end{notation}

\begin{prop}
\label{prop:equivalence beetwen left fibration}
Let $C$ be an $\io$-category, and $E$, $F$ two objects of $\Lfib(\Noiun C)$ corresponding to morphisms $X\to \Noiun C$, $Y\to \Noiun C$. Let $\phi:E\to F$ be a morphism.
The following are equivalent:
\begin{enumerate}
\item $\phi$ is an equivalence,
\item for any object $x$ of $C$, the induced morphism $\Rb x^*\phi:\Rb x^*E\to \Rb x^*E$ is an equivalence,
\item for any object $x$ of $C$, the induced morphism $\phi(x):X(x)\to Y(x)$ is an equivalence,
\end{enumerate}
\end{prop}
\begin{proof}
The implication $(1)\Rightarrow (2)$ is direct.
The implication $(2)\Rightarrow (3)$ comes from the fact that for any object $x$ of $C$, the value on $0$ of the simplicial object $\Rb x^*E$ (resp. $\Rb x^*F$) is $X(x)\to 1$ (resp. $Y(x)\to 1$). 

Suppose now that $\phi$ fulfills the last condition. As $\Noiun C$ is $C_0\sim \coprod_{C_0}1$, we have equivalences 
$$X_0\sim \coprod_{x:C_0} X(x)~~~~~Y_0\sim \coprod_{x:C_0} Y(x).$$
The morphism $\phi_0:X_0\to Y_0$ is then an equivalence. Eventually, as $E$ and $F$ are left fibrations, we have 
$$X_n\sim X_{\{0\}}\times_{ (\Noiun C)_{\{0\}} }(\Noiun C)_n\sim Y_{\{0\}}\times_{ (\Noiun C)_{\{0\}} }(\Noiun C)_n\sim Y_n.$$
This implies $(3)\Rightarrow (1)$, which concludes the proof. 
\end{proof}

\begin{prop}
\label{prop:lfib and W 2}
There is an equivalence, natural in $C:\ocat^{op}$,
between $\Lfib(\Noiun [C,1])$ and the $\iun$-category whose objects are arrows  of shape
$$X(0)\times C\to X(1)$$
and morphisms are natural transformations such that the induced morphism
$X(0)\times C\to Y(0)\times C$
is of shape $f\times id_C$.

For a left fibration $E$ corresponding to a morphism $X\to \Noiun [C,1]$, this arrow is the one appearing in the diagram:
\[\begin{tikzcd}[sep =0.3cm]
	& {X_1} && {X_0} \\
	{X(0)^\flat\times C^\flat} && {X(1)^\flat} \\
	& {\Noiun([C,1])_1} && {\Noiun([C,1])_{\{1\}}} \\
	{(C^\flat,0,1)} && {\{1\}}
	\arrow[from=4-1, to=3-2]
	\arrow[from=3-2, to=3-4]
	\arrow[from=2-1, to=1-2]
	\arrow[from=1-2, to=1-4]
	\arrow[from=2-1, to=4-1]
	\arrow[from=1-2, to=3-2]
	\arrow[from=1-4, to=3-4]
	\arrow[from=4-1, to=4-3]
	\arrow[from=2-3, to=4-3]
	\arrow[from=4-3, to=3-4]
	\arrow[from=2-3, to=1-4]
	\arrow[from=2-1, to=2-3]
\end{tikzcd}\]
where the left and the right squares are cartesian.
\end{prop}
\begin{proof}
Left fibrations are detected on pullback along representable. The functor $\Lfib(\uvar)$ then sends colimits of $\iPsh{\Theta\times \Delta}$ to limits.
Remark that we have a cocartesian square
\[\begin{tikzcd}[cramped]
	{\langle C,\{0\}\rangle\coprod\langle C,\{1\}\rangle} & {\langle C,1\rangle} \\
	{\langle [0],\{0\}\rangle\coprod \langle [0],\{1\}\rangle} & {\Noiun[C,1]}
	\arrow[from=1-1, to=1-2]
	\arrow[from=1-1, to=2-1]
	\arrow[from=1-2, to=2-2]
	\arrow[from=2-1, to=2-2]
	\arrow["\lrcorner"{anchor=center, pos=0.125, rotate=180}, draw=none, from=2-2, to=1-1]
\end{tikzcd}\]
According to proposition \ref{prop:lfib and W}, and as $\Lfib(\uvar)$ send colimits to limits, $\Lfib(\Noiun [C,1])$ fits in the cartesian square
\[\begin{tikzcd}
	{\Lfib(\Noiun [C,1])} & {\Fun([1],\ocat_{/C})} \\
	{\ocat\times \ocat} & {\ocat_{/C}\times \ocat_{/C}}
	\arrow[""{name=0, anchor=center, inner sep=0}, from=2-1, to=2-2]
	\arrow[from=1-1, to=2-1]
	\arrow[from=1-2, to=2-2]
	\arrow[from=1-1, to=1-2]
	\arrow["\lrcorner"{anchor=center, pos=0.125}, draw=none, from=1-1, to=0]
\end{tikzcd}\]
Using the adjunction 
\[\begin{tikzcd}
	{\dom:\ocat_{/C}} & {\ocat:\uvar\times C}
	\arrow[""{name=0, anchor=center, inner sep=0}, shift left=2, from=1-1, to=1-2]
	\arrow[""{name=1, anchor=center, inner sep=0}, shift left=2, from=1-2, to=1-1]
	\arrow["\dashv"{anchor=center, rotate=-90}, draw=none, from=0, to=1]
\end{tikzcd}\]
the $\iun$-category $\Lfib(\Noiun [C,1])$ fits in the cartesian square
\[\begin{tikzcd}
	{\Lfib(\Noiun [C,1])} & {\Fun([1],\ocat)} \\
	{\ocat\times \ocat} & {\ocat\times \ocat}
	\arrow["{(\uvar\times C,id)}"', from=2-1, to=2-2]
	\arrow[from=1-1, to=2-1]
	\arrow[from=1-2, to=2-2]
	\arrow[from=1-1, to=1-2]
	\arrow["\lrcorner"{anchor=center, pos=0.125}, draw=none, from=1-1, to=2-2]
\end{tikzcd}\]
The first assertion then follows from the last cartesian square and the proposition \ref{prp:to show fully faithfullness3} applied to $I:=1.$
 The second is obtained by walking through the equivalences used in the proof of proposition \ref{prop:lfib and W}.
\end{proof}

\begin{prop}
\label{prop:lfib and W 3}
There is an equivalence natural in $C:\ocatm^{op}$ between $\Lfib(([C,1]\otimes[1]^\sharp)^\natural)$ and the $\iun$-category whose objects are diagrams of shape
\[\begin{tikzcd}
	{X(0,0)\times C^\natural\otimes\{0\}} & {X(0,1)\times C^\natural} \\
	& {X(0,0)\times (C\otimes[1]^\sharp)^\natural} & {X(1,1)} \\
	{X(0,0)\times C^\natural\otimes\{1\}} & {X(1,0)}
	\arrow[from=1-1, to=1-2]
	\arrow[from=1-2, to=2-3]
	\arrow[from=3-2, to=2-3]
	\arrow[from=1-1, to=2-2]
	\arrow[from=2-2, to=2-3]
	\arrow[from=3-1, to=3-2]
	\arrow[from=3-1, to=2-2]
\end{tikzcd}\]
such that $X(0,0)\times C^\natural\otimes\{0\}\to X(0,1)\times C^\natural$ is of shape $f\times id_{C^\natural}$. Morphisms are natural transformations such that the induced morphisms 
$X(0,1)\times C^\natural\to Y(0,1)\times C^\natural$ and $X(0,0)\times (C\otimes[1]^\sharp)^\natural\to Y(0,0)\times (C\otimes[1]^\sharp)^\natural$
are of shape $g\times C^\natural$ and $h\times (C\otimes[1]^\sharp)^\natural$.
\end{prop}
\begin{proof}
The proposition \ref{prop:eq for cylinder marked} implies that $([C,1]\otimes[1]^\sharp)^\natural$ is the colimit of the diagram
\[\begin{tikzcd}
	{[1]\vee[C,1]^\natural} & {[C\otimes^\natural\{0\},1]} & {[C\otimes[1]^\sharp,1]^\natural} & {[C^\natural\otimes\{1\},1]} & {[C,1]^\natural\vee[1]}
	\arrow[from=1-4, to=1-3]
	\arrow[from=1-4, to=1-5]
	\arrow[from=1-3, to=1-2]
	\arrow[from=1-1, to=1-2]
\end{tikzcd}\]
According to proposition \ref{prop:example of a special colimit3 marked case} and lemma \ref{lemma:a otimes 1 is strict}, this colimit is special, and the  $\iun$-category $\Noiun ([C,1]\otimes[1]^\sharp)^\natural$ is then the colimit, computed in $\Psh{\Theta\times\Delta}$, of the diagram
\[\begin{tikzcd}[column sep =0.2cm]
	{\langle C^\natural,1\rangle\coprod\langle C^\natural,\{2\}\rangle} & {\langle C^\natural\otimes\{0\},1\rangle} && {\langle C^\natural\otimes\{1\},1\rangle} & {\langle C^\natural,\{0\}\rangle\coprod \langle C^\natural,1\rangle} \\
	{\langle [0],1\rangle\coprod \langle [0],1\rangle} & {\langle C^\natural,2\rangle} & {\langle (C\otimes[1]^\sharp)^\natural,1\rangle} & {\langle C^\natural,2\rangle} & {\langle [0],1\rangle\coprod \langle [0],1\rangle}
	\arrow[from=1-5, to=2-4]
	\arrow[from=1-1, to=2-2]
	\arrow[from=1-2, to=2-2]
	\arrow[from=1-2, to=2-3]
	\arrow[from=1-4, to=2-3]
	\arrow[from=1-4, to=2-4]
	\arrow[from=1-1, to=2-1]
	\arrow[from=1-5, to=2-5]
\end{tikzcd}\]
We then deduce the result from the proposition \ref{prop:lfib and W} in the same way as in the previous proof.
\end{proof}

\begin{prop}
\label{prop:Lfib commue with colimit}
Let $F:I\to \ocat$ be a $\Wcard$-small diagram. The canonical functor
$$\Lfib(\Noiun \colim_IF)\to \lim_I\Lfib(\Noiun F)$$
is an equivalence, where $\colim_IF$ denotes the colimit taken in $\ocat$.
\end{prop}
\begin{proof}
Let $C$ be an object of $\iPsh{\Theta}$.
As left fibrations are detected by unique right lifting property against morphisms whose codomains are of shape $\langle a,n\rangle$, a morphism $p:X\to \Noiun C$ is a left fibration if and only if for any $i:[a,n]\to C$, $(\Noiun i)^*p$ is a left fibration. 
The functor 
$$\begin{array}{ccl}
\Psh{\Delta[\Theta]}^{op}&\to &\icat_{\Wcard}\\
X&\mapsto & \Lfib(\Noiun X)
\end{array}$$
then sends colimits to limits, where $\icat_{\Wcard}$ denotes the (huge) $\iun$-category of $\Wcard$-small $\iun$-categories. To conclude the proof, we then have to show that it sends any morphism $f\in\M$ to an equivalence. If $f$ is of shape $[g,1]$ for $g\in\W$, this directly follows from proposition \ref{prop:lfib and W 2}. Suppose now that $f$ is $[a,\Sp_n]\to [a,n]$. Remark that we have a cocartesian square:
\[\begin{tikzcd}
	{\langle a, \Sp_n\rangle} & {\Noiun ([a,\Sp_n])} \\
	{\langle a,n\rangle} & {\Noiun ([a,n])}
	\arrow[from=1-1, to=2-1]
	\arrow[from=2-1, to=2-2]
	\arrow[from=1-1, to=1-2]
	\arrow[from=1-2, to=2-2]
	\arrow["\lrcorner"{anchor=center, pos=0.125, rotate=180}, draw=none, from=2-2, to=1-1]
\end{tikzcd}\]
The morphism $\Lfib(\Noiun [a,\Sp_n])\to \Lfib(\Noiun [a,n])$ then fits in the cartesian square: 
\[\begin{tikzcd}
	{\Lfib(\Noiun [a,n])} & {\Lfib(\langle a,n\rangle)} \\
	{\Lfib(\Noiun [a,\Sp_n])} & {\Lfib(\langle a, \Sp_n\rangle)}
	\arrow[from=1-2, to=2-2]
	\arrow[from=1-1, to=1-2]
	\arrow[""{name=0, anchor=center, inner sep=0}, from=2-1, to=2-2]
	\arrow[from=1-1, to=2-1]
	\arrow["\lrcorner"{anchor=center, pos=0.125}, draw=none, from=1-1, to=0]
\end{tikzcd}\]
According to proposition \ref{prop:lfib and W}, we have equivalences
$$\Lfib(\langle a ,\Sp_n\rangle)\sim \lim_{[k]\to\Sp_n}\Fun([k],\ocat_{/a})\sim \Fun([n],\ocat_{/a})\sim \Lfib(\langle a ,n\rangle)$$
It remains the case $f:=E^{eq}\to 1$. We have equivalences $\Noiun E^{eq}\sim \langle [0],E^{eq}\rangle$ and $\Noiun 1\sim 1$ .
The proposition \ref{prop:lfib and W} induces equivalences
$$ \Lfib( \langle [0],E^{eq}\rangle) \sim \lim_{[k]\to E^{eq}}\Fun([k],\ocat)\sim \Fun(1,\ocat)$$
which concludes the proof.
\end{proof}

\begin{definition} 

Let $A$ be an $\ioun$-category. An object $E:\ouncat_{/A}$ is \wcnotion{$\U$-small}{small object@$\U$-small object of $\ouncat_{/A}$} if for any morphism $i:\langle b,n\rangle\to A$, the space of morphism between $i$ and $E$ is $\U$-small. Remark that an object $F$ of $\Lfib(\Noiun A)$ corresponding to a left fibration $X\to \Noiun A$ is $\U$-small if an only if for any object $a$ of $A$, $X(a)$ is $\U$-small .

Eventually, we define $\Lfib_{\U}( A)$ as the full sub $\iun$-category of $\Lfib( A)$ whose objects correspond to $\U$-small left fibrations. In particular, $\Lfib_{\U}( A)$ is a $\V$-small $\iun$-category unlike $\Lfib( A)$ which is a $\Wcard$-small $\iun$-category. 
\end{definition}
\begin{construction}
\label{cons:defi of uni}
The proposition \ref{prop:Lfib commue with colimit} implies that the functor 
$$C:\ocat\mapsto \tau_0\mbox{$\Lfib_{\U}$}(\Noiun C)$$
sends colimits to limits. We then define $\uni$ as the $\io$-category that represents this object:
\begin{equation}
\label{eq:defi of uni}
\begin{array}{rcll}
\uni:&\Theta^{op} &\to &\igrd\\
& a&\mapsto & \tau_0\mbox{$\Lfib_{\U}$}(\Noiun a)
\end{array}
\end{equation}

 We then have by definition an equivalence 
\begin{equation}
\Hom(C,\uni)\sim \tau_0 \mbox{$\Lfib_{\U}$}(\Noiun C).
\end{equation}
As the functor $\Noiun$ preserves product, for any $\io$-category $D$, 
we also have a canonical equivalence
\begin{equation}
\Hom(C,\uHom(D,\uni))\sim \tau_0(\mbox{$\Lfib_{\U}$}(\Noiun C\times \Noiun D)).
\end{equation}
Eventually, by construction, the $\infty$-groupoid of objects of $\uni$ corresponds to the $\infty$-groupoid of $\U$-small $\io$-categories, and according to proposition \ref{prop:lfib and W 2}, we have an equivalence 
\begin{equation}
\label{eq:hom of uni}
\hom_{\uni}(C,D)\sim \uHom(C,D).
\end{equation}
The $\iota$-category $\uni$ is called the $\iota$-category of small $\iota$-categories.
\end{construction}

\begin{construction}  \label{cons:dualities fo omega}
Let $S$ be a subset of $\Nb^*$. We define the subset $\Sigma S=\{i+1,i\in S\}$. 
Remark that for any $n$, we have \ssym{((b49@$(\uvar)^S$}{for $\uni$}
$$((\Noiun C)_n)^S\sim (\Noiun C^{\Sigma S})_n$$
We then set the functor 
$$(\uvar)^S:\uni\to (\uni)^{\Sigma S}$$
sending a $\U$-small left fibration $X\to \Noiun C$ to the left fibration $n\mapsto (X_n^S\to (\Noiun C^{\Sigma S})_n^S)$. These functors are called \snotion{dualities}{for $\uni$}.
In particular, we have the \snotionsym{odd duality}{((b60@$(\uvar)^{op}$}{for $\uni$} $(\uvar)^{op}:\uni\to \uni^{co}$, corresponding to the set of odd integer, the \snotionsym{even duality}{((b50@$(\uvar)^{co}$}{for $\uni$} $(\uvar)^{co}:\uni\to (\uni^{t })^{op}$, corresponding to the subset of non negative even integer, the \snotionsym{full duality}{((b80@$(\uvar)^{\circ}$}{for $\uni$} $(\uvar)^{\circ}:\uni \to \uni^{t\circ}$, corresponding to $\Nb^*$ and the \snotionsym{transposition}{((b70@$(\uvar)^t$}{for $\uni$} $(\uvar)^t:\uni \to \uni^{\Sigma t}$, corresponding to the singleton $\{1\}$. Eventually, we have equivalences
$$((\uvar)^{co})^{op}\sim (\uvar)^{\circ} \sim ((\uvar)^{op})^{co}.$$ 
\end{construction}

\subsection{Grothendieck construction}
\begin{notation*}
Through this section, we will identify any marked $\io$-categories $C$ with the canonical induced morphism $C\to1$. If $f:X\to Y$ is a morphism, $f\times C$ then corresponds to the canonical morphism $X\times C\to Y$.
\end{notation*}
\begin{definition} Let $A$ be an $\io$-category and $a$ an object of $A$, we denote by \wcnotation{$h_a^A$}{(h@$h_{a}^{A}$} the morphism $1\to A^\sharp$ induces by $a$.
\end{definition}

We recall that the notion of left cartesian fibrant replacement is defined in construction \ref{cons:of fb for fibration}. Given a morphism $f:X\to I$, its left cartesian fibrant replacement is denoted $\Fb f:\tilde{X}\to I$. Proposition \ref{prop:explicit factoryzation} states that the left fibrant replacement of $h_a^A$, which we then denote by \wcnotation{$\Fb h^A_a$}{(fh@$\Fb h_{a}^{A}$)}, is the fibration $A^\sharp_{a/}\to A^\sharp$. Remark \ref{rem:cartesian square slicdes} provides, for any object $b$ of $A^\sharp$, a cartesian square
\begin{equation}
\label{eq:fiber of slice}
\begin{tikzcd}
	{\hom_A(a,b)^\flat} & {A^{\sharp}_{a/}} \\
	{\{b\}} & {A^{\sharp}}
	\arrow[from=2-1, to=2-2]
	\arrow[from=1-1, to=2-1]
	\arrow[from=1-1, to=1-2]
	\arrow["{\Fb h_a^A}", from=1-2, to=2-2]
\end{tikzcd}
\end{equation}
which induces a canonical morphism $h^A_b\times \hom_A(a,b)^\flat\to \Fb h^A_a$, and consequently, a morphism $\Fb h^A_b\times \hom_A(a,b)^\flat\to \Fb h^A_a$.

The case of $A:=[C,1]$ will be of particular interest. The morphism $\Fb h^{[C,1]}_{1}$ is just $ h^{[C,1]}_{1}$ and theorem \ref{theo:equivalence between slice and join} implies that $\Fb h^{[C,1]}_{0}$ is the canonical morphism $1\costar C^\flat\to [C,1]^\sharp$. In this last case, the square \eqref{eq:fiber of slice} corresponds to the square
\[\begin{tikzcd}
	{C^\flat} & {1\costar C^\flat} \\
	{\{1\}} & {[C,1]^\sharp}
	\arrow[from=2-1, to=2-2]
	\arrow[from=1-1, to=2-1]
	\arrow[from=1-1, to=1-2]
	\arrow["{\Fb h_0^{[C,1]}}", from=1-2, to=2-2]
\end{tikzcd}\]
induces by the one of theorem \ref{theo:formula between pullback of slice and tensor marked case}.

\begin{notation}
When nothing is specified, the morphism $C^\flat \to \Fb h_0^{[C,1]}$ will always corresponds to this square.
\end{notation}

\begin{definition}
\sym{(fh@$\Fb h^C_{\cdot}$}\sym{(cpoint@$C_{\cdot/}$}
Let $C$ be an $\io$-category. We define the simplicial marked $\io$-category $C_{\cdot/}$ and the simplicial arrow of marked $\io$-categories $\Fb h^C_{\cdot}$ whose value on an integer $n$ is given by the following pullback 
\[\begin{tikzcd}
	{(C_{\cdot/})_n} & {(C^{\sharp})^{[n+1]^\sharp}} \\
	{(\Noiun C)_n^\flat\times C^{\sharp}} & {(C^\sharp)^{[n]^\sharp}\times (C^\sharp)^{\{n+1\}}}
	\arrow[""{name=0, anchor=center, inner sep=0}, from=2-1, to=2-2]
	\arrow[from=1-2, to=2-2]
	\arrow["{(\Fb h_{\cdot})_n}"', from=1-1, to=2-1]
	\arrow[from=1-1, to=1-2]
	\arrow["\lrcorner"{anchor=center, pos=0.125}, draw=none, from=1-1, to=0]
\end{tikzcd}\]
 and where the functoriality in $n$ is induced by the universal property of pullback.
Unfolding the definition, on all integer $n$, the canonical morphism $(C_{\cdot/})_n\to C^\sharp$ corresponds to the morphism 
$$ \coprod\limits_{x_0,...,x_n:C_0} \hom_C^\flat(x_0,...,x_n)\times \Fb h_{x_n}^C$$
and is then a left cartesian fibration according to theorem \ref{theo:left cart stable by colimit}.
\end{definition}

\begin{construction}
\label{cons:definition of integral de grot}
Let $E$ be an object of $\ouncat_{/\Noiun C}$ corresponding to an arrow $X \to\Noiun C$. The \wcnotion{Grothendieck construction}{grothendieck construction@Grothendieck construction} of $E$, is the object of $\ocatm_{/C^\sharp}$ defined by the formula
$$\int_CE:=\colim_n (X^\flat \times_{(\Noiun C)^\flat } \Fb h_{\cdot})_n.$$
As the Grothendieck construction is by definition a colimit of left cartesian fibrations, the theorem \ref{theo:left cart stable by colimit} implies that it is also a left cartesian fibration. The Grothendieck construction then defines a functor
$$\int_{C}: \ouncat_{/\Noiun C}\to \LCart(C^\sharp).$$
Unfolding the definition, if $E$ is a left fibration, $\int_C E$ is the colimit of a simplicial diagram whose value on $n$ is:
$$
\coprod\limits_{x_0,...,x_n:C_0}X(x_0)\times \hom_C^\flat(x_0,...,x_n)\times \Fb h_{x_n}^C$$
\end{construction}

\begin{example}
\label{exe:of int}
Let $C$ be an $\io$-category.
Let $E$ be an object of $\Lfib(\Noiun [C,1])$ corresponding to a morphism $X\to \Noiun ([C,1])$. According to proposition \ref{prop:lfib and W 2}, this object corresponds to a morphism $X(0)\times C\to X(1)$. The arrow $\int_{[C,1]}E$ corresponds to the colimit of the following diagram:
\[\begin{tikzcd}
	{X(0)^\flat\times\Fb h^{[C,1]}_{0}} & {X(0)^\flat\times C^\flat} & {X(1)^\flat}
	\arrow[from=1-2, to=1-1]
	\arrow[from=1-2, to=1-3]
\end{tikzcd}\]
The domain of this arrow is then the colimit of the following diagram:
\[\begin{tikzcd}
	{X(0)^\flat\times[C,1]^\sharp_{0/}} & {X(0)^\flat\times C^\flat} & {X(1)^\flat}
	\arrow[from=1-2, to=1-1]
	\arrow[from=1-2, to=1-3]
\end{tikzcd}\]
\end{example}

\begin{lemma}
\label{lemma:intpreserces initial}
The functor $\int_C:\ouncat_{/\Noiun C}\to \LCart(C^\sharp)$ preserves colimits. Moreover, it sends morphisms of $\J$ to equivalences. 
\end{lemma}
\begin{proof}
According to corollary \ref{cor:inclusion of lcatt into the slice preserves colimits}, it is sufficient to show that the composite 
$$\ouncat_{/\Noiun C}\xrightarrow{ \int_C} \LCart(C^\sharp)\xrightarrow{\dom}\ocatm$$
preserves colimits.

To this extend, we consider the functor
$$\alpha: \iPsh{\Theta\times \Delta}_{/\Noiun C}\to \iPsh{t\Theta\times \Delta}$$
sending an object $E$ of $\Lfib(\Noiun C)$ corresponding to a morphism $X\to (\Noiun C)$ to 
$X\times_{(\Noiun C)^\flat } C_{\cdot/}$, 
and the functor 
$$\beta:\iPsh{t\Theta\times \Delta}\to \ocatm$$
that is the left Kan extension of the functor $t\Theta\times \Delta\to t\Theta\to \mPsh{\Theta}$. As $\iPsh{\Theta\times \Delta}$ is locally cartesian closed, $\alpha$ preserves colimits.
The composite $\beta\circ\alpha$ then preserves colimits. Moreover, we have a commutative diagram
\[\begin{tikzcd}
	{\iPsh{\Theta\times \Delta}_{/\Noiun C}} & \ocatm \\
	{\ouncat_{/\Noiun C}} & {\LCart(C^\sharp)}
	\arrow["\beta\circ\alpha", from=1-1, to=1-2]
	\arrow["\Fb"', from=1-1, to=2-1]
	\arrow["{\int_C}"', from=2-1, to=2-2]
	\arrow["\dom"', from=2-2, to=1-2]
\end{tikzcd}\]
According to proposition \ref{prop:if left fib the fib}, one then has to show that $\beta\circ\alpha$ sends any morphism of $\J$ to an equivalence to conclude. Indeed, it will implies that $\beta\circ \alpha$ lifts to a colimit preserving functor $$\Db(\beta\circ \alpha):\ouncat_{/\Noiun C}\to \ocatm,$$ and the previous square implies that this morphism is equivalent to $\dom \int_C$.

Suppose given two cartesian squares
\[\begin{tikzcd}
	X & {X'} & { C_{\cdot/}} \\
	{\langle a, \{0\}\rangle} & {\langle a, [n]\rangle} & {(\Noiun C)^\flat}
	\arrow["f"', from=2-1, to=2-2]
	\arrow[from=1-3, to=2-3]
	\arrow[from=1-1, to=2-1]
	\arrow[from=1-2, to=2-2]
	\arrow[from=2-2, to=2-3]
	\arrow[from=1-2, to=1-3]
	\arrow["g", from=1-1, to=1-2]
	\arrow["\lrcorner"{anchor=center, pos=0.125}, draw=none, from=1-1, to=2-2]
	\arrow["\lrcorner"{anchor=center, pos=0.125}, draw=none, from=1-2, to=2-3]
\end{tikzcd}\]
By currying, we see these objects as functors $t\Theta^{op}\to \iPsh{\Delta}$. The right vertical morphism is then pointwise a right fibration of $\iun$-categories fibered in $\infty$-groupoids, as it corresponds, for a fixed $a:t\Theta$ and $n:\Delta$, to the morphism of $\infty$-groupoid:
$$\coprod_{x_0,...,x_n:C_0}\Hom(a,\hom_C(x_0,...,x_n)^\flat)\times\Hom(a,C^\sharp_{x_n/})\to \coprod_{x_0,...,x_n:C_0}\Hom(a,\hom_C(x_0,...,x_n)^\flat).$$

As the morphism $f$ is pointwise initial, so is $g$. 
As $\beta$ sends pointwise initial morphisms to equivalence, this implies that $\beta\alpha (f):= \beta(g)$ is an equivalence. 

Suppose now given two cartesian squares
\[\begin{tikzcd}
	X & {X'} & { C_{\cdot/}} \\
	{\langle a, 0\rangle} & {\langle b, 0\rangle} & {(\Noiun C)^\flat}
	\arrow["{\langle f,0\rangle}"', from=2-1, to=2-2]
	\arrow[from=1-3, to=2-3]
	\arrow[from=1-1, to=2-1]
	\arrow[from=1-2, to=2-2]
	\arrow[from=2-2, to=2-3]
	\arrow[from=1-2, to=1-3]
	\arrow["g", from=1-1, to=1-2]
	\arrow["\lrcorner"{anchor=center, pos=0.125}, draw=none, from=1-1, to=2-2]
	\arrow["\lrcorner"{anchor=center, pos=0.125}, draw=none, from=1-2, to=2-3]
\end{tikzcd}\]
with $f\in \W$. By currying, we see these objects as functors $\Delta\to\iPsh{t\Theta}$. The right vertical morphism is then pointwise a right cartesian fibration. As the morphism $\langle f,0\rangle$ is pointwise in $\widehat{\Wm}$, so is $g$. The morphism $\colim_n g_n$ is then in $\widehat{\Wm}$ and $\beta\alpha (f):= \beta(g)$ is an equivalence.
\end{proof}

\begin{construction} 
 We will denote also by 
$$\int_C:\Lfib(\Noiun C)\to \LCart(C^\sharp)$$ 
the restriction of the Grothendieck construction. 
This will not cause any confusion as from now on we will only consider the 
Grothendieck construction of left fibration.
 The lemma \ref{lemma:intpreserces initial} then implies that this functor is colimit preserving, and it is then part of an adjunction \index[notation]{(partial@$\partial_C$}
\begin{equation}
\label{eq:underived GR constuction}
\begin{tikzcd}
	{\int_C:\Lfib(\Noiun C)} & { \LCart(C^\sharp):\partial_C}
	\arrow[""{name=0, anchor=center, inner sep=0}, shift left=2, from=1-1, to=1-2]
	\arrow[""{name=1, anchor=center, inner sep=0}, shift left=2, from=1-2, to=1-1]
	\arrow["\dashv"{anchor=center, rotate=-90}, draw=none, from=0, to=1]
\end{tikzcd}
\end{equation}
\end{construction}
\begin{lemma}
\label{lemma:partial fiber}
Let $i:C^\sharp\to D^\sharp$ be a morphism. The natural transformation $$\partial_{C}\circ\Rb i^*\to \Rb (\Noiun{i})^*\circ\partial_D$$ is an equivalence.
\end{lemma}
\begin{proof}
As equivalences between left fibrations are detected on fibers, one can suppose that $C$ is the terminal $\io$-category. Let $c$ denote the object of $D$ corresponding to $i$.
Let $E$ be an object of $\Lfib(\Noiun1)$, corresponding to a morphism $A\to 1$. According to lemma \ref{lemma:intpreserces initial}, we then have equivalences
$$\begin{array}{rclr}
\Lb i_! \int_1 E &\sim & \Lb i_! ( A^\flat\times h_1^1)\\
&\sim & A^\flat\times \Fb h_c^D\\
	&=: &\int_D {\Noiun{i}}_!E\\
	&\sim &\int_D \Lb(\Noiun{i})_!E& (\ref{lemma:intpreserces initial})\\
\end{array}$$
The canonical morphism $\Lb i_!\circ \int_1 \to \int_D \circ \Lb{(\Noiun{i})}_!$ is then an equivalence, which implies by adjunction that $\partial_{1}\circ\Rb_i^*\to \Rb (\Noiun{i})^*\circ\partial_D$ also is.
\end{proof}
\begin{definition}
 Let $C$ be an $\io$-category and $c$ an object of $C^\sharp$.
We define $(\Noiun C)_{/c}$ as the simplicial object in $\ocat$ whose value on $(a,n)$ fits in the cocartesian square 
\[\begin{tikzcd}
	{((\Noiun C)_{/c})_{(a,n)}} & {(\Noiun C)_{(a,n+1)}} \\
	{\{c\}} & {(\Noiun C)_{(a,\{n+1\})}}
	\arrow[from=2-1, to=2-2]
	\arrow[from=1-1, to=2-1]
	\arrow[from=1-2, to=2-2]
	\arrow[from=1-1, to=1-2]
	\arrow["\lrcorner"{anchor=center, pos=0.125, rotate=45}, draw=none, from=1-1, to=2-2]
\end{tikzcd}\]
Unfolding the definition, $(\Noiun C)_{/c}$ is the simplicial diagram whose value on $n$ is
$$\coprod_{x_0,...,x_n}\hom_C(x_0,...,x_n,c)$$
\end{definition}

\begin{lemma}
\label{lemma:fiber of F h .}
There is an equivalence
$$((\Noiun C)_{/c})^\flat \sim c^* \Fb h_{\cdot}.$$
\end{lemma}
\begin{proof}
A morphism $\langle a,n\rangle\to (c^* \Fb h_{\cdot})^\natural$ is the data of a commutative diagram
\[\begin{tikzcd}
	{a^\flat\otimes\{n+1\}} & {a^\flat\otimes[n+1]^\sharp} & {\coprod_{k\leq n} a^\flat\otimes\{k\}} \\
	{\{c\}} & {C^\sharp} & {\coprod_{k\leq n} \{k\}}
	\arrow[from=1-1, to=1-2]
	\arrow[from=1-1, to=2-1]
	\arrow[from=1-2, to=2-2]
	\arrow[from=1-3, to=1-2]
	\arrow[from=1-3, to=2-3]
	\arrow[from=2-1, to=2-2]
	\arrow[from=2-3, to=2-2]
\end{tikzcd}\]
which is, according to proposition \ref{prop:crushing of Gray tensor is identitye marked case}, equivalent to a diagram
\[\begin{tikzcd}
	{\{n+1\}} & {[a,n]} \\
	{\{c\}} & {C^\sharp}
	\arrow[from=1-1, to=1-2]
	\arrow[from=1-1, to=2-1]
	\arrow[from=1-2, to=2-2]
	\arrow[from=2-1, to=2-2]
\end{tikzcd}\]
and so to a morphism $\langle a,n\rangle\to (\Noiun C)_{c/}$. As $c^*\Fb h_{\cdot}$ has a trivial marking, this shows the desired equivalence.
\end{proof}
\begin{lemma}
\label{lemma:int fiber 1}
Let $p:X\to \Noiun C$ be a left fibration, and $c$ an object of $C$. 
The canonical morphism 
$$X(c)\to \colim_n (X\times_{\Noiun C} (\Noiun C)_{/c})_n$$
is an equivalence.
\end{lemma}
\begin{proof}
We will show a slightly stronger statement, which is that the morphism
$$X(c)\to \colim_n (X\times_{(\Noiun C)} (\Noiun C)_{/c})_n$$
is an equivalence when the colimit is taken in $\infty$-presheaves on $\Theta$.
As the colimit in presheaves commutes with evaluation, one has to show that for any globular sum $a$, the canonical morphism of $\infty$-groupoids
$$\Hom(a,X(c))\to \colim_n (\Hom(a,X_n)\times_{\Hom(a,(\Noiun C)_n)}\Hom(a, (\Noiun C)_{/c})_n)$$
is an equivalence. Remark that the simplicial $\infty$-groupoid $ \Hom(a, ((\Noiun C)_{/c})_\bullet)$ is equivalent to the simplicial $\infty$-groupoid $(\Hom(a,\Noiun C)_\bullet)_{/c}$.
If we denote also by $\Hom(a,X(c))$ the constant simplicial $\infty$-groupoid $n\mapsto \Hom(a,X(c))$, we have a cartesian square
\[\begin{tikzcd}[column sep =0.3cm]
	{\Hom(a,X(c))} & { \Hom(a,X_\bullet)\times_{\Hom(a,(\Noiun C)_\bullet)}\Hom(a, (\Noiun C)_\bullet)_{/c}} & { \Hom(a,X_\bullet)} \\
	{\{c\}} & {\Hom(a, (\Noiun C)_\bullet)_{/c}} & {\Hom(a, (\Noiun C)_\bullet)}
	\arrow[from=1-3, to=2-3]
	\arrow[from=1-1, to=1-2]
	\arrow[from=1-2, to=2-2]
	\arrow[from=1-1, to=2-1]
	\arrow[""{name=0, anchor=center, inner sep=0}, from=2-1, to=2-2]
	\arrow[""{name=1, anchor=center, inner sep=0}, from=2-2, to=2-3]
	\arrow[from=1-2, to=1-3]
	\arrow["\lrcorner"{anchor=center, pos=0.125}, draw=none, from=1-2, to=1]
	\arrow["\lrcorner"{anchor=center, pos=0.125}, draw=none, from=1-1, to=0]
\end{tikzcd}\]
Moreover, the left vertical morphism is a left fibration of $\iun$-category fibered in $\infty$-groupoid.
As pullbacks along left fibrations preserve final morphisms,
the morphism
$$\Hom(a,X(c))\to \Hom(a,X_\bullet)\times_{\Hom(a,(\Noiun C)_\bullet)}\Hom(a, (\Noiun C)_\bullet)_{/c}$$
is final. Taking the colimit, this implies the result.
\end{proof}

\begin{lemma}
\label{lemma:int fiber 2}
Let $i:C^\sharp\to D^\sharp$ be a morphism. The natural transformation 
$$\int_D\circ \Rb(\Noiun i)^*\to \Rb i^* \circ\int_C$$
is an equivalence.
\end{lemma}
\begin{proof}
 As equivalences between left cartesian fibrations are detected on fibers, one can suppose that $C$ is the terminal $\io$-category. Let $c$ denote the object of $D$ corresponding to $i$ and let $E$ be an object of $\Lfib(\Noiun C)$, corresponding to a left fibration $X\to \Noiun C$. 
 By construction, $\int_CE$ is a colimit of left cartesian fibrations. However, as proposition \ref{prop:fiber preserves colimits} states that $\Rb i^*$ commutes with colimit, we have 
$$\begin{array}{rclc}
\Rb i^*\int_CE&\sim &\colim_n X_n^\flat\times_{(\Noiun C)_n^\flat}\Rb i^*\Fb h^C_{\cdot}\\
&\sim &\colim_{n}(X\times_{\Noiun C} (\Noiun C)_{/c})^\flat_n&(\ref{lemma:fiber of F h .})
\end{array}$$
Moreover, remark that $\int_1 \Rb (\Noiun i)^* E$ is equivalent to $X(c)$, and the canonical morphism
$\int_D \Rb(\Noiun i)^*E\to \Rb i^* \int_CE$ is then the image by $(\uvar)^\flat$ of the equivalence given by lemma \ref{lemma:int fiber 1}.
\end{proof}

\begin{prop}
\label{prop: derived int and partial are natural}
The functors $\int_C$ and $\partial_C$ are natural in $C:\ocat^{op}$.
\end{prop}
\begin{proof}
We denote by $\Arr^{fib}(\ocatm)$ (resp. $\Arr^{fib}(\ouncat)$) the full sub $\iun$-category of $\Arr(\ocatm)$ (resp. $\Arr(\ouncat)$) whose objects are $\U$-small left cartesian fibrations (resp. $\U$-small left fibrations). 
We also set $\ocat\times_{\ocatm}\Arr^{fib}(\ocatm)$ and $\ocat\times_{\ouncat}\Arr^{fib}(\ouncat)$ as the pullbacks:
\[\begin{tikzcd}
	{\ocat\times_{\ocatm}\Arr^{fib}(\ocatm)} & {\Arr^{fib}(\ocatm)} \\
	\ocat & \ocatm \\
	{\ocat\times_{\ouncat}\Arr^{fib}(\ouncat)} & {\Arr^{fib}(\ouncat)} \\
	\ocat & \ouncat
	\arrow[""{name=0, anchor=center, inner sep=0}, "{(\uvar)^{\sharp}}"', from=2-1, to=2-2]
	\arrow["\codom", from=1-2, to=2-2]
	\arrow[from=1-1, to=2-1]
	\arrow[from=1-1, to=1-2]
	\arrow[""{name=1, anchor=center, inner sep=0}, "\Noiun"', from=4-1, to=4-2]
	\arrow[from=3-1, to=4-1]
	\arrow["\codom", from=3-2, to=4-2]
	\arrow[from=3-1, to=3-2]
	\arrow["\lrcorner"{anchor=center, pos=0.125}, draw=none, from=1-1, to=0]
	\arrow["\lrcorner"{anchor=center, pos=0.125}, draw=none, from=3-1, to=1]
\end{tikzcd}\]
The two left vertical morphism inherit from the right vertical morphisms of a structure of Grothendieck fibrations fibered in $\iun$-categories, where cartesian liftings are given by morphisms between arrows corresponding to cartesian squares.

As the assignation $C\mapsto \Fb h_{\cdot}^C$ can be promoted in a functor $\ocat\to \Arr(\Fun(\Delta,\ocatm))$
the functors $\int_C$ and $\partial_C$ are the restrictions of two functors $\int$ and $\partial$ fitting in commutative triangles:
\[\begin{tikzcd}
	& {\ocat\times_{\ocatm}\Arr^{fib}(\ocatm)} \\
	{\ocat\times_{\ouncat}\Arr^{fib}(\ouncat)} & \ocat \\
	& {\ocat\times_{\ouncat}\Arr^{fib}(\ouncat)} \\
	{\ocat\times_{\ocatm}\Arr^{fib}(\ocatm)} & \ocat
	\arrow[from=1-2, to=2-2]
	\arrow[from=3-2, to=4-2]
	\arrow[from=2-1, to=2-2]
	\arrow["\int", from=2-1, to=1-2]
	\arrow["\partial", from=4-1, to=3-2]
	\arrow[from=4-1, to=4-2]
\end{tikzcd}\]
Lemmas \ref{lemma:partial fiber} and \ref{lemma:int fiber 2} imply that these two functors preserve cartesian arrows, and the Grothendieck deconstruction then implies the desired result.
\end{proof}

\begin{theorem}
\label{theo:gr construction}
For any $\io$-category $C$, the adjunction 
$$\begin{tikzcd}
	{\int_C:\Lfib(\Noiun C)} & { \LCart(C^\sharp):\partial_C}
	\arrow[""{name=0, anchor=center, inner sep=0}, shift left=2, from=1-1, to=1-2]
	\arrow[""{name=1, anchor=center, inner sep=0}, shift left=2, from=1-2, to=1-1]
	\arrow["\dashv"{anchor=center, rotate=-90}, draw=none, from=0, to=1]
\end{tikzcd}$$
defined in \eqref{eq:underived GR constuction}, is an adjoint equivalence.
\end{theorem}
\begin{proof}
As equivalences between left fibrations and between left cartesian fibrations are detected on fibers, and as the two functors are natural in $C$, it is sufficient to show the result for $C$ being the terminal $\io$-category. In this case remark that $\Lfib(\Noiun1)\sim \LCart(1)$ and that both $\int_1$ and $\partial_1$ are the identities. 
\end{proof}

\begin{cor}
\label{cor:fib over a colimit2}
Let $F:I\to \ocatm$ be a $\Wcard$-small diagram. The canonical functor
$$\LCartc(\colim_IF) \to \lim_I \LCartc(F)$$
is an equivalence.
\end{cor}
\begin{proof}
This functor fits in an adjunction:
\[\begin{tikzcd}
	{\colim_I:\lim_I\LCartc(F)} & {\LCartc(\colim_I F)}
	\arrow[""{name=0, anchor=center, inner sep=0}, shift left=2, from=1-2, to=1-1]
	\arrow[""{name=1, anchor=center, inner sep=0}, shift left=2, from=1-1, to=1-2]
	\arrow["\dashv"{anchor=center, rotate=-90}, draw=none, from=1, to=0]
\end{tikzcd}\]
The corollary \ref{cor:fib over a colimit} implies that the counit of this adjunction is an equivalence. To conclude, we have to show that the right adjoint is essentially surjective.
By definition, the morphism $\tau_0\LCart(I^\sharp)\to \tau_0\LCartc(I)$ is an equivalence.
According to theorem \ref{theo:gr construction}, on the $\infty$-groupoid of objects, the right adjoint corresponds to the equivalence
$$\tau_0\Lfib(\Noiun\colim_I F^\sharp) \to \lim_I \tau_0\Lfib(\Noiun F^\sharp)$$
given in proposition \ref{prop:Lfib commue with colimit}.
\end{proof}

\begin{cor}
\label{cor:antecedant of slice}
Let $C$ be an $\io$-category and $c$ be an object of $c$. The left fibration $\partial_C \Fb h_c$ is the morphism of simplicial objects:
\[\begin{tikzcd}[column sep =0.5cm]
	\cdots & {\coprod_{x_0,x_1,x_2:C_0}\hom_C(y,x_0,x_1,x_2)} & {\coprod_{x_0,x_1:C_0}\hom_C(y,x_0,x_1)} & {\coprod_{x_0:C_0}\hom_C(y,x_0)} \\
	\cdots & {\coprod_{x_0,x_1,x_2:C_0}\hom_C(x_0,x_1,x_2)} & {\coprod_{x_0,x_1:C_0}\hom_C(x_0,x_1)} & {\coprod_{x_0:C_0}1}
	\arrow[shift right=4, from=2-2, to=2-3]
	\arrow[shift left=4, from=2-2, to=2-3]
	\arrow[from=2-2, to=2-3]
	\arrow[shift left=2, from=2-3, to=2-2]
	\arrow[shift right=2, from=2-3, to=2-2]
	\arrow[shift left=2, from=2-3, to=2-4]
	\arrow[shift right=2, from=2-3, to=2-4]
	\arrow[from=2-4, to=2-3]
	\arrow[from=1-3, to=2-3]
	\arrow[from=1-4, to=2-4]
	\arrow[shift left=2, from=1-3, to=1-4]
	\arrow[from=1-4, to=1-3]
	\arrow[shift right=2, from=1-3, to=1-4]
	\arrow[shift right=4, from=1-2, to=1-3]
	\arrow[from=1-2, to=1-3]
	\arrow[shift left=4, from=1-2, to=1-3]
	\arrow[shift right=2, from=1-3, to=1-2]
	\arrow[shift left=2, from=1-3, to=1-2]
	\arrow[from=1-2, to=2-2]
\end{tikzcd}\]
\end{cor}
\begin{proof}
We denote by $E:=X\to \Noiun C$ this left fibration.
According to theorem \ref{theo:gr construction}, we can equivalently show that the Grothendieck integral of $E$ is the morphism $C^{\sharp}_{c/}\to C$. 
Remark that we have by construction a family of cartesian squares
\[\begin{tikzcd}
	{X_n\times_{(\Noiun C)_n} (C_{\cdot/})_n} & {(C^\sharp)^{[1+n+1]^\sharp}} & {(C^\sharp)^{[1]^\sharp}} \\
	{\{c\}\times (\Noiun C)_n\times C^\sharp} & {C^\sharp \times (C^\sharp)^{[n]^\sharp}\times C^\sharp} & {C^\sharp\times C^\sharp}
	\arrow[from=1-1, to=2-1]
	\arrow[from=1-1, to=1-2]
	\arrow[""{name=0, anchor=center, inner sep=0}, from=2-1, to=2-2]
	\arrow[from=1-2, to=2-2]
	\arrow["{(C^{\sharp})^{h_n}}", from=1-2, to=1-3]
	\arrow[from=2-2, to=2-3]
	\arrow[from=1-3, to=2-3]
	\arrow["\lrcorner"{anchor=center, pos=0.125}, draw=none, from=1-2, to=2-3]
	\arrow["\lrcorner"{anchor=center, pos=0.125}, draw=none, from=1-1, to=0]
\end{tikzcd}\]
natural in $n$, where $h_n:[1]\to [1+n+1]$ is the simplicial morphism preserving the extremal points. The outer square factors in two cartesian squares:
\[\begin{tikzcd}
	{X_n\times_{(\Noiun C)_n} (C_{\cdot/})_n} & {C^{\sharp}_{c/}} & {(C^\sharp)^{[1]^\sharp}} \\
	{\{c\}\times (\Noiun C)_n\times C^\sharp} & {\{c\}\times C^\sharp} & {C^\sharp\times C^\sharp}
	\arrow[from=1-1, to=2-1]
	\arrow[from=1-1, to=1-2]
	\arrow[from=2-1, to=2-2]
	\arrow[from=1-2, to=2-2]
	\arrow["\lrcorner"{anchor=center, pos=0.125}, draw=none, from=1-1, to=2-2]
	\arrow[from=2-2, to=2-3]
	\arrow[from=1-3, to=2-3]
	\arrow[from=1-2, to=1-3]
	\arrow["\lrcorner"{anchor=center, pos=0.125}, draw=none, from=1-2, to=2-3]
	\arrow["\lrcorner"{anchor=center, pos=0.125}, draw=none, from=1-1, to=2-2]
\end{tikzcd}\]
This provides a canonical morphism 
$$\int_{C} E := \colim_n ( X_n\times_{(\Noiun C)_n} (\Fb h_{\cdot})_n)\to \Fb h_c^C$$
Using the naturality of the integral given in proposition \ref{prop: derived int and partial are natural}, we can see that on fibers, this morphism corresponds to the identity
$$\text{id}:\hom_C(c,d)\to \hom_C(c,d).$$

As $\int_{C} E$ and $\Fb h_c^C$ are left cartesian fibrations, this implies that they are equivalent.
\end{proof}

\begin{cor}
\label{cor:recapitulatif}
There is an equivalence, natural in $C:\ocat$,
between $\LCart([C,1]^\sharp)$ and the $\iun$-category whose objects are arrows of shape
$$\phi:X_0\times C\to X_1$$ and morphisms are natural transformations such that the induced morphism
$X_0\times C\to Y_0\times C$
is of shape $f\times \text{id}_C$.

For a left cartesian fibration $p:X\to [C,1]^\sharp$, the corresponding arrow is 
$$\phi:X_0\times C\to X_1$$
where $X_0$ and $X_1$ fit in the cartesian squares
\[\begin{tikzcd}
	{X_0^\flat} & X & {X_1^\flat} \\
	{\{0\}} & {[C,1]^\sharp} & {\{1\}}
	\arrow[from=1-1, to=1-2]
	\arrow[from=1-1, to=2-1]
	\arrow["\lrcorner"{anchor=center, pos=0.125}, draw=none, from=1-1, to=2-2]
	\arrow["p", from=1-2, to=2-2]
	\arrow[from=1-3, to=1-2]
	\arrow["\lrcorner"{anchor=center, pos=0.125, rotate=-90}, draw=none, from=1-3, to=2-2]
	\arrow[from=1-3, to=2-3]
	\arrow[from=2-1, to=2-2]
	\arrow[from=2-3, to=2-2]
\end{tikzcd}\]
and $\phi^\flat$ is the restriction to $X_0^\flat\times C^\flat \to X_0^\flat\times [C,1]^\sharp_{0/}$ of the lift appearing in the square
\[\begin{tikzcd}
	{X_0^\flat} & X \\
	{X_0^\flat\times[C,1]^{\sharp}_{0/}} & {[C,1]^\sharp}
	\arrow[from=1-1, to=1-2]
	\arrow[from=1-1, to=2-1]
	\arrow["p", from=1-2, to=2-2]
	\arrow["\phi"{description}, from=2-1, to=1-2]
	\arrow[from=2-1, to=2-2]
\end{tikzcd}\]

Conversely, given a morphism $\phi:X_0\times C\to X_1$, the corresponding left cartesian fibration is the horizontal colimit of the diagram
\[\begin{tikzcd}
	{X_0^\flat \times \Fb h_0^{[C,1]}} & {X_0^\flat\times C^\flat} & {X_1}
	\arrow[from=1-2, to=1-1]
	\arrow["\phi", from=1-2, to=1-3]
\end{tikzcd}\]
\end{cor}
\begin{proof}
This is a consequence of proposition  \ref{prop:lfib and W 2}, theorem \ref{theo:gr construction}, and example \ref{exe:of int}.
\end{proof}

\begin{cor}
\label{cor:explicit partial}
Let $E$ be an object of $\ocat_{/[b,1]^\sharp}$ corresponding to a morphism $p:X\to [b,1]^\sharp$. 
Consider the induced cartesian squares:
\[\begin{tikzcd}[sep=small]
	{ X_{0}\times b^\flat} && { X_{/1}} \\
	& { X_{0}} && X \\
	{b^\flat} && {[b,1]^\sharp_{/1}} \\
	& {\{0\}} && {[b,1]^\sharp}
	\arrow[from=1-1, to=2-2]
	\arrow["g", from=1-1, to=1-3]
	\arrow["f", from=1-3, to=2-4]
	\arrow[from=3-3, to=4-4]
	\arrow[from=4-2, to=4-4]
	\arrow[from=3-1, to=3-3]
	\arrow[from=3-1, to=4-2]
	\arrow[from=1-3, to=3-3]
	\arrow[from=2-4, to=4-4]
	\arrow[from=1-1, to=3-1]
	\arrow[from=2-2, to=2-4]
	\arrow[from=2-2, to=4-2]
\end{tikzcd}\]
The arrow associated to $\Fb E$ via the equivalence of corollary \ref{cor:recapitulatif} is equivalent to
\begin{equation}
\label{eq:cor:explicit parital}
(\bot X_{0})\times b\xrightarrow{\bot g} \bot X_{/1}
\end{equation}
where $\bot$ is the functor defined in \ref{defi:of bot}.
\end{cor}
\begin{proof}
We denote $\tilde{X}\to [b,1]^\sharp$ the morphism associated with the left cartesian fibrant replacement of $E$, denoted $\Fb E$.
As $[b,1]^\sharp_{/1}\to [b,1]^\sharp$ and $\{0\}\to [b,1]^\sharp$ are right cartesian fibrations, they are smooth, and the canonical morphisms
$$X_{/1}\to \tilde{X}_{/1}~~~~~~~\mbox{ and }~~~~~~~X_{0}\to \tilde{X}_{0}$$
are initial.
As $\bot$ sends initial morphisms to equivalences, the induced morphisms
$$\bot X_{/1}\to \bot \tilde{X}_{/1}~~~~~~~\mbox{ and }~~~~~~~\bot X_{0}\to \bot \tilde{X}_{0}$$
are equivalences. We can then suppose that $E$ corresponds to a left cartesian fibration.

As $\{1\} \to [b,1]^{\sharp}$ is a right Gray deformation retract, so is the inclusion $X_1\to X_{/1}$ according to proposition \ref{prop:left Gray transfomration stable under pullback along cartesian fibration}. The right Gray deformation retract structure induces a diagram:
\[\begin{tikzcd}
	{X_{/1}\otimes\{0\}} \\
	& {X_{/1}\otimes[1]^\sharp} & {X_{/1}} \\
	{X_{/1}\otimes\{1\}} & {X_1\otimes\{1\}}
	\arrow["id", curve={height=-18pt}, from=1-1, to=2-3]
	\arrow[from=3-1, to=2-2]
	\arrow[from=3-2, to=2-3]
	\arrow["r"', from=3-1, to=3-2]
	\arrow[from=1-1, to=2-2]
	\arrow["\phi"{description}, from=2-2, to=2-3]
\end{tikzcd}\]
By post composing with $g:X_0\otimes b^\flat \to X_{/1}$ and post composing $f:X_{/1}\to X$, we get a diagram:
\[\begin{tikzcd}
	{( X_0\times b^\flat)\otimes\{0\}} & {X_0} \\
	& {( X_0\times b^\flat)\otimes[1]^\sharp} & X \\
	{( X_0\times b^\flat)\otimes\{1\}} & {X_1\otimes\{1\}}
	\arrow[from=1-1, to=2-2]
	\arrow[from=3-1, to=2-2]
	\arrow[from=2-2, to=2-3]
	\arrow[from=1-1, to=1-2]
	\arrow["rg"', from=3-1, to=3-2]
	\arrow[from=3-2, to=2-3]
	\arrow[from=1-2, to=2-3]
\end{tikzcd}\]
Remark furthermore that the following diagram:
\[\begin{tikzcd}
	{( X_0\times b^\flat)\otimes\{0\}} & {X_0\times\{0\}} & {X_0\times \Fb h^{[b,1]}_{0/}} \\
	{( X_0\times b^\flat)\otimes[1]^\sharp} & X & {[b,1]^\sharp}
	\arrow[from=1-1, to=1-2]
	\arrow[from=1-1, to=2-1]
	\arrow[from=1-2, to=1-3]
	\arrow[from=1-3, to=2-3]
	\arrow["l"{description}, dashed, from=2-1, to=1-3]
	\arrow[from=2-1, to=2-2]
	\arrow[from=2-2, to=2-3]
\end{tikzcd}\]
admits a lift $l$. Indeed, the left vertical morphism is initial, and the right vertical one is a left cartesian fibration. 
All put together, we get a diagram 
\[\begin{tikzcd}
	{X_0\times b^\flat} & {X_0\times \Fb h^{[b,1]}_{0}} \\
	{X_{1}} & E
	\arrow[from=1-1, to=1-2]
	\arrow["rg"', from=1-1, to=2-1]
	\arrow[from=1-2, to=2-2]
	\arrow[from=2-1, to=2-2]
\end{tikzcd}\]
where the upper horizontal morphism is induced by the restriction of $l$ to $(X_0\times b^\flat)\otimes\{1\}$.
As $X_1\to X_{/1}$ is initial, we have $ \bot X_1 \sim \bot X_{/1}$ and $\bot r$ is an equivalence. As $X_0$ and $X_1$ are fibers of a left cartesian fibration, they are trivially marked, and we then have $(\bot X_0)^\flat \sim X_0$ and $(\bot X_1)^\flat \sim X_1$.
The previous diagram then corresponds to a diagram
\[\begin{tikzcd}
	{(\bot X_0)^{\flat}\times b^\flat} & {(\bot X_0)^{\flat}\times  \Fb h_{0}^{[b,1]}} \\
	{(\bot X_{/1})^{\flat}} & E
	\arrow[from=1-1, to=1-2]
	\arrow["{(\bot g)^\flat}"', from=1-1, to=2-1]
	\arrow[from=1-2, to=2-2]
	\arrow[from=2-1, to=2-2]
\end{tikzcd}\]

 We denote by $F$ the left fibration associated to the span \eqref{eq:cor:explicit parital} by corollary \ref{cor:recapitulatif}. The previous square then corresponds to a morphism 
$$\int_{[b,1]}F\to E$$
Using the naturality of $\int_{[b,1]}$, one can see that this morphism induces an equivalence on fibers, and is then an equivalence. Applying $\partial_{[b,1]}$ and using theorem \ref{theo:gr construction}, this concludes the proof.
\end{proof}

\begin{definition}
 A left cartesian fibration is \wcnotion{$\U$-small}{small left@$\U$-small left cartesian fibration} if its fibers are $\U$-small $\io$-categories.

 For an $\io$-category $A$, we denote by $\LCart_{\U}(A^\sharp)$ the full sub $\iun$-category of $\LCart(A^\sharp)$ whose objects correspond to $\U$-small left cartesian fibrations over $A^\sharp$.
 \end{definition}
\begin{cor}
\label{cor: Grt equivalence}
Let $\uni$ be the $\V$-small $\io$-category of $\U$-small $\io$-categories and $A$ a $\V$-small $\io$-category. There is an equivalence
$$\int_A:\Hom(A,\uni)\to \tau_0 \LCart_{\U}(A^\sharp)$$
natural in $A:\ocat^{op}$.
\end{cor}
\begin{proof}
This is a direct consequence of the theorem \ref{theo:gr construction} and the definition of $\uni$.
\end{proof}
\begin{cor}
\label{cor: universal fibration}
The left cartesian fibration $\int_{\uni}id$ is the universal left cartesian fibration with $\U$-small fibers, i.e for any left cartesian fibration $X\to A^\sharp$ with $\U$-small fibers, there exists a unique morphism $X\to \uni$ and a unique cartesian square:
\[\begin{tikzcd}
	X & {\dom\int_{\uni}id} \\
	{A^\sharp} & {\uni^\sharp}
	\arrow[from=1-1, to=2-1]
	\arrow[from=2-1, to=2-2]
	\arrow["{\int_{\uni}id}", from=1-2, to=2-2]
	\arrow[from=1-1, to=1-2]
	\arrow["\lrcorner"{anchor=center, pos=0.125}, draw=none, from=1-1, to=2-2]
\end{tikzcd}\]
\end{cor}
\begin{proof}
This is a direct consequence of the corollary \ref{cor: Grt equivalence} and the functoriality of the Grothendieck construction given in proposition \ref{prop: derived int and partial are natural}.
\end{proof}

\subsection{Lax Univalence}
\begin{notation*}
Through this section, we will identify any marked $\io$-category $C$ with the canonical induced morphism $C\to1$. If $f:X\to Y$ is a morphism, $f\times C$ then corresponds to the canonical morphism $X\times C\to Y$.
\end{notation*}

For the remaining of this section, we fix a marked $\io$-category $I$. As $[n]^\sharp_{\{k\}/} \sim [n-k]^\sharp$, proposition \ref{prop:explicit factoryzation} implies that $\Fb h^{[n]}_k$, the left cartesian fibrant replacement of $h_k:\{k\}\to [n]$, corresponds to the inclusion $(d_{0}^\sharp)^k:[n-k]^\sharp\to [n]^\sharp$.

\begin{construction}
 We define the functor \index[notation]{(intt@${{\oint}_{n,I}}$}
$$\oint_{n,I}: \Fun([n],\ocatm_{/I})\to \ocatm_{/I\otimes[n]^\sharp}$$
whose value on a morphism $E:[n]\to \ocatm_{/I}$ corresponding to a sequence $E_0\to ....\to E_n$, is
$$\oint_{n,I}E:=\colim_{m} \coprod_{i_0\leq... \leq i_m\leq n} E_{i_0}\otimes \Fb h_{i_m}^{[n]}.$$
As this functor is colimit preserving, it induces an adjunction \index[notation]{(partiall@$\ringpartial_{n,I}$}
\begin{equation}
\label{eq:Gr adj lax 1}
\begin{tikzcd}
	{\oint_{n,I}:\Fun([n],\ocatm_{/I})} & {\ocatm_{/I\otimes[n]^\sharp}:\ringpartial_{n,I}}
	\arrow[""{name=0, anchor=center, inner sep=0}, shift left=2, from=1-1, to=1-2]
	\arrow[""{name=1, anchor=center, inner sep=0}, shift left=2, from=1-2, to=1-1]
	\arrow["\dashv"{anchor=center, rotate=-90}, draw=none, from=0, to=1]
\end{tikzcd}
\end{equation}
\end{construction}

\begin{lemma}
\label{lemma:oint preserves init}
The functor $\oint_{n,I}$ sends a natural transformation that is pointwise initial to an initial morphism.
\end{lemma}
\begin{proof}
As initial morphisms are closed under colimits, we have to show that for any integer $k$, and any morphism $E\to F$ of $\ocatm_{/I}$ corresponding to a triangle $X\xrightarrow{i} Y\to I$, the induced morphism $X\otimes [n-k]^\sharp\to Y\otimes [n-k]^\sharp$ over $I\otimes[n]^\sharp$ is initial whenever $i$ is. For this, remark that there is a square
\[\begin{tikzcd}
	{X\otimes\{0\}} & {X\otimes[n-k]^\sharp} \\
	{Y\otimes\{0\}} & {Y\otimes[n-k]^\sharp}
	\arrow["i", from=1-1, to=2-1]
	\arrow[from=2-1, to=2-2]
	\arrow[from=1-1, to=1-2]
	\arrow[from=1-2, to=2-2]
\end{tikzcd}\]
where the two horizontal morphisms are initial.
By stability by composition and left cancellation of initial morphism, this implies the result.
\end{proof}

\begin{construction}
According to the last lemma, the adjunction \eqref{eq:Gr adj lax 1} induces a derived adjunction
\begin{equation}
\label{eq:Gr adj lax 2}
\begin{tikzcd}
	{\Lb \oint_{n,I}:\Fun([n],\LCart(I))} & {\LCart(I\otimes[n]^\sharp):\Rb \ringpartial_{n,I}}
	\arrow[""{name=0, anchor=center, inner sep=0}, shift left=2, from=1-1, to=1-2]
	\arrow[""{name=1, anchor=center, inner sep=0}, shift left=2, from=1-2, to=1-1]
	\arrow["\dashv"{anchor=center, rotate=-90}, draw=none, from=0, to=1]
\end{tikzcd}
\end{equation}
where $\Rb \ringpartial_{n,I}$ is just the restriction of $\ringpartial_{n,I}$ to $\LCart(I\otimes[n]^\sharp)$.
\end{construction}

\begin{remark}
Let $E:[n]\to \ocatm_{/I}$ corresponding to a sequence $E_0\to \dots \to E_n$. By construction, we have
$$\Lb \oint_{n,I}E:=\colim_{m} \coprod_{i_0\leq\dots \leq i_m\leq n} \Fb (E_{i_0}\otimes \Fb h_{i_m}^{[n]}).$$
\end{remark}

\begin{lemma}
\label{lemma:ringpartial fiber}
Let $i:[n]^\sharp \to [m]^\sharp$ and $j:I\to J$ be two morphisms. Let $E$ be an object of $\LCart(I\otimes [m]^\sharp)$. The natural transformation
$$\ringpartial_{n,I} (j\otimes i)^*E\to j^*\circ \ringpartial_{m,J} E\circ i^\natural $$
 is an equivalence.
\end{lemma}
\begin{proof}
As invertible natural transformations are detected pointwise, one can suppose that $n=0$, and let $k$ be the image of $[0]$ by $i$.
Let $E_0\to E_1\to.. \to E_m $ be the sequence of morphisms of $\LCart(J)$ corresponding to $\ringpartial_{m,J} E$.

The object $ j^*\circ \ringpartial_{m,J} E\circ i^\natural$ is then equivalent to $ j^* E_k$ by definition. 
As $\ringpartial_{0,I}$ is the identity, we have to show that the canonical morphism $(j\otimes \{k\})^*E\to j^*E_k$ is an equivalence. Remark that for any
 $F$ of $\ocatm_{/I}$, we have by adjunction a commutative square: 
\[\begin{tikzcd}
	{\Hom(F,(j\otimes  \{k\})^*E)} & {\Hom(F,j^* E_k)} \\
	{\Hom((j\otimes  \{k\})_!F,E)} & {\Hom((j_!F)\otimes h_k^{[n]},E)}
	\arrow[from=1-1, to=1-2]
	\arrow["\sim", from=1-1, to=2-1]
	\arrow["\sim", from=1-2, to=2-2]
	\arrow[from=2-1, to=2-2]
\end{tikzcd}\]
where the two vertical morphisms are equivalences. As $((j\otimes \{k\})_!F\sim (j_!F)\otimes h_k^{[n]}$, the lower morphism is an equivalence, and so is the top one. This implies the desired result.
\end{proof}

\vspace{1cm} In the following lemmas and proposition, we focus on the case where $I$ is of the form $A^\sharp$, where everything happens more simply thanks to the equivalence $A^\sharp\otimes[n]^\sharp\sim A^\sharp\times[n]^\sharp$ provided by proposition \ref{prop:associativity of Gray amput 1.5}.
\begin{lemma}
\label{lemme:oint a sharp is natural1}
Let $j:A\to B$ be a morphism between $\io$-categories and $i:[n]\to [m]$ a morphism of $\Delta$. Let $E$ be an object of $\Fun([n],\LCart(A^\sharp))$.
The canonical morphism 
$$\Lb\oint_{n,A^\sharp}( \Rb j^*\circ E\circ i)\to \Rb(j\times i^\sharp)^* \Lb \oint_{m,B^\sharp} E$$
is an equivalence.
\end{lemma}
\begin{proof}
As equivalences in $\Fun([m],\LCart(B^\sharp))$ are detected on points, an equivalences on $\LCart(B^\sharp\times [m]^\sharp)$ are detected on fibers, we can suppose that $n=0$, $A=1$, and we denote by $k$ the image of $i$ and $a$ the image of $B$. As $\Lb\oint_{0,1}$ is the identity, one has to show that the canonical morphism 
\begin{equation}
\label{eq:equationoint a sharp is natural1}
\Rb a^* E_k\to \Rb(a\times \{k\})^* \Lb \oint_{m,B^\sharp} E
\end{equation}
 is an equivalence.

Moreover, for any $l\leq n$, the proposition \ref{prop:cotimes 1 to ctimes 1 is a trivialization} implies that the canonical morphism $\Fb(E_l\otimes \Fb h_l^{[n]})\to E_l\times \Fb h_l^{[n]}$ is an equivalence, as this two left cartesian fibrations are replacement of $E_l\otimes h_l^{[n]}\sim E_l\times h_l^{[n]}$.
Furthermore, we have $\Rb\{k\}^* h_l^{[n]}\sim \emptyset$ if $l>k$ and $\{k\}^* h_l^{[n]}\sim 1$ if not. As a consequence, 
$$\Rb(a\times \{k\})^*\Fb(E_l\otimes \Fb h_l^{[n]})\sim \emptyset\mbox{ if $l>k$, and }\Rb(a\times \{k\})^*\Fb(E_l\otimes \Fb h_l^{[n]})\sim \Rb a^*E_l\mbox{ if not.}$$ According to proposition \ref{prop:fiber preserves colimits}, $\Rb(a\times \{k\})^*$ preserves colimits, we then have 
$$ \Rb(a\times \{k\})^* \Lb \oint_{m,B^\sharp} E\sim \colim_m\coprod_{i_0\leq ...\leq i_m\leq k}\Rb a^*E_{i_0}\sim \colim_{i:[k]}\Rb a^*E_i\sim \Rb a^*E_k.$$
The morphism \eqref{eq:equationoint a sharp is natural1} is then an equivalence, which concludes the proof.
\end{proof}

\begin{prop}
\label{prop:ring partial is natural}
The functor $\Rb\ringpartial_{n,I}$ is natural in $n:\Delta^{op}$ and $I:\ocatm^{op}$. The functor $\Lb \oint_{n,A^\sharp}$ is natural in $n:\Delta^{op}$ and $A:\ocat^{op}$.
\end{prop}
\begin{proof}
The proof is similar to the one of proposition \ref{prop: derived int and partial are natural}, using lemma \ref{lemma:ringpartial fiber} and lemma \ref{lemme:oint a sharp is natural1} instead of lemma \ref{lemma:partial fiber} and lemma \ref{lemma:int fiber 2}. 
 \end{proof}

\begin{prop}
\label{prop:lax gr construction particular case}
For any $\io$-category $A$ and any integer $n$, the adjunction 
\[\begin{tikzcd}
	{\Lb \oint_{n,A^\sharp}:\Fun([n],\LCart(A^\sharp))} & {\LCart((A\times[n])^\sharp):\Rb \ringpartial_{n,A^\sharp}}
	\arrow[""{name=0, anchor=center, inner sep=0}, shift left=2, from=1-1, to=1-2]
	\arrow[""{name=1, anchor=center, inner sep=0}, shift left=2, from=1-2, to=1-1]
	\arrow["\dashv"{anchor=center, rotate=-90}, draw=none, from=0, to=1]
\end{tikzcd}\]
 is an adjoint equivalence.
\end{prop}
\begin{proof}
As in both case equivalences are detected on fibers, and as these functors are natural in $A$ and $n$, one can show the result for $A$ being the terminal $\io$-category and $n=0$. In this case remark that these two functors are the identities. 
\end{proof}

\vspace{1cm} We return to the general case.

\begin{construction}
We set \wcnotation{$\Fun^c([n],\LCart(I))$}{(func@$\Fun^c([\uvar],\uvar)$} as the pullback
\[\begin{tikzcd}
	{\Fun^c([n],\LCart(I))} & {\Fun([n],\LCart(I))} \\
	{\prod_{k\leq n}\LCart(I^\sharp)} & {\prod_{k\leq n}\Fun(\{k\},\LCart(I))}
	\arrow[from=1-1, to=2-1]
	\arrow[""{name=0, anchor=center, inner sep=0}, from=2-1, to=2-2]
	\arrow[from=1-2, to=2-2]
	\arrow[from=1-1, to=1-2]
	\arrow["\lrcorner"{anchor=center, pos=0.125}, draw=none, from=1-1, to=0]
\end{tikzcd}\]
where $I^\sharp$ stand for $(I^\natural)^\sharp$. An object of this $\iun$-category is then a sequence in $\LCart(I)$:
\[\begin{tikzcd}
	{F_0} & {...} & {F_ n}
	\arrow[from=1-1, to=1-2]
	\arrow[from=1-2, to=1-3]
\end{tikzcd}\]
such that for any integer $i\leq n$, $F_i$ is classified. 
A $1$-cell of this $\iun$-category is a sequence of square in $\LCart(I)$:
\[\begin{tikzcd}
	{F_0} & {...} & {F_ n} \\
	{G_0} & {...} & {G_n}
	\arrow[from=1-1, to=1-2]
	\arrow[from=1-2, to=1-3]
	\arrow[from=1-2, to=2-2]
	\arrow[from=1-1, to=2-1]
	\arrow[from=1-3, to=2-3]
	\arrow[from=2-1, to=2-2]
	\arrow[from=2-2, to=2-3]
\end{tikzcd}\]
such that for any $k\leq n$, the morphism $F_k\to G_k$ comes from a morphism beetwen the corresponding objects of $\LCart(I^\sharp)$.
\end{construction}

\begin{prop}
\label{prop:Fun preserve colimies}
Let $F:I\to \ocatm$ be a $\Wcard$-small diagram. The canonical functor
$$\Fun^c([n],\LCart(\colim_I F))\to \lim_I\Fun^c([n],\LCart(F))$$
is an equivalence.
\end{prop}
\begin{proof}
This morphism fits in an adjunction:
\[\begin{tikzcd}
	{\colim_I:\lim_I\Fun^c([n],\LCart(F))} & {\Fun^c([n],\LCart(\colim_I F))}
	\arrow[""{name=0, anchor=center, inner sep=0}, shift left=2, from=1-2, to=1-1]
	\arrow[""{name=1, anchor=center, inner sep=0}, shift left=2, from=1-1, to=1-2]
	\arrow["\dashv"{anchor=center, rotate=-90}, draw=none, from=1, to=0]
\end{tikzcd}\]
The corollary \ref{cor:fib over a colimit} implies that the counit of this adjunction is an equivalence. To conclude, we have to show that the right adjoint is essentially surjective. On objects, this adjunction corresponds to the canonical equivalence 
$$\lim_I\Hom([n],\LCartc(F))\sim \Hom([n],\LCartc(\colim_IF))$$
induced by corollary \ref{cor:fib over a colimit2}
\end{proof}

\begin{construction}
 As $\Rb\ringpartial_{0,I}$ is the identity, lemma \ref{lemma:ringpartial fiber} implies that the functor
$$\LCart((I\otimes[n]^\sharp)^\sharp)\to \LCart(I\otimes[n]^\sharp) \xrightarrow{\Rb\ringpartial_{n,I}} \Fun([n],\LCart(I))$$
 factors through a functor \index[notation]{(partialll@$\ringpartial_{n,I}^c$}
 \begin{equation}
 \label{eq:def of right partial classified}
\ringpartial_{n,I}^c:\LCart((I\otimes[n]^\sharp)^\sharp)\to \Fun^c([n],\LCart(I))
\end{equation}
\end{construction}
We are now willing to show that this functor is an equivalence, and to this extent, we will construct an inverse.

\begin{notation}
For $a$ an object of $t\Theta$. We set $[a,1]^\sharp:=([a,1]^\natural)^\sharp$
 and we denote $\iota$ the canonical inclusion $[a,1]\to [a,1]^\sharp$.
 \end{notation}

\begin{lemma}
\label{lemma:replacement of unmarked slice}
Let $a$ be an object of $t\Theta$.
We have an equivalence
$$\Lb\iota_!\Rb\iota^*\Fb h_{1}^{[a^\natural,1]}\sim \Fb h_{1}^{[a^\natural,1]}$$
and an equivalence
$$\Lb\iota_!\Rb\iota^*\Fb h_{0}^{[a^\natural,1]}\sim \Fb h_{0}^{[a^\natural,1]}\coprod_{a^\flat\otimes\{0\}}(a\otimes[1]^\sharp)^\flat.$$
Moreover the morphism 
$\Lb\iota_!(a^\flat \to \Fb h_{0}^{[a^\natural,1]})$ corresponds to the inclusion 
$$(a\otimes\{0\})^\flat \to (a\otimes[1]^\sharp)^\flat \to \Fb h_{0}^{[a^\natural,1]}\coprod_{a^\flat\otimes\{0\}}(a\otimes[1]^\sharp)^\flat.$$
\end{lemma}
\begin{proof}
The first assertion is trivial after noting that $\Fb h_{1}^{[a^\natural,1]}$ is the inclusion $\{1\}\to [a,1]^\sharp$. The theorem \ref{theo:equivalence between slice and join} implies that 
$\iota_!\Rb\iota^*\Fb h_{0}^{[b,1]}$ and $\iota_!\Rb\iota^*\Fb h_{0}^{[(\Db_n)_t,1]}$ are respectively equivalent to 
$$(1\costar b)^\flat\to [b,1]^\sharp~~~~ \mbox{and}~~~~ (1\costar \Db_n)^{\sharp_{n+1}}\to [\Db_n,1]^\sharp$$
The theorem \ref{theo:formula between pullback of slice and tensor marked case} induces cartesian diagrams
\[\begin{tikzcd}[sep =0.1cm]
	{b^\flat\otimes\{0\}} && {(b\otimes[1])^\flat} && {\Db_n^\flat\otimes\{0\}} && {(\Db_n\otimes[1])^{\sharp_{n+1}}} \\
	& 1 && {(1\costar b)^\flat} && 1 && {(1\costar \Db_n)^{\sharp_{n+1}}} \\
	{b^\flat} && {b^\flat\star 1} && {\Db_n^\flat} && {\Db_n^\flat\star 1} \\
	& {\{0\}} && {[b,1]^\sharp} && {\{0\}} && {[\Db_n,1]^\sharp}
	\arrow[from=4-2, to=4-4]
	\arrow[from=2-4, to=4-4]
	\arrow[from=1-3, to=3-3]
	\arrow[from=1-1, to=3-1]
	\arrow[from=2-2, to=4-2]
	\arrow[from=3-1, to=4-2]
	\arrow[from=3-3, to=4-4]
	\arrow[from=1-3, to=2-4]
	\arrow[from=1-1, to=2-2]
	\arrow[from=1-1, to=1-3]
	\arrow[from=2-2, to=2-4]
	\arrow[from=3-1, to=3-3]
	\arrow[from=1-5, to=2-6]
	\arrow[from=1-5, to=1-7]
	\arrow[from=1-7, to=2-8]
	\arrow[from=3-7, to=4-8]
	\arrow[from=3-5, to=4-6]
	\arrow[from=1-5, to=3-5]
	\arrow[from=2-6, to=4-6]
	\arrow[from=1-7, to=3-7]
	\arrow[from=2-8, to=4-8]
	\arrow[from=3-5, to=3-7]
	\arrow[from=4-6, to=4-8]
	\arrow[from=2-6, to=2-8]
\end{tikzcd}\]
Remark furthermore that we have an equivalence
$$\bot(\Db_n\otimes[1])^{\sharp_{n+1}}\sim \tau^i_{n}(\Db_n\otimes[1])=: ((\Db_n)_t\otimes[1]^\sharp)^\natural.$$
Applying the full duality to theorem \ref{theo:equivalence between slice and join} and using the corollary \ref{cor:explicit partial}, this proves the first assertion.

The second assertion follows from the naturality in $E$ of the construction given in corollary \ref{cor:explicit partial} and from the squares
\[\begin{tikzcd}
	{b^\flat\otimes\{1\}} & {b^\flat} & {\Db_n^\flat\otimes\{1\}} & {\Db_n^\flat} \\
	{(b\otimes[1])^\flat} & {(1\costar b)^\flat} & {(\Db_n\otimes[1])^{\sharp_{n+1}}} & {(1\costar \Db_n)^{\sharp_{n+1}}} \\
	{b^\flat\star 1} & {[b,1]^\sharp} & {\Db_n^\flat\star 1} & {[\Db_n,1]^\sharp}
	\arrow[from=2-1, to=3-1]
	\arrow[from=2-1, to=2-2]
	\arrow[from=3-1, to=3-2]
	\arrow[from=2-2, to=3-2]
	\arrow[from=3-3, to=3-4]
	\arrow[from=2-4, to=3-4]
	\arrow[from=2-3, to=3-3]
	\arrow[from=2-3, to=2-4]
	\arrow[from=1-1, to=2-1]
	\arrow[from=1-1, to=1-2]
	\arrow[from=1-2, to=2-2]
	\arrow[from=1-4, to=2-4]
	\arrow[from=1-3, to=2-3]
	\arrow[from=1-3, to=1-4]
\end{tikzcd}\]
that are cartesian according to theorem \ref{theo:formula between pullback of slice and tensor marked case}.
\end{proof}

\begin{lemma}
\label{lemma:explicit iota excalmation}
Let $a$ be an object of $t\Theta$ and $E$ an object of $\LCart([a,1]^\sharp)$. 
The left cartesian fibration $\Lb\iota_!\Rb \iota^*E$ is the left cartesian fibration
$$
X_0^\flat\times (\Fb h_0^{[a^\natural,1]}\coprod_{ a^\flat} (a\otimes[1]^\sharp)^\flat) \coprod_{X_0^\flat\times (a\otimes\{1\})^\flat}X_1^\flat 
$$
where $X_0\times a^\natural \to X_1$ is the arrow corresponding to $E$ via the equivalence given in corollary \ref{cor:recapitulatif}.
\end{lemma}
\begin{proof}
By corollary \ref{cor:recapitulatif}, $E$ corresponds to the left cartesian fibration
$$X_0^\flat\times \Fb h_0^{[a^\natural,1]}\coprod_{X_0^\flat\times a^\flat}X_1^\flat.$$
Lemma \ref{lemma:replacement of unmarked slice} provides an initial morphism from $\iota_!\Rb \iota^*E$ to this object, and theorem \ref{theo:left cart stable by colimit} implies that this object is a left cartesian fibration.
\end{proof}

\begin{lemma}
\label{lemma:characterisation of natural transoformation}
Let $\psi: \iota_!\Rb \iota^*\to \Lb\iota_!\Rb \iota^*$ be a natural transformation, endowed with a family of natural commutative squares:
\[\begin{tikzcd}
	{ \iota_!\Rb \iota^*(B^\flat\times E)} & {\Lb\iota_!\Rb \iota^*(B^\flat\times E)} \\
	{B^\flat \times\iota_!\Rb \iota^*E} & {B^\flat \times\iota_!\Rb \iota^*E}
	\arrow["{\psi_{B^\flat\times E}}", from=1-1, to=1-2]
	\arrow[from=1-1, to=2-1]
	\arrow["{B^\flat\times\psi_{E}}"', from=2-1, to=2-2]
	\arrow[from=1-2, to=2-2]
\end{tikzcd}\]
where we identify marked $\io$-categories with their canonical morphims to the terminal marked $\io$-category.

The natural transformation $\psi$ is then the one obtained by the functorial factorization in initial morphisms followed by left cartesian fibrations. 
\end{lemma}
\begin{proof}
The natural transformation $\psi$ induces a natural transformation $\Db \psi:\Lb\iota_!\Rb \iota^*\to \Lb\iota_!\Rb \iota^*$ and we have to check that this last natural transformation is the identity. The corollary \ref{cor:recapitulatif} states that $E$ is a colimit of left cartesian fibration of shape $B^\flat \times \Fb h^{[a^\natural,1]}_{\epsilon}$ for $\epsilon\in \{0,1\}$. The hypothesis implies that we just have to show that $\Db \psi_{\Fb h^{[a^\natural,1]}_{0}}$ and $\Db \psi_{\Fb h^{[a^\natural,1]}_{1}}$ are equivalences, and we will check this on fibers.

Using the explicit expression of $\Lb\iota_!\Rb \iota$ given in lemma \ref{lemma:explicit iota excalmation}, we have equivalences
$$\{0\}^*\Lb\iota_!\Rb \iota \Fb h_0^{[a^\natural,1]} \sim 1~~~~~~~~
\{0\}^*\Lb\iota_!\Rb \iota \Fb h_1^{[a^\natural,1]} \sim \emptyset~~~~~~~~
\{1\}^*\Lb\iota_!\Rb \iota \Fb h_0^{[a^\natural,1]} \sim 1$$
which directly implies that $\{0\}^*\Db \psi_{\Fb h_0^{[a^\natural,1]}}$, $\{0\}^*\Db \psi_{\Fb h_1^{[a^\natural,1]}}$ and $\{1\}^*\Db \psi_{\Fb h_1^{[a^\natural,1]}}$ are equivalences. The only case remaining is $\{1\}^*\Db \psi_{\Fb h_0^{[a^\natural,1]}}$. This morphism corresponds to an endomorphism of $(a\otimes[1]^\sharp)^\natural$, which is a strict object according to \ref{prop:tensor of glboer are strics}. By right cancellation, the morphism induced by the domain of $\Db\psi_{\Fb h_0^{[a^\natural,1]}}$ is a left cartesian fibration. There exists then a lift in the following diagram
\begin{equation}
\label{eq:square in proof of replement ofneofoeijfoepaj}
\begin{tikzcd}[cramped]
	{\{0\}} & {[a,1]^{\sharp}_{0/}\coprod_{a^\flat\otimes\{0\}}(a\otimes[1]^\sharp)^\flat} \\
	{[a,1]^{\sharp}_{0/}} & {[a,1]^{\sharp}_{0/}\coprod_{a^\flat\otimes\{0\}}(a\otimes[1]^\sharp)^\flat}
	\arrow[from=1-1, to=1-2]
	\arrow[from=1-1, to=2-1]
	\arrow["{\dom\Db\psi_{\Fb h_0^{[a^\natural,1]}}}", from=1-2, to=2-2]
	\arrow["l"{description}, from=2-1, to=1-2]
	\arrow["\gamma"', from=2-1, to=2-2]
\end{tikzcd}
\end{equation}
where $\gamma$ is the canonical inclusion. As $l$ and $\gamma$ are lifts in the following diagram:
\[\begin{tikzcd}
	{\{0\}} & {[a,1]^{\sharp}_{0/}\coprod_{a^\flat\otimes\{0\}}(a\otimes[1]^\sharp)^\flat} \\
	{[a,1]^{\sharp}_{0/}} & {[a,1]^\sharp}
	\arrow[from=1-2, to=2-2]
	\arrow[from=1-1, to=2-1]
	\arrow[from=2-1, to=2-2]
	\arrow[from=1-1, to=1-2]
	\arrow[from=2-1, to=1-2]
\end{tikzcd}\]
they are equivalent. Taking the fiber on $\{1\}$ of the  square \eqref{eq:square in proof of replement ofneofoeijfoepaj}, this induces a commutative triangle:
\[\begin{tikzcd}
	{(a\otimes\{0\})^\natural} & {(a\otimes[1]^\sharp)^\natural} \\
	& {(a\otimes[1]^\sharp)^\natural}
	\arrow[from=1-1, to=1-2]
	\arrow["{\{1\}^*\Db \psi_{\Fb h_0^{[a^\natural,1]}}}", from=1-2, to=2-2]
	\arrow[from=1-1, to=2-2]
\end{tikzcd}\]
Eventually, the naturality induces a commutative squares. 
\[\begin{tikzcd}[cramped]
	{(a\otimes\{1\})^\natural} & {(a\otimes[1]^\sharp)^\natural} \\
	{(a\otimes\{1\})^\natural} & {(a\otimes[1]^\sharp)^\natural}
	\arrow[from=1-1, to=1-2]
	\arrow["{\{1\}^*\Db\psi_{a^\flat\times \Fb h_1^{[a^\natural,1]}}\sim id}"', from=1-1, to=2-1]
	\arrow["{\{1\}^*\Db\psi_{\Fb h_0^{[a^\natural,1]}}}", from=1-2, to=2-2]
	\arrow[from=2-1, to=2-2]
\end{tikzcd}\]
The restriction of the morphism $\{1\}^*\Db\psi_{\Fb h_0^{[a^\natural,1]}}:(a\otimes[1]^\sharp)^\natural\to (a\otimes[1]^\sharp)^\natural$ to $a\otimes\{0\}$ and $a\otimes\{1\}$ is therefore the identity.
Using Steiner theory, we can easily show that it forces $\{1\}^*\Db\psi_{\Fb h_0^{[a^\natural,1]}}$ to also be the identity.
\end{proof}

\begin{construction}
Let $F$ be an object of $\LCart([a,1]^\sharp)$, and $\phi:\Rb \iota^* E\to \Rb \iota^* F$ a morphism. By adjunction, this corresponds to a morphism $\tilde{\phi}: \iota_! \Rb\iota^*E\to F$, and as $F$ corresponds to a left cartesian fibration, this induces a morphism $\Db\tilde{\phi}:\Lb\iota_! \Rb\iota^*E\to F$. Using once again theorem \ref{theo:gr construction}, this induces a morphism 
$\partial_{[a^\natural,1]} \Lb\iota_! \Rb\iota^*E\to \partial_{[a^\natural,1]} F$, that corresponds, according to the explicit expression of $\Lb\iota_!\Rb \iota$ given in lemma \ref{lemma:explicit iota excalmation}, to a commutative square
\[\begin{tikzcd}
	{X_0\times a^\natural} & {Y_0\times a^\natural} \\
	{X_0\times (a\otimes[1]^\sharp)^\natural \coprod_{X_0\times a^\natural}X_1} & {Y_1}
	\arrow["{\Db \tilde{\phi}_0\times a^\natural}", from=1-1, to=1-2]
	\arrow[from=1-1, to=2-1]
	\arrow[from=1-2, to=2-2]
	\arrow["{\Db \tilde{\phi}_1}"', from=2-1, to=2-2]
\end{tikzcd}\]
where $Y_0\times a^\flat \to Y_1$ corresponds to $\partial_{[a^\natural,1]} F$ via the equivalence given in proposition \ref{prop:lfib and W 2}.
This is equivalent to a diagram
\begin{equation}
\label{eq:lax technical big diagram}
\begin{tikzcd}
	{X_0\times a^\natural} && {Y_0\times a^\natural} \\
	& {X_0\times (a\otimes [1]^\sharp)^\natural} && {Y_1} \\
	{X_0\times a^\natural} && {X_1}
	\arrow["{\Db \tilde{\phi}_0\times a^\natural}", from=1-1, to=1-3]
	\arrow[from=1-1, to=2-2]
	\arrow[from=1-3, to=2-4]
	\arrow[from=2-2, to=2-4]
	\arrow[from=3-1, to=2-2]
	\arrow[from=3-1, to=3-3]
	\arrow[from=3-3, to=2-4]
\end{tikzcd}
\end{equation}
According to proposition \ref{prop:lfib and W 3}, this corresponds to an object $\xi(\phi)$ of $\Lfib(\Noiun([a,1]\otimes[1]^\sharp)^\natural))$ endowed with two equivalences:
$$\partial_{[a^\natural,1]}E\sim \Noiun([a,1]\otimes\{0\})^*\xi(\phi)~~~~~~~
\partial_{[a^\natural,1]} F\sim \Noiun([a,1]\otimes\{1\})^*\xi(\phi)$$
Using the naturality of $\int_C$ demonstrated in proposition \ref{prop: derived int and partial are natural}, these equivalences induce equivalences:
\begin{equation}
\label{eq:fiber of xi}
E\sim ([a,1]\otimes\{0\})^*\int_{([a,1]\otimes[1]^\sharp)^\natural}\xi(\phi)~~~~~~~
 F\sim ([a,1]\otimes\{1\})^*\int_{([a,1]\otimes[1]^\sharp)^\natural}\xi(\phi)
\end{equation}
All the operations we performed were functorial and admitted inverses. 
We then have constructed an equivalence
\begin{equation}
\label{eq:inverse of ring partial}
\int_{([a,1]\otimes[1]^\sharp)^\natural}\xi:\Fun^c([1],\LCart([a,1]))\to \LCart(([a,1]\otimes[1]^\sharp)^\sharp)
\end{equation}
\end{construction}
 
\begin{lemma}
\label{lemma:lax Gr construction technical}
Let $a$ be an object of $t\Theta$,  $E,F$  two objects of $\LCart([a,1]^\sharp)$ and $\phi:\Rb \iota^* E\to \Rb \iota^* F$ a morphism.
By adjunction, this corresponds to a morphism $\tilde{\phi}: \iota_! \Rb\iota^*E\to F$.

There is a unique commutative square of shape
\begin{equation}
\label{eq:lemma:lax Gr construction technical}
\begin{tikzcd}
	{\iota_!\iota^*E\otimes \{0\}} & {E\otimes\{0\}} \\
	{\iota_!\iota^*E\otimes id_{[1]^\sharp}} & {\int_{([a,1]\otimes[1]^\sharp)^\natural}\xi(\phi)} \\
	{\iota_!\iota^*E\otimes\{1\}} & {F\otimes \{1\}}
	\arrow[from=1-2, to=2-2]
	\arrow[from=3-2, to=2-2]
	\arrow[from=1-1, to=2-1]
	\arrow[from=3-1, to=2-1]
	\arrow[from=3-1, to=3-2]
	\arrow[from=2-1, to=2-2]
	\arrow[from=1-1, to=1-2]
\end{tikzcd}
\end{equation}
where the upper horizontal morphism is induced by the unit of the adjunction $(\iota_!,\iota^*)$.

Moreover, the bottom horizontal morphism is $\tilde\phi$.
\end{lemma}
\begin{proof}
The unicity and existence of the middle horizontal morphism come from the initiality of the morphism $\iota_!\iota^*E\otimes \{0\}\to \iota_!\iota^*E\otimes [1]^\sharp$. The unicity and existence of the lower horizontal morphism is a consequence of the equation \eqref{eq:fiber of xi}. 
As the diagram \eqref{eq:lax technical big diagram} factors as
\[\begin{tikzcd}
	& {Y(0)\times a^\natural} \\
	{X(0)\times a^\natural} & {X(0)\times a^\natural} & {Y(1)} \\
	& {X(0)\times (a\otimes [1]^\sharp)^\natural} & {X(0)\times (a\otimes [1]^\sharp)^\natural\coprod_{X(0)\times a^\natural}X(1)} \\
	{X(0)\times a^\natural} & {X(1)}
	\arrow[from=1-2, to=2-3]
	\arrow[from=4-1, to=4-2]
	\arrow[from=2-1, to=3-2]
	\arrow[from=4-1, to=3-2]
	\arrow["{\Db \tilde{\phi}(0)\times a^\natural}", from=2-1, to=2-2]
	\arrow[from=3-2, to=3-3]
	\arrow[from=4-2, to=3-3]
	\arrow[from=2-2, to=1-2]
	\arrow[from=2-2, to=3-3]
	\arrow[from=3-3, to=2-3]
\end{tikzcd}\]
the downer square of the diagram of \eqref{eq:lemma:lax Gr construction technical} factors as
\[\begin{tikzcd}
	{\iota_!\iota^*E\otimes [1]^\sharp} & {\int_{([a,1]\otimes[1]^\sharp)^\natural}\xi(\mu_E)} & {\int_{([a,1]\otimes[1]^\sharp)^\natural}\xi(\phi)} \\
	{\iota_!\iota^*E\otimes\{1\}} & {\Lb \iota_! \iota^*E\otimes\{1\}} & {F\otimes \{1\}}
	\arrow[from=2-3, to=1-3]
	\arrow[from=2-1, to=1-1]
	\arrow[from=2-1, to=2-2]
	\arrow[from=1-1, to=1-2]
	\arrow[from=2-2, to=1-2]
	\arrow[from=1-2, to=1-3]
	\arrow["{\Db \tilde{\phi}}"', from=2-2, to=2-3]
\end{tikzcd}\]
where $\mu_E$ denotes the canonical morphism $\iota_!\iota^*E\to \Lb \iota_! \iota^*E$. To conclude, one has to show that the lower left horizontal morphism is $\mu_E$. As these constructions are natural, and commute with the cartesian product with $B^\flat\to 1$ for $B$ an $\io$-category, the lemma \ref{lemma:characterisation of natural transoformation} implies the desired result.
\end{proof}

\begin{lemma}
\label{lemma:ring partial eq for a 1}
Let $a$ be an object of $t\Theta$.
The functor $\ringpartial^c_{1,[a,1]}$ defined in \eqref{eq:def of right partial classified}  is an equivalence.
\end{lemma}
\begin{proof}
The lemma \ref{lemma:lax Gr construction technical} induces a diagram
\[\begin{tikzcd}
	{\iota^*E\otimes\Fb h^{[1]}_1} & {\iota^*E\otimes\Fb h^{[1]}_0} \\
	{\iota^*F\otimes\Fb h^{[1]}_1} & {(\iota\otimes id_{[1]})^*\int_{([a,1]\otimes[1]^\sharp)^\natural}\xi(\phi)}
	\arrow[from=2-1, to=2-2]
	\arrow[from=1-1, to=2-1]
	\arrow[from=1-1, to=1-2]
	\arrow[from=1-2, to=2-2]
\end{tikzcd}\]
which corresponds to a natural transformation 
$$\oint_{1,[a,1]} \phi \to (\iota\otimes id_{[1]})^* \int_{([a,1]\otimes[1]^\sharp)^\natural}\xi(\phi)~~~\leftrightsquigarrow~~~ \phi\to \ringpartial^c_{1,[a,1]}\int_{([a,1]\otimes[1]^\sharp)^\natural}\xi(\phi)$$
Eventually, remark that proposition \ref{prop:ring partial is natural} and the equivalences \eqref{eq:fiber of xi} imply that this natural transformation is pointwise an equivalence. 
The functor \eqref{eq:inverse of ring partial} is then a left inverse of $\ringpartial^c_{1,[a,1]}$. As it is an equivalence, so is $\ringpartial^c_{1,[a,1]}$.
\end{proof}

\begin{prop}
\label{prop:ring partial eq for I n}
For any marked $\io$-category $I$, and integer $n$, the morphism 
$$\ringpartial^c_{n,I}:\LCart((I\otimes[n]^\sharp)^\sharp) \to \Fun^c([n],\LCart(I))$$
defined in \eqref{eq:def of right partial classified} is an equivalence of $\iun$-categories.
\end{prop}
\begin{proof}
Corollary \ref{cor:fib over a colimit2}, and propositions \ref{prop:otimes marked preserves colimits} and \ref{prop:Fun preserve colimies} imply that the two functors on $\Delta^{op}\times \ocatm^{op}$:
$$\begin{array}{rcl}
(n,I)&\mapsto & \LCartc(I\otimes[n]^\sharp)\\
(n,I)&\mapsto &\Fun^c([n],\LCartc(I))
\end{array}$$
send colimits to limits. We can then reduce to the case where $I$ is an element of $t\Theta$ and $n=1$. 
If $I$ is $[1]^\sharp$, remark that $\ringpartial^c_{n,[1]^\sharp}$ is equivalent to $\ringpartial_{n,[1]^\sharp}$ which is an equivalence according to proposition \ref{prop:lax gr construction particular case}.
If $I$ is of shape $[a,1]$ for $a$ in $t\Theta$, this is the content of lemma \ref{lemma:ring partial eq for a 1}.
\end{proof}

\begin{definition} We recall that a left cartesian fibration is $\U$-small if its fibers are $\U$-small $\io$-categories. For an $\io$-category $A$, we denote by $\LCart_{\U}(A^\sharp)$ the full sub $\iun$-category of $\LCart_{\U}(A^\sharp)$ whose objects correspond to $\U$-small left cartesian fibrations over $A^\sharp$. For a marked $\io$-category $I$, we define similarly $\LCartc_{\U}(I)$ as the full sub $\iun$-category of $\LCartc_{\U}(I)$ whose objects correspond to $\U$-small classified left cartesian fibrations over $I$.
\end{definition}

\begin{cor}
\label{cor:univalence}
Let $\uni$ be the $\V$-small $\io$-category of $\U$-small $\io$-categories.
Let $n$ be an integer and $I$ be a $\V$-small marked $\io$-category. We denote by $I^\sharp$ the marked $\io$-category obtained from $I$ by marking all cells, and $\iota:I\to I^\sharp$ the induced morphism. There is an equivalence, natural in $[n]:\Delta^{op}$ and $I:\ocatm^{op}$, between functors
$$f:I\otimes[n]^\sharp\to \uni^\sharp$$
and sequences
$$\iota^*\int_{I^\natural}f_0\to ... \to \iota^*\int_{I^\natural}f_n$$
where for any $k\leq n$, $f_k$ is the functor $I^\natural\to \uni$ induced by $I\otimes\{k\}\to I\otimes[n]^\sharp\to \uni^\sharp$.
\end{cor}
\begin{proof}
This is a direct application of the equivalence 
$$\tau_0\LCart((I\otimes[n]^\sharp)^\sharp) \to \Hom([n],\LCartc(I))$$
induced by proposition \ref{prop:ring partial eq for I n}.
\end{proof}

\begin{cor}
\label{cor:parametric univalence}
Let $n$ be an integer, $I$ a $\V$-small marked $\io$-category, and $A$ an $\io$-category. We denote by $I^\sharp$ the marked $\io$-category obtained from $I$ by marking all cells, and $\iota:I\to I^\sharp$ the induced morphism. There is an equivalence, natural in $[n]:\Delta^{op}$ and $I:\ocatm^{op}$, between functors
$$f:I\otimes[n]^\sharp\to \uHom(A,\uni)$$
and sequences
$$(\iota\times A^\sharp)^*\int_{I^\natural\times A}f_0\to ... \to (\iota\times A^\sharp)^*\int_{I^\natural\times A}f_n$$
where for any $k\leq n$, $f_k$ is the functor $I^\natural\times A\to \uni$ induced by $(I\otimes\{k\})\times A^\sharp\to (I\otimes[n]^\sharp)\times A^\sharp\to \uni^\sharp$.
\end{cor}
\begin{proof}
This is a direct application of the last corollary and the equivalence $(I\otimes[n]^\sharp)\times A^\sharp\sim (I\times A^\sharp)\otimes[n]^\sharp$ given in proposition \ref{prop:associativity of Gray2}.
\end{proof}

\begin{cor}
\label{cor:univalence tranche}
Let $I$ be a $\V$-small marked $\io$-category and $c$ an object of $\uni$. We denote by $I^\sharp$ the marked $\io$-category obtained from $I$ by marking all cells, and $\iota:I\to I^\sharp$ the induced morphism. There is an equivalence, natural in $I:\ocatm^{op}$, between functors
$$f:I\to \uni^\sharp_{c/}$$
and arrows:
$$I\times \int_1c\to \iota^* \int_{I^\natural}\tilde{f}$$
where $\tilde{f}$ is the induced functor $I^\natural\to \uni_{c/}\to\uni $.
\end{cor}
\begin{proof}
By construction, we have a cocartesian square.
\[\begin{tikzcd}
	{I\otimes\{0\}} & {I\otimes[1]^\sharp} \\
	1 & {1\costar I}
	\arrow[from=1-1, to=2-1]
	\arrow[from=2-1, to=2-2]
	\arrow[from=1-1, to=1-2]
	\arrow[from=1-2, to=2-2]
	\arrow["\lrcorner"{anchor=center, pos=0.125, rotate=180}, draw=none, from=2-2, to=1-1]
\end{tikzcd}\]
As $\tau_0\LCart(\uvar)$ sends colimits to limits, this is a consequence of the last corollary.
\end{proof}

\begin{cor}
\label{cor:parametric univalence tranche}
Let $I$ be a $\V$-small marked $\io$-category, $A$ an $\io$-category, and $g$ an object of $\uHom(A,\uni)$. We denote by $I^\sharp$ the marked $\io$-category obtained from $I$ by marking all cells, and $\iota:I\to I^\sharp$ the induced morphism. There is an equivalence, natural in $I:\ocatm^{op}$, between functors
$$f:I\to \uHom(A,\uni)^\sharp_{g/}$$
and arrows:
$$I\times \int_Ag\to (\iota\times A^\sharp)^* \int_{I^\natural\times A}\tilde{f}$$
where $\tilde{f}:I^\natural\times A\to \uni$ is the functor corresponding to $I^\natural\to\uHom(A,\uni)_{g/}\to \uHom(A,\uni)$.
\end{cor}
\begin{proof}
We once again have a cocartesian square
\[\begin{tikzcd}
	{I\otimes\{0\}} & {I\otimes[1]^\sharp} \\
	1 & {1\costar I}
	\arrow[from=1-1, to=2-1]
	\arrow[from=2-1, to=2-2]
	\arrow[from=1-1, to=1-2]
	\arrow[from=1-2, to=2-2]
	\arrow["\lrcorner"{anchor=center, pos=0.125, rotate=180}, draw=none, from=2-2, to=1-1]
\end{tikzcd}\]
As $\tau_0\LCart(\uvar)$ sends colimits to limits, this is a consequence of the last corollary and the equivalence $(I\otimes[1]^\sharp)\times A^\sharp\sim (I\times A^\sharp)\otimes[1]^\sharp$ given in proposition \ref{prop:associativity of Gray2}.
\end{proof}

\subsection{Lax Grothendieck construction}

\begin{definition} For $I$ a marked $\io$-category and $A$ an $\io$-category, we define the $\io$-category \wcnotation{$\gHom(I,A)$}{(hom@$\gHom(\uvar,\uvar)$}, whose value on a globular sum $a$, is given by 
$$\Hom(a,\gHom(I,A)):=\Hom(I\ominus a^\sharp,A^\sharp)$$
\end{definition}

The section is devoted to the proof of the following theorem:
\begin{theorem}
\label{theo:lcartc et ghom}
Let $I$ be a $\U$-small marked $\io$-category.
Let $\uni$ be the $\V$-small $\io$-category of $\U$-small $\io$-categories, and $\uLCartc_{\U}(I)$ the $\V$-small $\io$-category of $\U$-small left cartesian fibrations. 
There is an equivalence
$$\gHom(I,\uni)\sim \uLCartc_{\U}(I)$$
natural in $I$.
On the maximal sub $\infty$-groupoid, this equivalence corresponds to the Grothendieck construction of theorem \ref{theo:gr construction}.
\end{theorem}

\begin{cor}
\label{cor:lcar et hom}
Let $A$ be a $\U$-small $\io$-category.
Let $\uLCartc_{\U}(A^\sharp)$ be the $\V$-small $\io$-category of $\U$-small left cartesian fibrations. 
There is an equivalence
$$\uHom(A,\uni)\sim \uLCart_{\U}(A^\sharp)$$
natural in $A$.
On the maximal sub $\infty$-groupoid, this equivalence corresponds to the Grothendieck construction of theorem \ref{theo:gr construction}.
\end{cor}
\begin{proof}
This is a consequence of the equivalences $\uLCart_{\U}(A^\sharp)\sim \uLCartc_{\U}(A^\sharp)$, of the previous theorem and of the equivalence 
$\uHom(A,\uni)\sim \gHom(A^\sharp,\uni)$ induced by the second assertion of proposition \ref{prop:associativity of ominus}.
\end{proof}

\begin{construction}
\label{cons: i pull and push beetwe io category of morphism}
The theorem \ref{theo:lcartc et ghom} provides equivalences \index[notation]{(f8@$f^*:\gHom(I,\uni)\to \gHom(J,\uni)$}
$$ \gHom(I,\uni)\sim \uLCartc_{\U}(I) ~~~~\mbox{and}~~~~\uHom(A,\omega)\sim \uLCart_{\U}(A^\sharp)$$
By construction, for any morphism $f:I\to J$ between marked $\omega$-categories, we have a morphism 
$$f^*:\gHom(J,\uni)\to \uHom(I,\uni)$$
Suppose now that the codomain of $f$ is of shape $A^\sharp$.
 The morphism \eqref{eq:i pull} induces a morphism \index[notation]{(f7@$f_{\mbox{$\exclam$}}:\gHom(I,\uni)\to \uHom(A,\uni)$}
$$f_!:\gHom(I,\uni)\to \uHom(A,\uni)$$ and \eqref{eq:i pull unit an counit} induces natural transformations:
$$
\mu:id\to f^*f_!~~~~ \epsilon:f_!f^*\to id
$$
coming along with equivalences:
$(\epsilon\circ_0 f_!)\circ_1(f_!\circ_0 \mu) \sim id_{f_!}$ and $(f^*\circ_0 \epsilon)\circ_1 (\mu \circ_0 f^* )\sim id_{f^*}$.
When $f$ is proper, the morphism \eqref{eq:i push op} induces a morphism \index[notation]{(f9@$f_*:\gHom(I,\uni)\to \uHom(A,\uni)$}
$$f_*:\gHom(I,\uni)\to \uHom(A,\uni)$$
 and \eqref{eq:i pull unit an counit op} induces natural transformations:
 $$
\mu: id\to f_*f^*~~~~ \epsilon:f^*f_*\to id
$$
coming along with equivalences:
$(\epsilon\circ_0 f^*)\circ_1(f^*\circ_0 \mu) \sim id_{f^*}$ and $(f_*\circ_0 \epsilon)\circ_1 (\mu \circ_0 f_* )\sim id_{f_*}$.
Moreover, for every morphism $j:C\to D^\sharp$, \eqref{eq:commutative pull push} 
induces a canonical commutative square
\[\begin{tikzcd}
	{\gHom(D^\sharp\times I,\uni)} & {\uHom(D\times A ,\uni)} \\
	{\gHom(C^\sharp\times I,\uni)} & {\uHom(C\times A,\uni)}
	\arrow["{( id_{D^\sharp}\times f)_!}", from=1-1, to=1-2]
	\arrow["{(j\times id_{I})^*}"', from=1-1, to=2-1]
	\arrow["{( id_{C^\sharp}\times f)_!}"', from=2-1, to=2-2]
	\arrow["{(j\times id_{A^\sharp})^*}", from=1-2, to=2-2]
\end{tikzcd}\]
and when $f$ is proper, \eqref{eq:commutative pull push op} induces a canonical commutative square
\[\begin{tikzcd}
	{\gHom(D^\sharp\times I,\uni)} & {\uHom(D\times A ,\uni)} \\
	{\gHom(C^\sharp\times I,\uni)} & {\uHom(C\times A,\uni)}
	\arrow["{( id_{D^\sharp}\times f)_*}", from=1-1, to=1-2]
	\arrow["{(j\times id_{I})^*}"', from=1-1, to=2-1]
	\arrow["{( id_{C^\sharp}\times f)_*}"', from=2-1, to=2-2]
	\arrow["{(j\times id_{A^\sharp})^*}", from=1-2, to=2-2]
\end{tikzcd}\]
\end{construction}

\vspace{1cm}
We now turn our attention back to the proof of the theorem \ref{theo:lcartc et ghom}.

\begin{notation}
We denote by $\pi_b:I\times b^\flat\to I$ the canonical projection.
\end{notation}
\begin{lemma}
\label{lemma:lax univalence 0}
Let $I$ be a marked $\io$-category and $b^\flat$ a globular sum. 
There is an equivalence of $\iun$-categories:
$$\LCart(I\times b^\flat)\sim \LCart(I)_{/\pi_b}$$
\end{lemma}
\begin{proof}
Remark first that we have an equivalence 
$$(\ocatm_{/I})_{/\pi_b}\sim \ocatm_{/I\times b}$$
Now suppose given a triangle 
\[\begin{tikzcd}
	& {I\times b^\flat} \\
	X & I
	\arrow[from=2-1, to=1-2]
	\arrow["{\pi_b}", from=1-2, to=2-2]
	\arrow[from=2-1, to=2-2]
\end{tikzcd}\]
As left cartesian fibrations are stable by composition and right cancellation, and as $\pi_b$ is a left cartesian fibration,  the diagonal morphism is a left cartesian fibration if and only if the horizontal morphism is. 

The $\iun$-categories $\LCart(I)_{/\pi_b}$ and  $\LCart(I\times b^\flat)$ then identify with the same full sub $\iun$-category of $(\ocatm_{/I})_{/\pi_b}\sim \ocatm_{/I\times b}$.
\end{proof}

\begin{lemma}
\label{lemma:lax univalence 1}
There is a family of cartesian squares
\[\begin{tikzcd}
	{\tau_0\LCart([a\times b,n]^\sharp)} & {\tau_0\LCart([a,n]^\sharp\times b^\flat)} \\
	{\prod_{k\leq n}\tau_0\LCart(\{k\})} & {\prod_{k\leq n}\tau_0\LCart(\{k\}\times b^\flat)}
	\arrow[from=2-1, to=2-2]
	\arrow[from=1-2, to=2-2]
	\arrow[from=1-1, to=2-1]
	\arrow[from=1-1, to=1-2]
\end{tikzcd}\]
natural in $a,b$ and $n$.
\end{lemma}
\begin{proof}
Remark first that the proposition
\ref{prop:crushing of Gray tensor is identitye marked case} provides cocartesian squares:
\[\begin{tikzcd}
	{\coprod_{k\leq n}(a^\flat\times b^\flat)\otimes\{k\}} & {(a^\flat\times b^\flat)\otimes[n]^\sharp} & {\coprod_{k\leq n}a^\flat\otimes\{k\}} & {a^\flat\otimes[n]^\sharp} \\
	{\coprod_{k\leq n}\{k\}} & {[a\times b,n]^\sharp} & {\coprod_{k\leq n}\{k\}} & {[a,n]^\sharp}
	\arrow[from=1-1, to=2-1]
	\arrow[from=2-1, to=2-2]
	\arrow[from=1-1, to=1-2]
	\arrow[from=1-2, to=2-2]
	\arrow["\lrcorner"{anchor=center, pos=0.125, rotate=180}, draw=none, from=2-2, to=1-1]
	\arrow[from=2-3, to=2-4]
	\arrow[from=1-4, to=2-4]
	\arrow[from=1-3, to=2-3]
	\arrow[from=1-3, to=1-4]
	\arrow["\lrcorner"{anchor=center, pos=0.125, rotate=180}, draw=none, from=2-4, to=1-3]
\end{tikzcd}\]
According to the corollary \ref{cor:fib over a colimit2}, and proposition \ref{prop:ring partial eq for I n}, and as $\Rb (\pi_{\uvar})_!:\LCart(1)\to \LCart(\uvar^\flat)$ factors through $\LCartc(\uvar^\flat)$, 
this induces cartesian squares:
\begin{equation}
\label{eq:lemma:lax univalence 2}
\begin{tikzcd}[column sep = 0.3cm]
	{\LCart([a\times b,n]^\sharp)} & {\Fun([n],\LCart(a^\flat\times b^{\flat}))} & {\LCart([a,n]^\sharp)} & {\Fun([n],\LCart( a^{\flat}))} \\
	{\prod_{k\leq n}\LCart(1)} & {\prod_{k\leq n}\LCart(a^\flat\times b^{\flat})} & {\prod_{k\leq n}\LCart(1)} & {\prod_{k\leq n}\LCart(a^\flat)}
	\arrow[from=1-3, to=2-3]
	\arrow[""{name=0, anchor=center, inner sep=0}, from=2-3, to=2-4]
	\arrow[from=1-3, to=1-4]
	\arrow[from=1-4, to=2-4]
	\arrow[from=1-1, to=1-2]
	\arrow[from=1-1, to=2-1]
	\arrow[""{name=1, anchor=center, inner sep=0}, from=2-1, to=2-2]
	\arrow[from=1-2, to=2-2]
	\arrow["\lrcorner"{anchor=center, pos=0.125}, draw=none, from=1-1, to=1]
	\arrow["\lrcorner"{anchor=center, pos=0.125}, draw=none, from=1-3, to=0]
\end{tikzcd}
\end{equation}
 As the $\iun$-categorical slice and the maximal full sub $\infty$-groupoid preserve cartesian squares,
the second cartesian square induces a cartesian square
\[\begin{tikzcd}
	{\tau_0(\LCart( [a,n]^\sharp)_{/\pi_b})} & {\Hom([n],\LCart( a^{\flat})_{/\pi_b})} \\
	{\prod_{k\leq n}\tau_0(\LCart(1)_{/b^\flat})} & {\prod_{k\leq n}\tau_0(\LCart( a^{\flat})_{/\pi_b})}
	\arrow[from=1-1, to=2-1]
	\arrow[""{name=0, anchor=center, inner sep=0}, from=2-1, to=2-2]
	\arrow[from=1-1, to=1-2]
	\arrow[from=1-2, to=2-2]
	\arrow["\lrcorner"{anchor=center, pos=0.125}, draw=none, from=1-1, to=0]
\end{tikzcd}\]
and according to  lemma \ref{lemma:lax univalence 0}, this corresponds to  a cartesian square
\[\begin{tikzcd}
	{\tau_0\LCart([a,n]^\sharp\times b^\flat)} & {\Hom([n],\LCart( a^{\flat}\times b^\flat))} \\
	{\prod_{k\leq n}\tau_0\LCart(  b^\flat)} & {\prod_{k\leq n}\tau_0\LCart( a^{\flat}\times b^\flat)}
	\arrow[from=1-1, to=2-1]
	\arrow[""{name=0, anchor=center, inner sep=0}, from=2-1, to=2-2]
	\arrow[from=1-1, to=1-2]
	\arrow[from=1-2, to=2-2]
	\arrow["\lrcorner"{anchor=center, pos=0.125}, draw=none, from=1-1, to=0]
\end{tikzcd}\]
Combined with the first cartesian square of \eqref{eq:lemma:lax univalence 2}, this induces a commutative diagram
\[\begin{tikzcd}
	{\tau_0\LCart([a\times b, n]^\sharp)} & {\tau_0\LCart([a,n]^\sharp\times b^\flat)} & {\Hom([n],\LCart( a^{\flat}\times b^\flat))} \\
	{\prod_{k\leq n}\tau_0\LCart(  1)} & {\prod_{k\leq n}\tau_0\LCart(  b^\flat)} & {\prod_{k\leq n}\tau_0\LCart( a^{\flat}\times b^\flat)}
	\arrow[from=1-2, to=2-2]
	\arrow[""{name=0, anchor=center, inner sep=0}, from=2-2, to=2-3]
	\arrow[from=1-2, to=1-3]
	\arrow[from=1-3, to=2-3]
	\arrow[from=2-1, to=2-2]
	\arrow[from=1-1, to=2-1]
	\arrow[from=1-1, to=1-2]
	\arrow["\lrcorner"{anchor=center, pos=0.125}, draw=none, from=1-2, to=0]
\end{tikzcd}\]
where the right and the outer square are cartesian. By right cancellation, the left square is cartesian which concludes the proof.
\end{proof}

\begin{remark}
\label{rem:explaining lax univalence}
Remark that $\LCartc(1)\sim \LCart(1)$.
The proposition \ref{prp:to show fully faithfullness3} and the construction of $\LCartc(\uvar)$ implies that the bottom horizontal morphism of the square given in lemma \ref{lemma:lax univalence 1} is a limit of fully faithful morphisms. As fully faithful morphisms are stable by limits and pullbacks, we can rephrase lemma \ref{lemma:lax univalence 1} by saying that $\tau_0\LCart([a\times b,n]^\sharp)$ is the sub $\infty$-groupoid of $\tau_0\LCart([a,n]^\sharp\times b^\flat)$ whose objects are left cartesian fibrations $X\to [a,n]^\sharp\times b^{\flat}$ such that for any $k\leq n$, there exists a cartesian square of the shape
\[\begin{tikzcd}
	{Y\times b^{\flat}} & X \\
	{\{k\}\times b^{\flat}} & {[a,n]\times b^{\flat}}
	\arrow[from=1-1, to=1-2]
	\arrow[from=1-1, to=2-1]
	\arrow["\lrcorner"{anchor=center, pos=0.125}, draw=none, from=1-1, to=2-2]
	\arrow[from=1-2, to=2-2]
	\arrow[from=2-1, to=2-2]
\end{tikzcd}\]
\end{remark}

\begin{lemma}
\label{lemma:lax univalence 2.5}
Let $b$ be a globular sum and let  $F:I\to \ocat$ be a $\Wcard$-small diagram. 
The canonical morphism
$$\LCart(\colim_IF^\sharp\times b^\flat) \to \lim_I \LCart(F^\sharp\times b^\flat)$$
is an equivalence.
\end{lemma}
\begin{proof}
The corollary \ref{cor:fib over a colimit2} implies that the canonical morphism
$$\LCart(\colim_IF^\sharp) \to \lim_I \LCart(F^\sharp)$$
is an equivalence.
As the $\iun$-categorical slice preserves limits, the previous equivalence induces an equivalence
$$\LCart(\colim_IF^\sharp)_{/\pi_b} \to \lim_I \LCart(F^\sharp)_{/\pi_b}.$$
The results then follows from lemma \ref{lemma:lax univalence 0}.
\end{proof}

\begin{lemma}
\label{lemma:lax univalence 2}
There is a family of cartesian squares
\[\begin{tikzcd}
	{\tau_0\LCart((I\ominus[b,n]^\sharp)^\sharp)} & {\tau_0\LCart((I\otimes[n]^\sharp)^\sharp\times b^\flat)} \\
	{\prod_{k\leq n}\tau_0\LCart(I^\sharp\otimes\{k\})} & {\prod_{k\leq n}\tau_0\LCart((I^\sharp\otimes\{k\})\times b^\flat)}
	\arrow[from=2-1, to=2-2]
	\arrow[from=1-2, to=2-2]
	\arrow[from=1-1, to=2-1]
	\arrow[from=1-1, to=1-2]
\end{tikzcd}\]
natural in $I,b$ and $n$.
\end{lemma}
\begin{proof}
By definition, $(I\ominus [b,n]^\sharp)^\sharp$ fits in the following cocartesian square:
\[\begin{tikzcd}
	{\colim_{[a,m]\to \amalg_kI^\natural\otimes \{k\}}[a\times b,m]^\sharp} & {\colim_{[a,m]\to \amalg_kI^\natural\otimes \{k\}}[a,m]^\sharp} \\
	{\colim_{[a,m]\to (I\otimes[n]^\sharp)^\natural}[a\times b,m]^\sharp} & {(I\ominus [b,n]^\sharp)^\sharp}
	\arrow[from=1-1, to=2-1]
	\arrow[from=2-1, to=2-2]
	\arrow[from=1-2, to=2-2]
	\arrow[""{name=0, anchor=center, inner sep=0}, from=1-1, to=1-2]
	\arrow["\lrcorner"{anchor=center, pos=0.125, rotate=180}, draw=none, from=2-2, to=0]
\end{tikzcd}\]
Combined with corollary \ref{cor:fib over a colimit2}, this implies that the $\infty$-groupoid $\tau_0\LCart((I\ominus [b,n]^\sharp)^\sharp)$ fits in the cartesian square:
\[\begin{tikzcd}
	{\tau_0\LCart((I\ominus [b,n]^\sharp)^\sharp)} & {\lim_{[a,m]\to \amalg_kI^\natural\otimes \{k\}}\tau_0\LCart([a,m]^\sharp)} \\
	{\lim_{[a,m]\to (I\otimes[n]^\sharp)^\natural}\tau_0\LCart([a\times b,m]^\sharp)} & {\lim_{[a,m]\to \amalg_kI^\natural\otimes \{k\}}\tau_0\LCart([a\times b,m]^\sharp)}
	\arrow[from=1-1, to=1-2]
	\arrow[from=1-1, to=2-1]
	\arrow[from=1-2, to=2-2]
	\arrow[""{name=0, anchor=center, inner sep=0}, from=2-1, to=2-2]
	\arrow["\lrcorner"{anchor=center, pos=0.125}, draw=none, from=1-1, to=0]
\end{tikzcd}\]
By the remark \ref{rem:explaining lax univalence} and by the fact that any morphism $\{l\}\to [a,m]\to (I\otimes[n]^\sharp)^\natural$ uniquely factors through $\coprod_{k}I^\natural \otimes\{k\}$,
we get a cartesian square
\[\begin{tikzcd}
	{\tau_0\LCart((I\ominus [b,n]^\sharp)^\sharp)} & {\lim_{[a,m]\to \amalg_kI^\natural\otimes \{k\}}\tau_0\LCart([a,m]^\sharp)} \\
	{\lim_{[a,m]\to (I\otimes[n]^\sharp)^\natural}\tau_0\LCart([a,m]^\sharp\times b^\flat)} & {\lim_{[a,m]\to \amalg_kI^\natural\otimes \{k\}}\tau_0\LCart([a,m]^\sharp\times b^\flat)}
	\arrow[from=1-1, to=1-2]
	\arrow[from=1-1, to=2-1]
	\arrow[from=1-2, to=2-2]
	\arrow[""{name=0, anchor=center, inner sep=0}, from=2-1, to=2-2]
	\arrow["\lrcorner"{anchor=center, pos=0.125}, draw=none, from=1-1, to=0]
\end{tikzcd}\]
Eventually, the lemma \ref{lemma:lax univalence 2.5} induces equivalences
$$\begin{array}{rll}
\lim_{[a,m]\to (I\otimes[n]^\sharp)^\natural}\tau_0\LCart([a,m]^\sharp\times b^\flat)&\sim&\tau_0 \LCart((I\otimes[n]^\sharp)^\sharp\times b^\flat)\\
\lim_{[a,m]\to I^\natural}\tau_0\LCart([a,m]^\sharp\times b^\flat)&\sim&\tau_0 \LCart(I^\sharp\times b^\flat)\\
\lim_{[a,m]\to I^\natural}\tau_0\LCart([a,m]^\sharp)&\sim &\tau_0\LCart(I^\sharp)
\end{array}
$$
This concludes the proof.
\end{proof}

\begin{lemma}
\label{lemma:lax univalence 3}
There is a family of cartesian squares
\[\begin{tikzcd}
	{\Hom([n], \LCartc(I;b))} & {\tau_0\LCart((I\otimes[n]^\sharp)^\sharp\times b^\flat)} \\
	{\prod_{k\leq n}\tau_0\LCart(I^\sharp\otimes\{k\})} & {\prod_{k\leq n}\tau_0\LCart((I^\sharp\otimes\{k\})\times b^\flat)}
	\arrow[from=2-1, to=2-2]
	\arrow[from=1-2, to=2-2]
	\arrow[from=1-1, to=2-1]
	\arrow[from=1-1, to=1-2]
\end{tikzcd}\]
natural in $I,b$ and $n$.
\end{lemma}
\begin{proof}
By the construction of $\LCartc(I;b)$, we have a cartesian square 
\[\begin{tikzcd}
	{\Hom([n], \LCartc(I;b))} & {\Hom([n], \LCart(I\times b^\flat))} \\
	{\prod_{k\leq n} \tau_0\LCartc(I)} & {\prod_{k\leq n} \tau_0\LCart(I\times b^\flat)}
	\arrow[""{name=0, anchor=center, inner sep=0}, from=2-1, to=2-2]
	\arrow[from=1-2, to=2-2]
	\arrow[from=1-1, to=2-1]
	\arrow[from=1-1, to=1-2]
	\arrow["\lrcorner"{anchor=center, pos=0.125}, draw=none, from=1-1, to=0]
\end{tikzcd}\]
According to lemma \ref{lemma:lax univalence 0}, this implies that the outer square of the following diagram
\[\begin{tikzcd}
	{\Hom([n], \LCartc(I;b))} & {\Hom([n], \LCartc(I)_{/\pi_b})} & {\Hom([n], \LCart(I)_{/\pi_b})} \\
	{\prod_{k\leq n} \tau_0\LCartc(I)} & {\prod_{k\leq n}\tau_0 (\LCartc(I)_{/\pi_b})} & {\prod_{k\leq n}\tau_0 (\LCart(I)_{/\pi_b})}
	\arrow[from=1-1, to=1-2]
	\arrow[from=1-1, to=2-1]
	\arrow[from=1-2, to=1-3]
	\arrow[from=1-2, to=2-2]
	\arrow[from=1-3, to=2-3]
	\arrow[from=2-1, to=2-2]
	\arrow[from=2-2, to=2-3]
\end{tikzcd}\]
is cartesian. As $\LCart^c(\uvar)\to \LCart(\uvar)$ is fully faithful, the right-hand square is cartesian, and so is the left one by right cancellation.
Now, remark that the naturality in $n$ of the equivalence $\ringpartial^c_{n,I}$ from proposition \ref{prop:ring partial eq for I n} implies that it sends $\pi_b:(I\otimes[n]^\sharp)^\sharp\times b^\flat\to b^{\flat}$ to the constant functor with value $\pi_b:I\otimes b^{\flat}\to b^{\flat}$. Combined with proposition \ref{prop:ring partial eq for I n}, the left-hand square in the previous diagram induces a cartesian square
\[\begin{tikzcd}
	{\Hom([n], \LCartc(I;b))} & {\tau_0(\LCart((I\otimes[n]^\sharp)^\sharp)_{/\pi_b})} \\
	{\prod_{k\leq n}\tau_0\LCart(I^\sharp\otimes\{k\})} & {\prod_{k\leq n}\tau_0(\LCart(I^\sharp\otimes\{k\})_{/\pi_b})}
	\arrow[""{name=0, anchor=center, inner sep=0}, from=2-1, to=2-2]
	\arrow[from=1-2, to=2-2]
	\arrow[from=1-1, to=2-1]
	\arrow[from=1-1, to=1-2]
	\arrow["\lrcorner"{anchor=center, pos=0.125}, draw=none, from=1-1, to=0]
\end{tikzcd}\]
Eventually, a last application of lemma \ref{lemma:lax univalence 0} concludes the proof.
\end{proof}

\begin{remark}
\label{rem:explaining lax univalence}
The proposition \ref{prp:to show fully faithfullness3} and the construction of $\LCartc(\uvar)$ imply that the bottom horizontal morphism of squares given in lemmas \ref{lemma:lax univalence 2} and \ref{lemma:lax univalence 3} is a limit of fully faithful morphisms. As fully faithful morphisms are stable by limits and pullbacks, we can rephrase lemmas \ref{lemma:lax univalence 2} and \ref{lemma:lax univalence 3} by saying that 
$\tau_0\LCart((I\ominus[b,n]^\sharp)^\sharp)$ and 
$\Hom([n], \LCartc(I;b))$ are equivalent to the full subcategory of $\tau_0\LCart((I\otimes[n]^\sharp)^\sharp\times b^\flat)$
whose objects are left cartesian fibrations $X\to (I\otimes[n]^\sharp)^\sharp\times b^\flat$ such that for any $k\leq n$, there exists a cartesian square of the shape
\[\begin{tikzcd}
	{Y\times b^{\flat}} & X \\
	{(I^\sharp\otimes \{k\})\times b^\flat} & {(I\otimes[n]^\sharp)^\sharp\times b^\flat}
	\arrow[from=1-1, to=1-2]
	\arrow[from=1-1, to=2-1]
	\arrow["\lrcorner"{anchor=center, pos=0.125}, draw=none, from=1-1, to=2-2]
	\arrow[from=1-2, to=2-2]
	\arrow[from=2-1, to=2-2]
\end{tikzcd}\]
\end{remark}

\begin{lemma}
\label{lemma:lax univalence 4}
There is an equivalence 
$$\tau_0(\LCart((I\ominus [b,n]^\sharp)^\sharp) \sim \Hom([n],\LCartc(I;b))$$
natural in $I:\ocatm^{op}$, $b:\Theta^{op}$ and $[n]:\Delta^{op}$.
\end{lemma}
\begin{proof}
This is a direct consequence of lemmas \ref{lemma:lax univalence 2} and \ref{lemma:lax univalence 3} as these two objects fit in the same cartesian square.
\end{proof}

\begin{proof}[Proof of theorem \ref{theo:lcartc et ghom}]
The lemma \ref{lemma:lax univalence 4} provides an equivalence
$$\tau_0(\LCart((I\ominus [b,n]^\sharp)^\sharp) \sim \Hom([n],\LCartc(I;b))$$
that preserves $\U$-smallness.
\end{proof}

\section{Yoneda lemma and other results}

\subsection{Yoneda lemma}
\begin{definition} An $\io$-category $C$ is \wcnotion{locally $\U$-small}{locally $\U$-small $\io$-category} if for any pair of objects $x$ and $y$, $\hom_C(x,y)$ is $\U$-small. 
\end{definition}

\begin{prop}
\label{prop:when Hom A B is locally small}
Let $A$ be a $\U$-small $\io$-category, and $C$ is a locally $\U$-small $\io$-category. The $\io$-category $\uHom(A,C)$ is locally $\U$-small. 
\end{prop}
\begin{proof}
We have to check that for any globular sum $b$, the morphism 
$$\Hom(A\times [b,1],C)\to \Hom(A\times (\{0\}\amalg\{1\}),C)$$
has $\U$-small fibers. As $A$, seen as an $\infty$-presheaves on $\Theta$, is a $\U$-small colimit of representables, we can reduce to the case where $A\in \Theta$. As $C$ is local with respect to Segal extensions, and as the cartesian product conserves them, we can reduce to the case where $A$ is of shape $[a,1]$ for $a$ a globular sum. We now fix a morphism $f:[a,1]\times (\{0\}\amalg\{1\})\to C$. Eventually, using the canonical equivalence between $[a,1]\times [b,1]$ and the colimit of the span
$$[a,1]\vee[b,1]\leftarrow [a\times b,1]\to [b,1]\vee[a,1],$$
the $\infty$-groupoid $\Hom([a,1]\times [b,1],C)_f$ fits in a cartesian square:
\[\begin{tikzcd}
	{\Hom([a,1]\times [b,1],C)_f} & {\Hom(b,\hom_C(f(0,0),f(0,1)))} \\
	{\Hom(b,\hom_C(f(1,0),f(1,1)))} & {\Hom(a\times b,\hom_C(f(0,0),f(1,1)))}
	\arrow[from=1-1, to=1-2]
	\arrow[from=1-1, to=2-1]
	\arrow[from=1-2, to=2-2]
	\arrow[from=2-1, to=2-2]
\end{tikzcd}\]
As all these objects are $\U$-small by assumption, this concludes the proof.
\end{proof}

\begin{construction} Let $C$ be an $\io$-category $C$. We define the simplicial object $S(\Noiun C)$ by the formula
$$S(\Noiun C)_n:= \coprod_{x_0,...,x_n:A_0} \coprod_{y_0,...,y_n:A_0}\hom_C(x_n,...,x_0,y_0,...,y_n)$$
This object comes along with a canonical projection 
\begin{equation}
\label{eq:definition of the hom0}
S(\Noiun C)\to \Noiun{C^t}\times \Noiun C.
\end{equation}
which obviously is a left fibration. As this construction if functorial, it induces a functor:
$$\begin{array}{rcl}
\ocat &\to &\Arr(\ouncat)\\
C&\mapsto & (S(\Noiun C)\to \Noiun{C^t}\times \Noiun C)
\end{array}$$
\end{construction}

Until the end of this section, we fix a locally $\U$-small $\io$-category $C$. 
\begin{construction}
The left fibration \eqref{eq:definition of the hom0} is then $\U$-small, and by the definition of $\uni$ given in \ref{cons:defi of uni}, this induces a morphism
\begin{equation}
\label{eq:definition of the hom}
\hom_C(\uvar,\uvar):C^t\times C\to \uni
\end{equation}
\end{construction}

\begin{remark}
We recall for a $\io$-category $A$ and an object $x$ of $A$, $\Fb h^{A}_{x}$ denote the left cartesian fibrant replacement of the morphism $h^A_x:\{x\}\to A$.
By the explicit construction of $\Fb h^{A}_{x}$ given in proposition \ref{prop:explicit factoryzation}, we have a canonical equivalence
$$\Fb h^{C^t\times C}_{(x,y)}\sim \Fb h^{C^t}_{x}\times \Fb h^{C}_{y}$$
The Grothendieck construction (defined in \eqref{eq:definition of the hom}) of the morphism \eqref{eq:definition of the hom} is then the colimit of a simplicial object whose value on $n$ is given by:
$$\coprod_{x_0,...,x_n}\coprod_{y_0,...,y_n} \Fb h^{C^t}_{x_n}\times \hom_{C}(x_n,...,x_0,y_0,...,y_n)^\flat\times \Fb h^{C}_{y_n}$$\end{remark}

\begin{definition} We define the \wcnotion{$\io$-category of $\io$-presheaves on $C$}{presheaves@$\io$-presheaves} \sym{(c@$\w{C}$}:$$\w{C}:=\uHom(C^t,\uni ).$$ This $\io$-category is locally $\U$-small according to proposition \ref{prop:when Hom A B is locally small}. The \notion{Yoneda embedding}\sym{(y@$y_{\uvar}$} $y: C\to \w{C}$ is the functor induced by the hom functor \eqref{eq:definition of the hom} by currying.
\end{definition}

\begin{definition}
An $\io$-presheaves is \wcnotion{representable}{representable $\io$-presheaves} if it is in the image of $y$.
\end{definition}

\begin{construction}
We recall that for a subset $S$ of $\Nb^*$, and an object $X$ of $\ouncat$, we denote by $X^S$ the simplicial object $n\mapsto X_n^S$. We also set $\Sigma S:=\{i+1,i\in S\}$. We then have equivalences
$$(\Noiun C)^S\sim \Noiun (C^{\Sigma C}) ~~~\mbox{and}~~~ S(\Noiun C))^S\sim S(\Noiun (C^{\Sigma C}))$$
For an object $X$ of $\ouncat$, we denote by $X_{op}$ the simplicial object $n\mapsto X_{n^{op}}$. We then have equivalences 
$$(\Noiun C)_{op}\sim \Noiun (C^{t}) ~~~\mbox{and}~~~ S(\Noiun C))_{op}\sim S(\Noiun (C^t))$$
Using the dualities defined in construction \ref{cons:dualities fo omega}, we then have commutative diagrams
\[\begin{tikzcd}
	{(C^{t\Sigma S}\times C^{\Sigma S})^{\Sigma S}} && {\uni^{\Sigma S}} && {C\times C^t} \\
	{C^t\times C} && \uni && {C^t\times C} & \uni
	\arrow["{\hom_{C}}"', from=2-1, to=2-3]
	\arrow["{\hom^{\Sigma S}_{C^{\Sigma S}}}", from=1-1, to=1-3]
	\arrow["\sim"', from=1-1, to=2-1]
	\arrow["{(\uvar)^S}", from=1-3, to=2-3]
	\arrow["{\hom_C}"', from=2-5, to=2-6]
	\arrow["\tw"', from=1-5, to=2-5]
	\arrow["{\hom_{C^t}}", from=1-5, to=2-6]
\end{tikzcd}\]
where $\tw$ is the functor exchanging the argument. This two diagram corresponds to invertible natural transformations
$$\hom_{C^{\Sigma S}}(x,y)\sim \hom_{C}(x,y)^S~~~\mbox{and}~~~\hom_{C^t}(x,y)\sim \hom_{C}(y,x).$$

In combining the two previous diagrams, we get a commutative square:
\[\begin{tikzcd}
	{(C^{\circ t}\times C^{\circ})^{\circ t}} && {\uni^{{\circ t}}} \\
	{C^t\times C} && \uni
	\arrow["{\hom_{C}}"', from=2-1, to=2-3]
	\arrow["{\hom^{\circ t}_{C^\circ}}", from=1-1, to=1-3]
	\arrow["\tw"', from=1-1, to=2-1]
	\arrow["{(\uvar)^\circ}", from=1-3, to=2-3]
\end{tikzcd}\]
corresponding to the natural transformation
$$\hom_{C^\circ}(x,y)\sim \hom_{C}(y,x)^\circ.$$
\end{construction}
\begin{prop}
\label{prop:Yoneda is Fb} 
Let $A$ be an locally $\U$-small $\io$-category.
Let $a$ be an object of $A$.
There is an equivalence
$$\int_{A}\hom_A(a,\uvar)\to \Fb h^{A}_a$$
Taking the fibers on $a$, the induced morphism $\hom_A(a,a)\to \hom_A(a,a)$ preserves the identity.

 In particular, for any object $c$ of $C$, this induces an equivalence
$$\int_{C^t}y_c\to \Fb h^{C^t}_c$$
\end{prop}
\begin{proof}
By construction, $\int_{A}\hom_A(a,\uvar)$ is the Grothendieck construction of the left fibration:
\[\begin{tikzcd}[column sep =0.4cm]
	\cdots & {\coprod_{x_0,x_1,x_2:A_0}\hom_{A}(a,x_0,x_1,x_2)} & {\coprod_{x_0,x_1:A_0}\hom_{A}(a,x_0,x_1)} & {\coprod_{x_0:A_0}\hom_{A}(a,x_0)} \\
	\cdots & {\coprod_{x_0,x_1,x_2:A_0}\hom_{A}(x_0,x_1,x_2)} & {\coprod_{x_0,x_1:A_0}\hom_{A}(x_0,x_1)} & {\coprod_{x_0:A_0}1}
	\arrow[shift right=4, from=2-2, to=2-3]
	\arrow[shift left=4, from=2-2, to=2-3]
	\arrow[from=2-2, to=2-3]
	\arrow[shift left=2, from=2-3, to=2-2]
	\arrow[shift right=2, from=2-3, to=2-2]
	\arrow[shift left=2, from=2-3, to=2-4]
	\arrow[shift right=2, from=2-3, to=2-4]
	\arrow[from=2-4, to=2-3]
	\arrow[from=1-3, to=2-3]
	\arrow[from=1-4, to=2-4]
	\arrow[shift left=2, from=1-3, to=1-4]
	\arrow[from=1-4, to=1-3]
	\arrow[shift right=2, from=1-3, to=1-4]
	\arrow[shift right=4, from=1-2, to=1-3]
	\arrow[from=1-2, to=1-3]
	\arrow[shift left=4, from=1-2, to=1-3]
	\arrow[shift right=2, from=1-3, to=1-2]
	\arrow[shift left=2, from=1-3, to=1-2]
	\arrow[from=1-2, to=2-2]
\end{tikzcd}\]
The results then follow from the corollary \ref{cor:antecedant of slice}.
\end{proof}

\begin{definition}
The identity $\w{C}\to \w{C}$ induces by currying a canonical morphism \sym{(ev@$\ev$}
$$\ev: C^t\times \w{C}\to \uni$$
called the \textit{evaluation functor}.
\end{definition}

\begin{remark}
 Given an object $c$ of $C$ and $f$ of $\widehat{C}$, we then have $\ev(c,f)\sim f(c)$ and so 
$$(c,\{f\})^*\int_{C\times \w{C}}\ev\sim c^*\int_{C^t}f$$
\end{remark}

\begin{prop}
\label{prop:un fonctorial Yoneda}
For any object $c$ of $C$, there exists a unique pair consisting of a morphism
$$\int_{\w{C}} \hom_{\w{C}}(y_c,\uvar)\to \int_{\w{C}}\ev(c,\uvar)$$
and a commutative square of shape
\begin{equation}
\label{eq:prop:un fonctorial Yoneda}
\begin{tikzcd}
	{\{id_{y_c}\}} & {\hom_{\widehat{C}}(y_c,y_c)\sim \{y_c\}^* \int_{\w{C}}\hom_{\w{C}}(y_c,\uvar)} \\
	{\{id_c\}} & {\hom_C(c,c)\sim \{y_c\}^* \int_{\w{C}}\ev(c,\uvar)}
	\arrow[Rightarrow, no head, from=1-1, to=2-1]
	\arrow[from=1-1, to=1-2]
	\arrow[from=2-1, to=2-2]
	\arrow[from=1-2, to=2-2]
\end{tikzcd}
\end{equation}
Moreover, this comparison morphism is an equivalence.
\end{prop}
\begin{proof}
The proposition \ref{prop:Yoneda is Fb} implies that $\int_{\w{C}}\hom_{\w{C}}(y_c,\uvar)$ is equivalent to $\Fb h^{\w{C}}_{y_c}$. A natural transformation $\int_{\w{C}}\hom_{\widehat{C}}(y_c,\uvar)\to  \int_{\w{C}}\ev(c,\uvar)$ then corresponds to a morphism 
$\Fb h^{\w{C}}_{y_c}\to  \int_{\w{C}}\ev(c,\uvar)$ and is then uniquely characterized by the value on $\{id_{y_c}\}$, which proves the uniqueness.

It remains to show the existence.
Let $E$ be an object of $\ocatm_{/\w{C}^\sharp}$ corresponding to a morphism $g:X\to \w{C}^\sharp$ . We denote $\iota:X\to X^\sharp$ the canonical inclusion. According to proposition \ref{prop:Yoneda is Fb}, a morphism $E\to \int_{\w{C}}\hom_{\widehat{C}}(y_c,\uvar)$ corresponds to a morphism $E\to \Fb h^{\w{C}}_{y_c}$, and by proposition \ref{prop:explicit factoryzation}, to a triangle
\[\begin{tikzcd}
	& {\w{C}^\sharp_{y_c/}} \\
	X & {\w{C}^\sharp}
	\arrow[from=2-1, to=1-2]
	\arrow[from=1-2, to=2-2]
	\arrow[from=2-1, to=2-2]
\end{tikzcd}\]
According to corollary \ref{cor:parametric univalence tranche}, this data is equivalent to the one of  a morphism
$$X\times \int_{C^t}y_c\to (\iota\times (C^t)^\sharp)^*\int_{X^\natural\times C^t}\tilde{g}$$
where $\tilde{g}$ is the morphism defined by currying from $g^\natural:X^\natural\to \w{C}$. 

As the proposition \ref{prop:Yoneda is Fb} states that $\int_{C^t}y_c\sim \Fb h_c^{C^t}$, we have then constructed an equivalence.
\begin{equation}
\label{eq:evaluation 4}
\Hom(E, \int_{\w{C}}\hom_{\w{C}}(y_c,\uvar)) \sim \Hom(X\times \Fb h_c^{C^t}, (\iota\times (C^t)^\sharp)^*\int_{X^\natural\times C^t}\tilde{g})
\end{equation}
natural in $E$.

Now, a morphism 
 $$E\to \int_{\w{C}}\ev(c,\uvar)$$
  corresponds by adjunction to a morphism 
\begin{equation}
\label{eq:evaluation 1}
id_X\to g^*\int_{\w{C}}\ev(c,\uvar)
\end{equation}
However, we have a canonical commutative square
\[\begin{tikzcd}
	{X^\natural} & {\w{C}} \\
	{X^\natural\times C^t} & \uni
	\arrow["{\ev(c,\uvar)}", from=1-2, to=2-2]
	\arrow["{g^\natural}", from=1-1, to=1-2]
	\arrow["{\tilde{g}}"', from=2-1, to=2-2]
	\arrow["{X^\natural\times \{c\}}"', from=1-1, to=2-1]
\end{tikzcd}\]
where $\tilde{g}$ is the morphism obtained by currying from $g^\natural:X^\natural\to \w{C}$. Using the naturality of the Grothendieck construction, the previous commutative square implies that the data of \eqref{eq:evaluation 1} corresponds to a morphism
$$id_X\to (\iota\times \{c\})^* \int_{X^\natural\times C^t}\tilde{g}$$
an by adjunction, to a morphism 
$$
X\times \Fb h_c^{C^t}\to (\iota\times (C^t)^\sharp)^*\int_{X^\natural\times C^t}\tilde{g}
$$
We then have constructed an equivalence 
\begin{equation}
\label{eq:evaluation 3}
\Hom(E, \int_{\w{C}}\ev(c,\uvar))\sim \Hom(X\times \Fb h_c^{C^t}, (\iota\times (C^t)^\sharp)^*\int_{X^\natural\times C^t}\tilde{g})
\end{equation}
natural in $E$.

Combining the equivalences \eqref{eq:evaluation 4} and \eqref{eq:evaluation 3}, we get an equivalence 
$$ \Hom(E, \int_{\w{C}} \hom_{\w{C}}(y_c,\uvar))\sim \Hom(E, \int_{\w{C}}\ev(c,\uvar))$$
natural in $E$.
Walking through all the equivalences, we can  see that when $E$ is $h^{\w{C}}_{y_c}$, this equivalence sends the upper horizontal morphism of \eqref{eq:prop:un fonctorial Yoneda} to the lower horizontal one.

By the Yoneda lemma for $(\infty,1)$-categories, this induces an equivalence 
$$\int_{\w{C}} \hom_{\w{C}}(y_c,\uvar)\sim \int_{\w{C}}\ev(c,\uvar).$$
that comes along with the desired commutative square.
\end{proof}

\begin{lemma}
\label{lemma:a particular Kan extension}
Let $i:C\to D$ be a morphism between locally $\U$-small $\io$-categories.
The canonical morphism of $\LCart((C^t)^\sharp\times D^\sharp)$:
$$\Lb(id\times i)_!\int_{C^t\times C}\hom_{C} \to \int_{C^t\times D}\hom_D(i(\uvar),\uvar)$$
is an equivalence.
\end{lemma}
\begin{proof}
Let $c,d$ be any objects of respectively $C$ and $D$.  We then have equivalences
$$\begin{array}{rcll}
\Rb (c,d)^* \Lb (id\times i)_!\int_{ C^t\times C^t}\hom_{ C}&\sim&\Rb \{d\}^*  \Lb i_! \Rb (id\times \{c\})^*\int_{ C^t\times C}\hom_{C}& (\ref{prop:BC condition})\\
&\sim& \Rb \{d\}^* \Lb i_!\Fb h^{C}_{c} &(\ref{prop:Yoneda is Fb})\\
&\sim & \Rb \{d\}^* \Fb h^{D}_{i(c)}\\
&\sim & \hom_D(i(c),d)^\flat
\end{array}$$
Remark that we also have an equivalence 
$$\Rb (c,d)^*\int_{C^t\times D}\hom_D(i(\uvar),\uvar)\sim \hom_D(i(c),d)^\flat$$
and that the induced endomorphism of $ \hom_D(i(c),d)^\flat$ is the identity. As equivalences are detected pointwise, this concludes the proof.
\end{proof}

\begin{theorem}
\label{theo:Yoneda lemma}
Let $C$ be a locally $\U$-small $\io$-category. 
The Yoneda embedding $C\to \widehat{C}$ is fully faithful. Furthermore,
there is an equivalence between the functor
$$\hom_{\w{C}}(y_{\uvar},\uvar):C^t\times \w{C}\to \uni$$ and
the functor 
$$\ev:C^t\times \w{C}\to \uni.$$

Given an object $c$ of $C$, the induced equivalence on fibers:
$$\hom_{\widehat{C}}(y_c,y_c)\sim \hom_C(c,c)$$
sends $\{id_{y_c}\}$ onto $\{id_c\}$.
\end{theorem}
\begin{proof}
The triangle
\[\begin{tikzcd}
	C \\
	{\w{C}} & {\w{C}}
	\arrow["y", from=1-1, to=2-2]
	\arrow["y"', from=1-1, to=2-1]
	\arrow["id"', from=2-1, to=2-2]
\end{tikzcd}\] 
induces by adjunction a triangle
\[\begin{tikzcd}
	{C^t\times C} \\
	{C^t\times\w{C}} & {\w{C}}
	\arrow["\hom", from=1-1, to=2-2]
	\arrow["{id\times y}"', from=1-1, to=2-1]
	\arrow["\ev"', from=2-1, to=2-2]
\end{tikzcd}\]
This corresponds to an equivalence
$$\int_{C^t\times C}\hom_{C}(\uvar,\uvar)\to (id\times y)^*\int_{C^t\times \w{C}}\ev.$$
By naturality, for any object $c$ of $C$, the pullback of the previous equivalence along $C^t\times\{c\}$ is the identity. In particular, the induced morphism $\hom(c,c)\to \hom(c,c)$ between the fibers over $(c,c)$ preserves the object $\{id_c\}$. According to lemma \ref{lemma:a particular Kan extension}, the previous equivalence induces by adjunction a morphism
\begin{equation}
\label{eq:proof of yoneda}
\Lb(id\times y)_!\int_{C^t\times C}\hom_{C}(\uvar,\uvar)\sim \int_{C^t\times \w{C}}\hom_{\w{C}}(y_{\uvar},\uvar)\to \int_{C^t\times \w{C}}\ev.
\end{equation}
that comes along, by construction, with a commutative square
\[\begin{tikzcd}
	{\{id_{y_c}\}} & {\hom_{\widehat{C}}(y_c,y_c)\sim \{y_c\}^* \int_{\w{C}}\hom_{\w{C}}(y_c,\uvar)} \\
	{\{id_c\}} & {\hom_C(c,c)\sim \{y_c\}^* \int_{\w{C}}\ev(c,\uvar)}
	\arrow[Rightarrow, no head, from=1-1, to=2-1]
	\arrow[from=1-1, to=1-2]
	\arrow[from=2-1, to=2-2]
	\arrow[from=1-2, to=2-2]
\end{tikzcd}\]
for any object $c$ of $C$. The restriction of the morphism \eqref{eq:proof of yoneda} to $\w{C}\times \{c\}$ is then equivalent to the natural transformation given in proposition \ref{prop:un fonctorial Yoneda}, and is an equivalence.  As equivalences between left cartesian fibrations are detected on fibers, the morphism \eqref{eq:proof of yoneda} is an equivalence, and by the Grothendieck deconstruction given in corollary \ref{cor: Grt equivalence}, this induces the desired equivalence between the functors $\hom_{\w{C}}(y_{\uvar},\uvar)$ and $\ev_{\w{C}}(y_{\uvar},\uvar)$. 

It remains to show that the Yoneda embedding is fully faithful. Remark that the morphism \eqref{eq:proof of yoneda} induces  by adjunction a commutative triangle
\[\begin{tikzcd}
	{\int_{C^t\times C}\hom_{C}(\uvar,\uvar)} & {(id\times y)^*\Lb(id\times y)_!\int_{C^t\times C}\hom_{C}(\uvar,\uvar)\sim (id\times y)^*\int_{C^t\times \w{C}}\hom_{\w{C}}(y_{\uvar},\uvar)} \\
	& {(id\times y)^*\int_{C^t\times \w{C}}\ev}
	\arrow[from=1-1, to=1-2]
	\arrow[from=1-1, to=2-2]
	\arrow[from=1-2, to=2-2]
\end{tikzcd}\]
where the diagonal and the right vertical morphism are equivalences. By two out of three, so is the horizontal one. By taking the fibers, this implies that the canonical morphism
$$\hom(y_{\uvar},y_{\uvar}):\hom_C(c,d)\sim \hom_{\widehat{C}}(y_c,y_d)$$ 
is an equivalence. This shows that the Yoneda embedding is fully faithful.
\end{proof}

\begin{cor}
\label{cor: universal fibration 2}
The universal left cartesian fibration with $\U$-small fibers is the canonical projection 
$\uni^\sharp_{1/}\to \uni^\sharp$.
\end{cor}
\begin{proof}
The corollary \ref{cor: universal fibration 2} states that universal left cartesian fibration with $\U$-small fibers is $\int_{\uni}id$. The Yoneda lemma implies that this left cartesian fibration is equivalent to $\int_{\uni}\hom_{\uni}(1,\uvar)$. Eventually, the proposition \ref{prop:Yoneda is Fb} states that this left cartesian fibration is equivalent to $\uni^\sharp_{1/}\to \uni^\sharp$.
\end{proof}

\subsection{Adjoint functors}

\begin{definition}
Let $C$ and $D$ be two locally $\U$-small $\io$-categories and $u:C\to D,$ $v:D\to C$ two functors. An \notion{adjunction structure} for the pair $(u,v)$ is the data of a invertible natural transformation
$$\phi: \hom_D(u(\uvar),\uvar)\sim \hom_C(\uvar,v(\uvar))$$
In this case, $u$ is a \wcnotion{left adjoint}{left or right adjoint} of $v$ and $v$ is a \textit{right adjoint} of $u$.
\end{definition}

\begin{prop}
\label{prop:adj if slice as terminal}
Let $u:C\to D$ be a functor between locally $\U$-small $\io$-categories. 
For $b$ an object of $D$, we define $(C^t)^\sharp_{b/}$ and $C^\sharp_{b/}$ as the marked $\io$-categories fitting in the cartesian squares:
\[\begin{tikzcd}
	{(C^t)^\sharp_{/b}} & {(D^t)^\sharp_{b/}} & {C^\sharp_{b/}} & {D^\sharp_{b/}} \\
	{(C^t)^\sharp} & {(D^t)^\sharp} & {C^\sharp} & {D^\sharp}
	\arrow[from=1-3, to=2-3]
	\arrow["u"', from=2-3, to=2-4]
	\arrow[from=1-4, to=2-4]
	\arrow["\lrcorner"{anchor=center, pos=0.125}, draw=none, from=1-3, to=2-4]
	\arrow[from=1-3, to=1-4]
	\arrow[from=1-1, to=2-1]
	\arrow[from=1-1, to=1-2]
	\arrow[from=1-2, to=2-2]
	\arrow["{u^t}"', from=2-1, to=2-2]
	\arrow["\lrcorner"{anchor=center, pos=0.125}, draw=none, from=1-1, to=2-2]
\end{tikzcd}\]
The following are equivalent.
\begin{enumerate}
\item The functor $u$ admits a right adjoint.
\item For any element $b$ of $D$, the marked $\io$-category $(C^t)^\sharp_{b/}$ 
admits an initial element.
\end{enumerate}
Similarly, the following are equivalent.
\begin{enumerate}
\item[(1)'] The functor $u$ admits a left adjoint.
\item[(2)'] For any element $b$ of $D$, $C^\sharp_{b/}$ admits an initial element.
\end{enumerate}
\end{prop}
\begin{proof}
Suppose first that $(1)$ is fulfilled, and let $v:D\to C$ be a functor and $\phi:\hom(u(a),b)\sim\hom(a,v(b))$ be an invertible natural transformation. In particular, this implies that we have an equivalence
$$\int_{C^t\times D}\hom_D(u(a),b)\sim \int_{C^t\times D}\hom_C(a,v(b))$$
Pulling back along $C^t\times \{b\}$ where $b$ is any object of $D$, we get an equivalence between 
$(C^t)^\sharp_{b/}$ and $(C^t)^\sharp_{v(b)/}$. As this last marked $\io$-category admits an initial element, given by the image $id_{v(b)}$, this shows the implication $(1)\Rightarrow (2)$.

For the converse, suppose that $u$ fulfills condition $(2)$. The functor $\hom_D(u(\uvar),\uvar)):C^t\times D\to \uni$ corresponds by adjonction to a functor $v':D\to \w{C}$. By assumption, for any $b\in B$, $v'(b)$ is a representable $\io$-presheaf. The Yoneda lemma then implies that $v$ factors through a functor $v:D\to C$. Using once again Yoneda lemma, we have a sequence of equivalences
$$\hom_D(u(a),b)\sim v'(b)(a)\sim \hom_C(b,v(a)).$$

The equivalence between $(1)'$ and $(2)'$ is proved similarly.
\end{proof}

\begin{construction}
\label{cons:of unit and counit}
 Let $(u,v,\phi)$ be an adjunction structure. There is a transformation 
$$\hom_C(a,a')\to \hom_D(u(a),u(a'))\to \hom_C(a,vu(a'))$$
natural in $a:C^t$, $a':C$. According to the Yoneda lemma, this corresponds to a natural transformation $\mu: id_C \to vu$, called the \wcnotion{unit of the adjunction}{unit and counit of an adjunction}. Similarly, the natural transformation:
$$\hom_D(b,b')\to \hom_C(v(b),v(b'))\to \hom_C(uv(b),b')$$
induces a natural transformation $\epsilon:uv\to id_D$, called \textit{the counit of the adjunction.}
\end{construction}

\begin{theorem}
\label{theo:two adjunction definition}
Let $u:C\to D$ and $v:D\to C$ be two functors between locally $\U$-small $\io$-categories. 
The two following are equivalent. 
\begin{enumerate}
\item The pair $(u,v)$ admits an adjunction structure.
\item Their exists a pair of natural transformations $\mu: id_C \to vu$ and $\epsilon:uv\to id_D$ together with equivalences $(\epsilon\circ_0 u)\circ_1(u\circ_0 \mu) \sim id_{u}$ and $(v\circ_0 \epsilon)\circ_1 (\mu \circ_0 v )\sim id_{v}$.
\end{enumerate}
Given an adjunction structure, the two natural transformations are the units and the counits defined in construction \ref{cons:of unit and counit}. Given the data of $(2)$, the adjunction structure of $(u,v)$ is given by the composite 
$$\hom_D(u(a),b)\to \hom_C(vu(a),v(b))\xrightarrow{(\mu_a)_!} \hom_C(a,v(b)).$$
Moreover, the unit of this adjunction is $\mu$ and its counit is $\epsilon$.
\end{theorem}
Before giving the proof of this theorem, we give some corollaries.

\begin{cor}
\label{cor:adjonction induced adjunction by post composition}
Let $(u:B\to C,v:C\to B)$ be an adjoint pair between locally $\U$-small $\io$-categories and $D$ a locally $\U$-small $\io$-category.
If $C$ and $B$ are $\U$-small, this induces an adjunction
\[\begin{tikzcd}
	{\uvar\circ u:\uHom(C,D)} & {\uHom(B,D):\uvar\circ v}
	\arrow[""{name=0, anchor=center, inner sep=0}, shift left=2, from=1-1, to=1-2]
	\arrow[""{name=1, anchor=center, inner sep=0}, shift left=2, from=1-2, to=1-1]
	\arrow["\dashv"{anchor=center, rotate=-90}, draw=none, from=0, to=1]
\end{tikzcd}\]
and if $D$ is $\U$-small an adjunction
\[\begin{tikzcd}
	{u\circ \uvar:\uHom(D,C)} & {\uHom(D,B):v\circ\uvar}
	\arrow[""{name=0, anchor=center, inner sep=0}, shift left=2, from=1-1, to=1-2]
	\arrow[""{name=1, anchor=center, inner sep=0}, shift left=2, from=1-2, to=1-1]
	\arrow["\dashv"{anchor=center, rotate=-90}, draw=none, from=0, to=1]
\end{tikzcd}\]
\end{cor}
\begin{proof}
Let $\mu$ and $\epsilon$ be the unit and the counit of the adjunction. We define $\mu': \uHom(C,D)\times [1]\to \uHom(C,D)$, induced by currying the morphism 
$$ \uHom(C,D)\times [1]\times C\xrightarrow{id\times \mu} \uHom(C,D)\times C\xrightarrow{\ev} D$$
and $\epsilon':\uHom(B,D)\times [1]\to \uHom(B,D)$ by currying the morphism 
$$ \uHom(B,D)\times [1]\times B\xrightarrow{id\times \epsilon} \uHom(B,D)\times B\xrightarrow{\ev} B$$
We can easily check that $\mu'$ and $\epsilon'$ fulfill the triangle identities, and theorem \ref{theo:two adjunction definition} then implies that the pair $(\uvar\circ u,\uvar\circ v)$ admits an adjunction structure. We proceed similarly for the second assertion. 
\end{proof}

\begin{cor}
\label{cor:naive kan extension}
Let $i:I\to A^\sharp$ be a morphism between $\U$-small $\io$-category. The functor $i^*:\uHom(A,\uni)\to \gHom(I,\uni)$ has a left adjoint given by the functor $i_!:\gHom(I,\uni)\to \uHom(A,\uni)$. If $i$ is proper, the functor $i^*$ has a right adjoint $i_*:\gHom(I,\uni)\to \uHom(A,\uni)$.
\end{cor} 
\begin{proof}
In construction \ref{cons: i pull and push beetwe io category of morphism}, for a morphism $i:I\to A^\sharp$ between marked $\io$-categories, we defined the morphism $i_!:\gHom(I,\uni)\to \uHom(A,\uni)$ and when $i$ is proper, a morphism $i_*:\gHom(I,\uni)\to \uHom(A,\uni)$.

With the characterization of adjunction given in theorem \ref{theo:two adjunction definition}, this is a direct consequence of natural transformations given in construction \ref{cons: i pull and push beetwe io category of morphism} and of the lax Grothendieck construction given in theorem \ref{theo:lcartc et ghom}
\end{proof}

\vspace{1cm}
The remaining of the section is devoted to the proof of theorem \ref{theo:two adjunction definition}.

\begin{lemma}
\label{lemma:naturality of hom apply to natural transformation}
Suppose we have two morphisms $f:C\to D$ and $g:C\to D$ between locally $\U$-small $\io$-categories as well as a natural transformation $\nu:f\to g$. This induces a commutative diagram 
\[\begin{tikzcd}
	{\hom_C(a,b)} & {\hom_D(g(a),g(b))} \\
	{\hom_D(f(a),f(b))} & {\hom_D(f(a),g(b))}
	\arrow[from=1-1, to=2-1]
	\arrow[from=1-1, to=1-2]
	\arrow["{(\nu_{a})_!}", from=1-2, to=2-2]
	\arrow["{(\nu_{b})_!}"', from=2-1, to=2-2]
\end{tikzcd}\]
natural in $a:C^t, b:C$.
\end{lemma}
\begin{proof}
Remark that $\hom_{[1]}(0,1)\sim \hom_{[1]}(1,1)\sim \hom_{[1]}(0,0)=1$.
Using the naturality of the hom functor, we have a commutative diagram 
\[\begin{tikzcd}
	{\hom_C(a,b)\times \hom_{[1]}(0,0)} & {\hom_D(f(a),f(b))} \\
	{\hom_C(a,b)\times \hom_{[1]}(0,1)} & {\hom_D(f(a),g(b))} \\
	{\hom_C(a,b)\times \hom_{[1]}(1,1)} & {\hom_D(g(a),g(b))}
	\arrow["{(\nu_{a})_!}"', from=3-2, to=2-2]
	\arrow["{(\nu_{b})_!}", from=1-2, to=2-2]
	\arrow["\sim"', from=1-1, to=2-1]
	\arrow["\sim", from=3-1, to=2-1]
	\arrow[from=2-1, to=2-2]
	\arrow[from=3-1, to=3-2]
	\arrow[from=1-1, to=1-2]
\end{tikzcd}\]
where the left-hand vertical morphisms are equivalences.
\end{proof}

\begin{lemma}
\label{lemma:If unit and counit so adjunction}
Let $u:C\to D$ and $v:D\to C$ be two functors between locally $\U$-small $\io$-categories, $\mu:id_C\to vu$, $\epsilon:uv\to id_D$ be two natural transformations coming along with equivalences 
$$(\epsilon\circ_0 u)~\circ_1~(u\circ_0 \mu) \sim id_{u}~~~~ (v\circ_0 \epsilon)~\circ_1(\mu \circ_0 v )\sim id_{v}.$$
If we set $\phi$ as the composite 
$$\hom_D(u(a),b)\to \hom_C(vu(a),v(b))\xrightarrow{(\mu_a)_!} \hom_C(a,v(b)),$$
the triple $(u,v,\phi)$ is an adjunction structure.
Moreover, the unit of the adjunction is $\mu$ and its counit is $\epsilon$.
\end{lemma}
\begin{proof}
Suppose we have such data. We define $\psi$ as the composite
$$\hom_C(a,v(b))\to \hom_D(u(a),uv(b))\xrightarrow{(\epsilon_a)_!} \hom_D(u(a),b)
$$
natural in $a:C^t$ and $b:D$. We then have to show that $\psi$ and $\phi$ are inverse of each other. For this consider the diagram
\[\begin{tikzcd}
	{\hom_D(u(a),b)} & {\hom_C(vu(a),v(b))} & {\hom_C(a,v(b))} \\
	& {\hom_D(uvu(a),uv(b))} & {\hom_D(u(a),uv(b))} \\
	{\hom_D(u(a),b)} & {\hom_D(uvu(a),b)} & {\hom_D(u(a),b)}
	\arrow["{(\mu_a)_!}", from=1-2, to=1-3]
	\arrow["{(u(\mu_{a}))_!}"', from=3-2, to=3-3]
	\arrow[from=1-1, to=1-2]
	\arrow["{(u(\mu_{a}))_!}", from=2-2, to=2-3]
	\arrow[from=1-2, to=2-2]
	\arrow[from=1-3, to=2-3]
	\arrow["{(\epsilon_b)_!}", from=2-2, to=3-2]
	\arrow["{(\epsilon_b)_!}", from=2-3, to=3-3]
	\arrow["{(\epsilon_{u(a)})_!}"', from=3-1, to=3-2]
	\arrow[Rightarrow, no head, from=1-1, to=3-1]
\end{tikzcd}\]
which is commutative thanks to lemma \ref{lemma:naturality of hom apply to natural transformation} applied to the natural transformation $\epsilon:uv \to id$ and thanks to the naturality of the $\hom$.
By hypothesis, the left lower horizontal morphism is equivalent to the identity.
The outer square then defines an equivalence between $\psi\circ \phi$ and the identity. We show similarly $\phi\circ \psi\sim id$.

For the second assertion, remark that the composition 
$$\hom_C(a,a')\to \hom_D(u(a),u(a'))\xrightarrow{\phi(a,u(a'))} \hom_C(a,vu(a'))$$
is by definition equivalent to 
$$\hom_C(a,a')\to \hom_D(vu(a),vu(a'))\xrightarrow{(\mu_a)_!} \hom_C(a,vu(a'))$$
and according to the lemma \ref{lemma:naturality of hom apply to natural transformation}, to
$$\hom_C(a,a')\xrightarrow{(\mu_{a'})_!} \hom_C(a,vu(a'))$$
The Yoneda lemma then implies that the unit of the adjunction is $\mu$. We proceed similarly for the counit.
\end{proof}

\vspace{1cm} 

For the remaining, we fix two functors $u:C\to D$ and $v:D\to C$ between $\io$-categories as well as an adjunction structure
$$\phi:\hom_{D}(u(a),b)\sim \hom_C(a,v(b))$$
natural in $a:C^t$ and $b:D$.
We will consider the unit $\mu:id_C\to vu$ and the counit $\epsilon:uv\to id$ constructed in \ref{cons:of unit and counit}.

\begin{lemma}
\label{lemma: if ajdunction then unit 1}
The natural transformation
$$\hom_D(u(a),b)\to \hom_C(vu(a),v(b))\xrightarrow{(\mu_a)_!}\hom_C(a,v(b))$$
is equivalent to $\phi:\hom_D(u(a),b)\to \hom_D(a,v(b))$.
Similarly, the natural transformation

$$\hom_C(a,v(b))\to \hom_D(u(a),uv(b))\xrightarrow{(\epsilon_b)_!}\hom_D(u(a),b)$$
is equivalent to $\phi^{-1}:\hom_D(a,v(b))\to \hom_D(u(a),b)$.
\end{lemma}
\begin{proof}
Remark that we have a commutative diagram
\[\begin{tikzcd}
	{\hom_C(a,b)} & {\hom_D(u(a),u(b))} & {\hom_C(vu(a),vu(b))} \\
	{\hom_D(u(a),u(b))} && {\hom_D(a,vu(b))}
	\arrow["{(\mu_a)_!}", from=1-3, to=2-3]
	\arrow[from=1-2, to=1-3]
	\arrow[from=1-1, to=1-2]
	\arrow["{(\mu_b)_!}"{description}, from=1-1, to=2-3]
	\arrow["\phi"{description}, from=2-1, to=2-3]
	\arrow[from=1-1, to=2-1]
\end{tikzcd}\]
The commutativity of the left triangle comes from the definition of $\mu$, and the second one, from the lemma \ref{lemma:naturality of hom apply to natural transformation}, applied to $\mu$.
This then induces a commutative square
\[\begin{tikzcd}
	{\int_{C^t\times C}\hom_C} && {\Rb(id\times u)^*\int_{C^t\times D}\hom_D(u(\uvar),\uvar)} \\
	{\Rb(id\times u)^*\int_{C^t\times D}\hom_D(u(\uvar),\uvar)} && {\Rb(id\times u)^*\int_{C^t\times D}\hom_C(\uvar,v(\uvar))}
	\arrow[from=1-1, to=1-3]
	\arrow["{\Rb (id\times u)^*\int_{C^t\times D}(\mu_a)_!\circ \hom_v}", from=1-3, to=2-3]
	\arrow[from=1-1, to=2-1]
	\arrow["{\Rb(id\times u)^*\int_{C^t\times D}\phi}"', from=2-1, to=2-3]
\end{tikzcd}\]
By adjunction, this corresponds to a commutative square 
\[\begin{tikzcd}
	{\Lb(id\times u)_!\int_{C^t\times C}\hom_C} && {\int_{C^t\times D}\hom_D(u(\uvar),\uvar)} \\
	{\int_{C^t\times D}\hom_D(u(\uvar),\uvar)} && {\int_{C^t\times D}\hom_C(\uvar,v(\uvar))}
	\arrow[from=1-1, to=1-3]
	\arrow["{\int_{C^t\times D}(\mu_a)_!\circ \hom_v}", from=1-3, to=2-3]
	\arrow["{\int_{C^t\times D}\phi}"', from=2-1, to=2-3]
	\arrow[from=1-1, to=2-1]
\end{tikzcd}\]
However, the top horizontal and left vertical morphisms are equivalences according to lemma \ref{lemma:a particular Kan extension}.
We then have an equivalence $$ \int_{C^t\times D}(\mu_a)_!\circ \hom_v\sim \int_{C^t\times D}\phi$$
which implies the result.
The other assertion is shown similarly.
\end{proof}

\begin{lemma}
\label{lemma:if ajdunction then unit 2}
There are equivalences
$(\epsilon\circ_0 u)\circ_1(u\circ_0 \mu) \sim id_{u}$ and $(v\circ_0 \epsilon)\circ_1 (\mu \circ_0 v )\sim id_{v}$.
\end{lemma}
\begin{proof}
As the proof of the two assertions are similar, we will only show the second one.
To demonstrate this, it is enough to show that the induced natural transformation 
\begin{equation}
\label{eq:sequence ajdunction}
\hom_C(a,v(b))\xrightarrow{(\mu_{v(b)})_!} \hom_C(a,vuv(b)) \xrightarrow{(v(\epsilon_{(b)}))_!} \hom_C(a,v(b))\xrightarrow{\phi^{-1}}\hom_D(u(a),b)
\end{equation}
is equivalent to $\phi^{-1}$.
By definition, the first morphism is equivalent to the composition
$$\hom_C(a,v(b))\to \hom_D(u(a),uv(b))\xrightarrow{\phi} \hom_C(a,vuv(b))$$
and as $\phi^{-1}$ is a natural transformation, we have a commutative square
\[\begin{tikzcd}
	{\hom_C(a,vuv(b))} & {\hom_C(a,v(b))} \\
	{\hom_C(u(a),uv(b))} & {\hom_D(u(a),b)}
	\arrow["{(v(\epsilon_{b}))_!}", from=1-1, to=1-2]
	\arrow["{\phi^{-1}}", from=1-2, to=2-2]
	\arrow["{\phi^{-1}}"', from=1-1, to=2-1]
	\arrow["{(\epsilon_{b})_{!}}"', from=2-1, to=2-2]
\end{tikzcd}\]
The composite of the sequence \eqref{eq:sequence ajdunction} is then equivalent to 
$$\hom_C(a,v(b))\to \hom_D(u(a),uv(b))\xrightarrow{(\epsilon_{b})_{!}} \hom_D(u(a),b)$$
which is itself equivalent to $\phi^{-1}$ according to lemma \ref{lemma: if ajdunction then unit 1}.
\end{proof}

\begin{proof}[Proof of theorem \ref{theo:two adjunction definition}]
The implication $(1)\Rightarrow (2)$ is given by lemma \ref{lemma:If unit and counit so adjunction} and the contraposed by the lemma \ref{lemma:if ajdunction then unit 2}.
\end{proof}

\subsection{Lax colimits}

 For $I$ a marked $\io$-category and $A$ an $\io$-category,
we recall that $\gHom(I,A)$ is the $\io$-category whose value on a globular sum $a$ is given by:
$$\Hom(a,\gHom(I,A)):=\Hom(I\ominus a^\sharp, A^\sharp)$$
\begin{remark}
Let $B$ be an $\io$-category.
We want to give an intuition of the object $\gHom(B^\flat,\omega)$. The objects of this $\io$-category are the functors $I\to \omega$. The $1$-cells are the lax transformations $F\Rightarrow G$. For $n>1$, the $n$-cells are the lax transformations $F^{\times \Db_{n-1}}\Rightarrow G$ where $F^{\times \Db_{n-1}}:I\to \omega$ is the functor that sends $i$ onto $F(i)\times \Db_{n-1}$.
This last assertion is a consequence of the equivalence 
$$\tau_0(\LCart((I\ominus [b,n]^\sharp)^\sharp) \sim \Hom([n],\LCartc(I;b))$$
provided by the lemma \ref{lemma:lax univalence 4}.
\end{remark}
\begin{prop}
If $I$ is $\U$-small and $A$ is locally $\U$-small, the $\io$-category $\gHom(I,A)$ is locally $\U$-small.
\end{prop}
\begin{proof}
We have to check that for any globular sum $b$, the morphism 
$$\Hom(I\ominus [b,1]^\sharp,A^\sharp)\to \Hom(I\ominus (\{0\}\amalg\{1\}),A^\sharp)$$
has $\U$-small fibers. As $I$, seen as an $\infty$-presheaves on $t\Theta$, is a $\U$-small colimit of representables, we can reduce to the case where $I\in t\Theta$. As $A$ is local with respect to Segal extensions, and as $\ominus$ conserves them, we can reduce to the case where $I$ is of shape $[1]^\sharp$ or $[a,1]$ for $a$ a in $t\Theta$. If $I$ is $[1]^\sharp$, according to the second assertion of proposition \ref{prop:associativity of ominus}, $[1]^\sharp\ominus [b,1]^\sharp$ is equivalent to $([1]\times [b,1])^\sharp$ and the result follows from proposition \ref{prop:when Hom A B is locally small}.

For the second case, we fix a morphism $f:[a,1]\times (\{0\}\amalg\{1\})\to A$. 
Using the canonical equivalence between $[a,1]\ominus [b,1]^\sharp$ and the colimit of the diagram \eqref{eq:formula for the ominus marked case},
the $\infty$-groupoid $\Hom(I\ominus [b,1]^\sharp,A^\sharp)_f$ is the limit of the diagram:
\[\begin{tikzcd}[column sep=0.1cm]
	{\Hom(a^\natural,\hom(f(1,0),f(1,1)))} & {\Hom((a\otimes\{0\}^\sharp)^\natural\times b,\hom(f(0,0),f(1,1))} \\
	{\Hom((a\otimes[1]^\sharp)^\natural\times b,\hom(f(0,0),f(1,1)))} \\
	{\Hom(a^\natural,\hom(f(0,0),f(1,0)))} & {\Hom((a\otimes\{0\}^\sharp)^\natural\times b,\hom(f(0,0),f(1,1))))}
	\arrow[from=2-1, to=3-2]
	\arrow[from=2-1, to=1-2]
	\arrow[from=3-1, to=3-2]
	\arrow[from=1-1, to=1-2]
\end{tikzcd}\]
As all these objects are $\U$-small by assumption, this concludes the proof.
\end{proof}

\begin{definition}
Let $I$ be a $\U$-small marked $\io$-category, $A$ a locally $\U$-small $\io$-category $A$ and $F:I\to A^\sharp$ a functor. 
A \notion{lax colimit} of $F$ is an object \wcnotation{$\laxcolim_IF$}{(laxcolim@$\laxcolim$} of $A$ together with an equivalence
$$\hom_{A}(\laxcolim_IF, b)\sim \hom_{\gHom(I,A)}(F,\cst b)$$
natural in $b:A$. 
Conversely, a \notion{lax limit} of $F$ is an object \wcnotation{$\laxlim_IF$}{(laxlim@$\laxlim$} of $A$ together with an equivalence
$$\hom_{A}(b,\laxlim_IF)\sim \hom_{\gHom(I,A)}(\cst b,F)$$
natural in $b:A$. 
\end{definition}

\begin{remark}
The proposition \ref{prop:ominus and opmarked}  induces an equivalence
$$\gHom(I,A)^{\circ}\sim\gHom(I^{\circ},A^{\circ})$$
As a consequence, a functor $F:I\to A^\sharp$ admits a lax colimit if and only if $F^\circ:I^\circ\to (A^\circ)^\sharp$ admits a lax limit. If $F$ admits such lax colimit, the lax limit of $F^\circ$ is the image by the canonical equivalence $A_0\sim A^\circ_0$ of the lax colimit of $F$ and we then have
$$\laxcolim_IF\sim \laxlim_{I^{\circ}}F^{\circ}.$$
\end{remark}

\begin{definition}
We say that a locally $\U$-small $\io$-category $C$ is \notion{lax $\U$-complete} (resp. \notion{lax $\U$-cocomplete}), if for any $\U$-small marked $\io$-category $I$ and any functor $F:I\to C$, $F$ admits limits (resp. colimits).
\end{definition}

\begin{remark}
Using proposition \ref{prop:adj if slice as terminal}, $C$ is lax $\U$-complete (resp. lax $\U$-cocomplete) if and only if for any $\U$-small marked $\io$-category $I$, the functor $\cst:C\to \gHom(I,C)$ admits a right adjoint (resp. a left adjoint).
\end{remark}

\begin{remark}
We want to give an intuition of the lax colimits.
Let $I$ be a $\U$-small marked $\io$-category, $A$ a locally $\U$-small $\io$-category $A$ and $F:I\to A^\sharp$ a functor admitting a lax colimit $\laxcolim_IF$. For any $1$-cell $i:a\to b$ in $I$, we have a triangle
\[\begin{tikzcd}
	{} & {F(b)} \\
	{F(a)} & {\laxcolim_IF}
	\arrow["{F(i)}", curve={height=-30pt}, from=2-1, to=1-2]
	\arrow[from=2-1, to=2-2]
	\arrow[shorten <=8pt, shorten >=8pt, Rightarrow, from=1-2, to=2-1]
	\arrow[draw=none, from=1-1, to=2-1]
	\arrow[from=1-2, to=2-2]
\end{tikzcd}\]
If $i$ is marked, the preceding $2$-cell is an equivalence. 
For any $2$-cell $u:i\to j$, we have a diagram
\[\begin{tikzcd}
	& {F(b)} & {} & {F(b)} \\
	{F(a)} & {\laxcolim_IF} & {F(a)} & {\laxcolim_IF}
	\arrow[""{name=0, anchor=center, inner sep=0}, "{F(i)}"{description}, from=2-1, to=1-2]
	\arrow[""{name=1, anchor=center, inner sep=0}, from=2-1, to=2-2]
	\arrow[from=1-2, to=2-2]
	\arrow[""{name=2, anchor=center, inner sep=0}, from=1-2, to=2-2]
	\arrow[""{name=3, anchor=center, inner sep=0}, "{F(j)}", curve={height=-30pt}, from=2-1, to=1-2]
	\arrow["{F(j)}", curve={height=-30pt}, from=2-3, to=1-4]
	\arrow[from=2-3, to=2-4]
	\arrow[from=1-4, to=2-4]
	\arrow[shorten <=8pt, shorten >=8pt, Rightarrow, from=1-4, to=2-3]
	\arrow[""{name=4, anchor=center, inner sep=0}, draw=none, from=1-3, to=2-3]
	\arrow[shift right=2, shorten <=12pt, shorten >=12pt, Rightarrow, from=2, to=1]
	\arrow[shorten <=4pt, shorten >=4pt, Rightarrow, from=3, to=0]
	\arrow[shift left=0.7, shorten <=14pt, shorten >=16pt, no head, from=2, to=4]
	\arrow[shorten <=14pt, shorten >=14pt, from=2, to=4]
	\arrow[shift right=0.7, shorten <=14pt, shorten >=16pt, no head, from=2, to=4]
\end{tikzcd}\]
If $u$ is marked, the $3$-cell is an equivalence. We can continue these diagrams in higher dimensions and we have
similar assertions for lax limits.

The marking therefore allows us to play on the "lax character" of the universal property that the lax colimit must verify.
\end{remark}

\begin{prop}
\label{prop:expliciti colimit for presheaves}
Let $A$ be a $\U$-small $\io$-category and $I$ a $\U$-small marked $\io$-category. The $\io$-category $\widehat{A}$ is lax $\U$-complete and lax $\U$-cocomplete.

For a morphism $g:I\to \w{A}^\sharp$ corresponding to an object $E$ of $\LCartc(I\times (A^t)^\sharp)$, we then have $$
\int_{A^t}\laxcolim_I g \sim \Lb (t\times id_{(A^t)^\sharp})_!E~~~~~~~\int_{A^t}\laxlim_I g \sim \Rb (t\times id_{(A^t)^\sharp})_*E
$$
\end{prop}
\begin{proof}
Recall that $\gHom(I,\w{A})$ is equivalent to $\gHom(I\times (A^t)^\sharp,\uni)$. Let $t$ be the canonical morphism $I\to 1$. 
As $t$ is smooth, corollary \ref{cor:naive kan extension} induces adjunctions
\[\begin{tikzcd}
	{\gHom(I,\w{A})} && {\w{A}}
	\arrow[""{name=0, anchor=center, inner sep=0}, "{(t\times id_A)^*}"{description}, from=1-3, to=1-1]
	\arrow[""{name=1, anchor=center, inner sep=0}, "{(t\times id_A)_!}", shift left=5, from=1-1, to=1-3]
	\arrow[""{name=2, anchor=center, inner sep=0}, "{(t\times id_A)_*}"', shift right=5, from=1-1, to=1-3]
	\arrow["\dashv"{anchor=center, rotate=-90}, draw=none, from=0, to=2]
	\arrow["\dashv"{anchor=center, rotate=-90}, draw=none, from=1, to=0]
\end{tikzcd}\]
As the functor $$(t\times id_A)^*:\gHom(I,\widehat{A})\to \w{A}$$ is equivalent to the functor $\cst:\gHom(I,\widehat{A})\to \w{A}$, this concludes the proof.
\end{proof}

\begin{prop}
\label{prop:eval commutes with colimit and limits}
Let $i:A\to B$ be a morphism between $\U$-small $\io$-categories and $I$ a $\U$-small marked $\io$-category. We have canonical commutative squares:
\[\begin{tikzcd}[cramped]
	{\gHom(I,\w{B})} & {\w{B}} & {\gHom(I,\w{B})} \\
	{\gHom(I,\w{A})} & {\w{A}} & {\gHom(I,\w{A})}
	\arrow["{\laxcolim_I}", from=1-1, to=1-2]
	\arrow["{(id_I\times i^t)^*}"', from=1-1, to=2-1]
	\arrow["{i^*}", from=1-2, to=2-2]
	\arrow["{\laxlim_I}"', from=1-3, to=1-2]
	\arrow["{(id_I\times i^t)^*}", from=1-3, to=2-3]
	\arrow["{\laxcolim_I}"', from=2-1, to=2-2]
	\arrow["{\laxlim_I}", from=2-3, to=2-2]
\end{tikzcd}\]
\end{prop}
\begin{proof}
This directly follows from the squares given in \ref{cons: i pull and push beetwe io category of morphism} and from proposition \ref{prop:expliciti colimit for presheaves}.
\end{proof}

\begin{remark}
In particular, choosing $A:=1$, the previous proposition implies that the lax colimits and lax limits in $\io$-presheaves commute with evaluation.
\end{remark}

The next proposition implies that limits and colimits in $\io$-presheaves can be detected as the level of the sub maximal $\iun$-categories of $\gHom(I,\w{A})$ and $\w{A}$. We recall that the sub maximal $\iun$-categories of $\gHom(I,\w{A})$, denoted by $\tau_1\gHom(I,\w{A})$, is the adjoint of the functor $[n]\mapsto (I\otimes[n]^\sharp)^\natural$.
\begin{prop}
Let $I$ be a $\U$-small marked $\io$-category, and $g:I\to A^\sharp$ a functor. An object $f$ of $\w{A}$ has a structure of colimit of the functor $g$ if and only if there exists an equivalence
$$\Hom_{\tau_1\w{A}}(f,h)\sim \Hom_{\tau_1\gHom(I,\w{A})}(F,\cst h)$$
natural in $h:(\tau^1 \w{A})^{op}$.
Similarly, the object $f$ has a structure of limit of the functor $F$ if and only if there exists an equivalence
$$\Hom_{\tau_1\w{A}}(h,f)\sim \Hom_{\tau_1\gHom(I,\w{A})}(\cst h,F)$$
natural in $h:(\tau^1 \w{A})^{op}$.
\end{prop}
\begin{proof}
We recall that theorem \ref{theo:lcartc et ghom} and corollary \ref{cor:lcar et hom} induces equivalences
$$\tau_1\w{A}\sim \LCart_{\U}((A^t)^\sharp)~~~ \tau_1\gHom(I,A)\sim \LCartc_{\U}(I\otimes (A^t)^\sharp)$$
and that we have a triplet of adjoints
\[\begin{tikzcd}
	{\LCartc_{\U}(I\otimes (A^t)^\sharp)} && {\LCart_{\U}((A^t)^\sharp)}
	\arrow[""{name=0, anchor=center, inner sep=0}, "{\Lb (t\times id_{A^t})_!}", shift left=5, from=1-1, to=1-3]
	\arrow[""{name=1, anchor=center, inner sep=0}, "{\Rb(t\times id_{A^t})_*}"', shift right=5, from=1-1, to=1-3]
	\arrow[""{name=2, anchor=center, inner sep=0}, "{\Rb(t\times id_{A^t})^*}"{description}, from=1-3, to=1-1]
	\arrow[""{name=3, anchor=center, inner sep=0}, "\dashv"{anchor=center, rotate=-90}, draw=none, from=0, to=2]
	\arrow["\dashv"{anchor=center, rotate=-90}, draw=none, from=3, to=1]
\end{tikzcd}\]
which is the image by $\tau_1$ of the triplet of adjoints given in the proof of proposition \ref{prop:expliciti colimit for presheaves}. As $\Rb(t\times id_{A^t})^*:\LCart_{\U}((A^t)^\sharp)\to\LCartc_{\U}(I\otimes (A^t)^\sharp)$ is equivalent to the functor $\cst:\tau_1\gHom(I,A)\to \tau_1\w{A}$, this concludes the proof.
\end{proof}

\begin{example}
We recall that we denote by $\bot:\Arr(\ocatm)\to \ocat$ the functor sending a left fibration $Y\to A$ to the localization of $Y$ by marked cells. This functors sends initial and final morphisms to equivalences. If $E$ is a left cartesian fibration over a marked $\io$-category $I$, we then have $\bot E\sim \Lb t_! E$ where $t$ denotes the morphism $I\to 1$.

 Let $g:I\to \uni$ be a diagram. We denote $\iota:I\to I^\sharp$ the canonical inclusion.
By the explicit expression of lax colimit given in proposition \ref{prop:expliciti colimit for presheaves}, we then have an equivalence 
$$\laxcolim_I g \sim \bot \iota^*\int_{I^{\natural}}g^\natural.$$
If $I$ is equivalent to $I^\flat$, we then have
$$\laxcolim_I g \sim \dom(\int_{I^{\natural}}g^\natural)^\natural.$$
 \begin{enumerate}
 \item[$-$]
Let $c:1\to \uni$ be a morphism corresponding to an $\io$-category $C$. For any $\io$-category $A$, we then have 
$$\laxcolim_{A^\sharp} \cst_c\sim (\tau_0 A)\times C~~~~~\laxcolim_{A^\flat} \cst_c\sim A\times C$$

 \item[$-$] Let $f:[b,1]\to \uni$ be a morphism corresponding to a morphism $A\times b\to B$. We then have 
 $$\laxcolim_{[b,1]^\flat} f\sim A\times (1\costar b)\coprod_{A\times b}B$$
 \end{enumerate}
\end{example}

\begin{example}
Using the explicit expression of lax limit given in proposition \ref{prop:expliciti colimit for presheaves}, we have an equivalence
$$\laxlim_I g \sim \Map(id_I,\iota^*\int_{I^{\natural}}g^\natural)$$
 \begin{enumerate}
 \item[$-$]
Let $c:1\to \uni$ be a morphism corresponding to an $\io$-category $C$. For any $\io$-category $A$, we then have 
$$\laxlim_{A^\sharp} \cst_c\sim \uHom(\tau_0 A, C)~~~~~\laxlim_{A^\flat} \cst_c\sim \uHom(A,C)$$

 \item[$-$] Let $f:[b,1]\to \uni$ be a morphism corresponding to a morphism $A\times b\to B$. Let $c$ be a globular sum. According to corollary \ref{cor:univalence tranche}, a morphism $id_{[b,1]^\flat}\times c^\flat\to \iota^*\int_{[b,1]^{\flat}}g^\natural$ corresponds to a diagram 
\[\begin{tikzcd}
	1 \\
	& {1\costar [b,1]} & \uni \\
	{[b,1]}
	\arrow[from=1-1, to=2-2]
	\arrow["{\{c\}}", curve={height=-12pt}, from=1-1, to=2-3]
	\arrow[from=2-2, to=2-3]
	\arrow[from=3-1, to=2-2]
	\arrow["f"', curve={height=12pt}, from=3-1, to=2-3]
\end{tikzcd}\]
and according to proposition \ref{prop:lfib and W 3}, to a diagram	
\begin{equation}
\label{eq:in exemple of limit}
\begin{tikzcd}[cramped]
	{c\times b\otimes\{0\}} && {A\times b} \\
	& {c\times( b\otimes[1])} && B \\
	{c\times b\otimes\{1\}} && c
	\arrow[from=1-1, to=1-3]
	\arrow[from=1-1, to=2-2]
	\arrow[from=1-3, to=2-4]
	\arrow[from=2-2, to=2-4]
	\arrow[from=3-1, to=2-2]
	\arrow[from=3-1, to=3-3]
	\arrow[from=3-3, to=2-4]
\end{tikzcd}
\end{equation}
where the upper horizontal morphism is of shape $g\times b$. 
As the cartesian product commutes with colimits, proposition \ref{prop:eq for cylinder} implies that the following square is cocartesian:
\[\begin{tikzcd}
	{c\times b\otimes\{1\}} && c \\
	{c\times( b\otimes[1])} && {c\times 1\star b}
	\arrow[from=1-1, to=1-3]
	\arrow[from=1-1, to=2-1]
	\arrow[from=1-3, to=2-3]
	\arrow[""{name=0, anchor=center, inner sep=0}, from=2-1, to=2-3]
	\arrow["\lrcorner"{anchor=center, pos=0.125}, draw=none, from=1-1, to=0]
\end{tikzcd}\]
The diagram \eqref{eq:in exemple of limit} then corresponds to a diagram:
\[\begin{tikzcd}[cramped]
	{c\times b\otimes\{0\}} && {A\times b} \\
	{c\times b\star 1} && B
	\arrow[from=1-1, to=1-3]
	\arrow[from=1-1, to=2-1]
	\arrow[from=1-3, to=2-3]
	\arrow[from=2-1, to=2-3]
\end{tikzcd}\]
We then have 
 $$\laxlim_{[b,1]^\flat} f\sim A\prod_{\uHom(b,B)} \uHom(b\star 1,B).
 $$
 \end{enumerate}
 \end{example}

\begin{lemma}
\label{lemma:tehcnical colimit}
Let $F:I\to A^\sharp$ be a morphism between $\U$-small marked $\io$-categories. 
There is an equivalence 
$$ \hom_{\gHom(I,A)}(\cst_a,F)\sim \laxlim_I\hom_A(a,F(\uvar))$$
natural in $F: \gHom(I,A)$ and $a:A^t$.
\end{lemma}
\begin{proof}
Remark that there is a commutative square:
\[\begin{tikzcd}
	A & {\gHom(I,A)} \\
	{\w{A}} & {\gHom(I,\w{A})}
	\arrow["\cst", from=1-1, to=1-2]
	\arrow["\cst"', from=2-1, to=2-2]
	\arrow["y"', from=1-1, to=2-1]
	\arrow["{\gHom(I,y)}", from=1-2, to=2-2]
\end{tikzcd}\]
and that the right vertical morphism is fully faithful as $y$ is.
We then have a sequence of equivalences
$$
\begin{array}{rcll}
\hom_{\gHom(I,A)}(\cst_a,F)&\sim& \hom_{\gHom(I,\w{A})}(\cst_{y_a},\gHom(I,y)(F))\\
&\sim& \hom_{\w{A}}(y_a,\laxlim_{I}\gHom(I,y)(F)))\\
&\sim& (\laxlim_{I}\gHom(I,y)(F))(a)&\mbox{(Yoneda lemma)}\\
&\sim& \laxlim_I\hom_A(a,F(i))
\end{array}$$
where the last one comes from the fact that evaluations commute with lax limits.
\end{proof}

\begin{prop}
\label{prop:other characthereisation of limits}
Consider a functor $F:I\to A^\sharp$ between $\U$-small marked $\io$-categories. We have an equivalence
$$\hom_A(a,\laxlim_IF)\sim \laxlim_I\hom_A(a,F(\uvar))$$
natural in $a:A^t$.
Dually, we have an equivalence
$$\hom_A(\laxcolim,a)\sim \laxlim_{I^t}\hom_A(F(\uvar),a)^\circ$$
natural in $a:A$. 
\end{prop}
\begin{proof}
By lemma \ref{lemma:tehcnical colimit} and by definition of lax limit, we have two equivalences
 $$\hom_A(a,\laxlim_IF)  \sim \laxlim_I\hom_A(a,F(\uvar))\sim \hom(\cst_a,F)$$
This proves the first assertion. The second one follows by duality, using the fact that the functor 
$$(\uvar)^\circ:\uni\to \uni^{t\circ}$$
preserves limits as it is an equivalence.
\end{proof}

\begin{cor}
\label{cor:left adjoint preserves limits}
Left adjoints between $\U$-small $\io$-categories preserve lax colimits and right adjoints preserve lax limits.
\end{cor}
\begin{proof}
Let $u:C\to D$ and $v:D\to C$ be two adjoint functors. Let $F:I\to C^\sharp$ be a functor admitting a lax colimit.
We then have a sequence of equivalences
$$
\begin{array}{rclc}
\hom_C(u(\laxcolim_IF),b)&\sim &\hom_D(\laxcolim_IF,v(b))\\
&\sim & \laxlim_{I^t}\hom_D(F,v(b))^\circ&(\ref{prop:other characthereisation of limits})\\
&\sim &\laxlim_{I^t}\hom_C(u(F),b)^\circ\\
&\sim &\hom_C(\laxcolim_Iu(F),b)&(\ref{prop:other characthereisation of limits})
\end{array}
$$
natural in $b:D$. The result then follows from the Yoneda lemma applied to $C^t$. The other assertion is proved similarly.
\end{proof}

\begin{cor}
Consider a functor $F:I\to A^\sharp$ between $\U$-small marked $\io$-categories. Then $F$ admits a lax limit if and only if there exists an object $l$ and an equivalence
$$\hom_A(a,l)\sim \hom_{\gHom(I,\uni)}(\cst 1,\hom_A(a,F(\uvar))$$
natural in $a:A^t$. If such an object exists, then $l$ is a lax limit of $F$.
\end{cor}
\begin{proof}
Remark that we have an equivalence
$$\hom_{\gHom(I,\uni)}(\cst 1,\hom_A(a,F(\uvar)))\sim \hom_{\uni}(1,\laxlim_I\hom_A(a,F(\uvar))$$
Eventually, the Yoneda lemma implies that 
$$\hom_{\uni}(1,\laxlim_I\hom_A(a,F(\uvar))\sim\laxlim_I\hom_A(a,F(\uvar))$$
The result then follows from proposition \ref{prop:other characthereisation of limits}.
\end{proof}

\begin{remark}
The characterization of  lax limit given in previous corollary is the generalization to the case $\io$ of the characterization of lax limit for $(\infty,2)$-categories given in \cite[corollary 5.1.7]{Gagna_fibrations_and_lax_limit_infini_2_categories}. The other follows by duality.
\end{remark}

The proof of the following proposition is a direct adaptation of the one of proposition 5.1 of \cite{Gepner_Lax_colimits_and_free_fibration}.
\begin{prop}
\label{prop:explicit hom between morphism}
Let $f,g:A\to B$ be two morphisms between $\U$-small $\io$-categories.
There is an equivalence
$$\hom_{\uHom(A,B)}(f,g)\sim \laxlim_{a\to b: S(A)}\hom_{B}(f(a),g(a)).$$
\end{prop}
\begin{proof}
Remark first that the right term is in fact equivalent to 
$$\laxlim_{a\to b: S(A)}h^*\hom_{B}(\uvar,\uvar)$$
where $h$ is the left cartesian fibration $S(A)\to A^t\times A$ corresponding to $\hom_A: A^t\times A\to \uni$. We then have 
$$
\begin{array}{rcll}
\laxlim_{a\to b: S(A)}\hom_{B}(f(a),g(a)) &\sim &\hom_{\uni}(1,\laxlim_{a\to b: S(A)}h^*\hom_{B}(\uvar,\uvar)) &(\ref{theo:Yoneda lemma})\\
&\sim &\hom_{\gHom(S(A),\uni)}(\cst 1,h^*\hom_{B}(\uvar,\uvar))\\
&\sim &\hom_{\uHom(A^t\times A,\uni)}(h_! \cst 1,\hom_{B}(\uvar,\uvar))&(\ref{cor:naive kan extension})\\
\end{array}$$
By construction, $h_! \cst 1$ is the Grothendieck deconstruction of the left cartesian fibration $\Lb h_!id \sim h$, and so is equivalent to $\hom_A$. 
We then have 
$$\laxlim_{a\to b: S(A)}\hom_{B}(f(a),g(a))\sim \hom_{\uHom(A^t\times A,\uni)}(\hom_A(\uvar,\uvar),\hom_B(f(\uvar),g(\uvar)))$$
We have a canonical equivalence $\uHom(A^t\times A,\uni)\sim \uHom(A,\w{A})$ sending the functor $\hom_A$ to the Yoneda embedding $y^A$, and $\hom_B(f(\uvar),g(\uvar))$ is $f^*(y^B\circ g)$. 
We then have 
$$
\begin{array}{rcll}
\hom_{}(\hom_A(\uvar,\uvar),\hom_B(f(\uvar),g(\uvar)))&\sim &\hom_{\uHom(A,\w{A})}(y^A,f^*(y^B\circ g))\\
&\sim &\hom_{\uHom(A,\w{B})}(f_!\circ y^A,y^B\circ g)&(\ref{cor:naive kan extension})\\
&\sim &\hom_{\uHom(A,\w{B})}(y^B\circ f,y^B\circ g)&(\ref{prop:left extension commutes with Yoneda})\\
&\sim &\hom_{\uHom(A,B)}(f, g)&(\mbox{Yoneda lemma})\\
\end{array}$$
\end{proof}

\begin{theorem}
\label{theo:initiality colimit}
Let  $i:I\to J$ be a  morphisms between $\U$-small marked $\io$-categories. We denote $\bot:\ocatm\to \ocat$ the functor sending a marked $\io$-category to its localization by marked cells.

The following are equivalent 
\begin{enumerate}
\item For any morphism  $F:J\to A^\sharp$  that admits a lax colimit, the functor $F\circ i$ also admits a lax colimit, and the canonical morphism:
$$ \laxcolim_{J}F\circ i \to \laxcolim_{I}F$$
is an equivalence.
\item  For any object $j$ in $J$, the canonical morphism 
$$\bot I_{j/}\to \bot J_{j/}$$
is an equivalence, where $I_{j/}$ and $J_{j/}$ fit in the cartesian squares
\[\begin{tikzcd}
	{I_{j/}} & {J_{j/}} & {J^{\sharp}_{j/}} \\
	I & J & {J^\sharp}
	\arrow[from=1-1, to=1-2]
	\arrow[from=1-1, to=2-1]
	\arrow["\lrcorner"{anchor=center, pos=0.125}, draw=none, from=1-1, to=2-2]
	\arrow[from=1-2, to=1-3]
	\arrow[from=1-2, to=2-2]
	\arrow["\lrcorner"{anchor=center, pos=0.125}, draw=none, from=1-2, to=2-3]
	\arrow[from=1-3, to=2-3]
	\arrow[from=2-1, to=2-2]
	\arrow[from=2-2, to=2-3]
\end{tikzcd}\]
\end{enumerate}
Dually, the following are equivalent 
\begin{enumerate}
\item[(1)'] For any morphism  $F:J\to A^\sharp$  that admits a lax limit, the functor $F\circ i$ also admits a lax limit, and the canonical morphism:
$$\laxlim_{I}F\to \laxlim_{J}F\circ i$$
is an equivalence.
\item[(2)']
For any object $j$ in $J$, the canonical morphism 
$$\bot I_{/j}\to \bot J_{/j}$$
is an equivalence, where $I_{/j}$ and $J_{/j}$ fit in the cartesian squares
\[\begin{tikzcd}
	{I_{/j}} & {J_{/j}} & {J^{\sharp}_{/j}} \\
	I & J & {J^\sharp}
	\arrow[from=1-1, to=1-2]
	\arrow[from=1-1, to=2-1]
	\arrow["\lrcorner"{anchor=center, pos=0.125}, draw=none, from=1-1, to=2-2]
	\arrow[from=1-2, to=1-3]
	\arrow[from=1-2, to=2-2]
	\arrow["\lrcorner"{anchor=center, pos=0.125}, draw=none, from=1-2, to=2-3]
	\arrow[from=1-3, to=2-3]
	\arrow[from=2-1, to=2-2]
	\arrow[from=2-2, to=2-3]
\end{tikzcd}\]
\end{enumerate}
\end{theorem} 
We directly give a corollary.
\begin{cor}
\label{cor:limit and final}
Let $i:I\to J$ and $F:J\to A^\sharp$ be two morphisms between $\U$-small marked $\io$-categories.
If $i$ is final, and $F$ admits a lax colimit, the functor $F\circ i$ also admits a lax colimit, and the canonical morphism:
$$ \laxcolim_{J}F\circ i \to \laxcolim_{I}F$$
is an equivalence.

Dually, if $i$ is initial, and $F$ admits a lax limit, the functor $F\circ i$ also admits a lax limit, and the canonical morphism:
$$\laxlim_{I}F\to \laxlim_{J}F\circ i$$
is an equivalence.
\end{cor}
\begin{proof}
This directly follows from theorem \ref{theo:initiality colimit} and proposition \ref{prop:quillent theorem A}.
\end{proof}
To prove theorem \ref{theo:initiality colimit}, we need several lemmas.

\begin{lemma}
\label{lemma:colimit restricted to final}
Let $i:I\to J$ be a morphism between $\U$-small marked $\io$-categories and $f:J\to \uni$ a morphism.  If $i$ is initial, then the canonical morphism
$$\laxlim_{I}f\to \laxlim_{J}f	\circ i$$
is an equivalence.
\end{lemma}
\begin{proof}
We denote by $E$ (resp. $H$) the object of $\LCart(J)$ (resp. $\LCart(I)$) corresponding to $f$ (resp. $f\circ i$), and $p:X\to J$ (resp. $q:Y\to I$) the corresponding left cartesian fibration. We then have a cartesian square:
\[\begin{tikzcd}
	Y & X \\
	I & J
	\arrow["{i'}", from=1-1, to=1-2]
	\arrow["q"', from=1-1, to=2-1]
	\arrow["p", from=1-2, to=2-2]
	\arrow["i"', from=2-1, to=2-2]
\end{tikzcd}\]
By proposition \ref{prop:expliciti colimit for presheaves}, for any $\io$-category $C$, we have 
$$\Hom(C,\laxlim f)\sim \Hom_{\ocatm_{/J}}(C\times id_i,p)$$
$$
\Hom(C,\laxlim f\circ i)\sim \Hom_{\ocatm_{/I}}(C\times id_j,q)\sim \Hom_{\ocatm_{/I}}(C\times i,p)$$
where we denote by $C$ the canonical morphism $C\to 1$.
As $i$ is initial and $p$ is a left cartesian fibration, we have 
$$ \Hom_{\ocatm_{/J}}(C\times id_i,p)\sim  \Hom_{\ocatm_{/I}}(C\times i,p)$$ which concludes the proof.
\end{proof}

\begin{lemma}
\label{lemma:colimit restricted to final2}
Let $i:I\to J$ be a morphism between $\U$-small marked $\io$-categories, $A$ an locally $\U$-small $\io$-category, and $f:J\to A^\sharp$ a morphism. If $i$ is final, then the canonical morphism
$$\laxcolim_{I}f\circ i\to \laxcolim_{J}f$$
is an equivalence.
\end{lemma}
\begin{proof}
Suppose that $i$ is final. Remark that according to \ref{prop:other characthereisation of limits}, it is sufficient to demonstrate for any integer $a$ that the morphism
$$ \laxlim_{J^t}\hom(f(\uvar),a)^\circ\to  \laxlim_{I^t}\hom(f\circ i(\uvar),a)^\circ$$ 
is an equivalence. Using the equivalence $(\uvar)^\circ:\uni^{t\circ}\to \uni$,
this is equivalent to demonstrating that  
the morphism
$$ \laxlim_{J^\circ}\hom(f(\uvar),a)^{\circ}\to  \laxlim_{I^\circ}\hom(f\circ i(\uvar),a)^{\circ}$$
is an equivalence. As $I^{\circ}\to J^{\circ}$ is initial, this follows from lemma \ref{lemma:colimit restricted to final2}.
\end{proof}

\begin{proof}[ Proof of theorem \ref{theo:initiality colimit}]
Let  $i:I\to J$ be a morphism fulfilling condition $(1)$. Let $f:J\to A^\sharp$ be a morphism. Remark first that $f$ factors as
$$J\xrightarrow{\iota} J^\sharp\xrightarrow{g} A^\sharp.$$
We denote $p:K\to J^{\sharp}$ the left fibrant cartesian replacement of $\iota$. By proposition \ref{prop:quillent theorem A}, we then have a canonical triangle:
\[\begin{tikzcd}
	I & J \\
	& K
	\arrow["i", from=1-1, to=1-2]
	\arrow[from=1-1, to=2-2]
	\arrow[from=1-2, to=2-2]
\end{tikzcd}\]
where the two morphisms with codomain $K$ are initial. This induces a triangle:
\[\begin{tikzcd}
	{\laxcolim_I f\circ i} & {\laxcolim_Jf} \\
	& {\laxcolim_Kp}
	\arrow[from=1-1, to=1-2]
	\arrow[from=1-1, to=2-2]
	\arrow[from=1-2, to=2-2]
\end{tikzcd}\]
According to lemma \ref{lemma:colimit restricted to final2}, the  vertical and the diagonal morphisms are equivalences, and so is the horizontal one by two out of three. This concludes the proof of the implication $(1)\Rightarrow (2)$.

To show the converse, let $i:I\to J$ be a morphism fulfilling condition $(2)$ and $j$ an element of $J$. 
Consider the functor 
$$f:J\to J^{\sharp}\xrightarrow{\hom_{J^{\natural}}(j,\uvar))} \uni$$
By proposition \ref{prop:Yoneda is Fb}, $f$ and $f\circ i$ correspond via the Grothendieck construction to the classified left cartesian fibrations $J_{j/}\to J$ and $I_{j/}\to I$. By proposition \ref{prop:expliciti colimit for presheaves}, and by assumption, we then have equivalences
$$\bot I_{j/}\sim \laxcolim_If\circ i \sim \laxcolim_Jf  \sim\bot J_{j/}$$

The equivalence between $(1)$ and $(2)$ follows by duality.

\end{proof}

\vspace{1cm} We suppose the existence of a Grothendieck universe $\Z$ containing $\Wcard$. As a consequence, we can use all the results of the last three subsections to respectively $\V$-small and locally $\V$-small objects.

\begin{definition}
Let $A$ be a $\U$-small $\io$-category. Let $f$ be an object of $\w{A}$. We define $A^\sharp_{/f}$ as the following pullback
\[\begin{tikzcd}
	{A^\sharp_{/f}} & {\w{A}^\sharp_{/f}} \\
	{A^\sharp} & {\w{A}^\sharp}
	\arrow[from=1-1, to=2-1]
	\arrow[from=1-1, to=1-2]
	\arrow[from=1-2, to=2-2]
	\arrow[from=2-1, to=2-2]
\end{tikzcd}\]
\end{definition}

\begin{theorem}
\label{theo:presheaevs colimi of representable}
The colimit of the functor 
$\pi:A^\sharp_{/f}\to A^\sharp\to \w{A}^\sharp$ is $f$.
\end{theorem}
\begin{proof}
We denote by $\pi'$ the canonical projection $\w{A}^\sharp_{/f}\to \w{A}^\sharp$, and 
$t_{A^\sharp_{/f}}:A^\sharp_{/f}\to 1$, $t_{ \w{A}^\sharp_{/f}}:\w{A}^\sharp_{/f}\to 1$ the canonical morphisms.
Let  $E$ be the object of $\LCart(A^\sharp\times A^\sharp_{/f})$ induced by currying $\pi$, 
and $F$  the object of $\LCart(A^\sharp\times \w{A}^\sharp_{/f})$ induced by currying $\pi'$.
By proposition \ref{prop:expliciti colimit for presheaves}, we have equivalences 
$$\int_{A^t}\colim_{A^\sharp_{/f}}\pi \sim (id_{(A^t)^\sharp}\times t_{A^\sharp_{/f}})_!E
~~~~~~~~
\int_{A^t}\colim_{\w{A}^\sharp_{/f}}\pi' \sim (id_{(A^t)^\sharp}\times t_{ \w{A}^\sharp_{/f}})_!F$$

We denote by $X\to A^\sharp\times A^\sharp_{/f}$ the left cartesian fibration corresponding to $E$,
and by $Y\to (A^t)^\sharp\times \w{A}^\sharp_{/f}$ the left fibration corresponding to $F$. All this data fits in the diagram
\[\begin{tikzcd}
	X && {S(A)} \\
	& Y && {\dom(\int_{A^t\times \w{A}}\ev)} \\
	{(A^t)^\sharp\times A^{\sharp}_{/f}} && {(A^t)^\sharp\times A^{\sharp}} \\
	& {(A^t)^\sharp\times (\w{A})^{\sharp}_{/f}} && {(A^t)^\sharp\times \w{A}^\sharp}
	\arrow["E", from=1-1, to=3-1]
	\arrow[from=3-1, to=3-3]
	\arrow[from=1-1, to=1-3]
	\arrow[from=1-3, to=3-3]
	\arrow[from=3-3, to=4-4]
	\arrow[from=3-1, to=4-2]
	\arrow["i", from=1-3, to=2-4]
	\arrow[from=2-4, to=4-4]
	\arrow[from=4-2, to=4-4]
	\arrow["j", from=1-1, to=2-2]
	\arrow[from=2-2, to=2-4]
	\arrow["F"{pos=0.4}, from=2-2, to=4-2]
\end{tikzcd}\]
where all squares are cartesian. 
Furthermore, according to the Yoneda lemma, $\dom(\int_{A^t\times \w{A}}\ev))$ is equivalent to $\dom(\int_{A^t\times \w{A}}\hom_{\w{A}}(y_{\uvar},\uvar))$, and 
lemma \ref{lemma:a particular Kan extension} implies that $i$ is initial. As the lower horizontal morphism is a right cartesian fibration, and the dual version of proposition \ref{prop:left cartesian fibration are smooth} induces that $j$ is initial. 
This implies that the canonical morphism
$$(id_{(A^t)^\sharp}\times \bot_{A^\sharp_{/f}})_!E\to (id_{(A^t)^\sharp}\times \bot_{ \w{A}^\sharp_{/f}})_!F$$
is an equivalence, and we then have
$$\colim_{A^\sharp_{/f}}\pi\sim \colim_{\w{A}^\sharp_{/f}}\pi'$$
However, 
$A^\sharp_{/f}$ admits a terminal element, given by $id_f$, and according to corollary \ref{cor:limit and final}, we have 
$$\colim_{A^\sharp_{/f}}\pi\sim f.$$
\end{proof}

\begin{cor}
\label{cor:if cocomplete then Yoneda right adjoint}
A $\U$-small $\io$-category $A$ is lax $\U$-cocomplete if and only if the Yoneda embedding has a left adjoint, which we will also note by \wcnotation{$\laxcolim$}{(laxcolim@$\laxcolim:\widehat{C}\to C$}.
\end{cor}
\begin{proof}
If such a left adjoint exists, as $\w{A}$ is lax $\U$-cocomplete, corollary 
\ref{cor:left adjoint preserves limits} implies that $A$ is lax $\U$-cocomplete. Suppose now that $A$ is lax $\U$-cocomplete and let $f:A^t\to \uni$ be a functor. Let $c$ be the colimit of the functor $A^\sharp_{/f}\to A^\sharp$.
According to theorem \ref{theo:presheaevs colimi of representable}, we have a sequence of equivalences
$$\begin{array}{rcll}
 \hom_{\w{A}}(f,y(a))&\sim &\hom_{\w{A}}(\laxcolim_{A^\sharp_{/f}}y(\uvar),y(a))\\
 &\sim & \laxlim_{(A^\sharp_{/f})^t}\hom_{\w{A}}(y(\uvar),y(a))^\circ\\
 &\sim & \laxlim_{(A^\sharp_{/f})^t} \hom_A(\uvar,a)^\circ&\\
 &\sim &\hom_{A} (c,a)\\
\end{array}$$
natural in $a:A^t$. The functor 
$$a:A\mapsto \hom_{\w{A}}(f,y(a))$$
is then representable, which concludes the proof according to proposition \ref{prop:adj if slice as terminal}.
\end{proof}

\begin{definition} Let $i:A\to B$ be a functor between two $\U$-small $\io$-categories. We define $N_i:B\to \w{A}$ as
$$a:A^t, b:B\mapsto \hom_B(i(a),b)$$
\end{definition}
\begin{cor}
\label{cor:adjonction with prehseaves}
Let $i:A\to B$ be a functor between two $\U$-small $\io$-categories with $B$ lax $\U$-cocomplete. 
The morphism $N_i:B\to \w{A}$
admits a left adjoint that sends an $\io$-presheaf $f$ to $\laxcolim_{A^\sharp_{/f}} i(\uvar)$
\end{cor}
\begin{proof}
The proof is similar to the one of corollary \ref{cor:if cocomplete then Yoneda right adjoint}.
\end{proof}

\subsection{Kan extentions}

We suppose the existence of a Grothendieck universe $\Z$ containing $\Wcard$. As a consequence, we can use all the results of the last three subsections to respectively $\V$-small and locally $\V$-small objects.

\begin{definition}
Let $f:A\to B^\sharp$ be a morphism between marked $\U$-small $\io$-categories. This induces for any $\io$-category $C$ a morphism 
$$\uvar\circ f:\gHom(B,C)\to \uHom(A,C).$$
Let $g:A\to C$ be a morphism.
A \notion{left Kan extension} of $g$ along $f$ is a functor \sym{(lanf@$\Lan_fg$}$\Lan_fg:B\to C$ and an equivalence
$$\hom_{\uHom(B,C)}(\Lan_fg,h)\sim \hom_{\gHom(A,C)}(g, h\circ f).$$
Remark that if the left Kan extension along $f$ exists for any $g$, the proposition \ref{prop:adj if slice as terminal} implies that
the assignation $g\mapsto \Lan_fg$ can be promoted to a left adjoint, which is called the \notion{global left Kan extension} of $f$. 
\end{definition}

\begin{prop}
\label{prop:Kan extension an naive kan extension}
Let $C$ be a $\U$-small $\io$-category, $f:I\to B^\sharp$ a functor between $\U$-small $\io$-categories and $g:I\to \uHom(C,\uni)$ a functor. The functor $g$ then corresponds to a morphism $\tilde{g}:\gHom( C^\sharp\times I,\uni)$.
The left Kan extension of $f$ along $g$ corresponds to the morphism $(id_{C^\sharp}\times f)_!\tilde{g}$.
\end{prop}
\begin{proof}
This is a direct consequence of corollary \ref{cor:naive kan extension}.
\end{proof}
	
\begin{cor}
\label{cor:Kan extension of Yonedal along i}
Let $i:A\to B$ be a morphism between $\U$-small $\io$-categories. The left Kan extension of the Yoneda embedding $y:A\to \w{A}$ along $i$ is $N_i:B\to \widehat{A}$.
\end{cor}
\begin{proof}
According to proposition \ref{prop:Kan extension an naive kan extension}, the desired left Kan extension is given by 
$$(B^t\times i)_!\hom_B$$
which is $N_i$ according to lemma \ref{lemma:a particular Kan extension}.
\end{proof}

\begin{construction}
\label{cons:adjoint presheaves}
 According to corollary \ref{cor:naive kan extension}, a morphism $f:A\to B$ between $\U$-small $\io$-categories induces an adjoint pair:
\begin{equation}
\begin{tikzcd}
	{f_!:\w{A}} & {\w{B}:f^*}
	\arrow[""{name=0, anchor=center, inner sep=0}, shift left=2, from=1-1, to=1-2]
	\arrow[""{name=1, anchor=center, inner sep=0}, shift left=2, from=1-2, to=1-1]
	\arrow["\dashv"{anchor=center, rotate=-90}, draw=none, from=0, to=1]
\end{tikzcd}
\end{equation}
\end{construction}

\begin{prop}
\label{prop:left extension commutes with Yoneda}
Let $f:A\to B$ be a morphism between $\U$-small $\io$-categories.
There is an equivalence 
$$f_!(y_a)\sim y_{f(a)}$$
natural in $a:A$.
\end{prop}
\begin{proof}
Consider the sequence of equivalences
$$\begin{array}{rcll}
\hom_{\w{B}}(f_!(y_a), g)&\sim &\hom_{\w{A}}(y_a, f^*(g))& (\ref{cons:adjoint presheaves})\\
&\sim &\ev(a,f^*(g))&(\mbox{Yoneda lemma})\\
&\sim &\ev(f(a),g)&(\mbox{naturality of $\ev$})\\
&\sim & \hom_{\w{B}}(y_{f(a)}, g)&(\mbox{Yoneda lemma})\\
\end{array}$$
Eventually, the Yoneda lemma applied to $(\w{B})^t$ concludes the proof.
\end{proof}

\begin{prop}
Let $i:A\to B$ a functor between $\U$-small $\io$-categories. The left Kan extension of $y^B\circ i$ along $y^A$ is given by $i_!$.
\end{prop}
\begin{proof}
Let $i:A\to B$ be any functor. Remark first that the Yoneda lemma and the corollary \ref{cor:Kan extension of Yonedal along i} imply that the left Kan extension of $y:A\to \w{A}$ along $y:A\to \w{A}$ is the identity of $\w{A}$.
We then have a sequence of equivalences
$$
\begin{array}{rcll}
\hom_{\uHom(\w{A},\w{A})}(i_!,f)&\sim &\hom_{\uHom(\w{A},\w{A})}(id,i^*\circ f)&(\ref{cor:naive kan extension}) \\
&\sim & \hom_{\uHom(A,\w{A})}(y_A,i^*\circ f\circ y^A)&(\mbox{Yoneda lemma}) \\
&\sim & \hom_{\uHom(A,\w{B})}(i_! \circ y^A, f\circ y^A)&(\ref{cor:naive kan extension}) \\
&\sim & \hom_{\uHom(A,\w{B})}( y_B\circ i, f\circ y^A)&(\ref{prop:left extension commutes with Yoneda})\\
\end{array}
$$
natural in $f:\uHom(\w{A},\w{B})$.
\end{proof}

\begin{cor}
For any morphism $A\to B$ between $\U$-small $\io$-categories with $B$ lax $\U$-cocomplete, there exists a unique colimit preserving functor $\w{A}\to B$ extending $i$.
\end{cor}
\begin{proof}
Let $|\uvar|_i: \w{A}\to B$ be the functor defined in corollary \ref{cor:adjonction with prehseaves}. 
As this functor is an extension of $A$, it fulfills the desired condition, that shows the existence.
The $\io$-category of functors verifying the desired property is given by the pullback 
\[\begin{tikzcd}
	{\uHom_!(\w{A},B)_{i}} & {\uHom_!(\w{A},B)} \\
	{\{i\}} & {\uHom(A,B)}
	\arrow[from=2-1, to=2-2]
	\arrow[from=1-2, to=2-2]
	\arrow[from=1-1, to=2-1]
	\arrow[from=1-1, to=1-2]
\end{tikzcd}\]
where $\uHom_!(\w{A},B)$ is the full sub $\io$-category of $\uHom(\w{A},B)$ whose objects are colimit preserving functors.
As $|\uvar|_i$ is the left Kan extension of $i$ along the Yoneda embedding, there is a transformation 
$$|\uvar|_i\to h$$ natural in $h:\uHom(\w{A},B))_{i}$. To conclude, one has to show that for any object $h$ of $\uHom(\w{A},B))_{i}$ , $|\uvar|_i\to h$ is an equivalence, and so that for any object $f$ of $\w{A}$, $|f|_i\to h(f)$ is an equivalence. As $f$ is a lax colimit of representables as shown in theorem \ref{theo:presheaevs colimi of representable} and as both $|\uvar|_i$ and $h$ preserve lax colimits, this is immediate.
\end{proof}

\begin{cor}
Let $A,B$ and $C$ be three $\U$-small $\io$-categories with $B$ lax $\U$-cocomplete, and
$i:A\to C$ and $f:A\to B$ two functors. The left Kan extension of $i$ along $f$ is given by the composite functor.
$$B\xrightarrow{N_f}\w{A}\xrightarrow{i_!}\w{C}\xrightarrow{\laxcolim_{}} C$$
\end{cor}
\begin{proof}
We have a sequence of equivalences
$$\begin{array}{rcll}
\hom_{\uHom(C,B)}(\laxcolim_{}\circ i_!\circ N_f,h)&\sim & \hom_{\uHom(C,\w{A})}( N_f,i^*\circ y^B\circ h)\\
&\sim & \hom_{\uHom(A,\w{A})}( y^A,i^*\circ y^B\circ h\circ f)&(\ref{cor:Kan extension of Yonedal along i})\\
&\sim & \hom_{\uHom(A,\w{B})}(i_!\circ y^A, y^B\circ h\circ f)&(\ref{cor:naive kan extension})\\
&\sim & \hom_{\uHom(A,\w{B})}(y^B\circ i, y^B\circ h\circ f)&(\ref{prop:left extension commutes with Yoneda})\\
&\sim & \hom_{\uHom(A,B)}( i, h\circ f)&(\ref{theo:Yoneda lemma})
\end{array}$$
natural in $h:\uHom(C,B)$.
\end{proof}

%
%
%

\cleardoublepage
\phantomsection
\addcontentsline{toc}{part}{Index of symbols} 
\printindex[notation]
\clearpage
\phantomsection
\addcontentsline{toc}{part}{Index of notions} 
\printindex[notion]

\cleardoublepage
\phantomsection
\addcontentsline{toc}{part}{Bibliography} 
\bibliography{biblio}{}
\bibliographystyle{alpha}

\end{document}